\documentclass[a4paper,titlepage=false]{scrreprt}

\usepackage{etex}
\usepackage[utf8]{inputenc}
\usepackage{fontenc}
\usepackage{amsmath,amssymb,enumerate,fullpage,float}

\usepackage[dvipsnames]{xcolor}
\usepackage{epsfig}
\usepackage{titling,makeidx}
\usepackage{microtype}
\usepackage[english]{babel}
\usepackage{amsthm}
\usepackage{mathabx}
\usepackage{booktabs,csquotes}
\usepackage{longtable}
\usepackage{comment,setspace}
\usepackage{stmaryrd}
\usepackage{relsize,enumitem}
\usepackage{mathtools,afterpage}
\usepackage{calc,import}
\usepackage{caption,subcaption}
\usepackage{faktor}
\usepackage[all]{xy}
\usepackage{array}
\usepackage{ragged2e}
\usepackage{textcomp,xspace,mparhack}
\usepackage{tabularx,scrhack}

\usepackage[backend=bibtex,bibencoding=utf8,language=auto,style=alphabetic,sorting=nyt,maxbibnames=10,natbib=true,isbn=false,url=false]{biblatex}

\usepackage{listings}
\usepackage[capitalize,nameinlink]{cleveref}

\newcommand\numberthis{\addtocounter{equation}{1}\tag{\theequation}}

\newcommand{\RR}{\mathbb{R}}
\newcommand{\CC}{\mathbb{C}}
\newcommand{\ZZ}{\mathbb{Z}}
\newcommand{\NN}{\mathbb{N}}
\newcommand{\QQ}{\mathbb{Q}}

\newcommand{\HH}{\mathbb{H}}
\newcommand{\DD}{\mathbb{D}}
\newcommand{\LL}{\mathbb{L}}

\newcommand{\mft}{\mathfrak t}

\newcommand{\mfK}{\mathfrak K}
\newcommand{\mfh}{\mathfrak h}

\newcommand{\mcM}{\mathcal M}
\newcommand{\mcP}{\mathcal P}
\newcommand{\mcC}{\mathcal C}

\newcommand{\mcX}{\mathcal X}
\newcommand{\mcR}{\mathcal R}

\newcommand{\mcA}{\mathcal A}
\newcommand{\mcF}{\mathcal F}
\newcommand{\mcE}{\mathcal E}
\newcommand{\mcU}{\mathcal U}

\newcommand{\mcT}{\mathcal T}
\newcommand{\mcD}{\mathcal D}
\newcommand{\mcB}{\mathcal B}

\newcommand{\mcG}{\mathcal G}

\newcommand{\mcQ}{\mathcal Q}

\newcommand{\mcH}{\mathcal H}
\newcommand{\mcV}{\mathcal V}
\newcommand{\mcW}{\mathcal W}
\newcommand{\mcI}{\mathcal I}

\newcommand{\dd}{\mathrm{d}}
\newcommand{\D}{\mathrm{D}}

\newcommand{\ii}{\mathrm{i}}
\newcommand{\si}{\Sigma}
\newcommand{\del}{\partial}

\newcommand{\cin}{C^{\infty}}
\newcommand{\lra}{\longrightarrow}
\newcommand{\ol}[1]{\overline{#1}}
\newcommand{\ul}[1]{\underline{#1}}
\newcommand{\wt}[1]{\widetilde{#1}}
\newcommand{\ddt}{\frac{\dd}{\dd t}_{|_{t=0}}}
\newcommand{\glue}{\operatorname{glue}}
\newcommand{\sglue}{\operatorname{sglue}}
\newcommand{\Ad}{\operatorname{Ad}}
\newcommand{\cinl}{C^{\infty}_{\text{loc}}}
\newcommand{\ga}{g_{\text{a}}}

\newcommand{\bq}{\textbf{q}}
\newcommand{\bw}{\textbf{w}}
\newcommand{\bp}{\textbf{p}}
\newcommand{\bz}{\textbf{z}}
\newcommand{\br}{\textbf{r}}
\newcommand{\wh}[1]{\widehat{#1}}
\newcommand{\wwp}{\omega_{\text{WP}}}

\newcommand{\fgt}{\operatorname{fgt}}

\newcommand{\pr}{\operatorname{pr}}
\newcommand{\bG}{\boldsymbol{\Gamma}}

\newcommand{\id}{\operatorname{id}}
\newcommand{\Hom}{\operatorname{Hom}}

\newcommand{\Aut}{\operatorname{Aut}}

\newcommand{\ts}{\mathrm{T}}
\newcommand{\rank}{\operatorname{rk}}

\newcommand{\End}{\operatorname{End}}

\newcommand{\rot}{\operatorname{rot}}

\newcommand{\obj}{\operatorname{Obj}}
\newcommand{\mcg}{\operatorname{MCG}}
\newcommand{\lcm}{\operatorname{lcm}}
\newcommand{\Mor}{\operatorname{Mor}}
\newcommand{\Ob}{\operatorname{Ob}}
\newcommand{\FN}{\operatorname{FN}}
\newcommand{\vol}{\operatorname{vol}}
\newcommand{\Isom}{\operatorname{Isom}}
\newcommand{\dist}{\operatorname{dist}}
\newcommand{\arsinh}{\operatorname{arsinh}}

\renewcommand{\Im}{\operatorname{Im}}

\newcommand{\pt}{\mathrm{pt}}

\newcommand{\ev}{\mathrm{ev}}

\newcommand{\PSL}{\mathrm{PSL}}

\newcommand{\Diff}{\mathrm{Diff}}

\newcounter{dummy}

\theoremstyle{definition}
\newtheorem{thm}{Theorem}[chapter]
\newtheorem{prop}[thm]{Proposition}
\newtheorem{lem}[thm]{Lemma}
\newtheorem{cor}[thm]{Corollary}
\newtheorem{example}[thm]{Example}
\newtheorem{def-lemma}[thm]{Definition/Lemma}

\theoremstyle{definition}
\newtheorem{definition}[thm]{Definition}

\newtheorem{rmk}[thm]{Remark}

\title{\vspace{-2em} Symplectic Geometry of Moduli Spaces of Hurwitz Covers}
\author{Sven Prüfer\thanks{now at Deutsches Zentrum für Luft- und Raumfahrt e.V., thesis was written at Universität Augsburg}}

\makeindex

\everymath{\displaystyle}

\begin{document}

\maketitle

\begin{abstract}
We extend results by Mirzakhani in \cite{mirzakhani_weil-petersson_2007} to moduli spaces of Hurwitz covers. In particular we obtain equations relating Weil–Petersson volumes of moduli spaces of Hurwitz covers, Hurwitz numbers and certain Hurwitz cycles on Deligne--Mumford space related to those Riemann surfaces admitting Hurwitz covers of a specified branching profile. We state the precise orbifold structure of the moduli space of Hurwitz covers by applying ideas and results from Robbin--Salamon in \cite{robbin_construction_2006}. Furthermore we prove compactness of the involved moduli spaces by applying SFT-compactness in the Cieliebak--Mohnke version from \cite{cieliebak_compactness_2005}.
\end{abstract}

\tableofcontents

\chapter{Introduction}

\section{The General Picture of Hurwitz Numbers}

In 1891 Adolf Hurwitz started the systematic analysis of the counts of branched covers of Riemann surfaces with prescribed branching profile in his paper \cite{hurwitz_uber_1891} which are nowadays called \emph{Hurwitz numbers}. Since then there have been many publications devoted to this problem in various areas of mathematics such as complex analysis, combinatorics, algebraic geometry, topology and symplectic geometry. In particular there has been a growing interest in Hurwitz numbers starting in the 1990's as they appeared in mathematical physics, Gromov--Witten theory and algebraic geometry.

In order to specify a Hurwitz number one needs to fix a closed target surface $X$ as well as $n$ partitions of the degree $d\in\NN_{>0}$. One then considers equivalence classes of holomorphic maps $u:C\lra X$ of degree $d$ which have $n$ branch points whose fibres have degrees corresponding to the given partitions. Two such Hurwitz covers $u:C\lra X$ and $u':C'\lra X$ are called equivalent if there exists a biholomorphism $\Phi:C\lra C'$ such that $u'=u\circ\Phi$. The Hurwitz number is now defined as the sum of these equivalence classes weighted by the inverse of the number of automorphisms of the Hurwitz cover.

Due to the Riemann existence theorem one can relate these Hurwitz numbers easily to combinatorial properties of the symmetric group $S_d$ by building an appropriate cover of $X\setminus\{n\text{ points}\}$ with monodromy around these $n$ points specified by choices of $d$ permutations with cycle type given by the partition and then gluing in holomorphic discs at the punctures. This way Hurwitz was able to calculate lots of examples of Hurwitz numbers, see \cite{hurwitz_uber_1891} and \cite{hurwitz_ueber_1901}. Nowadays this can be used to calculate any concrete Hurwitz number with the help of a computer. However, these algorithms might become very slow for higher degrees.

Besides the pure interest in calculating these numbers it turns out that they are related to various other areas of mathematics. One example of a rather obvious connection is the relation to Gromov--Witten theory. This subject, originally introduced by Mikhail Gromov in \cite{gromov_pseudo_1985} in 1985, consists of the study of moduli spaces of $J$-holomorphic maps from a Riemann surface into a symplectic manifold. So-called Gromov--Witten invariants count equivalence classes of these maps with prescribed marked points or tangential conditions. In this picture Hurwitz numbers appear as numbers related to the Gromov--Witten invariants of the target Riemann surface $X$. So the theory of Hurwitz numbers appears as the lowest-dimensional example of Gromov--Witten theory which does not have a point as the target.\footnote{Note that the Gromov--Witten theory of a point corresponds to intersection theory on the Deligne--Mumford space which is far from trivial!} This point of view was used for example in \cite{okounkov_gromov-witten_2009} by Andrei Okounkov and Rahul Pandharipande by applying virtual localization techniques to equivariant Gromov--Witten theory on $\CC P^1$.

It is common to restrict one's attention to a subclass of Hurwitz numbers. For example one can look at hyperelliptic ones which have a degree-$2$ cover over the sphere. One very famous class are simple Hurwitz numbers corresponding to an arbitrary branching profile over a special point and with all other branch points being simple, i.e.\ only one degree $2$ critical point and all others unbranched in such a fibre. The double Hurwitz numbers are the same definition just with arbitrarily chosen branching profiles over two such points. The simple Hurwitz numbers have been intensely studied by e.g.\ Hurwitz in \cite{hurwitz_uber_1891} and Ekedahl--Lando--Shapiro--Vainshtein in \cite{ekedahl_hurwitz_2001}. In the latter they prove the so-called ELSV formula

\begin{equation}
  h_{g;k_1,\ldots,k_n}=\frac{(K+n+2g-2)!}{\#\Aut(k_1,\ldots,k_n)}\prod_{i=1}^n\frac{k_i^{k_i}}{k_i!}\int_{\ol{\mcM}_{g,n}}\frac{c(\Lambda^{\ast}_{g,n})}{(1-k_1\psi_1)\cdots (1-k_n\psi_n)},
  \label{eq:elsv}
\end{equation}

where $g$ is the source genus, $K=k_1+\cdots+k_n$ is the sum of the branching degrees over the special point, $\ol{\mcM}_{g,n}$ is the Deligne--Mumford space, $c(\Lambda^{\ast}_{g,n})$ is the total Chern class of the dual of the Hodge bundle over Deligne--Mumford space and $\psi_i$ is the $i$-th $\psi$-class on $\ol{\mcM}_{g,n}$. \cref{eq:elsv} relates Hurwitz numbers with intersection numbers on Deligne--Mumford space. One might expect that there is some relation between these two things as for a fixed target Riemann surface together with a fixed covering one obtains a unique complex structure on the source such that the covering is holomorphic. So Hurwitz numbers count certain Riemann surfaces admitting Hurwitz covers as a source which in turn might be expressable by intersections of characteristic classes that contain some kind of geometric information such as $\lambda$-, $\psi$- and $\kappa$-classes on Deligne--Mumford space.

However, the relation between Hurwitz numbers and Deligne--Mumford spaces goes further than that. Suppose we consider the moduli space of Hurwitz covers with varying source and target complex structure but fixed branching profile. Denote this space by $\mcM_{g,k,h,n}(T)$, where $(g,k)$ and $(h,n)$ correspond to the genera and marked points on the source and target surface, respectively, and $T$ denoting the branching profile. Then we have maps

\begin{equation}
  \xymatrix{
    & \mcM_{g,k,h,n}(T) \ar[ld]^{\fgt} \ar[dr]^{\ev}& \\
    \ol{\mcM}_{g,k} & & \ol{\mcM}_{h,n}
    }
  \label{eq:correspondence}
\end{equation}

which assign to a Hurwitz cover its source and target surface, respectively. Such a triangle is called a correspondence in algebraic geometry and is a main tool in determining the intersection theory of the target spaces, see e.g.\ \cite{fulton_intersection_1998}. Together with the additional structure of Deligne--Mumford spaces such as forgetful maps and gluing maps describing the compactification divisors geometers have been able to say a lot on the intersection theory of Deligne--Mumford spaces.

Regarding intersection theory on Deligne--Mumford spaces there is another famous result by Maryam Mirzakhani in \cite{mirzakhani_weil-petersson_2007} relating Weil--Petersson symplectic volumes of moduli spaces of bordered Riemann surfaces and $\psi$-class intersections on Deligne--Mumford space. This is achieved by applying a localization method for Hamiltonian torus actions on suitable moduli spaces of hyperbolic Riemann surfaces. She then proceeds to cut the surfaces into smaller ones and applies a generalized McShane identity for lengths of closed simple geodesics on hyperbolic surfaces to obtain a recursion formula for the Weil--Petersson volumes as well as the $\psi$-class intersections on Deligne--Mumford space. In this thesis we will adapt the first part of these ideas to moduli spaces of Hurwitz covers.

Note that one corollary that Mirzakhani obtained was that the generating function of $\psi$-class intersections on Deligne--Mumford space satisfies the Korteweg--de Vries hierarchy, a system of partial differential equations originally coming from the theory of waves on water surfaces. This result was first shown by Maxim Kontsevich in \cite{kontsevich_intersection_1992}. Similarly the generating functions of various types of Hurwitz numbers are related to other hierarchies, too. In particular Okounkov proved in \cite{okounkov_toda_2000} that the generating function of double Hurwitz numbers satisfies the Toda lattice hierarchy. This in turn implies that the generating function of connected simple Hurwitz numbers is a solution to the Kadomtsev--Petviashvili hierarchy, as was shown in \cite{kazarian_algebro-geometric_2007}. These properties of generating function of Hurwitz numbers make them very interesting to mathematical physicists as these hierarchies appear in various toy models for quantum or topological gravity.

Yet another approach to Hurwitz numbers is the detailed investigation of properties of the symmetric group $S_d$ as was done in \cite{goulden_transitive_1997}. They obtain the cut--and--join equation whose name refers to the distinction that multiplying a permutation with a transposition can either increase the number of cycles by one or decrease it by one. As we do not want to introduce too much notation we will just state that this can be formulated as a second-order partial differential equation for the generating function of simple Hurwitz numbers.

One more common idea for calculating simple Hurwitz numbers is to try to understand what happens if one deforms the target Riemann surface of a Hurwitz cover by moving a simple branch point to the special branch point. This makes the target surface nodal but it is indeed possible to understand the relation between the corresponding Hurwitz numbers. This gives a degeneration formula which can be found for example in the appendix of \cite{okounkov_gromov-witten_2009}. Note that this can be a viable strategy in more general cases: Deform the target Riemann surface to make it completely nodal such that all its components are spheres, calculate their Hurwitz numbers and then put them back together by recalling how the Hurwitz numbers change when collapsing curves.

The theory of double Hurwitz numbers was investigated in detail in \cite{okounkov_gromov-witten_2006}. They introduced \emph{completed cycles} as special formal linear combinations of cycles in the symmetric group which could be used to formulate a precise relation between gravitational descendants\footnote{Gravitational descendants refers to the fact that one incorporates integrals of $\ev^{\ast}\psi$-classes over the Gromov--Witten moduli space.} of relative Gromov--Witten numbers and Hurwitz numbers involving these completed cycles. This allowed them to calculate for example generating functions for Gromov--Witten invariants of elliptic curves in the so-called infinite wedge formalism, a special operator calculus adapted to calculations involving Toda hierarchies.

We have seen that there exist many different approaches to Hurwitz numbers and we will explain below how this thesis makes use of some of them, in particular the correspondence of Deligne--Mumford spaces, Mirzakhanis ideas of applying fixed-point formulas for Hamiltonian torus actions and relating Hurwitz numbers with $\psi$-class intersections on Deligne--Mumford space. Before describing our approach in more detail let us say a few more words on compactifications and generalizations.

Hurwitz numbers are defined as counts of smooth Hurwitz covers meaning that both source and target are smooth Riemann surfaces. One major ingredient of Gromov--Witten theory as well as intersection theory is compactness of the underlying moduli spaces, so we need some kind of compactification of the space of smooth Hurwitz covers. There are different ways how to approach this. One idea for simple Hurwitz covers with a sphere as the target is for example to consider them as meromorphic functions on a Riemann surface with $\infty$ as the special point. In \cite{ekedahl_hurwitz_2001} this space is embedded into the space of generalized principal parts built from Laurent coefficients of the branched cover at the critical points in the fibre over $\infty$. They then consider the closure of the space of the set of smooth Hurwitz covers in that space and denote this as the \emph{completed Hurwitz space}. Another approach is the \emph{stable-maps compactification} used e.g.\ in \cite{okounkov_gromov-witten_2006} and \cite{okounkov_gromov-witten_2009}. This treats the Hurwitz covers as holomorphic maps and adds stable maps as in Gromov--Witten theory. This means that one allows for example constant components in the domain as well as components with an unstable underlying domain. One advantage is that one can add arbitrary marked points to the moduli space which gives additional structure that can be used. However, this space adds components ``far away'' from the smooth locus, for example there are moduli spaces of stable maps consisting entirely of nodal maps. Another common compactification of the space of Hurwitz covers are \emph{admissible covers}. Here one adds holomorphic maps between nodal Hurwitz covers which map nodes to nodes and such that the cover has the same degree at the nodes from both sides. These degrees are called \emph{kissing numbers}. Admissible covers are used for example by Abramovich, Corti and Vistoli in \cite{abramovich_dan_twisted_2001}.

There are various ways how one can generalize Hurwitz numbers, we will mention only two. One of them is to include a spin structure on the target surface in the data of the Hurwitz cover and then count weighted equivalence classes of these pairs of Hurwitz covers with suitable spin structures. The resulting number is called a \emph{spin Hurwitz number} and various equations similar to results on relative Gromov--Witten numbers were obtained by Junho Lee and Thomas H. Parker in \cite{parker_recursion_2013} using concentration principles for elliptic operators. Interestingly the same result was obtained independently using completely different techniques, namely topological quantum field theories, by Sam Gunningham in \cite{gunningham_spin_2016}. Another generalization that has become very popular in the recent years are \emph{tropical Hurwitz numbers}. In \cite{cavalieri_tropical_2010} Renzo Cavalieri, Paul Johnson and Hannah Markwig define a tropical version of Hurwitz numbers and use them to reprove for example the cut--and--join relation.

\section{Our Approach}

Now that we have mentioned various old and some new ideas surrounding Hurwitz numbers we want to see how we will approach the topic in this thesis and how it relates to the other techniques. The primary goal of this thesis is to apply the symplectic techniques from \cite{mirzakhani_weil-petersson_2007} to moduli spaces of Hurwitz covers. This means that we will define a moduli space of Hurwitz covers with varying source and target complex structures giving us maps to the source and target Deligne--Mumford space, respectively, as in \cref{eq:correspondence}.

We will mark the covers in the following way: Every branch point and every preimage of a branch point will be marked, this means that even if a branch point has non-critical preimages we still mark them. Also we do not mark any other points of the source. Both of these conditions are in fact crucial for two reasons. First, this will allow us to use a mixture of stable-maps and admissible-cover compactifications which captures the geometric intuition about nodal Hurwitz covers very well and allows us to apply the techniques from \cite{robbin_construction_2006} for the orbifold structure of the moduli space. Secondly it solves the issue about differences of various $\psi$-classes in Gromov--Witten theory. In Gromov--Witten theory there are two different notions of $\psi$-classes. One comes from the pull-backs of the $\psi$-classes on Deligne--Mumford space via the evaluation map and the others are directly defined on the Gromov--Witten moduli space as the Chern classes of the line bundle whose fibre is the cotangent space at a marked point. But since Gromov--Witten moduli spaces contain maps whose underlying domains are not stable these $\psi$-classes are different. By marking points in our way we will exclude non-stable domains which allows us to prove that the $\psi$-classes are in fact equal.

Similar to \cite{mirzakhani_weil-petersson_2007} we will consider also moduli spaces of bordered Hurwitz covers equipped with marked points on the boundary. This will give us the structure of a torus action which turns out to be Hamiltonian for the Weil--Petersson symplectic structure. We will then use the Duistermaat--Heckman localization theorem to relate $\psi$-classes of these torus bundles with volume integrals. The volume integrals will be computed from the Weil--Petersson volume of Deligne--Mumford spaces of bordered Riemann surfaces and the degree of the evaluation map given by Hurwitz numbers. This way there will be lots of hyperbolic aspects of Hurwitz covers entering the discussion. However, we will \emph{not} apply the McShane identity in the end because it seems not clear how to simply cut Hurwitz covers in a usable way along closed simple geodesics in the source surface.

Regarding subclasses of Hurwitz numbers we do not need to specialize beyond requiring that the underlying surfaces are stable when including the marked points. It is possible that the techniques developed in this thesis might simplify in special situations but apart from calculating examples we have not pursued this direction any further.

\section{Main Results}

The main results of this thesis are as follows. Let $\mcR_{h,n}$ be the groupoid of stable Riemann surfaces of type $(h,n)$ as well as $\mcR_{g,k,h,n}(T)$ the groupoid of Hurwitz covers $(C,u,X,\bq,\bp)$, where $u:C\lra X$ is a branched cover such that $\bp$ is a tuple of points on $X$ containing all branch points and $\bq$ is precisely the set of preimages of the points $\bp$, $u$ gives a bijection between the set of nodes of the two surfaces and it has the same degree on both sides of a node and $u$ satisfies the branching profile specified in $T$. An orbifold structure on $|\mcR_{g,k,h,n}(T)|$ is an orbifold category $\mcM_{g,k,h,n}(T)$ together with a functor $\mcM_{g,k,h,n}(T)\lra\mcR_{g,k,h,n}(T)$ which is a bijection on the orbit spaces. Furthermore we define the \emph{Hurwitz number} $H_{g,k,h,n}(T)\in\QQ$ as the sum over all equivalence classes of branched covers of fixed topological type and fixed branching profile weigthed by the inverse of the size of the automorphism group.\footnote{Note, however, that our Hurwitz numbers differ from the usual ones by a combinatorial factor explained in \cref{sec:hurwitz-numbers}.} In \cref{sec:main-results} we prove

\begin{thm}
  There exists an orbifold category $\mcM_{g,k,h,n}(T)$ of real dimension $6h-6+2n$ together with functors $\iota:\mcM_{g,k,h,n}(T)\lra\mcR_{g,k,h,n}(T)$ and $\ev:\mcM_{g,k,h,n}(T)\lra\mcM_{h,n}$ such that\footnote{Here, $\mcM_{h,n}$ denotes the orbifold category built from universal unfoldings as in \cite{robbin_construction_2006}.}
  \begin{itemize}
    \item the functor $\iota$ is essentially surjective, faithful and full on the full subcategories of smooth Hurwitz covers and smooth Riemann surfaces and has finite preimages on the locus of nodal Hurwitz covers and nodal Riemann surfaces,
    \item the functor $\ev$ makes the following diagram commute:
      \begin{equation*}
        \xymatrix{
          \mcM_{g,k,h,n}(T) \ar[r]^{\iota} \ar[d]^{\ev} & \mcR_{g,k,h,n}(T) \ar[d]^{\ev} \\
          \mcM_{h,n} \ar[r] & \mcR_{h,n}
          }
        \end{equation*}
      \item there exist local coordinates around $(C,u,X,\bq,\bp)\in\Ob\mcM_{g,k,h,n}(T)$ as well as $(X,\bp)\in\Ob\mcM_{h,n}$ such that $\ev$ looks like
        \begin{align*}
          \DD^{3k-3+n-N}\times\DD^N & \lra\DD^{3k-3+n-N}\times\DD^N \\
          (x,z_1,\ldots,z_N) & \longmapsto \left(x,z_1^{K_1},\ldots,z_N^{K_N}\right),
        \end{align*}
        where $N$ is the number of nodes in $X$ and
      \item the functor $\ev$ is a morphism covering of orbifolds on the full subcategories of smooth Hurwitz covers with degree equal to the Hurwitz number $H_{g,k,h,n}(T)\in\QQ$.      
  \end{itemize}
  \label{thm:main-result-1}
\end{thm}

The meaning of \cref{thm:main-result-1} is as follows. The category $\mcR_{g,k,h,n}(T)$ is not by itself an orbifold category because there are morphisms which give rise to families of a lower dimension. This can be imagined by noticing that from \cite{robbin_construction_2006} we see that morphisms between surfaces extend in a unique way to surfaces nearby in the universal unfolding. So if

\begin{equation*}
  \xymatrix{
    C \ar[r]^{\Phi} \ar[d]^u & C' \ar[d]^{u'} \\
    X \ar[r]^{\phi} & X'
    }
\end{equation*}

is a morphism then $\phi$ extends to a family of morphisms of the correct dimension when varying the target surface. However, this fixes variations of the source surfaces as well and depending on $\Phi$ it is possible that these variations are not identical but only intersect in lower dimensional manifolds.\footnote{Recall that the real dimension of Deligne--Mumford space is $6h-6+2n$ but $6g-6+2k\geq 6h-6+2n$ using Riemann--Hurwitz. This shows that there is enough space to have intersections which have a lower dimension.} In the construction we will see that this phenomenon can only happen at nodal Hurwitz covers.

In any case we need to exclude these morphism in the definition of $\mcM_{g,k,h,n}(T)$ in order to obtain an orbifold. The first two statements in particular imply that for smooth Hurwitz covers the orbifold $\mcM_{g,k,h,n}(T)$ contains all information about $\mcR_{g,k,h,n}(T)$ and in particular the evaluation functor on $\mcM_{g,k,h,n}(T)$ is the same as the standard one. The other two statements say something about the local structure of the evaluation functor. In particular it is not a covering but it is branched over the locus of nodal Hurwitz covers. But again we see that on the smooth Hurwitz covers it is a covering and its (orbifold-)degree computes the Hurwitz number. This idea is the basis of many approaches to Hurwitz numbers, for example it can be found in \cite{ekedahl_hurwitz_2001} and \cite{okounkov_gromov-witten_2009}.

The second main result is a rigorous proof of compactness of $|\mcM_{g,k,h,n}(T)|$. Note that in a few other compactifications such as the completed Hurwitz space from \cite{ekedahl_hurwitz_2001} compactness holds by construction and the problem consists more of finding a way to understand and describe the compactification locus. In our case, as the target surface can degenerate to a nodal surface it is not completely clear how to apply e.g.\ Gromov compactness to such a surface. Instead we apply SFT-compactness developed by Frédéric Bourgeois, Yakov Eliashberg, Helmut Hofer, Kris Wysocki and Eduard Zehnder in \cite{bourgeois_compactness_2003} in the version of Kai Cieliebak and Klaus Mohnke in \cite{cieliebak_compactness_2005}. \cref{thm:main-result-2} is proven in \cref{sec:SFT-compactness}.

\begin{thm}
  The moduli spaces of Hurwitz covers $|\mcM_{g,k,h,n}(T)|$ and $|\mcR_{g,k,h,n}(T)|$ are compact.
  \label{thm:main-result-2}
\end{thm}

The third main result deals with the locus of source surfaces admitting a Hurwitz cover of prescribed type. This is encoded by the \emph{Hurwitz class}
\begin{equation*}
  D_{g,k,h,n}(T)\coloneqq \fgt_*[\mcM_{g,k,h,n}]\in H_{6g-6+2n}(|\mcM_{g,k}|,\QQ),
\end{equation*}
where $|\mcM_{g,k}|$ is the Deligne--Mumford space of the source surface. Here we use the \emph{forgetful functor}
\begin{equation*}
  \fgt(C,u,X,\bq,\bp)\coloneqq(C,\bq).
\end{equation*}
We denote by $V_{h,n}(L)$ the Weil--Petersson volume of the moduli space of bordered Riemann surfaces
\begin{equation*}
  V_{h,n}(L)\coloneqq \int_{\mcM_{h,n}(L)}\frac{\wwp^{3h-3+n}}{(2\pi)^{3h-3+n}},
\end{equation*}
 which is a polynomial in $L_i^2$, where the $L_i$ with $i=1,\ldots,n$ are hyperbolic lengths of the geodesic boundaries.

In \cref{sec:application-duistermaat-heckman} we prove \cref{thm:main-result-3}.

\begin{thm}
  We have
  \begin{align*}
    & K \cdot H_{g,k,h,n}(T)V_{h,n}(L)[3h+n-3] = \\
    & \sum_{\substack{\alpha\in\NN^n,\\ |\alpha|= 3h+n-3}} \sum_{\substack{\beta_1\in\NN^{I_1}\\ |\beta_1|=\alpha_1}}\cdots \sum_{\substack{\beta_n\in\NN^{I_n}\\ |\beta_n|=\alpha_n}} \frac{L^{2\alpha}l_1^{2(\beta_{\nu(1)})_1}\cdots l_k^{2(\beta_{\nu(k)})_k}}{(2d)^{3h+n-3}\beta_1!\cdots\beta_n!} \left\langle {\psi_1}^{(\beta_{\nu(1)})_1}\cdots{\psi_k}^{(\beta_{\nu(k)})_k},D_{g,k,h,n}(T) \right\rangle,
  \end{align*}
  where $[3h+n-3]$ denotes the terms of degree $3h+n-3$ of $V_{h,n}$, $d$ is the degree of the Hurwitz cover, $l_j$ are the degrees of the Hurwitz cover at $q_j\in C$, $K\coloneqq \prod_{i=1}^nK_i$ with $K_i\coloneqq\lcm\{l_j\mid j\in\nu^{-1}(i)\}$, and $\psi_j\in H^2(|\mcM_{g,k}|,\QQ)$ are the $\Psi$-classes on the source Deligne--Mumford space.
  \label{thm:main-result-3}
\end{thm}
Note that this equation relates Hurwitz numbers, Weil--Petersson volumes of moduli spaces of bordered surfaces and $\Psi$-class intersections on Deligne--Mumford space. In \cref{sec:appl-exampl} we will workout a few explicit examples.

\section{A Few Comments on Notation}

Before giving an overview of the individual chapters let us give a warning and some comments about notation. Looking into \cite{robbin_construction_2006} one sees that the notation for a rigorous construction of the orbifold structure on Deligne--Mumford space is already pretty involved. In our case every object and morphism contains twice as many surfaces plus a map between them. This forces us to make a few compromises and to try to stay consistent with our use of symbols. In particular we will try to obey the following rules:

\begin{itemize}
  \item Branching profile data will be denoted by $T$'s.
  \item The topological type of Riemann surfaces will be $(g,k)$ for source surfaces and $(h,n)$ for target surfaces.
  \item Source Riemann surfaces will be denoted by capital roman letters at the beginning of the alphabet so in particular $C$ and $D$, target Riemann surfaces will be denoted by capital roman letters at the end of the alphabet.
  \item Whenever we put hyperbolic metrics on a Riemann surface $C$ with marked points $\bq$ such that these points are cusps we will talk about a hyperbolic metric on $C$ although this metric is not defined at the cusps.
  \item Marked points on the source surface will be denoted by $q$'s and marked points on the target surface by $p$'s.
  \item Tuples will often be denoted by boldface letters.
    \item Running indices for marked points on the source will be $j$'s and those for marked points on the target will be $i$'s. However, we will use these indices also for completely different purposes.
  \item Biholomorphisms will usually be denoted by variations of the Greek letter $\phi$, in particular $\Phi$ for biholomorphisms of the source surface and $\phi$ or $\varphi$ for those of the target surface.
  \item Quite often we will have to do constructions at nodes. In that case we will denote the two sides of a node in the normalization of the surface by $\dagger$ and $\ast$, respectively.
  \item General categories will be denoted by $\mcC, \mcG$ or $\mcH$. This will be important in the chapter on general orbifold structures. Categories of surfaces or Hurwitz covers will be denoted by $\mcR$, in particular the actual moduli space as a set is the orbit space $|\mcR|$.\footnote{Note that morphisms will always be isomorphisms of the corresponding objects and thus categories will be groupoids.}
  \item A category that carries more structure, in particular an orbifold category, will be denoted by $\mcM$. This means that the $\mcM$ categories will contain far fewer objects and morphisms than the corresponding $\mcR$ ones. This is the same idea as in \cite{robbin_construction_2006} or \cite{hofer_applications_2011}.
  \item In particular the Deligne--Mumford space will be denoted by $|\mcM_{g,k}|$ or $|\mcR_{g,k}|$. Both spaces will be in bijection but the former category has a manifold structure whereas the latter contains \emph{all} possible Riemann surfaces as objects.
\end{itemize}

\section{Chapter Overview}

We will now give an overview over the individual chapters.

In \cref{sec:fundamentals-hurwitz-number} we define our notion of Hurwitz covers. Recall that there are a few different notions in particular in regards to nodal Hurwitz covers and marked points. One issue we have is that we need to enumerate marked points on the source surface. This means that we have a different notion of automorphisms and thus need to be careful how our notion of Hurwitz numbers compares to the usual one. Next we prove a nodal version of the Riemann--Hurwitz formula. This is because in \cref{sec:SFT-compactness} we need to consider the limit of this equation for nodal curves in order to prove compactness. Furthermore we give a small recap of the combinatorial description of Hurwitz numbers for completeness.

Afterwards we collect various fundamental statements about hyperbolic geometry of Riemann surfaces in \cref{sec:basics-hyperbolic-geometry}. We need this for our later considerations regarding the Weil--Petersson symplectic structure as well as the compactness theorem. This includes some statements on uniformization, multicurves, Teichmüller spaces and collar neighborhoods. The latter subsection is particularly important for the discussion of the gluing map which is used to treat bordered Hurwitz covers in \cref{sec:orbifold-structure-moduli-space-borderd-huwritz-covers}.

The \cref{sec:orbifolds} deals with orbifolds. We recall fundamental definitions and theorems and give a few comments on orbifold structures and maps between orbifolds. In particular we give a mostly self-contained introduction to bundles over orbifolds as well as their symplectic geometry. Besides this we introduce the notion of a \emph{morphism covering} which captures the properties of the evaluation functor on the category of smooth Hurwitz covers. This is different from an orbifold covering as it can have more automorphisms on the source orbifold whereas an orbifold cover has the base automorphism group acting on the fibre.

In \cref{chap:orbi-structure} we do the main technical work, constructing the orbifold structure on the moduli space of Hurwitz covers. This essentially applies the ideas from \cite{robbin_construction_2006} to our case. A local calculation at nodes shows how a variation of the target complex structure can be lifted to the source complex structure. This then allows to define deformation families by gluing in an appropriate way at the nodes. The object manifold of $\mcM_{g,k,h,n}(T)$ will then be defined by

\begin{equation*}
  \Ob\mcM_{g,k,h,n}(T)\coloneqq \bigsqcup_{\lambda\in\Lambda}O^{\lambda},
\end{equation*}

where $O^{\lambda}$ are domains of the gluing maps $\Psi^{\lambda}:O^{\lambda}\lra\Ob\mcR_{g,k,h,n}(T)$ and the $\lambda$ contain all the necessary choices for the gluing construction. This is the same procedure as in \cite{robbin_construction_2006}. We then prove that morphisms come in families with a dimension determined by the relation between certain holomorphic discs in the universal unfolding of the source surface. This allows us to define the morphisms of $\mcM_{g,k,h,n}(T)$ by only using those which have a family of the appropriate dimension. We also prove \cref{thm:main-result-1} in that chapter.

Besides the usual Hurwitz covers we also need to deal with bordered Hurwitz covers similar to Mirzakhani. To this end we define a gluing map in \cref{sec:orbifold-structure-moduli-space-borderd-huwritz-covers} which glues a hyperbolic pair of pants to the target surface with two marked points. We extend the source surface by gluing pairs of pants at every boundary component and extend the map such that one new marked point is maximally branched and the other is regular. This allows us to pull back the orbifold structure from a covering of the orbifold of closed Hurwitz covers.

In \cref{sec:SFT-compactness} we prove \cref{thm:main-result-2}, i.e.\ the compactness result. We first define the topologies on all the moduli spaces and then recall SFT-compactness from \cite{cieliebak_compactness_2005}. Next we construct various objects on a sequence of Hurwitz covers such that we can interpret this sequence as a neck-stretching sequence. After applying SFT-compactness to this neck-stretching sequence we then show that the limit object is in fact a Hurwitz cover and that SFT-convergence implies converges in our moduli space.

After having established all the technical statements we define the Weil--Petersson symplectic structure together with various Hamiltonian torus actions on our moduli spaces in \cref{sec:sympl-geom-moduli}. This is done akin to \cite{mirzakhani_weil-petersson_2007}. Note that we define the various cotangent line bundles on moduli spaces needed for $\psi$-classes in a rigorous way. Although this is basic we could not find detailed explanations and proofs in the language of orbifolds and so provide them here. Apart from this we use the Duistermaat-Heckman theorem in this chapter to prove \cref{thm:main-result-3}.

In the last \cref{sec:appl-exampl} we calculate the formula \cref{thm:main-result-3} in various examples. In particular we calculate Hurwitz numbers with the help of a computer program and the Weil--Petersson volumes using Mirzakhanis recursion relation. This way we obtain explicit formulas for evaluations of $\psi$-classes on the Hurwitz class $D_{g,k,h,n}(T)$.

For completeness we give some more details in the appendix in \cref{appendix}. This includes details on the shifting maps in the SFT-compactness statement, calculations of the Weil--Petersson volumes used in \cref{sec:appl-exampl} as well as the SAGE/Python source code used for calculating the Hurwitz numbers.

\section{Acknowledgments}

%\pdfbookmark[1]{Acknowledgments}{acknowledgments}

This document is the authors PhD thesis, written at the Universität Augsburg under the supervision of Kai Cieliebak and handed in on May 24th, 2017. At the time of creation of this version the author was a member of the Deutsches Zentrum für Luft- und Raumfahrt e.V. During parts of the work on this thesis the author was supported by the Studienstiftung des Deutschen Volkes.

The author would like to thank his parents and his family for their support throughout his life. Also the author is very grateful to Meru Alagalingam, Ingo Blechschmidt, Kathrin Helmsauer, Rüdiger Kürsten, Manousos Maridakis, \'Akos Nagy and Peter Uebele for fruitful discussions and a lot of support.

Furthermore the author wants to thank the various people who have helped with their insights and discussions during various research stays, in particular Rahul Pandharipande, Tom Parker, Dietmar Salamon and Kai Zehmisch.

And of course the author wants to thank his advisor Kai Cieliebak for always being supportive and teaching him symplectic topology.

\chapter{Fundamentals on Hurwitz Numbers}

\label{sec:fundamentals-hurwitz-number}

\section{Combinatorial Data and Hurwitz Covers}

\label{sec:indexc-data}
\index{Hurwitz Cover}

First we will fix the auxiliary parameters for our Hurwitz covers. We will try to be as consistent as possible and use the same types of indices and abbreviations throughout the whole document.

The source surface will always be denoted by capital Latin letters $C,D,\ldots$ and the target surface by later capital Latin letters like $X,Y,\ldots$. These surfaces may be closed or compact with boundary and this will usually be fixed at the beginning of the chapters. The source genus will be usually denoted by $g$ and the target genus by $h$.

If we have marked points or boundary components on the source these will be enumerated from $1$ to $k$ using the letter $j$ as index and $q$ for the point or $\del C$ for the boundary components, i.e.\ $q_j\in C$ or $\del_jC$, respectively. The marked points $p$ or boundary components $\del X$ on the target will be enumerated by $i=1,\ldots,n$, i.e.\ $p_i\in X$ or $\del_iX$, respectively.

Also we fix a surjective map $\nu:\{1,\ldots,k\}\lra\{1,\ldots,n\}$ and partitions $T_1,\ldots,T_n$ of a fixed natural number $d$, called the degree of the Hurwitz covering. We require that the length of the partition $T_i$ is equal to $|\nu^{-1}(i)|$. In principle this is enough to define Hurwitz numbers but we will make a few more choices in order to simplify our treatment of the corresponding moduli spaces. This means that we do in fact not just enumerate the $n$ branched points but enumerate \emph{all} the $k$ critical points. This means we write

\begin{equation*}
  T_i=\sum_{j\in\nu^{-1}(i)}l_j\qquad\forall i=1,\ldots,n.
\end{equation*}

\index{Combinatorial Data}

These numbers $l_1,\ldots,l_k$ will be part of the combinatorial data we choose. Summarizing, the \emph{combinatorial data} that we fix is given by a tuple

\begin{equation*}
  T=\left(d,\nu,\{T_i\}_{i=1}^n,\{l_j\}_{j=1}^k\right),
\end{equation*}

where of course some of the elements contain the same information. Our main object of interest will be \emph{Hurwitz covers}.

\index{Hurwitz Cover!Closed}

\begin{definition}
  Given some combinatorial data as above, the category $\mcR_{g,k,h,n}(T)$ of \emph{closed Hurwitz covers} is defined as follows. Its objects and morphisms are given by
  \begin{align*}
    \obj\mcR_{g,k,h,n}(T) & \coloneqq \left\{(C,u,X,\bq,\bp)\right\}, \\
    \Mor_{\mcR_{g,k,h,n}(T)}\left( (C,u,X,\bq,\bp),(C',u',X',\bq',\bp') \right) & \coloneqq \left\{(\Phi,\phi)\right\},
  \end{align*}
  where
  \begin{itemize}
    \item $C$ is a connected closed stable nodal Riemann surface of arithmetic genus $g$,
    \item $X$ is a connected closed stable nodal Riemann surface of arithmetic genus $h$,
    \item $u:C\lra X$ is a holomorphic map which is non-constant on every nodal component,\footnote{Here, a \emph{nodal} or \emph{smooth} or \emph{irreducible} component is a connected component of the normalization of a Riemann surface.}
    \item $\bq=(q_1,\ldots,q_k)$ is a set of pairwise distinct points $q_j\in C$,
    \item $\bp=(p_1,\ldots,p_n)$ is a set of pairwise distinct points $p_i\in X$,
    \item $\Phi:C\lra C'$ is a biholomorphism,
    \item $\phi:X\lra X'$ is a biholomorphism,
  \end{itemize}
  such that
  \begin{itemize}
    \item $u(q_j)=p_{\nu(j)}$ for all $j=1,\ldots,k$, $u$ maps nodes to nodes and all preimages of nodes are nodes,
    \item all critical points of $u$ are contained in $\bq$ or are nodes of $C$ and thus all branch points of $u$ are contained in $\bp$ or are nodes of $X$,
    \item the branching profile of $u$ over $p_i$ is given by $T_i$ and $\deg_{q_j}u=l_j$ for all $j=1,\ldots,k$ and $i=1,\ldots,n$,
    \item the degrees of $u$ at the two sides of a node on $C$ coincide and $u$ is surjective on a neighborhood of the target node when restricted to a small neighborhood of any node,\footnote{We will call this property \emph{locally surjective at nodes} for later reference. This means that $u$ can not map both sides of a node $q\in C$ to only one side of a node $p\in X$, i.e.\ it excludes the third Figure in Figure~\cref{fig:bad-hurwitz-covers}.} 
    \item $\Phi(q_j)=q_j'$ and $\phi(p_i)=p_i'$ for all $j=1,\ldots,k$ and $i=1,\ldots,n$ and
    \item the following diagram
      \begin{equation*}
        \xymatrix{ C \ar[r]^{\Phi} \ar[d]^{u} & C' \ar[d]^{u'} \\
          X \ar[r]^{\phi} & X'
          }
      \end{equation*}
      commutes.
  \end{itemize}
  \label{def:hurwitz-cover}
\end{definition}

\index{Locally Surjective at Nodes}

\begin{figure}[!ht]
    \centering
    \def\svgwidth{0.5\textwidth}
    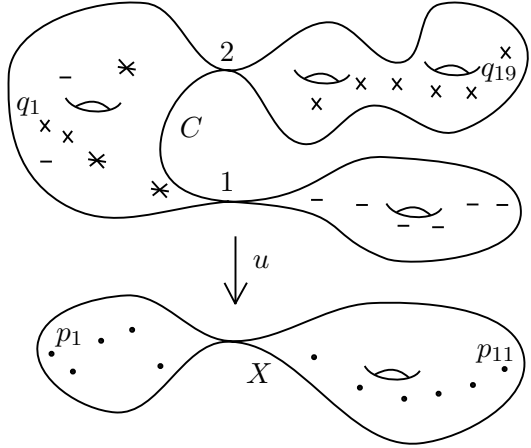
    \caption{This is an example of a (nodal) Hurwitz cover of type $(3,\nu,\{2+1,2+1,3,3,3,2+1,2+1,2+1,2+1,2+1,2+1\},$ $(2,1,1,2,3,3,3,2,1,2,1,2,1,2,1,2,1,2,1))$ where we have not written all the enumerations of the various points for readability. However, we can infer e.g.\ $\nu(1)=1$ and $\nu(19)=11)$.}
    \label{fig:example-hurwitz-cover}
  \end{figure}

\begin{rmk}
  Note that in the definition of $\mcR_{g,k,h,n}(T)$ we do not fix the target surface and instead allow for morphisms to change the target surface as well. Also note that we immediately allow for \emph{nodal Hurwitz covers}, i.e.\ those where $C$ and $X$ contain nodes  in contrast to the usual \emph{smooth Hurwitz covers}, where $C$ and $X$ are smooth. Although the preimage of a node cannot be a circle since $u$ is nonconstant, we need to explicitly require that nodes are mapped to nodes and that there are no smooth points mapped to nodes as these objects cannot arise from the kind of degeneration that we want to consider in this thesis. This excludes the following cases.

\index{Hurwitz Cover!Nodal}
\index{Hurwitz Cover!Smooth}

\begin{figure}[!ht]
    %\centering
    \def\svgwidth{\textwidth}
    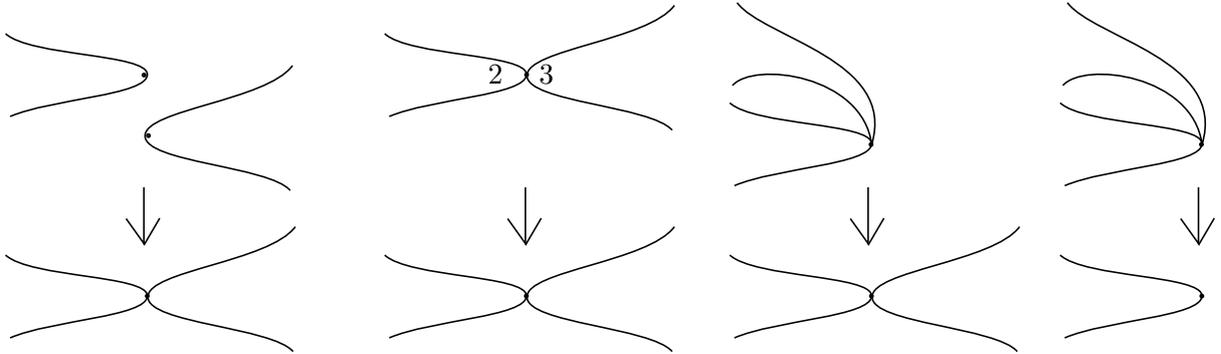
    \caption{The various conditions on the nodes in Definition~\cref{def:hurwitz-cover} of Hurwitz covers exclude in particular the above illustrated cases. From left to right these pictures correspond to non-nodal points being mapped to nodes, the two sides of a node having different degrees, two sides of a node mapped to just one side of a node and a node mapped to a smooth point. One other important forbidden case not depicted is the one of a constant component.}
    \label{fig:bad-hurwitz-covers}
  \end{figure}

We need to understand how the various possible definitions of Hurwitz covers relate to each other which will be explained in the next section.
\end{rmk}

\section{Hurwitz Numbers}

\label{sec:hurwitz-numbers}

The usual definition of a Hurwitz number is as follows. Fix a smooth closed target surface $X$ of genus $h$ as well as $n$ branched points $p_1,\ldots,p_n\in X$. Furthermore we are given $n$ partitions $T_1,\ldots,T_n$ of the degree $d\in\NN$. Now a \emph{standard Hurwitz cover} is a pair $(C,u)$ with $C$ a smooth closed Riemann surface and $u:C\lra X$ such that $u$ is holomorphic, its branched points are given by $p_1,\ldots,p_n$ and the branching profile over $p_i$ is given by $T_i$. Two such standard Hurwitz covers $(C,u)$ and $(C',u')$ are called equivalent if there exists a biholomorphism $\phi:C\lra C'$ such that
\begin{equation}
  \xymatrix{
    C \ar[dr]_u \ar[rr]^{\phi} & & C' \ar[dl]^{u'} \\
    & X &
  }
  \label{eq:automorphism-hurwitz-cover}
\end{equation}
commutes. Automorphisms of $(C,u)$ are such maps $\phi:C\lra C$. Then we define the standard Hurwitz numbers as follows.
\begin{definition}
  The category of \emph{standard Hurwitz covers} $\mcR_h(T_1,\ldots,T_n,X,\bp)$ is defined by
  \begin{align*}
    \obj\mcR_h(T_1,\ldots,T_n,X,\bp) & \coloneqq \left\{(C,u) \text{ standard Hurwitz cover}\right\}, \\
    \Mor_{\mcR_h(T_1,\ldots,T_n,X,\bp)}\left( (C,u),(C',u') \right) & \coloneqq \left\{\Phi:C\lra C'\text{ a biholomorphism}\right. \\
     & \left. \phantom{\coloneqq}\quad \text{s.t. }u'\circ \Phi= u\right\}.
  \end{align*}
  Note that this category is a groupoid with finite automorphism groups and we define the \emph{standard Hurwitz numbers} $\mcH_h(T_1,\ldots,T_n)$ by
  \begin{equation*}
    \mcH_h(T_1,\ldots,T_n)=\sum_{[x]\in|\mcR_h(T_1,\ldots,T_n,X,\bp)|}\frac{1}{|\Aut_{\mcR_h(T_1,\ldots,T_n,X,\bp)}(x)|},
  \end{equation*}
  where $\Aut_{\mcR_h(T_1,\ldots,T_n,X,\bp)}(x)$ is the automorphism group of $x\in\Ob\mcR_h(T_1,\ldots,T_n,X,\bp)$ and thus by definition the same as the automorphisms of the Hurwitz cover in the sense of \cref{eq:automorphism-hurwitz-cover}, sometimes denoted by $\Aut(C,u)$ as well. By $|\mcG|$ we denote the orbit space of a groupoid which is the same as the set of equivalence classes.
\end{definition}
\index{Hurwitz Number!Standard}

\begin{rmk}
  Note that any branching profile $T_i$ determines the degree $d$ and by Riemann--Hurwitz this determines the genus of the source surface. This is why this data is usually suppressed in the notation. Also in contrast to the objects in our Hurwitz spaces from \cref{sec:indexc-data} the branched points are fixed and the critical points are \emph{not} numbered. Furthermore one might think that this definition depends on the target $(X,p_1,\ldots,p_n)$ but the purely combinatorial description in the next Section will show that these numbers are indeed independent of these choices.
\end{rmk}

Recall from the last section that we will consider instead tuples of Hurwitz covers where the target space $(X,\bp)$ is allowed to vary and our notion of morphism takes into account morphisms of the target surface. In order to obtain Hurwitz numbers in our setting we thus need to fix the equivalence class of the target surface instead of the actual surface. To this end we introduce the category of stable Riemann surfaces. Note that all our categories will have markings and so we will usually drop the word ``marked'' in front of ``Riemann surfaces''.
\begin{definition}
  The category $\mcR_{h,n}$ of \emph{closed stable nodal Riemann surfaces} is defined by
  \begin{align*}
    \obj\mcR_{h,n} & \coloneqq \left\{(X,\bp)\right\}\text{ and} \\
    \Mor_{\mcR_{h,n}}\left((X,\bp),(X',\bp')\right) & \coloneqq \left\{\phi:X\lra X'\right\},
  \end{align*}
  such that $X$ is a closed stable nodal Riemann surface of genus $h$ with marked points $\bp=\{p_i\}_{i=1}^n$ such that $p_i\in X$ and $p_i\neq p_j$ for all $i\neq j$. A morphism is a biholomorphism $\phi$ such that $\phi(p_i)=p_i'$ for all $i=1,\ldots,n$.
\end{definition}
\index{Deligne--Mumford Space}
\index{Evaluation Functor}

\begin{rmk}
  The category $\mcR_{h,n}$ is obviously a groupoid and its orbit space $|\mcR_{h,n}|$ is usually called \emph{Deligne--Mumford space}. Note that we have a well defined \emph{evaluation functor}
  \begin{equation*}
    \ev:\mcR_{g,k,h,n}(T)\lra\mcR_{h,n}
  \end{equation*}
  by mapping $(C,u,X,\bq,\bp)\longmapsto (X,\bp)$ and $(\Phi,\phi)\longmapsto \Phi$. This functor descends to a well-defined map
  \begin{equation*}
    \ev:\left|\mcR_{g,k,h,n}(T)\right|\lra\left|\mcR_{h,n}\right|.
  \end{equation*}
\end{rmk}

We will use the following definition for Hurwitz numbers. So now $T,g,k,h$ and $n$ are combinatorial data as in \cref{sec:indexc-data}. 
\index{Hurwitz Number}
\begin{definition}
  Given combinatorial data $T=\left(d,\nu,\{T_i\}_{i=1}^n,\{l_j\}_{j=1}^k\right)$ and an equivalence class $[Y,\br]\in \left|\mcR_{h,n}\right|$ of a smooth target surface we can define the \emph{Hurwitz number} $H_{g,k,h,n}(T)$ as
  \begin{equation}
    H_{g,k,h,n}(T)\coloneqq \left|\Aut(Y,\br)\right|\sum_{\substack{[C,u,X,\bq,\bp]\in|\mcR_{g,k,h,n}(T)| \\ \text{s.t. }\ev([C,u,X,\bq,\bp])=[Y,\br] }}\frac{1}{\left|\Aut(C,u,X,\bq,\bp)\right|}.
    \label{eq:hurwitz-number-definition}
  \end{equation}
\end{definition}

\begin{rmk}
  As we said earlier this definition differs from the standard one in two ways: First, we vary the target surface $(X,\bp)$ and thus include automorphisms of it, which forces us to account for this overcounting by introducing an additional factor. And secondly we enumerate all preimages, which means that we overcount by some combinatorial factor. This is explained in \cref{thm:rel-std-hurwitz-numbers}.
\end{rmk}

\begin{thm}
  The standard Hurwitz numbers $\mcH_h(T)$ and our version $H_{g,k,h,n}(T)$ are related by
  \begin{equation}
    \label{eq:relation-standard-hurwitz-number}
    H_{g,k,h,n}(T)=\mfK\cdot \mcH_h(T_1,\ldots,T_n),
  \end{equation}
  where the combinatorial factor $\mfK$ is given by
  \begin{equation*}
    \mfK\coloneqq\prod_{i=1}^n\prod_{u=1}^d\left(\# \{1\leq j \leq k\mid \nu(j)=i, l_j=u\}\right)!
  \end{equation*}
  and depends on $n,k,d,\nu$ and $\{l_j\}_{j=1}^k$.
  \label{thm:rel-std-hurwitz-numbers}
\end{thm}

\begin{rmk}
  Although the combinatorial factor $\mfK$ seems strange it can be easily understood. The difference between the two Hurwitz numbers comes from the fact that we enumerated the preimages of the branch points $\bp$. However, if you fix the enumeration of the critical points $\bq$ as well as the map $\nu$ the only choice you have are permutations of the marking $j$ in the same fibre and with the same degree of $u$. Thus we obtain as a factor the product of all these factorials.

Due to the two differences in the definitions of Hurwitz numbers we will introduce an intermediate Hurwitz space where the target surface varies but the critical points are not enumerated. 
\end{rmk}

\begin{definition}
  We define the groupoid $\mcR'_{g,k,h,n}(T)$ by
  \begin{align*}
    \obj\mcR'_{g,k,h,n}(T) & \coloneqq \left\{(C,u,X,\bp)\right\}, \\
    \Mor_{\mcR'_{g,k,h,n}(T)}\left( (C,u,X,\bp),(C',u',X',\bp') \right) & \coloneqq \left\{(\Phi,\phi)\right\},
  \end{align*}
  where $u:C\lra X$ is a smooth Hurwitz cover, $\bp\subset X$ are $n$ enumerated points including all branched points of $u$. In contrast to $\mcR_{g,k,h,n}(T)$ the preimages of $\bp$ are \emph{not} enumerated. Morphisms $(\Phi,\phi)$ are commuting diagrams
  \begin{equation*}
    \xymatrix{
      C \ar[r]^{\Phi} \ar[d]_{u} & C' \ar[d]_{u'} \\
      X \ar[r]^{\phi} & X'
      }
  \end{equation*}
  such that $\phi(p_i)=p_i'$ for all $i=1,\ldots n$. Furthermore we define the \emph{intermediate Hurwitz number} $\mcH'_{g,k,h,n}(T)$ as
  \begin{equation*}
    \mcH'_{g,k,h,n}(T)\coloneqq |\Aut(Y,\br)|\sum_{\substack{[C,u,X,\bp] \\ \text{s.t. } \ev[C,u,X,\bp]=[Y,\br]}}\frac{1}{|\Aut_{\mcR'_{g,k,h,n}(T)}(C,u,X,\bp)|},
  \end{equation*}
  where the evaluation map is defined in the obvious way.
\end{definition}

We will prove two propositions.
\begin{prop}
  Given some combinatorial data we have
  \begin{equation*}
    \mcH'_{g,k,h,n}(T)=\mcH_h(T_1,\ldots,T_n,Y,\br).
  \end{equation*}
  \label{prop:rel-hurwitz-1}
\end{prop}

\begin{prop}
  Given some combinatorial data we have
  \begin{equation*}
    H_{g,k,h,n}(T)=\mfK\cdot\mcH'_{g,k,h,n}(T).
  \end{equation*}
  \label{prop:rel-hurwitz-2}
\end{prop}

\cref{prop:rel-hurwitz-1} and \cref{prop:rel-hurwitz-2} together prove \cref{thm:rel-std-hurwitz-numbers}.

Both propositions will be proven by choosing representatives for the equivalence classes and rewriting the equivalence relations as group actions.

\begin{proof}{(\cref{prop:rel-hurwitz-1})}
  Let $\{(C_s,u_s)\}_{s=1}^m\subset\obj\mcR_h(T_1,\ldots,T_n,Y,\br)$ be representatives for the $m$ equivalence class of $|\mcR_h(T_1,,\ldots T_n,Y,\br)|$. Denote by $\mcB$ the full subcategory generated by these $m$ objects. It has only automorphisms which are precisely the automorphisms from $\mcR_h(T_1,\ldots,T_n,Y,\br)$. Thus
  \begin{equation*}
    \mcH_h(T_1\ldots,T_n)=\sum_{s=1}^m\frac{1}{|\Aut_{\mcB}(C_s,u_s)|}.
  \end{equation*}
  There exists an action by the group $\mcG\coloneqq\Aut(Y,\br)$ on $\obj\mcB$ defined as follows. Let $\varphi: (Y,\br)\lra (Y,\br)$ be a biholomorphism and $(C_s,u_s)\in\obj\mcB$. Then $\varphi\circ u_s:C_s\lra Y$ is again a Hurwitz cover of the fixed type and therefore there exists $t\in\{1,\ldots,m\}$ such that $(C_s,\varphi\circ u_s)\sim_{\mcR_h(T_1,\ldots,T_n)}(C_t,u_t)$. We define $\varphi\cdot(C_s,u_s)\coloneqq (C_t,u_t)$. This is obviously well defined and is a group action by $\mcG$ on $\obj\mcB$.

Next we show that every class $[C,u,X,\bp]\in|\mcR'_{g,k,h,n}(T)$ such that $\ev[C,u,X,\bp]=[Y,\br]$ has a representative $[C_s,u_s,Y,\br]$. First we see that $(C,u,X,\bp)\sim_{\mcR'_{g,k,h,n}(T)}(C,\varphi\circ u,Y,\br)$ because there exists a biholomorphism $\varphi:(X,\bp)\lra(Y,\br)$ and thus the diagram
\begin{equation*}
  \xymatrix{
    C \ar[r]^{\id} \ar[d]_u & C \ar[d]^{\varphi\circ u} \\
    X \ar[r]^{\varphi} & Y
}
\end{equation*}
commutes. Since this implies that $(C,\varphi\circ u)$ is a Hurwitz cover with target $(Y,\br)$ it is equivalent in $\mcR_h(T_1,\ldots,T_n)$ to $(C_s,u_s)\in\obj \mcB$ for some $s\in\{1,\ldots,m\}$. Thus there exists a biholomorphism $\Phi:C\lra C_s$ such that
\begin{equation*}
  \xymatrix{
    C \ar[r]^{\Phi} \ar[d]_{\varphi\circ u} & C \ar[d]^{u_s} \\
    Y \ar[r]^{\id} & Y
}
\end{equation*}
commutes. Composing the two diagrams shows $(C,u,X,\bp)\sim_{\mcR'_{g,k,h,n}(T)}(C_s,u_s,Y,\br)$. This implies that the set $\{(C_s,u_s,Y,\br)\}_{s=1}^m$ contains representatives of all equivalence classes in $|\mcR'_{g,k,h,n}(T)|$. However, some of them might still be identified in that groupoid as we will see next.

We need to show three more statements to conclude the result:
\begin{enumerate}[label=(\roman*), ref=(\roman*)]
  \item We have $(C_s,u_s,Y,\br)\sim_{\mcR'_{g,k,h,n}(T)}(C_t,u_t,Y,\br)\Longleftrightarrow (C_s,u_s)\sim_{\mcG}(C_t,u_t)$ for all $s,t\in\{1,\ldots,m\}$. \label{prop1:item1}
  \item There exist surjective group homomorphisms $$F_s:\Aut_{\mcR'_{g,k,h,n}(T)}(C_s,u_s,Y,\br)\lra\Aut_{\mcG}(C_s,u_s)$$ for all $s=1,\ldots,m$. \label{prop1:item2}
  \item The kernel $\ker F_s$ is isomorphic to $\Aut_{\mcB}(C_s,u_s)$ for all $s=1,\ldots,m$. \label{prop1:item3}
\end{enumerate}
Before proving these statements let us show that this does indeed imply the statement. Combining \cref{prop1:item2} and \cref{prop1:item3} we get 
\begin{equation*}
  |\Aut_{\mcR'_{g,k,h,n}(T)}(C_s,u_s,Y,\br)=|\Aut_{\mcB}(C_s,u_s)|\cdot |\Aut_{\mcG}(C_s,u_s)|\qquad\forall\,s=1,\ldots,m.
\end{equation*}
All in all we have 
\begin{align*}
  \mcH'_{g,k,h,n}(T) & = |\mcG|\sum_{\substack{[C,u,X,\bp]\in|\mcR'_{g,k,h,n}(T)| \\ \text{s.t. } \ev[C,u,X,\bp]=[Y,\br]}}\frac{1}{|\Aut_{\mcR'_{g,k,h,n}(T)}(C,u,X,\bp)|} \\
  & = \sum_{\substack{[C_s,u_s,Y,\br]\in|\mcR'_{g,k,h,n}(T)| \\ \text{for }s=1,\ldots,m}}\frac{|\mcG|}{|\Aut_{\mcB}(C_s,u_s)|\cdot |\Aut_{\mcG}(C_s,u_s)|} \\
  & = \sum_{[C_s,u_s]_{\mcG}} \#[C_s,u_s]_{\mcG}\frac{1}{|\Aut_{\mcB}(C_s,u_s)|} \\
  & = \sum_{(C_s,u_s)\in\obj\mcB}\frac{1}{|\Aut_{\mcB}(C_s,u_s)|} \\
  & = \mcH_g(T_1,\ldots,T_n)
\end{align*}
where the intermediate sums go over all equivalence classes only, the subscript $\mcG$ denotes equivalence classes with respect to the $\mcG$-action and $\#[\cdot]$ denotes the number of elements in the equivalence class.

Now we prove \cref{prop1:item1}, \cref{prop1:item2} and \cref{prop1:item3}.
\begin{enumerate}[label=(\roman*), ref=(\roman*)]
  \item Suppose $(C_s,u_s,Y,\br)\sim_{\mcR'_{g,k,h,n}}(C_t,u_t,Y,\br)$, i.e.\ there exists a morphism $(\Phi,\varphi)$ such that
    \begin{equation}
      \label{eq:equivalence-notions-group-action}
      \xymatrix{
        C_s \ar[r]^{\Phi} \ar[d]_{u_s} & C_t \ar[d]^{u_t} \\
        Y \ar[r]^{\varphi} & Y
}
    \end{equation}
    commutes and $\varphi(r_i)=r_i$ for $i=1,\ldots,n$. You can read this square as a triangle implying $(C_s,\varphi\circ u_s)\sim_{\mcR_h(T_1,\ldots,T_n)}(C_t,u_t)$ and thus $(C_t,u_t)=\varphi\cdot(C_s,u_s)$ for $\varphi\in\mcG$ which was to show. For the other direction notice that $(C_t,u_t)=\varphi\cdot (C_s,u_s)$ implies $(C_s,\varphi\circ u_s)\sim_{\mcR_h(T_1,\ldots,T_n)}(C_t,u_t)$ which gives a commuting triangle that expands to the commuting square \cref{eq:equivalence-notions-group-action}, which proves this direction as well.
  \item Let $(\Phi,\varphi)\in\Aut_{\mcR'_{g,k,h,n}(T)}(C_s,u_s,Y,\br)$. Then we have
    \begin{equation*}
      \xymatrix{
        C_s \ar[r]^{\Phi} \ar[d]_{u_s} & C_s \ar[d]^{u_s} \\
        Y \ar[r]^{\varphi} & Y
}
    \end{equation*}
    implying $\varphi\cdot(C_s,u_s)\sim_{\mcR_h(T_1,\ldots,T_n,Y,\br)}(C_s,u_s)$ and therefore $\varphi\in\Aut_{\mcG}(C_s,u_s)$. This map $(\Phi,\varphi)\mapsto \varphi$ is clearly a group homomorphism and its surjectivity comes from the fact that $\varphi\cdot(C_s,u_s)=(C_s,u_s)$ implies the existence of $\Phi:C_s\lra C_s$ such that
    \begin{equation*}
      \xymatrix{
        C_s \ar[r]^{\Phi} \ar[d]_{\varphi\circ u_s} & C_s \ar[dl]^{u_s} \\
        Y &
}
    \end{equation*}
    commutes which can be expanded to the above square. At this point it is important that in $\mcR'_{g,k,h,n}(T)$ we do not consider the critical points enumerated.
  \item This is obvious as $(\Phi,\id)\in\Aut_{\mcR'_{g,k,h,n}(T)}(C_s,u_s,Y,\br)$ satisfies
    \begin{equation*}
      \xymatrix{
        C_s \ar[r]^{\Phi} \ar[d]_{u_s} & C_s \ar[d]^{u_s} \\
        Y \ar[r]^{\id} & Y 
}
    \end{equation*}
    and thus $\Phi\in\Aut_{\mcB}(C_s,u_s)$, and similarly for the other direction.
\end{enumerate}
\end{proof}

\begin{proof}{\cref{prop:rel-hurwitz-2}}
  Similarly to the proof of \cref{prop:rel-hurwitz-1} we begin by choosing exactly one representative $\{(C_s,u_s,Y,\br)\}_{s=1}^m$ for every equivalence class in $|\mcR'_{g,k,h,n}(T)|$ and denote by $\mcB$ the full subcategory of $\mcR'_{g,k,h,n}(T)$ generated by these elements which again contains only automorphisms. Therefore we have
  \begin{equation*}
    \mcH'_{g,k,h,n}(T)=|\Aut(Y,\br)|\sum_{s=1}^m\frac{1}{|\Aut_{\mcB}(C_s,u_s,Y,\br)|}.
  \end{equation*}
  Now we define another category $\mcA$ built out of $\mcB$ including enumerations. We define
  \begin{align*}
    \obj\mcA & \coloneqq \begin{aligned} \left\{(C_s,u_s,Y,\bq,\br)\mid s=1,\ldots,m, \right. \\ \left. \bq\text{ any admissible enumeration}\right\} \end{aligned} \\
    \Mor_{\mcA}((C_s,u_s,Y,\bq,\br),(C_t,u_t,Y,\bq',\br)) & \coloneqq \\
    \Mor_{\mcR_{g,k,h,n}(T)} & ((C_s,u_s,Y,\bq,\br),(C_t,u_t,Y,\bq',\br)),
  \end{align*}
  i.e.\ the full subcategory generated by the $m$ Hurwitz covers including all admissible enumerations. Here, ``admissible'' means that the Hurwitz cover with the enumeration satisfies $T$, so in particular $\deg_{q_j}u=l_j$ and $\nu(j)=i$. A little bit of combinatorics tells us that given a Hurwitz cover there are $\mfK$ possible enumerations as you can permute precisely those indices which have the same degree inside their fibre. Thus $\mcA$ has $\mfK\cdot m$ objects. Furthermore we define the map
  \begin{align*}
    F:\obj\mcA & \lra \obj\mcB \\
    (C_s,u_s,Y,\bq,\br) & \longmapsto (C_s,u_s,Y,\br)
  \end{align*}
forgetting the enumeration. We now show a few statements about $\mcA, \mcB$ and $F$.
\begin{enumerate}[label=(\roman*), ref=(\roman*)]
  \item Every class $[C,u,X,\bq,\bp]\in|\mcR_{g,k,h,n}(T)|$ with $\ev[C,u,X,\bq,\bp]=[Y,\br]$ has a representative in $\obj\mcA$ and $$|\mcA|=\{[C,u,X,\bq,\bp]\in|\mcR_{g,k,h,n}(T)|\text{ s.t. }\ev([C,u,X,\bq,\bp])=[Y,\br]\}$$ as well as $$\Aut_{\mcR_{g,k,h,n(T)}}(C_s,u_s,Y,\bq,\br)=\Aut_{\mcA}(C_s,u_s,Y,\bq,\br)$$ for every $s=1,\ldots m$. \label{prop2:item1}
  \item If $(C_s,u_s,Y,\bq,\br)\sim_{\mcR_{g,k,h,n}(T)}(C_t,u_t,Y,\bq',\br)$ then $(C_s,u_s,Y,\br)=(C_t,u_t,Y,\br)$, i.e.\ only elements in the fibres of $F$ are identified in $\mcR_{g,k,h,n}(T)$ or $\mcA$, respectively. \label{prop2:item2}
  \item For every $s=1,\ldots,m$ there exists a group action by
    \begin{equation*}
      \mcG_s\coloneqq \Aut_{\mcR'_{g,k,h,n}(T)}(C_s,u_s,Y,\br)=\Aut_{\mcB}(C_s,u_s,Y,\br)
    \end{equation*}
    on the fibre $F^{-1}(C_s,u_s,Y,\br)$. \label{prop2:item3}
  \item For every $s=1,\ldots,m$ we have 
    \begin{equation*}
      (C_s,u_s,Y,\bq,\br)\sim_{\mcA}(C_s,u_s,Y,\bq',\br)\Longleftrightarrow (C_s,u_s,Y,\bq,\br)\sim_{\mcG_s}(C_s,u_s,Y,\bq',\br).
    \end{equation*}
    \label{prop2:item4}
  \item For every $s=1,\ldots,m$ we have 
    \begin{equation*}
      \Aut_{\mcR_{g,k,h,n}(T)}(C_s,u_s,Y,\bq,\br)\simeq \Aut_{\mcG_s}(C_s,u_s,Y,\bq,\br).
    \end{equation*}
    \label{prop2:item5}
\end{enumerate}

  This situation is illustrated in \cref{fig:subcategories-hurwitz-covers}.

  \begin{figure}[!ht]
    \centering
    \def\svgwidth{0.4\textwidth}
    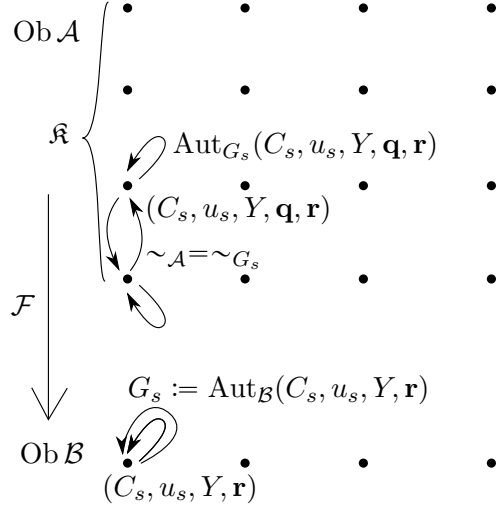
    \caption{In every fibre of the map $\mcF$ the group $\Aut_{\mcB}(C_s,u_s,Y,\br)$ acts by permuting the enumerations. This way we can rewrite the groupoid structure of $\mcR_{g,k,h,n}(T)$ via the one of $\mcA$ and then via these group actions and use the class formula. Notice that the group $G_s$ does indeed depend on $s$.}
    \label{fig:subcategories-hurwitz-covers}
  \end{figure}

Before proving these five statements let us show that this does indeed imply the proposition. We have
\begin{align*}
  H_{g,k,h,n}(T) & \coloneqq \left|\Aut(Y,\br)\right|\sum_{\substack{[C,u,X,\bq,\bp]\in|\mcR_{g,k,h,n}(T)| \\ \text{s.t. }\ev([C,u,X,\bq,\bp])=[Y,\br]}}\frac{1}{\left|\Aut(C,u,X,\bq,\bp)\right|} \\
  & = |\Aut(Y,\br)|\sum_{[C_s,u_s,Y,\bq,\br]\in|\mcA|}\frac{1}{|\Aut_{\mcG_s}(C_s,u_s,Y,\bq,\br)|} \\
  & = |\Aut(Y,\br)|\sum_{s=1}^m\frac{1}{|\mcG_s|}\sum_{\substack{[C_s,u_s,Y,\bq,\br]\in|\mcA| \\ \text{s.t. } F(C_s,u_s,Y,\bq,\br) = (C_s,u_s,Y,\br) }}\frac{|\mcG_s|}{|\Aut_{\mcG_s}(C_s,u_s,Y,\bq,\br)|} \\
  & = |\Aut(Y,\br)|\sum_{s=1}^m\frac{1}{|\mcG_s|}\cdot\mfK \\
  & = \mfK\cdot \mcH'_{g,k,h,n}(T),
\end{align*}
where we have used the class equation for the group actions $\mcG_s$ for $s=1,\ldots,m$.

It remains to prove the above five statements.
\begin{enumerate}[label=(\roman*), ref=(\roman*)]
  \item This is essentially by definition of $\mcA$. In the same way as in the proof of \cref{prop1:item1} for \cref{prop:rel-hurwitz-1} we can show that there exists a morphism $(\Phi,\varphi)$ via which $(C,u,X,\bq,\bp)$ is equivalent to some $(C_s,u_s,Y,\bq',\br)$ where $\bq'=\Phi(\bq)$ is some admissible enumeration of the fibres over $\bp$. Since $\obj\mcA$ contains the combinations of $(C_s,u_s,Y,\br)$ with any admissible enumeration we have $(C_s,u_s,Y,\bq',\br)\in\obj\mcA$. Since $\mcA$ is a full subcategory of $\mcR_{g,k,h,n}(T)$ we have by definition
    \begin{align*}
      (C_s,u_s,Y,\bq,\br)\sim_{\mcR_{g,k,h,n}(T)}(C_t,u_t,Y,\bq',\br) &  \\
      \Longleftrightarrow & (C_s,u_s,Y,\bq,\br) \sim_{\mcA}(C_t,u_t,Y,\bq',\br) \\
      \Aut_{\mcR_{g,k,h,n(T)}}(C_s,u_s,Y,\bq,\br) & =\Aut_{\mcA}(C_s,u_s,Y,\bq,\br).
    \end{align*}
  \item This comes from the fact that morphisms in $\mcR_{g,k,h,n}(T)$ or $\mcA$ are also morphisms in $\mcR'_{g,k,h,n}(T)$ as there are less conditions on the latter ones. Thus equivalent objects in $\mcA$ are mapped to equivalent ones in $\mcB$ but by definition $\obj\mcB$ contains only one representative per class and thus they get mapped to the same object.
  \item We define the group action of $(\Phi,\varphi)\in\mcG_s$ in the obvious way. The element $(C_s,u_s,Y,\bq,\br)$ gets mapped to $(C_s,u_s,Y,\Phi(\bq),\br)$. This is indeed a group action preserving the fibre of $\mcF$.
  \item This follows immediately from the definitions of the morphisms in $\mcA$ and the group action of $\mcG_s$. Notice in particular that by \cref{prop2:item2}.\ in $\mcA$ only objects in the fibre of $F$ have morphisms in between them.
  \item Again this follows from the definitions of $\mcA$ and the group action of $\mcG_s$.
\end{enumerate}
\end{proof}

\section{Riemann--Hurwitz Formula}

\index{Riemann--Hurwitz Formula|(}

The combinatorial data cannot be chosen freely, in particular it needs to satisfy the well-known Riemann--Hurwitz formula. Unfortunately this relation is usually stated for smooth Hurwitz covers and we need the appropriate version for nodal Hurwitz covers e.g.\ in the compactness proof. Also we will need a version for branched covers of nodal surfaces with boundary that have branched points only in the nodes. Before stating our version let us discuss various objects describing properties of nodal surfaces.

\begin{definition}
  Let $C$ be a Riemann surface. Then according to \cite{robbin_construction_2006} and \cite{hummel_gromovs_1997} we can define
  \begin{itemize}
    \item its \emph{normalization} $\wt{C}$ as the Riemann surface obtained by removing all nodes and gluing in two separate points, i.e.\ the unique (up to biholomorphism) closed Riemann surface birationally equivalent\footnote{Here, birationally equivalent means that there exists a holomorphic map $\sigma:\wt{C}\lra C$ such that $\sigma$ is a biholomorphism on $C\setminus\{\text{nodes}\}$.} to $C$ such that the preimage of the nodes consist of finitely many points,
    \item its \emph{signature graph} $\Gamma$ which has one vertex for each irreducible component with its genus as a label and one edge for every node between the corresponding vertices and
    \item its \emph{arithmetic genus} $\ga(C)\in\NN$ which is defined as $\ga(C)\coloneqq b_1(\Gamma)+\sum_{\alpha}g_{\alpha}$ where $b_1\coloneqq \rank H_1(\Gamma,\ZZ)$ is the rank of the first homology of the signature graph $\Gamma$ of $C$ and $g_{\alpha}$ is the genus of the $\alpha$-th connected component of the normalization $\wt{C}$. 
  \end{itemize}
\end{definition}

\index{Euler Characteristic}
\index{Arithmetic Genus}
\index{Normalization}

Notice that we furthermore can consider the topological Euler characteristics of $C$ as well as its normalization $\wt{C}$ which can be calculated in an easier way as it is smooth. They are related as follows.

\begin{lem}
  Let $C$ be a closed nodal Riemann surface with $\delta_C\in\NN$ nodes and $\wt{C}$ its normalization. Then we have for the topological Euler characteristics $\chi$
  \begin{equation*}
    \chi(\wt{C})=\chi(C)+\delta_C.
  \end{equation*}
\end{lem}

\begin{proof}
  We show that gluing together the surface $\wt{C}$ at two points corresponding to a node on $C$ reduces the topological Euler characteristics by one. By induction the result then follows.

  Of course, this follows directly from additivity of the Euler characteristic but we give a short proof anyway. So denote by $\sigma:\wt{C}\lra C$ the normalization and by $\{x,y\}\lra z$ the node and its preimages. Choose two open discs around $x$ and $y$ in $\wt{C}$ as well as the complement of two slightly smaller closed discs around $x$ and $y$. Denote the union of the two discs by $B\subset \wt{C}$ and define $B'\coloneqq\sigma(B)\subset C$. The complement of the smaller disc is denoted by $A\subset \wt{C}$ and its corresponding set under $\sigma$ is $A'\coloneqq\sigma(A)\subset C$. We then have $A\cup B=\wt{C}$ and $A'\cup B'=C$ as well as $A\simeq A'$ because $\sigma$ is a homeomorphism outside the node $z$. Furthermore it is easy to see that $B$ is homotopy equivalent to two points, $B'$ is homotopy equivalent to one point and $A\cap B$ as well as $A'\cap B'$ are homotopy equivalent to two circles, respectively. This is illustrated in \cref{fig:mayer-vietoris}.

  \begin{figure}[!ht]
    \centering
    \def\svgwidth{0.8\textwidth}
    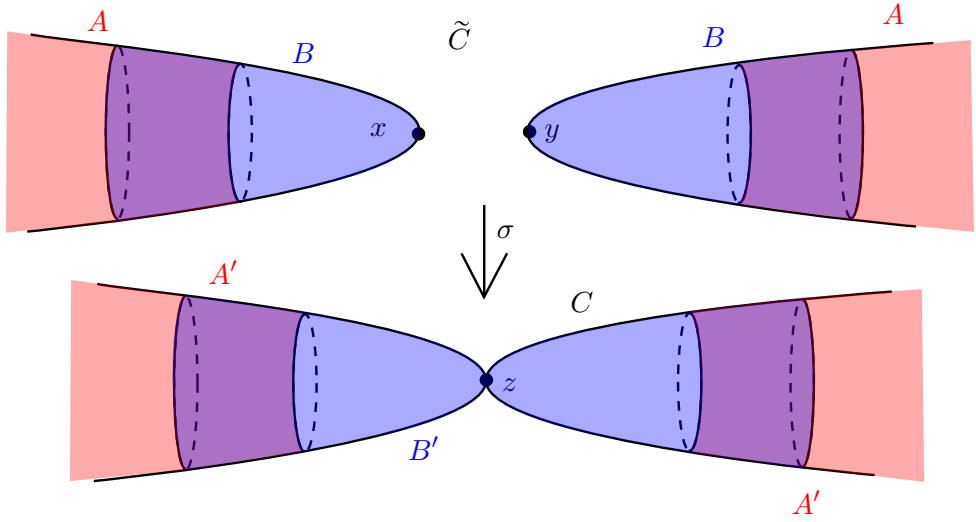
    \caption{This figure illustrates the open sets by which the nodal surface as well as its normalization at one node is covered. The sets $A$ and $A'$ are in red, $B$ and $B'$ are in blue and $A\cap B$ as well as $A'\cap B'$ are in violet.}
    \label{fig:mayer-vietoris}
  \end{figure}

  Since $A$ and $B$ as well as $A'$ and $B'$ cover $\wt{C}$ and $C$, respectively, we obtain long exact Mayer--Vietoris sequences in cohomology, that is
  \begin{align*}
    \ldots \lra H^n(\wt{C},\ZZ) \lra & H^n(A,\ZZ)\oplus H^n(B,\ZZ) \lra H^n(A\cap B,\ZZ) \lra \ldots, \\
    \ldots \lra H^n(C,\ZZ) \lra & H^n(A',\ZZ)\oplus H^n(B',\ZZ) \lra H^n(A'\cap B',\ZZ) \lra \ldots
  \end{align*}
  which implies
  \begin{align*}
    \chi(A)+\chi(B) & =\chi(\wt{C})+\chi(A\cap B) \\
    \chi(A')+\chi(B') & =\chi(C)+\chi(A'\cap B')
  \end{align*}
  and thus
  \begin{equation*}
    \chi(\wt{C})=\chi(A)+2-0=\chi(A)+\chi(B')+1=\chi(C)+1.
  \end{equation*}
\end{proof}

Next we relate the topological Euler characteristic of a connected closed nodal Riemann surface to the number of nodes and the arithmetic genus. Note that the arithmetic genus has the interpretation as the actual genus of the connected closed curve you obtain when replacing discs around the nodes by cylinders.

\begin{lem}
  If $C$ is a connected closed nodal Riemann surface we have
  \begin{equation*}
    \chi(\wt{C})=2-2\ga(C)+2\delta_C.
  \end{equation*}
  \label{lem:euler-nodal-arithmetic}
\end{lem}

\begin{proof}
  Denote by $v\in\NN$ the number of vertices of the signature graph $\Gamma$ of $C$. The number of edges is given by $\delta_C$, i.e.\ the number of nodes on $C$. Then we can calculate the Euler characteristic of $\Gamma$ as $b_0-b_1=v-\delta_C$. Note that since $\Gamma$ is connected we have $b_0=1$ and $v$ is the number of connected components of $\wt{C}$. Thus $v+b_1=1+\delta_C$.

As $\wt{C}$ is smooth we can calculate its Euler characteristics by adding $2-2g_{\alpha}$ for all connected components $\alpha$ and obtain
\begin{align*}
  \chi(\wt{C}) & =\sum_{\alpha}(2-2g_{\alpha})=2v-2\sum_{\alpha}g_{\alpha} \\
  & =2v-2(\ga(C)-b_1)=2(1+\delta_C)-2\ga(C)=2-2\ga(C)+2\delta_C.
\end{align*}
\end{proof}

The first Riemann--Hurwitz formula we prove will be for smooth but disconnected Riemann surfaces. This is of course a corollary to the usual Riemann--Hurwitz formula for connected smooth branched covers.
\begin{lem}
  Let $u:C\lra X$ be a smooth Hurwitz cover of type $T$ with possibly disconnected surfaces $C$ and $X$. In this case we have
  \begin{equation*}
    \chi(C)=d\chi(X)-\sum_{j=1}^k(l_j-1)=d\chi(X)-nd+k.
  \end{equation*}
  \label{lem:disconnected-riemann-hurwitz}
\end{lem}

\begin{proof}
  Decompose the target $X$ into its connected components $X=\bigsqcup_{\alpha}X_{\alpha}$. The set of connected components of $C$ can be decomposed into disjoint sets $I_{\alpha}$ of those components mapped to $X_{\alpha}$. If we restrict $u$ further to one component $C_{\beta}$ with $\beta\in C_{\alpha}$ we obtain a connected branched Hurwitz cover of some degree $d_{\beta}$. Thus we can calculate using the Riemann--Hurwitz formula for connected smooth surfaces from e.g.\ \cite{forster_lectures_1981}
  \begin{align*}
    \chi(C) & = \sum_{\alpha}\sum_{\beta\in I_{\alpha}} \chi(C_{\beta}) \\
    & = \sum_{\alpha}\sum_{\beta\in I_{\alpha}} (d_{\beta}\chi(X_{\alpha})-\sum_{p\in C_{\beta}}(\deg_pu-1)) \\
    & = \sum_{\alpha}\left(d\chi(X_{\alpha})-\sum_{p\in u^{-1}(X_{\alpha})}(\deg_pu-1)\right) \\
    & = d\chi(X)-\sum_{j=1}^k(l_j-1).
  \end{align*}
\end{proof}

\begin{rmk}
  Note that a very similar statement also holds true for \emph{bordered Hurwitz covers} which we will encounter in \cref{sec:orb-structure-mod-space-bordered-hurwitz-covers-definitions}. In this case we require that $C$ and $X$ have $k$ or $n$ boundaries, respectively, which are mapped onto each other according to the map $\nu$ and the restriction $u:\del_jC\lra\del_{\nu(j)}X$ is of degree $l_j$ with no branched points on $C$. Then the map is an actual covering and we have $\chi(C)=d\chi(X)$ where $\chi$ now also counts the boundary components. Writing this out we obtain $2-2g(C)=d(2-2g(X))-nd+k$.
\end{rmk}

\begin{prop}
  A nodal Hurwitz cover $(C,u,X,\bp,\bq)\in\obj\mcR_{g,k,h,n}(T)$ satisfies the \emph{Riemann--Hurwitz formula}
  \begin{equation}
    \label{eq:riemann-hurwitz}
    \chi(C)=d\chi(X)-\sum_{j=1}^k(l_j-1)-\sum_{p\in\text{nodes}(C)}(\deg_pu-1)=d\chi(X)-nd+k-d\delta_X+\delta_C,
  \end{equation}
  where the Euler characteristic $\chi$ for a nodal compact Riemann surface refers to its \emph{topological} Euler characteristic. Note that for smooth Hurwitz covers this formula is the usual Riemann--Hurwitz formula.\footnote{Also note that every node in $C$ appears only once in the sum.}
  \label{prop:riemann-hurwitz}
\end{prop}

\begin{rmk}
  Recall that our Hurwitz covers are by definition connected. However, normalizations of the surfaces might not be connected.
\end{rmk}

\begin{proof}
  By definition of the normalization there exist birational maps $\sigma_C:\wt{C}\lra C$ and $\sigma_X:\wt{X}\lra X$ which are biholomorphic outside the nodes. Thus we can pull back $u$ to a branched cover $\wt{u}:\wt{C}\lra\wt{X}$ having the critical points and degrees as $u$ including two new critical points with equal degree for every node. By \cref{lem:disconnected-riemann-hurwitz} we have
  \begin{equation*}
    \chi(\wt{C})=d\chi(\wt{X})-\sum_{p\in\wt{C}}(\deg_p\wt{u}-1)
  \end{equation*}
  and thus
  \begin{align*}
    \chi(C) & = \chi(\wt{C})-\delta_C \\
            & = d\chi(\wt{X}) - \sum_{p\in\wt{C}}(\deg_p\wt{u}-1)-\delta_C \\
            & = d\chi(X) - \sum_{p\in\wt{C}}(\deg_p\wt{u}-1) - \delta_C + d\delta_X \\
            & = d \chi(X) - \sum_{j=1}^k(l_j-1) - 2 \sum_{p\in\text{nodes}(C)}(\deg_pu-1) + \overbrace{\sum_{p\in\text{nodes}(C)}(\deg_pu-1)}^{=d\delta_X-\delta_C} \\
            & = d \chi(X) - \sum_{j=1}^k(l_j-1) - \sum_{p\in\text{nodes}(C)}(\deg_pu-1).
  \end{align*}
\end{proof}

From this we can deduce the version that we need later in the compactness proof.

\begin{cor}
  Let $(C,u,X,\bp,\bq)\in\obj\mcR_{g,k,h,n}(T)$ be a nodal Hurwitz cover. Then we have
  \begin{equation*}
    2-2\ga(C)=d(2-2\ga(X))-\sum_{j=1}^k(l_j-1)=d(2-2\ga(X))-nd+k.
  \end{equation*}
  \label{cor:nodal-riemann-hurwitz}
\end{cor}

\begin{proof}
  Using \cref{prop:riemann-hurwitz} we get
  \begin{equation*}
    \chi(C)=d\chi(X)-nd+k-d\delta_X+\delta_C
  \end{equation*}
  where we can apply \cref{lem:euler-nodal-arithmetic} to obtain
  \begin{equation*}
    2-2\ga(C)+\delta_C=d(2-2\ga(X)+\delta_X)-nd+k-d\delta_X+\delta_C
  \end{equation*}
  from which the claim follows.
\end{proof}

\index{Riemann--Hurwitz Formula|)}

\section{Combinatorial Description of Hurwitz Numbers}

\label{sec:comb-desciption-hurwitz-numbers}

So far we implicitly assumed that there only exist finitely many equivalence classes in the sums appearing for Hurwitz numbers. In this section we state the well-known result that Hurwitz numbers can be calculated in terms of the symmetric group only which will tell us that
\begin{itemize}
  \item all the sums are finite,
  \item Hurwitz numbers don't depend on the target surface $(Y,\br)$ and
  \item there exists an easy (but not very efficient) algorithm for computing these numbers.
\end{itemize}
The last point will be important for examples. In the following we will just collect the most important statements to understand the statement for higher-genus surfaces and the standard Hurwitz numbers. From those it is possible to compute our Hurwitz numbers using \cref{thm:rel-std-hurwitz-numbers}.

Given combinatorial data $(T_1,\ldots,T_n), d, n$ and a target surface $(X,\bp)$ with genus $h$ we can look at standard Hurwitz covers $u:C\lra X$. First we need the fundamental group of the target $X$ with base point $x_0\in X$ such that $x_0\not\in\bp$.

\begin{prop}
  Let $(X,\bp)$ be a closed connected smooth Riemann surface  of genus $h$ and with $n$ marked points $p_1,\ldots,p_n$ and choose some base point $x_0\in X\setminus\bp$. Then its fundamental group is given by
  \begin{equation*}
    \pi_1(X\setminus\bp,x_0)\cong \left\langle x_1,\ldots,x_h,y_1,\ldots,y_h,z_1,\ldots,z_n\middle|  \; \prod_{i=1}^h[x_i,y_i]=z_1\cdots z_n\right\rangle.
  \end{equation*}
  \label{prop:pi-1-surface}
\end{prop}

\begin{proof}
  This is well-known and can be easily seen by moving all the punctures to a disc and then using the van-Kampen theorem on the genus-$g$ surface with boundary as well as the $n$-punctured disc.
\end{proof}

Now one can use covering theory to find an algebraic description of a smooth Hurwitz cover. Since this is well-known we will only hint at the proof.

\begin{thm}
  Equivalence classes of standard Hurwitz covers of combinatorial type $(T_1,\ldots,T_n,d,h)$ are in bijection with equivalence classes of tuples 
  \begin{equation*}
    (\eta_1,\rho_1,\ldots,\eta_h,\rho_h,\sigma_1,\ldots,\sigma_{n})\in (S_d)^{2h+n}
  \end{equation*}
  such that
  \begin{itemize}
    \item $\prod_{i=1}^h[\eta_i,\rho_i]=\prod_{i=1}^{n}\sigma_i$,
    \item the cycle decomposition of $\sigma_i$ agrees with $T_i$ for all $i=1,\ldots,n$ and
    \item the subgroup of $S_d$ generated by the $\sigma_i$ acts transitively on $\{1,\ldots,d\}$,
  \end{itemize}
  where
  \begin{align*}
    (\eta_1,\rho_1,\ldots,\eta_h,\rho_h,\sigma_1,\ldots,\sigma_n) & \sim (\eta_1',\rho_1',\ldots,\eta_h',\rho_h',\sigma_1',\ldots,\sigma_n') \\
    \Longleftrightarrow & \; \exists \tau\in S_d: \sigma_i'=\tau\cdot \sigma_i\cdot\tau^{-1} \; \forall i=1,\ldots,n \\
    & \text{ and }\eta_i'=\tau\cdot\eta_i\cdot\tau^{-1}, \rho_i'=\tau\cdot\rho_i\cdot\tau^{-1} \; \forall i=1,\ldots,h.
  \end{align*}  
  Furthermore the corresponding automorphism groups are also isomorphic and thus
  \begin{equation*}
    \mcH_h(T_1,\ldots,T_n)=\sum_{\substack{[\eta_1,\rho_1,\ldots,\eta_h,\rho_h,\sigma_1,\ldots,\sigma_n] \\ \text{as above}  }}\frac{1}{|\Aut_{S_d}(\eta_1,\rho_1,\ldots,\eta_h,\rho_h,\sigma_1,\ldots,\sigma_n)|}.
  \end{equation*}
\end{thm}

\begin{proof}
  This is again well-known. The proof proceeds as follows. For one direction you label the preimages of the fibre $u^{-1}(x_0)$ and lift all elements in $\pi_1(X\setminus\bp,x_0)$ to automorphisms of the fibre and thus elements in $S_d$. The images of the generators from \cref{prop:pi-1-surface} then satisfy the three conditions above. Automorphisms of the cover together with reenumerations of the fibre over $x_0$ correspond to the common-conjugacy-class equivalence relation as stated in the theorem.

  For the other direction note that such a tuple of permutations defines a surjective group homomorphism $\pi_1(X\setminus\bp,x_0)\lra S_d$ and there exists exactly one equivalence class of coverings $C'\lra X\setminus\bp$ inducing the map as monodromy map. Using the Riemann existence theorem you can then prove that there is exactly one closed Riemann surface $C$ together with a branched cover $u:C\lra X$ such that $C'\lra X\setminus\bp$ comes from $u$ by removing finitely many points.
\end{proof}

\begin{rmk}
  The last theorem shows in particular that there are only finitely many terms appearing in the sums for the standard Hurwitz numbers as well as our version. Furthermore it allows for algorithms to compute these numbers as we will need them for examples in \cref{sec:appl-exampl}. Notice, however, that the naive algorithms by making complete lists of all admissible tuples are not very efficient as you need to compute common conjugacy classes of $(d!)^{2h+n-1}$ elements.
\end{rmk}

\chapter{Some Basics in Hyperbolic Geometry of Surfaces}

\label{sec:basics-hyperbolic-geometry}

Since we will switch between the complex and the hyperbolic description of Riemann surfaces quite often we collect various useful statements and definitions in this chapter.

\section{Uniformization}

\label{sec:uniformization}
\index{Uniformization}

Under certain topological conditions any complex Riemann surface admits a unique hyperbolic metric inducing the complex structure. The process of passing from a complex structure to this particular hyperbolic structure is called \emph{uniformization}.

\index{Riemann Surface!Admissible}
\index{Riemann Surface!Of Finite Type}

\begin{definition}
  An \emph{admissible Riemann surface} $C$ is a nodal punctured Riemann surface $C$ possibly with boundary of finite type\footnote{Here, \emph{finite type} means that $C$ has finitely many smooth or irreducible components and every such component is homeomorphic to a compact surface with finitely many boundary components of finite genus with finitely many interior points removed.} equipped with a complex structure. This means that around every point $z\in C$ there exists a coordinate chart $\phi$ defined on $z\in U\subset C$ such that
  \begin{itemize}
    \item \emph{(smooth points)} $\phi:U\lra\CC$ with $\phi(z)=0$, 
    \item \emph{(boundary points)} $\phi:U\lra\HH$ with $\phi(z)=0$ or
    \item \emph{(nodal points)} $\phi:U\lra\{(x,y)\in\CC^2\mid xy=0\}$ with $\phi(z)=0$
  \end{itemize}
  where $\phi$ is a homeomorphism onto its image and of course transition functions are biholomorphisms. Furthermore, we require that every smooth component is stable, i.e.\ it has only finitely many automorphisms\footnote{An automorphism is a biholomorphism which preserves cusps, nodes and boundary components. Usually, we will consider enumerated cusps and boundary components so we will require the biholomorphism to respect that enumeration.} or equivalently
  \begin{equation*}
    2-2g-\#\{\text{nodes, cusps, boundary components}\}<0
  \end{equation*}
  on each component.
  \label{def:admissible-riemann-surface}
\end{definition}

\index{Hyperbolic Metric}

\begin{prop}[See \cite{hummel_gromovs_1997}, \cite{buser_geometry_2010} and \cite{hubbard_teichmuller_2006}]
  Each component of an admissible Riemann surface $C$ carries a unique hyperbolic structure $g$ such that
  \begin{enumerate}[label=(\roman*), ref=(\roman*)]
    \item at smooth points the complex structure agrees with rotation by $90^{\degree}$ in the direction of the orientation induced by the complex structure,\footnote{This is the same as $g$ inducing the same \emph{conformal} structure as the complex structure.}
    \item the boundary components are geodesics,
    \item every marked point is a cusp, i.e.\ there exists a neighborhood isometric to
      \begin{equation*}
        \faktor{\{z\in\HH\mid \Im(z)\geq 1\}}{(z\sim z+T)}
      \end{equation*}
      for some $T\in\RR$,
    \item $(C,g)$ has finite area and
    \item $(C,g)$ is complete.\footnote{Of course we allow that geodesics reach boundary components (not cusps) in finite time.}
  \end{enumerate}
  \label{prop:unique-hyperbolic-metric}
\end{prop}

Of course we can also relate holomorphic maps to local isometries for the corresponding hyperbolic metrics.

\begin{lem}
  Let $u:C\lra X$ be a holomorphic map between admissible Riemann surfaces $C$ and $X$. Now equip both Riemann surfaces with their unique hyperbolic metrics $g_C$ and $g_X$ from \cref{prop:unique-hyperbolic-metric}. Then the map $u$ is a local isometry of nodal surfaces.
  \label{lem:hol-map-loc-isom}
\end{lem}

\begin{proof}
  Since $u$ is holomorphic and $C$ and $X$ are two-dimensional this implies that $u$ is conformal, i.e.\ there exists a smooth function $f\in\cin(C,\RR)$ such that $u^*g_X=f\cdot g_C$. Note that this is because the pull back of the conformal structure of $X$ under $u$ yields the same conformal structure as the complex structure on $C$ which in turn is the same conformal structure induced by $g_C$. Also, the curvature of the pulled-back metric $u^*g_X$ is constant $-1$ and thus $u^*g_X$ is hyperbolic. Furthermore, it is also complete, has finite area and the boundary components are geodesics because they are preimages of the geodesic boundaries on $X$. By uniqueness of the metric in \cref{prop:unique-hyperbolic-metric} this implies $u^*g_X=g_C$
\end{proof}

\section{Multicurves}

\label{sec:multicurves}
\index{Multicurve}

\begin{definition}
  A \emph{$k$-multicurve} $\Gamma=([\gamma_i])_{i=1}^k$ on an admissible surface $(C,\bq)$ with its hyperbolic metric is a collection of free homotopy classes of simple closed curves $\gamma_i$ on $(C\setminus\{\text{nodes, marked points}\})$ satisfying
  \begin{itemize}
    \item that all classes $[\gamma_i]$ are not contractible on $C\setminus\{\text{nodes, marked points}\}$ and
    \item that intersection numbers $[\gamma_i]\cap[\gamma_j]=0$ for all $i\neq j$, i.e. all classes have pairwise disjoint representatives.
  \end{itemize}
\end{definition}

By an \emph{essential} homotopy class of a simple closed curve we mean a class that is not freely homotopic to a point, a marked point or a node on $C$. The following lemma summarizes various well-known statements on multicurves.

\index{Essential Simple Closed Curve}

\begin{lem}[See e.g.\ \cite{buser_geometry_2010}]
  Let $(C,\bq)$ be an admissible hyperbolic surface. Then 
  \begin{enumerate}[label=(\roman*), ref=(\roman*)]
    \item in every essential free homotopy class there exists a unique geodesic representative and two such geodesics realize the minimal intersection number of the corresponding two free homotopy classes  and
    \item every $k$-multicurve with all non-essential classes having zero self-intersection can be completed to an essentially maximal multicurve, of which the number of essential components only depends on the topology of the surface.
  \end{enumerate}
  Here, \emph{essentially maximal} means that every other essential homotopy class of simple curves intersects the maximal one and we do not have any restriction on the non-essential simple curves in the multicurve. This implies that cutting the surface at the geodesic representatives of the essential classes in the multicurve separates the surface into possibly degenerate hyperbolic pairs of pants.
\end{lem}

\index{Multicurve!Maximal}

\begin{rmk}
  Note that we include the non-essential curves in the multicurve because later we want to consider reference curves close to cusps. Imagine that one of the essential geodesics collapses to a node. Then we want to replace the corresponding curve in the multicurve by some simple curve close to that new node which corresponds to two cusps in the hyperbolic picture.
\end{rmk}

\begin{figure}[!ht]
  \centering
  \def\svgwidth{\textwidth}
  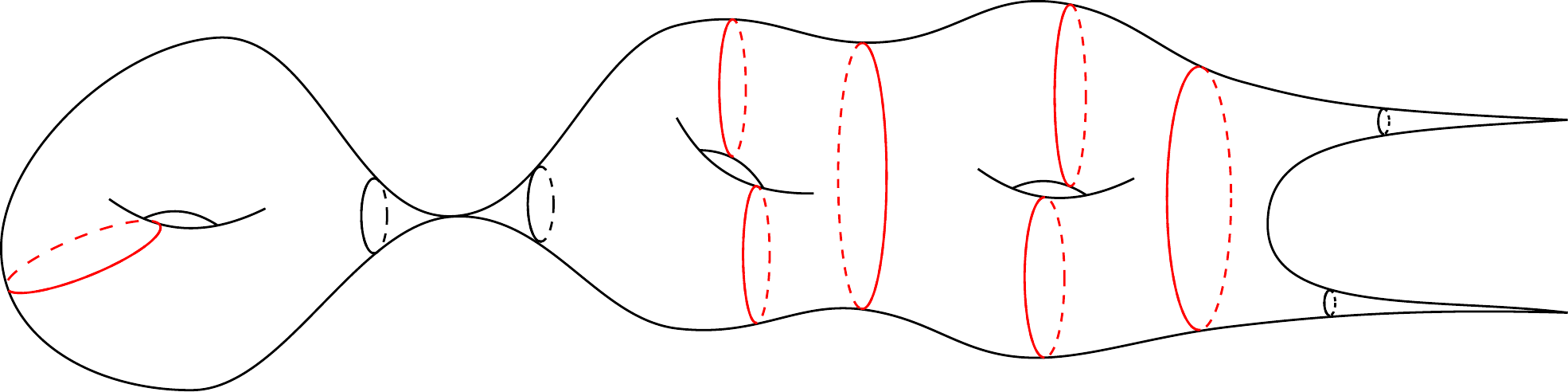
  \caption{This is an example of a maximal multicurve. However, notice that only the red curves are essential and are needed for separating the surface into pairs of pants.}
  \label{fig:example-multicurves}
\end{figure}

\section{Teichmüller Spaces}

\label{sec:teichmueller-spaces}

In this section we will define Teichmüller spaces and fix some notation that we are going to use later on. Note that all our Teichmüller spaces will be those of possibly disconnected \emph{smooth} Riemann surfaces. There is a notion of augmented Teichmüller spaces which incorporates nodes on this level but it is more difficult in particular because it is not even locally compact. As we will later construct coordinates on moduli spaces using gluing techniques we will only need coordinates on the common Teichmüller spaces.

\index{Teichmüller Space}

\begin{def-lemma}
  Let $\si$ be an oriented smooth surface without boundary with genus $g$ and fix $n\in\NN$ such that $2-2g-n<0$. Then the following spaces are in bijection and every one is defined to be the \emph{Teichmüller space} $\mcT_{g,n}$:
  \begin{enumerate}[label=(\roman*), ref=(\roman*)]
    \item The set of complex structures $J\in\cin(\si,\End(\ts \si))$ with $n$-tuples of mutually different marked points $\bq$ on $\si$ inducing the given orientation up to diffeomorphisms of $\si$ isotopic to the identiy.
    \item The set of complete hyperbolic structures on $\si\setminus\bq$ having finite area inducing the given orientation up to diffeomorphisms of $\si$ isotopic to the identity.
    \item The set of \emph{marked Riemann surfaces of type $(g,n)$}, i.e.\ pairs $(C,f)$ where $f:\si\lra C$ is a homeomorphism, $\bq\subset C$ is a set of $n$ fixed points and $C$ is a complex surface. Two pairs $(C,f)$ and $(C',f')$ are identified if the map $(f')^{-1}\circ f:\si\lra\si$ is homotopic to the identity relative to the points $\bq$.
  \end{enumerate}
\end{def-lemma}

\index{Riemann Surface!Marked}

\begin{rmk}
  Note that we did not require the surfaces to be connected. As the complex structures on different connected components are independent of each other, we have
  \begin{equation*}
    \mcT_{g,n}=\bigsqcup_{\substack{0\leq n_1,\ldots,n_l\leq n \\ n_1+\cdots + n_l=n}}\prod_{i=1}^l\mcT_{g_i,n_i},
  \end{equation*}
  where the index $i$ runs over all connected components who have genus $g_i$ and $n_i$ marked points.

  Furthermore, in the case of marked Riemann surfaces one should think of the images $f(\bq)$ as the marked points.

  Also note that there exist corresponding Teichmüller spaces for bordered Riemann surfaces as well. In this case we need to modify the definitions slightly:
  \begin{enumerate}[label=(\roman*), ref=(\roman*)]
    \item Holomorphic transition functions are required to be restrictions of biholomorphisms to the upper half plane at boundary points.
    \item Hyperbolic structures are such that the boundary components are geodesics and away from the boundary the hyperbolic structure is supposed to be complete.
    \item Markings are required to map boundary components to boundary components.
    \item When identifying any two objects in any of the cases above we require the diffeomorphism to map boundary components to boundary components and do \emph{not} require the isotopy to be the identity on the boundary components.
  \end{enumerate}
  The corresponding Teichmüller spaces can then be constructed for example by doubling the surface and restricting to the fixed point set of the involution swapping the two halves.
\end{rmk}

The following is well-known.

\begin{prop}[See e.g.\ \cite{hubbard_teichmuller_2006}]
  The Teichmüller space of connected Riemann surfaces of genus $g$ and $n$ marked points such that $2-2g-n<0$ is a contractible differentiable manifold.
\end{prop}

\begin{definition}
  There is a complex manifold together with a holomorphic map $\pi:\mcC_{g,n}\lra\mcT_{g,n}$ over Teichmüller space with $n$ holomorphic sections $s_1,\ldots,s_n$, called the \emph{universal curve}, such that for all $t\in\mcT_{g,n,}$ the fibre $\pi^{-1}(t)$ together with the points $s_1(t),\ldots,s_n(t)$ is contained in the Teichmüller equivalence class $t$.
\end{definition}

\index{Universal Curve!Over Teichmüller Space}

\begin{prop}[See \cite{hubbard_analytic_2014}]
  The universal curve $\pi:\mcC_{g,n}\lra\mcT_{g,n}$ is a globally differentiably trivializable fibre bundle. It does not possess any holomorphic sections except those of the $n$ marked points. Also there is a continuous bundle metric on $\pi$ which restricts to the hyperbolic metric induced by uniformization on the fibre.
  \label{prop:universal-teichmueller-family-vertical-hyp-metric}
\end{prop}

\begin{rmk}
  The universal curve $\pi$ is globally trivializable as a smooth fibre bundle because the base Teichmüller space is contractible. However, $\pi$ is also holomorphic and so one could ask if this trivialization is holomorphic. It turns out that it is not because there are not enough holomorphic sections. Note that the fibres are still biholomorphic to a surface in the Teichmüller class corresponding to the image of the fibre under $\pi$. So in particular the universal curve $\mcC_{g,n}$ has non-trivial topology and hence is \emph{not} given by the forgetful map $\mcT_{g,n+1}\lra\mcT_{g,n}$. Nevertheless, we can use uniformization on each of the fibres to obtain a vertical hyperbolic metric and thus a map from the square of the vertical subbundle of $\ts\mcC_{(S,Z)}$ to $\RR$ which turns out to be continuous in the base point and restricts to the hyperbolic structure on the fibre. The continuity is shown in \cite{hubbard_analytic_2014} by using the Kobayashi metric description of the hyperbolic metric.
\end{rmk}

\begin{rmk}
  Note that the universal curve over the Teichmüller space has a universal property: For any flat family of surfaces $\eta:P\lra T$ such that the fibres are closed Riemann surfaces of type $(g,n)$ together with a marking of the family\footnote{A marking of a family is not a family of markings and its definition is rather subtle. See \cite{hubbard_analytic_2014} for details. As we will not need this here, we skip a discussion.} there exists a unique map $\varphi:T\lra\mcT_{g,n}$ such that $\varphi^*C_{g,n}\cong P$ as a marked flat family of surfaces.

  Also there is a properly discontinuous group action of the mapping class group $\mcg_{g}=\pi_0(\Diff_+(\si))\simeq\faktor{\Diff_+(\si)}{\Diff_0(\si)}$ on the Teichmüller space. The quotient is the moduli space of (smooth) Riemann surfaces which carries an orbifold structure that we will use very intensively in \cref{chap:orbi-structure}.
\end{rmk}

\subsection{Fenchel--Nielsen Coordinates}

In order to define Fenchel--Nielsen coordinates on Teichmüller space we first need a set of simple closed essential decomposing curves $\Gamma\coloneqq\{\gamma_i\}_{i=1}^{3g-3+n}$ on $\si$. Here, $\si$ has genus $g$ and $n$ marked points. Now suppose we have a complex structure on $\si$. Then we can uniformize this surface to obtain a compatible hyperbolic structure. In particular, every free homotopy class of each of the $\gamma_i$ has a unique hyperbolic geodesic representative. This way we obtain $2g-2+n$ hyperbolic pairs of pants with geodesic boundaries or cusps. Denote the lengths of these $3g-3+n$ closed geodesics by $l_1,\ldots,l_{3g-3+n}$. These give half of the coordinates and it remains to define the twisting coordinates. However, we need some kind of ``reference'' for measuring the twisting which is why these coordinates are best defined on a model for Teichmüller space using marked surfaces. Since we only need these coordinates for local descriptions, e.g.\ the symplectic structure on the moduli spaces or for constructing orbifold atlases, we will not give full details but rather describe a local version. Details can be found in \cite{hubbard_teichmuller_2006}.

For the local version we still need to choose a multicurve $\Gamma'$ on $\si$ such that the $\gamma_i$ are intersected by the curves in $\Gamma'$ exactly twice and the arcs on every pair of pants each join two boundary curves in $\Gamma$ on that pair of pants. Such a multicurve can be constructed by choosing two points on every $\gamma_i\in\Gamma$ and connecting these via three pairwise disjoint simple arcs on every embedded pair of pants.

It is now possible to consider the unique geodesic representatives of two such arcs $\alpha$ and $\alpha'$ meeting at a common point on $\gamma_i$ such that they are perpendicular to the boundary curves. The resulting geodesic arcs will no longer meet and we can consider the oriented normalized length of the arc between these two points. Note that this is only well-defined ``locally'' as without a reference we can not count how often we needed to wind around the geodesic $\gamma_i$. We can repeat this construction for every geodesic $\gamma_i\in\Gamma$ to obtain angles $\tau_1,\ldots,\tau_{3g-3+n}$ which are contained in a small interval. Note that they do not depend on which pair of arcs we use as the geodesic representatives are precisely opposite of each other. The following picture summarizes the construction.

\begin{figure}[!ht]
  \centering
  \def\svgwidth{0.6\textwidth}
  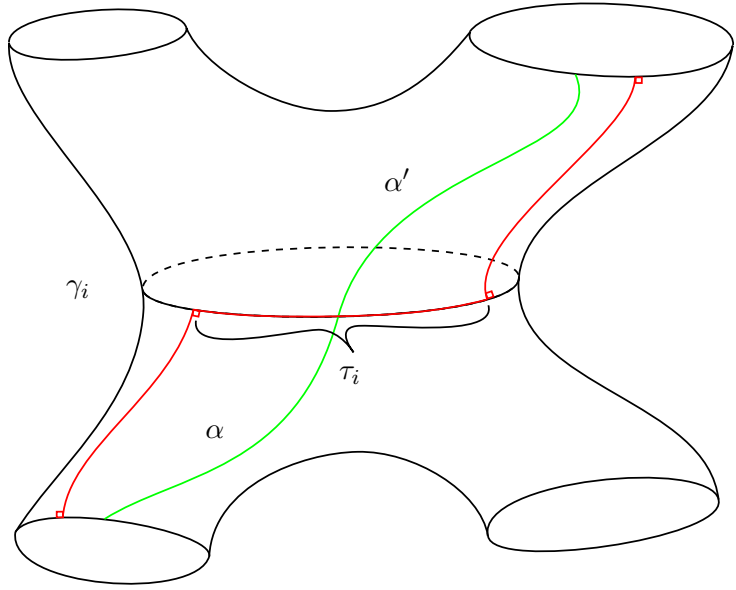
  \caption{The green curve is part of the multicurve $\Gamma'$ which is used as a reference for calculating twist coordinates. It consists of two arcs $\alpha$ and $\alpha'$ which have unique homotopic geodesic representatives such that they are perpendicular to the boundary geodesics. The red curve then consists of these two geodesic arcs as well as a connecting arc along $\gamma_i$. Notice that orientation conventions are not depicted here and we can not count complete wraps of the red arc on $\gamma_i$ around that geodesic as we have no reference curve for measuring this.}
  \label{fig:fenchel-nielsen}
\end{figure}

In total, this defines a homeomorphism onto its image on a small neighborhood $U\subset\mcT_{g,n}$ of $C$,
\begin{equation*}
  \FN_{\Gamma}:U\lra\RR_+^{3g-3+n}\times\RR^{3g-3+n}.
\end{equation*}
Again notice that it is possible to make this globally well defined and then this map will be a homeomorphism, see \cite{hubbard_teichmuller_2006}. In \cref{sec:hyperbolic-gluing} we will use a slightly different version of Fenchel--Nielsen coordinates on the parameter space of the plumbed family, i.e.\ after gluing surfaces at a node. Since the plumbing construction does not give points in Teichmüller space but rather on the quotient by the infinite cyclic subgroup generated by the Dehn twist of the collapsed curve, one obtains twist coordinates defined up to multiples of $2\pi$ but which extend to the nodal curve with length equal to zero. This is explained in \cite{hubbard_analytic_2014}. In particular we have Fenchel--Nielsen coordinates
\begin{equation*}
  \FN_{\Gamma}:U\lra\CC^{3g-3+n},
\end{equation*}
given by $l_ie^{\tau_i}$ in every component.

\subsection{Weil--Petersson Symplectic Structure}

\label{sec:weil-peterss-sympl}
\index{Weil--Petersson!Symplectic Structure}
\index{Weil--Petersson!Inner Product}

This section deals with the Weil--Petersson symplectic structure on Teichmüller space and later on the moduli space of Riemann surfaces and its compactification, the Deligne--Mumford space.

\begin{definition}
  If one identifies the cotangent space of Teichmüller space at a hyperbolic metric $h$ on $\si$ with $\ts_{[h]}^*\mcT_{g,n}\cong Q^2(\si)$, i.e.\ the space of quadratic differentials on $\si$, one defines the \emph{Weil--Petersson inner product} by
  \begin{equation*}
    \langle q_1,q_2\rangle\coloneqq\int_{\si}\frac{\ol{q_1}q_2}{\vol_h^2},
  \end{equation*}
  where the notation for the integrand means $\lambda^{-2}(z)\ol{q_1}(z)q_2(z)\dd z\wedge \dd\ol{z}$ if $\vol_h=\lambda\dd z\wedge\dd\ol{z}$ and $q_i=q_i(z)\dd z^2$. Also, quadratic differentials here are allowed to have certain types of poles at the punctures.
\end{definition}

It turns out that this $L^2$-inner product on the space of quadratic differentials is inherently connected to the hyperbolic geometry of the underlying curve.

\begin{thm}[L. Ahlfors, \cite{ahlfors_remarks_1961-1}]
  The Weil--Petersson inner product is Kähler.
\end{thm}

\begin{thm}[S. Wolpert, \cite{wolpert_symplectic_1983}]
  In Fenchel--Nielsen coordinates on Teichmüller space $\mcT_{g,n}$ the Weil--Petersson symplectic structure has standard form
  \begin{equation*}
    \wwp=\sum_{i=1}^{3g-3+n}\dd\theta_i\wedge\dd l_i.
  \end{equation*}
\end{thm}

\begin{thm}[H. Masur, \cite{masur_extension_1976}]
  The Weil--Petersson Kähler structure extends to the Deligne--Mumford space.
\end{thm}

\begin{rmk}
  Note that we can put two complex structures on the Deligne--Mumford space, one using the complex structure on Teichmüller space coming from the complex structure on the underlying surface $\si$ acting on $Q^2(\si)$ and the other coming from the Weil--Petersson Kähler structure. It turns out that these structures are not compatible but isomorphic, even on the differentiable level. This issue is well explained in \cite{hubbard_analytic_2014}, \cite{wolf_real_1992-1} and \cite{wolpert_weil-petersson_1985}. In particular, the map
  \begin{equation*}
    \Phi:\DD\lra\RR_{\geq 0}\times [0,2\pi)
  \end{equation*}
  which takes a gluing parameter $z$ to the Fenchel--Nielsen coordinates corresponding to the corresponding glued-in cylinder is not differentiable at the origin. This is illustrated in \cref{fig:complex-gluing-hyperbolic-geodesic}.
  \label{rmk:different-differentiable-structures-deligne-mumford}
\end{rmk}

\begin{figure}[!ht]
  \centering
  \def\svgwidth{0.7\textwidth}
  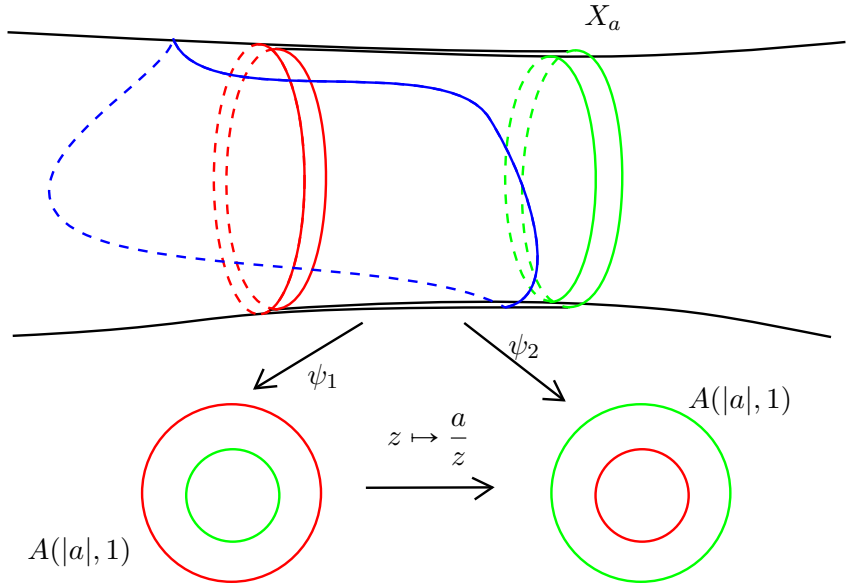
  \caption{This figure illustrates the complex gluing procedure where the colored boundaries correspond to each other. Notice that the hyperbolic structure is determined by a global PDE meaning that it is hard to determine where the corresponding hyperbolic geodesic (drawn in blue) of the new annulus is actually situated. In particular, it is not clear that the geodesic is contained inside of the glued-in annulus. Therefore its length and twist are not easily computable as a function of $a\in\CC$.}
  \label{fig:complex-gluing-hyperbolic-geodesic}
\end{figure}

\section{Collar Neighborhoods and Hyperbolic Gluing}

In this section we will discuss various suitable coordinate neighborhoods of hyperbolic geodesics and cusps. In particular, we will find a local picture of Hurwitz covers in neighborhoods of geodesics and define reference curves which we need later to define marked points close to cusps and boundaries.

\subsection{Collar and Cusp Neighborhoods}

\label{sec:collar-neighborhoods}

As we will do various constructions close to hyperbolic cusps and hyperbolic geodesic boundaries we need a few statements on collar and cusp neighborhoods. In the following, $C$ is a complete hyperbolic surface of finite area.

\index{Collar Neighborhood}

\begin{definition}
  Let $\gamma\subset C$ be a closed simple geodesic with length $l(\gamma)$. For $w\in\RR$ define $A_w(\gamma):=\{p\in C\mid \operatorname{dist}(p,\gamma)\leq w\}$. Then we call $A_w(\gamma)$ a \emph{standard collar neighborhood} of $\gamma$ if $\sinh(w)\sinh\left(\frac{l(\gamma)}{2}\right)= 1$.
\end{definition}

The following statements give various properties and descriptions of standard collar neighborhoods. Note that $A_{w'}(\gamma)$ is contained in $A_{w}(\gamma)$ for $w'<w$.

\begin{lem}
  Any standard collar neighborhood $A_w(\gamma)$ is isometric to
  \begin{enumerate}[label=(\roman*), ref=(\roman*)]
    \item the set $[-w,w]\times \,\faktor{\RR}{\ZZ}$ with the metric $\dd \rho^2+l^2(\gamma)\cosh^2\rho\;\dd t^2$ with $(\rho,t)\in[0,w]\times\, \faktor{\RR}{\ZZ}$ and
    \item the subset of the upper half plane defined as the interior of the following four curves: the circle of radius $1$, the circle of radius $e^{l(\gamma)}$ and the straight lines through the origin with angle $\phi$ to the $y$-axis where $\phi$ satisfies $\tan\phi\sinh\left(\frac{l(\gamma)}{2}\right)=1$.
  \end{enumerate}
  The latter model shows that the boundaries of the collar neighborhoods are curves of constant geodesic curvature\footnote{However, these curves are \emph{not} horocycles which have a geodesic curvature of $1$.} of length $\frac{l(\gamma)}{\tanh\left(\frac{l(\gamma)}{2}\right)}$. Also, the geodesic is the vertical line from $(0,1)$ to $(0,e^{l(\gamma)})$. See \cref{fig:collar-neighborhoods} for an illustration.
  \label{lem:collar-neighborhood}
\end{lem}

\begin{proof}
  See \cite{buser_geometry_2010} for the well-known first part. For the second statement we can choose a universal covering of the hyperbolic surface such that $\gamma$ lifts to the vertical line through the origin. This geodesic now corresponds to the element $z\mapsto e^{l(\gamma)}\cdot z$ in $\PSL(2,\ZZ)$ so we can pick one representative for the closed geodesic from $1$ to $e^{l(\gamma)}$. Now the points of constant distance from this line are given by hypercircles which are defined as these curves. It is well-known, see e.g.\ \cite{hubbard_teichmuller_2006} that they are straight lines through the same ideal point, i.e.\ the origin in our case. It remains to calculate the angle $\phi$ such that the distance $w$ satisfies $\sinh(w)\sinh\left(\frac{l(\gamma)}{2}\right)=1$. This distance $w$ is the hyperbolic length of any circular arc with the origin as the center between the imaginary axis and the line through the origin with angle $\phi$. Recall the following formula for the distance in the Poincaré upper half plane between two points on a circle of radius $r$ with angle $\phi$ about the origin
  \begin{equation*}
    \dist\left((0,r),(r\sin\phi,r\cos\phi)\right)=\arsinh(\tan\phi).
  \end{equation*}
  We thus have for $\phi$
  \begin{equation}
    \tan\phi=\sinh(w)=\frac{1}{\sinh\left(\frac{l(\gamma)}{2}\right)}.
    \label{eq:phi-l-gamma}
  \end{equation}
  Now we can calculate the length of the boundary hypercircle, i.e.\ the straight line segment from $(\sin\phi,\cos\phi)$ to $e^{l(\gamma)}(\sin\phi,\cos\phi)$ which is parametrized by
  \begin{equation*}
    \eta(t)\coloneqq (te^{l(\gamma)}+(1-t))(\sin\phi,\cos\phi)
  \end{equation*}
  for $t\in[0,1]$. We obtain
  \begin{align*}
    L(\eta) & = \int_{\eta}|\eta'(t)|_{\text{hyp}}\,\dd t =\int_0^1\frac{e^{l(\gamma)}-1}{(te^{l(\gamma)}+1-t)\cos\phi} \\
    & = \frac{e^{l(\gamma)}-1}{\cos\phi}\int_0^1\frac{1}{1+(e^{l(\gamma)}-1)t}\dd t =\frac{l(\gamma)}{\cos\phi} \\
    & = l(\gamma)\sqrt{1+\tan^2(\phi)} = \frac{l(\gamma)}{\tanh\left(\frac{l(\gamma)}{2}\right)},
  \end{align*}
  where we have used \cref{eq:phi-l-gamma} in the last step. Notice that hypercircles are automatically curves of constant geodesic curvature as can be seen in e.g.\ \cite{hubbard_teichmuller_2006}.
\end{proof}

\begin{figure}[!ht]
  \centering
  \def\svgwidth{\textwidth}
  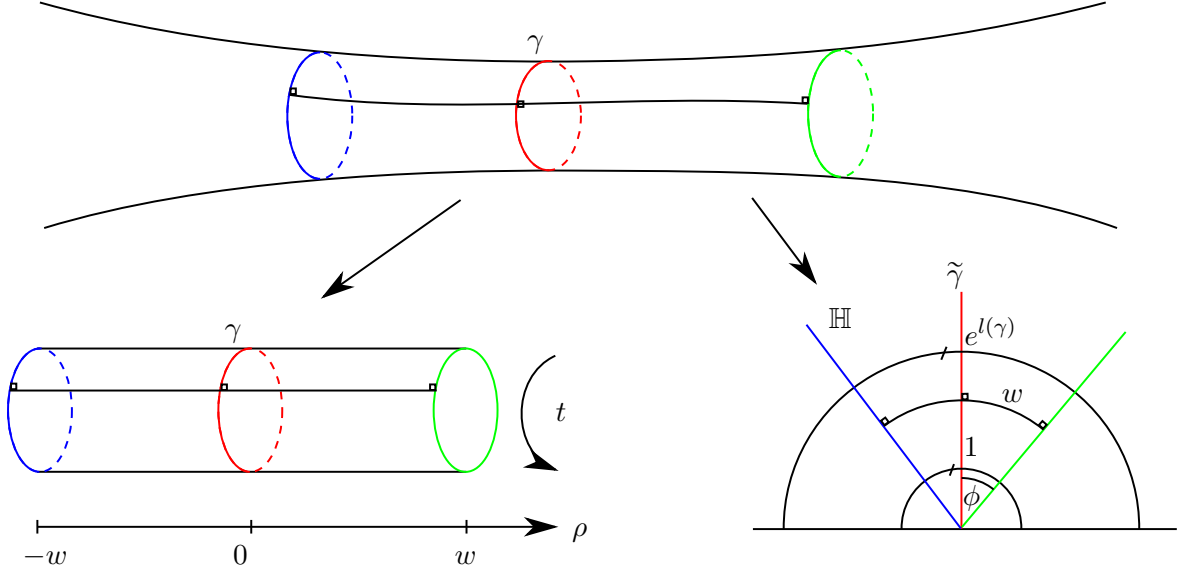
  \caption{Two possible parametrizations of a collar neighborhood. On the left hand side the metric is given by $\dd\rho^2+l^2(\gamma)\cosh^2\rho\;\dd t^2$ where $t\in[0,1]$ and on the right hand side the metric is the one from the upper half plane $\frac{\dd x^2+\dd y^2}{y^2}$. The black line perpendicular to all the colored curves is a geodesic perpendicular to the central geodesic $\gamma$ and the boundary hypercircles of constant geodesic curvature.}
  \label{fig:collar-neighborhoods}
\end{figure}

Recall that a \emph{horocycle} on a hyperbolic surface is a curve with constant geodesic curvature $1$. In the Poincar\'e half-plane model those are the circles tangent to the real axis and the lines parallel to the real axis. On the contrary the curves of arbitrary constant geodesic curvature are the straight lines and the circles intersecting the real axis. Also notice that horocycles have \emph{centers}, i.e.\ all perpendicular geodesics through a horocycle meet in one point called the center of the horocycle. In the circle case this is the ideal (tangent) point on the real axis and for the horizontal lines this is the point at infinity.

\index{Horocycle}
\index{Cusp Neighborhood}

\begin{definition}
  We call a neighborhood of a point $z\in C$ isometric to
  \begin{equation*}
    (-\infty,\log 2)\times\, \faktor{\RR}{\ZZ}
  \end{equation*}
  with the metric $\dd\rho^2+e^{2\rho}\dd t^2$ for $(\rho,t)\in(-\infty,\log 2)\times\, \faktor{\RR}{\ZZ}$ a \emph{cusp neighborhood} of $z$.
\end{definition}

\begin{prop}[See e.g.\ \cite{buser_geometry_2010}]
  A cusp neighborhood is also isometric to
  \begin{equation*}
    \{z\in\HH\mid \Im(z)\geq 1\}/(z \sim z+2)
  \end{equation*}
  via the map $(r,t)\longmapsto 2(t+\ii e^{-r})$.
\end{prop}

\begin{lem}[See e.g.\ \cite{buser_geometry_2010}]
  Any pair of standard collar neighborhoods of simple closed geodesics or cusp neighborhoods are disjoint in the interior.
  \label{lem:collar-and-cusp-neighborhood-disjoint}
\end{lem}

\begin{lem}[See \cref{lem:collar-neighborhood} and \cite{buser_geometry_2010}]
  The boundaries of cusp neighborhoods are horocycles of length $2$. The boundaries of collar neighborhoods are curves of constant geodesic curvature.
  \label{lem:unique-cusp-neighborhood}
\end{lem}

\subsection{Hyperbolic Description of Hurwitz Covers Close to the Boundary}

Consider a Hurwitz cover on surfaces with boundary, i.e.\ $u:C\lra X$ is holomorphic and $u$ maps the boundaries $\del_jC$ to boundaries $\del_iX$ for $\nu(j)=i$. The map
\begin{equation*}
  u|_{\del_jC}:\del_jC:\del_iX
\end{equation*}
has degree $l_j$ and we have $\sum_{\substack{j=1,\ldots,k \\ \nu(j)=i}}l_j=d$. Now equip both surfaces $C$ and $X$ with their uniformized hyperbolic metric as in \cref{sec:uniformization}. Thus the boundaries $\del_jC$ and $\del_iX$ are geodesics and $u$ is a local isometry. Furthermore, assume that we are given marked points on the boundaries $z_j\in\del_jC$ for all $j=1,\ldots,k$ such that $u(z_j)=u(z_{j'})$ if $\nu(z_j)=\nu(z_{j'})$.

\begin{lem}
  If $\gamma$ is a simple closed geodesic on $X$ of length $l$ with respect to the hyperbolic metric then its preimages under $u$ are again simple closed geodesics with respect to the corresponding hyperbolic metric. They have lengths $m_jl$ where $m_j$ is the degree of $u$ on the $j$-th preimage of $\gamma$. If we parametrize one such preimage $\eta$ by arc length $t\mapsto\eta(t)$ then we have $u(\eta(t))=\gamma(t)$ for all $t\in[0,\ldots,m_jl]$ if we pick corresponding starting points such that $u(\eta(0))=\gamma(0)$. The last statement also holds for geodesic arcs.
\label{lem:hyperbolic-lift-geodesic}
\end{lem}

\begin{proof}
  Since $u$ is a local isometry with respect to the uniformized hyperbolic metrics, see \cref{lem:hol-map-loc-isom}, the preimage of a geodesic is again a geodesic. As $u$ has finite degree the preimage of $\gamma$ is compact and thus consists of closed geodesics. If the preimage was not simple then the image $u(x)$ of an intersection point $x$ of the preimage would again be an intersection point of $\gamma$ which does not exist by assumption.

  Now pick one connected component of a preimage of a curve $\gamma$. The map $u$ restricted to this curve is an isometric cover of degree $m$ implying that its length is given by $m\cdot l(\gamma)$. Furthermore we can pick two points $x\in C$ and $y\in X$ such that $u(x)=y$ and $y\in \gamma$. Then the parametrization by arc-length $t\mapsto c(t)$ of the preimage has $|c'(t)|=1$ for all $t\in[0,1]$ and therefore $|(u\circ c)'(t)|=1$ as $u$ is a local isometry. Thus the parametrization by arc-length is mapped under $u$ to the parametrization by arc-length. The statement from the lemma follows since we have chosen corresponding starting points.
\end{proof}

The following gives some kind of hyperbolic normal form for Hurwitz covers close to boundaries. However, we postpone the precise definition of bordered Hurwitz covers until \cref{sec:orbifold-structure-moduli-space-borderd-huwritz-covers}.
\index{Hurwitz Cover!Bordered}

\begin{lem}
  Let a bordered Hurwitz cover $u:C\lra X$ and corresponding boundaries $\bigsqcup_{\nu(j)=i}\del_jC\lra\del_iX$ with degrees $l_j\geq 1$ for $j\in\{1,\ldots,k\}$ and $\nu(j)=i$ together with marked points on the boundary as above be given. Then there exist hyperbolic collar neighborhoods $\mcU_j$ and $\mcV_i$ close to $\del_jC$ and $\del_iX$ with charts $\phi_j:\mcU_j\lra(-\epsilon,0]\times S^1$ and $\psi_i:\mcV_i\lra(-\epsilon,0]\times S^1$, respectively, such that $u$ satisfies
  \begin{align*}
    \psi_i\circ u \circ \sqcup_j\phi_{j}^{-1}:  \bigsqcup_{\nu(j)=i}(-\epsilon,0]\times S^1 & \lra (-\epsilon,0]\times S^1 \\
     (r,\theta) & \longmapsto \psi_i(u(\phi_j^{-1}(r,\theta)))=(r,l_j\theta)\qquad (r,\theta)\in\mcU_j
  \end{align*}
  and $\phi_j^{-1}\left(\{0\}\times S^1\right)=\del_jC$ as well as $\phi_j(z_j)=0\in \faktor{\RR}{\ZZ}$ for all $j=1,\ldots,k$. This is illustrated in \cref{fig:u-in-collar-neighborhood}.
  \label{lem:collar-neighborhoods-boundary-hurwitz-cover}
\end{lem}

\begin{figure}[!ht]
  \centering
  \def\svgwidth{0.9\textwidth}
  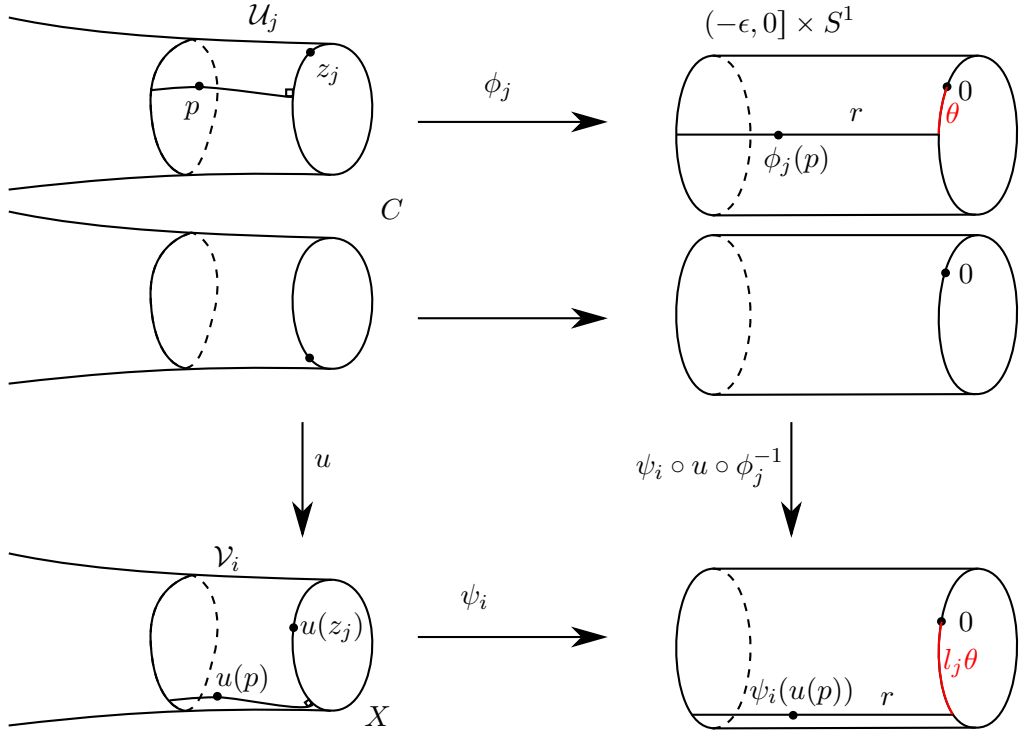
  \caption{This illustrates \cref{lem:collar-neighborhoods-boundary-hurwitz-cover}. Notice that the boundary curve as well as the perpendicular line are geodesics.}
  \label{fig:u-in-collar-neighborhood}
\end{figure}

\begin{proof}

Choose a positive real number $w\in\RR_{>0}$ such that
\begin{equation*}
\sinh(w)\sinh\left(\frac{d\cdot l(\del_iX)}{l(\del_iX)}\right)<1.
\end{equation*}
This will ensure that preimages of collar neighborhoods of width $w$ are within collar neighborhoods of preimages of the geodesic $\del_iX$ as the covering there has degree less than or equal to $d$ and using \cref{lem:hyperbolic-lift-geodesic}.

For $\del_iX$ we define the collar neighborhood $\mcV_i\coloneqq A_{w}(\del_iX)$ together with a standard hyperbolic coordinate chart $\psi_i:\mcV_i\lra(-w,0]\times S^1 $ such that $\psi_i(\del_iX)=\{0\}\times S^1$. Note that such a coordinate chart is unique up to rotation after fixing the collar neighborhood. Fix this rotation by requiring that the point $0\in S^1$ is mapped to the marked point $u(z_j)$ for any $j\in\nu^{-1}(i)$.\footnote{Note that this is well-defined by the requirement $u(z_j)=u(z_{j'})$ for $\nu(j)=\nu(j')$ as stated in the introduction of this section. Furthermore, recall that here $S^1=\faktor{\RR}{\ZZ}$.}

Now define the collar neighborhoods $\mcU_j$ with $\nu(j)=i$ as the connected component of $u^{-1}(\mcU_i)$ containing $\del_jC$. Consider any point $z\in \ol{\mcU_j}\setminus\del_jC$. The distance of $z$ to $\del_jC$ is realized by a geodesic from $z$ perpendicular to $\del_jC$ which is contained in $\mcU_j$. To see this recall that $\mcU_j$ is contained in the standard collar neighborhood of $\del_jC$. Call this geodesic $\gamma$. If $\gamma$ was not mapped injectively to $\mcV_i$ then its image would contain a geodesic loop which is not possible as its image needs to be a geodesic from $u(z)$ and perpendicular to $\del_iX$ because $u$ is a local isometry. Thus $l(\gamma)=l(u(\gamma))$ and we see that $\mcU_j=A_w(\del_jC)$. 

So we define $\phi_j:\mcU_j\lra(-w,0]\times S^1$ as the standard collar neighborhood parametrization such that $\phi_j(z_j)=0$. This is possible since the chart in \cref{lem:collar-neighborhood} is only defined up to a rotation. We can now infer in a similar way as above how the map $u$ looks in these charts.

Is is easy to see that the parametrization $\phi:S^1\lra \del_jC$ by $\phi(t)=\phi_j(0,t)$ is proportional to the parametrization by arc-length and similarly for $\del_iX$. \cref{lem:hyperbolic-lift-geodesic} thus gives $\psi_i\circ u \circ\phi_j^{-1}(0,\theta)=(0,l_j\theta)$. Any other point $(r,\theta)$ with $r\neq 0$ is mapped to the point constructed in the following way: First follow the unique geodesic from $(r,\theta)$ to the boundary at $(0,\theta)$, then $\psi_i\circ u\circ\phi_j^{-1}$ maps this point to $(0,l_j\theta)$ and then follow the geodesic perpendicular to $(0,l_j\theta)$ for the same arc-length. Since the metric on the cylinder has no $l$-dependent prefactor in front of the $\dd r^2$-term we end up with the same $r$ component, i.e.\ $\psi_i\circ u\circ\phi_j^{-1}(r,\theta)=(r,l_j\theta)$.

We thus have
\begin{equation*}
	u^{-1}(\mcU_i)=\bigsqcup_{\substack{1\leq j\leq k\\ \nu(j)=i}}A_w(\del_jC)
\end{equation*}
and the required local picture which finishes the proof.
\end{proof}

\begin{rmk}
  Notice that equivalent charts also exist at inner hyperbolic geodesics as the standard collar neighborhoods exist around those geodesics as well.
\end{rmk}

\subsection{Hyperbolic Gluing}

This section deals with gluing bordered Hurwitz covers along common geodesic boundaries of equal length. This is essentially a hyperbolic version of unique continuation for holomorphic maps.

\begin{lem}
  Let $C_1$ and $C_2$ be two oriented hyperbolic surfaces with geodesic boundary and $\gamma_i\subset C_i$ be two boundary geodesics of the same length. In the same way let $X_1$ and $X_2$ be two hyperbolic surfaces with two given geodesic boundaries $\eta_1$ and $\eta_2$. Furthermore, let points $q_1\in \gamma_1, q_2\in\gamma_2, p_1\in \eta_1$ and $p_2\in\eta_2$ be given. Also suppose we are given coverings $u_1:C_1\lra X_1$ and $u_2:C_2\lra X_2$ of degree $d$ which are local isometries and satisfy $u(q_1)=p_1$ and $u(q_2)=p_2$. Then the hyperbolic surfaces can be glued along the boundary to hyperbolic surfaces $C$ and $X$ and the maps fit together to yield a smooth covering $u:C\lra X$ of degree $d$ which is a local isometry and restricts to $u_i$ on $C_i\subset C$.
\label{lem:local-gluing}
\end{lem}

\begin{rmk}
Note that as usual when gluing along boundaries we need that the map identifying the boundary components is orientation-reversing. See \cref{fig:gluing-maps-along-hyperbolic-boundary} for the setup.
\end{rmk}

\begin{figure}[ht!]
\centering
\def\svgscale{0.8}
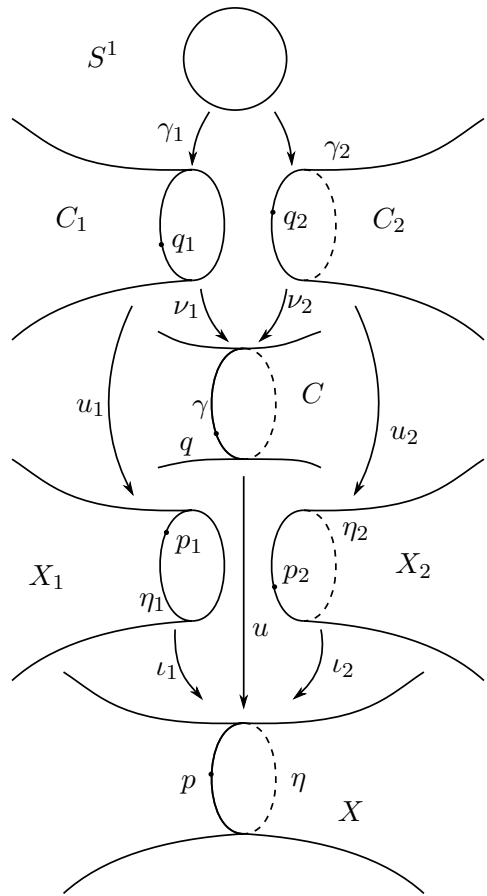
\caption{This figure shows the procedure of gluing surfaces with local isometries along common geodesic boundaries.}
\label{fig:gluing-maps-along-hyperbolic-boundary}
\end{figure}

\begin{proof}
As $\gamma_1$ and $\gamma_2$ are hyperbolic geodesics of the same length we can parametrize them by arc length in the direction of their induced orientation with $q_1$ and $q_2$ as zeros. We denote these parametrizations again by $\gamma_i:[0,1]\lra C_i$. Now we can define the topological manifold
\begin{equation*}
  C\coloneqq\faktor{C_1\sqcup C_2}{\gamma_1(t)\sim\gamma_2(-t)}
\end{equation*}
which has a manifold atlas on $C\setminus\gamma$ where $\gamma$ is the image of $\gamma_1$ and $\gamma_2$ on $C$. Since both curves $\gamma_1$ and $\gamma_2$ are geodesics of the same length we can build charts around a point $z\in C$ as follows. For both geodesics choose collar neighborhood parametrizations $\phi_1:(-w,0]\times S^1\lra C_1$ and $\phi_2:[0,w)\times S^1\lra C_2$. In these charts both hyperbolic metrics are given by $\phi_i^*g=\dd r^2+l^2\cosh^2r \;\dd \theta^2$. Now pick a small positive real number $\epsilon$ and define
\begin{align*}
  \phi_z:B_{\epsilon}(z) & \lra \RR^2 \\
  x & \longmapsto\begin{cases} \phi_1(x) & x\in C_1, \\ \phi_2(x) & x\in C_2,\end{cases}
\end{align*}
where $B_{\epsilon}(z)$ is the union of the two metric (half-)balls in $C_1$ and $C_2$, respectively. As the transition functions for these charts are either contained in one of the atlases of the $C_i$ or are given by the identity map on the collar neighborhoods, we see that this gives a smooth manifold structure. Furthermore, the Riemannian metrics on both sides of $\gamma$ are equal in these charts so we obtain a hyperbolic surface such that $\gamma$ is a geodesic.

We can do the same construction with the target surface $X$ using the points $p_1$ and $p_2$ as reference points for the gluing.

Since $u_1$ and $u_2$ both are local isometries they satisfy $u_1(\gamma_1(t))=u_2(\gamma_2(t))$ as $\gamma_1$ and $\gamma_2$ are both parametrized by arc length and use the same point $q_1=q_2=q$ as base point. Therefore the map $u$ given by $u_i$ on $C_i$ is continuous on $C$. 

Next we show that $u$ is in fact $C^1$ and holomorphic. To this end, choose coordinates at $z\in\gamma$ using collar neighborhoods on $C$ and on $X$. As $u$ is a local isometry it needs to map geodesics to geodesics and preserve angles between geodesics. Thus it looks like $(r,\theta)\longmapsto (r,d\theta)$ in the standard collar neighborhood charts where $d$ is the degree of $u$ on $\gamma$. This map is obviously smooth and a local isometry as the lengths of $\gamma$ and $u(\gamma)$ differ by a factor of $d$ which cancels. Thus it is also conformal and therefore holomorphic. This proves \cref{lem:local-gluing}.
\end{proof}

\subsection{Reference Curves}

\label{sec:reference-curve}
\index{Reference Curve}

In this section we define reference curves which will be used later for marking points close to a boundary component in such a way that the marking persists in the limit of geodesic boundary lengths going to zero. This was introduced by Mirzakhani in \cite{mirzakhani_weil-petersson_2007}.

First we choose a continuous strictly decreasing function $F:\RR_{\geq 0}\lra\RR_{>0}$ such that
\begin{equation*}
  l< F(l)\leq \frac{l}{\tanh\left(\frac{l}{2}\right)}
\end{equation*}
and $0<\lim_{l\to 0}F(l)<2$.

\begin{definition}
  Let $C$ be an admissible Riemann surface with its hyperbolic structure. A curve $\gamma\subset C$ is called a \emph{reference curve} at the boundary $\del_jC$ or the cusp $q_j$ if
  \begin{enumerate}[label=(\roman*), ref=(\roman*)]
    \item at a boundary component $\gamma$ is given as the hypercircle\footnote{\emph{Hypercircles} are curves equidistant from a geodesic.} of length $F(l(\del_jC))$ or
    \item at a cusp $\gamma$ is given as a simple closed horocycle with length $F(0)$ with the cusp as a center and which is contained in its cusp neighborhood.
  \end{enumerate}
\end{definition}

\begin{lem}
  Reference curves exist and are unique.
  \label{lem:existence-reference-curves}
\end{lem}

\begin{proof}
  Consider a boundary component $\del_jC$ with a standard collar neighborhood $A_w$. By lifting to the hyperbolic plane in such a way that $\del_jC$ lifts to the vertical line through the origin we can use the second model from \cref{lem:collar-neighborhood} to describe the situation.

Hypercircles in the hyperbolic plane are given by circular arcs intersecting the $x$-axis non-perpendicularly or straight lines intersecting the $x$-axis non-perpendicularly. In the former case their axis is the geodesic with the same intersection points as ideal end points and in the latter case the axis is the vertical geodesic through the same ideal point. Notice that these are precisely curves of constant curvature.

Using these results on hypercircles in the hyperbolic plane it is clear that there exists precisely one such reference curve. Notice that the function $F$ was chosen such that the length stays between the length of the geodesic and the boundary of the collar neighborhood.

In the case of a cusp we can use the model in the upper half plane, i.e.\ $\{z\in\HH\mid \Im(z)\geq 1\}$ with $z$ and $z+2$ identified for all $z\in\HH$. Now the horocycles are the horizontal lines which are of length $\frac{2}{h}$ and thus there exists exactly one such horocycle of length $F(0)$.
\end{proof}

\begin{rmk}
  Note that there exists a bijection between points on the boundary component $\del_jC$ and a reference curve located at this boundary component by following the geodesics perpendicular to the boundary component. This bijection extends to a bijection of a reference curve at a cusp and the circle of directions in the tangent space at the cusp. Again, see \cite{mirzakhani_weil-petersson_2007} for details.

  Also note that the reference curves of geodesic boundaries converge to the reference curve of the cusp when the boundary length goes to zero. This is because away from the boundary the hyperbolic metrics converge in $\cin_{\text{loc}}$ and thus the lengths converge and the curves stay curves of constant geodesic curvature converging to one.
  
  Although we don't need this precise value notice that the hypercircle of length $F(l)$ has distance $\arsinh\sqrt{\frac{F(l)^2}{l^2}-1}$ from the boundary geodesic of length $l$. This can be seen by the same calculation as in the proof of \cref{lem:collar-neighborhood}. In particular this length goes to infinity for $l\to 0$ as is expected because horocycles are ``infinitely far away'' from the cusp.
\end{rmk}

\chapter{Orbifold Structures}

\label{sec:orbifolds}

\section{Orbifold Groupoids}

\subsection{Groupoids}

\index{Groupoid}

In this chapter we will describe the type of structure we will put on the moduli space of Hurwitz covers. As this will not be precisely an orbifold structure due to the local description as products of discs glued at exactly one point we will denote this structure as $s$-orbifolds.

\begin{definition}
  A (small) \emph{groupoid} is a (small) category $\mcC=(\Ob\mcC,\Mor\mcC)$ such that all morphisms are invertible. This means that there exist structure maps
  \begin{itemize}
    \item $i:\Mor\mcC\lra\Mor\mcC$ which maps a morphism to its inverse,
    \item $s:\Mor\mcC\lra\Ob\mcC$ which maps a morphism to its source object,
    \item $t:\Mor\mcC\lra\Ob\mcC$ which maps a morphism to its target object,
    \item $u:\Ob\mcC\lra\Mor\mcC$ which maps an object to its identity morphism and
    \item $m:\Mor\mcC {_s\times_t}\Mor\mcC\lra\Mor$ which maps two composable morphisms to their composition.\footnote{By $\Mor\mcC {_s\times_t}\Mor\mcC$ we mean the set $\{(f,g)\in\Mor\mcC\times\Mor\mcC\mid s(f)=t(g)\}$.}
  \end{itemize}
\end{definition}

\index{Groupoid!Structure Maps}

\begin{definition}
  A groupoid $\mcC$ is called
  \begin{itemize}
    \item \emph{Lie} if $\Mor\mcC$ and $\Ob\mcC$ are smooth manifolds, $s$ and $t$ are submersions (such that in particular $\Mor\mcC{_s\times_t}\Mor\mcC$ is a smooth manifold) and all structure maps are smooth,
    \item \emph{étale} if $\mcC$ is Lie and in addition all structure maps are local diffeomorphisms,
    \item \emph{proper} if the map $s\times t:\Mor\mcC\lra\Ob\mcC\times\Ob\mcC$ is proper,
    \item \emph{ep Lie} or \emph{orbifold groupoid} if it is étale, proper and Lie,
    \item \emph{non-singular} if all stabilizers $G_x:=\Mor(x,x)$ for $x\in\Ob\mcC$ are trivial,
    \item \emph{effective} if for each $x\in\Ob\mcC$ and $g\in G_x$ every neighborhood $V\subset\Mor\mcC$ of $g$ contains a morphism $h\in V$ such that $s(h)\neq t(h)$,
    \item \emph{stable} if all stabilizers $G_x$ are finite and
    \item \emph{connected} if $|\mcC|$ is path connected.
  \end{itemize}
\end{definition}

\index{Groupoid!Orbifold|(}
\index{Groupoid!Lie}
\index{Groupoid!Proper}
\index{Groupoid!\'Etale}
\index{Groupoid!Effective}
\index{Groupoid!Stable}
\index{Groupoid!Connected}

We denote the \emph{orbit space} by $|\mcC|:=\faktor{\Ob\mcC}{(x\sim y\Longleftrightarrow \Mor(x,y)\neq\emptyset )}$.

\begin{example}
  An important example is the Lie groupoid associated to a smooth action of a Lie group $G$ on a smooth manifold $M$. Is is called the \emph{translation groupoid} and is denoted by $G\ltimes M$. Its objects are $\Ob G\ltimes M:= M$, its morphisms $\Mor G\ltimes M:= G\times M$ and the structure maps $s(g,x)=x$, $t(g,x)=g\cdot x$ and $m((g_1,x_1),(g_2,x_2))=(g_1\circ g_2,x_2)$ if $x_1=g_2\cdot x_2$.
\end{example}

The following are a few statements about ep-Lie groupoids and implications for their orbit spaces taken from \cite{mcduff_groupoids_2005}.

\begin{definition}
  An ep-Lie groupoid is called symplectic or complex if the object and morphism sets are in fact symplectic or complex manifolds and the structure maps are symplectic or holomorphic. 
\end{definition}

An element $g\in G_x$ is an arrow $g:x\lra x$ but if it is a morphism in a stable étale Lie groupoid $\mcC$ it actually ``acts'' on a neighborhood $U\subset \Ob\mcC$ in the following sense.

\begin{lem}
  Given a stable étale Lie-groupoid $\mcC$ and an element $x\in\Ob\mcC$ there exist a neighborhood $U\subset \Ob\mcC$ of $x$ and pairwise disjoint neighborhoods $N_g\subset\Mor\mcC$ with $g\in N_g$ for $g\in G_x$ such that $s$ and $t$ map each $N_g$ diffeomorphically onto $U$.
\end{lem}

\begin{proof}
  See \cite{robbin_construction_2006}.
\end{proof}

\begin{rmk}
Given $g\in G_x$ we thus obtain a diffeomorphism $\phi_g:U\lra U$ by mapping $s(h)\longmapsto t(h)$ for $h\in N_g$, i.e.\ $\phi_g:=t\circ s^{-1}$. This means that for each $x\in\Ob\mcC$ there exists a neighborhood $U$ of $x$ such that the stabilizer group $G_x$ acts on that neighborhood via diffeomorphisms. Note that if $g\in\Mor(x,y)$ then there still exist neighborhoods $U,V\subset\Ob\mcC$ with $x\in U$ and $y\in V$ such that the above construction gives a diffeomorphism $\phi_g:U\lra V$.

Note that such a neighborhood is not yet small enough to give some kind of orbifold chart around $[x]\in|\mcG|$ because there might be morphisms identifying two elements in $U$ which do not come from $g\in G_x$ in this manner. The following lemma from \cite{adem_orbifolds_2007} shows that it is possible to choose an even smaller neighborhood good enough for this task.
\label{rmk:action-groupoid-neighborhood}
\end{rmk}

\begin{lem}
  Let $\mcG$ be an effective ep-Lie groupoid. Then for any $x\in\Ob\mcG$ there exist neighborhoods $U_x\subset\Ob\mcG$ such that
  \begin{enumerate}[label=(\roman*), ref=(\roman*))]
    \item $G_x$ acts on $U_x$ as above,
    \item any $h\in\Mor\mcG$ such that $s(h),t(h)\in U_x$ comes from the action of some $g\in G_x$ on $U$ in the sense of \cref{rmk:action-groupoid-neighborhood} and
    \item $\faktor{U_x}{G_x}\subset|\mcG|$ is an open embedding.
  \end{enumerate}
  \label{lem-existence-good-neighborhoods}
\end{lem}

\begin{proof}
  See the proof of Proposition~1.44 in \cite{adem_orbifolds_2007}.
\end{proof}

\begin{definition}
  An \emph{orbifold chart} around a point $x\in\Ob\mcG$ for an effective ep-Lie groupoid $\mcG$ is a triple $(U_x,G_x,\pi)$ where $x\in U_x\subset\Ob\mcG$ is an open neighborhood of $x$ such that $G_x=\Mor(x,x)\subset\Mor\mcG$ acts on $U_x$ as in \cref{lem-existence-good-neighborhoods} together with a projection $\pi:U_x\lra \faktor{U_x}{G_x}$ such that $\faktor{U_x}{G_x}$ is homeomorphic to its image in $|\mcG|$.
  \label{def-orb-chart}
\end{definition}

\index{Orbifold!Chart}
\index{Orbifold!Atlas}

\begin{definition}
  An \emph{orbifold atlas} of an orbifold groupoid $\mcG$ is a collection of orbifold charts $\{(U_i,G_i,\pi_i)\}_{i\in I}$ as in \cref{def-orb-chart} such that the open sets $|U_i|=\faktor{U_i}{G_i}$ cover all of $|\mcG|$. Notice that on every chart we have chosen implicitly a central point $x_i\in\Ob\mcG$ with $x_i\in U_i$ and $G_i=\Aut_{\mcG}(x_i)$.
\end{definition}

\begin{rmk}
  Notice that this definition of an orbifold atlas is \emph{not} the standard one and in particular is note enough to recover the whole orbifold. We are missing how to glue coordinate charts, i.e.\ morphisms between different $U_i$'s as well as some information on how the groups $U_i$ are related. However, all this information is contained in the groupoid and as we will not deal with orbifolds defined only via such an atlas this is good enough for us. One way how to formulate such an actual atlas in our language is by choosing the full subcategory of $\mcG$ defined by $\bigcup_{i\in I}U_i$ and taking the fibre product along $\bigsqcup_{i\in I}U_i\lra \bigcup_{i\in I} U_i$. This category is equivalent to $\mcG$ in the sense of \ref{def:homomorphisms_of_groupoids} and it contains everything from the atlas in addition to information on how to glue things together.
\end{rmk}

\begin{prop}[See \cite{mcduff_groupoids_2005}]
  If $\mcC$ is an étale Lie-groupoid then the following holds:
  \begin{enumerate}[label=(\roman*), ref=(\roman*)]
    \item If $\mcC$ is proper then it is stable.
    \item The projection $\pi:\Ob\mcC\lra |\mcC|$ is open.
    \item $|\mcC|$ is Hausdorff if and only if $s\times t$ has closed image. In particular, if $\mcC$ is in addition proper then the topology on the orbit space is Hausdorff.
    \item If $\mcC$ is non-singular and proper, then its orbit space $|\mcC|$ is a manifold.
    \item If $\mcC$ is connected and proper and for some $x\in\Ob\mcC$ the stabilizer group $G_x$ acts effectively on a neighborhood of $x$, then $\mcC$ is effective.
    \item If $\mcC$ is connected and proper, then the isomorphism class of the subgroup $K_x\subset G_x$ which acts trivially on a neighborhood of $x\in\Ob\mcC$ is independent of $x\in\Ob\mcC$.
    \item If $\mcC$ is connected and proper, then there exists an effective groupoid $\mcC_{\text{eff}}$ with the same objects as $\mcC$ and morphisms $\faktor{\Mor_{\mcC}(x,y)}{K_y}$.
  \end{enumerate}
  \label{prop:properties-ep-lie-groupoids}
\end{prop}

\begin{lem}
  Let $\mcG$ be an effective ep-Lie groupoid. There exists a full subcategory $\mcB\subset\mcG$ such that the inclusion is an equivalence of ep-Lie groupoids (see \cref{def:homomorphisms_of_groupoids} in the next section) and such that $\mcB$ is again an ep-Lie groupoid such that for every object $x\in\Ob\mcB$ the preimages $s^{-1}(x)$ and $s^{-1}(t)$ in $\Mor\mcG$ are finite sets.
  \label{lem:existence-nice-orb-chart}
\end{lem}

\begin{rmk}
  Note that the finiteness condition is equivalent to the fact that every equivalence class $[x]$ has only finitely many representatives $x\in\Ob\mcB$. Also this does not follow from the general definition of an ep-Lie groupoid since we require $s$ and $t$ to be submersions from which follows that $s^{-1}(x)$ is zero-dimensional but the properness only gives finiteness for $(s\times t)^{-1}(x,y)$, i.e.\ after fixing source \emph{and} target.
\end{rmk}

\begin{proof}
  Cover $\Ob\mcG=\bigcup_{x\in X}U_x$ by orbifold charts as in \cref{def-orb-chart}. Notice that $|\mcG|$ is Hausdorff and second-countable because it is the image of an open quotient map and $\Ob\mcG$ is second-countable and because of \cref{prop:properties-ep-lie-groupoids}. Thus we can pick a subset $I\subset X$ such that $\bigcup_{i\in I}|U_i|$ covers $|\mcG|$ in a locally finite way, i.e.\ for any class $[x]\in\mcG$ there exist only finitely many $i\in I$ such that $[x]\in |U_i|$. Now define $\mcB$ as the full subcategory defined by $\bigcup_{i\in I}U_i\subset\Ob\mcG$. By construction every equivalence class has only finitely many representatives in $\obj\mcB$ with each only finitely many automorphisms and thus source and target map have only finitely many preimages. This $\mcB$ is an étale Lie groupoid as $\Ob\mcB$ is an open subset of $\Ob\mcG$ and all the structure maps are restrictions of smooth maps to open subsets. The properness is also clear as the diagram
  \begin{equation*}
    \xymatrix{
      \Mor\mcB \ar[d] \ar[r]^-{s\times t} & \Ob\mcB\times\Ob\mcB \ar[d] \\
      \Mor\mcG  \ar[r]^-{s\times t} & \Ob\mcG\times\Ob\mcG
      }
  \end{equation*}
commutes where the vertical arrows are homeomorphisms onto their images. The equivalence is also true by construction.
\end{proof}

\subsection{Equivalences}

Next we define a few notions of morphisms between ep-Lie groupoids in order to explain the notion of equivalent orbifold atlases and maps between orbifolds.

\index{Orbifold!Homomorphism}
\index{Orbifold!Equivalence}
\index{Morita Equivalence}
\index{Orbifold!Fibered Product of}

\begin{definition}
  Let $\mcC$, $\mcD$ and $\mcG$ be ep-Lie groupoids.
  \begin{enumerate}[label=(\roman*), ref=(\roman*)]
    \item A \emph{homomorphism} $\phi:\mcC\lra\mcD$ is a functor whose maps on objects and morphisms are smooth and commute with all structure maps.
    \item Two homomorphisms $\phi,\eta:\mcC\lra\mcD$ are identified if there exists a natural transformation $\alpha:\phi\Rightarrow\eta$ which is smooth as a map $\alpha:\Ob\mcC\lra\Mor\mcD$.
    \item If $\phi:\mcC\lra\mcG$ and $\psi:\mcD\lra\mcG$ are homomorphisms then we define the \emph{fibered product of groupoids} $\mcC\times_{\mcG}\mcD$ as the groupoid with objects $\Ob\mcC\times_{\mcG}\mcD=\{(y,g,z)\mid y\in\Ob\mcC,z\in\Ob\mcD,g:\phi(y)\lra\psi(z)\}$ and morphisms
      \begin{equation*}
        \Mor_{\mcC\times_{\mcG}\mcD}((y,g,z),(y',g',z')) = \left\{ \begin{aligned} (h,k)\mid h:y\lra y',k:z\lra z', \\ \text{ s.t. }g'\phi(h)=\psi(k)g \end{aligned} \right\}.
      \end{equation*}
      Composition and structure maps are defined in the obvious ways.
    \item A homomorphism $\phi:\mcC\lra\mcD$ is called an \emph{equivalence} if the map $t\circ\pr_1:\Mor\mcD{}_s\times_{\phi}\Ob\mcC\lra \Ob\mcD$ is a surjective submersion and the square
      \begin{equation*}
        \xymatrix{
          \Mor\mcC \ar[r]^{\phi} \ar[d]_{(s,t)} & \Mor\mcD \ar[d]^{(s,t)} \\
          \Ob\mcC\times\Ob\mcC \ar[r]^{\phi\times\phi} & \Ob\mcD\times\Ob\mcD
          }
      \end{equation*}
      is a fibered product of manifolds. This means that $\Mor\mcC$ is diffeomorphic to the pull-back of $\Mor\mcD$ along $\phi\times\phi$ in such a way that all the maps are smooth and commute.
   % \item A \emph{strong equivalence} is an equivalence $\phi:\mcC\lra\mcD$ such that the map $\phi:\Ob\mcC\lra\Ob\mcD$ is already a surjective submersion.
    \item $\mcC$ and $\mcD$ are called \emph{Morita equivalent} if there exists a third ep-Lie groupoid $\mcG$ together with two equivalences $\mcC\overset{\phi}{\longleftarrow}\mcG\overset{\phi'}{\lra}\mcD$.
  \end{enumerate}
  \label{def:homomorphisms_of_groupoids}
\end{definition}

\begin{rmk}
  A few comments are in order.
  \begin{enumerate}[label=(\roman*), ref=(\roman*)]
    \item The definitions also make sense in the Lie groupoid setting. There are various statements which properties are preserved under Morita equivalence like properness and effectiveness and which ones are not like \'etalness. However, we will only deal with orbifolds and thus restrict everything to the ep-Lie groupoid case.
    \item Fibered products of groupoids are not always Lie groupoids because of transversality issues. However, the submersion condition for an equivalence ensures that you can ``compose'' two equivalences by their ep-Lie groupoid fibered product as in
      \begin{equation*}
        \xymatrix{
          & & \mcG{}_{\epsilon'}\times_{\epsilon''}\mcH \ar[dl]^{\rho} \ar[dr]^{\rho'} & & \\
          & \mcG \ar[dl]^{\epsilon} \ar[dr]^{\epsilon'} & & \mcH \ar[dl]^{\epsilon''} \ar[dr]^{\epsilon'''} & \\
          \mcC & & \mcD & & \mcE
          }
      \end{equation*}
      In \cite{moerdijk_introduction_2003} it is proven that the fibered product always exists if one of the homomorphisms is an equivalence. Also in the square the functor $\rho$ is an equivalence whenever $\epsilon''$ is. Composition of two equivalences is again an equivalence. Also note that the square does commute only up to a (smooth) natural transformation.
      \item The last statements show that Morita equivalence is indeed an equivalence relation where transversality is proved via fibered products and symmetry is achieved by changing the type of the diagrams as you can just ``rotate'' the triangle.
      \item Notice that all natural transformations are automatically natural isomorphisms as all the categories are groupoids.
      \item It is reasonable to consider homomorphisms only up to smooth natural transformations as two such related homomorphisms descend to the same map on the quotient and induce maps on morphism sets that are conjugated. For example if $\alpha:\mcF\Rightarrow\mcG$ for $\mcF,\mcG:\mcC\lra\mcD$, then we have maps $\mcF:\Hom_{\mcC}(x,y)\lra\Hom_{\mcD}(\mcF(x),\mcF(y))$ and $\mcG:\Hom_{\mcC}(x,y)\lra\Hom_{\mcD}(\mcG(x),\mcG(y))$ which have the same kernel and whose images are isomorphic by the map $(f:x\to y)\mapsto \alpha_y\circ f \circ\alpha_x^{-1}$.
      \item The conditions for an equivalence imply in particular that the functor is essentially surjective, i.e.\ all equivalence classes of the target groupoid are hit and that it is fully faithful as the commuting diagram implies $\Hom_{\mcC}(x,y)\simeq\Hom_{\mcD}(\phi(x),\phi(y))$. Thus equivalences are in particular category equivalences.
      \item This whole set of definitions and statements is actually an instance of the calculus of fractions, see \cite{gabriel_calculus_1967}. In this language our equivalences are weak equivalences that we try to invert artificially.
      \item Note that an equivalence $\phi:\mcC\lra\mcD$ induces a homeomorphism on the quotient spaces $|\mcC|\lra|\mcD|$.
    \end{enumerate}
\end{rmk}

\subsection{Orbifolds and their Morphisms}

\index{Orbifold!Groupoid}

From now on we will refer to ep-Lie groupoids as \emph{orbifold groupoids}.

\begin{definition}
  An \emph{orbifold} is a topological space $X$ together with a Morita equivalence class of orbifold groupoids $\mcG$ together with a homeomorphism $g:|\mcG|\lra X$. This means $(\mcG,g)\sim(\mcH,h)$ if and only if there exists an orbifold groupoid $\mcC$ with two equivalences $\epsilon:\mcC\lra\mcG$ and $\epsilon':\mcC\lra\mcH$ inducing a homeomorphism $\phi:|\mcH|\lra|\mcG|$ such that
  \begin{equation*}
    \xymatrix{
      |\mcH| \ar[rr]^{\phi} \ar[rd]^h & & |\mcG| \ar[dl]^g \\
      & X &
      }
  \end{equation*}
  commutes. An orbifold groupoid and the corresponding homeomorphism representing the Morita equivalence class of the orbifold structure is called an \emph{orbifold representation}.
\end{definition}

Now we can define various differential-geometric notions of maps and functions.

\index{Strong map}

\begin{definition}
  Let $(\mcC,c)$ and $(\mcD,d)$ be orbifold representations for orbifolds $X$ and $Y$, respectively.
  \begin{enumerate}[label=(\roman*), ref=(\roman*)]
    \item A \emph{smooth function} $f$ on the orbifold $X$ is a smooth function $f:\Ob\mcC\lra\RR$ such that $f(x)=f(y)$ for all points $x,y\in\Ob\mcC$ such that $\Hom(x,y)\neq\emptyset$.
    \item A \emph{(strong) map between orbifolds $X$ and $Y$} is an equivalence class of a tuple $(\epsilon,\mcG,\phi)$ and a map $f:X\lra Y$, where $\mcC\overset{\epsilon}{\longleftarrow}\mcG\overset{\phi}{\lra}\mcD$ such that $\mcG$ is an orbifold groupoid, $\epsilon$ is an equivalence, $\phi$ is a homomorphism of groupoids and the diagram
      \begin{equation*}
        \xymatrix{
          |\mcC| \ar[r]^{\phi_*\circ\epsilon_*^{-1}} \ar[d]^c & |\mcD| \ar[d]^d \\
          X \ar[r]^f & Y
          }
      \end{equation*}
      commutes. Two such strong maps
      \begin{align*}
        (\epsilon,\mcG,\phi,f) & :(X,\mcC,c)\lra(Y,\mcD,d)\text{ and} \\
        (\epsilon',\mcG',\phi',f') & :(X,\mcC,c)\lra(Y,\mcD,d)
      \end{align*}
      are identified if and only if there exists a homomorphism $\gamma:\mcG\lra\mcG'$ such that
      \begin{equation*}
        \xymatrix{
          & \mcG \ar[dl]^{\epsilon} \ar[dr]^{\phi} \ar[dd]^{\gamma} & \\
          \mcC & & \mcD \\
          & \mcG' \ar[ul]^{\epsilon'} \ar[ur]^{\phi'} &
          }
      \end{equation*}
      commutes up to smooth natural transformations.
  \end{enumerate}
\end{definition}

\begin{rmk}
  Again, a few comments are in order.
  \begin{enumerate}[label=(\roman*), ref=(\roman*)]
    \item Notice that in the definition of equivalence between strong maps we did not say anything about the map $f:X\lra Y$ because the two homomorphisms induce the same map on the quotient space since they are related by a natural transformation. Also it is not trivial to show that this is indeed an equivalence relation.
    \item The obvious definition of a map between orbifolds would be a homomorphism inducing a continuous map on the orbit spaces but this definition depends on the orbifold presentation so we allow for refinements of the structure and allow for homomorphisms that factor through Morita equivalence.
    \item The last point means that we need in fact to extend the definition of equivalence of maps to Morita equivalent groupoids. However, this is easily done as any strong map between orbifold groupoids induces a map on any Morita equivalent one by the fibered product construction. Since we do not need this in this document we will skip the relevant diagrams.
    \item Let us spell out the corresponding notion of isomorphism between two orbifolds $(X,\mcC,c)$ and $(Y,\mcD,d)$. They are isomorphic if and only if there exist two strong maps $(\epsilon,\mcG,\phi):(X,\mcC,c)\lra(Y,\mcD,d)$ and $(\rho,\mcG',\eta):(Y,\mcD,d)\lra(X,\mcC,c)$ such that their compositions are naturally isomorphic to the identity, i.e.
      \begin{equation*}
        \xymatrix{
          & & \mcG{}_{\phi}\times_{\rho}\mcG' \ar[dl]^{\rho'} \ar[dr]^{\phi'} \ar@/^1pc/[ddd]^<<<<<<<<{\epsilon\circ\rho'}  & & \\
          & \mcG \ar[dl]^{\epsilon} \ar[dr]^{\phi} & & \mcG' \ar[dl]^{\rho} \ar[dr]^{\eta} & \\
          \mcC & & \mcD & & \mcC \\
          & & \mcC \ar[ull]^{\id} \ar[urr]^{\id} & &
          }
      \end{equation*}
      commutes up to smooth natural transformation. The smoothness is crucial here to make sure that e.g.\ the dimension is an actual invariant of orbifolds. Additionally one can trace the automorphism groups through the diagram to see that these are indeed invariants as well.
    \item In \cite{moerdijk_orbifolds_2002} it is explained that although the notion of a strong map between orbifolds is somewhat complicated one can in fact show that any strong map $(\epsilon,\mcG,\phi):(X,\mcC,c)\lra(Y,\mcD,d)$ can be represented in the following way. Choose an open cover $\mcU=\{U_i\}_{i\in I}$ of $\Ob\mcC$ and define $\mcC_{\mcU}$ by $\Ob\mcC_{\mcU}\coloneqq \bigsqcup_{i\in I}U_i$ and $\Hom_{\mcC_{\mcU}}(x\in U_i,y\in U_j)\coloneqq\Hom_{\mcC}(x,y)$. Then every strong map has a representative
      \begin{equation*}
        \xymatrix{
          & \mcC_{\mcU} \ar[dl]^{\epsilon} \ar[dr]^{\phi'} & \\
          \mcC & & \mcD
          }
      \end{equation*}
      for some cover $\mcU$ fine enough where $\phi$ is a homomorphism. This means we do not need to pass to arbitrary wild equivalent groupoid categories.
  \end{enumerate}
\end{rmk}

\section{Orbifold bundles and coverings}

\label{sec:orb-bundles-coverings}
\index{$\mcG$-space}

\begin{definition}
  Let $\mcG$ be an ep-Lie groupoid and $E$ be a manifold. We call $E$ a \emph{$\mcG$-space} if there exists an action of $\mcG$ on $E$, i.e.
  \begin{itemize}
    \item an \emph{anchor} $\pi:E\lra\Ob\mcG$ and
    \item an \emph{action} $\mu:\Mor\mcG {}_s\times_{\pi} E\lra E$, written as $\mu(g,e)=g\cdot e$
  \end{itemize}
  which satisfy $\pi(g\cdot e)=t(g)$, $1_x\cdot e=e$ and $g\cdot (h\cdot e)=(gh)\cdot e$ for all suitable $g,h$ and $e$.
  \label{def:g-space}
\end{definition}

\index{Groupoid!Translation}

\begin{rmk}
  Given a $\mcG$-space $E$ we can associate to it the \emph{translation groupoid} $\mcE:=\mcG\ltimes E$ with objects $\Ob\mcE:= E$ and morphisms $\Mor\mcE:=\Mor\mcG{_s\times_{\pi}}E$ with the obvious structure maps. There exists a homomorphism of groupoids $\pi_{\mcE}:\mcE\lra\mcG$ and the fibre of $\Ob \mcE\lra|\mcE|$ is given by $\faktor{\pi_{\mcE}^{-1}(x)}{G_x}$.
\end{rmk}

\begin{definition}
  Let $\mcG$ be an ep-Lie groupoid and $E$ a $\mcG$-space. The we call $E$
  \begin{enumerate}[label=(\roman*), ref=(\roman*)]
    \item an \emph{orbifold cover} over $\mcG$ if the map $\pi:E\lra\Ob\mcG$ is a covering. The translation groupoid $\mcE$ can be considered as the total space of the covering.
    \item $E$ is called a \emph{(orbi-)vector bundle} over $\mcG$ if $\pi:E\lra\Ob\mcG$ is a vector bundle such that the action of $\mcG$ on $E$ is fibrewise linear. Again, the ep-Lie groupoid $\mcE=\mcG\ltimes E$ associated to this bundle is the total space of the vector bundle.
    \item $E$ is called a \emph{principal $T^n$-(orbi-)bundle} over $\mcG$ if $\pi:E\lra\Ob\mcG$ is a principal torus bundle and the action of $\mcG$ on $E$ is a fibrewise $T^n$-equivariant map.
    \item A \emph{section} $\sigma$ of a vector bundle $\mcE$ over $\mcG$ as above is an invariant section of $E\lra\Ob\mcG$, i.e.\ $g\cdot \sigma(x)=\sigma(y)$ for $g:x\lra y$.
  \end{enumerate}
  \label{def:orbifold-bundles}
\end{definition}

\index{Orbifold!Vector Bundle}
\index{Orbifold!Cover}
\index{Orbifold!Principal Bundle}

\begin{rmk}
 Note that the homomorphism $\pi_{\mcE}:\mcE\lra\mcG$ induces a projection $\pi_{|\mcE|}:|\mcE|\lra|\mcG|$ which is no longer a vector bundle as the type of the fibre might jump. Also, $G_x$ acts linearly on the fibre $E_x$ for $x\in\Ob{\mcG}$.
\end{rmk}

The local picture for an orbifold cover $\pi:E\lra\Ob\mcG$ is as follows. Let $x\in\Ob\mcG$ and $U_x\subset\Ob\mcG$ a neighborhood of $x$ such that $(U_x,G_x,\eta)$ is an orbifold chart as in \cref{def-orb-chart}. Then we have $\pi^{-1}(U_x)\cong\bigsqcup_{i\in I}U_i$ where each $U_i$ is homeomorphic to $U_x$. Now $G_x$ acts on $\pi^{-1}(U_x)$. Fix one component $U_i$ and consider the subgroup $G_i$ of $G_x$ mapping this component to itself. Then the induced map $\pi_{|\mcE|}:|\mcG\ltimes E|\lra|\mcG|$ in a neighborhood of $\pi_{|\mcE|}|_{U_i}^{-1}([x])$ looks like the map $\faktor{U_i}{G_i}\lra\faktor{U_x}{G_x}$ given by $\pi$ on the representatives.

Similarly we can look at vector bundles over orbifolds. So assume $\pi:E\lra\Ob\mcG$ is a vector bundle and consider an orbifold chart $(U_x,G_x,\eta)$ around a point $x\in\Ob\mcG$. Thus locally the map on orbit spaces $\pi_{|\mcE|}:|\mcG\ltimes E|\lra|\mcG|$ looks like $\faktor{\pi^{-1}(U_x)}{G_x}\lra \faktor{U_x}{G_x}$. In particular the fibre over $[x]$ is given by $\faktor{\pi^{-1}(x)}{G_x}$. Note that the fibre over $[x]$ thus do not need to be a vector space and may depend on the point $[x]$.

\index{Orbifold! Vector Bundle Morphism}

\begin{definition}
  An \emph{orbifold vector bundle morphism} $(\Phi,\phi):E\lra F$ between two orbifold vector bundles $E\lra\Ob\mcG$ and $F\lra\Ob\mcH$ is a vector bundle morphism
  \begin{equation*}
    \xymatrix{
      E \ar[d]^{\pi} \ar[r]^{\Phi} & F \ar[d]^{\pi} \\
      \Ob\mcG \ar[r]^{\phi_{\Ob}} & \Ob\mcH
      }
  \end{equation*}
  together with a homomorphism $\phi:\mcG\lra\mcH$ such that
  \begin{equation*}
    \xymatrix{
      \Mor\mcG{_s\times_{\pi}}E \ar[r]^-{\mu_E} \ar[d]_{\phi_{\Mor}\times\Phi} & E \ar[d]^{\Phi} \\
      \Mor\mcH{_s\times_{\pi}}F \ar[r]^-{\mu_F} & F \\
      }
  \end{equation*}
  commutes. An \emph{isomorphism} of orbifold vector bundles is an invertible orbifold vector bundle morphism.
  \label{def:orbibundle-morphism}
\end{definition}

\begin{rmk}
  Of course one can generalize this notion of orbifold vector bundle morphisms. It would be enough to require everything only up to Morita equivalence. However, we will only need this notion for homomorphisms. Also we can modify this definition for e.g.\ principal bundles by requiring the map $\Phi$ to be equivariant instead of linear. Furthermore note that these definitions also make sense for a stable groupoid which has only a topology on the object space.
\end{rmk}

Before we look at examples we will say a few words about bundles on non-effective orbifolds which might be such that they do not admit any local sections.

\begin{definition}
  Let $\mcG$ be an ep-Lie groupoid. Then we call
  \begin{enumerate}[label=(\roman*), ref=(\roman*)]
    \item a morphism $g\in\Mor\mcG$ \emph{ineffective} if there exists a neighborhood $V\subset\Mor\mcG$ with $g\in V$ such that $s(h)=t(h)$ for all $h\in V$. The set of ineffective morphisms is denoted by $\Mor_{\text{ineff}}\mcG$. Also
    \item we call a vector bundle $E\lra\Ob\mcG$ \emph{good} if $G_x\cap\Mor_{\text{ineff}}\mcG$ acts trivially on each fibre $E_x$ for all $x\in\Ob\mcG$.
  \end{enumerate}
  \label{def:good-vector-bundle}
\end{definition}

\index{Good Vector Bundle}

\begin{prop}
  A good vector bundle $E\lra\mcG$ has non-trivial sections in the sense that for any $e\in \pi^{-1}(p)$ there exists a local section $s$ such that $s(p)=e$. Also, any bundle over an effective ep-Lie groupoid is good. Furthermore the space of global sections is a vector space.
\end{prop}

\begin{proof}
  If $\mcG$ is effective then by definition $G_x\cap\Mor_{\text{ineff}}\mcG=\{\id_x\}$ which acts trivially on $E_x$ implying that the bundle $E\lra\mcG$ is good. Given two sections $\sigma_1,\sigma_2:\obj\mcG\lra E$ their sum $\sigma_1+\sigma_2$ and scalar multiples $\lambda\sigma_1$ are again invariant sections of $E$ and thus the space of sections forms a vector space.  

  Now use \cref{lem:existence-nice-orb-chart} to define an equivalent ep-Lie groupoid having a subset $\Ob\mcB\subset\Ob\mcG$ as objects. It is clear that it is enough to define a section $\ol{\sigma}:\Ob\mcG\lra E$ on the subset $\Ob\mcB$ as every other object is identified with some object in $\mcB$ by the equivalence and we can define $\ol{\sigma}(y)\coloneqq g\cdot\ol{\sigma}(x)$ for any $g:x\lra y$ with $x\in\Ob\mcB$. This is well defined because for another $z\in\Ob\mcB$ and $h:x\lra z$ we have $g\cdot\ol{\sigma}(x)=h\cdot\ol{\sigma}(z)$ because $h^{-1}g\cdot\ol{\sigma}(x)=\ol{\sigma}(x)$ as $h^{-1}g:x\lra x\in\Mor\mcB$. Also smoothness of the extension is clear as the action of $g:x\lra y$ on a neighborhood of $x\in\Ob\mcG$ is smooth.

  So it remains to construct an invariant section of $E\lra\mcB$ with prescribed value $e$ at $p\in\obj\mcB$. We will do this by averaging some arbitrary section of $E\lra\mcB$, hence why we passed to a sort of ``finite'' full subcategory. First choose a section $\sigma:\obj\mcB\lra E$ such that $\sigma(q)=e$ for all $q$ equivalent to $p$. This is possible as $E\lra\obj\mcB$ is a vector bundle over a manifold. Then define
  \begin{equation*}
    \ol{\sigma}(x)\coloneqq\frac{1}{|s^{-1}(x)|}\sum_{\substack{h\in\Mor\mcG \\ s(h)=x}}h^{-1}\cdot\sigma(t(h))
  \end{equation*}
  for any $x\in\obj\mcB$. As we have seen various times in \cref{sec:hurwitz-numbers} the prefactor $|s^{-1}(x)|$ takes care of the fact that the number of summands can jump. This map is indeed smooth and satisfies $\ol{\sigma}(p)=e$. Furthermore it is invariant as for $g:x\lra y$
  \begin{align*}
    g\cdot\ol{\sigma}(x) & = \frac{1}{|s^{-1}(x)|}\sum_{\substack{h\in\Mor\mcG \\ s(h)=x}}gh^{-1}\cdot\sigma(t(h)) \\
    & = \frac{1}{|s^{-1}(y)|}\sum_{\substack{k\in\Mor\mcG \\ s(k)=y}}k^{-1}\cdot\sigma(t(k)) \\
    & = \ol{\sigma}(y),
  \end{align*}
  where we have used that $g\in G_x\cap\Mor_{\text{ineff}}$ acts as identity and that precomposition with $g^{-1}$ induces a bijection $s^{-1}(x)\lra s^{-1}(y)$.
\end{proof}

If a bundle is not good and we have a morphism $g\in\Mor(x,x)$ which acts non-trivially on the fibre $E_x$ we obtain a condition on the sections, namely $g\cdot\sigma(x)=\sigma(x)$. If $g$ does not act as the identity there is a vector in $E_x$ through which there is no local section.

We will now discuss pull-backs of bundles.

\begin{definition}
  Let $E$ be a orbifold vector bundle over $\mcH$ and $\phi:\mcG\lra \mcH$ a homomorphism. Then we can define the pull-back bundle $\phi^*E$ over $\mcG$ in the following way. Define the vector bundle $\phi^*E$ as the pull-back under $\phi:\Ob\mcG\lra\Ob\mcH$ with the projection as the anchor map. Now define the action of $\mcG$ on $\phi^*E$ by
  \begin{align*}
    \mu: \Mor\mcG{}_s\times_{\pi}\phi^*E & \lra \phi^*E \\
    (g,(x,e)) & \longmapsto (t(g),\phi(g)\cdot e)
  \end{align*}
  where $g\in\Mor\mcG$ such that $s(g)=x$ and $e\in E_{\phi(x)}$ and $\phi(g)\cdot e$ means the action of $\Mor\mcH$ on $E$.
  \label{def:pull-back-orbibundle}
\end{definition}

\begin{rmk}
  Given an equivalence $\epsilon:\mcC\lra\mcD$ one can in fact define a pushforward $\epsilon_*E$ of an orbifold vector bundle or covering $E\lra\Ob\mcC$ to $\mcD$ which in turn allows one to pull-back bundles along strong maps $\mcC\overset{\epsilon}{\longleftarrow}\mcG\overset{\phi}{\lra}\mcD$ and consequently also along Morita equivalences. However, these definitions are somewhat tricky and induce category equivalences instead of category isomorphisms of vector bundles over $\mcC$ and $\mcD$ only. All the maps in later applications will in fact be given as homomorphisms so we will not need these details.
  
  Also note that some sets of morphisms act on local sections of $E$ in the following way. Given some open set $U\subset\Ob\mcG$ we can look at smooth sections $\varphi:U\lra s^{-1}(U)\cap t^{-1}(U)$ and define
  \begin{align*}
    \Gamma(U,E) & \lra \Gamma(\varphi(U),E) \\
    s & \mapsto (x\mapsto\varphi(x)\cdot s(x))
  \end{align*}
  which is nothing but a local version of the pointwise \cref{def:g-space}.
\end{rmk}

\begin{example}
  \begin{enumerate}[label=(\roman*), ref=(\roman*)]
    \item Consider the bundle $\ts\Ob\mcG\lra\Ob\mcG$. This has the action
      \begin{align*}
        \mu:\Mor\mcG {_s\times_{\pi}}\ts\Ob\mcG & \lra\ts\Ob\mcG \\
        (g,X)&\longmapsto \dd_x\varphi_g\cdot X
      \end{align*}
      where $\varphi_g:U\lra U$ is the unique smooth action of $s^{-1}(U)\cap t^{-1}(U)\subset\Mor\mcG$ on a sufficiently small neighborhood $U\subset\Ob\mcG$ of $x=s(g)=\pi(x)$ with $\varphi_g(x)=g\cdot x\in U$. This is clearly a bundle and if $\mcG$ is effective then it is even a good bundle. Sections of this bundle are sections $s:\Ob\mcG\lra\ts\Ob\mcG$, i.e.\ vector fields on the object manifold, that are invariant under the morphism action. This is of course the intuitive definition of a vector field on an orbifold.
    \item Of course we can dualize the last example to obtain the cotangent bundle
      \begin{equation*}
        \ts^*\Ob\mcG\lra\Ob\mcG
      \end{equation*}
      which is an orbibundle via the action
      \begin{align*}
        \mu:\Mor\mcG {_s\times_{\pi}}\ts^*\Ob\mcG & \lra\ts^*\Ob\mcG \\
        (g,\omega)&\longmapsto \left(X\mapsto\omega(\dd_y\varphi_g^{-1}\cdot X)\right)
      \end{align*}
      where $X\in \ts^*_y\Ob\mcG$. Again, a section of such a bundle is a $1$-form on $\Ob\mcG$ which is invariant under the local action of the morphisms via pullback. This can be stated slightly differently as we will see in the next section.
    \end{enumerate}
\end{example}

\section{Differential Forms on Orbifolds}

% Definition, Integration, Existence

Differential forms on orbifolds are essentially equivariant differential forms on the local orbifold charts. In the ep-Lie groupoid setting this can be formulated as follows.

\begin{definition}
  Let $\mcG$ be a ep-Lie groupoid. Then we define the de-Rham complex of differential forms on $\mcG$ by $\Omega^*(\mcG):=\{\alpha\in\Omega^*(\Ob\mcG)\mid s^*\alpha=t^*\alpha\}$ with the usual exterior differential.
\end{definition}

In order to define integration of differential forms we need to choose a locally finite covering with some sort of partition of unity in order to ``localize'' the computations.

\begin{definition}
  Let $\mcG$ be an ep-Lie groupoid. Then we call a covering $\{U_i\}_{i\in I}$ \emph{locally finite} if
  \begin{enumerate}[label=(\roman*), ref=(\roman*)]
    \item the sets $U_i\subset \Ob\mcG$ are open,
    \item the union of their images on the orbit space covers the whole space, i.e.\ $|\mcG|=\bigcup_{i\in I}|U_i|$, and
    \item for every $x\in|\mcG|$ there exist only finitely many $i\in I$ such that $x\in|U_i|$.
  \end{enumerate}
  Given a locally finite covering such that the sets $U_i$ are orbifold charts together with some automorphism group $G_i$ we define a \emph{partition of unity} subordinate to $\{U_i\}_{i\in I}$ as a family of smooth real-valued functions $f_i:U_i\lra\RR$ for $i\in I$ such that\footnote{Notice that these partition functions are defined on the object set! It is of course possible to define these function on the orbit space as well.}
  \begin{enumerate}[label=(\roman*), ref=(\roman*)]
    \item $f_i$ is $G_i$-equivariant such that it factors through the open set $\ol{f}_i:\faktor{U_i}{G_i}=|U_i|\lra\RR$ in the quotient space,
    \item the function $f_i:U_i\lra\RR$ has compact support, and
    \item $\sum_{\substack{i\in I\\ [x]\in |U_i|}}\ol{f}_i([x])=1$ for all $[x]\in|\mcG|$.
  \end{enumerate}
\end{definition}

\begin{rmk}
  Let us note a few observations.
  \begin{enumerate}[label=(\roman*), ref=(\roman*))]
    \item A locally finite covering such that the $(U_i,G_i,\pi_i)$ are orbifold charts is in particular an atlas with the local finiteness condition. Recall that for us orbifold charts are always centered around points $x\in \Ob U_i$ such that $G_i=\Aut_{\mcG}(x)$.
    \item There always exists a locally finite covering via orbifold charts on an ep-Lie groupoid. As the object set is locally compact the quotient space is it, too. Thus we can start by choosing a cover via orbifold charts around all points and then choosing a locally finite subcover.
    \item Also there always exists a partition of unity subordinate to such a locally finite cover by orbifold charts as the quotient is also second countable in addition to Hausdorff and thus paracompact. So choose a continuous partition of unity $\{\ol{f}_i\}$ subordinate to the cover of $|\mcG|$ by the $|U_i|$ for $i\in I$ and pull back $\ol{f}_i$ to $U_i$. Now make the pulled-back function smooth on $U_i$ in a $G_i$-equivariant way.
    \item As usual we need the local finiteness condition in order to have only finite sums when defining integrals. Note that an equivalence class in $|U_i|$ can still have more than one representative in $U_i$ as the central point can have automorphisms so we will need to divide by the size of this automorphism group when defining integrals.
  \end{enumerate}
\end{rmk}

%Integration of differential forms is defined as follows. Cover the orbit space $|\mcG|$ by a locally finite atlas of orbifold charts $(U_i,G_i,\phi_i)_{i\in I}$ with $\phi_i:U_i\lra \faktor{U_i}{G_i}$ and $U_i\subset\Ob\mcG$ as well as that $\faktor{U_i}{G_i}$ is homeomorphic to its image in $|\mcG|$\footnote{These images are the actual coordinate neighborhoods on $|\mcG|$ which we require to cover all of $|\mcG|$ and such that at each point there are only finitely many neighborhoods.} together with a partition of unity $(f_i)_{i\in I}$ subordinate to the atlas of real valued functions $f_i:\faktor{U_i}{G_i}\lra\RR$. Note that this construction is related to the proof of \cref{lem:existence-nice-orb-chart}.

\begin{definition}
  Given a locally finite covering of a $n$-dimensional orbifold groupoid $\mcG$ via orbifold charts $(U_i,G_i,\phi_i)_{i\in I}$ together with a partition of unity $(f_i)_{i\in I}$ subordinate to this atlas, define the integral of an $n$-form $\alpha\in\Omega^n(\mcG)$ over $\mcG$ by
  \begin{equation*}
    \int_{\mcG}\alpha:=\sum_{i\in I}\frac{1}{|G_i|}\int_{U_i}f_i\cdot\alpha|_{U_i}.
  \end{equation*}
\end{definition}

Note that this might not be finite as we did not require $|\mcG|$ to be compact. As usual, this is independent of the choices as is explained e.g.\ in \cite{adem_orbifolds_2007}. Furthermore this is independent of the representative in the Morita equivalence class by pulling back the orbifold charts and partition of unity to the common ep-Lie groupoid via the equivalences and comparing the integrals there. For this to make sense we need a notion of pulling back differential forms on $\mcG$ via a homomorphism $\phi:\mcG\lra\mcH$. This is defined as pulling back $\alpha\in\Omega^*\mcH$ via $\phi:\Ob\mcG\lra\mcH$ and noting that $\phi_{\Ob}\circ s = s\circ\phi_{\Mor}$ and similarly for the target map implying that $s^*\phi_{\Ob}^*\alpha=\phi_{\Mor}^* s^*\alpha=\phi_{\Mor}^* t^*\alpha=t^*\phi_{\Ob}^*\alpha$.

\section{Morphism Coverings}

We will need one more notion of a map between orbifolds, so let $\mcG$ and $\mcH$ be orbifold groupoids.

\begin{definition}
  A homomorphism $\phi:\mcG\lra \mcH$ is called a \emph{morphism covering} if it is a covering and local diffeomorphism on objects as well as morphisms (not necessarily of the same degree) such that the \emph{lifting property} holds:

  \begin{center}
    \parbox{0.9\textwidth}{%
      For all $x\in\Ob\mcG$   and all $h\in\Mor\mcH$ such that $\phi(x)=s(h)$ there exists a $g\in\Mor\mcG$ such that $\phi(g)=h$ and $s(g)=x$.}
 \end{center}
\end{definition}

\index{Morphism Covering}
\index{Lifting Property}

\begin{figure}[!ht]
  \centering
  \begin{subfigure}[b]{0.4\textwidth}
    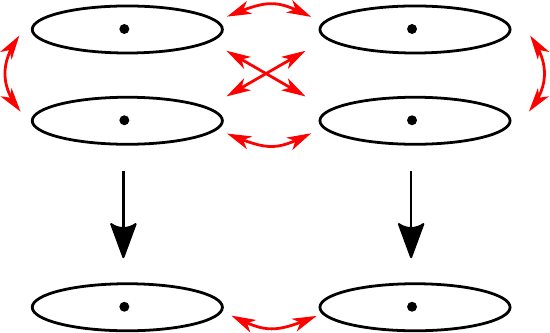
    \caption{}
    \label{fig:examples-morphism-covering-b}
  \end{subfigure}
  \hfill
  \begin{subfigure}[b]{0.4\textwidth}
    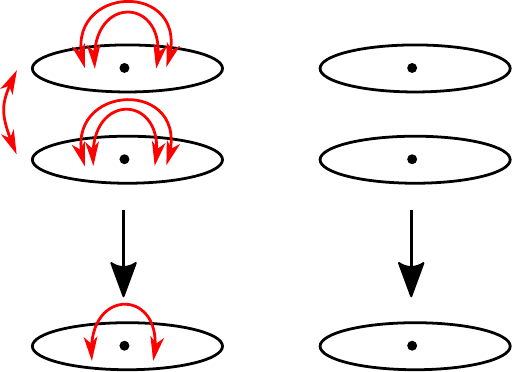
    \caption{}
    \label{fig:examples-morphism-covering-f}
  \end{subfigure}
  \\
  \begin{subfigure}[b]{0.8\textwidth}
    \centering
    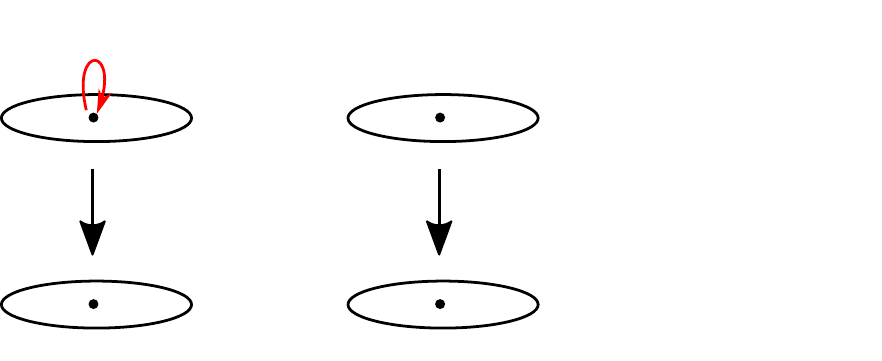
    \caption{}
    \label{fig:examples-morphism-covering-g}
  \end{subfigure}
  \\
  \begin{subfigure}[b]{0.79\textwidth}
    \centering
    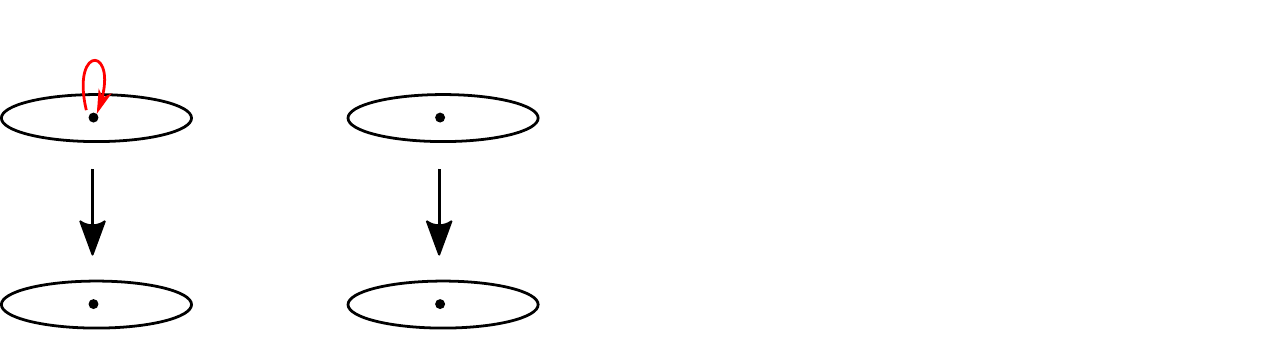
    \caption{}
    \label{fig:examples-morphism-covering-h}
  \end{subfigure}
  \caption{This figure comprises some examples of morphism coverings. Morphisms are drawn in red. Note that we did not draw the identity morphisms always explicitly, also arrows indicate isomorphisms of course. \cref{fig:examples-morphism-covering-b} is in fact an equivalence as well as a morphism covering. In \cref{fig:examples-morphism-covering-f} one sees an example of a morphism covering whose target quotient space is not connected and thus its degree is not well defined -- the left hand side has degree $\frac{1}{2}$ and the right hand side $2$. \cref{fig:examples-morphism-covering-g} and \cref{fig:examples-morphism-covering-h} are both morphism coverings of degree $\frac{1}{2}$. On the left side of the dashed line one sees a ``summary'' of the picture and on the right hand side one can see the objects and morphisms separately. The horizontal arrows refer to morphisms reflecting the disc at the central point giving rise to one automorphism of the central point. Note that in \cref{fig:examples-morphism-covering-h} the degree on the morphism set is two because we can compose the automorphism on the source with the reflection. Also note that the morphism discs already contain the inverse morphisms in form of the opposite point on that same disc.}
  \label{fig:examples-morphism-covering}
\end{figure}

\begin{figure}[!ht]
  \centering
  \begin{subfigure}[b]{0.4\textwidth}
    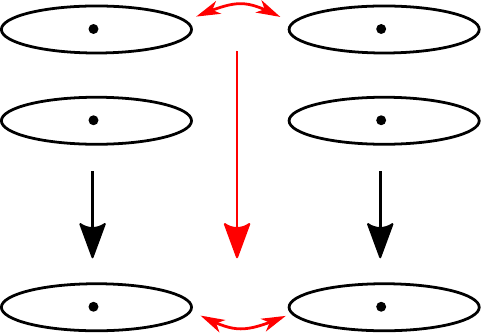
    \caption{}
    \label{fig:examples-morphism-covering-a}
  \end{subfigure}
  \hfill
  \begin{subfigure}[b]{0.4\textwidth}
    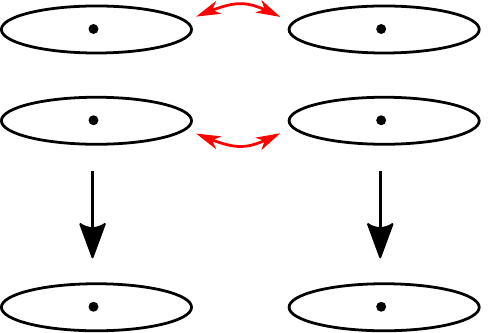
    \caption{}
    \label{fig:examples-morphism-covering-d}
  \end{subfigure}
  \\
  \begin{subfigure}[b]{0.4\textwidth}
    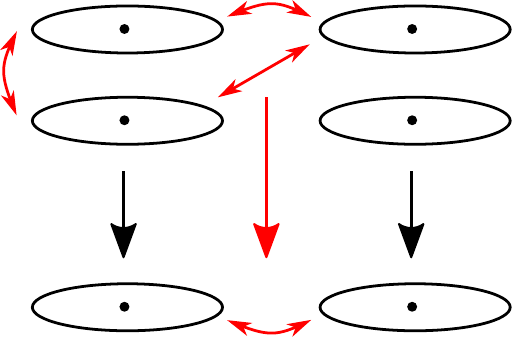
    \caption{}
    \label{fig:examples-morphism-covering-e}
  \end{subfigure}
  % \hfill
  % \begin{subfigure}[b]{0.4\textwidth}
  %   \input{examples-morphism-covering-c.pdf_tex}
  %   \caption{}
  %   \label{fig:examples-morphism-covering-c}
  % \end{subfigure}
  \caption{These are non-examples of morphism coverings. Note that we did not draw the identity morphisms explicitly, also arrows indicate isomorphisms. \cref{fig:examples-morphism-covering-a} and \cref{fig:examples-morphism-covering-e} are not morphism coverings because they do not satisfy the lifting property. \cref{fig:examples-morphism-covering-d} is not a functor. % \cref{fig:examples-morphism-covering-c} has a well-defined degree but is not surjective on automorphism groups.
}
  \label{fig:non-examples-morphism-covering}
\end{figure}

\begin{rmk}
  Note that a morphism covering is in fact different from an orbifold covering because in the orbifold covering case the automorphism groups of the preimages in the ``total'' space are subgroups of the automorphism group of the base point. This means that there are fewer automorphisms on the total space than on the base. In particular, the automorphism group of a point in the fibre over a smooth point (which are dense if the base orbifold is effective) is automatically trivial. In the case of the evaluation map for Hurwitz covers it is clear that a Hurwitz cover can have automorphisms although the target surface has none. Therefore we can not use the notion of an orbifold covering.
\end{rmk}

For morphism coverings there is a well-defined notion of a covering degree of $\phi$. Before defining this notion and proving its well-definedness we will prove a lemma about the local picture of a morphism covering as well as the existence of a particular kind of charts adapted to morphism coverings.

\begin{definition}
  A \emph{compatible pair of atlases and partitions} for the morphism covering $\phi:\mcG\lra\mcH$ consists of a locally finite atlas of orbifold charts $\{(U_i,H_i,\pi_i)\}_{i\in I}$ for $\mcH$, a locally finite atlas of orbifold charts $\{(V_i^j,G_i^j,\rho_i^j)\}_{i\in I,j\in J_i}$ on $\mcG$, a partition of unity subordinate to the atlas on $\mcH$ denoted by $\{f_i\}_{i\in I}$ with $f_i:U_i\lra\faktor{U_i}{H_i}\lra\RR$ such that
  \begin{enumerate}[label=(\roman*), ref=(\roman*)]
    \item $\bigsqcup_{j\in J_i}V_i^j\subset \phi^{-1}(U_i)$ and $\phi|_{V_i^j}:V_i^j\lra U_i$ is a diffeomorphism,
    \item $\bigsqcup_{j\in J_i}G_i^j\subset \phi^{-1}(H_i)$ and $\phi|_{G_i^j}:G_i^j\lra H_i$ is a diffeomorphism and
    \item $\{f_i\circ \phi:V_i^j\lra\RR\}_{i\in I,j\in J_i}$ is a partition of unity subordinate to the atlas on $\mcG$.
  \end{enumerate}
\end{definition}

\begin{rmk}
  Note that this means in particular that the $V_i^j$ cover all equivalence classes in $|\mcG|$ but because of the local finality condition we do not take all of the preimages of $U_i$ under $\phi$ but rather a finite subset and then use as $G_i^j$ all the necessary morphisms of these $V_i^j$.

  The third condition plays a rather subtle role which is merely an artifact of how things are set up. Consider the following situation: We have an actual covering between smooth manifolds $\mcG$ and $\mcH$, such as $\Ob\mcH=\DD$ with $\Mor\mcH=\DD$ representing the identity morphisms. On $\mcG$ we have three discs $\DD_1,\DD_2$ and $\DD_3$ together with three discs representing the identity morphisms and two discs identifying two of the discs. This is illustrated in \cref{fig:example-property-3}. Now every element in $|\mcH|$ has two equivalence classes in $|\mcG|$ as preimage. However, we could choose all three discs in $\Ob\mcG$ as a locally finite atlas of orbifold charts which would satisfy all the conditions. But if we now pull back the constant $1$-function as a partition of unity on $\mcH$ we have two functions with value $1$ on one of the charts in $\mcG$ which doesn't add up to one anymore. In order to make this impossible we require the third condition which forced us to pick only one such orbifold chart on $\mcG$ as we will see in the proof of \cref{lem:existence-comp-atlas-mor-cover}.
\end{rmk}

\begin{figure}[!ht]
  \centering
  \def\svgwidth{0.5\textwidth}
  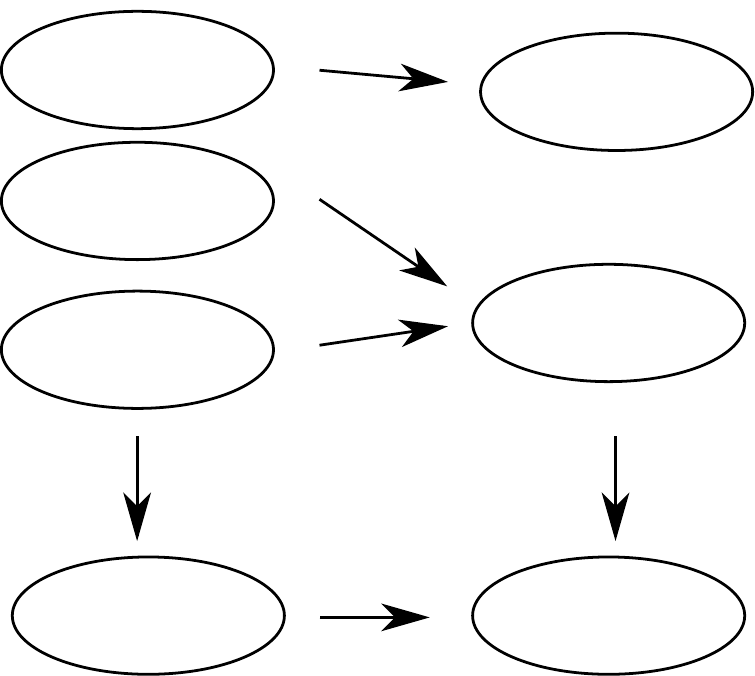
  \caption{On the upper half we see the category $\mcG$ with objects left and orbit space right. The two vertical arrows identify the two discs in $\Ob\mcG$. The horizontal arrows are the projection maps $\Ob\mcG\lra|\mcG|$ and $\Ob\mcH\lra|\mcH|$. Pulling back the constant $1$-function on $\Ob\mcH$ gives us two charts for the same disc in $|\mcG|$ which each are equipped with the constant $1$-function. The pulled-back partition of unity is thus no longer a partition of unity.}
  \label{fig:example-property-3}
\end{figure}

\begin{lem}
  For any morphism covering $\phi:\mcG\lra\mcH$ between compact ep-Lie groupoids there exists a compatible pair of atlases and partitions.
  \label{lem:existence-comp-atlas-mor-cover}
\end{lem}

\begin{proof}
  Notice from \cref{lem-existence-good-neighborhoods} and in particular its proof in \cite{adem_orbifolds_2007} that given a neighborhood $U\subset\Ob\mcH$ around a point $x\in\Ob\mcH$ we can find an orbifold chart $(V_x,G_x,\phi_x)$ around the point $[x]\in|\mcH|$  with $V_x\subset U$ by restricting the neighborhoods for the local diffeomorphism $s$ in the construction to $U$. Thus we can find orbifold charts contained in some given neighborhood.

  So first, for any $x\in\Ob\mcH$ choose a neighborhood such that its preimages under $\phi$ are open neighborhoods of $\Ob\mcG$ diffeomorphic to the one around $x$. Then choose an orbifold chart contained in this neighborhood around $x$ and denote it by $(U_x,G_x,\phi_x)$. Now we need to pick appropriate points $x$ in order to ``cover everything''.

  Note that because of the lifting condition in the definition of morphism coverings it is enough to consider $\phi^{-1}(U_x)$ in order to have representatives for all preimages in $\phi^{-1}([x])$. This is because of the following. It is clear that $\bigcup_{x'\sim x}\phi^{-1}(x')$ covers $\phi^{-1}([x])$. But if $h:x\lra x'$ is a morphism in $\mcH$ then for any $y\in\Ob\mcG$ such that $\phi(y)=x$ we have a morphism $g:y\lra y'$ such that $\phi(g)=h$ and thus $\phi(y')=x'$ as well as $y\sim y'$, i.e.\ $[y]=[y']$. This means that any element in $\phi^{-1}([x])$ has a representative $y\in \phi^{-1}(x)$.
  
  \begin{figure}[!ht]
    \centering
    \def\svgwidth{0.7\textwidth}
    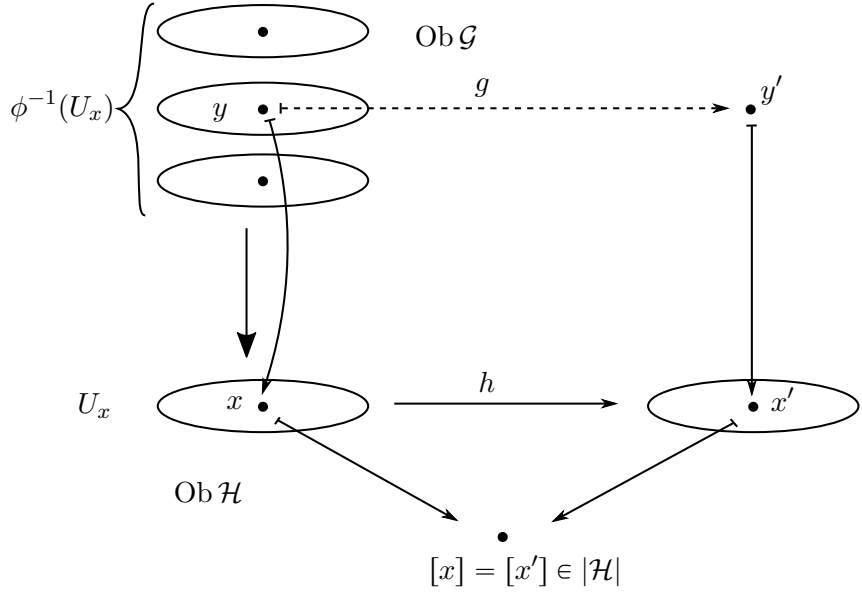
    \caption{This figure explains how the lifting property in the definition of morphism covering is used to find representatives of every class in $\phi^{-1}([x])$ by only looking at the fibre $\phi^{-1}(x)$.}
    \label{fig:lifting-property-morphism-covering}
  \end{figure}

  Now pick a (possibly infinite) subset $I\subset \Ob\mcH$ such that these $\{(U_i,H_i,\pi_i)\}_{i\in I}$ form a locally finite atlas for $\mcH$. Then pick a partition of unity subordinate to this covering. This way we have found a locally finite atlas of $\mcH$ such that preimages of the $U_i\subset\Ob\mcH$ under $\phi$ contain representatives of all classes in $|\mcH|$. Furthermore we can make the orbifold charts $(U_i,H_i,\pi_i)$ smaller such that the source and target map in $\mcG$ are diffeomorphisms on the connected components of $\phi^{-1}(U_i)$. This way we can make sure that we obtain actual orbifold charts from $\phi^{-1}(U_i)$.

  Now we can look at the connected components of the preimages $\phi^{-1}(U_x)\subset\Ob\mcG$ which are by construction diffeomorphic to $U_x$ via $\phi$. Since $\phi$ could have infinite degree on $\Ob\mcG$ we need to pick a finite subset of these such that every equivalence class in $\phi^{-1}([x])\in|\mcG|$ has exactly one representative. This is possible because $|\mcG|$ is compact and thus $\phi:|\mcG|\lra|\mcH|$ is proper meaning that $\phi^{-1}([x])$ is finite. Furthermore as we have seen above every element in $\phi^{-1}([x])$ does have a representative in $\phi^{-1}(U_x)$ by the lifting condition. Therefore we can choose finitely many connected components covering $|\phi^{-1}(U_x)|$.

  Doing this we obtain a finite set of points $y$ in $\phi^{-1}(x)\in\Ob\mcG$ for every $x\in I$ with neighborhoods $V_y\subset\Ob\mcG$ defined as connected components of $\phi^{-1}(U_x)$ such that $y\in V_y$. Their projection to $|\mcG|$ covers the whole space by construction. We denote the index set for the neighborhoods in the fibre over $i\in I$ by $J_i$. So we have now a covering of $|\mcG|$ by the open sets $V_i^j\subset\Ob\mcG$ where the index $j\in J_i$ runs over the preimages of $i$. By construction
  \begin{equation*}
    \bigsqcup_{j\in J_i}V_i^j\subset \phi^{-1}(U_i).
  \end{equation*}
  
  \begin{figure}[!ht]
    \centering
    \def\svgwidth{0.7\textwidth}
    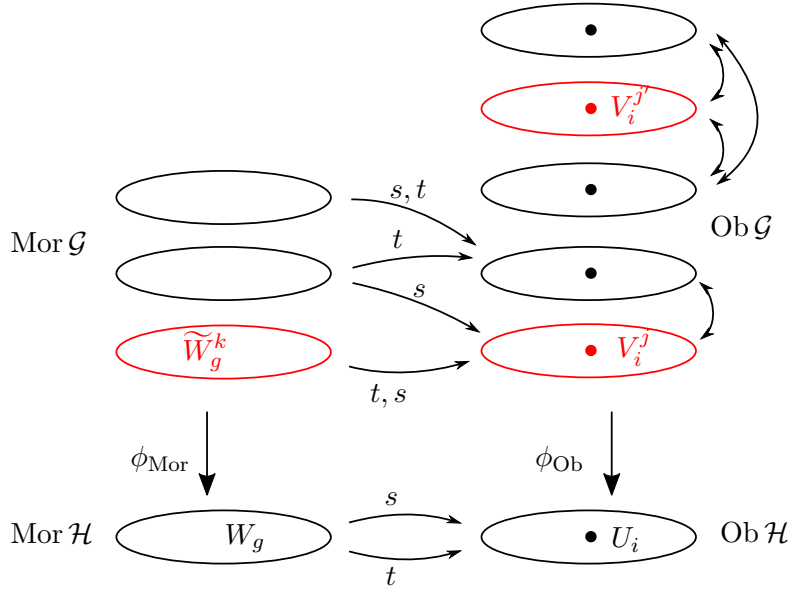
    \caption{On the right hand side one can see the preimage $\phi^{-1}(U_i)\subset\Ob\mcG$ for which we choose a connected component as a representative for every class. On the left hand side we pick one $g\in H_i$ together with its neighborhood $W_g\subset\Mor\mcH$ and its lifts $\wt{W}^k_g$. These sets are morphisms relating the various connected components in $\phi^{-1}(U_i)$ and we pick only those ones having source and target in the chosen $V_i^j$. An example choice is marked in red. Note that this is nothing but the full subcategory generated by the red objects.}
    \label{fig:illustration-choice-compatible-atlas}
  \end{figure}

  Next we need to understand what happens with the morphisms. Let $W_g\subset\Mor\mcH$ be the connected component of $s^{-1}(U_x)\cap t^{-1}(U_x)$ containing $g\in H_x$ for an arbitrary $x\in\Ob\mcH$. Then its preimage under $\phi:\Mor\mcG\lra\Mor\mcH$ consists of a disjoint union of open sets $\bigsqcup_k\wt{W}_g^k\subset\Mor\mcG$. Recall that the $V_i^j, U_i, W_g$ and $\wt{W}_g^k$ were chosen such that $s:\wt{W}_g^k\lra V_i^j$, $\phi_{\Ob}:V_i^j\lra U_i$ and $s:W_g\lra U_i$ are all diffeomorphisms. Because $\phi$ is a homomorphism we have that
  \begin{align*}
    \phi_{\Ob}\circ s & = s\circ\phi_{\Mor},\\
    \phi_{\Ob}\circ t & = t\circ\phi_{\Mor}
  \end{align*}
  which implies that all morphisms in $\wt{W}_g^k$ have source and target in $\phi^{-1}(U_x)$. Furthermore we see that on $\wt{W}_g^k$ the map $\phi_{\Mor}=s^{-1}\circ\phi_{\Ob}\circ s$ restricts to a diffeomorphism on every $\wt{W}_g^k\lra W_g$. By continuity the source of all morphisms in $\wt{W}_g^k$ is contained in the same component of $\phi^{-1}(U_x)$. The same holds for the target but it might be a different component. For every $g\in H_x$ we will only keep those preimages $\wt{W}_g^k$ which have source and target contained in the same $V_i^j$. This means that we throw away all but finitely many connected components. Then we define
  \begin{equation*}
    G_i^j\coloneqq \bigsqcup_{g\in G_i}\bigsqcup_{\substack{k\text{ s.t.}\\ s(\wt{W}_g^k),t(\wt{W}_g^k)\subset V_i^j}}\wt{W}_g^k.
  \end{equation*}
  We know by construction that $G_i^j$ acts on $V_i^j$ but it remains to see that these are all the morphisms acting on this set. But this is clear as $g\in\Mor\mcG$ with $s(g),t(g)\in V_i^j$ satisfies $s(\phi(g)),t(\phi(g))\in U_i$ and thus there exists a $h\in H_i$ such that $\phi(g)=h$ and thus $\phi(g)\in W_g$.

  It remains to verify the statement for the partition functions. For this purpose we need to figure out how many orbifold charts $V_i^j$ we constructed around a given point $[y]\in\Ob\mcG$. So suppose $\phi([y])=[x]$ and the orbifold atlas $(U_i,\phi_i,H_i)_{i\in i}$ contains charts $U_1,\ldots,U_k$ with $x_i\in U_i$ for $i=1,\ldots,k$ such that $[x_1]=\cdots=[x_k]=[x]$ as the only charts around $[x]$. Then every $U_i$ has preimages $V_i^j\subset\Ob\mcG$ for $j\in J_i$. They were chosen such that no $V_i^{j_1}$ and $V_i^{j_2}$ are identified for $j_1\neq j_2$. Thus there can be at most one $V_i^j$ containing $[y]$ in its quotient. However, there also needs to exist at least one neighborhood containing a representative of $[y]$ as we have seen earlier. Thus the only partition functions that have support in $[y]\in|\mcG|$ are these $V_i^j$ for $i=1,\ldots,k$ and $j$ uniquely determined by $i$. Denote the preimages of $x_i$ in $V_i^j$ by $y_i$. On $V_i^j$ we use the function $f_i\circ\phi$. Thus we have
  \begin{equation*}
    \sum_{i=1}^kf_i\circ\phi([y])=\sum_{i=1}^kf_i(\phi(y_i))=\sum_{i=1}^kf_i(x_i)=\sum_{i=1}^kf_i([x])=1
  \end{equation*}
  where we abused notation by denoting the functions on objects as well as on the quotient by the same name $f_i$.
\end{proof}

\begin{rmk}
  Note that we used the assumption of compactness in \cref{lem:existence-comp-atlas-mor-cover} only because we need that there are only finitely many preimages $\phi^{-1}([x])\in|\mcG|$ as it is possible for infinitely many preimages that some neighborhoods that we choose become smaller and we end up with just a point in the intersection. However, we will apply this Lemma to the evaluation functor between moduli spaces of Hurwitz covers and Deligne--Mumford spaces which will only be a morphism covering on a dense open subset. But it is clear that we can still apply the Lemma if we know that every equivalence class has only finitely many preimages which will be obvious for the evaluation functor.
\end{rmk}

\begin{lem}
  Let $\phi:\mcG\lra\mcH$ be a morphism covering between compact ep-Lie groupoids. Furthermore let $[x]\in |\mcH|$ together with representatives $y_1,\ldots,y_k$ of all preimages in $\phi^{-1}([x])$ be given. Then there exist neighborhoods $\mcU$ of $[x]\in|\mcH|$ and $U_1,\ldots,U_k\subset\Ob\mcG$ of $y_1,\ldots,y_k$, respectively, which are small enough such that for any $[x']\in\mcU$ all preimages $[y'_1],\ldots,[y'_m]\in |\mcG|$ of $[x']$ can be represented by elements in $U_1,\ldots,U_k$.
  \label{lem:morphism-coverings-comp-pairs-prop}
\end{lem}

\begin{proof}
  This is because $\phi_{\Ob}:\Ob\mcG\lra\Ob\mcH$ is a covering and $\Ob\mcH\lra|\mcH|$ is open. It was essentially proven in the proof of \cref{lem:existence-comp-atlas-mor-cover}. Choose the open neighborhood $\mcU$ around $x\in\Ob\mcH$ small enough such that $\phi_{\Ob}$ gives a diffeomorphism $U_i\lra \mcU$ for every $i\in I$ and $\phi_{\Ob}^{-1}(\mcU)\cong \bigsqcup_{i\in i}U_i$. Then the sets $U_i$ contain representatives of all preimages of $[x']\in |\mcU|$. Now pick $k$ of these sets containing representatives $y_1,\ldots,y_k$ of $|\phi|^{-1}([x])$. Next we make these open subsets $\mcU,U_1,\ldots,U_k$ smaller such that source and target maps in $\mcG$ and $\mcH$, respectively, define diffeomorphisms on them. Now for some $|x'|\in|\mcU|$ a representative of a preimage $z\in\Ob\mcG$ under $|\phi|$ satisfies $\phi_{\Ob}(z)=x'\in\mcU$. Therefore $z\in U_m$ for some $m\in I$. But since $U_1,\ldots,U_k$ contain representatives of all preimages of $x$ in $\Ob\mcG$ there must exist a morphism $g$ with $s(g)\in U_m$ and $t(g)=y_i$ for some $i\in\{1,\ldots,k\}$. As the source and target maps are diffeomorphisms by construction this morphism $g$ extends to a morphism $g'\in\Mor\mcG$ with $s(g')=z$ and $t(g')\in U_i$ proving the lemma. See also \cref{fig:preimages-morphism-covering}.
\end{proof}

\begin{figure}[!ht]
  \centering
  \def\svgwidth{0.5\textwidth}
  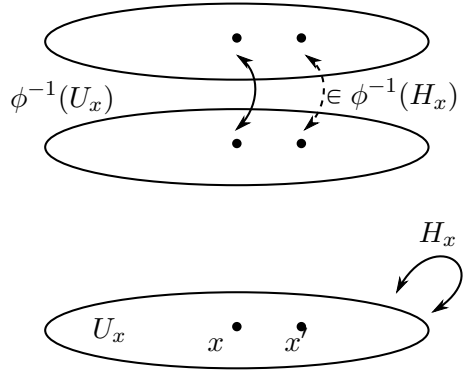
  \caption{In the picture one sees a morphism $g:y\lra z$ between two preimages of $x$. Due to the fact that $s,t$ and $\phi$ are local diffeomorphisms this $g$ extends to a morphism between two preimages of a point $x'$ sufficiently close to $x$.}
  \label{fig:preimages-morphism-covering}
\end{figure}

\begin{prop}
  If $\phi:\mcG\lra\mcH$ is a morphism covering between compact ep-Lie groupoids and $|\mcH|$ is connected then the number
  \begin{equation*}
    \deg \phi := |H_x|\sum_{[y]\in \phi^{-1}([x])}\frac{1}{|G_y|},
  \end{equation*}
  where $|G_y|$ denotes the number of automorphisms of $y$, is independent of the point $x\in\Ob\mcG$. We call this number the \emph{degree} of the morphism covering.
  \label{prop:morphism-coverings-deg}
\end{prop}

\begin{proof}
  Consider a point $[x]\in|\mcH|$. From \cref{lem:existence-comp-atlas-mor-cover} and \cref{lem:morphism-coverings-comp-pairs-prop} we see that we can choose a neighborhood $U_x\subset \Ob\mcH$ together with a group $H_x=\Aut_{\mcH}(x)$ and disjoint $(V_1,G_1),\ldots,(V_k,G_k)$ in $\mcG$ such that $\phi^{-1}([x])=\{[y_1],\ldots,[y_k]\}$ for $y_i\in V_i$ and $G_i=G_{y_i}$ as part of a compatible pair of atlases. Furthermore we can assume that $\phi$ restricts on objects to a diffeomorphism on each component $V_i\lra U_x$ and that for a point $[x']\in|\mcH|$ all preimages in $|\mcH|$ have representatives in some $V_i$ for $i=1,\ldots,k$. This situation as well as the nomenclature is illustrated in \cref{fig:nbhd-charts-morphism-cover}.

  \begin{figure}[!ht]
    \centering
    \def\svgwidth{0.5\textwidth}
    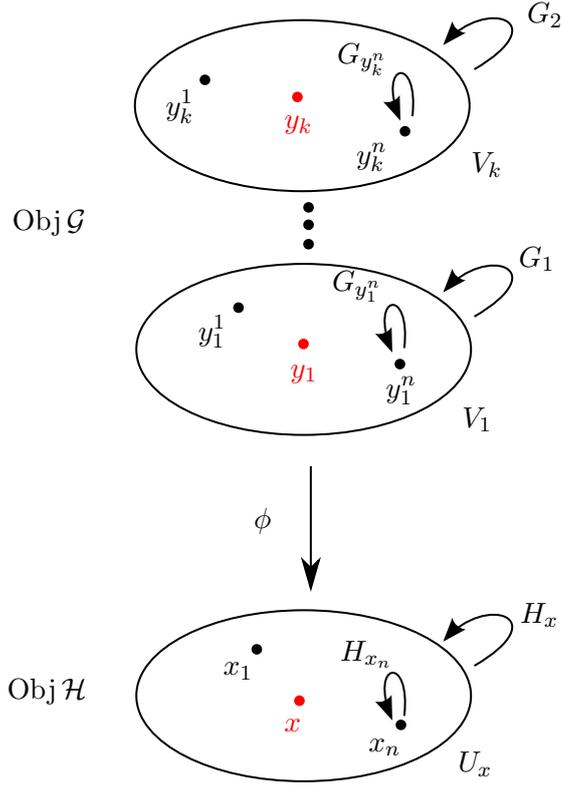
    \caption{The charts $(U_x,G_x)$ and $(V_1,G_1),\ldots,(V_k,G_k)$ are centered around the representatives of the preimages of $[x]$ under $\phi$. We denote the representatives of $[x']$ in $U_x$ by $x_1,\ldots,x_n$ and using the covering property of $\phi$ we denote the $n$ preimages of these points in $V_i$ by $y_i^1,\ldots,y_i^n$ and their respective automorphism groups by $G_{y_i^j}$.}
    \label{fig:nbhd-charts-morphism-cover}
  \end{figure}

  Using \cref{lem:morphism-coverings-comp-pairs-prop} we see that the $y_i^j$ cover all elements in $\phi^{-1}([x'])$. Furthermore by construction of the charts there do not exist any morphisms between the elements of $V_i$ and $V_j$ for $i\neq j$. However, it is possible that various elements from $\{y_i^1,\ldots,y_i^n\}$ are identified. Again by construction of the neighborhoods this identification comes the group action $G_i$. We can therefore write
  
  \begin{equation*}
    \{y_j^1,\ldots,y_j^n\}=I_j^1\sqcup\ldots\sqcup I_j^{n_j}
  \end{equation*}

  for every $j=1,\ldots,k$ where the $\{I_j^i\}_{i=1}^n$ are the $n_j\in\NN$ orbits of the $G_j$-action on $V_j$. We therefore have the equalities

  \begin{align*}
    \sum_{i=1}^{n_j}|I_j^i| & = n \\
    |H_x| & = n\cdot|H_{x'}|,\\
    |G_j| & = |G_j^i| \cdot |I_j^i|\qquad \forall i=1,\ldots,n_j,
  \end{align*}
  
  where $|G_j^i|$ is the number of elements of the automorphism group along the orbit $I_j^i$ of the $G_j$-action. We can now calculate

  \begin{align*}
    \deg_{[x]}\phi & = |H_x|\sum_{j=1}^k\frac{1}{|G_j|}=|H_{x'}|\sum_{j=1}^k\frac{n}{|G_j|} \\
    & = |H_{x'}|\sum_{j=1}^k\sum_{i=1}^{n_j}\frac{|I_j^i|}{|G_j|}=|H_{x'}|\sum_{j=1}^k\sum_{i=1}^{n_j}\frac{1}{|G_j^i|}=\deg_{[x']}\phi,
  \end{align*}
  
  as $G_i^j$ is actually isomorphic to the automorphism group of the elements in $I_j^i$ in $\mcG$ and there are no identifications between the leaves and as was said at the beginning of the proof, all preimages of $[x']$ have representatives in some $V_j$.
\end{proof}

\begin{prop}
  If $\phi:\mcG\lra\mcH$ is a morphism covering of two compact orbifold groupoids of dimension $n$ with degree $\deg \phi$, then we have for any $n$-form $\alpha\in\Omega^n(\mcH)$
  \begin{equation*}
    \int_{|\mcG|}\phi^*\alpha=\deg \phi\cdot\int_{|\mcH|}\alpha.
  \end{equation*}
\end{prop}

\begin{proof}
Choose any compatible pair of atlases and partitions for the morphism covering $\phi:\mcG\lra\mcH$ and use the earlier notation. We calculate
\begin{align*}
  \int_{|\mcG|}\phi^*\alpha&=\sum_{i\in I}\sum_{j\in J_i}\frac{1}{|G_i^j|}\int_{V_i^j}\phi^*\alpha\cdot f_i\circ \phi\\
  &=\sum_{i\in I}\sum_{j \in J_i}\frac{1}{|G_i^j|}\int_{U_i}\alpha\cdot f_i\\
  &=\sum_{i\in I}\frac{1}{|H_i|}\left(\int_{U_i}\alpha\cdot f_i\right)\left(|H_i|\cdot\sum_{j\in J_i}\frac{1}{|G_i^j|}\right)\\
  &=\deg \phi\cdot\int_{|\mcH|}\alpha,
\end{align*}
where we have used the fact that $\phi:V_i^j\lra U_i$ is a diffeomorphism in the first step and \cref{prop:morphism-coverings-deg} for the last step.
\end{proof}

\begin{lem}
  If a morphism covering $\phi:\mcG\lra\mcH$ is such that $\phi:\Aut_{\mcG}(x)\lra\Aut_{\mcH}(\phi(x))$ is surjective for all $x\in\Ob\mcG$ then it induces an actual topological covering on the orbit spaces.
  \label{lem:morphism-covering-actual-covering}
\end{lem}

\begin{proof}
  Recall \cref{fig:nbhd-charts-morphism-cover} and the proof of \cref{prop:morphism-coverings-deg}. We know that the image of a neighborhood $U_x\subset\Ob\mcH$ of $x\in U_x$ gives a neighborhood of $[x]\in |\mcH|$ and this is also true for the preimages of $y_1,\ldots,y_k\in\Ob\mcG$. Furthermore we know that there are no morphisms between elements in $V_i$ and $V_j$ for $i\neq j$. Also recall that $x_1,\ldots,x_n$ are the representatives of an arbitrary class $[x']$ close to $[x]\in|\mcH|$. It is thus enough to show that for every $j=1,\ldots,k$ all preimages $y_j^1,\ldots,y_j^n$ are equivalent. For this notice that by assumption $\phi:\Hom_{\mcG}(x,x)\lra\Hom_{\mcH}(\phi(x),\phi(x))$ is surjective. This implies that for a morphism $g\in\Aut_{\mcH}(x)$ whose extension to $U_x$ satisfies $g:x_1\lra x_i$ there exists an automorphism $h\in\Aut_{\mcG}(y_j)$ whose extension to $V_j$ identifies $h:y_j^1\lra y_j^i$. This works for all $i\in\{1,\ldots,n\}$ and therefore $y_j^1\sim\cdots\sim y_j^n$.
\end{proof}

\section{Algebraic Topology and Symplectic Geometry of Orbifolds}

\subsection{Algebraic Topology}

\label{sec:algebraic-topology-orbifolds}

Regarding the algebraic topology of orbifolds we will only need a few statements which are taken from \cite{adem_orbifolds_2007}, \cite{satake_gauss-bonnet_1957} and \cite{satake_generalization_1956}.

\begin{definition}
  Let $\mcG$ be an orbifold groupoid and $R$ a ring. Then we can define
  \begin{enumerate}[label=(\roman*), ref=(\roman*)]
  \item its \emph{(quotient) singular cohomology} $H^*(|\mcG|,R)$ and \emph{(quotient) singular homology} $H_*(|\mcG|,R)$,
    \item its \emph{de Rham cohomology} $H^*(\mcG,\RR)\coloneqq H(\Omega^*(\mcG),\dd)$ defined via invariant differential forms on $\Ob\mcG$,
    \item its \emph{orbifold singular cohomology} $H_{\text{orb}}^*(\mcG,R)\coloneqq H^*(\mcB\mcG,R)$ as well as its \emph{orbifold singular homology} $H^{\text{orb}}_*(\mcG,R)\coloneqq H_*(\mcB\mcG,R)$ and
    \item its \emph{orbifold fundamental group} $\pi_1^{\text{orb}}(\mcG,x_0)\coloneqq \pi_1(\mcB\mcG,x_0)$     
  \end{enumerate}
  where $\mcB\mcG\coloneqq |\mcG_{\bullet}|$ is the geometric realization of the nerve of the groupoid $\mcG$, i.e.\ a model for its classifying space, see e.g.\ \cite{moerdijk_classifying_1995}.
\end{definition}

\begin{thm}[See \cite{satake_generalization_1956} and \cite{adem_orbifolds_2007}]
  The following statements hold for a compact oriented\footnote{An ep-Lie groupoid is called \emph{oriented} if the object and morphism manifolds are oriented and all structure maps are orientation preserving.} orbifold groupoid.
  \begin{enumerate}[label=(\roman*), ref=(\roman*)]
    \item There are isomorphisms of graded algebras $H^*_{\text{orb}}(\mcG,\RR)\cong H^*(\mcG,\RR)\cong H^*(|\mcG|,\RR)$ and $H^*_{\text{orb}}(\mcG,\QQ)\cong H^*(|\mcG|,\QQ)$.
    \item We have $H_n(|\mcG|,\QQ)\cong \QQ$ for $\dim\mcG=n$ with a natural fundamental class $[\mcG]\in H_n(|\mcG|,\QQ)$ as the generator.
    \item $\mcB\mcG$ satisfies Poincar\'e duality over the rational numbers.
  \end{enumerate}
  \label{thm:alg-top-orbifolds}
\end{thm}

\begin{rmk}
  Note that for global quotients of manifolds by finite free group actions the first statement is easy to see as an averaging argument gives you immediately that the de-Rham cohomology of invariant forms calculates de-Rham cohomology of the quotient. Thus all the stabilizers and the orbifold structure can be seen in the torsion part of $H^*_{\text{orb}}(\mcG,\ZZ)$ only.

  Also note that the these statements show that as long as we use rational or real coefficients we can treat the algebraic topology as if $|\mcG|$ was a manifold.
\end{rmk}

\begin{definition}
  The \emph{Chern class} $c_1(E)\in H^2(\mcG,\RR)$ of a good complex orbifold line bundle $E\lra\Ob\mcG$ is defined by first choosing an \emph{invariant connection} $\nabla$ on $E\lra\Ob\mcG$ and
  \begin{equation*}
    c_1(E)\coloneqq \frac{1}{2\pi\ii}[F^{\nabla}]
  \end{equation*}
  where $F^{\nabla}\in\Omega^2(\mcG,\ii\RR)$ is the curvature form of $\nabla$. Here, an invariant connection means a $\CC$-linear map $\nabla:\Gamma(U,E)\lra\Omega^1(U,E)$ for open sets $U\subset\Ob\mcG$ satisfying the Leibniz rule such that the following diagram commutes
  \begin{equation*}
    \xymatrix{
      \Gamma(U,E) \ar[r]^{\nabla} \ar[d]^g & \Omega^1(U,E) \ar[d]^g \\
      \Gamma(U,E) \ar[r]^{\nabla} & \Omega^1(U,E) \\
      }
  \end{equation*}
  for every smooth section $g:U\lra s^{-1}(U)\cap t^{-1}(U)\subset\Mor\mcG$ acting in the obvious way pointwise.
\end{definition}

\begin{rmk}
  A few comments are in order.
  \begin{enumerate}[label=(\roman*), ref=(\roman*)]
    \item Notice that an invariant connection does map invariant vector fields, i.e.\ invariant sections of $\ts\Ob\mcG$, to invariant forms. From this one can easily see that $F^{\nabla}\in\Omega^*(\mcG)$.
    \item Also convex linear combinations of the corresponding invariant connection forms are again invariant and thus the space of invariant connections is convex. This implies that the usual proof of the independence of the Chern class on the choice of connection works in the case of orbifolds, too.
    \item Suppose we are given a homomorphism $\phi:\mcG\lra\mcH$ together with a bundle $E\lra\Ob\mcH$. Then we can pull back the bundle together with its connection and see
      \begin{equation*}
        c_1(\phi^*E)=\phi^*c_1(E).
      \end{equation*}
      Here it is of course essential that the map $\phi$ is indeed smooth.
    \item In the case of a complex line bundle $E$ over a manifold $M$ one can show that $c_1(E)$ lifts to $H^2(M,\ZZ)$. For a good complex line bundle over an orbifold $\mcG$ we have $c_1(E)\in H^2(\mcG,\QQ)$, i.e.\ it lifts to singular cohomology with rational coefficients, see e.g.\ \cite{adem_orbifolds_2007}.
    \item As usual if we talk about $S^1$ or $T^n$-bundles then we can also define Chern classes. The data of a $S^1$-bundle over $\Ob\mcG$ defines a complex line bundle over $\Ob\mcG$ up to isomorphism and we use this Chern class as the one for the $S^1$-bundle. Similarly a $T^n$-bundle $E$ has $n$ sub-$S^1$-bundles by the $n$ injective Lie-group homomorphisms $S^1\lra T^n$ into the components. Then we define the Chern class $c_1(E)$ as the vector of the Chern classes of these subbundles.
    \item Unfortunately we will encounter bad bundles over our later non-reduced orbifold moduli spaces. For these bundles it is not a priori clear how to define the Chern class in terms of Chern--Weil theory as the definition of a connection might be empty if there are no non-trivial sections. However, it urns out that one can always consider a bad orbifold vector bundle as a restriction of a good orbifold vector bundle for which one case use Chern--Weil theory to define characteristic classes. This can be done by recognizing that the vertical tangent bundle $\operatorname{V}E\lra E$ to the bad orbifold vector bundle $E\lra\Ob\mcG$ is in fact a good bundle and one has $E\cong \operatorname{V}E|_{\iota(\Ob\mcG)}$ where $\iota:\Ob\mcG\lra E$ is the zero-section. The Chern--Weil characteristic classes defined in this way by pulling back via $\iota$ do end up in $H^*(\mcG,\RR)$ and satisfy the same properties as the usual characteristic classes on a good bundle since they are restrictions of these classes. This is shown and explained in \cite{seaton_characteristic_2007}. 
  \end{enumerate}
  \label{rmk:chern-classes-orbibundles}
\end{rmk}

\subsection{Symplectic Geometry}

Recall that a symplectic orbifold is an ep-Lie groupoid such that object and morphism manifolds are symplectic and all structure maps are symplectomorphisms together with a homeomorphism of the quotient to the topological space. This section describes the orbifold notion of symplectic reduction and the Duistermaat-Heckman theorem. First, we have of course the usual Darboux theorem.

\begin{thm}[Darboux]
  Let $\mcG$ be a symplectic orbifold groupoid. Then every point $[x]\in|\mcG|$ has an orbifold chart\footnote{Note that we use the word orbifold chart slightly differently from before.} $(U_x,G_x,\pi_x)$ with $U_x\subset\RR^n$ and $\pi_x:U_x\lra \Ob\mcG$ continuous such that $U_x\lra\Ob\mcG\lra|\mcG|$ is a homeomorphism onto its image including $[x]$ such that $\pi_x^*\omega=\omega_0$ and $G_x$ acts on $U_x$ via symplectomorphisms.
\end{thm}

\begin{rmk}
  Note that we can reformulate this in a global way as follows. There exists a symplectic orbifold groupoid $\mcH$ with an equivalence $\epsilon:\mcH\lra\mcG$ such that $\epsilon$ is symplectic on objects and morphisms and such that every $[x]\in|\mcH|$ has only finitely many representatives in $\Ob\mcH$. This groupoid can be constructed by covering $\Ob\mcG$ with symplectic charts and then choosing a locally finite sub covering $\mcU$ and defining $\mcH\coloneqq \mcG_{\mcU}$.
\end{rmk}

\index{Orbifold!Group Action on}

\begin{definition}
  Let $\mcG$ be a orbifold groupoid. We call a smooth group homomorphism $H\lra\Isom(\mcG)$ a \emph{group action by $H$ on $\mcG$}. Here $\Isom(\mcG)$ means homomorphisms $\phi:\mcG\lra\mcG$ of orbifold groupoids such that there exists an inverse $\eta:\mcG\lra\mcG$ satisfying $\phi\circ\eta=\eta\circ\phi=\id_{\mcG}$. Note that these are equalities and not just identities up to natural transformation. Furthermore, smooth means that the maps
  \begin{align*}
    \Psi:H\times\Ob\mcG & \lra\Ob\mcG \\
    \Phi:H\times\Mor\mcG & \lra \Mor\mcG
  \end{align*}
  are smooth maps.
  \label{def:group-action-on-orbifold}
\end{definition}

\begin{rmk}
  \begin{enumerate}[label=(\roman*), ref=(\roman*))]
    \item Comparing the definition of group actions on manifolds with this one, one sees that we do not just require the group to act on the underlying topological space but rather on the objects. This might not be necessary but it simplifies things a lot if we have actual group actions on the object and morphism manifolds. This is why we require actual invertible homomorphisms.
    \item Note that isomorphisms induce homeomorphisms of the quotient space and that they preserve automorphism groups.
    \item Also note that the group action $H$ on $\mcG$ gives two group actions on the object and morphism sets which are compatible in a very specific way because $(\Psi_h,\Phi_h)$ is a homomorphism of $\mcG$ for every $h\in H$. 
  \end{enumerate}
\end{rmk}

\begin{lem}
  For every $h\in H$ and every $g:x\lra y$ in $\Mor\mcG$ we have
  \begin{equation*}
    \Phi_h(g)(\Psi_h(x))=\Psi_h(g(x)).
  \end{equation*}
  \label{lem:relation-group-action-orbifold}
\end{lem}

\begin{proof}
  Because $(\Psi_h,\Phi_h)$ is a homomorphism we have
  \begin{align*}
    \Psi_h\circ s & = s\circ\Phi_h, \\
    \Psi_h\circ t & = t\circ\Phi_h.
  \end{align*}
  Now fix neighborhoods $U_x,U_y,U_g$ and $U_{\Phi_h(g)}$ close to $x,y,g:x\lra y$ and $\Phi_h(g):\Psi_h(x)\lra\Psi_h(y)$, respectively, such that we can invert $s:U_g\lra U:x$ and $s:U_{\Phi_h(g)}\lra U_y$. We thus have $s|_{U_{\Phi_h(g)}}^{-1}\circ \Psi_h=\Phi_h\circ s|_{U_g}^{-1}$ and therefore
  \begin{equation*}
    t\circ s|_{U_{\Phi_h(g)}}^{-1}\circ\Psi_h=t\circ\Phi_h\circ s|_{U_g}^{-1}=\Psi_h\circ t \circ s|_{U_g}^{-1}
  \end{equation*}
  which gives the result as $t\circ s|_{U_{\Phi_h(g)}}^{-1}(x)=\Phi_h(g)(x)$ and $t \circ s|_{U_g}^{-1}(x)=g(x)$.
\end{proof}

\begin{lem}
  Given a group action $H\lra\Isom(\mcG)$ by a Lie group $H$ on an orbifold groupoid $\mcG$ one can define a Lie algebra homomorphism $\mfh\lra\mcX(\mcG)$ by sending $X\mapsto\ul{X}$, so in particular infinitesimal vector fields of the action on objects are invariant.
\end{lem}

\begin{proof}
  Recall that
  \begin{equation*}
    \ul{X}(x)=\ddt\exp(tX)\cdot x
  \end{equation*}
  and that we need to check $g\cdot \ul{X}=\ul{X}(y)$ for $g:x\lra y$. Denote by $\varphi_g$ the map acting on a neighborhood $U_x\subset \Ob\mcG$ of $x$ defined by $g$. First note that for $h\in H$ sufficiently close  to the identity in $H$ we have
  \begin{equation*}
    \varphi_{\Phi_h(g)}=\varphi_g
  \end{equation*}
  where they are both defined, i.e.\ $U_{\Phi_h(g)}\cap U_g$. This is because $\Phi_h(g)$ and $g$ are very close to each other and the source and target maps are local diffeomorphisms defining $\varphi_{\Phi_h(g)}$ and $\varphi_g$. From this and \cref{lem:relation-group-action-orbifold} follows
  \begin{equation*}
    \Psi_h\circ\varphi_{g}=\varphi_{\Phi_h(g)}\circ\Psi_h=\varphi_g\circ\Psi_h,
  \end{equation*}
  where we have rewritten \cref{lem:relation-group-action-orbifold} as $\varphi_{\Phi_h(g)}\circ\Psi_h=\Psi_h\circ\varphi_g$. Now we can calculate
  \begin{align*}
    g\cdot \ul{X}(x) & = \dd_x\varphi_g\left(\ddt\exp(tX)\cdot x\right) = \ddt\varphi_g(\Psi_{\exp(tX)}(x)) \\
                     & = \ddt\Psi_{\exp(tX)}(\varphi_g(x)) \\
    & = \ul{X}(\varphi_g(x)) = \ul{X}(y).
  \end{align*}
  The fact that $X\longmapsto\ul{X}$ is a Lie algebra homomorphism follows from the usual statement on manifolds.
\end{proof}

\begin{definition}
  Let $H$ be a Lie group and $\mcG$ an orbifold groupoid together with an action $H\lra\Isom(\mcG)$. Then we call
  \begin{enumerate}[label=(\alph*), ref=(\alph*)]
    \item the action \emph{symplectic} if the actions $\Psi$ and $\Phi$ are symplectic, i.e.\ if $\Psi_h^*\omega_{\Ob}=\omega_{\Ob}$ and $\Psi_h^*\omega_{\Mor}=\omega_{\Mor}$ for all $h\in H$.
    \item A map $\mu:\Ob\mcG\lra\mfh^*$ such that $s^*\mu=t^*\mu$ is called a \emph{moment map} if
      \begin{enumerate}[label=(\roman*), ref=(\roman*)]
        \item its quotient map $\mu:|\mcG|\lra\mfh^*$ is proper,
        \item it is Ad-equivariant, i.e.\ $\Psi_h^*\mu=\Ad_{h^{-1}}^*\circ\mu$ and
        \item it satisfies 
          \begin{equation}\dd\langle\mu,X\rangle=-i_{\ul{X}}\omega_{\Ob} \label{eq:moment-map} \end{equation} for all $X\in\mfh$ for the action $\Psi$ on the object manifold.
      \end{enumerate}
    \item A symplectic action is called \emph{Hamiltonian} if there exists a moment map $\mu$ for this action.
  \end{enumerate}
\end{definition}

\begin{rmk}
  \begin{enumerate}[label=(\roman*), ref=(\roman*)]
    \item Note that these definitions are verbatim the ones for manifolds except that everything is required to be invariant under the morphism action.
    \item In the definition of a symplectic action it is enough to require that one of the two actions on objects or morphisms is symplectic.
  \end{enumerate}
\end{rmk}

\begin{lem}
  Let $H$ be a Hamiltonian Lie group action on an orbifold groupoid $\mcG$ with moment map $\mu:\Ob\mcG\lra\mfh^*$. Then the action of $\Phi$ on $\Mor\mcG$ is Hamiltonian, too, with moment map $s^*\mu=t^*\mu$.
\end{lem}

\begin{proof}
  As stated in the remark we have for $h\in H$
  \begin{equation*}
    \Phi_h^*\omega_{\Mor}=\Phi_h^*s^*\omega_{\Ob}=s^*\Psi_h^*\omega_{\Ob}=s^*\omega_{\Ob}=\omega_{\Mor}
  \end{equation*}
  showing that $\Phi$ is a symplectic action on the morphism manifold.

  Now define $\eta:\Mor\mcG\lra\mfh^*$ by $\eta\coloneqq \mu\circ s$. This is a smooth map and it is Ad-equivariant as
  \begin{equation*}
    \Phi_h^*\eta=\eta\circ\Phi_h=\mu\circ s\circ\Phi_h=\mu\circ\Psi_h\circ s =\Ad_h^*\circ\mu\circ s=\Ad_h^*\circ\eta.
  \end{equation*}
  It remains to check that it satisfies \cref{eq:moment-map}. For this purpose we prove that the infinitesimal vector fields of the $\Phi$-action map under $s$ and $t$ to the infinitesimal vector fields of the $\Psi$-action. This follows for any $X\in\mfh$ and $g:x\lra y$ from
  \begin{equation*}
    s_*\ul{X}_{\Mor}(s(g))=\ddt s\left(\Phi_{\exp(tX)}(g)\right)=\ddt \Psi_{\exp(tX)}(s(g))=\ul{X}_{\Ob}(x).
  \end{equation*}
  But then $\eta,\ul{X}_{\Mor}$ and $\omega_{\Mor}$ are all pullbacks of the corresponding objects on $\Ob\mcG$ via the local diffeomorphisms $s$ and $t$ and thus \cref{eq:moment-map} holds on $\Mor\mcG$ as well.
\end{proof}

\begin{thm}[Symplectic Reduction]
  Let $H$ be a compact Lie group acting on the orbifold groupoid $\mcG$ in a Hamiltonian way with moment map $\mu:\Ob\mcG\lra\mfh^*$. Furthermore assume that $H$ acts freely on $\mu^{-1}(0)\subset\Ob\mcG$. Then $H$ acts freely on $\eta^{-1}(0)\subset\Ob\mcG$ and the symplectic quotient  manifolds $\faktor{\mu^{-1}(0)}{H}$ and $\faktor{\eta^{-1}(0)}{H}$ form a symplectic orbifold groupoid $\mcG\sslash H$ together with a symplectic homomorphism $\mcG|_{\mu^{-1}(0)}\lra\mcG\sslash H$ where $\mcG|_{\mu^{-1}(0)}$ means the full subcategory orbifold groupoid defined by $\mu^{-1}(0)\subset\Ob\mcG$.
  \label{thm:symplectic-reduction}
\end{thm}

\begin{proof}
  If the $\Psi$-action on the objects is free then so is the action on morphisms because if there were a fixed point $g:x\lra y$ such that $\Phi_h(g)=g$ then $g\in\Hom_{\mcG}(\Psi_h(x),\Psi_h(y))$ and thus $\Psi_h(x)=x$.

  Since $H$ is compact we can therefore form the two symplectic manifolds
  \begin{equation*}
    \left(\faktor{\mu^{-1}(0)}{H},\ol{\omega}_{\Ob}\right)\text{ and }\left(\faktor{\eta^{-1}(0)}{H},\ol{\omega}_{\Mor}\right)
  \end{equation*}
  satisfying $\pi_{\Ob}^*\ol{\omega}_{\Ob}=\omega_{\Ob}|_{\mu^{-1}(0)}$ and similarly for the morphisms. Notice that 
  \begin{equation*}
    s^{-1}(\mu^{-1}(0))=\eta^{-1}(0)=t^{-1}(\mu^{-1}(0))
  \end{equation*}
  and therefore the morphisms in $\mcG$ acting on $\mu^{-1}(0)$ are precisely those contained in $\eta^{-1}(0)$. Recall that because $h$ acts via homomorphisms all the structure maps of $\mcG$ are equivariant with respect to the actions $\Psi$ and $\Phi$. Therefore they descend to the quotients giving rise to a groupoid $\mcG\sslash H$ with symplectic structures on object and morphism manifold.

  Next we argue that this groupoid is in fact ep-Lie. All the structure maps are smooth because they arise from equivariant smooth maps\footnote{Even the multiplication map $m:\Mor\mcG{_s\times_t}\Mor\mcG\lra\Mor\mcG$ restricts to $m:\eta^{-1}(0){_s\times_t}\eta^{-1}(0)\lra\eta^{-1}(0)$ which is again a manifold on which $H$ acts diagonally.} and so the local slice theorem gives charts for the quotients on which the maps are smooth by assumption. Thus $\mcG\sslash H$ is Lie. For the source map $\ol{s}$ consider the diagram
  \begin{equation*}
    \xymatrix{
      \eta^{-1}(0) \ar[r]^s \ar[d]^{\pi_{\Mor}} & \mu^{-1}(0) \ar[d]^{\pi_{\Ob}} \\
      \faktor{\eta^{-1}(0)}{H} \ar[r]^{\ol{s}} & \faktor{\mu^{-1}(0)}{H}
      }
  \end{equation*}
  Because $s$ is a local diffeomorphism we know that $\ker(\dd_g\pi_{\Ob}\circ s)=\ts(\Phi_H(x))=\ker(\dd_g\pi_{\Mor})$ and $\pi_{\Ob}\circ s$ is a submersion. Therefore $\ol{s}$ is a submersion, too. The same argument works for the other structure maps showing that $\mcG\sslash  H$ is in fact étale. The structure maps are also symplectic as we have $\pi_{\Ob}^*\ol{\omega}_{\Ob}=\omega_{\Ob}|_{\mu^{-1}(0)}$ and $\pi_{\Mor}^*\ol{\omega}_{\Mor}=\omega_{\Mor}|_{\eta^{-1}(0)}$ and therefore their pullbacks are along the projections agree. Next we show that $\mcG\sslash H$ is in fact proper. Consider the diagram
  \begin{equation*}
    \xymatrix{
      \eta^{-1}(0) \ar[r]^{s\times t} \ar[d]^{\pi_{\Mor}} & \mu^{-1}(0)\times\mu^{-1}(0) \ar[d]^{\pi_{\Ob}\times\pi_{\Ob}} \\
      \faktor{\eta^{-1}(0)}{H} \ar[r]^-{\ol{s}\times\ol{t}} & \faktor{\mu^{-1}(0)}{H}\times\faktor{\mu^{-1}(0)}{H}
      }
  \end{equation*}
  of commuting continuous maps. Notice that $\faktor{\mu^{-1}(0)}{H}\times\faktor{\mu^{-1}(0)}{H}$ is locally compact and Hausdorff and $\pi_{\Ob}\times\pi_{\Ob}$ is a closed map by definition whose preimage of a point is $H$ and thus compact. Also, the restriction of $s\times t$ to $\eta^{-1}(0)$ is proper as $\eta^{-1}(0)=(s\times t)^{-1}(\mu^{-1}(0),\mu^{-1}(0))$ and thus preimages of compact subsets of $\mu^{-1}(0)\times\mu^{-1}(0)$ are contained in $\eta^{-1}(0)$. Therefore preimages of compacts subsets of $\faktor{\mu^{-1}(0)}{H}\times\faktor{\mu^{-1}(0)}{H}$ in $\eta^{-1}(0)$ are compact and thus their image under $\pi_{\Mor}$ are also compact. Thus $\mcG\sslash H$ is indeed an orbifold groupoid.

  The symplectic homomorphism $\mcG|_{\mu^{-1}(0)}$ is given by the projection onto the quotients by $H$ on objects and morphisms.  
\end{proof}

\begin{rmk}
  Of course, if $H$ is a compact Abelian Lie group then the equivariance condition on the moment map says $\Psi_h^*\mu=\mu$ and we can do symplectic reduction not just at zero but at any regular value $\xi\in\mfh^*$. So in particular we can try to compare the cohomology classes of the reduced symplectic forms at nearby values $\xi$ and $\xi_0$.
\end{rmk}

\begin{thm}[Duistermaat-Heckman]
  Consider a Hamiltonian torus action $T^n$ on the orbifold groupoid $\mcG$. Let $\xi_0,\xi_1\in\mft^*$ be two points contained in an open connected component $U\subset\mft^*$ such that every $\xi\in U$ is a regular value of the moment map $\mu$ and $T^n$ acts freely on all $\mu^{-1}(\xi)$. Then there is a natural identification of the de Rham cohomologies $H^*(\mcG\sslash H[\xi_0],\RR)$ and $H^*(\mcG\sslash H[\xi_1],\RR)$ such that in $H^*(\mcG\sslash H[\xi_0],\RR)$ we have
  \begin{equation*}
    [\omega_{\xi_1}]=[\omega_{\xi_0}]+2\pi\left\langle\xi_1-\xi_0,c_1\left(\mu^{-1}(\xi_0)\lra\mcG\sslash H[\xi_0]\right)\right\rangle.
  \end{equation*}
  \label{thm:duistermaat-heckman}
\end{thm}

\index{Duistermaat-Heckman}

\begin{rmk}
   Note that in \cref{thm:duistermaat-heckman} we assume \emph{free} group actions of $T^n$ on the object manifold. This is probably not necessary as \cite{duistermaat_variation_1982} allows for this by passing to orbifold quotients. In our case, however, we already start with an orbifold and thus have to figure out what the symplectic quotient should be exactly. For example, if $x\in\Ob\mcG$ has automorphisms and is a fixed point of the $T^n$-action, what should the automorphisms of the quotient be? This task seems non-trivial. As we don't need this case, we state everything for the free action on objects.
\end{rmk}

\begin{proof}
  The original proof from \cite{duistermaat_variation_1982} can be adapted to work on the object symplectic quotients which yields the result.
\end{proof}

\index{Groupoid!Orbifold|)}

\chapter{Orbifold Structure on the Moduli Space of Closed Hurwitz Covers}

\label{chap:orbi-structure}

This section deals with the construction of the orbifold structure on the moduli space of closed Hurwitz covers. It will be built from the corresponding orbifold structure on Deligne--Mumford space constructed in \cite{robbin_construction_2006}. All the later results build on this section.

We will need many similar moduli spaces and a somewhat consistent notation to distinguish them. Essentially the main difference is that some spaces contain maps between Riemann surfaces, some contain only Riemann surfaces and some also have marked points on the boundary. Furthermore, some spaces contain marked complex curves (or hyperbolic surfaces with cusps) whereas others contain Riemann surfaces with boundaries (or hyperbolic surfaces with geodesic boundary components of arbitrary length including cusps).

These spaces are summarized in \cref{tab:moduli-spaces-wo-maps} and \cref{tab:moduli-spaces-with-maps} in \cref{sec:moduli-spaces-3}. In this chapter we will deal with the moduli space of closed Hurwitz covers from \cref{def:hurwitz-cover}. For this we first choose and fix combinatorial data as in \cref{sec:indexc-data}, i.e.\ a tuple 
\begin{equation*}
  \left(g,h,k,n,T=\left(d,\nu,\{T_i\}_{i=1}^n\right),\{l_j\}_{j=1}^k\right). 
\end{equation*}
In general we equip moduli spaces with orbifold structures in the following way. We will have a (non-small) groupoid category $\mcC$ with objects $\obj\mcC$ whose orbit space $|\mcC|$ we want to equip with an orbifold structure. This category includes all possible objects and so in particular will not have any manifold structure. However it will have a topology on the orbit space. The orbifold structure for $|\mcC|$ is then given by a second groupoid $\mcG$ with objects and morphisms as in $\mcC$ but fewer of them and which do have a manifold structure such that $\mcG$ forms an orbifold groupoid together with a homeomorphism $|\mcG|\lra|\mcC|$. This homeomorphism will usually be induced by a functor $\mcG\Longrightarrow\mcC$ whose induced map on the orbit spaces is continuous. The following diagram demonstrates the involved spaces.

\begin{equation*}
  \xymatrix{
    \Mor\mcG \ar[r]^{\phi_{\Mor}} & \Mor\mcC \\
    \Ob\mcG \ar[r]^{\phi_{\Ob}} \ar[d] & \Ob\mcC \ar[d] \\
    |\mcG| \ar[r]^{|\phi|} & |\mcC|
    }
\end{equation*}

In this diagram the category $\mcC$ is usually a large category including all considered objects whereas the category $\mcG$ is a small one and carries the geometric structure we want to put on $|\mcC|$ by making sure that the map $|\phi|$ is a homeomorphism. In our case this will usually be an orbifold structure on $|\mcC|$, i.e.\ the category $\mcG$ will be an ep-Lie groupoid.

Unfortunately this ideal picture will not be possible to realize as the moduli space of closed Hurwitz covers will not be of this type. We will see in \cref{sec:constr-an-orbif} that this space does not look locally like an orbifold but rather somewhat branched. Thus we will describe a different orbifold category whose orbit space has a continuous map to the actual moduli space and which is ``branched'' over codimension two subsets. This will not cause problems for our later calculations in \cref{sec:sympl-geom-moduli} and \cref{sec:appl-exampl} because we will be interested in integrals of volume forms and pointwise degrees mainly.

We will adhere to the following convention. Groupoids containing all possible objects will be denoted by a cursive $\mcR$. Note that this is the category, so the orbit space -- which is usually called the moduli space -- will be denoted by $|\mcR|$. By contrast, we will use $\mcM$ for the orbifold categories which will contain far less objects and morphisms. So for example the Deligne--Mumford space for us is $|\mcR_{g,n}|$ for the appropriate category $\mcR_{g,n}$ instead of the more common $\ol{\mcM}_{g,n}$. In order to distinguish various similar moduli spaces we will use the number of indices to denote whether a space contains just surfaces or pairs of surfaces with maps in between. This means that one pair of indices like in $\mcR_{g,n}$ denotes a category of surfaces with genus $g$ and $n$ marked points (or boundaries) and $\mcR_{g,k,h,n}$ implies that its objects are at least triples $(C,X,u)$ with $u:C\lra X$ a map having certain properties. We will use hats and tildes to denote that the objects of a category have more data like additional marked points or free homotopy classes of curves.

Again, for reference, \cref{tab:moduli-spaces-wo-maps} and \cref{tab:moduli-spaces-with-maps} summarize most of the categories used in this thesis. The precise definitions can be found in \cref{def:hurwitz-cover} as well as \cref{sec:orb-structure-mod-space-bordered-hurwitz-covers-definitions}.

\section{Complex Gluing}

\label{sec:complex-gluing}

\index{Gluing!Complex}

We will now describe how to glue Hurwitz covers close to nodes. Recall that a Hurwitz cover $u:c\lra X$ maps nodes to nodes and all preimages of nodes are nodes. Also, the degree of the Hurwitz covers from both sides of the node is equal and non-zero.

\subsection{Setup}

Let $C$ and $X$ be nodal closed stable connected Riemann surfaces of genus $g$ and $h$, respectively. Let $u:C\lra X$ be a branched covering of degree $d$. By $\wt{X}$ and $\wt{C}$ we denote the normalization of the surfaces, i.e.\ possibly disconnected closed stable smooth Riemann surfaces with a holomorphic map to $\wt{C}$ and $\wt{X}$, respectively, such that it is biholomorphic outside of the nodes. Thus every node has two preimages in the normalization.

We will describe a neighborhood in the moduli space of Hurwitz covers $|\mcR_{g,k,h,n}(T)|$ of $(C,u,X,\bq,\bp)$. Now fix a branch point $p\in X$ and its preimages $q_1,\ldots,q_l\in C$ which are nodes in $C$ and the map $u$ has local degree $k_i$ from both sides at $q_i$ for all $i=1,\ldots,l$. We will denote the points corresponding to the node in the normalization by $q_i^{\dagger}$ and $q_i^{\ast}$. Looking at the nodal points in the normalization each component close to $q_i$ is mapped to only one component close to $p$. This means we can distinguish between ``sides'' of a node $q_i$, i.e.\ we will denote all neighborhoods and nodal points in the normalization with $j=\dagger$ if they are mapped under $u$ to the component containing $p^{\dagger}$ in $\wt{X}$ and likewise for $q_i^{\ast}$ and $p^{\ast}$. Please be aware that the notation might be slightly misleading. We are not interested in the marked points of $u:C\lra X$ but we are interested in the fibre over a node. This is why we will use the latter $l$ to denote the number of nodal preimages of a node and $k_1,\ldots,k_l$ for the degrees at the nodes which are in particular \emph{not} fixed by the data $T$.

\subsection{Choice of Coordinate Neighborhoods on the Surfaces}

In order to glue annuli into the surface we need appropriate coordinate neighborhoods on the surfaces. Note that the usual normal-form theorem for holomorphic maps between Riemann surfaces actually only modifies the chart in the domain. Thus the following slight generalization holds.
\begin{lem}
  Suppose $u:C\lra X$ is a holomorphic map between closed \emph{smooth} Riemann surfaces and $u(q_1)=\cdots=u(q_l)=p$. Then there exist coordinate charts $\phi_i:U_i\lra V_i\subset\CC$ and $\psi:U\lra V\subset\CC$ such that $q_i\in U_i$, $p\in U$, $\phi_i(q_i)=\psi(p)=0$ together with unique integers $k_1,\ldots,k_l$ such that for all $i=1,\ldots,l$ we have $\psi\circ u\circ\phi_i^{-1}(z)=z^{k_i}$.
\label{lem1}
\end{lem}

\begin{proof}
  Uniqueness of the degrees follows as usual. For existence choose first any coordinate charts $\wt{\phi}_i$ and $\wt{\psi}$ such that $\wt{\phi}_i(q_i)=\psi(p)=0$ for all $i=1,\ldots,l$. Now for some integers $k_i$ we have $\psi\circ u\circ(\wt{\phi}_i)^{-1}(\zeta)=\zeta^{k_i}e^{h_i(\zeta)}$. Define $F_i(\zeta)\coloneqq \zeta e^{\frac{h_i(\zeta)}{k_i}}$ which implies $F_i'(0)\neq 0$ meaning that $F_i$ is invertible in the neighborhood of the origin. Then define $\phi_i(x)\coloneqq F_i\circ\wt{\phi}_i(x)$ for all $i=1,\ldots,l$. These maps are local biholomorphisms around the origin and satisfy
  \begin{equation*}
    \psi\circ u\circ (\phi_i)^{-1}(z)=\psi\circ u\circ(\wt{\phi}_i)^{-1}\circ F_i^{-1}(z)= (F_i(F_i^{-1}(z)))^{k_i}=z^{k_i}
  \end{equation*}
  close to zero.
\end{proof}

Now we apply \cref{lem1} to the map $u:\wt{C}\lra\wt{X}$ for the points $q_i^{\dagger}\longmapsto p^{\dagger}$ and $q_i^{\ast}\longmapsto p^{\ast}$ and thus obtain charts $\wt{\phi}_i^j:\wt{U}_i^j\lra\wt{O}_i^j\subset\CC$ and $\wt{\psi}^j:\wt{V}^j\lra\wt{W}_j\subset\CC$ such that
\begin{itemize}
    \item $\wt{\phi}_i^j(q_i^j)=0$, $\wt{\psi}^j(p^j)=0$ and
    \item for all $i=1,\ldots,l$ and $j=\dagger,\ast$ we have $\wt{\psi}^j\circ u\circ (\wt{\phi}_i^j)^{-1}(z)=z^{k_i}$.
\end{itemize}

\begin{lem}
  Given $u:C\lra X$ as above there exist biholomorphic coordinate charts $\phi_i^j:U_i^j\lra\DD$  for $i=1,\ldots,l$ and $j=\dagger,\ast$ as well as $\psi^j:V^j\lra\DD$ such that $\phi_i^j(q_i^j)=0$ and $\psi^j(p^j)=0$ and such that $\psi^j\circ u\circ(\phi_i^j)^{-1}(z)=z^{k_i}$.
\label{lem2}
 \end{lem}

\begin{proof}
  Since $p$ has only finitely many preimages under $u$ we can find an $\epsilon>0$ such that $u^{-1}\left((\wt{\psi}^j)^{-1}(B_{\epsilon}(0))\right)\subset\bigcup_{i=1}^l\wt{U}_i^j$ for $j=\dagger,\ast$. Now define
  \begin{itemize}
    \item $(\psi^j)^{-1}:\DD\lra\wt{X}$ by $(\psi^j)^{-1}(z)\coloneqq (\wt{\psi}^j)^{-1}(\epsilon z)$ and
    \item $(\phi_i^j)^{-1}:\DD\lra\wt{X}$ by $(\phi_i^j)^{-1}(z)\coloneqq (\wt{\phi}_i^j)^{-1}(\sqrt[k_i]{\epsilon} z)$ as well as
    \item $U_i^j\coloneqq \phi_i^j(\DD)$ and $V^j\coloneqq \psi^j(\DD)$.
  \end{itemize}
Then we have
\begin{align*}
  \psi^j\circ u\circ(\phi_i^j)^{-1}(z)&=\psi^j\left( u\left(\left(\wt{\phi}_i^j\right)^{-1}(\sqrt[k_i]{\epsilon} z)\right)\right)=\psi^j\left(\left(\wt{\psi}_j\right)^{-1}\left((\sqrt[k_i]{\epsilon}z)^{k_i}\right)\right) \\
  &=\frac{(\sqrt[k_i]{\epsilon}z)^{k_i}}{\epsilon}=z^{k_i}.
\end{align*}
\end{proof}

\begin{lem}
  If the $V^j$'s are fixed in \cref{lem2} then the maps $\psi^j$ are unique up to rotation. For fixed $\psi^j$'s and open sets $U_i^j$ the charts $\phi_i^j$ are unique up to rotations by $k_i$-th roots of unity.
\end{lem}

\begin{proof}
  This is just the standard argument. Suppose there are two such charts $\psi_j$, then their transition function would be a biholomorphism from the unit disc to the unit disc fixing the origin and is thus a rotation. Similarly, if for fixed $\psi_j$ and $U_i^j$ we have two such charts $\phi_i^j$ then their transition function must be rotation and because of the standard from of $u$ in these charts the angle $\theta$ must be such that $k_i\theta=2\pi m$ for some $m\in\ZZ$.
\end{proof}

\subsection{Gluing Surfaces}

\label{sec:gluing}

In this section we want to define a parametrized disc in the space of branched covers close to the given $u:C\lra X$ by opening up the node $p\in X$. This map will be denoted by $\Phi:\DD\lra \Ob\mcR_{g,k,h,n}(T)$. It will depend on the following data.

\index{Disc Structure}

\begin{definition}
  A \emph{disc structure} for a Hurwitz cover $u:C\lra X$ and a node $p\in X$ is a collection of
  \begin{itemize}
    \item neighborhoods $V^{\dagger},V^{\ast}$ around the nodal points corresponding to $p$,
    \item neighborhoods $U_1^{\dagger},U_1^{\ast},\ldots,U_l^{\dagger},U_l^{\ast}$ around the preimages of $p$ and
    \item coordinate charts $\phi_i^j:U_i^j\lra\DD$ and $\psi^j:V^j\lra\DD$ such that $\psi^j\circ u \circ \phi_i^j(z)=z^{k_i}$.
  \end{itemize}
  \label{def:disc-structure}
\end{definition}

This is illustrated in \cref{fig:disc-structure}. The name disc structure is taken from \cite{hofer_applications_2011}.

\begin{figure}[!ht]
  \centering
  \def\svgwidth{0.9\textwidth}
  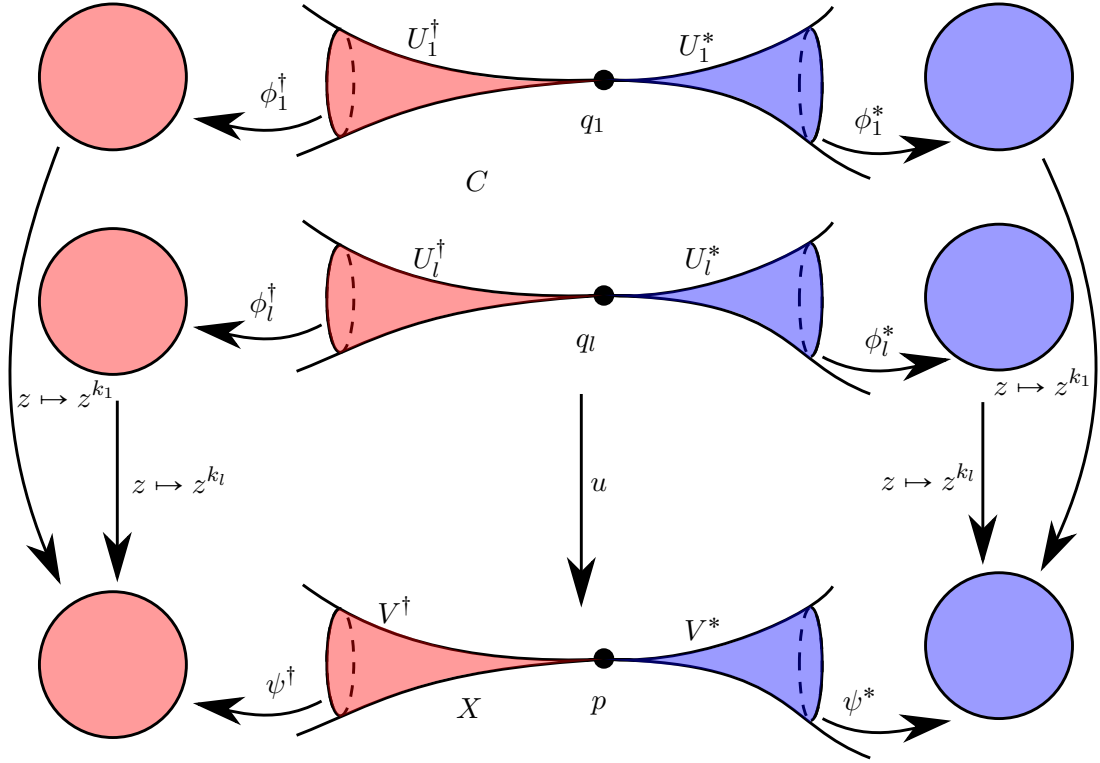
  \caption{The left-hand side corresponds to the $\dagger$-data and the right-hand side to $\ast$-data. The outer discs are actual unit discs in the complex plane.}
  \label{fig:disc-structure}
\end{figure}

First we define for each $a\in \DD$ with $a\neq 0$ a surface $X_a$ by
\begin{equation}
  X_a\coloneqq \frac{\wt{X}\setminus\left((\psi_1)^{-1}(\mathring{B}_{|a|}(0))\sqcup(\psi_2)^{-1}(\mathring{B}_{|a|}(0))\right)}{z\sim z'\qquad\Longleftrightarrow \qquad\parbox{4.1cm}{$z\in V_1\setminus(\psi_1)^{-1}(\mathring{B}_{|a|}(0)),\\ z'\in V_2\setminus(\psi_2)^{-1}(\mathring{B}_{|a|}(0)),\\ \psi_1(z)\cdot\psi_2(z')=a$}}
  \label{eq:glue-complex-cylinder}
\end{equation}
and for $a=0$ we set $X_a\coloneqq X$.

Of course the equivalence relation is meant to also hold if the roles of $z$ and $z'$ are reversed. Note that this corresponds precisely to the \emph{logarithmic gluing profile} in Hofer--Wysocki--Zehnder, \cite{hofer_applications_2011}.

Now we want to define the branched covering $\Phi(a)=(u_a:C_a\lra X_a)$. Denote by $A(|a|,1)=\{z\in\CC\mid |a|\leq |z|\leq 1\}$ the standard annulus of modulus $\frac{1}{2\pi}\ln\left(\frac{1}{|a|}\right)=-\frac{\ln(|a|)}{2\pi}$.

Repeat the gluing construction at a node $q_i$ in $C$ defining a surface $C_b$ for a gluing parameter $b\in\DD$. Using the coordinate charts we obtain injective holomorphic maps $(\phi_i^j)^{-1}|_{A(|b|,1)}:A(|b|,1)\lra C_b$ for $j=\dagger$ and $\ast$. Since the map $u$ is fixed outside this image we have to extend the map on the boundary to the annulus as a holomorphic map onto the target annulus. This will give us conditions on the admissible $b\in\DD$ for fixed parameter $a\in\DD$.

Note that if such an extension to the annulus exists then it is unique. One can see this by representing such a holomorphic map as a convergent Laurent series. Observe that the holomorphic map can be extended slightly over the boundary of the annulus and thus the coefficients are given by a circle integral along any one of the boundaries.

Let us denote the glued in cylinder by
\begin{equation*}
    Z_a\coloneqq \frac{(V^{\dagger}\sqcup V^{\ast})\setminus\left((\psi^{\dagger})^{-1}(\mathring{B}_{|a|}(0))\sqcup(\psi^{\ast})^{-1}(\mathring{B}_{|a|}(0))\right)}{w\sim w'\qquad\Longleftrightarrow \qquad\parbox{4.3cm}{$w\in V^{\dagger}\setminus(\psi^{\dagger})^{-1}(\mathring{B}_{|a|}(0)),\\ w'\in V^{\ast}\setminus(\psi^{\ast})^{-1}(\mathring{B}_{|a|}(0)),\\ \psi^{\dagger}(w)\cdot\psi^{\ast}(w')=a$}}\subset X_a
\end{equation*}
and in the same way
\begin{equation*}
    Y_i(b)\coloneqq \frac{\left(U_i^{\dagger}\sqcup U_i^{\ast}\right)\setminus\left((\phi_i^{\dagger})^{-1}(\mathring{B}_{|b|}(0))\sqcup(\phi_i^{\ast})^{-1}(\mathring{B}_{|b|}(0))\right)}{w\sim w'\qquad\Longleftrightarrow \qquad\parbox{4.3cm}{$w\in U_i^{\dagger}\setminus(\phi_i^{\dagger})^{-1}(\mathring{B}_{|b|}(0)),\\ w'\in U_i^{\ast}\setminus(\phi_i^{\ast})^{-1}(\mathring{B}_{|b|}(0)),\\ \phi_i^{\dagger}(w)\cdot\phi_i^{\ast}(w')=b$}}.
\end{equation*}

\begin{lem}
   There exists a holomorphic map $f:Y_i(b)\lra Z_a$ extending the boundary conditions $u|_{\del U_i^1}$ and $u|_{\del U_i^2}$ if and only if $b^{k_i}=a$.
\label{lem3}
\end{lem}

\begin{proof}
  We will use the $\phi_i^{\dagger}$ and $\psi^{\ast}$ charts to reformulate the extension problem on the standard annulus. We have biholomorphisms (note that we denote these maps in the same way as we denoted the charts by a slight abuse of notation)
  \begin{align*}
    \phi_i^{\dagger}:Y_i(b)&\lra A(|b|,1),\\
    \phi_i^{\ast}:Y_i(b)&\lra A(|b|,1),\\
    \psi^{\dagger}:Z_a&\lra A(|a|,1),\\
    \psi^{\ast}:Z_a&\lra A(|a|,1).
  \end{align*}
Let us denote the local representative of $f$  in the $\dagger$-charts by $v:A(|b|,1)\lra A(|a|,1)$. The local representation of the boundary condition from $u$ at $|z|=1$ looks like $v(z)=z^{k_i}$ because the coordinate charts were chosen in such a way that $u$ was locally given by $z\longmapsto z^{k_i}$. The same is true at the other boundary component but here the coordinates $\phi_i^{\ast}$ and $\psi^{\ast}$ were used so we have to reformulate the boundary condition via the coordinates coming from $\phi_i^{\dagger}$ and $\psi^{\dagger}$.

We have
\begin{equation*}
  v|_{\{|z|=|b|\}}=\psi^{\dagger}\circ u\circ(\phi_i^{\dagger})^{-1}=\psi^{\dagger}\circ (\psi^{\ast})^{-1}\circ \psi^{\ast}\circ u\circ (\phi_i^{\ast})^{-1}\circ\phi_i^{\ast}\circ(\phi_i^{\dagger})^{-1}
\end{equation*}
and thus for $|z|=|b|$
\begin{equation*}
  v(z)=\psi^{\dagger}\circ (\psi^{\ast})^{-1}\circ \psi^{\ast}\circ u \circ (\phi_i^{\ast})^{-1}\left(\frac{b}{z}\right)=\psi^{\dagger}\circ (\psi^{\ast})^{-1}\left(\left(\frac{b}{z}\right)^{k_i}\right)=\frac{a}{b^{k_i}}z^{k_i}.
\end{equation*}

Now $v:A(|b|,1)\lra A(|a|,1)\subset \CC$ is a holomorphic map which needs to be extendable to a small open neighborhood of the annulus. Thus it can be written as a Laurent series converging on a small open neighborhood of the annulus and therefore its coefficients can be calculated by the usual Cauchy integral around a circle, e.g.\ $|z|=1$. Thus the holomorphic extension is unique. Since $v(z)=z^{k_i}$ extends the boundary condition to a map on any open annulus containing $|z|=1$ we see that $v(z)=z^{k_i}$ holds on the whole cylinder $A(|b|,1)$. Thus the other boundary condition gives $v(z)=z^{k_i}=\frac{a}{b^{k_i}}z^{k_i}$ and therefore $a=b^{k_i}$. This proves necessity of the relation.

The condition $a=b^{k_i}$ is also sufficient for the existence of an extension because we can take the above construction as the definition of $v:A(|a|,1)\lra A(|b|,1)$ and then define $f\coloneqq (\psi^{\dagger})^{-1}\circ v\circ \phi_i^{\dagger}$ which is well-defined and satisfies the requirements by construction .
\end{proof}

Thus we see that the set of $b=(b_1,\ldots,b_l)$ that allow for holomorphic maps $u:C_b\lra X_a$ extending the given map outside of the glued cylinder is given by
\begin{equation*}
  \{(b_1,\ldots,b_l)\in\DD^l\mid b_1^{k_1}=\cdots=b_l^{k_l}\}
\end{equation*}
and the map to the coordinate disc $a\longmapsto X_a$ is given by $(b_1,\ldots,b_k)\longmapsto b_1^{k_1}$.

\begin{rmk}
  Note that the set $\{(b_1,\ldots,b_l)\in\DD^l\mid b_1^{k_1}=\cdots=b_l^{k_l}\}$ is not a manifold in general because e.g.\ for $l=2$ and $k_1=k_2=2$ we have
  \begin{equation*}
    \{(b_1,b_2)\in\DD^2\mid b_1^2=b_2^2\}=\{(b_1,b_2)\in\DD^2\mid b_1=\pm b_2\},
  \end{equation*}
  which is the union of two discs intersecting at the origin only. However, some of the cylinders might be such that gluing them into the surface gives isomorphic covers which might have to be identified when parametrizing inequivalent covers.
\end{rmk}

In any case we can try to parametrize a disc of deformations of $u:C\lra X$ in the following way.

\begin{definition}
   Define the map $\Phi:\DD\lra\Ob\mcR_{g,k,h,n}(T)$ by
   \begin{equation*}
     \Phi(w)\coloneqq (u_w:C_{b(w)}\lra X_{a(w)}),
   \end{equation*}
   where $b(w)=(w^{\frac{K}{k_1}},\ldots,w^{\frac{K}{k_l}})$ and $a(w)=w^K$ and $u$ is defined as above by gluing in the cylinders $Y_i(b(z)_i)$ for $i=1,\ldots,l$ and extending the map as before. Here, $K=\lcm(k_1,\ldots,k_l)$ denotes the least common multiple. Also the marked points are enumerated in the same way as on the original $u$ and we do not have to specify anything new because the nodes are not enumerated.
   \label{def:deformation-of-hurwitz-covers}
\end{definition}

\begin{lem}
  The map $b:\DD\lra\{(b_1,\ldots,b_l)\in\DD^l\mid b_1^{k_1}=\cdots=b_l^{k_l}\}$ given by $b(w)=(w^{\frac{K}{k_1}},\ldots,w^{\frac{K}{k_l}})$ is injective. 
\label{lem5}
\end{lem}

\begin{proof}
  Abbreviate $a_i\coloneqq \frac{K}{l_i}$ for $i=1,\ldots,l$. Suppose $b$ was not injective, then there exist two complex numbers $w_1,w_2\in\DD$ different from zero such that $w_1^{a_1}=w_2^{a_1},\ldots,w_1^{a_l}=w_2^{a_l}$. Then there exists $\xi\in\CC$ such that $w_1=\xi w_2$ which satisfies
\begin{equation*}
  \xi^{a_1}=\cdots=\xi^{a_l}=1.
\end{equation*}
Now notice that if the greatest common divisor $\gcd(a_1,\ldots,a_l)=\lambda$ was not equal to one there would exist natural numbers $m_1,\ldots,m_l$ such that for all $i=1,\ldots,l$ we had $a_i=\frac{K}{k_i}=m_i\lambda$ and thus $\frac{K}{\lambda}=m_ik_i$ meaning that $\frac{K}{\lambda}$ was also a common multiple of $k_1,\ldots,k_l$. By assumption $K$ was the least common multiple, so $\lambda=1$ and thus the $a_1,\ldots,a_l$ are coprime.

Since the lemma of B\'ezout holds for $l$ numbers there exist integers $m_1,\ldots,m_l$ such that $1=\gcd (a_1,\ldots,a_l)=m_1a_1+\cdots m_la_l$. Therefore
\begin{equation*}
  \xi=\xi^{\gcd(a_1,\ldots,a_l)}=\xi^{m_1a_1+\cdots +m_la_l}=(\xi^{a_1})^{m_1}\cdots(\xi^{a_l})^{m_l}=1.
\end{equation*}
\end{proof}

\subsection{Gluing Isomorphisms}

If we want to talk about orbifold structures on the moduli space we need to consider automorphisms and more general any biholomorphisms between Hurwitz covers and the relation between biholomorphisms of a nodal cover and biholomorphisms of nearby glued ones. Recall that a morphism between two Hurwitz covers $u:C\lra X$ and $v:D\lra Y$ is given by biholomorphisms $\Phi:C\lra D$ and $\varphi: X\lra Y$ such that $\varphi\circ u=u\circ\Phi$. We denote such a morphism by $(\Phi,\varphi):u\Longrightarrow v$.

\begin{rmk}
  Note that a biholomorphism of a nodal surface is a biholomorphism of the (disconnected) normalization which agrees on the two preimages of nodes and thus induces a homeomorphism of the nodal surfaces. Secondly, note that such a biholomorphism needs to map nodes to nodes and critical points to critical points preserving the degrees.
\end{rmk}

\begin{rmk}
\label{rmk1}
  Since we need twice as many objects as before let us first fix the notation. We will consider two Hurwitz covers between nodal surfaces $u:C\lra X$ and $v:D\lra Y$ with nodal points $p\in X$ and $r\in Y$ together with disc structures around those nodes. So in particular we will describe biholomorphisms at every node locally. We will assume that the nodes have both $l$ preimages, that the kissing numbers of both Hurwitz covers $k_1,\ldots,k_l$ are the degrees of the $u$ at the two sides of the nodes and are equal, respectively, and that the enumerations of the disc structures correspond to each other. The parameters for gluing $u$ will be denoted by $b\in\DD^l$ and $a\in\DD$ and those for gluing $v$ by $b'\in\DD^l$ and $a'\in\DD$. The disc structure of $u$ has charts $\phi_i^{\dagger}$ and $\phi_i^{\ast}$ as well as $\psi^{\dagger}$ and $\psi^{\ast}$. The disc structure charts of $v$ will be denoted by $\rho_i^{\dagger}$, $\rho_i^{\ast}$, $\eta^{\dagger}$ and $\eta^{\ast}$. The glued in cylinders will be denoted by $C_i(b_i)$, $D_i(b_i')$, $\ol{X}(a)$ and $\ol{Y}(a')$. Then we have the following diagram of maps for the normalizations of the nodal surfaces
  \begin{equation}
    \xymatrix{%
      \DD & & & \DD & & \DD & & & \DD \\
      & U_i^{\dagger} \ar[lu]_{\phi_i^{\dagger}} \ar[d]_u \ar@/^1pc/[rrrrr]^{\Phi} & U_i^{\ast} \ar[ru]^{\phi_i^{\ast}} \ar[d]_u \ar@/_1pc/[rrrrr]^{\Phi} & & & & O_i^{\dagger} \ar[lu]_{\rho_i^{\dagger}} \ar[d]_v & O_i^{\ast} \ar[ru]^{\rho_i^{\ast}} \ar[d]_v  & \\
      & V^{\dagger} \ar[ld]_{\psi^{\dagger}} \ar@/^1pc/[rrrrr]^{\varphi} & V^{\ast} \ar[rd]^{\psi^{\ast}} \ar@/_0.8pc/[rrrrr]^{\varphi} & & & & W^{\dagger} \ar[ld]_{\eta^{\dagger}} & W^{\ast} \ar[rd]^{\eta^{\ast}}  & \\
      \DD & & & \DD & & \DD & & & \DD \\
    }
    \label{diag2}
  \end{equation}
and for the glued surfaces
  \begin{equation}
    \xymatrix{%
      A(|b|,1) & & A(|b|,1) & & A(|b'|,1) & & A(|b'|,1) \\
      & C_i(b_i) \ar[lu]_{\phi_i^{\dagger}} \ar@/^1pc/[rrrr]^{\Phi_{bb'}} \ar[ru]^{\phi_i^{\ast}} \ar[d]_{u_b}  & & & & D_i(b_i') \ar[lu]_{\rho_i^{\dagger}} \ar[d]_{v_{b'}} \ar[ru]^{\rho_i^{\ast}} & \\
      & X_a \ar[ld]_{\psi^{\dagger}} \ar@/^1pc/[rrrr]^{\varphi_{bb'}} \ar[rd]^{\psi^{\ast}}  & & & & Y_{a'} \ar[ld]_{\eta^{\dagger}} \ar[rd]^{\eta^{\ast}}  & \\
      A(|a|,1) & & A(|a|,1) & & A(|a'|,1) & & A(|a'|,1) \\
    }
    \label{diag1}
  \end{equation}
Note that we denoted the open sets underlying the disc structures by $U_i^j$, $V^j$, $O_i^j$ and $W^j$, respectively. Also we will denote the morphisms by $(\Phi,\varphi):u\Longrightarrow v$ and $(\Phi_{bb'},\varphi_{bb'}):u_b\Longrightarrow v_{b'}$.
\end{rmk}

\begin{prop}
  Let two nodal Hurwitz covers $u:C\lra X$ and $v:D\lra Y$ be given together with nodes $p\in X$ and $r\in Y$ as well as disc structures around these nodes. Let furthermore a morphism $(\Phi,\varphi):u\Longrightarrow v$ be given such that $\varphi(p)=r$ and which maps all the coordinate neighborhoods of the disc structure of $p\in X$ to the coordinate neighborhoods of the disc structure of $r\in Y$ keeping the enumeration of the preimages.

Then for each $b\in\DD^l$ there exists a unique $b'\in\DD^l$ such that the morphism $(\Phi,\varphi)$ induces a morphism $(\Phi_{bb'},\varphi_{bb'}):u_b\Longrightarrow v_{b'}$ which coincides with $(\Phi,\varphi)$ outside of $\bigsqcup_{i=1}^lC_i(b_i)\subset C_b$.
\label{prop1}
\end{prop}

\begin{proof}
  Use the notation from \cref{rmk1}. Given parameters $a$ and $b$ such that $b_i^{k_i}=a$ for $i=1,\ldots,l$ we ask what the conditions are that the morphism $(\Phi,\varphi)$ induces boundary conditions on the circles of the glued annuli $C_i(b_i)$ such that these extend to biholomorphisms over the annuli $C_i(b_i)\lra D_i(b_i')$ and such that the center square in \cref{diag1} commutes.

First note that the local representatives $\rho_i^j\circ\Phi\circ(\phi_i^j)^{-1}:\DD\lra\DD$ are biholomorphisms of the unit disc fixing zero, therefore by the Schwarz lemma there exists an angle $\beta_i^j\in[0,2\pi)$ such that $\rho_i^j\circ\Phi\circ(\phi_i^j)^{-1}(z)=e^{\ii\beta_i^j}z$ on the whole disc. Analogously there exists an angle $\alpha^j\in[0,2\pi)$ such that $\eta^j\circ\varphi\circ(\psi^j)^{-1}(z)=e^{\ii\alpha^j}z$. Thus we see that the induced boundary conditions in the respective charts are rotations by these corresponding angles.

Since $(\Phi,\varphi):u\Longrightarrow v$ is a morphism we can write down the relation $v\circ\Phi=\varphi\circ u$ in the $\dagger$- and $\ast$-representatives to obtain
\begin{align}
  \begin{split}
  e^{\ii\alpha^{\dagger}}&=e^{\ii k_i\beta_i^{\dagger}}, \\
  e^{\ii\alpha^{\ast}}&=e^{\ii k_i\beta_i^{\ast}}.
  \end{split}
\label{eq:cond-angles}
\end{align}
To see this also recall that $\psi^j\circ u \circ (\phi_i^j)^{-1}(z)=z^{k_i}$ and correspondingly for $v$. This gives conditions on the rotation angles the morphism can induce in our charts.

Now suppose we glue $u$ and $v$ using the parameters $b$ and $b'$. Then we need $a=b_i^{k_i}$ and $a'=(b_i')^{k_i}$ in order to obtain glued maps. Note also that the maps
\begin{align*}
  \rho_i^{\dagger}\circ\Phi_{bb'}\circ (\phi_i^{\dagger})^{-1}&:A(|b|,1)\lra A(|b'|,1)\\
  \rho_i^{\ast}\circ\Phi_{bb'}\circ (\phi_i^{\ast})^{-1}&:A(|b|,1)\lra A(|b'|,1)\\
  \eta^{\dagger}\circ\varphi_{bb'}\circ (\psi^{\dagger})^{-1}&:A(|a|,1)\lra A(|a'|,1)\\
  \eta^{\ast}\circ\varphi_{bb'}\circ (\psi^{\ast})^{-1}&:A(|a|,1)\lra A(|a'|,1)
\end{align*}
need to satisfy the boundary condition $z\longmapsto e^{\ii\gamma}z$ on the boundary $|z|=1$ for the appropriate angle $\gamma$. As above such a holomorphic function defined on an annulus is actually uniquely defined by its boundary because the holomorphic function is extendable to a small open neighborhood and thus the Laurent coefficients can be calculated by the usual Cauchy integral. Thus the maps do actually coincide with $z\longmapsto e^{\ii\gamma}z$ on all of the annuli. Thus we only need to find necessary and sufficient conditions that the map coincides on the inner boundary as well. E.g.\ reformulating the first identity by inserting the transition functions we obtain
\begin{align*}
  e^{\ii\beta_i^{\dagger}}z&=\rho_i^{\dagger}\circ(\rho_i^{\ast})^{-1}\circ\rho_i^{\ast}\circ\Phi_{bb'}\circ(\phi_i^{\ast})^{-1}\circ\phi_i^{\ast}\circ(\phi_i^{\dagger})^{-1}(z)\\
&=\rho_i^{\dagger}\circ(\rho_i^{\ast})^{-1}\circ\rho_i^{\ast}\circ\Phi_{bb'}\circ(\phi_i^{\ast})^{-1}\left(\frac{b_i}{z}\right)=\rho_i^{\dagger}\circ(\rho_i^{\ast})^{-1}\left(e^{\ii\beta_i^{\ast}}\frac{b_i}{z}\right)=\frac{b_i'}{b_i}e^{-\ii\beta_i^{\ast}}z
\end{align*}
and thus $b_i'=b_ie^{\ii(\beta_i^{\dagger}+\beta_i^{\star})}$. Similarly one obtains $a'=ae^{\ii(\alpha^{\dagger}+\alpha^{\ast})}$ by looking at the local representatives of $\varphi$.

Note that the angles $\alpha^{\dagger}$, $\alpha^{\ast}$, $\beta_i^{\dagger}$ and $\beta_i^{\ast}$ are fixed by the given morphism, so these equations fix unique $b'$ and $a'$. Also note that if these relations are satisfied we do indeed get a morphism $(\Phi_{bb'},\varphi_{bb'}):u_b\Longrightarrow v_{b'}$ because the maps extend over the cylinder to a global biholomorphic map, $a'=b_i'^{l_i}$ holds indeed and $v_{b'}\circ\Phi_{bb'}=\varphi_{bb'}\circ u_b$ is true. The latter holds by definition outside of the cylinders and on the cylinder local calculations show that it is again equivalent to \cref{eq:cond-angles} which holds again by assumption.
\end{proof}

\begin{cor}
  Let $u:C\lra X$ be a given Hurwitz cover with a nodal point $p\in X$ and a disc structure around $p$. Then there exists a morphism between $u_b:C_b\lra X_a$ and $u_{b'}:C_{b'}\lra X_{a'}$ which is equal to the identity outside of $\bigsqcup_{i=1}^kY_i(b_i)\subset C_b$  if and only if $b=b'$.
  \label{cor:hurwitz-deformations-are-not-equivalent}
\end{cor}

\begin{proof}
Choose $u=v$ in the last theorem and use the same disc structure everywhere. Then since the morphism is the identity outside of the glued cylinders all angles are zero and we obtain $b=b'$.
\end{proof}

\begin{prop}
  Let two nodal Hurwitz covers $u: C\lra X$ and $v:D\lra Y$ be given together with nodes $p\in X$ and $r\in Y$ together with disc structures around $p$ and $r$. Suppose there exists a morphism $(\Phi_{bb'},\varphi_{bb'}):u_b\Longrightarrow v_{b'}$ such that $\varphi_{bb'}(X_a)=Y_{a'}$ for some $b\in\DD^l$ and $b'\in\DD^l$ and $a=b_i^{k_i}$ as well as $a'=b_i'^{k_i}$. Here $u_b:C_b\lra X_a$ and $v_{b'}:D_{b'}\lra Y_{a'}$ are defined as in \cref{sec:gluing} using the given disc structures and such that $\Phi_{bb'}$ maps the glued cylinders onto each other preserving the enumeration, i.e.\ $\Phi_{bb'}(C_i(b_i))=D_i(b_i')$. Then there exists a unique morphism $(\Phi,\varphi):u\Longrightarrow v$ which coincides with $(\Phi_{bb'},\varphi_{bb'})$ outside of $\bigsqcup_{i=1}^lU_i^{\dagger}\sqcup U_i^{\ast}\subset C$ and $V^{\dagger}\sqcup V^{\ast}\subset D$.
  \label{prop:morphisms-nodal-hurwitz-covers}
\end{prop}

\begin{proof}
First note that the only biholomorphisms of an annulus fixing the boundary components setwise are given by rotations, see e.g.\ Farkas and Kra.\footnote{Without the condition on the boundary components there also exist biholomorphisms that switch these boundary components, namely $z\mapsto \frac{r}{z}$ for $r\leq|z|\leq 1$.}

Recall the proof of \cref{prop1} and the coordinate chart diagrams \cref{diag2} and \cref{diag1}. The local representatives
\begin{align*}
    \rho_i^{\dagger}\circ\Phi_{bb'}\circ (\phi_i^{\dagger})^{-1}&:A(|b|,1)\lra A(|b'|,1)\\
  \rho_i^{\ast}\circ\Phi_{bb'}\circ (\phi_i^{\ast})^{-1}&:A(|b|,1)\lra A(|b'|,1)\\
  \eta^{\dagger}\circ\varphi_{bb'}\circ (\psi^{\dagger})^{-1}&:A(|a|,1)\lra A(|a'|,1)\\
  \eta^{\ast}\circ\varphi_{bb'}\circ (\psi^{\ast})^{-1}&:A(|a|,1)\lra A(|a'|,1)
\end{align*}
are biholomorphisms of annuli preserving the boundary at $|z|=1$ and are thus given by rotations by angles $\beta_i^{\dagger}$, $\beta_i^{\ast}$, $\alpha^{\dagger}$ and $\alpha^{\ast}$, respectively. Again, because the diagram has two commute and because the transition maps of the charts are holomorphic we obtain conditions on these angles as in the proof of \cref{prop1}, namely
\begin{align*}
  b_i'&=b_ie^{\ii(\beta_i^{\dagger}+\beta_i^{\ast})},\\
  a'&=ae^{\ii(\alpha^{\dagger}+\alpha^{\ast})},\\
  e^{\ii\alpha^{\dagger}}&=e^{\ii k_i\beta_i^{\dagger}},\\
  e^{\ii\alpha^{\ast}}&=e^{\ii k_i\beta_i^{\ast}},
\end{align*}
which are again the same conditions for the maps as above.

Therefore we can define a morphism $\Phi:C\lra D$ by $\Phi_{bb'}$ outside of $\bigsqcup_{i=1}^lU_i^{\dagger}\sqcup U_i^{\ast}$ and extend it on $U_i^j$ by $z\longmapsto e^{\ii\beta_i^j}z$ in local coordinates $\phi_i^j$ and $\rho_i^j$ and correspondingly for $\varphi$.
\end{proof}

\begin{rmk}
  Note that the condition for $(\Phi,\varphi)$ to map the glued cylinders onto each other is rather restrictive and does not need to be satisfied for general Hurwitz covers and arbitrary disc structures.
\end{rmk}

\subsection{Analyzing Singularities of the Parametrization}

\label{sec:analyze-singularities}

\begin{lem}
  The set $B\coloneqq \{(b_1,\ldots,b_l)\in\DD^l\mid b_1^{k_1}=\cdots=b_l^{k_l}\}$ is a smooth complex curve except at the origin for $k_1,\ldots,k_l\geq 1$ if at least one degree is strictly bigger than one. Otherwise it is smooth everywhere.
  \label{lem:param-set-b}
\end{lem}

\begin{proof}
  The set from the lemma is the zero set of the map $\mcF:\DD^l\lra\CC^{l-1}$ given by
  \begin{equation*}
    \DD^l\ni b\longmapsto (b_1^{k_1}-b_2^{k_2},b_2^{k_2}-b_3^{k_3},\ldots,b_{l-1}^{k_{l-1}}-b_l^{k_l})\in\CC^{l-1}.
  \end{equation*}
  Since $\mcF$ is holomorphic its differential is $\ii$-linear and thus $\rank_{\RR}\dd_b\mcF=2\rank_{\CC}\dd_b^{\CC}\mcF$ and is given by the matrix
  \begin{equation*}
    \begin{pmatrix}
      k_1b_1^{k_1-1} & -k_2b_2^{k_2-1} & 0 & \cdots & 0 & 0 & 0 \\
      0 & k_2b_2^{k_2-1} & -k_3b_3^{k_3-1} & \cdots & 0 & 0 & 0\\
      0 & 0 & k_3b_3^{k_3-1} & \cdots & 0 & 0 & 0\\
      \vdots & \vdots & \vdots & \ddots & \vdots & \vdots & \vdots \\
      0 & 0 & 0 & \cdots & k_{l-2}b_{l-2}^{k_{l-2}-1} & -k_{l-1}b_{l-1}^{k_{l-1}-1} & 0\\
      0 & 0 & 0 & \cdots & 0 & k_{l-1}b_{l-1}^{k_{l-1}-1} & -k_lb_l^{k_l-1}
    \end{pmatrix}.
  \end{equation*}
  Thus we see
  \begin{equation*}
    \rank_{\RR}\dd_b\mcF=2(l-1)-2\#\{i\in\{1,\ldots,l-1\}\mid k_i>1\text{ and }b_i=0\}
  \end{equation*}
  which implies that $B$ is a smooth complex submanifold of complex dimension one except when one $b_i$ and therefore all $b_i$ are zero.
\end{proof}

\begin{definition}
 Define $K\coloneqq \lcm(k_1,\ldots,k_l)$ and $b_{\xi}\coloneqq b_{\xi_1,\ldots,\xi_l}:\DD\lra B$ by 
 \begin{equation*}
   b_{\xi_1,\ldots,\xi_l}(w)\coloneqq (\xi_1 w^{\frac{K}{k_1}},\ldots,\xi_l w^{\frac{K}{k_l}})
 \end{equation*}
 where the $\xi_i$ are such that $\xi_i^{k_i}=1$ for $i=1,\ldots,l$.
 \label{def:hol-discs}
\end{definition}

\begin{lem}
  Each map $b_{\xi_1,\ldots,\xi_l}:\DD\setminus\{0\}\lra B$ is an injectively immersed holomorphic (punctured) disc.
\end{lem}

\begin{proof}
  Each $b_{\xi}$ is an injective map by \cref{lem5}. As a polynomial in $w$, $b_{\xi}$ is holomorphic as a map from $\DD$ to $\CC^l$. Since the complex structure on $B$ is given by restricting the complex structure on $\CC^l$ we get that $b_{\xi}:\DD\setminus\{0\}\lra B$ is holomorphic. Its differential is given by
  \begin{equation*}
    \dd_wb_{\xi}=\left(\frac{K}{k_1}\xi_1w^{\frac{K}{k_1}-1},\ldots,\frac{K}{k_l}\xi_lw^{\frac{K}{k_l}-1}\right)\dd w
  \end{equation*}
which is injective whenever $w\neq 0$.
\end{proof}

\begin{lem}
  The discs $b_{\xi_1,\ldots,\xi_l}$ cover all of $B$, i.e.\ $\bigcup_{j_1=1,\ldots,j_l=1}^{k_1,\ldots,k_l}b_{\xi_{j_1},\ldots,\xi_{j_l}}(\DD)=B$, where the indices are such that all combinations of roots are covered in \cref{def:hol-discs}.
\end{lem}

\begin{proof}
  In order to show the lemma, first notice that $b_{\xi_1,\ldots,\xi_l}(w)\in B$ because 
  \begin{equation*}
    \left(\xi_iw^{\frac{K}{k_i}}\right)^{k_i}=w^K=\left(\xi_jw^{\frac{K}{k_j}}\right)^{k_j}
  \end{equation*}
  for all $i\neq j$ and $|\xi w^{\frac{K}{k_i}}|\leq 1$. For the other inclusion we need to show that for any given $b=(b_1,\ldots,b_l)$ such that $b_1^{k_1}=\cdots=b_l^{k_l}$ there exist $\xi_1,\ldots,\xi_l$ and $w\in\DD$ such that $b_{\xi_1,\ldots,\xi_l}(w)=b$. For this, denote by $a\coloneqq b_1^{k_1}=\cdots=b_l^{k_l}$ and $a=|a|e^{\ii\alpha}$ with $\alpha\in[0,2\pi]$. Then define $w\coloneqq \sqrt[K]{|a|}e^{\ii\frac{\alpha}{K}}$ and $\xi_i\coloneqq \frac{b_ie^{-\ii\frac{\alpha}{k_i}}}{\sqrt[k_i]{|a|}}$ for $i=1,\ldots,l$. Then we have that $|w|\leq 1$ and $\xi_i^{k_i}=\frac{b_i^{k_i}e^{-\ii\alpha}}{|a|}=1$ as well as
\begin{equation*}
  b_{\xi_1,\ldots,\xi_l}(w)=(\xi_1\sqrt[k_1]{|a|}e^{\ii\frac{\alpha}{k_1}},\ldots,\xi_l\sqrt[k_l]{|a|}e^{\ii\frac{\alpha}{k_l}})=(b_1,\ldots,b_l).
\end{equation*}
This proves the statement.
\end{proof}

\begin{lem}
  Let $\xi_1,\ldots,\xi_l,\eta_1,\ldots,\eta_l\in\CC$ be given such that $\xi_i^{k_i}=\eta_i^{k_i}=1$ for all $i=1,\ldots,l$ and consider $b_{\xi}:\DD\lra B$ and $b_{\eta}:\DD\lra B$, where we abbreviate $b_{\xi}\coloneqq b_{\xi_1,\ldots,\xi_l}$ and $b_{\eta}\coloneqq b_{\eta_1,\ldots,\eta_l}$. Then the following are equivalent:
  \begin{enumerate}[label=(\roman*), ref=(\roman*)]
    \item There exist nonzero $v,w\in\DD$ such that $b_{\xi}(v)=b_{\eta}(w)$. \label{lem9-item1}
    \item We have $\Im (b_{\xi})=\Im (b_{\eta})$. \label{lem9-item2}
    \item There exist $\alpha_i\in S^1$ such that $\xi_i=\alpha_i\eta_i$ for all $i=1,\ldots,l$ and a solution $\zeta\in S^1$ of the system of equations $\alpha_i=\zeta^{\frac{K}{k_i}}$ for $i=1,\ldots,l$. \label{lem9-item3}
  \end{enumerate}
\label{lem9}
\end{lem}

\begin{proof}
  \cref{lem9-item1}$\Longrightarrow$\cref{lem9-item2}: Suppose there exist nonzero $v,w\in\DD$ such that $b_{\xi}(v)=b_{\eta}(w)$ which implies $\xi_iv^{\frac{K}{k_i}}=\eta_iw^{\frac{K}{k_i}}$ for $i=1,\ldots,l$. If $v'\in\DD$ then there exists $\lambda\in\CC$ such that $v'=\lambda v$. Then we have
  \begin{align*}
    b_{\xi}(v')&=b_{\xi}(\lambda v)=(\xi_1(\lambda v)^{\frac{K}{k_1}},\ldots,\xi_l(\lambda v)^{\frac{K}{k_l}})=(\lambda^{\frac{K}{k_1}}\xi_1 v^{\frac{K}{k_1}},\ldots,\lambda^{\frac{K}{k_l}}\xi_l v^{\frac{K}{k_l}})\\
    &=(\lambda^{\frac{K}{k_1}}\eta_1 w^{\frac{K}{k_1}},\ldots,\lambda^{\frac{K}{k_l}}\eta_l w^{\frac{K}{k_l}})=b_{\eta}(\lambda w)
  \end{align*}
  which shows $\Im(b_{\xi})\subset\Im(b_{\eta})$ and the other inclusion follows in the same way.

  \cref{lem9-item2}$\Longrightarrow$\cref{lem9-item3}: Since $b_{\xi}:\DD\setminus\{0\}\lra B$ is a holomorphic injective immersion between complex curves its inverse (defined on its image) is also holomorphic and thus $b_{\xi}^{-1}\circ b_{\eta}:\DD\setminus\{0\}\lra\DD\setminus\{0\}$ is a bounded holomorphic map defined on the punctured disc and can therefore be extended to a biholomorphism $b_{\xi}^{-1}\circ b_{\eta}:\DD\lra\DD$. Since this map also fixes the origin it must be a rotation by some $\zeta\in S^1$. This means that for all $v\in\DD$ we have $b_{\xi}(\zeta v)= b_{\eta}(v)$ which implies for $i=1,\ldots,l$ that $\xi_i(\zeta v)^{\frac{K}{k_i}}=\eta_i v^{\frac{K}{k_i}}$. Since this holds also for $v\neq 0$ we obtain $\xi_i \zeta^{\frac{K}{k_i}}=\eta_i$. Define $\alpha_i\coloneqq \zeta^{\frac{K}{k_i}}$ which satisfy $\alpha^{k_i}=1$ because $\xi_i^{k_i}=\eta_i^{k_i}=1$ and $\zeta^{K}=\alpha_i^{k_i}=1$.

  \cref{lem9-item3}$\Longrightarrow$\cref{lem9-item1}: We want to show that there exists a point $v\in\DD$ such that $b_{\xi}(1)=b_{\eta}(v)$. This equation is equivalent to $\xi_i=\eta_iv^{\frac{K}{k_i}}$ for $i=1,\ldots,l$. Since $\xi_i=\alpha_i\eta_i$ by assumption we need to find a solution to the equations $\alpha_i=v^{\frac{K}{k_i}}$ for $i=1,\ldots,l$ which is again possible by (iii).
\end{proof}

\begin{rmk}
  Note that condition (iii) implies $\alpha_i^{k_i}=1$ and $\zeta^K=1$ because $\xi_i^{k_i}=\eta_i^{k_i}$. 
	%Note also that given numbers $\alpha_1,\ldots,\alpha_k$ a solution to $\alpha_i=\zeta^{\frac{K}{l_i}}$ is unique because if there was another solution $\zeta'$ we had $1=\left(\frac{\zeta}{\zeta'}\right)^{\frac{K}{l_i}}$ for $i=1,\ldots,l$ which implies $\zeta'=\zeta$ by the proof of \cref{lem5}.
\end{rmk}

\begin{definition}
  Define $\mcP\coloneqq \{(\xi_1,\ldots,\xi_l)\in\CC^l\mid \xi_i^{k_i}=1\}$. The group $\ZZ/K\ZZ$ of $K$-th unit roots acts on $\mcP$ by $\zeta\circ\xi\coloneqq \left(\xi_1\zeta^{\frac{K}{k_1}},\ldots,\xi_l\zeta^{\frac{K}{k_i}}\right)$.
\end{definition}

\begin{cor}
  We have $\Im(b_{\zeta\circ\xi})=\Im(b_{\xi})$.
\end{cor}

\begin{proof}
  By \cref{lem9} the images of two discs $b_{\xi}$ and $b_{\eta}$ are equal if and only if there exists $\zeta\in S^1$ such that $\xi=\zeta\circ\eta$.
\end{proof}

\begin{lem}
  The group $\ZZ/K\ZZ$ acts freely on $\mcP$.
\end{lem}

\begin{proof}
  Suppose $\zeta\circ\xi=\xi$. Then $\xi_i\zeta^{\frac{K}{k_i}}=\xi_i$ and thus $\zeta^{\frac{K}{k_i}}=1$ for $i=1,\ldots,l$. By the proof of \cref{lem5} this implies $\zeta=1$.
\end{proof}

\begin{cor}
  The set $B$ is parametrized by $\frac{k_1\cdots k_l}{K}$ injective holomorphic discs which are pairwise disjoint if you remove the origin of the disc.
  \label{cor:param-B-hol-discs}
\end{cor}

\begin{proof}
  The set $B$ is covered by $|\mcP|=k_1\cdots k_l$ injective holomorphic discs which are equal if and only if they lie on the same $\ZZ/K\ZZ$-orbit which consists of exactly $K$ elements by the last lemma.
\end{proof}

\section{Hyperbolic Description Of Gluing}

\label{sec:hyperbolic-gluing}
\index{Hurwitz Deformation Family}

This section deals with the construction of the \emph{Hurwitz deformation family} which will be used to construct the orbifold version of $\mcR_{g,k,h,n}(T)$.

\subsection{Setup}

\label{sec:choices-hyp}

Now we choose all those objects that we need in order to define a parametrization of a neighborhood of a given nodal branched cover $u:C\lra X$. The last point will depend on a compact set $K$ on the surface $X$ which will be chosen after the first three steps\footnote{The purpose of this set is to make sure that the disc structure on $X$ will be disjoint from all hyperbolic geodesics in the slightly perturbed hyperbolic structures nearby.}.

\index{Hyperbolic Deformation Preparation|(}

\begin{definition}
  Given an admissible cover $u:C\lra X$ we make the following choices.
\begin{enumerate}[label=(\roman*), ref=(\roman*)]
\item A hyperbolic metric $g$ on every connected component of the normalization $\wt{X}$ of $X$ such that all special points\footnote{Recall that special points include marked points and nodes.} are cusps and which is compatible with its complex structure and orientation,
\item a set of decomposing curves on $\wt{X}$, i.e.\ a multicurve $\Gamma$ of closed simple hyperbolic geodesics such that $\wt{X}\setminus\Gamma$ consists of (open) hyperbolic pairs of pants,
\item a completion of the multicurve $u^{-1}(\Gamma)$ to a set of decomposing curves of $\wt{C}$ for its hyperbolic metric. Note that preimages of closed simple hyperbolic geodesics under a local isometry are possibly disconnected simple closed geodesics, so $u^{-1}(\Gamma)$ can indeed be completed in this way. Also note that the preimage of a pair of pants under a degree-$d$ cover has Euler characteristics $-d$ and may be disconnected. Having a pair of pants decomposition we choose
%\item a compact set $\mcK\subset X\setminus\{p^{\ast},p^{\dagger}\}$ which includes the boundary horocycles of the cusp neighborhoods of $p^{\ast}\in \wt{X}$ and $p^{\dagger}\in\wt{X}$ and finally
\item a disc structure $(V^j,U_i^j,\phi_i^j,\psi^j)$ at every node $p\in X$ such that the $V^j$ are contained in the interior of the cusp neighborhood of the nodal points $p^{\dagger}$ and $p^{\ast}$.
\end{enumerate}
We call such a choice a \emph{hyperbolic deformation preparation}.
\end{definition}

\index{Hyperbolic Deformation Preparation|)}

\begin{rmk}
  A few comments are in order.
  \begin{enumerate}[label=(\roman*), ref=(\roman*)]
    \item We will refer to a hyperbolic metric on $\wt{X}\setminus\{\text{special pt.}\}$ as a hyperbolic metric on $\wt{X}$ with cusps in the special points in order to avoid introducing yet another symbol. It will always be clear from the context where the hyperbolic metric is actually defined.
    \item Note that the hyperbolic metric on $\wt{X}$ induces a unique hyperbolic metric on $\wt{C}$ such that $u$ is a local isometry compatible with its complex structure and orientation such that the special points on $\wt{C}$ are cusps. This is because we can pull back the hyperbolic metric outside of the special points which include all branched points so it is an actual cover.
    \item The disc structure at a node is chosen to be contained in the interior of a cusp neighborhood. The reason for this is that when we vary the hyperbolic structure the maximal cusp neighborhood will change but we want to keep the fixed disc structure within a cusp region.
  \end{enumerate}
\end{rmk}

\subsection{Local Parametrizations}

\label{sec:local-parametrizations}

Consider a branched covering $u:C\lra X$ in $\obj\mcR_{g,k,h,n}(T)$ of type $T$ including enumerations of the critical and marked points with one node $p\in X$ singled out as in \cref{sec:moduli-spaces-3}. We will define a map $\Psi:\mcU\lra\obj\mcR_{g,k,h,n}(T)$ with $\mcU\subset\RR^{6h-6+2n}$ and $u\in\Psi(\mcU)$ in the following steps:
\begin{enumerate}[label=(\roman*), ref=(\roman*)]
  \item Vary the complex structure on $X$ away from the node.
  \item Pull back the uniformized hyperbolic structure on $X$ via $u$.
  \item Glue in the complex cylinders as in \cref{sec:analyze-singularities}.
  \item Modify the domain of parametrization by using Fenchel--Nielsen coordinates.
\end{enumerate}
This family of Hurwitz covers constructed this way will later be used to define a manifold structure on the sets of objects and morphisms in $\mcR_{g,k,h,n}(T)$. The last step is used to obtain the symplectic Weil--Petersson structure on the resulting orbifold. Note that the word ``local'' means in this context that the construction takes place at one node in $X$ but we will of course later on repeat this construction at all nodes in $X$. So in the following we will assume that the target surface has only one node and then generalize later to more than one node.

\subsubsection{Varying the complex structure on \texorpdfstring{$X$}{X} away from the node}

\label{sec:vary-comp-str-away-from-node}

First do steps one to three in \cref{sec:choices-hyp}, i.e.\ choose a hyperbolic metric $g$ of the normalization $\wt{X}$ and a set of decomposing curves $\Gamma$ on $\wt{X}$ as well as a completion of the multicurve $u^{-1}(\Gamma)$ to a set of decomposing curves on $\wt{C}$.

Now consider the space $\mcQ\coloneqq \mcT_{\wt{X},Z\cup \{p^{\ast},p^{\dagger}\}}$ which is the Teichmüller space of the surface that we obtain if we normalize the node $p$ and consider the two new nodal points as marked points. Note that this space is a manifold as it is either a usual Teichmüller space if the normalization is connected or the product of two Teichmüller spaces if the node was separating. In both cases the corresponding Teichmüller space carries a universal family $\pi:\mcC\lra\mcQ$ and is of dimension $6h-8+2n$. This universal family is again locally trivial as was described in \cref{sec:teichmueller-spaces}. Also it carries a continuous fibre metric $\eta$ which on each fibre is the hyperbolic metric uniformizing the complex structure of the underlying fibre, see \cref{prop:universal-teichmueller-family-vertical-hyp-metric}.

The complex curve $X$ defines a point $[\wt{X}]\in\mcQ$ and so we can choose Fenchel--Nielsen coordinates $\Phi:U\lra V\subset\RR^{6h-8+2n}$ in a small neighborhood around $[X]\in U\subset\mcQ$. For these coordinates we use the earlier chosen set of decomposing curves. Using a trivialization $\Xi:\pi^{-1}(U)\lra U \times \wt{X}$ of the universal Teichmüller curve in a neighborhood of $[\wt{X}]$ we obtain for each $t\in V$ a hyperbolic structure on $\wt{X}$ without special points by defining
\begin{equation*}
  g_t\coloneqq \left(\Xi^{-1}\circ\iota_t\right)^*\eta,
\end{equation*}
where $\iota_t:\wt{X}\lra U\times\wt{X}$ is given by $\iota_t(p)=(t,p)$ and $\eta$ is the fibre metric as above. Thus we can now associate to each $t\in U$ a hyperbolic surface $(\wt{X},g_t)$. This surface has various cusps, two of which are interpreted as a nodal pair.

\subsubsection{Pulling back the hyperbolic structure}

Given the hyperbolic metric $g_t$ and the map
\begin{equation*}
  u:\wt{C}\setminus\{\text{special points}\}\lra X\lra\wt{X}\setminus\{\text{cusps}\}
\end{equation*}
we pull back the hyperbolic metric $g_t$ from $\wt{X}$ to $\wt{C}$ outside of the cusps. This map is by definition a local isometry and thus holomorphic for the corresponding complex structures. Denote the new complex curve by $\wt{C}_{u^*t}$ which is just the surface $\wt{C}$ with a new complex structure. Note that the pulled-back hyperbolic metric is again complete, has finite area and is compatible with the pulled-back complex structure outside the cusps and thus has cups at the punctures.

This construction does of course preserve the local degrees of the map $u$ as we only modified the structures on the surfaces and thus all the properties of a Hurwitz covering are preserved. Switching back to the nodal picture we have thus constructed for each $t\in U$ a nodal Hurwitz covering $u:C_{u^*t}\lra X_t$.

\subsubsection{Gluing of cylinders}

Recall from \cref{sec:analyze-singularities} that a neighborhood of nodal branched covers in Deligne--Mumford space of the source surface has possible gluing parameters at the $l$ nodes in a fibre parametrized by $\frac{k_1\cdots k_l}{K}$ discs which are pairwise disjoint except at the central point. Thus we fix one such disc $b_{\xi}\coloneqq b_{\xi_1,\ldots,\xi_l}:\DD\lra B\subset \CC^l$ and describe one deformation for every such disc. Next we build a ``mixed'' surface $X_{(t,a)}$ for each $(t,a)\in U\times\DD$ as follows.

Recall that we have the universal curve over Teichmüller space $\pi:\mcC\lra\mcQ$ together with a trivialization over $U\subset\mcQ$ and the hyperbolic fibre metric $\eta$. We choose a disc structure in the interior of the cusp neighborhood around the nodes. This way, by restricting to a sufficiently small neighborhood $U$ of $[\wt{X}]\in\mcQ$ we can make sure that the disc structure is still contained in the cusp neighborhood. To see this notice that $\eta$ is continuous and that by \cref{lem:unique-cusp-neighborhood} the cusp neighborhood is bounded by a horocycle of length two. This implies that the boundary of the cusp neighborhood depends continuously on the point in $U$ and therefore if $U$ is small enough this boundary does not intersect a disc structure in the interior of the cusp region.

Next we glue in cylinders with parameter $a$ as in \cref{eq:glue-complex-cylinder} for the various fibres $X_t$ in the family for $t\in U$. This way we obtain complex surfaces $X_{(t,a)}$ for $(t,a)\in U\times\DD$.

At the same time we glue in cylinders using $b_{\xi(z)}$ as a parameter in the surface $C_{u^*t}$ and define the map $u_{z}:C_{u^*t,b_{\xi(z)}}\lra X_{t,a(z)}$ as in \cref{def:deformation-of-hurwitz-covers} with $a(z)\coloneqq b_{\xi}(z)_i^{k_i}$ for any $i=1,\ldots,l$.

\subsubsection{Modifying the domain of parametrization}

\label{sec:modify-doma-param}

On the surfaces $X_{t,a(z)}$ we now have a set of decomposing curves by taking $\Gamma$ and adding a curve wrapping once around the glued-in cylinder. The surfaces fit together in a flat family $Y\lra U\times\DD$ with $Y\coloneqq\bigsqcup_{(t,z)\in U\times\DD}X_{t,a(z)}$ because $Y$ is just the pull back of the plumbing family constructed in \cite{robbin_construction_2006} and \cite{hubbard_analytic_2014} along the holomorphic map $U\times\DD\lra U\times\DD$ given by $(t,z)\mapsto (t,a(z))$. A flat family in our case is a proper holomorphic map between complex manifolds having at most nodal points.

Together with the set of decomposing curves we thus obtain Fenchel--Nielsen coordinates $\Phi:U\times\DD\lra \RR^{6h-6+2n}$, see Proposition~6.1 in \cite{hubbard_analytic_2014}. After shrinking the domain of $\Phi$ if necessary this map is a homeomorphism onto its image by Theorem~9.11 in \cite{hubbard_analytic_2014}. Thus we obtain an open set $\mcU\subset\RR^{6h-6+2n}$ such that the inverse $\Phi^{-1}:\mcU\lra U\times\DD$ is well-defined and thus we can define $\Psi_{\xi,u}$ by parametrizing a family of Hurwitz covers via
\begin{alignat}{5}
  \Psi_{\xi,u}:\mcU & \lra & \quad U\times\DD & \lra && \Ob\mcR_{g,k,h,n}(T) \nonumber\\
  w & \longmapsto & \Phi^{-1}(w) & && \label{eq:def-fn-map} \\
  & & (t,z) & \longmapsto && (u_{z}:C_{u^*t,b_{\xi}(z)}\lra X_{t,a(z)}). \nonumber
\end{alignat}
Note that this parametrization is defined in such a way that the point in $\mcU$ does indeed specify the Fenchel--Nielsen coordinates of the set of decomposing curves of the target surface, although we used the complex formulation as an intermediate step. The definition of the map $u_z$ was stated in \cref{def:deformation-of-hurwitz-covers}.

\begin{figure}[!ht]
  \centering
  \def\svgwidth{\textwidth}
  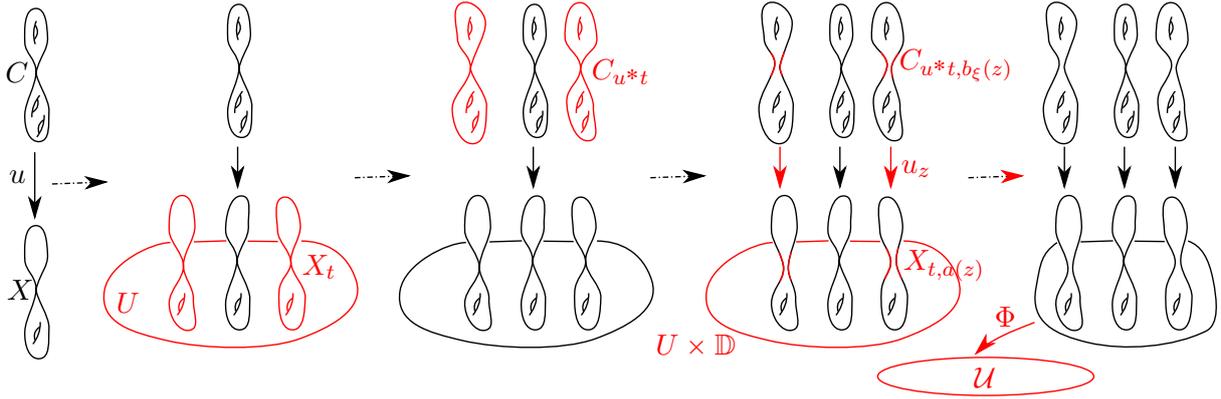
  \caption{This diagram illustrates the process how we constructed the Hurwitz deformation. First we vary the hyperbolic structure on the smooth part of $X$, then we pull this back to $C$, next we glue in the local construction from \cref{sec:complex-gluing} and then we reparametrize by using Fenchel--Nielsen coordinates as in \cite{hubbard_analytic_2014}.}
  \label{fig:construction-hurwitz-family}
\end{figure}

\begin{rmk}
\label{rmk:unfoldings-source-surface}
Let us node that we can of course repeat the construction at every note of the target surface independently as the choices for the hyperbolic deformation preparation were such that all the constructions happen ``far away'' from each other. This means that if the target $X$ has $0\leq l \leq 3h-3+n$ nodes then we obtain a map
\begin{alignat*}{5}
  \Psi_{\xi,u}:\mcU & \lra & \quad U\times\DD^l & \lra && \Ob\mcR_{g,k,h,n}(T) \\
  w & \longmapsto & \Phi^{-1}(w) & && \\
  & & (t,z) & \longmapsto && (u_{z}:C_{u^*t,b_{\xi}(z)}\lra X_{t,a(z)})
\end{alignat*}
where $U$ is a product of the Teichmüller spaces of the smooth components of $X$.
\end{rmk}

\begin{rmk}
  \label{rmk:definition-all-deformations-of-source-surface}
  By constructing this family of Hurwitz covers we also build families of nodal Riemann surfaces, namely
  \begin{align*}
    \bigsqcup_{(t,z)\in U\times\DD}C_{u^*t,b_{\xi}(z)}\lra U\times\DD^l
  \end{align*}
  for the source surface and
  \begin{align*}
    \bigsqcup_{(t,z)\in U\times\DD}X_{t,a(z)}\lra U\times\DD^l
  \end{align*}
  for the target surface. Note that the latter one was built in precisely the same way as the universal unfolding of the target surface, i.e.\ this is a universal unfolding of $X$, see \cref{sec:some-techn-lemm} for more explanations. Anyway, suppose we have fixed a nodal Hurwitz cover $u:C\lra X$ and equipped it with a hyperbolic deformation preparation together with maps $b_{\xi_p}:\DD\lra B_p\subset\CC^{\#\{\text{nodes over } p\}}$ for every node $p$ in $X$. Then we obtain an unfolding of the source surface $C$ and thus it comes with a unique holomorphic map into the universal unfolding $B^C$ of $C$ in a small neighborhood of the central point. We define the set $\mcA(u:C\lra X)\subset B^C$ or $\mcA(u)$ as the union of all these images over all possible choices of discs $b_{\xi_p}$ for all the nodes $p\in X$. Thus $\mcA(u:C\lra X)\subset B^C$ describes the set of complex structures close to $C$ constructed via the above Hurwitz deformations. We will see in \cref{lem:nod-hur-cov-unique} that the germ of this set is indeed independent of the choices.
  \label{rmk:unfoldings-surfaces}
\end{rmk}

\section{Construction of an Orbifold Structure on the Branched Cover of the Moduli Space of Closed Hurwitz Covers}

\sectionmark{Orbifold Structure on Moduli Space of Hurwitz Covers}

\label{sec:constr-an-orbif}

\subsection{Manifold Structure on the Set of Objects}

The purpose of this section is to use the prior gluing constructions to define an orbifold structure for the category $\mcR_{g,k,h,n}(T)$. Thus we will define an orbifold category $\mcM_{g,k,h,n}(T)$ with an homeomorphism from its orbit space to $|\mcR_{g,k,h,n}(T)|$. We will define an ep Lie-groupoid with compact orbit space which will be a morphism covering over the moduli space of complex structures of the target surface and which will have a continuous map to the actual moduli space of Hurwitz covers. This map will be branched in a certain way over isomorphism classes of nodal Hurwitz covers and comes from the local description in \cref{sec:analyze-singularities}. This chapter deals with the manifold structure on the objects.

So suppose we are given the following objects:
\begin{itemize}
  \item A countable set $\Lambda$ of Hurwitz covers.
  \item For any $\lambda\in\Lambda$ a family of Hurwitz covers $\Psi^{\lambda}$ defined on $O^{\lambda}$ which is constructed as in \cref{sec:local-parametrizations} deforming the Hurwitz cover $\lambda$.
\end{itemize}

Here, the index $\lambda$ actually contains a bit of extra data besides the (central) Hurwitz cover $u^{\lambda}:C^{\lambda}\lra X^{\lambda}$ necessary for defining $\Psi^{\lambda}$. This means that $\lambda$ includes
\begin{itemize}
  \item a hyperbolic deformation preparation (see \cref{sec:choices-hyp}),
  \item a choice of $\xi\in \CC^{m_p}$ with $\xi^{k_i}=1$ for $i=1,\ldots,m_p$ for each node $p$ in the target space where $m_p$ is the number of nodal preimages of the node $p$ in $C^{\lambda}$ and\footnote{In particular these values $\xi$ and $m_p$ depend on $\lambda$ and are not fixed in any way by the choice of combinatorics except that there are fewer than $3h-3+n$ nodes in the target and every node has at most $d$ preimages.}
  \item an enumeration of the branched and nodal points as well as 
  \item open neighborhoods in the Teichmüller spaces of the normalization of $X^{\lambda}$.
\end{itemize}

Since $|\mcM_{h,n}|$ is covered by finitely many strata which are manifolds and since every surface $X$ has only finitely many equivalence classes of Hurwitz covers of fixed type $T$ it is possible to choose a finite set $\Lambda$ such that all equivalence classes in $|\mcR_{g,k,h,n}(T)|$ are covered by $\bigcup_{\lambda\in\Lambda}|\Psi^{\lambda}|$. Also note that given such a set we can of course refine the covering by enlarging $\Lambda$ such that the sets $O^{\lambda}$ become smaller. This will be necessary because we need to choose the $O^{\lambda}$ such that a certain set of fibre-isomorphisms from $\Psi^{\lambda}(O^{\lambda})$ to $\Psi^{\lambda'}(O^{\lambda'})$ becomes a manifold, see \cref{sec:mfd-structure-morphisms}.

So, up to possible later refinements of the choice of $\Lambda$ we make the following definition.
\begin{definition}
  The object set of the orbifold category of Hurwitz covers of type $T$, denoted by $\mcM_{g,k,h,n}(T)$, is defined by
  \begin{equation*}
    \Ob\mcM_{g,k,h,n}(T)\coloneqq\bigsqcup_{\lambda\in\Lambda}O^{\lambda}.
  \end{equation*}
  Note that by construction the sets $O^{\lambda}$ already have a topology and even two (non-equivalent) manifold structures: First, $O^{\lambda}\subset\RR^{6h-6+2n}$ are values of Fenchel--Nielsen coordinates and secondly the map $\left(\Phi^{\lambda}\right)^{-1}:O^{\lambda}\lra U^{\lambda}\times\DD^{\ol{k}}$ from the definition of the family $\Psi^{\lambda}$ in \cref{eq:def-fn-map} is a local homeomorphism and $U^{\lambda}\times\DD^{\ol{k}}$ is a manifold because $U^{\lambda}$ is a neighborhood in a Teichmüller space of the normalization of $X^{\lambda}$. The index $\ol{k}$ is just used as a place holder for the number of nodes on $X^{\lambda}$ as we need a disc for describing the opening of every node in the target surface. By \cref{rmk:different-differentiable-structures-deligne-mumford} the map $\Phi^{\lambda}$ is a homeomorphism but it is not differentiable on the locus of nodal Hurwitz covers. We can thus distinguish between the Fenchel--Nielsen differentiable structure and the complex one via gluing discs $\DD^{\ol{k}}$.
\end{definition}

\subsection{Manifold Structure on the Set of Morphisms}

\label{sec:mfd-structure-morphisms}

\index{Fibre Isomorphism}

\begin{definition}
  A triple $(b,(\Phi,\phi),c)$ with $b\in O^{\lambda},c\in O^{\lambda'}$ and a morphism of Hurwitz covers $(\Phi,\phi):u^{\lambda}_b\Rightarrow u^{\lambda'}_c$ is called a \emph{fibre isomorphism}. Given two families $\Psi^{\lambda}:O^{\lambda}\lra\Ob\mcR_{g,k,h,n}(T)$ and $\Psi^{\lambda'}:O^{\lambda'}\lra\Ob\mcR_{g,k,h,n}(T)$ deforming two Hurwitz covers $u^{\lambda}:C^{\lambda}\lra X^{\lambda}$ and $u^{\lambda'}:C^{\lambda'}\lra X^{\lambda'}$ as in \cref{sec:local-parametrizations} we define
  \begin{equation*}
    M(\lambda,\lambda')\coloneqq \{(b,(\Phi,\phi),c)\mid b\in O^{\lambda},c\in O^{\lambda},(\Phi,\phi):u^{\lambda}_b\Rightarrow u^{\lambda'}_c\}.
  \end{equation*}
\end{definition}

\begin{rmk}
  We would like to show that these sets $M(\lambda,\lambda')$ are manifolds and the obvious structure maps from $\mcR_{g,k,h,n}(T)$ are smooth as we could then define
  \begin{equation*}
    \Mor\mcM_{g,k,h,n}(T)\coloneqq\bigsqcup_{(\lambda,\lambda')\in\Lambda\times\Lambda}M(\lambda,\lambda')
  \end{equation*}
  in order to obtain an orbifold structure for $|\mcR_{g,k,h,n}(T)|$. Unfortunately this is not the case because two deformations of the same Hurwitz cover using different maps $b_{\xi}$ (see \cref{sec:analyze-singularities}) have an isomorphism in the middle fibre which cannot be extended to the other fibres. Thus these $M(\lambda,\lambda')$ will sometimes be lower dimensional as they are isolated in the directions of the variation coming from the opening of the node. So in order to get an ep-Lie groupoid we need to restrict to those pairs $(\lambda,\lambda')$ such that $M(\lambda,\lambda')$ has the correct dimension. The resulting orbit space $|\mcM_{g,k,h,n}(T)|$ then has an orbifold structure with a continuous map to the actual moduli space of Hurwitz covers $|\mcR_{g,k,h,n}(T)|$ which will be however not a homeomorphism but ``branched'' over the locus of nodal Hurwitz covers, see \cref{sec:main-results}.
\end{rmk}

\begin{prop}
  Given two Hurwitz deformation families $\Psi^{\lambda}$ and $\Psi^{\lambda'}$ the set $M(\lambda,\lambda')$ is naturally a manifold and its components are either of the same dimension as $\Ob\mcM_{g,k,h,n}(T)$, i.e.\ $6h-6+2n$, or at least two dimensions less.
  \label{prop:manif-struct-set}
\end{prop}

The proof of this proposition is the goal of this \cref{sec:mfd-structure-morphisms}.

\subsubsection{Universal Unfoldings}

\label{sec:some-techn-lemm}

Before we begin with the proof of \cref{prop:manif-struct-set} we will cite two theorems from Robbin--Salamon~\cite{robbin_construction_2006} and then prove some lemmas that we need for the proposition.

\begin{prop}
  Every closed stable marked Riemann surface $C$ has a unique universal marked nodal unfolding $(\pi^C:Q^C\lra B^C,S_*,b)$ which
  \begin{itemize}
    \item consists of connected complex manifolds $Q^C$ and $B^C$ such that $\pi^C$ is a surjective proper holomorphic map and $\dim_{\CC} Q^C=\dim_{\CC}B^C+1$,
    \item every critical point of $\phi^C$ is \emph{nodal},
    \item $b\in B^C$ and $(\pi^C)^{-1}(b)= C$ is the central fibre and
    \item $S_*=(S_1,\ldots,S_n)$ are pairwise disjoint complex submanifolds of $Q^C$ which are mapped by $\pi^C$ diffeomorphically onto $B^C$.\footnote{These submanifolds correspond to marked points of the fibres.}
    \item Such a marked nodal unfolding $(\pi^C:Q^C\lra B^C)$ is called \emph{universal} if for every other nodal unfolding $\pi^D:Q^D\lra B^D$ as above and any fibre isomorphism $f:C\lra D$ there exits a unique germ of a morphism $(\Phi,\phi):(\pi^C,b^C)\lra(\pi^D,b^D)$ such that $\Phi(S^C_i)\subset S^D_i$ for all $i=1,\ldots,n$ extending $f$ on the central fibres. See \cite{robbin_introduction_2014} for details of the definition of fibre isomorphism, morphism and germ in this setting.
  \end{itemize}
  Here, uniqueness of the universal marked nodal unfolding has a particular meaning that we do not need to investigate further for our purposes.
  \label{prop:summary-properties-universal-unfoldings}
\end{prop}

\begin{proof}
  These are the main theorems~5.5 and 5.6  in \cite{robbin_introduction_2014}.
\end{proof}

\begin{prop}
  Let $\pi^C:Q^C\lra B^C$ be a universal unfolding of a closed nodal Riemann surface $C$. There exist coordinates in a small neighborhood of the base point $b^C\in B^C$ in $\mcU\times \DD^m$ where $\mcU$ is an open neighborhood of the product of Teichmüller spaces of the smooth components of $C$ and $m$ is the number of nodes in $C$ by the usual gluing construction. All the necessary data for the definition in Robbin--Salamon is included in the choice of a disc structure. Furthermore there exist local Fenchel--Nielsen coordinates on $B^C$ induced by the choices in a hyperbolic preparation deformation.
  \label{prop:univ-unfold-coord}
\end{prop}

\begin{proof}
  See the proof of Theorem~5.6 in \cite{robbin_introduction_2014} and notice that the construction of $Y_{\Gamma}\lra\mcP_{\Gamma}$ in \cite{hubbard_analytic_2014} corresponds to the one in Robbin--Salamon so we can use Proposition~9.1 in \cite{hubbard_analytic_2014}. However, the Fenchel--Nielsen coordinates give a non-equivalent smooth structure which is ``only'' diffeomorphic to the other one, so in particular the Fenchel--Nielsen coordinates are not smooth in the complex coordinates.
\end{proof}

\subsubsection{Comparing Hurwitz Deformations in Universal Unfolding of the Source } 

\begin{lem}
  Let $u:C\lra X$ be a nodal Hurwitz cover with $m$ nodes $p_1,\ldots,p_m\in X$ and a hyperbolic deformation preparation around these nodes. Furthermore denote the degrees of $u$ at the nodes over $p_i$ by $k_i^j\in\NN$ for $j=1,\ldots,m_i$. Here, a node $p_i$ has $m_i$ nodal preimages. There exists a neighborhood $\mcU^X\subset B^X$ and a neighborhood $\mcU^C\subset B^C$ such that for every smooth surface $Y\in B^X$ there exist at most 
  \begin{equation*}
    \prod_{i=1}^mk_i^1\cdots k_i^{m_i}
  \end{equation*}
  surfaces $D\in\mcU^C$ with the property that there exists a Hurwitz cover $v:D\lra Y$.
  \label{lem:nod-hur-cov-unique}
\end{lem}

\begin{proof}
  Use the disc structure to obtain Fenchel--Nielsen coordinates on a neighborhood $\mcU^X$ of the target surface. This includes in particular a choice of a set of decomposing curves as well as some specified geodesics perpendicular to the boundary geodesics on the pairs of pants. Once these objects are chosen we lift the geodesics bounding the pairs of pants to $C$ and complete this set of simple closed curves to a decomposing one and choose geodesic representatives in the additional free homotopy classes. Also we choose lifts of the specified perpendicular line segments to the pairs of pants containing the preimages of the new geodesics arising from the nodes in $C$. Denote all curves as in the \cref{fig:pairs-of-pants-gluing}. 

Note that we do this at every node in $X$ but in order to simplify notation we deal with every node separately. Thus we drop the index $i$ for the $m$ nodes and instead enumerate the preimages of one arbitrary but fixed node with index $j=1,\ldots,l$. We denote the degrees at the $j$-th preimage node by $k^j$.

\begin{figure}[H]
    \centering
    \def\svgwidth{0.6\textwidth}
    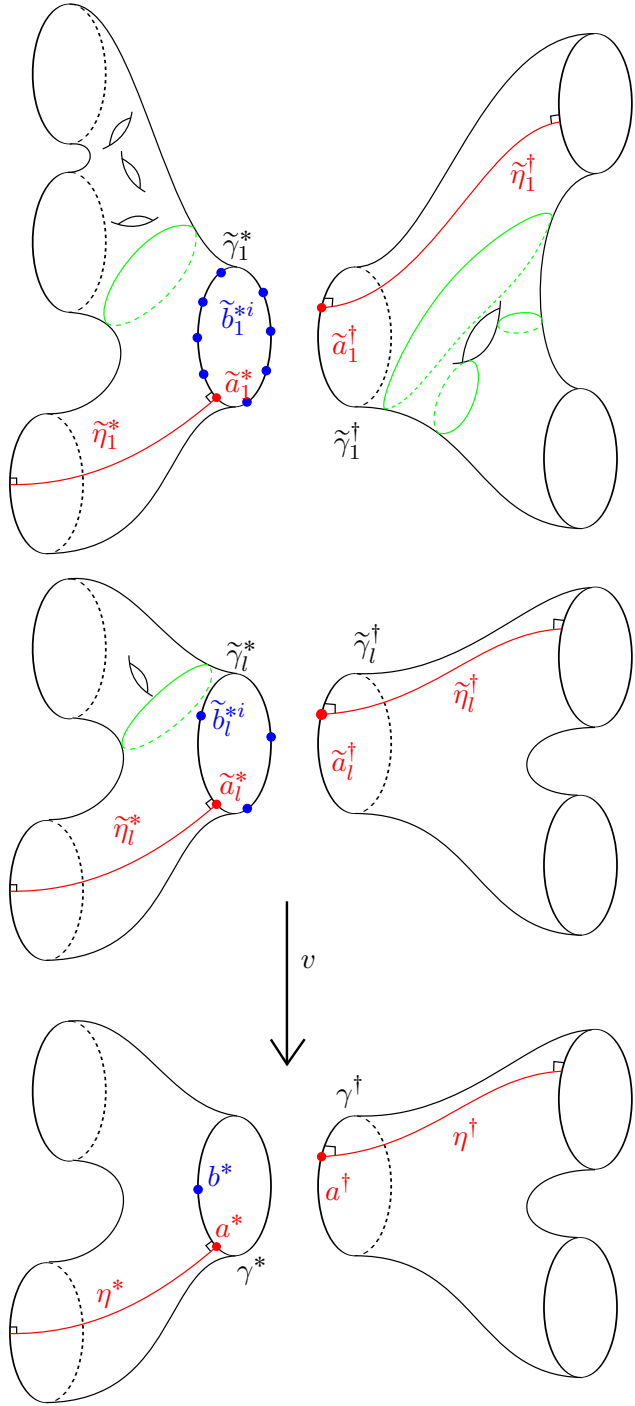
    \caption{Chosen curves and points for the proof of \cref{lem:nod-hur-cov-unique}. Depicted are only the pairs of pants bordering the node and the hyperbolic geodesic in the free homotopy class of the node and its preimages. Objects with a tilde on top come from lifts from the target. In particular the point $b^{\ast}$ is the point which is identified with $a^{\dagger}$ in the chosen target surface and the other blue points $\wt{b}_j^{\ast i}$ are all lifts of $b^{\ast}$. Here the lower index refers to the number of the geodesic and the upper index runs from $1$ to $k^j$ for $\wt{\gamma}_j^{\ast}$, where $k^j$ is the degree of $u:C\lra X$ at the $j$-th preimage of the node under investigation.}
    \label{fig:pairs-of-pants-gluing}
\end{figure}

Now consider a Hurwitz covering $v:D\lra Y$ such that $D$ and $Y$ are contained in small neighborhoods of the universal unfoldings of $C$ and $X$, respectively and such that the node $p\in X$ was smoothened to a curve in $Y$ and thus also its preimages are smooth. This means that the length of the geodesic representative $\gamma$ of the free homotopy class corresponding to $p$ is close to zero. As the covering $v$ is an actual covering close to this curve and thus isometric in the uniformized hyperbolic metrics, the length of $\gamma$ specifies the lengths of the preimages. So by the way we have chosen our Fenchel--Nielsen coordinates this fixes the lengths of the curves $\wt{\gamma}_j^{\ast}$ and $\wt{\gamma}_j^{\dagger}$ to be $k^jl(\gamma)$ for each $j=1,\ldots,l$. This means that in order for a surface $D$ to be close to $C$ in Teichmüller space and admit a covering over $Y$ we can only vary the twist parameters as the lengths of the geodesics are fixed by $Y$.

As the surface $Y$ is fixed we see that the point $a^{\dagger}$, which is one of the reference points used for measuring the twist coordinates, is identified with a unique point $b^{\ast}$ on $\gamma^{\ast}$. This means that in order for the glued map on $D$ to be well defined we need that the lifts $\wt{a}_j^{\dagger}$ are identified with lifts of $b^{\ast}$. But for each $j=1,\ldots,l$ there are only $k^j$ possibilities. Thus in total there are at most $k^1\cdots k^l$ possibilities for combinations of twist parameters at the opened preimages of the node $p\in X$. Since we can choose these independently for all nodes we obtain the result in the lemma.

Note that all these glued surfaces do indeed give smooth coverings as by \cref{lem:local-gluing} isometric coverings of surfaces with boundary geodesics of the same length can always be uniquely glued together if the maps coincide on one boundary point.
\end{proof}

\begin{rmk}
  We will need a few conclusions from this lemma.
  \begin{enumerate}[label=(\roman*), ref=(\roman*)]
    \item First notice that for every node with nodal degrees $k_1,\ldots,k_l$ we have constructed $\frac{k_1\cdots k_l}{K}$ families $\Psi_u:\mcU^X\lra\Ob\mcR_{g,k,h,n}(T)$ in \cref{eq:def-fn-map} that are each a $K$-fold cover of the constructed target family by \cref{cor:param-B-hol-discs}. Thus we actually have constructed all the possible complex structures $D$ having a Hurwitz cover $D\lra Y$ close to $D$ in $B^{C}$.
    \item Note that a hyperbolic deformation preparation for the central fibre $u^{\lambda}:C^{\lambda}\lra X^{\lambda}$ induces a hyperbolic deformation preparation for all the deformed Hurwitz covers $u^{\lambda}_b:C^{\lambda}_b\lra X^{\lambda}_b$ in a straight forward way because we only need the data at the not-yet-opened nodes. This way we can define $\Psi^{u^{\lambda}_b}$ using this data and all the corresponding objects are naturally elements of the family $\Psi^{\lambda}$.
    \item Also notice that the set $\mcA(u:C\lra X)$ does not depend on the choices for the hyperbolic deformation preparation in a small enough neighborhood of $C\in B^C$ as this set contains all complex structures close to $C$ admitting a Hurwitz cover to a surface close to $X$ and any other choice of data parametrizes locally the same complex structures close to $X$.
  \end{enumerate}
  \label{rmk:number-constructed-deformations}
\end{rmk}

Next we want to further describe the relation between the complex structures on the source surface constructed by two deformation families $\Psi^{\lambda}$ and $\Psi^{\lambda'}$ for which there exists a fibre isomorphism $(b,(\Phi,\phi),c)$ with $b\in O^{\lambda}$ and $c\in O^{\lambda'}$. Recall from \cref{rmk:unfoldings-source-surface} that both Hurwitz deformation families give rise to corresponding unfoldings parametrized by $O^{\lambda}$ and $O^{\lambda'}$ of the source surfaces $C^{\lambda}_b,C^{\lambda'}_c$ and $C^{\lambda'}$, where $C^{\lambda'}$ is the central source surface of the Hurwitz deformation family $\Psi^{\lambda'}$ and the indices $b$ and $c$ refer to the constructed surfaces at points $b\in O^{\lambda}$ and $c\in O^{\lambda'}$, respectively. Notice that all three surfaces have inclusions into the universal unfolding $Q^{C^{\lambda'}}$, $C^{\lambda'}$ as the central fibre, $C^{\lambda'}_c$ as an inclusion because of the way we constructed the deformation family and $C^{\lambda}_b$ via the morphism $\Phi$ and the inclusion into the universal unfolding. Now use the universal property of a universal unfolding from \cref{prop:summary-properties-universal-unfoldings} to extend these inclusions to holomorphic maps from small neighborhoods of $b\in O^{\lambda}$ and $c\in O^{\lambda'}$ to $B^{C^{\lambda'}}$, i.e.\

\begin{align*}
  \rho_b^{\lambda}:O^{\lambda}\supset\mcU_b & \xhookrightarrow{} B^{C^{\lambda'}} \\
  \rho_c^{\lambda'}:O^{\lambda'}\supset\mcU_c & \xhookrightarrow{} B^{C^{\lambda'}} \\
  \rho^{\lambda'}:O^{\lambda'} & \xhookrightarrow{} B^{C^{\lambda'}}.
\end{align*}

\begin{figure}[!h]
  \centering
  \def\svgwidth{0.9\textwidth}
  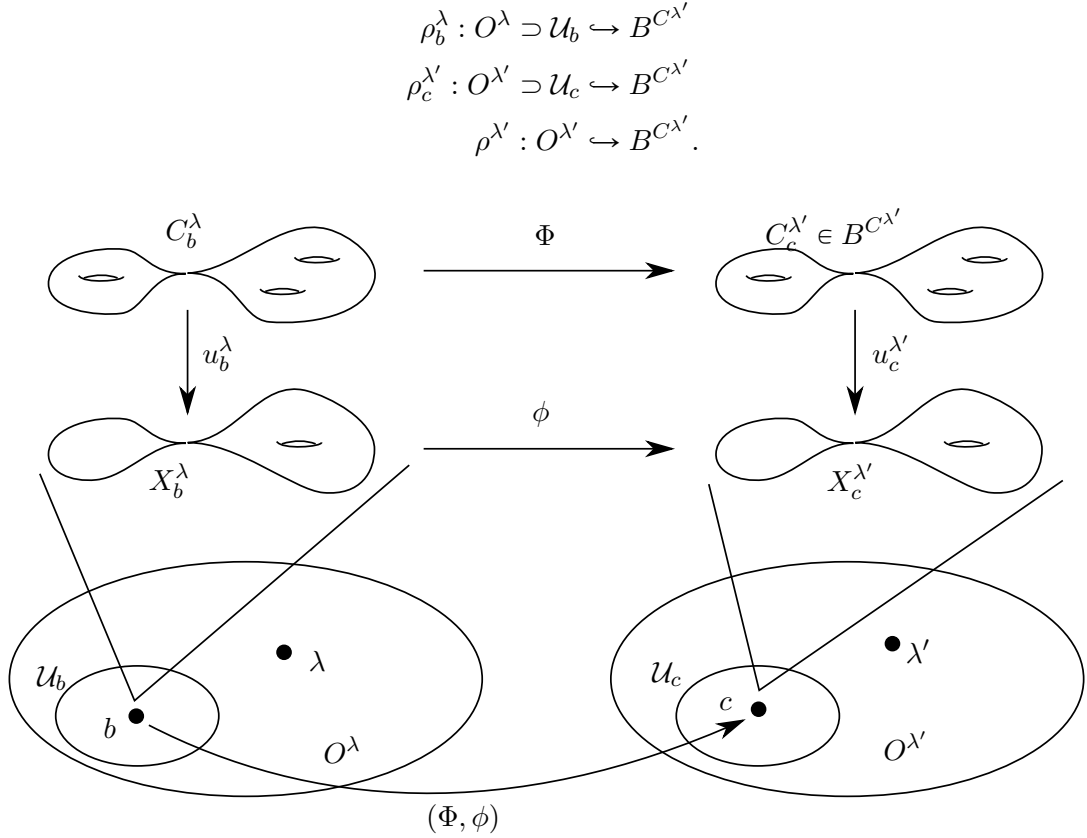
  \caption{These are the objects that we start with: A fibre isomorphism $(b,(\Phi,\phi),c)\in M(\lambda,\lambda')$ whose neighborhood we want to construct.}
  \label{fig:situation-mfd-structure-morphism-set}
\end{figure}

We will denote the corresponding maps on the families themselves by $\wt{\rho}$. Furthermore notice that the morphism $\phi:X^{\lambda}_b\lra X^{\lambda'}_c$ extends to a locally unique morphism between the unfoldings of $X^{\lambda}_b$ and $X^{\lambda'}_c$ as was shown in \cite{robbin_construction_2006}. Denote the extension between the total spaces of the universal unfoldings by $\wt{\eta}_{b,c}$. The situation is summarized in \cref{fig:situation-mfd-structure-morphism-set} and \cref{fig:universal-unfoldings-morphisms}.

\begin{figure}[!h]
  \centering
  \def\svgwidth{0.9\textwidth}
  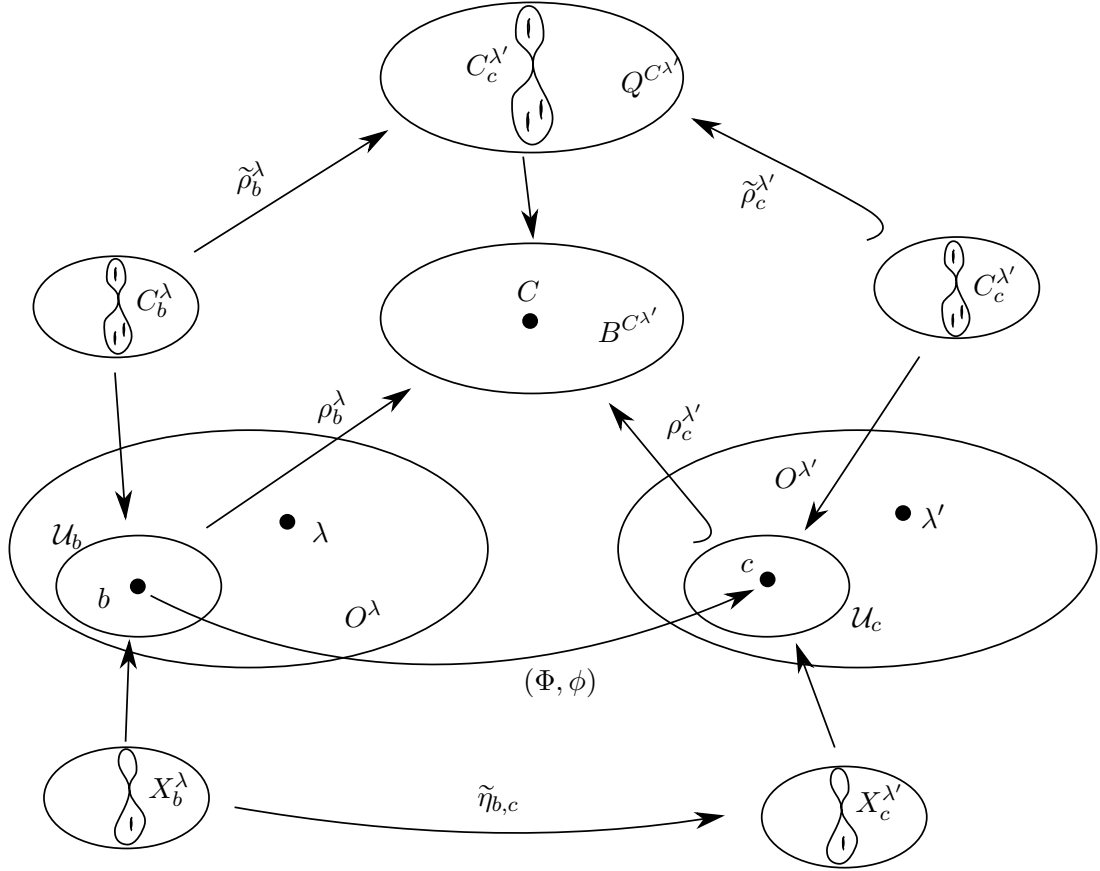
  \caption{The two central sets are those from \cref{fig:situation-mfd-structure-morphism-set}. Note that both smaller neighborhoods $\mcU_b$ and $\mcU_b$ define nodal unfoldings of the source surface (upper half) as well as of the target surface (lower half). These families then induce locally unique morphisms into the corresponding universal unfoldings.}
  \label{fig:universal-unfoldings-morphisms}
\end{figure}

\begin{lem}
  Consider the set $\mcA(u^{\lambda'}_c)\subset B^{C^{\lambda'}_c}$, defined in \cref{rmk:unfoldings-surfaces} as well as $B^{X^{\lambda'}_c}$.\footnote{The set $\mcA(u^{\lambda'}_c)$ describes the set of complex structures on the source surface close to $C^{\lambda'}_c$ admitting Hurwitz covers close to $u_c^{\lambda'}$ as constructed in the earlier sections.} Then there exist open neighborhoods $\mcV_c$ of $c\in\mcA(u^{\lambda'}_c)$ and $\mcW_c$ of $c\in B^{X^{\lambda'}_c}$ with a homeomorphism
  \begin{equation*}
    \mcV_c\lra V\times B_1\times\cdots\times B_p\subset V\times\DD^{m_1}\times\cdots\times\DD^{m_p}
  \end{equation*}
  and biholomorphism
  \begin{equation*}
    \mcW_c\lra V' \times \overbrace{\DD\times\cdots\times \DD}^{p}
  \end{equation*}
  such that in these coordinates the map $C_x^{\lambda'}\mapsto X_x^{\lambda'}$ is given by
  \begin{align*}
    V\times B_1\times\cdots\times B_p & \lra V' \times \DD\times\cdots\times \DD \\
    (t,(b_1^1,\ldots,b^1_{m_1}),\ldots,(b^p_1,\ldots,b^p_{m_p}) & \longmapsto (t',(b_1^1)^{l^1_1},\ldots,(b_1^p)^{l^p_1}).
  \end{align*}
  Here we use the convention that the target surface $X^{\lambda'}_c$ has $p$ nodes and the $i$-th node has $m_i$ nodal preimages in $C_c^{\lambda'}$ with degree $l^i_j$ with $j=1,\ldots,m_i$. Also the sets $V$ and $V'$ are open sets in Teichmüller space around the point corresponding to the smooth components of the source an target surface, respectively. Furthermore the set $B_i$ is the set of admissible gluing parameters defined in \cref{lem:param-set-b} for the nodes in $C_c^{\lambda'}$ over the $i$-th node in $X_c^{\lambda'}$. 
  \label{lem:coord-source-hurwitz-unfolding}
\end{lem}

\begin{proof}
  This follows immediately from \cref{def:deformation-of-hurwitz-covers} and the construction in \cref{sec:local-parametrizations}.
\end{proof}

\begin{lem}
  Both images of $\rho_b^{\lambda}$ and $\rho_c^{\lambda'}$ are contained in $\mcA(u_c^{\lambda'})\subset\mcA(u^{\lambda'})$, see \cref{rmk:definition-all-deformations-of-source-surface} for a definition of these sets. The intersection of these two sets is a complex submanifold of $B^{C^{\lambda'}}$ of real dimension $6h-6+2n$ or at least two dimensions less. The germ of this intersection and in particular its dimension is independent of all the choices except the maps $b_{\xi}$ for every node of the target surface.
  \label{lem:dimension-calculation-intersection-morphisms}
\end{lem}

\begin{proof}
  Notice that for $x\in \mcU_b\subset O^{\lambda}$ the curve $Q_{\rho^{\lambda}_b(x)}^{C^{\lambda'}}$ is biholomorphic to $C^{\lambda}_x$ via $\wt{\rho}^{\lambda}_b$. This curve does indeed possess a Hurwitz cover map of type $T$, namely
  \begin{equation*}
    \wt{\eta}_{b,c}|_{X^{\lambda}_x}\circ u^{\lambda}_x\circ\left(\wt{\rho}^{\lambda}_b|_{C_x^{\lambda}}\right)^{-1}:Q_{\rho^{\lambda}_b(x)}^{C^{\lambda'}}\lra X^{\lambda'}_{\eta_{b,c}(x)}.
  \end{equation*}
  If we choose $x$ sufficiently close to $b$ the target and source surface of this Hurwitz cover are contained in small neighborhoods of the fibre over $\rho^{\lambda}_b(b)=\rho^{\lambda'}_c(c)$ such that by \cref{lem:nod-hur-cov-unique} and \cref{rmk:number-constructed-deformations} the source surface must be contained in $\mcA(u^{\lambda}_c)$. Since this argument works for all points $x\in \mcU_b$ we have that $\Im\rho^{\lambda}_b\subset\mcA(u^{\lambda'}_c)$.

  The fact that $\Im\rho_c^{\lambda'}\subset\mcA(u_c^{\lambda'})$ follows in a similar but easier way directly from \cref{lem:nod-hur-cov-unique} and \cref{rmk:number-constructed-deformations}.

  We have that $\mcA(u_c^{\lambda'})\subset\mcA(u^{\lambda'})$ if we choose the induced hyperbolic deformation retractions as was explained in \cref{rmk:number-constructed-deformations}. As the families $x\mapsto C^{\lambda}_x$ and $y\mapsto C^{\lambda'}_y$ for $x\in O^{\lambda}$ and $y\in O^{\lambda'}$ factor through a parametrization via an open set in Teichmüller space and glued-in discs in the same way as the complex structure on universal unfoldings is defined it is easy to see that their images in $Q^{C^{\lambda'}}$ are indeed complex submanifolds. In order to compute their dimensions choose coordinates as in \cref{lem:coord-source-hurwitz-unfolding}, i.e. we have the diagram
  \begin{equation*}
    \xymatrix{
      V\times \DD\times\cdots\times\DD \ar[rrr]^{(t,z)\mapsto C^{\lambda}_{(u^{\lambda}_b)^*t,b_{\xi}(z)}} \ar[rrrd]_{(t,z)\mapsto X^{\lambda}_{t,a(z)}} & & & V''\times B_1\times\cdots\times B_p \ar[d]^{(t,y)\mapsto(t',y^l)}\\
      & & & V' \times \DD\times\cdots\times \DD
      }.
  \end{equation*}
  Here, the horizontal map is the local representative of the unfolding of the source surface constructed from $C^{\lambda}_b$, i.e.\ $\rho_b^{\lambda}$, the vertical map is the local representative of the map which maps the source surface to the corresponding target surface and the diagonal map is the local representative of the map $\eta_{bc}$. Since $\eta_{bc}$ is a submersion because both source and target are coordinate neighborhoods in universal unfoldings we see that the horizontal map needs to be a submersion for dimensional reasons. Thus we have $\dim_{\RR}\Im\rho_b^{\lambda}=6h-6+2n=\dim_{\RR}\Im\rho_c^{\lambda'}$.

  It remains to analyze the intersection of $\Im\rho_b^{\lambda}$ and $\Im\rho_c^{\lambda'}$ in $\mcA(u_c^{\lambda'})\subset\mcA(u^{\lambda'})\subset Q^{C^{\lambda'}}$. By construction the intersection is not empty as the central fibre corresponding to $\rho_b^{\lambda}(b)=\rho_c^{\lambda'}(c)$ is contained in both images. Also we know that it maps holomorphically into $\mcA(u^{\lambda'}_c)$ and is a submersion. Thus it needs to map holomorphically into every set $B_i$ and $V$. But since every $B_i$ is a one-point union of complex discs we can determine the intersection of $\Im\rho_b^{\lambda}$ and $\Im\rho_c^{\lambda'}$ quite easily. In every factor $B_i$ the two sets either agree or they consist of the central point only. From this the dimension count
  \begin{align*}
    \dim_{\RR}\Im\rho_b^{\lambda}\cap\Im\rho_c^{\lambda'} = 6h-6+2n- & 2\#\{i\text{ s.t. the intersection of } \\
    & \Im\rho_b^{\lambda}\text{ and } \Im\rho_c^{\lambda'}\text{ in }B_i\text{ is just a point}\}
  \end{align*}
  immediately follows. This proves the dimensional part of \cref{prop:manif-struct-set}.

  It is clear that the dimension of $\Im\rho_b^{\lambda}\cap\Im\rho_c^{\lambda'}$ does not depend on the choices for the coordinates except for the discs $b_{\xi_i}$ for the $p$ nodes. Also we get that the dimension is locally constant in $b$ and as $O^{\lambda}$ is connected we see that this dimension really only depends on the choice of the discs.
\end{proof}

The last lemma allows us to define the morphism set of our category $\mcM_{g,k,h,n}(T)$. Note that we can talk of the \emph{order} of a fibre isomorphism $(b,(\Phi,\phi),c)$ meaning the dimension of this intersection.

\begin{definition}
  We define
  \begin{equation*}
    \Mor\mcM_{g,k,h,n}(T)\coloneqq\bigsqcup_{\lambda,\lambda'\in\Lambda} M(\lambda,\lambda')
  \end{equation*}
  with
  \begin{align*}
    M(\lambda,\lambda')\coloneqq \{(b,(\Phi,\phi),c) \mid & \text{ fibre  isomorphisms with }b\in O^{\lambda}, c\in O^{\lambda'} \\
    & \quad \text{s.t.\ their order is } 6h-6+2n\}.
  \end{align*}
\end{definition}

It remains to define manifold charts on this set.
\begin{prop}
  The manifold structure on $\Mor\mcM_{g,k,h,n}(T)$ can be defined by the following charts on a neighborhood of $(b,(\Phi,\varphi),c)\in M(\lambda,\lambda')$:
  \begin{align*}
    \mcU_{b} & \lra M(\lambda,\lambda')\\
    x & \longmapsto \left(x,\left(\left(\wt{\rho}^{\lambda'}_{c}|_{C^{\lambda'}_x}\right)^{-1}\circ\wt{\rho}^{\lambda}_b|_{C^{\lambda}_x},\wt{\eta}_{b,c}|_{X^{\lambda}_x} \right),\left(\rho^{\lambda'}_c\right)^{-1}(\rho^{\lambda}_b(x))\right), \\
    & \phantom{\longmapsto} \eqqcolon (x,(\Phi_x,\varphi_x),y)
  \end{align*}
  where we vary the points $(b,(\Phi,\phi),c)$ and the neighborhoods $\mcU_b$ as long as they are small enough such that the unique extensions of the corresponding morphisms into the universal unfoldings of all involved curves exist. Also $y$ denotes the base point of $\varphi_x(x)\in B^{C^{\lambda'}}$ which is given by $y=\left(\rho^{\lambda'}_c\right)^{-1}(\rho^{\lambda}_b(x))$.
\end{prop}

\begin{proof}
  It remains to check that
  \begin{enumerate}[label=(\roman*), ref=(\roman*)]
    \item the image consists indeed of fibre isomorphisms, \label{enum:fibre-isom}
    \item these maps are injective and \label{enum:maps-injective}
    \item transition functions are smooth. \label{enum:transition-functions-smooth}
  \end{enumerate}
  For this purpose we use again the conventions from the last lemmas. Notice first that the universal property of a universal unfolding gives uniqueness of the extended morphisms and thus for any $x\in \mcU_b$ we have $(\eta^{\lambda'}_c)^{-1}(\eta_b^{\lambda}(x))=(\rho_c^{\lambda'})^{-1}(\rho_b^{\lambda}(x))$ and denote this point by $y\in O^{\lambda'}$. Then we have the diagram
  \begin{equation*}
    \xymatrix{
      C_x^{\lambda} \ar[rr]^{\Phi_x} \ar[d]^{u^{\lambda}_x} &  &C^{\lambda'}_y \ar[d]^{u^{\lambda'}_y} \\
      X_x^{\lambda} \ar[rr]^{\varphi_x} & & X^{\lambda'}_y
      }
  \end{equation*}
  where the horizontal maps $\Phi_x=\left(\wt{\rho}^{\lambda'}_{c}|_{C^{\lambda'}_x}\right)^{-1}\circ\wt{\rho}^{\lambda}_b|_{C^{\lambda}_x}$ and $\varphi_x=\wt{\eta}_{b,c}|_{X_x^{\lambda}}$ are biholomorphisms by construction. The diagram commutes because by \cref{rmk:number-constructed-deformations} and \cref{cor:hurwitz-deformations-are-not-equivalent} we have at most one Hurwitz cover between $C^{\lambda'}_y$ and $X_y^{\lambda'}$ if we fix the cover on the smooth part. Thus the image is indeed a fibre isomorphism.

  The maps are injective as they are injective on the first component. It remains to show that the transition functions are smooth. So let
  \begin{equation*}
    \mcU_b\lra M(\lambda,\lambda') \longleftarrow \mcU_c
  \end{equation*}
  be two such coordinate charts with a common intersection. Note that now $b$ and $c$ are points in $O^{\lambda}$. Then by definition of the charts we have that $\Im\rho_b^{\lambda}$ and $\Im\rho_c^{\lambda}$ are submanifolds of $B^{C^{\lambda'}}$. Furthermore the maps $\rho_b^{\lambda}$ and $\rho_c^{\lambda}$ coincide wherever both are defined by the uniqueness property of universal unfoldings. Thus the transition function for our charts are the restrictions of the transition functions of $\Mor\mcM_{g,k}$ to smooth submanifolds which are again diffeomorphisms.
\end{proof}

Note that this defines a manifold structure on the set $\Mor\mcM_{g,k,h,n}(T)$ because by \cref{enum:fibre-isom} and \cref{enum:maps-injective} we obtain subsets bijective to open subsets of $O^{\lambda}$, so we can define a topology on $M(\lambda,\lambda')$ by taking the coarsest topology such that all these sets are open. This way by \cref{enum:transition-functions-smooth} we obtain a smooth atlas on $\Mor\mcM_{g,k,h,n}(T)$.

\begin{thm}
  With the above definitions and the obvious structure maps the groupoid $\mcM_{g,k,h,n}(T)$ is an orbifold groupoid.
\end{thm}

\begin{proof}
  We need to check the following properties:
  \begin{enumerate}[label=(\roman*), ref=(\roman*)]
    \item $\mcM_{g,k,h,n}(T)$ is a groupoid.
    \item All structure maps are smooth submersions.
    \item The map $s\times t:\Mor\mcM_{g,k,h,n}(T)\lra \Ob\mcM_{g,k,h,n}(T)\times\Ob\mcM_{g,k,h,n}(T)$ is proper.
  \end{enumerate}
  For the first statement we need to check that if $(b,(\Phi,\phi),c)\in \Mor\mcM_{g,k,h,n}(T)$ then $(c,(\Phi^{-1},\phi^{-1}),b)\in\Mor\mcM_{g,k,h,n}(T)$. But for this we only need to check the dimension of the intersection of the images of the extended morphisms into the universal unfolding of the source surface of the target Hurwitz cover. By the uniqueness property of the extended morphism into the universal unfolding  we have that $\Im\rho_c^{\lambda'}=\Im\rho_v^{\lambda}$ implies the same for the image of this set in $B^{C^{\lambda}_b}$ which is the dimension in question for the inverse morphism. Thus $\mcM_{g,k,h,n}(T)$ is indeed a groupoid.

  Next we verify that the structure maps are smooth submersions. Let us spell out how the source maps looks in coordinates close to a point $(b,(\Phi,\phi),c)\in\Mor\mcM_{g,k,h,n}(T)$
  \begin{align*}
    \mcU_b \lra & M(\lambda,\lambda') \lra O^{\lambda} \\
    x \longmapsto & \left(x,\left(\Phi_x,\varphi_x\right),\left(\rho^{\lambda'}_c\right)^{-1}(\rho^{\lambda}_b(x))\right) \longmapsto x
  \end{align*}
  which is obviously a smooth submersion. Similarly the target map is given by $b\longmapsto \left(\rho^{\lambda'}_c\right)^{-1}(\rho^{\lambda}_b(x))$ which is a smooth submersion as this is precisely the transition function from the Deligne--Mumford orbifold groupoid $\mcM_{h,n}$ as was defined in \cite{robbin_construction_2006}. The inverse and identity maps are smooth submersions by the same argument as the target and source map, respectively. For the multiplication map we can choose coordinates close to $(b,(\Phi,\phi),c), (c,(\Phi',\phi'),d)$ and $(b,(\Phi'\circ \Phi,\phi'\circ\phi),d)$ to obtain charts in a neighborhood in $\Mor\mcM_{g,k,h,n}(T){_s\times_t}\Mor\mcM_{g,k,h,n}(T)$ such that
  \begin{align*}
    \mcU_b\lra & \Mor\mcM_{g,k,h,n}(T){_s\times_t}\Mor\mcM_{g,k,h,n}(T) \lra \Mor\mcM_{g,k,h,n}(T) \lra \mcU_b \\
    x \longmapsto & \left((x,(\Phi_x,\phi_x),y),(y,(\Phi'_y,\phi'_y),z)\right) \longmapsto (x,(\Phi'_y\circ\Phi_x,\phi'_y\circ\phi_x),z) \longmapsto x
  \end{align*}
  which is of course a smooth submersion on the manifold $O^{\lambda}$.

  It remains to prove the properness of the map $s\times t$. Notice that the map is given by
  \begin{equation*}
    s\times t(b,(\Phi,\phi),c)=(b,c)
  \end{equation*}
  and recall that the corresponding maps on the Deligne--Mumford orbifold groupoids $\mcM_{g,k}$ and $\mcM_{h,n}$ are proper. We will prove sequential compactness of the preimage of a compact set as $\Mor\mcM_{g,k,h,n}(T)$ is clearly second-countable. So consider a sequence $\{(b_k,(\Phi_k,\phi_k),c_k)\}_{k\in\NN}\subset\Mor\mcM_{g,k,h,n}$. Then we obtain a subsequence
  \begin{equation*}
    \{(b_k,(\Phi_k,\phi_k),c_k)\}_{k\in\NN}\subset\Mor\mcM_{g,k,h,n}
  \end{equation*}
  such that all four sequences converge individually because the $b_k$ and $c_k$ are contained in a compact subset of $\Ob\mcM_{g,k,h,n}(T)$ and $\Phi_k,\phi_k$ are contained in preimages of $s\times t$ on $\mcM_{g,k}$ and $\mcM_{h,n}$ of a convergent subsequence. Thus all the surfaces and maps do indeed converge in the appropriate spaces. It is clear that the property of being a morphism of Hurwitz covers is closed in our topology as all the surfaces and maps converge in $C^{\infty}_{\text{loc}}$ away from the nodes and thus preserve the diagram in the limit. Therefore this subsequence converges to a fibre isomorphism. The order of the limit fibre isomorphism is still the maximal one as the sets $M(\lambda,\lambda')$ are given by points where two manifolds agree and thus limit points of sequences contained in both manifolds are still contained in the same manifolds implying that the limit has the same order.
\end{proof}

\section{Properties of Maps between Moduli Spaces of Hurwitz Covers}

\label{sec:main-results}

Recall that until now we have constructed the following groupoids as well as defined a few more:
\begin{enumerate}[label=(\roman*), ref=(\roman*)]
  \item the target and source Deligne--Mumford groupoids $\mcR_{g,k}$ and $\mcR_{h,n}$ and their orbifold groupoids $\mcM_{g,k}$ and $\mcM_{h,n}$,
  \item the orbifold groupoid of nodal Hurwitz covers $\mcM_{g,k,h,n}(T)$ and
  \item the groupoid of all nodal Hurwitz covers $\mcR_{g,k,h,n}(T)$.
\end{enumerate}

In this section we will introduce various obvious maps between these spaces and prove some of their properties. In particular we need to talk about the difference between the orbifold groupoid $\mcM_{g,k,h,n}(T)$ and the actual moduli space of nodal Hurwitz covers $\mcR_{g,k,h,n}(T)$. We will assume that the same families of source and target surfaces that we used for constructing the orbifold structure on $\mcM_{g,k,h,n}(T)$ are also used for the orbifold structure of $\mcM_{g,k}$ and $\mcM_{h,n}$. We can assume this because it is known from \cite{robbin_construction_2006} that different choices of universal unfoldings give rise to Morita equivalent ep-Lie groupoids for the Deligne--Mumford moduli spaces.

\index{Forgetful Functor}

\begin{definition}
  We define the \emph{forgetful functors} $\fgt:\mcR_{g,k,h,n}(T)\lra\mcR_{g,k}$ by
  \begin{align*}
    \fgt_{\Ob}(C,u,X,\bq,\bp) & \coloneqq (C,\bq) \\
    \fgt_{\Mor}(\Phi,\phi) & \coloneqq \Phi
  \end{align*}
  and its ``restriction'' $\fgt:\mcM_{g,k,h,n}(T)\lra\mcM_{g,k}$ by
  \begin{align*}
    \fgt_{\Ob}(b) & \coloneqq (C^{\lambda}_b,\bq) \qquad\text{for }b\in O^{\lambda}\\
    \fgt_{\Mor}(b,(\Phi,\phi),c) & \coloneqq \Phi.
  \end{align*}
  These functors descend to maps on the corresponding orbit spaces.
\end{definition}

Additionally we have the following evaluation functors.

\index{Evaluation Functor}

\begin{definition}
  We define the \emph{evaluation functors} $\ev:\mcR_{g,k,h,n}(T)\lra\mcR_{h,n}$ by
  \begin{align*}
    \ev_{\Ob}(C,u,X,\bq,\bp) & \coloneqq (X,\bp) \\
    \ev_{\Mor}(\Phi,\phi) & \coloneqq \phi
  \end{align*}
  and $\ev:\mcM_{g,k,h,n}(T)\lra\mcM_{h,n}$ by
  \begin{align*}
    \ev_{\Ob}(b) & \coloneqq (X^{\lambda}_b,\bp) \qquad\text{for }b\in O^{\lambda}\\
    \ev_{\Mor}(b,(\Phi,\phi),c) & \coloneqq \phi.
  \end{align*}
  These functors descend to maps on the corresponding orbit spaces.
\end{definition}

Furthermore we have the inclusion from our moduli space into the actual moduli space of Hurwitz covers.

\index{Inclusion Map}

\begin{definition}
  We define the \emph{inclusion functor} $\iota:\mcM_{g,k,h,n}(T)\lra\mcR_{g,k,h,n}(T)$ as the obvious inclusion of categories
  \begin{align*}
    \iota_{\Ob}(b) & \coloneqq (C^{\lambda}_b,u^{\lambda}_b,X^{\lambda}_b,\bq,\bp) \qquad\text{for }b\in O^{\lambda}\\
    \iota_{\Mor}(b,(\Phi,\phi),c) & \coloneqq (\Phi,\phi).
  \end{align*}
\end{definition}

\begin{rmk}
  This functor descends to a not necessarily injective map
  \begin{equation*}
    \iota:|\mcM_{g,k,h,n}(T)|\lra|\mcR_{g,k,h,n}(T)|.
  \end{equation*}
  Notice that there are also inclusion functors for the Deligne--Mumford orbifolds which descend to actual homeomorphisms on their orbit spaces.
\end{rmk}

These functors obviously satisfy the following commuting diagram.

\begin{equation}
  \xymatrix{
    \mcM_{g,k,h,n}(T) \ar[r]^{\iota} \ar[d]^{\ev} \ar@/_2pc/[dd]_{\fgt} &  \mcR_{g,k,h,n}(T) \ar[d]^{\ev} \ar@/^2pc/[dd]^{\fgt} \\
    \mcM_{h,n} \ar[r]^{\iota} &  \mcR_{h,n} \\
    \mcM_{g,k} \ar[r]^{\iota} &  \mcR_{g,k}
}
\label{eq:diagram-moduli-spaces}
\end{equation}

The following proposition states the main properties of these functors.

\begin{prop}
  In the situation as above we have that
  \begin{enumerate}[label=(\roman*), ref=(\roman*)]
    \item all functors in \cref{eq:diagram-moduli-spaces} induce continuous maps on orbit spaces, \label{prop-functors-case-1}
    \item the functor $\fgt:\mcM_{g,k,h,n}(T)\lra\mcM_{g,k}$ is a homomorphism and \label{prop-functors-case-2}
    \item the functor $\iota:\mcM_{g,k,h,n}(T)\lra\mcR_{g,k,h,n}(T)$ is essentially surjective\footnote{Recall that a functor $f:\mcC\lra\mcD$ is called \emph{essentially surjective} if every object in $\mcD$ is isomorphic to an object in the image of $\Ob\mcC$ under $f$, it is called \emph{faithful}, if for every $x,y\in\Ob\mcC$ the map $f:\Hom_{\mcC}(x,y)\lra\Hom_{\mcD}(f(x),f(y))$ is injective and it is called \emph{full}, if this map is surjective for every $x,y\in\Ob\mcC$.} and faithful everywhere and its restriction to the full subcategory of smooth Hurwitz covers is additionally full. \label{prop-functors-case-3}
  \end{enumerate}
  \label{prop:properties-functors-between-groupoids}
\end{prop}

\begin{proof}
  Note that \cref{prop-functors-case-1} does not make sense so far as we have not defined any topology on $\mcR_{g,k,h,n}(T)$ yet. We will define a topology on its orbit space in \autoref{sec:topology-moduli-spaces}. It will be such that the induced map $\iota$ on orbit spaces is a quotient map and thus continuous by definition. Furthermore the functors $\iota$ induce homeomorphisms on Deligne-Mumford spaces which is the content of Theorem~13.6 in \cite{robbin_construction_2006}.

  The functor $\ev$ is continuous on objects and morphisms because in local coordinates it is given as product of a homeomorphism on an appropriate Teichmüller space times maps $a:z\mapsto z^l$ for certain powers $l\in\NN$ for every node in the target surface. In the next paragraph we will show that $\fgt$ is smooth on objects and morphisms. Using the universal property of quotient topologies we thus see that the corresponding maps on $|\mcR_{g,k,h,n}(T)|$ are continuous as well.
  
  Recall that the functor $\fgt$ on objects can be described as an inclusion of $O^{\lambda}\in B^{C^{\lambda}}$ as $O^{\lambda}$ defined a nodal family of source surfaces. This inclusion comes from the universal property of $B^{C^{\lambda}}$ and is thus smooth and in particular continuous. The same argument works for the morphisms. Thus $\fgt$ is a homomorphism.

  It remains to prove \cref{prop-functors-case-3}. By choice of the index set $\Lambda$ the functor $\iota$ is essentially surjective. It is obviously injective on morphisms as it is given by
  \begin{equation*}
    \iota_{\Mor}(b,(\Phi,\phi),c)=(\Phi,\phi).
  \end{equation*}
  Now consider the restriction of the functor $\iota$ to the full subcategory generated by the open submanifold $\mathring{\mcM}_{g,k,h,n}(T)$ of smooth Hurwitz covers. Then $\iota$ obviously restricts to a functor to the full subcategory $\mathring{\mcR}_{g,k,h,n}(T)$ generated by smooth Hurwitz covers. Now recall from the proof of \cref{lem:dimension-calculation-intersection-morphisms} that the order of a morphism 
  \begin{equation*}
    (\Phi,\phi)\in\Hom_{\mcR_{g,k,h,n}(T)}\left((C_b^{\lambda},u_b^{\lambda},X_b^{\lambda}\bq_b,\bp_b),(C_c^{\lambda'},u_c^{\lambda'},X_c^{\lambda'}\bq_c,\bp_c)\right)
  \end{equation*}
  is determined by the number of nodes in the target surface such that the discs $b_{\xi}$ and $b_{\xi'}$ in $\mcA(u^{\lambda'}_c)$ intersect in single points. But since there are no nodes in the target surface every morphism has full order and is thus included in the morphisms of $\Hom_{\mathring{\mcM}_{g,k,h,n}(T)}(b,c)$ already. Thus $\iota$ is full on the full subcategory of smooth Hurwitz covers.
\end{proof}

\begin{rmk}
  Note that \cref{prop:properties-functors-between-groupoids} shows in particular that the restriction of $\iota$ to the full subcategory of smooth Hurwitz covers is a bijection on orbit spaces. Thus the moduli space $|\mcR^{\circ}_{g,k,h,n}(T)|$ of \emph{smooth} Hurwitz covers carries an actual orbifold structure.
  \label{rmk:moduli-space-smooth-hurwitz-covers-orbifold}
\end{rmk}

Before continuing with the main theorem on the evaluation functor let us modify the definition of $\mcM_{g,k,h,n}(T)$ a little bit by possibly adding more elements to our index set $\Lambda$. We will choose $\Lambda$ by first covering the target moduli space by small enough neighborhoods $\{\mcU_i\}_{i\in I}$ such that the universal unfoldings of their central surfaces $X_i\in\mcU_i$ cover the target moduli space. Then we add representatives for \emph{all} equivalence classes of Hurwitz covers with these $X_i$ as target. Furthermore if $X_i$ has nodes we add \emph{all} choices of discs $b_{\xi}$ for the resolutions of the source surface. This gives a finite set $\lambda$ together with deformation spaces $O^{\lambda}$ for our Hurwitz deformations.

\begin{thm}
  The functor $\ev:\mcM_{g,k,h,n}(T)\lra\mcM_{h,n}$ is a morphism covering of degree $H_{g,k,h,n}(T)$ on the full subcategory of smooth Hurwitz covers. Additionally at a nodal point $(C,u,X,\bq,\bp)\in\Ob\mcM_{g,k,h,n}(T)$ there exist smooth coordinates such that the map $\ev_{\Ob}$ looks like
  \begin{align*}
    \DD^{3k-3+n-N}\times\DD^N & \lra\DD^{3k-3+n-N}\times\DD^N \\
    (x,z_1,\ldots,z_N) & \longmapsto \left(x,z_1^{K_1},\ldots,z_N^{K_N}\right),
  \end{align*}
  where $\DD\subset\CC$ is the unit disc, $N\geq 0$ corresponds to the number of nodes of $X$ and $K_i$ is the least common multiple of the degrees of $u$ at the nodes above the $i$-th node of $X$. Furthermore the set where $z_i\neq 0$ for all $i=1,\ldots,N$ corresponds exactly to the subset of smooth Hurwitz covers. We will refer to such a map as a \emph{branched morphism covering} between the orbifolds.
  \label{thm:ev-local-structure}
\end{thm}

\index{Morphism Covering!Branched}
\index{Morphism Covering}

\begin{proof}
  Before proving the morphism covering property let us first show the local statement on object spaces. Choose coordinates around a point 
  \begin{equation*}
    (C,u,X,\bq,\bp)\in\Ob\mcM_{g,k,h,n}(T)
  \end{equation*}
  as in \cref{eq:def-fn-map}, i.e.\ a disc in Teichmüller space of the normalization of the target surface and discs $b_{\xi_i}$ for smoothing all nodes with indices $i=1,\ldots,N$ of the target. This means that $\xi_i$ is a vector with roots of unity as entries, one for each node over the $i$-th node of corresponding order and the gluing parameters for the target surface are described by maps $a_i$ for $i=1,\ldots,N$. Then we can use the same disc structure for describing a neighborhood of the target surface in its universal curve. For these choices the map $\ev$ is just given as
  \begin{align}
    \ev: \DD^{3k-3+n-N}\times\DD^N & \lra\DD^{3k-3+n-N}\times\DD^N \label{eq:ev-map-loc} \\
    (t,z_1,\ldots,z_N) & \longmapsto (t,a_1(z_1),\ldots,a_N(z_N)) \nonumber
  \end{align}
  where the maps $a_i:\DD\lra\DD$ are given by \cref{def:deformation-of-hurwitz-covers}, i.e.\ $z\mapsto z^{K_i}$ if we denote by $K_i$ the least common multiple of the degrees of $u$ at the nodes over the $i$-th node in $X^{\lambda}$. This shows the local statement.

  Now let us consider the functor $\ev:\mathring{\mcM}_{g,k,h,n}(T)\lra\mathring{\mcM}_{h,n}$ on the full subcategory of smooth Hurwitz covers denoted by $\mathring{\mcM}_{g,k,h,n}(T)\subset\mcM_{g,k,h,n}(T)$. Consider a point $x\in\Ob\mathring{\mcM}_{h,n}$. This point is contained in the deformation space of a universal unfolding of some $X^{\lambda}$, possibly with nodes. By construction all preimages of $x$ in $\Ob\mathring{\mcM}_{g,k,h,n}(T)$ are contained in $\bigsqcup_{\lambda'\in\Lambda(\lambda)}O^{\lambda'}$ where $\Lambda(\lambda)$ is the subset of $\lambda'\in\Lambda$ such that $X^{\lambda'}=X^{\lambda}$.\footnote{Recall that $\lambda\in\Lambda$ contains more information than just the target surface.} Since $x$ corresponds to a smooth target surface it has all gluing parameters unequal to zero meaning that it is contained in the regular part of the map
  \begin{equation*}
    (t,z_1,\ldots,z_n)\mapsto (t,z_1^{l_1},\ldots,z_n^{l_n})
  \end{equation*}
  where it is obviously a covering. Thus we can restrict to a small neighborhood around $x\in\Ob\mathring{\mcM}_{h,n}$ and see that its preimage is given by the disjoint union of various ``roots'' of this set over all $\lambda$ which contain $x$ in a neighborhood. Thus $\ev$ is a covering on objects. As the morphisms are also parametrized by the same coordinates we can easily see that $\ev$ is also a covering on morphisms.

  It remains to show the lifting property for a morphism covering. To see this let $b=(C,u,X,\bq,\bp)\in\Ob\mathring{\mcM}_{g,k,h,n}(T)$ and $\phi\in\Mor\mathring{\mcM}_{h,n}$ with $\phi(X,\bp)=(X',\bp')$ be given. Then there is a morphism $(\id,\phi)$ between $(C,u,X,\bq,\bp)$ and $(C,\phi\circ u,X',\bq,\bp')$ in the category $\mcR_{g,k,h,n}(T)$. By construction of the set $\lambda$ this Hurwitz cover has a an equivalent cover in $\Ob\mathring{\mcM}_{g,k,h,n}(T)$ and the corresponding morphism is included in $\Mor\mathring{\mcM}_{g,k,h,n}(T)$ because the target surface is smooth and thus all morphisms have maximal order. 

  The degree of the morphism covering is clearly given by $H_{g,k,h,n}(T)$ by comparing \cref{prop:morphism-coverings-deg} and \cref{eq:hurwitz-number-definition}.
\end{proof}

We will later show that the category $\mcM_{g,k,h,n}(T)$ is a compact orbifold category and thus carries a rational fundamental class as was mentioned in \cref{thm:alg-top-orbifolds}. We define

\index{Hurwitz Class}

\begin{definition}
  The rational singular cohomology class
  \begin{equation*}
    D_{g,k,h,n,}(T)\in H^{6h-6+2n}(|\mcM_{g,k}|,\QQ)
  \end{equation*}
  is called the \emph{Hurwitz class} and is defined as
  \begin{equation*}
    D_{g,k,h,n}(T)\coloneqq \fgt_*[\mcM_{g,k,h,n}(T)].
  \end{equation*}
\end{definition}

\chapter{Orbifold Structure on the Moduli Space of Bordered Hurwitz Covers}

\label{sec:orbifold-structure-moduli-space-borderd-huwritz-covers}

\section{Definitions}

\label{sec:orb-structure-mod-space-bordered-hurwitz-covers-definitions}

In this section $C$ and $X$ denote admissible Riemann surfaces, see \cref{def:admissible-riemann-surface}. This means that they can have actual boundary components and punctures. We will interpret the punctures as marked points. Both types of objects will be referred to as boundary components. If the component is a circle we denote it by $\partial_j C$ for $j\in\{1,\ldots,k\}$ and if it is a marked point we denote it by $q_j$ for $j\in\{1,\ldots,k\}$. So in particular we will fix the total number of boundary components.

The surface $X$ will as usual be of genus $h$ and have $n$ such boundary components which might be either actual circles or marked points. Again we denote actual boundary components by $\partial_i X$ for $i\in\{1,\ldots,n\}$ and marked points by $p_i$ for $i\in\{1,\ldots,n\}$. In both cases we require that all boundary components are enumerated with indices $j\in\{1,\ldots,k\}$ and $i\in\{1,\ldots,n\}$.

Again, we fix a surjective map $\nu:\{1,\ldots,k\}\lra\{1,\ldots,n\}$ and partitions $T_1,\ldots,T_n$ of a fixed natural number $d$, called the degree of the Hurwitz covering. We require that the length of the partition $T_i$ is equal to $|\nu^{-1}(i)|$ and that $T_i=\{l_j\}_{j\in\nu^{-1}(i)}$ satisfies
\begin{equation*}
T_i=\sum_{j\in\nu^{-1}(i)}l_j\qquad\forall i=1,\ldots,n.
\end{equation*}

Also we will now need reference curves close to the boundary of the target surface. So recall from \cref{sec:reference-curve} that a reference curve at a boundary $\del_iX$ or $\bp_i$ is a closed simple curve of constant curvature $1$ and a fixed length $F(L(\del_jC))$ or $F(0)$, respectively.\footnote{To avoid confusion with the branch degrees we will temporarily denote the length of a curve by $L$.} Here we fixed a suitable function $F:\RR_{\geq 0}\lra\RR_{\geq 0}$ beforehand for which we require a slightly stronger inequality than in \cref{sec:reference-curve}. This is because a $d$-fold cover of a reference curve of length $F(L)$ has length $dF(L)$ and we want this curve to be still contained in the collar neighborhood of the boundary geodesic of length $dL$. As $dL$ is larger than $L$ the corresponding collar neighborhood on $C$ is actually thinner than the one on $X$. So if we want to ensure that preimages of the reference curves on $X$ are also reference curves on $C$, we require
\begin{equation*}
  L< F(L)< \frac{L}{\tanh\left(\frac{dL}{2}\right)}
\end{equation*}
and $0<d\lim_{L\to 0}F(L)<2$. Notice that close to a boundary $\del_jC$ with degree $l_j$ we then have
\begin{equation*}
  l_jL< l_jF(L) <\frac{L}{\tanh\left(\frac{dL}{2}\right)} < \frac{L}{\tanh\left(\frac{l_jL}{2}\right)}
\end{equation*}
which are precisely the inequalities we need to ensure that a reference curve of length $l_jF(L)$ exists close to the boundary of length $l_jL$. In principle the definition of our spaces will then depend on this function $F$ but since it won't matter for our considerations we will drop it.

Given such a choice of a function $F$ we can define the reference curves $\Gamma_i(X)$ close to $\del_iX$ or $p_i\in X$. Since $u:C\lra X$ is a hyperbolic local isometry the preimages of $\Gamma_i(X)$ are contained in the collar or cusp neighborhoods of $C$ and have also constant geodesic curvature and length $l_jF(L(\del_iX))\leq dF(L(\del_iX))$ implying that this is also a reference curve on $C$ close to $\del_jC$ or $p_j$. Denote this reference curve by $\Gamma_j(C)$. Note that this follows from \cref{lem:collar-neighborhoods-boundary-hurwitz-cover} as well as the uniqueness and existence of reference curves in \cref{lem:existence-reference-curves}.

\index{Hurwitz Cover!Bordered}

\begin{definition}
  \label{def:whmcm}
  We define the category $\wh{\mcR}_{g,k,h,n}(T)$ of \emph{bordered Hurwitz covers} with $g,k,h,n,d\in\NN$ and $T=(T_1,\ldots,T_n,d,\nu,\{l_j\}_{1}^k)$ as above to consist of 
  \begin{equation*}
    \Ob\wh{\mcR}_{g,k,h,n}(T)\coloneqq\{(C,u,X,\bq,\bp,\bz)\}
  \end{equation*}
  where
  \begin{itemize}
    \item $C$ is an admissible Riemann surface of genus $g$ with $k$ boundary components,
    \item $X$ is an admissible Riemann surface of genus $h$ with $n$ boundary components,
    \item the tuples $\bq$ denote the set of boundary components, i.e.\ for any $j\in \{1,\ldots,k\}$ the element $\bq_j$ is either a marked point $q_j$ or a boundary component $\partial_j C$ and similarly for any $i\in\{1,\ldots,n\}$ the element $\bq_i$ is either a marked point $p_i$ or a boundary component $\partial_i X$,
    \item $u:C\lra X$ is holomorphic such that all critical points and branch points are special (i.e.\ marked or nodal, not boundary) points on $C$, respectively on $X$, and at every node $u$ satisfies all the conditions of a Hurwitz cover,
    \item for all $j=1,\ldots,k$ we have $u(\bq_j)=\bp_{\nu(j)}$,
    \item $\textbf{z}$ is a tuple of marked points $z_j\in\Gamma_j(C)$ on the reference curve $\Gamma_j(C)$ defined above satisfying the condition $u(z_j)=u(z_l)\;\forall j,l=1,\ldots,k$ s.t. $\nu(j)=\nu(l)$ and
\item the branching profile over $p_i$ is given by $T_i$, i.e.\ the degree of $u$ at $\bq_j$ (either the degree $z\mapsto z^{l_j}$ or the degree $u:\del_jC\lra\del_iX$) with $\nu(j)=i$ is given by $l_j$.
  \end{itemize}
  Its morphisms are defined as
  \begin{equation*}
    \Hom_{\wh{\mcR}_{g,k,h,n}(T)}((C_1,X_1,u_1,\bz_1),(C_2,X_2,u_2,\bz_2))\coloneqq\{(\Phi,\phi)\}
  \end{equation*}
  where
  \begin{itemize}
    \item $\Phi:C_1\lra C_2$ and $\phi:X_1\lra X_2$ are biholomorphisms,
    \item the diagram 
      \begin{equation}
        \label{eq:diag_equivalence_whmcm}
        \xymatrix{%
          C_1 \ar[r]^{\Phi} \ar[d]_{u_1} & C_2 \ar[d]^{u_2} \\
          X_1 \ar[r]^{\phi} & X_2
        }
      \end{equation}
      commutes and
    \item the chosen enumerations agree, i.e.\ $\Phi(\bq_j(C_1))=\bq_j(C_2)$ and $\phi(\bp_i(X_1))=\bp_i(X_2)$ as well as $\Phi(z_i(X_1))=z_i(X_2)\;\forall i=1,\ldots,n$.
  \end{itemize}
  Note that such a map necessarily preserves types of boundary conditions as well as their degrees.
\end{definition}

\begin{rmk}
  Note that the condition $u(z_j)=u(z_l)$ with $\nu(j)=\nu(l)$ ensures that we actually have well-defined marked points close to the boundary $\del_{\nu(j)}X$ as well. We will use these marked points later on, too. Also note that we require that all critical points are marked, i.e.\ contained in $\bq$, which in turn means that if all boundary components are circles then the map $u$ is an actual cover.
\end{rmk}

We also need a modified version of the orbifold category of Hurwitz covers. This category will be denoted by $\wt{\mcR}_{g,k,h,n}(T)$ where again $g,k,h,n,d\in\NN$, $\nu:\{1,\ldots,k\}\lra\{1,\ldots,n\}$ and $T=(T_1,\ldots,T_n,d,\nu,\{l_j\}_{1}^k)$ are the combinatorial data.

\index{Moduli Space!Of Closed Hurwitz Covers With Multicurve}

\begin{definition}
  Define the category $\wt{\mcR}_{g,k,h,n}(T)$ with even $n$ as follows. Its objects are tuples
  \begin{equation*}
    \Ob\wt{\mcR}_{g,k,h,n}(T)\coloneqq \{(C,u,X,\bq,\bp,\boldsymbol{\Gamma})\}
  \end{equation*}
  where
\begin{itemize}
  \item $C$ and $X$ are closed stable nodal Riemann surfaces of genus $g$ and $h$, respectively,
  \item $\bq$ is a $k$-tuple of marked points on $C$ and $\bp$ is a $n$-tuple of marked points on $X$,
  \item $u:C\lra X$ is a holomorphic map of degree $d$ which satisfies $u(\bq_j)=\bp_{\nu(j)}$ for all $j=1,\ldots,k$,
  \item all critical and branched points of $u$ are special points, i.e. they are contained in the tuples $\bq$ and $\bp$, respectively, or are nodal,
  \item the branching profile over $\bp_i$ is given by $T_i$, i.e.\ the degree of $u$ at $\bq_j$ is given by $l_j$ which is contained in the partition $T_{\nu(j)}$,
  \item $u$ maps nodes to nodes and all preimages of nodes are nodes, the degrees of $u$ from both sides agree and $u$ is locally surjective at nodes and
  \item $\bG$ is a multicurve (see \cref{sec:multicurves}) on $X$, such that every curve in $\bG$ is simple and bounds either
    \begin{itemize}
      \item a disc with exactly two marked points or
      \item a disc with a node whose other smooth component is a sphere with exactly two marked points in addition to the node.
    \end{itemize}
    Furthermore we require that the marked points $\bp$ (i.e.\ the branch points of $u$) are either contained in such a pair of pants or a completely nodal spherical component. Also we require that the labels of the branch points contained in one such disc from above differ by exactly one.
  \end{itemize}
  Its morphisms are given by pairs of maps
  \begin{equation*}
    (\Phi,\varphi):(C,u,X,\bq,\bp,\bG)\Longrightarrow(C',u',X',\bq',\bp',\bG')
  \end{equation*}
  such that
  \begin{itemize}
    \item $\Phi:C\lra C'$ and $\varphi:X\lra X'$ are biholomorphisms which commute, i.e.\ $\varphi\circ u=u \circ\Phi$,
    \item $\Phi(\bq_j(C))=\bq_j(C')$, $\varphi(\bp_i(X))=\bp_i(X')$ for all $i=1,\ldots,n$ and $j=1,\ldots,k$ and
    \item $\varphi_*(\bG_r)=\bG'_r$ for all $r=1\ldots,|\bG|=|\bG'|$.
  \end{itemize}
\end{definition}

\index{Multicurve!Lift of}

\begin{rmk}
  Note that a simple free essential homotopy class on $X$ can be lifted to a set of simple free essential homotopy classes on $C$ such that every element covers a given representative curve and the sum over the degrees is the total degree of $u$. This can be seen e.g. by passing to the unique geodesic representative in the homotopy class. Given the multicurve $\bG$ on $X$ we thus obtain an associated multicurve on $C$ whose elements are simple free homotopy classes and the preimages of essential ones on $\bG$ will also be essential.
  \label{rmk:lift-simple-homotopy-classes}
\end{rmk}

\section{An Embedding of the Moduli Space of Admissible Hurwitz Covers}

\label{sec:def-glue-map}

\subsection{Definition of the Gluing Map on Objects}

\label{sec:embedding-moduli-space-boundary}

Given combinatorial data $g,k,h,n,T,\{l_j\}_{j=1}^k,d,\nu$ we want to associate new combinatorial data $\wt{g},\wt{k},\wt{h},\wt{n},\wt{T},\{\wt{l}_j\}_{j=1}^{\wt{k}},\wt{d},\wt{\nu}$ such that we can define a functor
\begin{equation}
  \label{eq:glue-map}
  \glue:\wh{\mcR}_{g,k,h,n}(T)\lra\wt{\mcR}_{\wt{g},\wt{k},\wt{h},\wt{n}}(\wt{T})
\end{equation}
which we can use to pull back the symplectic orbifold structure of the latter space. We will construct this last structure in \cref{sec:pull-back-orbifold-structure-via-glue}.

The associated combinatorial data will be the following:
\begin{alignat}{2}
  \wt{g}&=g && \wt{k}=\sum_{j=1}^{k}(l_j+1)=nd+k \nonumber \\
  \wt{h}&=h && \wt{n}=2n \label{eq:def-modified-comb-data}\\
  \wt{d}&=d && \hspace{-2mm} \wt{\nu}=j\mapsto\begin{cases}2\nu(i)-1 & j=\sum_{m=1}^{i-1}(l_m+1)+1\\ 2\nu(i)\phantom{-1} & \sum_{m=1}^{i-1}(l_m+1)+1<j<\sum_{m=1}^i(l_m+1)\end{cases} \nonumber \\[1em]
  \wt{T}_i&=\begin{cases} \overbrace{1+\cdots+1}^d & 2\mid i\\ T_{\frac{i+1}{2}}  & 2\nmid i\end{cases} \qquad && \wt{l}_j=\begin{cases}l_i\phantom{1} & j=\sum_{m=1}^{i-1}(l_m+1)+1\\ 1\phantom{l_j} & \sum_{m=1}^{i-1}(l_m+1)+1<j<\sum_{m=1}^i(l_m+1)\end{cases} \nonumber
\end{alignat}
Because this might look strange at the first glance, \cref{fig:comb-data-bounded-hurwitz-cover} summarizes the way this data is chosen.

\begin{rmk}
  Note that with this choice of $\wt{T}$ the objects in $\wt{\mcR}_{g,nd+k,h,2n}(\wt{T})$ have the property that the preimages of the ``bounded'' objects of the multicurve $\Gamma$ under $u$ consist of discs or spheres, respectively. This is because every connected component of the preimage might have some degree $d'$ and genus $g'$ and we obtain
  \begin{equation*}
    2-2g'=d'(2-2\cdot 0)-2(d'-1)=2.
  \end{equation*}
  The count for the degrees can be easily seen in \cref{fig:comb-data-bounded-hurwitz-cover}.
  \label{rmk:genus-preimage-adjoined-hurwitz-cover}
\end{rmk}

\begin{figure}[H]
    \centering
    \def\svgwidth{\textwidth}
    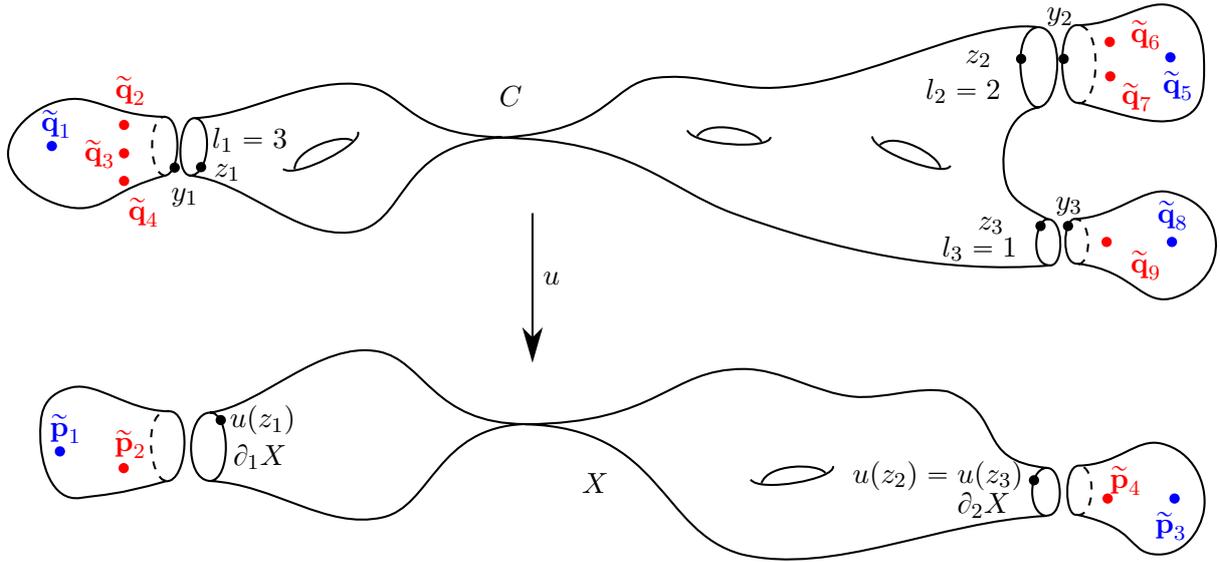
    \caption{Given a bordered Hurwitz cover $u:C\lra X$ we will glue discs (or rather hyperbolic punctures spheres with one geodesic boundary) to the source and target surface. These will always contain one fully branched point and one regular point which is marked nevertheless in order to make the target disc stable. The degree in the picture are $\wt{l}_1=3,\wt{l}_2=\wt{l}_3=\wt{l}_4=1,\wt{l}_5=2,\wt{l}_6=\wt{l}_7=1,\wt{l}_8=1,\wt{l}_9=1$.}
    \label{fig:comb-data-bounded-hurwitz-cover}
\end{figure}

Now suppose we are given an object $(C,u,X,\bq,\bp,\bz)\in\wh{\mcR}_{g,k,h,n}(T)$. We will do the following modifications:

\begin{enumerate}
  \item Glue in hyperbolic pairs of pants to the boundary components of the target surface $X$ and choose biholomorphic charts for these.
  \item Fix a particular covering of such a pair of pants and uniformize it.
  \item Glue together these covers along the boundary components.
  \item Add specific branched covers at the branched points.
  \item Enumerate everything in an appropriate way.
\end{enumerate}

\begin{rmk}
  There are various ways how we can modify a Hurwitz cover with boundary to obtain an actual branched cover. In particular we could choose collar neighborhoods of the boundary components, glue in a standard disc and extend $u$ over that disc as the map $z\mapsto z^{l_j}$. This is a standard technique which can for example also be used to prove the Riemann existence theorem. However, in our case it causes difficulty because it is hard to control how the complex structure depends on the choice of the collar neighborhood. In particular we get problems when trying to prove injectivity of the gluing map because we would need to restrict a biholomorphism to the interior but we can not ``refind'' the boundary as the uniformization might give us a different hyperbolic geodesic representative. In order to avoid this problem we will instead build the new Hurwitz cover by gluing together hyperbolic surfaces.
\end{rmk}

\subsubsection{Glue in Hyperbolic Pairs of Pants to Target Surface \texorpdfstring{$\boldsymbol{X}$}{X}.}

First we do the same thing as Mirzakhani, see~\cite{mirzakhani_weil-petersson_2007}, and glue appropriate hyperbolic pairs of pants to $X$.

To this end uniformize the surface $X$ such that all special points are cusps and the boundary components are geodesics. For any such boundary geodesic $\del_iX$ we can now build the corresponding pair of pants $\Sigma(i)$ with one geodesic boundary component of the same length as $\del_iX$ and two punctures. Such a pair of pants is unique up to unique isometry if we distinguish the two punctures. We do this by enumerating one puncture with $\wt{\bp}_{2i-1}$ and the other $\wt{\bp}_{2i}$.

Mark one point $y_i$ on the boundary of $\Sigma(i)$ which is the endpoint of the unique geodesic that is perpendicular to the boundary of the hyperbolic pair of pants and goes up the cusp $\wt{\bp}_{2i-1}$. We build a new hyperbolic surface by gluing all these pairs of pants to $X$ along their common boundaries and requiring that the marked point $u(z_j)\in\del_iX$ coincides with $y_i$, see \cref{fig:comb-data-bounded-hurwitz-cover}. However, the final surface $\wt{X}$ will also be modified along branched points.

Now we fix a biholomorphic chart $\phi_i:\Sigma(i)\lra\DD\subset\CC$ for this pair of pants. Such a chart exists by the Riemann mapping theorem and it is unique up to rotation if we require that $\phi_i(\wt{\bp}_{2i-1})=0$. Using this rotation we can assume that $\phi_i(\wt{p}_{2i})\in(-1,0)$ is on the negative real axis. Since reflection by the real axis is an anti-holomorphic map and thus an isometry for the uniformized pair of pants $\DD$ with nodes $0$ and $\phi_i(\wt{p}_{2i})$ we see that the positive real axis $[0,1]\subset\DD$ is a geodesic perpendicular to the boundary $S^1$ and going up the cusp $0$. Thus $\phi_i(y_i)=1$, see \cref{fig:glued-in-cover}.

\subsubsection{Fixing a cover of glued hyperbolic pairs of pants}

We will glue the standard degree-$l_j$ cover to the boundary $\del_jC\lra \del_iX$. For this purpose consider the map $z\longmapsto z^{l_j}$ from $\DD$ to $\DD$. On its image we have the marked points or cusps $0,\phi_i(\wt{\bp}_{2i})\in\DD$ as well as the marked point $1=\phi_i(y_i)$ which will be glued to $u(z_j)\in\del_iX$. Thus we mark the $l_j$ preimages of $\phi_i(\wt{\bp}_{2i})$ under this map which will be $l_j$ points contained in $\wt{\bq}$. However, their precise enumeration will be explained at the end. Furthermore we mark $0$ as a critical point of degree $l_j$ and for gluing purposes we mark $1\in \DD$. However, the latter point will not be included in the list of marked points $\wt{\bq}$ later on.

Now uniformize this surface such that the unit circle becomes a geodesic and such that all the $l_j+1$ interior marked points become cusps. This way we obtain hyperbolic surfaces $\wt{\Sigma}(j)$ for every $j$ such that $\nu(j)=i$ together with conformal coverings $f_j:\wt{\Sigma}(j)\lra\Sigma(i)$ mapping cusps to cusps and boundaries to boundaries which are thus local isometries in the hyperbolic metrics and coverings of degree $l_j$. Also this means that the length of the boundary of $\wt{\Sigma}(j)$ is $l_j\cdot l(\del_iX)=l(\del_jC)$.

Using \cref{lem:local-gluing} we see that we can glue the $f_j:\wt{\Sigma}(j)\lra\Sigma(i)$ to $u:\del_jC\lra\del_iX$ such that we obtain hyperbolic surfaces containing $C$ and $X$ together with a holomorphic extension of the map $u$. This extension will have new critical and branch points and cusps corresponding to the earlier marked points.

\begin{figure}[H]
  \centering
  \def\svgwidth{\textwidth}
  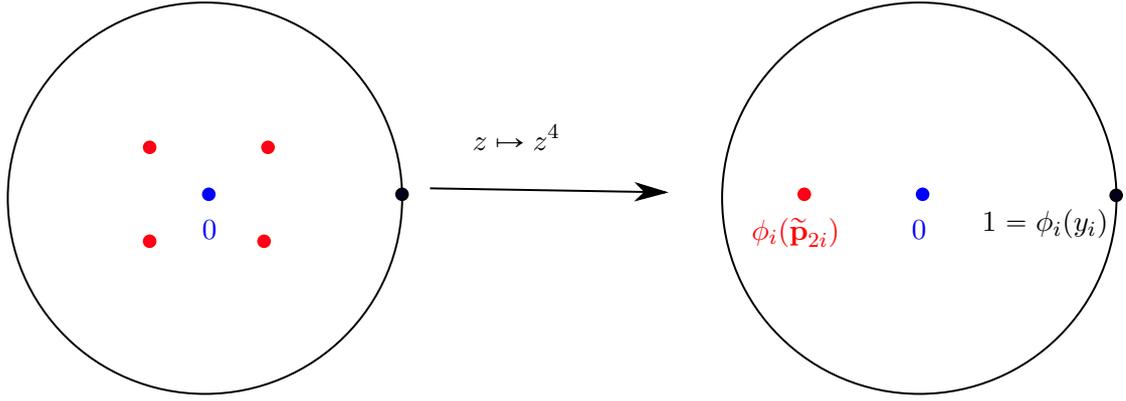
  \caption{These are the standard disc covers that we glue in. On the right side you can see the target disc with three marked points and on the left you see the source disc where one point is completely branched and the others are regular. Notice that in both parts we fix the rotation of the boundary geodesic by identifying $1$ with the corresponding marked point on the boundary which comes from the reference curve.}
  \label{fig:glued-in-cover}
\end{figure}

\begin{rmk}
Note that we don't mark the points $z_j$ as they can actually be recovered from a pair of pants decomposition and the usual Fenchel--Nielsen coordinate construction. However, we will remember the free homotopy classes of the former boundaries on $X$ and so by Remark~\ref{rmk:lift-simple-homotopy-classes} we obtain lifted free homotopy classes on the modified $C$. This way we will be able to refind these boundaries by looking at the unique geodesic representative in this class.
\label{rmk:difference-hyperbolic-complex-gluing}
\end{rmk}

Note that this cover is actually unique in the following sense.

\begin{prop}
  There exists only one equivalence class of branched degree-$k$ covers $f:U\lra\DD$ such that $0$ is the only branched point and is fully branched. Here, equivalence means that there exists a biholomorphism $\phi:U\lra U'$ such that $f'\circ\phi=f$ for $f:U\lra\DD$ and $f':U'\lra\DD$. Also every two such equivalent branched covers are biholomorphic in a unique way given by multiplication by a $k$-th root of unity.
  \label{prop:eq-class-hurwitz-disc}
\end{prop}

\begin{proof}
  Suppose we are given a branched cover $f:U\lra\DD$ of degree $k$. Denote $p\coloneqq f^{-1}(0)$. Since $f$ is branched only around zero it induces an actual covering $f:U\setminus \{p\}\lra\DD\setminus\{0\}$. Using the local form for the branched cover $f$ close to $p$ we see that a small curve in $U\setminus \{p\}$ around $p$ is mapped to the homotopy class of $t\mapsto e^{2\pi\ii kt}$, i.e.\ $k$ times the generator of $\pi_1(\DD\setminus\{0\},1)$ where we have chosen $1\in\DD\setminus\{0\}$ and $q\in U\setminus\{p\}$ with $f(q)=1$ as the base points for fundamental groups as well as for the covering arguments. Thus the image of $f_*$ in $\pi_1(U\setminus\{p\},q)$ is the same as the image of $g:\DD\setminus\{0\}\lra \DD\setminus\{0\}$ given by $g(z)=z^k$ in $\pi_1(\DD\setminus\{0\},1)$. Therefore the map $f$ lifts in the following diagram to a map $h:U\lra\DD\setminus\{0\}$.
  \begin{equation*}
    \xymatrix{
      & \DD\setminus\{0\} \ar[d]^g \\
      U\setminus\{p\} \ar@{-->}[ur]^h \ar[r]^f & \DD\setminus\{0\}
      }
  \end{equation*}
  The map $h:U\setminus\{p\}\lra\DD\setminus\{0\}$ is smooth and holomorphic as it is locally given by $f^{-1}\circ g$. By the removable singularity theorem it extends to a holomorphic map $h:U\lra\DD$. It is injective as $f$ and $g$ are both of degree $k$ and surjective by openness. Thus its differential is everywhere nonzero and it is therefore a biholomorphism. The lift $h$ was unique after choosing a preimage of a regular base point in $\DD\setminus\{0\}$ under $g$, i.e.\ there are $k$ different lifts. However this means that any branched cover $f:U\lra\DD$ of this type is equivalent to the standard branched cover $z\longmapsto z^k$.

  It remains to show that two such biholomorphisms differ only by a $k$-th root of unity. Suppose we have constructed biholomorphisms $h:U\lra\DD$ and $h':U\lra\DD$ in this way. Then we obtain a biholomorphism $h'\circ h^{-1}:\DD\lra\DD$ preserving the origin which is thus given by a rotation. Since this map needs to preserve fibres of $z\longmapsto z^k$ the rotation needs to be a $k$-th root of unity.
\end{proof}

\begin{rmk}
  \cref{prop:eq-class-hurwitz-disc} says in particular that after choosing the hyperbolic pair of pants $\Sigma(i)$ there exists only one way of gluing a branched covering of genus zero with type $(1+\cdots+1,l_j)$ and one boundary component to $\del_jC$ if we require that a marked point on $\del_jC$ is identified with $1\in\DD$. The $l_j$ different isomorphisms in the proposition correspond to the choice of the $l_j$ lifts which in turn correspond to the $l_j$ possible identifications of the boundary $\del_jC$ with $S^1$. This identification is then fixed by the choice of $z_j$ in the fibre $u^{-1}(u(z_j))$.
\end{rmk}

\subsubsection{Modifying the Cover at Critical and Branched Points}

\label{sec:glue-at-branch-point}

Now we modify the surfaces and the map at the critical and branch points. Let $\bp_i=p_i$ be a branch point and $\bq_j=q_j\in C$ a critical preimage.

Add a complex sphere $\CC P^1(i)$ to $X$ in such a way that $[0:1]\in\CC P^1(i)$ and $p_i$ are a nodal pair. Furthermore glue a complex sphere $\CC P^1(j)$ to $C$ such that $[0:1]\in\CC P^1(j)$ and $q_j$ form a nodal pair. On $\CC P^1(j)$ define the map $f_j:\CC P^1(j)\lra \CC P^1(i)$ by
\begin{equation*}
  f_j([x:y])=[x^{l_j}:y^{l_j}].
\end{equation*}
This map has two fully ramified critical points $[0:1]$ and $[1:0]$, one of which will be a nodal point in $\wt{X}$ and $[1:0]$ which will be marked as will be its image $[1:0]\in\CC P^1(i)$.  In order to make the image component stable (at the moment it is a sphere with two special points) we also mark the point $[1:1]\in\CC P^1(i)$ as well as its preimages $[\zeta_{l_j}^k:1]\in\CC P^1(j)$, where $\zeta_{l_j}$ is a primitive root of unity. Note that these preimages all have degree $1$ but we mark them anyway.

We do this for all $i$ such that $\bp_i$ is a branched point. The maps fit together as they coincide on the new nodal points by construction.

This way we obtain surfaces $\wt{X}$ and $\wt{C}$ as well as a map $u$ between them which can be either seen as a branched cover of closed nodal complex curves or as a locally-isometric covering between complete nodal hyperbolic surfaces with cusps and finite area.

\begin{figure}[H]
  \centering
  \def\svgwidth{\textwidth}
  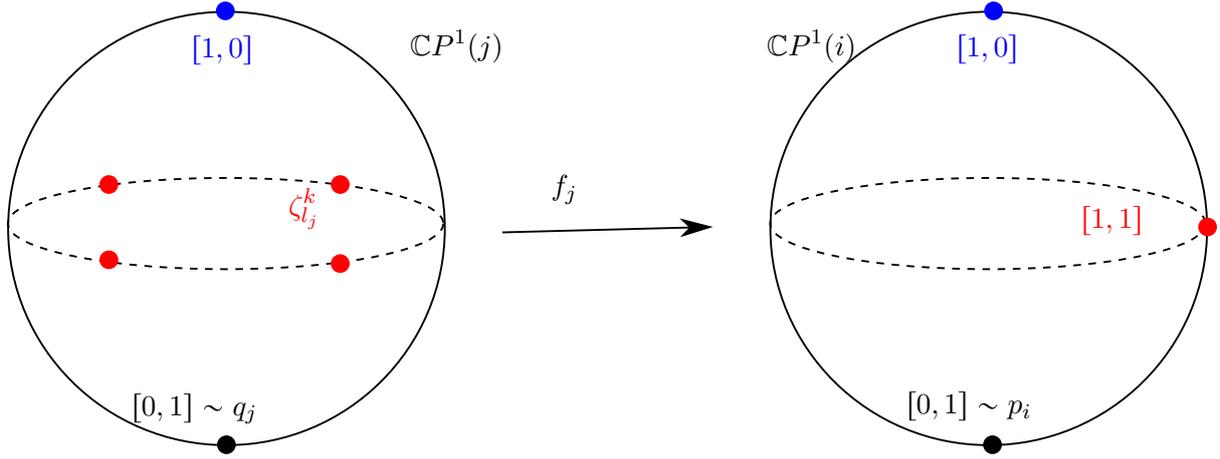
  \caption{This is an illustration of the spheres that we glue in at punctures.}
  \label{fig:glued-in-spheres}
\end{figure}

\subsubsection{Enumeration and Homotopy Classes of Curves}

Now that we have defined the map $\wt{u}:\wt{C}\lra\wt{X}$ we need to specify the remaining data.

As was described above on the surface $\wt{X}$ we mark the following points:
\begin{itemize}
  \item If $\bp_i=\del_iX$ is an actual boundary component then we mark the cusps of the glued-in hyperbolic pair of pants.
  \item If $\bp_i=p_i$ is a branched point then we mark $[1:0],[1:1]\in \CC P^1(i)$.
\end{itemize}
Here, one point is marked because we need to mark all branched points. The second one is marked in order to make the glued surface stable. Note that in both cases these additional points are not branched and thus we have $d$ preimages -- on each connected component $l_j$-many.

As we need to mark every preimage of a branched or marked point we mark the following points on the surface $\wt{C}$ as seen above:
\begin{itemize}
  \item If $\bq_j=\del_jC$ is an actual boundary component then we mark $0\in \DD(j)$ and the $l_j$ preimages of the image of the other cusp under the biholomorphic chart $\phi_i$.
  \item If $\bq_j=q_j$ is a critical point then we mark $[1:0]\in\CC P^1$ and all $[\zeta_{l_j}^k:1]\in\CC P^1$.
\end{itemize}
Note that every marked point is in one of the new glued components. Furthermore every such component in the target has one branch point and one marked non-branch point but the spheres also have a nodal point that is branched but not marked. Also the components in the source have one point of (maximal) degree $l_j$ and $l_j$ simple marked points.

Regarding the enumeration of the marked points we work as follows. First we order everything according to the index $j$ or $i$ that was used to glue. Then we start at the lowest index and assign the first not-yet-used index for the branched point or the completely branched preimage of that point, respectively. Then we denote the remaining point on the target component by that index plus one. On the preimage we mark again first the point of degree $l_j$ and then continue with the $l_j$ preimages of the other cusp in the positive direction starting from the direction of $1$. This ordering is also well-defined on $\CC P^1$ if we start with $[1:1]$ and then increase the power of the $l_j$-th root of unity in the first component. See the next picture as an illustration.

\begin{figure}[H]
  \centering
  \def\svgwidth{\textwidth}
  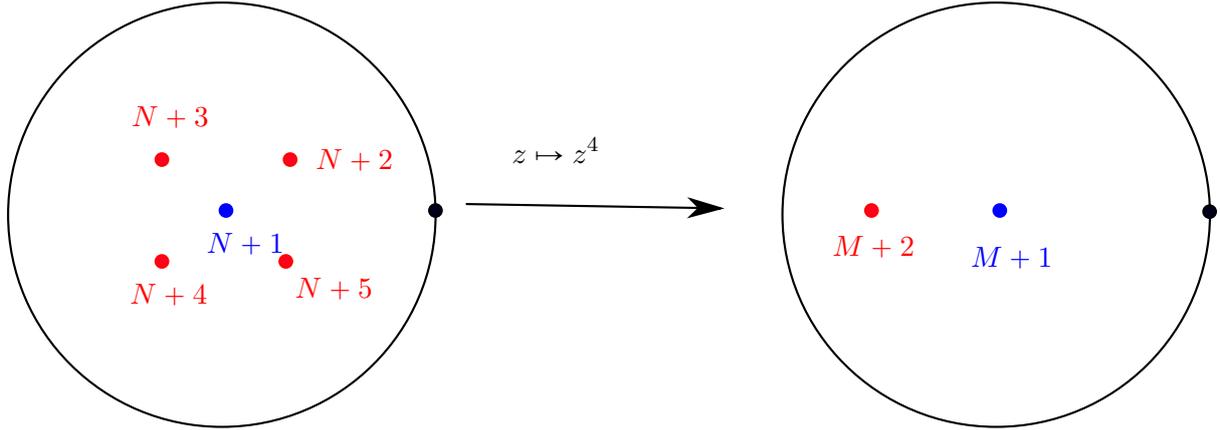
  \caption{This figure illustrates the enumeration scheme we are using for the example of a disc. Note that $M$ and $N$ are the number of marked points already assigned on the target and source surface, respectively, before arriving at this particular boundary component.}
  \label{fig:glued-in-cover-enumeration}
\end{figure}

It remains to define the free homotopy classes of simple closed curves on $\wt{X}$ denoted by $\bG$. Of course we choose all the images of the free homotopy classes of the boundaries $\del_iX$ as well as the reference curves on $X$ close to nodes under the inclusion $X\xhookrightarrow{}\wt{X}$. Recall that $\Gamma_j(C)$ denotes a reference curve on $C$ close to $\del_jC$ or $q_j$ and we chose the lengths such that they correspond to one reference curve close to $\del_{\nu(j)}X$ or $p_{\nu(j)}$, respectively. Note that the preimages of $\bG$ on $C$ then correspond to the simple free homotopy class of the curves $\Gamma_j(C)$, i.e.\ the boundary curves or the reference curves close to a node.

The enumeration and various definitions are illustrated in Figure~\ref{fig:glue-def-choices-enumeration}. Note that this figure draws the hyperbolic picture, i.e.\ after making all these choices we equip the surfaces including the spheres with the unique hyperbolic metric such that all marked points (i.e.\ \emph{not} the points $z_j$) and nodes are cusps. Here we stress once more that the construction was done in such a way that the former boundary components $\del_jC$ were geodesic in the hyperbolic structure on $C$ but are now still geodesic in the uniformized hyperbolic structure in $\wt{C}$ because we built the surfaces from gluing along geodesic boundaries instead of gluing along conformal cylinders.

To sum up we have defined a map
\begin{equation*}
  \glue:\Ob\wh{\mcR}_{g,k,h,n}(T)\lra\Ob\wt{\mcR}_{g,nd+k,h,2n}(\wt{T}).
\end{equation*}

\begin{figure}[H]
  \centering
  \def\svgwidth{\textwidth}
  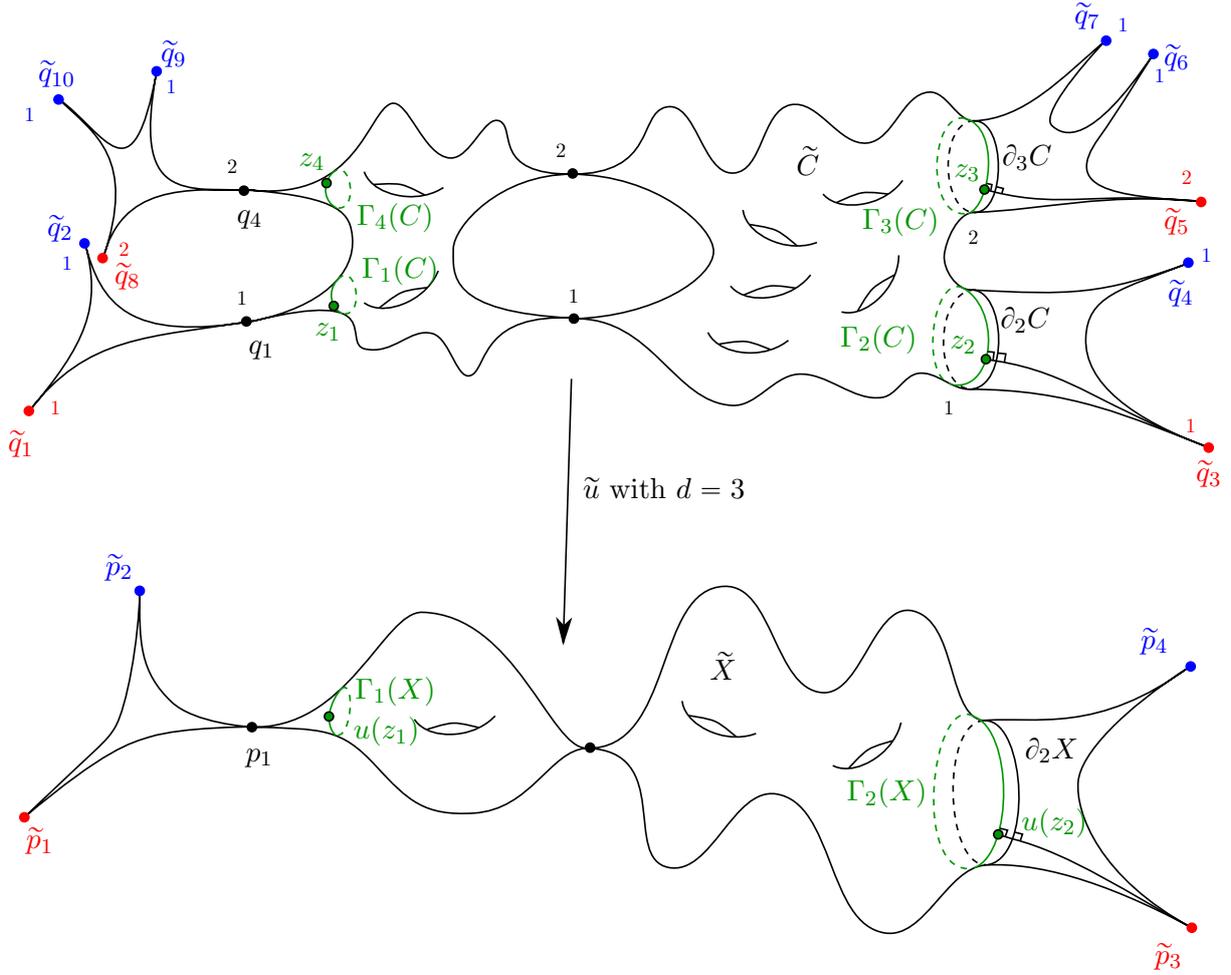
  \caption{This figure illustrates the various enumeration conventions. Note that the surfaces are drawn hyperbolically so the enumeration order of e.g.\ $\wt{q}_6$ and $\wt{q}_7$ is not clear from the picture. Also the colors are there just for visibility. Note that we do neither mark nor modify the interior nodes. However, there are new nodes appearing due to the existence of branched points on the left hand side. The small numbers show the local degree of the corresponding node/critical point/boundary component. Note that the reference curves $\Gamma_i(X)$ and $\Gamma_j(C)$ are not explicitly part of the data in $(C,u,X,\bq,\bp,\bz)$ but are uniquely fixed by the choice of a function $F$ beforehand, see \cref{sec:orb-structure-mod-space-bordered-hurwitz-covers-definitions}.}
  \label{fig:glue-def-choices-enumeration}
\end{figure}

\subsection{The Gluing Map on Morphisms}

Next we show how to extend a morphism
\begin{equation*}
  (\Phi,\varphi):(C,u,X,\bq,\bp,\bz)\lra (C',u',X',\bq',\bp',\bz')
\end{equation*}
in $\wt{\mcR}_{g,k,h,n}(T)$ to a morphism
\begin{equation*}
  (\wt{\Phi},\wt{\varphi}):(\wt{C},\wt{X},\wt{u},\wt{\bq},\wt{\bp},\wt{\bz})\lra (\wt{C'},\wt{X'},\wt{u'},\wt{\bq'},\wt{\bp'},\wt{\bz'})
\end{equation*}
in $\wt{\mcR}_{g,nd+k,h,2n}(\wt{T})$. This is a local problem. Since we modified the surfaces and their data only at boundaries and critical points we only need to extend the maps in those neighborhoods.

\subsubsection{Boundary Components}

Consider a boundary component $\del_iX$ together with its preimages $\del_jC$ where $u$ has degree $l_j$ and $\nu(j)=i$ on $(C,u,X,\bq,\bp,\bz)$ as well as the corresponding objects in $(C',u',X',\bq',\bp',\bz')$, denoted by $\del_jC'$ and $\del_iX'$. This means that $\Phi(\del_jC)=\del_jC'$ and $\varphi(\del_iX)=\del_iX'$.

In order to prove that the isomorphisms $\Phi$ and $\varphi$ extend over the glued in pairs of pants we need to show that the holomorphicity of these maps implies that on the boundary of the pairs of pants they are given by the identity map in the glued-in disc charts.

We will start with the target surfaces $\wt{X}$ and $\wt{X'}$. Consider $\varphi:\del_iX\lra\del_iX'$. Since $\varphi$ is an isometry it preserves the length of the boundary geodesic. Thus the glued-in pairs of pants are identical. Since they are glued in such a way that $y_i$ and $y_i'$ are identified with $z_i$ and $z_i'$ we can extend the map $\varphi$ on these pairs of pants by the identity map and use \cref{lem:local-gluing} to see that this gives a well-defined extension.

In the same way we can extend $\Phi$ as the identity map in charts over the pairs of pants $\Sigma(j)$. Recalling how we defined the enumeration and that orientations agree one sees easily that this extension of $\Phi$ to $\wt{\Phi}$ preserves enumerations of marked points.

\subsubsection{Punctures and Free Homotopy Classes}

There is nothing to show at the marked points $q_j$ and $p_i$ because we glued in some standard spheres with a fixed map. Thus we can extend $\Phi$ and $\varphi$ by the identity on these spheres. Thus we obtain maps $\wt{\Phi}:\wt{C}\lra\wt{C'}$ and $\wt{\varphi}:\wt{X}\lra\wt{X'}$ which intertwine $\wt{u}$ and $\wt{u'}$ and preserve all the marked points and their enumerations. Also, the free homotopy classes $\bG$ and $\bG'$ are mapped to each other via $\wt{\Phi}$ and $\wt{\varphi}$ because the original maps $\Phi$ and $\varphi$ preserved the boundary components of the bordered surfaces.

All in all we see that we obtain a map
\begin{equation*}
	\glue_{\Mor}:\Mor\wh{\mcR}_{g,k,h,n}(T)\lra\Mor\wt{\mcR}_{g,nd+k,h,2n}(\wt{T}).
\end{equation*}
It is easy to see that this gives indeed a functor $\glue:\wh{\mcR}_{g,k,h,n}(T)\lra\wt{\mcR}_{g,nd+k,h,2n}(\wt{T})$.

\begin{rmk}
  Notice that this extension of the morphisms is in fact unique under our assumptions. At a boundary the holomorphic map on the glued punctured disc is determined by its boundary condition and this boundary condition is determined by requiring that certain points agree, in our case the (arbitrarily fixed) point $y_i$ corresponding to $1\in\DD$ and the unique endpoint of the lift of the positive real line segment on the disc to the fully branched marked point and $z_i$. On the other hand at a pair of corresponding punctures we added spheres and required the marked points to be mapped to a specific marked point. But since these are at least three points this also fixes the biholomorphism.
\end{rmk}

\subsection{Properties of the \texorpdfstring{$\boldsymbol{\glue}$}{glue}-Functor}

Next we want to analyze the functor $\glue$ a bit further. First we show that $\glue$ is essentially surjective and afterwards we determine the fibre of $\glue|_{\Ob}$, it will consist of a single point or a product of circles if there are degenerate boundary components. At the end we will see that the same property holds for $\glue|_{\Mor}$.

\begin{lem}
  The functor $\glue:\wh{\mcR}_{g,k,h,n}(T)\lra\wt{\mcR}_{g,nd+k,h,2n}(\wt{T})$ is essentially surjective.
  \label{lem:glue-functor-surjectivity}
\end{lem}

\begin{proof}
  Let a Hurwitz cover $(\wt{C},\wt{X},\wt{u},\wt{\bq},\wt{\bp},\wt{\bG})\in\Ob\wt{\mcR}_{g,nd+k,h,2n}(\wt{T})$ be given. First uniformize $X$ to obtain a hyperbolic metric. We will cut the Hurwitz cover along the curves in the essential classes in $\wt{\bG}$ and the cusps corresponding to spheres with two marked points in the target.

  In every essential class in $\wt{\bG}$ we can find a unique closed simple geodesic representative whose preimages under $\wt{u}$ are closed simple geodesics and separate the Hurwitz cover into a covering of a pair of pants with boundary with two marked points, one of which is regular and the other one completely branched. Recall that by \cref{rmk:genus-preimage-adjoined-hurwitz-cover} the preimage of this pair of pants consists of a disjoint union of discs with marked points. Now cut the Hurwitz cover along these geodesics.

  At every nodal sphere with two branch points in $X$ we have again by \cref{rmk:genus-preimage-adjoined-hurwitz-cover} that its preimage consists of nodal spherical components. Remove these nodal spheres.

  As the map $T\lra\wt{T}$ defined in \cref{eq:def-modified-comb-data} is invertible we can define an enumeration of the new boundary components and cusps left over from the cutting procedure. This way we obtain a bordered Hurwitz cover $u:C\lra X$ of type $T$ with an enumeration of the boundary components and cusps. It remains to define the points on the reference curves. At cusps  we can use any point we want because they came from nodes where we do not have any kind of information that we can recover from the nodal sphere. At a essential curves giving rise to geodesic boundaries we can pick the end point of the unique hyperbolic geodesic perpendicular to this curve and going up the cusp with the lower index. This defines $\bz$.

  It is now easy to see that for these choices we have
  \begin{equation*}
    \glue(C,u,X,\bq,\bp,\bz)=(\wt{C},\wt{X},\wt{u},\wt{\bq},\wt{\bp},\wt{\bG})
  \end{equation*}
  and that $(C,u,X,\bq,\bp,\bz)\in\Ob\wh{\mcR}_{g,k,h,n}(T)$. This is due to the fact that by uniformizing and gluing hyperbolic surfaces along their geodesic boundary as well as cutting along geodesics we effectively realize $C$ as a subset of $\wt{C}$ and similarly for $X$. Also by \cref{prop:eq-class-hurwitz-disc} there exists only one equivalence class of Hurwitz covers that we can glue to the boundary component of the fixed type and using the given enumeration.
\end{proof}

Notice that $\glue$ is not injective on objects.

\begin{lem}
  The fibre of the $\glue$-functor on objects is given by
  \begin{equation*}
    \glue_{\Ob}^{-1}(\wt{C},\wt{X},\wt{u},\wt{\bq},\wt{\bp},\wt{\bG})=\{(C,u,X,\bq,\bp,\bz)\mid \bz_j\text{ arbitrary at punctures }q_j\}
  \end{equation*}
  where $C,u,X,\bq,\bp$ and $\bz_j$ at non-degenerate boundary components are fixed by the image $(\wt{C},\wt{X},\wt{u},\wt{\bq},\wt{\bp},\wt{\bG})$.
  \label{lem:glue-functor-objects-not-injective}
\end{lem}

\begin{proof}
  This is clear from the proof of \cref{lem:glue-functor-surjectivity}.
\end{proof}

\begin{rmk}
  Notice that the $T^n$-action on $\Ob\wt{\mcR}_{g,nd+k,h,2n}(\wt{T})$ given by Fenchel--Nielsen twisting along the curves in $\bG$ has a fixed point when their lengths are zero. On $\wh{\mcR}_{g,k,h,n}(T)$, however, the corresponding action by rotating the points $\bz$ is actually free. This was one of the major ideas of Mirzakhani in \cite{mirzakhani_weil-petersson_2007}.
\end{rmk}

 Now suppose we are given two bordered Hurwitz covers
 \begin{equation*}
   (C,u,X,\bq,\bp,\bz)\text{ and }(C',u',X',\bq',\bp',\bz')
 \end{equation*}
 which are mapped under $\glue$ to $(\wt{C},\wt{u},\wt{X},\wt{\bq},\wt{\bp},\wt{\boldsymbol{\Gamma}})$ and $(\wt{C'},\wt{u'},\wt{X'},\wt{\bq'},\wt{\bp'},\wt{\boldsymbol{\Gamma'}})$ and also a morphism $(\wt{\Phi},\wt{\varphi}):(\wt{C},\wt{u},\wt{X},\wt{\bq},\wt{\bp},\wt{\boldsymbol{\Gamma}})\lra(\wt{C'},\wt{u'},\wt{X'},\wt{\bq'},\wt{\bp'},\wt{\boldsymbol{\Gamma'}})$. Then we would like to show that we can construct from this morphism of Hurwitz covers a unique morphism of the corresponding Hurwitz covers with boundary, i.e.
\begin{equation*}
  (\Phi,\varphi):(C,u,X,\bq,\bp,\bz)\lra (C',u',X',\bq',\bp',\bz').
\end{equation*}
However, this statement can not be true because in the gluing construction at the nodes we forget the marked point $z$ so this map can only be injective up to the position of the points. We prove

\begin{prop}
  Consider the functor $\glue:\wh{\mcR}_{g,k,h,n}(T)\lra\wt{\mcR}_{g,nd+k,h,2n}(\wt{T})$ and a morphism $(\wt{\Phi},\wt{\varphi})\in\wt{\mcR}_{g,nd+k,h,2n}(\wt{T})$ with
  \begin{equation*}
    (\wt{\Phi},\wt{\varphi}):(\wt{C},\wt{u},\wt{X},\wt{\bq},\wt{\bp},\wt{\boldsymbol{\Gamma}})\lra(\wt{C'},\wt{u'},\wt{X'},\wt{\bq'},\wt{\bp'},\wt{\boldsymbol{\Gamma'}}).
  \end{equation*}
  Then for every $(C,u,X,\bq,\bp,\bz)\in\glue_{\Ob}^{-1}((\wt{C},\wt{u},\wt{X},\wt{\bq},\wt{\bp},\wt{\boldsymbol{\Gamma}}))$ there exists a unique morphism
  \begin{equation*}
    (\Phi,\varphi):(C,u,X,\bq,\bp,\bz)\lra (C',u',X',\bq',\bp',\bz'')
  \end{equation*}
  which agrees with a given $(\wt{\Phi},\wt{\varphi})$ everywhere except on the points $z_j\in\Gamma_j(C)$ at punctures $q_j$ and is thus mapped to $(\wt{\Phi},\wt{\varphi})$ under $\glue$.
  \label{prop:glue-morphisms-unique}
\end{prop}

\begin{proof}
To this purpose uniformize first the surfaces $\wt{C},\wt{X},\wt{C'}$ and $\wt{X'}$ to get hyperbolic metrics such that all the marked points and nodes are punctures. For these metrics the maps $\wt{\Phi}$ and $\wt{\varphi}$ are isometries and $\wt{u}$ and $\wt{u'}$ are local isometries.

Now consider the unique geodesic representatives in $\boldsymbol{\Gamma}$ and $\boldsymbol{\Gamma'}$ which are essential curves. By construction of the closed surfaces we know that these geodesics are given by the boundary curves $\del_jC\subset \wt{C}$, $\del_jC'\subset \wt{C'}$, $\del_iX\subset \wt{X}$ and $\del_iX'\subset \wt{X'}$. Since $\wt{\Phi}$ and $\wt{\varphi}$ are isometries and preserve the free homotopy classes they map these boundary components onto each other.

Thus we can restrict $\wt{\Phi}$ and $\wt{\varphi}$ to the interior surfaces $C, C', X$ and $X'$ by cutting at the boundaries and the nodes having one component with marked points. These maps $\Phi:C\lra C'$ and $\varphi:X\lra X'$ obviously intertwine $u$ and $u'$ and also preserve the enumeration of $\bq$ and $\bz$. However we need to check whether they preserve the marked points $\bz$ and $\bz'$ as well.

Consider first $z_j\in\Gamma_j(C)$ with $\bq_j=q_j$ a critical point. Then we can not say anything about the relation of $z_j'$ and $\Phi(z_j)$ as we do not use these points in the construction of the closed Hurwitz cover. Thus the preimage of the map $\glue$ always consists of all possible choices for such marked points $z_j$ close to critical points.

Now consider $z_j\in\Gamma_j(C)\simeq \del_jC$ close to an actual boundary component $\bq_j=\del_jC$. For this purpose switch to a local description of the pairs of pants $\Sigma(j)\subset\wt{C}$ and $\Sigma'(j)\subset\wt{C}$. We defined the maps $\wt{u}$ and $\wt{u'}$ on them by first choosing biholomorphisms with the unit disc. They were required to map the lowest-indexed cusp to $0\in\DD$ and the boundary to the unit circle meaning they are unique up to a rotation. However, since $\Phi$ also preserves the enumeration of the other marked points this means that the map $\Phi$ is given by the identity on these discs. Thus $1$ is mapped to $1$ and we can recover $z_j$ and $z_j'$ as the preimages of $1$ under their respective charts. Thus $\Phi(z_j)=z_j'$.
\end{proof}

\section{Pulling Back the Orbifold Structure}

\label{sec:pull-back-orbifold-structure-via-glue}

\subsection{Orbifold Structure on \texorpdfstring{$\boldsymbol{\wt{\mcR}_{g,k,h,n}(T)}$}{R(T)}}

\label{sec:orbifold-structure-wtmcr}

In this section we will define an orbifold structure for $\wt{\mcR}_{g,k,h,n}(T)$ for which we will use essentially that the object and morphism spaces are locally homeomorphic to those of $\mcM_{g,k,h,n}(T)$.\footnote{Note that we temporarily use $g,k,h,n$ and $T$ in this section for readability. This section will of course be applied to $\wt{g},\wt{k},\wt{h},\wt{n}$ and $\wt{T}$.}

Consider an object $\lambda=(C,u,X,\bq,\bp,\boldsymbol{\Gamma})$ of a class in $\wt{\mcR}_{g,k,h,n}(T)$. First we can build the family $\Psi^{\lambda}:O^{\lambda}\lra\Ob\mcR_{g,k,h,n}(T)$ as in \cref{sec:local-parametrizations} after choosing the objects from \cref{sec:choices-hyp}, i.e.\ hyperbolic metrics, sets of decomposing curves, compact sets including the boundary horocycles away from the nodes, disc structures as well as some roots of unity $\xi$ for all the nodes. We require that the sets of decomposing curves \emph{include} the free homotopy classes of curves $\bG$. Regarding notation these choices are again included in the symbol $\lambda$.

This family describes a variation of $C, X$ and $u$ as well as $\bq$ and $\bp$, so it remains to describe the free homotopy classes of curves at the neighboring Hurwitz covers. Suppose the marked points $\bp_i$ and $\bp_{i+1}$ are on the same pair of pants bounded by either a node or a curve in $\boldsymbol{\Gamma}$. If it is bounded by a curve then this curve still exists in the glued surface $\Psi^{\lambda}(b)$ for $b$ close to the original surface $\lambda$ and we can use this one. If it is bounded by a node then there are again two possibilities. Either on $\Phi^{\lambda}(b)$ this node was not modified and we can keep the same class of curves or it was opened. If it was replaced by a cylinder we define the corresponding new element in $\bG$ to be the curve wrapping once around the glued in cylinder. This means that as objects we can still use
\begin{equation*}
	\obj\wt{\mcM}_{g,k,h,n}(T)\coloneqq \bigsqcup_{\lambda\in\Lambda}O^{\lambda},
\end{equation*}
where the $\lambda$ now include a choice of free homotopy classes of curves $\bG$, the maps $\Psi^{\lambda}$ include the construction as outlined above and the set $\Lambda$ ranges over a large enough set of possible choices to cover all equivalence classes in $|\wt{\mcR}_{g,k,h,n}(T)|$. As in \cref{sec:constr-an-orbif} this defines a manifold $\Ob\wt{\mcM}_{g,k,h,n}(T)$.

Now we need to look at the morphism set. Again, as in \cref{sec:mfd-structure-morphisms} we define $M(\lambda,\lambda')$ as the set of all fibre isomorphisms of correct order between $\Psi^{\lambda}(O^{\lambda})$ and $\Psi^{\lambda'}(O^{\lambda'})$ where we require now that each fibre isomorphism also respects the homotopy classes of curves $\bG$. Also we define
\begin{equation*}
  \Mor\wt{\mcM}_{g,k,h,n}(T)\coloneqq \bigsqcup_{\lambda,\lambda'\in\Lambda}M(\lambda,\lambda').
\end{equation*}
 Now we need to show that these sets $M(\lambda,\lambda')$ are still manifolds. Recall that we obtained charts for these sets by continuing the morphism between the central fibres $\lambda$ and $\lambda'$ in a unique way by looking at the universal unfoldings of the source and target surfaces. Notice that the free homotopy classes of curves defined above in neighborhoods of $\lambda$ and $\lambda'$ in the family are clearly preserved by the extended morphisms from \cref{sec:mfd-structure-morphisms}. Thus we get again manifold charts for $M(\lambda,\lambda')$. All the other properties of an orbifold groupoid follow immediately from the fact that $\mcM_{g,k,h,n}(T)$ is an orbifold groupoid.

This equips $|\wt{\mcM}_{g,k,h,n}(T)|$ with an orbifold structure and we have the inclusion functor
\begin{equation}
  \iota:\wt{\mcM}_{g,k,h,n}(T)\lra\wt{\mcR}_{g,k,h,n}(T)
  \label{eq:inclusion-bordered-moduli-spaces}
\end{equation}
which still has the same properties as in \cref{prop:properties-functors-between-groupoids}, i.e.\ it is not surjective on morphisms as we dropped all those that had a lower order in $\wt{\mcM}_{g,k,h,n}(T)$.

\subsection{Orbifold Structure on \texorpdfstring{$\boldsymbol{\wh{\mcR}_{g,k,h,n}(T)}$}{R(T)}}

We will see later that there is a symplectic Hamiltonian $T^n$-action on $\wt{\mcM}_{g,nd+k,h,2n}(\wt{T})$ by rotating the glued-in pairs of pants in the target surface. Unfortunately this action is not free in the nodal case which is why we look at the space $\wh{\mcM}_{g,k,h,n}(T)$ instead. This space will have a free $T^n$-action by rotating the points on the reference curves close to the boundaries and cusps. Unfortunately this space is \emph{not} an orbifold as by \cref{lem:glue-functor-objects-not-injective} an inner point in the orbifold $\wt{\mcM}_{g,nd+k,h,2n}(\wt{T})$ gets replaced by a circle. Thus it is rather something like an orbifold with corners, a more difficult concept that would lead us to far away. Instead we will equip the ``inner'' part consisting of actual bordered Hurwitz covers with an orbifold structure and the whole orbit space with a topology only. This will be enough to describe the limit Chern class at the fixed locus of the torus action.

Define $\wh{\mcR}^{\square}_{g,k,h,n}(T)$ as the full subcategory of $\wh{\mcR}_{g,k,h,n}(T)$ such that the Hurwitz covers have actual boundary components and no cusps. Notice that they can still have \emph{interior} nodes, i.e.\ the full subcategory of smooth Hurwitz covers will still be denoted by $\wh{\mcR}^{\circ}_{g,k,h,n}(T)$. Also define $\wt{\mcR}^{\square}_{g,nd+k,h,2n}(\wt{T})$ as the full subcategory of
\begin{equation*}
  \glue_{\Ob}\left(\Ob\wh{\mcR}^{\square}_{g,k,h,n}(T)\right)
\end{equation*}
in $\wt{\mcR}_{g,nd+k,h,2n}(\wt{T})$. This category consists of those Hurwitz covers $(C,u,X,\bq,\bp,\bG)$ where every element in the multicurve $\bG$ bounds a disc with two marked points, i.e.\ the curves are not contractible to a node. Note that this is an open condition\footnote{Although we will define the topology only in \cref{sec:topology-moduli-spaces}.} and thus $\Ob\wt{\mcR}^{\square}_{g,nd+k,h,2n}(\wt{T})\subset\Ob\wt{\mcR}_{g,nd+k,h,2n}(\wt{T})$ is an open subset, i.e.\ we can restrict the orbifold structure $\wt{\mcM}_{g,nd+k,h,2n}(\wt{T})$ to this subcategory.

By \cref{lem:glue-functor-surjectivity}, \cref{lem:glue-functor-objects-not-injective} and \cref{prop:glue-morphisms-unique} we have that

\begin{equation*}
  \glue:\wh{\mcR}^{\square}_{g,k,h,n}(T)\lra\wt{\mcR}^{\square}_{g,nd+k,h,2n}(\wt{T})
\end{equation*}

is a category equivalence which induces a bijection on orbit spaces and thus we can pull back the orbifold structure from $\wt{\mcM}^{\square}_{g,nd+k,h,2n}(\wt{T})$ in an obvious way. Therefore we have

\begin{lem}
  The groupoid $\wh{\mcR}^{\square}_{g,k,h,n}(T)$ carries a natural orbifold structure given by
  \begin{equation*}
    \wt{\mcM}^{\square}_{g,nd+k,h,2n}(\wt{T}) \lra \wt{\mcR}^{\square}_{g,nd+k,h,2n}(\wt{T}) \xrightarrow{\glue^{-1}} \wh{\mcR}^{\square}_{g,k,h,n}(T).
  \end{equation*}
  \label{lem:orbifold-structure-rtildesquare}
\end{lem}

\begin{definition}
  We can thus define the orbifold groupoid
  \begin{equation*}
    \wh{\mcM}^{\square}_{g,k,h,n}(T)\coloneqq \glue^{-1}(\iota(\wt{\mcM}^{\square}_{g,nd+k,h,2n}(\wt{T}))).
  \end{equation*}
\end{definition}

Furthermore suppose we are given a topology on $\wt{\mcR}_{g,nd+k,h,2n}(\wt{T})$:\footnote{This topology will be defined in \cref{sec:topology-moduli-spaces}.} Then we can define a topology on $\wh{\mcR}_{g,k,h,n}(T)$ by taking the coarsest topology on $\wt{\mcR}_{g,nd+k,h,2n}(\wt{T})$ such that the glue functor is continuous. This makes the map $\glue:\wh{\mcR}_{g,k,h,n}(T)\lra\wt{\mcR}_{g,nd+k,h,2n}(\wt{T})$ an open quotient map.

\subsection{Maps to Moduli Spaces of Admissible Riemann Surfaces}

\label{sec:moduli-spaces-admissible-riemann-surfaces}

In addition to the moduli spaces of bordered Hurwitz covers we will also need moduli spaces of bordered Riemann surfaces. To this end we make the same definitions as \cite{mirzakhani_weil-petersson_2007}. Also note that these are the obvious definitions for the target surfaces from $\wh{\mcR}_{g,k,h,n}(T)$ and $\wt{\mcR}_{g,dk+n,h,2n}(\wt{T})$. We need to choose a function $F:\RR_{\geq 0}\lra\RR_{\geq 0}$ as in \cref{sec:orb-structure-mod-space-bordered-hurwitz-covers-definitions}.

\index{Moduli Space!Of Admissible Riemann Surfaces}

\begin{definition}
  The category of \emph{admissible Riemann surfaces with marked points on the boundary} $\wh{\mcR}_{h,n}$ is defined as follows. Objects are tuples $(X,\bp,\bz)$ where
  \begin{itemize}
    \item $X$ is an admissible Riemann surface with genus $h$ and $n$ possibly degenerate boundary components $\bp$,
    \item $\bz\in X^n$ is a tuple of marked points $z_i\in\Gamma_i(X)$ at the reference curves $\Gamma_i(X)$ defined with respect to the function $F$ on the uniformized hyperbolic surface $X$ close to the possible degenerate boundary component $\bp_i$
  \end{itemize}

  and morphisms between $(X,\bp,\bz)$ and $(X',\bp',\bz')$ are given by biholomorphisms $\phi:X\lra X'$ such that $\phi(\bp_i)=\bp_i'$ and $\phi(z_i)=z_i'$ for all $i=1,\ldots,n$.
  \label{def:moduli-space-admissible-surfaces}
\end{definition}

Similarly we define the corresponding moduli space of glued Riemann surfaces.

\index{Moduli Space!Of Closed Riemann Surfaces With Multicurve}

\begin{definition}
  The category of \emph{closed Riemann surfaces with a multicurve} $\wt{\mcR}_{h,2n}$ is defined as follows. Objects are tuples $(X,\bp,\bG)$ where
  \begin{itemize}
    \item $X$ is an closed Riemann surface with genus $h$ and $2n$ marked points $\bp$ and
    \item a multicurve $\bG$ on $X$ such that every curve in $\bG$ is simple and bounds either
      \begin{itemize}
        \item a disc with exactly two marked points or
        \item a disc with a node whose other smooth component is a sphere with exactly two marked points in addition to the node.
      \end{itemize}
      Furthermore we require that all the marked points $\bp$ are either contained in such a pair of pants or a completely nodal spherical component. Also we require that the labels of the branch points contained in one such disc from above differ by exactly one. 
  \end{itemize}
  The morphisms between $(X,\bp,\bG)$ and $(X',\bp',\bG')$ are defined to be biholomorphisms $\phi:X\lra X'$ such that $\phi(p_i)=p_i'$ and $\phi(\Gamma_i)=\Gamma_i'$, where the last equality means that the multicurves are freely homotopic.
  \label{def:moduli-space-riemann-surfaces-with-multicurve}
\end{definition}

In the same way as in \cref{sec:orbifold-structure-wtmcr} we obtain an orbifold structure $\wt{\mcM}_{h,n}$ on $\wt{\mcR}_{h,n}$ by incorporating the choice of the $\bG$ into the universal unfolding. This is the same orbifold structure as Mirzakhani defines in \cite{mirzakhani_weil-petersson_2007}. Also we can repeat the constructions in \cref{sec:def-glue-map} to obtain a functor

\begin{equation*}
  \sglue:\wh{\mcR}_{h,n}\lra\wt{\mcR}_{h,2n}
\end{equation*}

which also allows us to pull back the orbifold structure from $\wt{\mcR}^{\square}_{h,n}$ to $\wh{\mcR}_{h,n}^{\square}$, where the symbol $\square$ stands for the full subcategory of Riemann surface whose boundaries are non-degenerate and all elements in the multicurve are non-contractible, respectively. Again, the symbol $\circ$ refers to the full subcategory of completely smooth curves. The topology on $|\wt{\mcR}_{h,n}|$ is induced by the bijection $\iota:|\wt{\mcM}_{h,n}|\lra|\wt{\mcR}_{h,n}|$ and the topology on the moduli space $|\wh{\mcR}_{h,n}|$ will be defined in \cref{sec:topology-bordered-hurwitz-covers}.

It remains to say a few words about the obvious evaluation and forgetful functors.

\index{Forgetful Functor}
\index{Evaluation Functor}

\begin{definition}
  We define the obvious forgetful and evaluation functors
  \begin{align*}
    \wh{\fgt}: \wh{\mcR}_{g,k,h,n}(T) & \lra \wh{\mcR}_{g,k} \\
    \wh{\fgt}_{\Ob}(C,u,X,\bq,\bp,\bz) & \coloneqq (C,\bq,\bz) \\
    \wh{\fgt}_{\Mor}(\Phi,\phi) & \coloneqq \Phi
  \end{align*}
  and 
  \begin{align*}
    \wh{\ev}: \wh{\mcR}_{g,k,h,n}(T) & \lra \wh{\mcR}_{h,n} \\
    \wh{\ev}_{\Ob}(C,u,X,\bq,\bp,\bz) & \coloneqq (X,\bp,u(\bz)) \\
    \wh{\ev}_{\Mor}(\Phi,\phi) & \coloneqq \phi
  \end{align*}
  and
  \begin{align*}
    \wt{\ev}: \wt{\mcR}_{g,nd+k,h,2n}(T) & \lra \wt{\mcR}_{h,2n} \\
    \wt{\ev}_{\Ob}(C,u,X,\bq,\bp,\bG) & \coloneqq (X,\bp,\bG) \\
    \wt{\ev}_{\Mor}(\Phi,\phi) & \coloneqq \phi
  \end{align*}
  which clearly restrict to the corresponding functors on the $\mcM$-categories if we take care of using corresponding universal unfoldings in the definitions of the orbifold categories.
  \label{def:ev-fgt-functors-moduli-spaces-with-boundary}
\end{definition}

\begin{rmk}
  Notice that there is no $\wt{\fgt}$ as $\bG$ is a multicurve on $X$ bounding a pair of pants with two marked points or a node with a twice-punctured sphere attached. Although $u^{-1}(\bG)$ consists of a simple multicurve, the individual curves do not necessarily bound just two marked points each.

  In contrast on $\wh{\mcR}_{g,k,h,n}(T)$ we have the condition that $u(z_i)=u(z_j)$ if $\nu(i)=\nu(j)$ and thus $u(\bz)\coloneqq(u(z_j)\text{ for some }j\text{ s.t. }\nu(j)=i)_{i=1}^n$ is well-defined.
\end{rmk}

\index{Morphism Covering!Branched}

\begin{lem}
  The functors defined in \cref{def:ev-fgt-functors-moduli-spaces-with-boundary} are all homomorphisms on the corresponding orbifold categories $\wh{\mcM}_{g,k,h,n}^{\square}(T)$ and $\wt{\mcM}_{g,nd+k,h,2n}(\wt{T})$. Furthermore the following diagram commutes for the $\mcR$-categories as well as the corresponding $\mcM$-categories.
  \begin{equation}
    \xymatrix{
      \wh{\mcR}_{g,k,h,n}(T) \ar[d]_{\wh{\ev}} \ar[r]^-{\glue} & \wt{\mcR}_{g,nd+k,h,2n}(\wt{T}) \ar[d]_{\wt{\ev}} \ar[r] & \mcR_{g,nd+k,h,2n}(\wt{T}) \ar[d]^{\ev} \\
      \wh{\mcR}_{h,n} \ar[r]^{\sglue} & \wt{\mcR}_{h,2n} \ar[r] & \mcR_{h,2n}
      }
    \label{eq:diagram-glue-sglue}
  \end{equation}
  where the two right functors just forget the multicurve. Additionally we have that
  \begin{equation*}
    \wt{\ev}:\wt{\mcM}_{g,nd+k,h,2n}(\wt{T}) \lra \wt{\mcM}_{h,2n}
  \end{equation*}
  is a branched morphism covering. This implies that
  \begin{equation*}
    \wh{\ev}:\wh{\mcM}^{\square}_{g,nd+k,h,2n}(\wt{T}) \lra \wh{\mcM}^{\square}_{h,2n}
  \end{equation*}
  is a branched morphism covering.
\end{lem}

\begin{proof}
  In the same way as in the proof of \cref{prop:properties-functors-between-groupoids} we can easily see that all the functors are homomorphisms because we constructed the differentiable structure from the universal curve of the target surface.

  Also the fact that the diagram in \cref{eq:diagram-glue-sglue} commutes is directly clear from the definitions of the categories and the maps.

  Recall from \cref{sec:orbifold-structure-wtmcr} that the orbifold structure on $\wt{\mcM}_{g,nd+k,h,2n}(\wt{T})$ is defined via Hurwitz families $\Psi^{\lambda}:O^{\lambda}\lra\Ob\wt{\mcR}_{g,nd+k,h,2n}(\wt{T})$ that are constructed from the usual Hurwitz family from \cref{sec:constr-an-orbif} by equipping every fibre with the multicurve corresponding to the one of the central fibre that is part of the data in $\lambda$. This means that both horizontal functors in \cref{eq:diagram-glue-sglue} which forget the multicurves are covers on objects. Locally the functor $\wt{\ev}$ on objects thus looks like $\ev$ which has the required form of a branched morphism covering. The morphism covering condition for $\wt{\ev}:\wt{\mcM}_{g,nd+k,h,2n}^{\circ}(\wt{T}) \lra \wt{\mcM}^{\circ}_{h,2n}$ is also easily verified, again in the same way as in the proof of \cref{prop:properties-functors-between-groupoids}.

  Notice that for the right choices of parameter sets $\Lambda$ we can see that the category isomorphisms $\glue:\wh{\mcR}_{g,k,h,n}^{\square}(T)\lra\wt{\mcR}_{g,nd+k,h,2n}^{\square}(\wt{T})$ and $\sglue:\wh{\mcR}_{h,n}\lra\wt{\mcR}_{h,2n}$ give rise to isomorphisms of orbifold groupoids
  \begin{align*}
    \glue:\wh{\mcM}_{g,k,h,n}^{\square}(T) & \lra \wt{\mcM}_{g,nd+k,h,2n}^{\square}(\wt{T}), \\
    \sglue:\wh{\mcM}_{h,n} & \lra\wt{\mcM}_{h,2n}
  \end{align*}
  in the sense of compatible diffeomorphisms on object and morphism manifolds. But this implies that $\wh{ev}$ is a morphism covering, too.
\end{proof}

\begin{rmk}
  It is interesting to think about the degree of $\wt{\ev}$ and $\wt{\ev}$ in contrast to $\ev$. One can see that the horizontal maps forgetting multicurves in \cref{eq:diagram-glue-sglue} are coverings on orbit spaces of the same degree. Thus $\deg\wt{\ev}=\deg\ev$. However, the smooth Hurwitz covers in $\mcR_{g,nd+k,h,2n}(\wt{T})$ have a very particular structure, namely they are Hurwitz covers in $\mcR_{g,k,h,n}(T)$ with additionally marked $nd$ non-critical points. It is clear that this in fact a bijection. The corresponding Hurwitz numbers should hence differ by a factor coming from the possible additional enumerations of the trivial fibres which is $K=\prod_{i=1}^nK_i$ with $K_i=\lcm\{l_j\mid j=1,\ldots,k\text{ with }\nu(j)=i\}$. We will see this again in \cref{lem:degrees-evaluation-map-hurwitz}.
\end{rmk}

\chapter{SFT-Compactness and Topology}

\label{sec:SFT-compactness}

\section{Topology on Moduli Spaces}

\label{sec:topology-moduli-spaces}

\subsection{Closed Hurwitz Covers}

\label{sec:topology-moduli-spaces-closed-hurwitz-covers}

By defining an orbifold structure on $\mcM_{g,k,h,n}(T)$ and $\wh{\mcM}^{\circ}_{g,k,h,n}(T)$ we have equipped their orbit spaces with a topology. However, we still need to 
\begin{enumerate}[label=(\roman*), ref=(\roman*)]
  \item define a topology on $|\mcR_{g,k,h,n}(T)|$ and $|\wh{\mcR}_{g,k,h,n}(T)|$,
  \item prove that the orbit spaces $|\mcM_{g,k,h,n}(T)|$ and $|\mcR_{g,k,h,n}(T)|$ are compact and
  \item that the boundary length functions on $|\wh{\mcR}_{g,k,h,n}(T)|$ and $|\wt{\mcR}_{g,nd+k,h,2n}(\wt{T})$ are proper.
\end{enumerate}
   Before stating this in detail, let us recall the notion of convergence of surfaces in the Deligne--Mumford space.

\begin{definition}
  A sequence $[X_k,\bp_k]\in|\mcR_{h,n}|$ converges to $[X,\bp]\in|\mcR_{h,n}|$ if after removing finitely many elements in the sequence there exist maps $\phi_k:X_k\lra X$ together with curves $\gamma_k^i\subset X_k$ for $i=1,\ldots,m$ such that
  \begin{itemize}
    \item the maps $\phi_k$ map each $\gamma_k^i$ to a node in $X$ independent of $k$,
    \item the maps $\phi_k|_{X_k\setminus\bigcup_{i=1}^m\gamma_k^i}:X_k\setminus\bigcup_{i=1}^m\gamma_k^i\lra X\setminus \bigcup_{i=1}^m\phi_k(\gamma_k^i)$ are diffeomorphisms of nodal surfaces and we denote their inverses by $\psi_k$,
    \item the pull-back complex structures $\psi_k^*J_k$ converge in $\cin_{\text{loc}}$ to $J$ on $X\setminus\bigcup_{i=1}^m\phi_k(\gamma_k^i)$ and
    \item the marked points $\phi_k(\bp_k^i)$ converge to $\bp^i$ for $i=1,\ldots,n$.
  \end{itemize}
  \label{def:topology-dm}
\end{definition}

\begin{rmk}
  Let us make a few remarks for this section.
  \begin{enumerate}[label=(\roman*), ref=(\roman*)]
    \item Note that the convergence of the complex structures implies convergence of the corresponding hyperbolic structures to a hyperbolic metric on $(X,\bp)$. So if we denote by $g_k$ the unique complete hyperbolic metric with finite area induced by $J_k$ on $X_k$ we see that $\psi_k^*g_k\to g$ in $\cin_{\text{loc}}$ on $X\setminus \bigcup_{i=1}^m\phi_k(\gamma_k^i)$. Also we can choose the curves $\gamma_k^i$ to be closed simple geodesics of $g_k$.
    \item Also note that by the words ``diffeomorphism'' and ``$\cin_{\text{loc}}$-convergence'' we actually mean the corresponding version for nodal surfaces as we allow for already existing nodes. These notions are summarized in \cref{fig:convergence-deligne-mumford}.
    \item By Theorem~13.6 from \cite{robbin_construction_2006} we know that a sequence $[X_k,\bp_k]\in|\Ob\mcR_{h,n}|$ converges to $[X,\bp]\in|\Ob\mcR_{h,n}|$ in the sense as above if and only if the sequence converges in the topology induced by the orbifold structure $\mcM_{h,n}$ after possibly removing finitely many points in the sequence.
  \end{enumerate}
\end{rmk}

\begin{figure}[H]
    \centering
    \def\svgwidth{0.8\textwidth}
    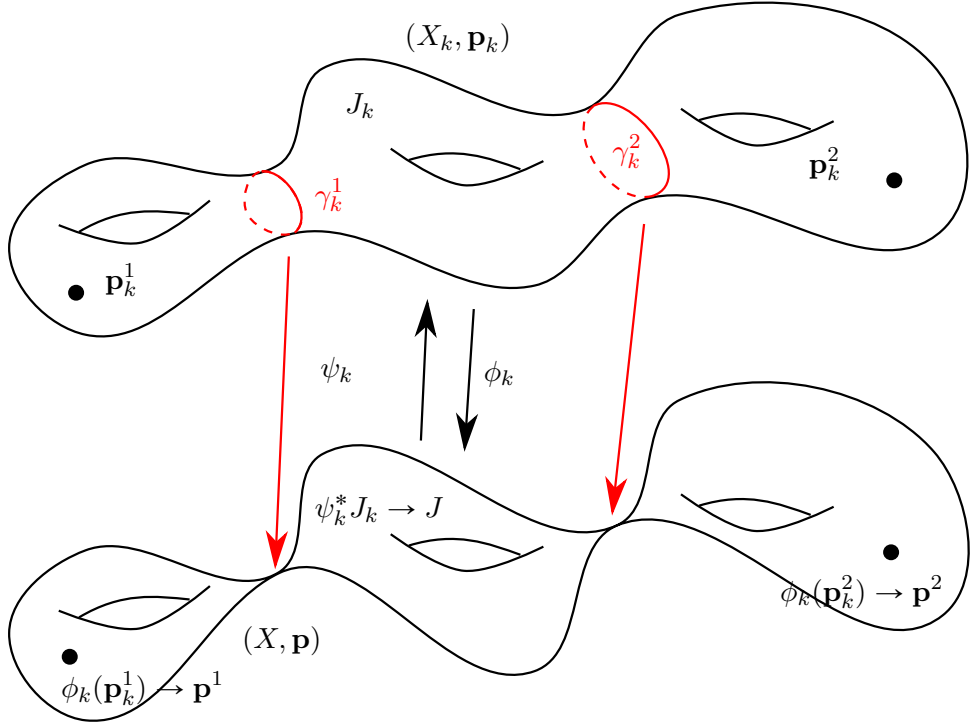
    \caption{This illustrates the objects needed in the definition of Deligne--Mumford convergence.}
    \label{fig:convergence-deligne-mumford}
\end{figure}

Next we define the topology on $|\mcR_{g,k,h,n}(T)|$.

\begin{definition}
  Define the topology on $|\mcR_{g,k,h,n}(T)|$ as the finest topology such that the surjective map $\iota:|\mcM_{g,k,h,n}(T)|\lra|\mcR_{g,k,h,n}(T)|$ is continuous. Thus the map $\iota$ is a quotient map and $|\mcR_{g,k,h,n}(T)|$ is equipped with the quotient topology.
  \label{def:topology-moduli-space-hurwitz-covers}
\end{definition}

We will need to relate this topology to a notion of convergence to broken holomorphic curves for which we have various compactness theorems available.

Note that we have the following basic statement.

\begin{lem}
  The map $\iota:|\mcM_{g,k,h,n}(T)|\lra|\mcR_{g,k,h,n}(T)|$ is proper.
  \label{lem:inclusion-proper}
\end{lem}

\begin{proof}
  Quotient maps are closed and thus it is enough to check that preimages of points are compact. But this is clear because for every Hurwitz cover there are only finitely many choices for the discs in the definition of the Hurwitz deformations at nodes.
\end{proof}

Next we show how to interpret this topology on $|\mcR_{g,k,h,n}(T)|$.

\begin{prop}
  A sequence $[C_k,u_k,X_k,\bq_k,\bp_k]\in |\mcR_{g,k,h,n}(T)|$ converges to
  \begin{equation*}
    [C,u,X,\bq,\bp]\in|\mcR_{g,k,h,n}(T)|
  \end{equation*}
  if and only if after removing finitely many members of the sequence
  \begin{itemize}
    \item there exist tuples of simple curves $\Gamma_k\subset C_k$ and $\Theta_k\subset X_k$ such that \label{item:collaping-curves}
      \begin{equation*}
        u_k^{-1}(\Theta_k^i)=\bigsqcup_{j\in I^i}\Gamma_k^j\text{ and }u_k(\Gamma_k^i)\in\Theta_k\quad\forall i,
      \end{equation*}
    \item there exist maps $\phi_k:C_k\lra C$ and $\rho_k:X_k\lra X$ which map $\Gamma_k^i$ and $\Theta_k^i$ for every $i$ to one fixed node in $C$ and $X$ depending on $i$, respectively,
    \item such that $\phi_k$ and $\rho_k$ are diffeomorphisms of nodal surfaces outside the curves and their images,
    \item the marked points $\phi_k(\bq_k^j)$ and $\rho_k^i(\bp_k^i)$ converge to $\bp$ and $\bq$, respectively,
    \item the complex structures $(\phi_k)_*j_k$ and $(\rho_k)_*J_k$ converge to $j$ and $J$, respectively, in $\cin_{\text{loc}}$ outside the images of the curves above and
    \item the maps $\rho_k\circ u_k\circ \phi_k^{-1}:C\setminus\bigcup_i\phi_k(\Gamma_k^i)\lra X\setminus\bigcup_i\rho_k(\Theta_k^i)$ converge in $\cin_{\text{loc}}$ to $u$ away from the images of the curves and uniformly on $C$.
  \end{itemize}
  \label{prop:equivalent-formulation-topology}
\end{prop}

\begin{rmk}
  We reformulate the topology on $|\mcR_{g,k,h,n}(T)|$ in this way because from SFT-compactness we will extract a subsequence having these properties and thus giving us sequential compactness. Note that this description is merely a consequence of the fact that we defined the orbifold structures from the data in the universal unfoldings of the target and source surface and thus can use the reformulation of the topology from \cite{robbin_construction_2006}.
\end{rmk}

\begin{proof}
  First assume that we are given a sequence $[C_k,u_k,X_k,\bq_k,\bp_k]\in |\mcR_{g,k,h,n}(T)|$ converging to $[C,u,X,\bq,\bp]\in|\mcR_{g,k,h,n}(T)|$ in the topology defined in \cref{def:topology-moduli-space-hurwitz-covers}. As $\Ob\mcM_{g,k,h,n}(T)\lra|\mcM_{g,k,h,n}(T)|\lra|\mcR_{g,k,h,n}(T)|$ is a quotient map there exists a sequence of representatives of $[C_k,u_k,X_k,\bq_k,\bp_k]$ in a neighborhood of $(C,u,X,\bq,\bp)\in\Ob\mcM_{g,k,h,n}(T)$ after removing finitely many members of the sequence. As the object set was built from the coordinate charts on the target universal curve we can apply Theorem~13.6 from \cite{robbin_construction_2006} to conclude that the sequence of target surfaces $(X_k,\bp_k)$ converges to $(X,\bp)$ in the sense of \cref{def:topology-dm}. Recalling that we built the source surfaces in the Hurwitz deformation as a subset of the universal unfolding of the source surface we get the same type of convergence on the source surface. 

  It remains to verify the convergence of the maps and to fix a choice of ``collapsing curves'' that satisfies \cref{item:collaping-curves}. We will choose the curves in the complex gluing description as it is not easy to show the uniform convergence of the maps over the nodes in the hyperbolic description. So choose a disc structure \index{Disc Structure} around the collapsed nodes in the limit Hurwitz cover $(C,u,X,\bq,\bp)\in\Ob\mcM_{g,k,h,n}(T)$ and associate gluing parameters for all the nodes in the sequence $(C_k,u_k,X_k,\bq_k,\bp_k)$ such that $\Psi(t_k,z_k)=(C_k,u_k,X_k,\bq_k,\bp_k)$ for some choices of discs in the definition of $\Psi$ as the sequence might ``jump'' between the branches. See \cref{fig:disc-structure-topology} for an illustration and recall that $\Psi:Q\times\DD^N\lra\Ob\mcR_{g,k,h,n}(T)$ was defined for an open subset $Q$ of Teichmüller space and one disc for every one of the $N$ nodes in the central target $X$. The gluing parameters for the $M$ source nodes of one node in $X$ were then specified by the choice of one disc $b:\DD\lra \DD^M$ for every node in $X$. See \cref{sec:local-parametrizations} for more details. We will denote the constructed Hurwitz covers by $\Psi(t_k,z_k)=(u_{b(z_k)}:C_{u^*t_k,b(z)}\lra X_{t_k,a(z)}$-

  Note that we need to investigate the convergence of the maps in $\cin_{\text{loc}}$ away from the nodes and in $C^0$ at the nodes. The first part is trivial as $\Psi$ was built such that $u_k$ is \emph{not} modified outside the disc structure and thus converges there in $\cin_{\text{loc}}$ trivially. Now we need to look more closely at the nodes but it is enough to do this separately for every new limit node $p\in X$. Its preimages will be enumerated by $j=1,\ldots,m$ with degrees $l_j$ which we should not confuse with the parameters specifying the degrees of $u$ at $q_j$. The gluing parameters will be denoted by $a_k$ for the one of the target surface in the $k$-th element in the sequence. Correspondingly the gluing parameters for the preimages of the node $p$ will be denoted by $b^j_k$. Recall from \cref{fig:disc-structure-topology} that we used the charts from the disc structure to glue in modified discs specified by the gluing parameters.

  \begin{figure}[H]
    \centering
    \def\svgwidth{\textwidth}
    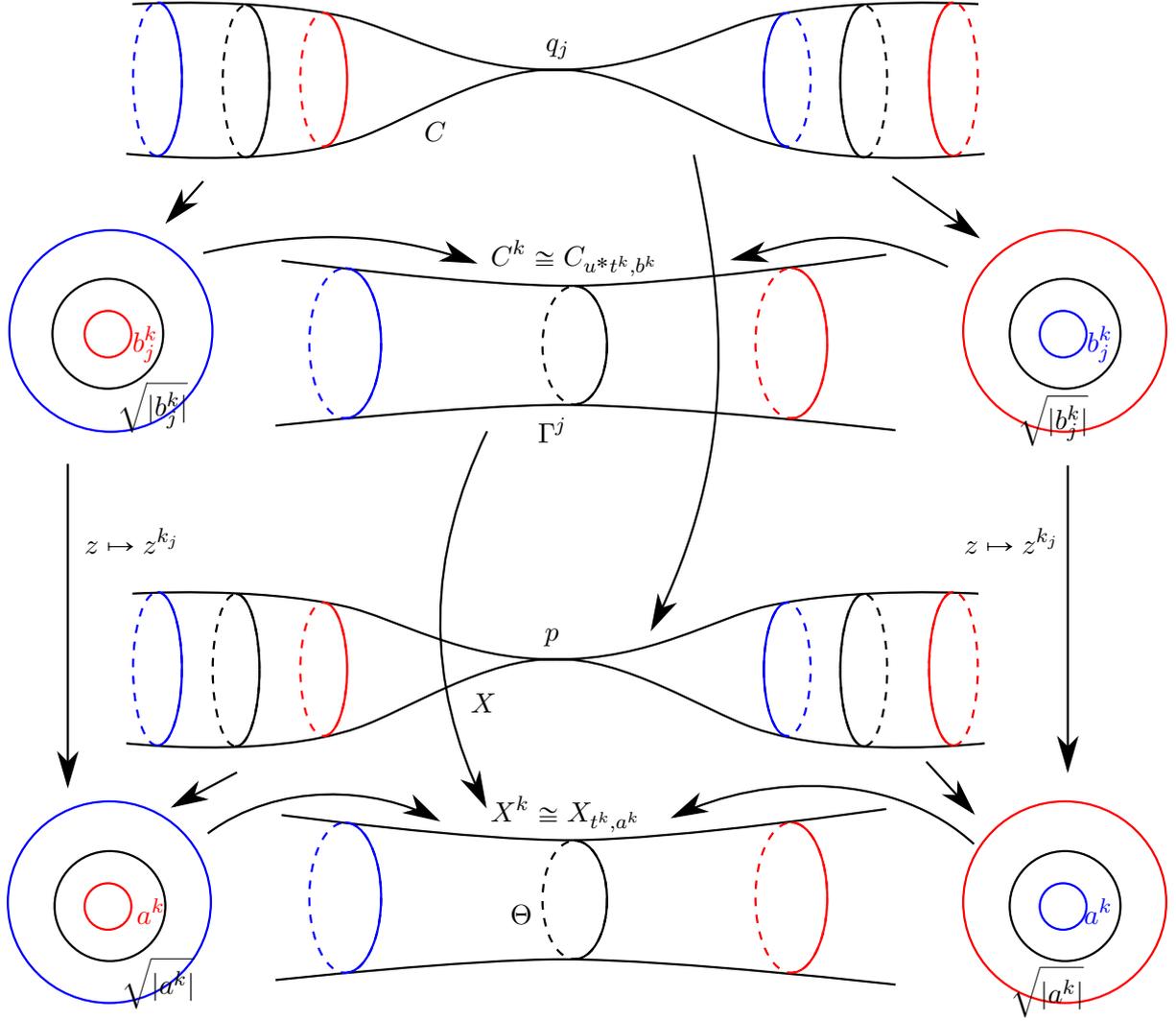
    \caption{Recall the gluing procedure from \cref{sec:complex-gluing}. We define the curves $\Gamma_k^j$ and $\Theta_k$ as those circles with radius the square root of the modulus. This way we can define the diffeomorphisms outside the curves onto the nodal surfaces in a straight-forward way.}
    \label{fig:disc-structure-topology}
  \end{figure}

  We will now define the curves that we remove from the surfaces together with diffeomorphisms to the target nodal surface with nodes removed. For every fixed index $k$ in the sequence we define a curve $\Gamma_k^j$ as the image of the circle of radius $\sqrt{|b_k^j|}$ on $C_{u^*t_k,b_k}$ and the curve $\Theta_k$ as the image of the circle of radius $\sqrt{|a_k|}$. Since
  \begin{equation*}
    \sqrt{|b_k^j|}^{l_j}=\sqrt{|a_k|}\text{ and }\left|\frac{x}{\sqrt{|x|}}\right|=\sqrt{|x|}
  \end{equation*}
  it does not matter on which side of the node we define the curve and $\Gamma_k^j$ gets indeed mapped to $\Theta_k$ under $u_{b_k}$. Now choose a family of diffeomorphisms $g_s:[\sqrt{s},1]\lra[0,1]$ for $s\in[0,1)$ such that
  \begin{enumerate}[label=(\roman*), ref=(\roman*)]
    \item $g_s(x)=x-\sqrt{s}$ in a small neighborhood of $\sqrt{s}$ and $g_s(x)=x$ in a small neighborhood of $1$ and the sizes of these neighborhoods are independent of $s$  for small $s$ and
    \item for every $S\in(0,1]$, $g_s|_{[\sqrt{S},1]}$ converges to the identity function in $\cin([\sqrt{S},1])$ as $s\lra 0$.
  \end{enumerate}
  This is illustrated in \cref{fig:function-choices}.

  \begin{figure}[H]
    \centering
    \def\svgwidth{0.5\textwidth}
    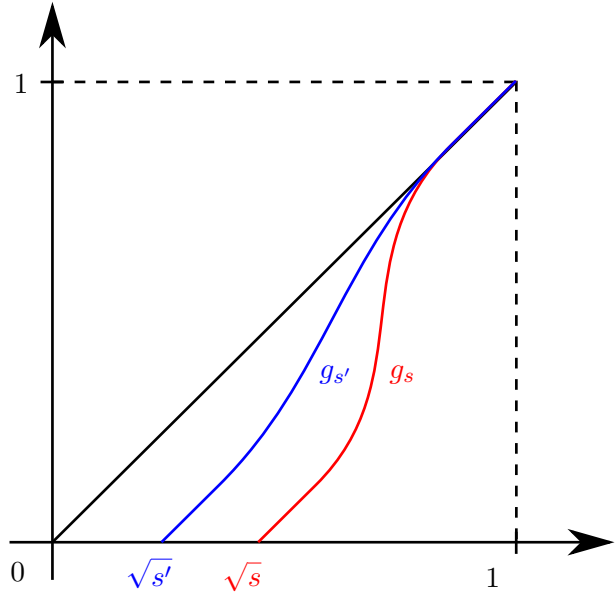
    \caption{This figure illustrates the kind of radial scaling function that we use for the diffeomorphisms outside the curves on the glued-in cylinders. Note that $s'< s$ and for every $s$ the smooth function has slope one next to the interval boundaries. Also they clearly converge in $\cin$ on compact sets in $(0,1)$ to the identity function.}
    \label{fig:function-choices}
  \end{figure}
  
  With this choice we can define maps on the glued-in cylinders by

  \begin{align*}
    \{z\in\CC\mid\sqrt{|b_k^j|}<|z|\leq 1\} & \lra \DD\setminus\{0\} \\
    re^{\ii\theta} & \longmapsto g_{|b_k^j|}(r)e^{\ii\theta}
  \end{align*}
  and correspondingly for $a_k$. Note that we define these maps on the \emph{outer} parts of the annuli and thus by defining them on both coordinate charts we obtain maps defined everywhere except at the curves. Furthermore these maps are given by the identity near radius $1$ and thus extend to well-defined diffeomorphisms
  \begin{align*}
    \phi_k: C_k\setminus\Gamma_k & \lra C\setminus\{\text{collapsed nodes}\} \\
    \rho_k: X_k\setminus\Theta_k & \lra X\setminus\{\text{collapsed nodes}\}
  \end{align*}
  by using the corresponding parameters $a_k$ and $b^j_k$ everywhere. Note that the different choices of the discs $b$ in the definition of $\Psi$ for the gluing parameters result in different twists for the complex structure on $C_k$ but the maps are the same. So we can calculate in a neighborhood of the $j$-th preimage of $p$
  \begin{align*}
    \rho_k\circ u_k\circ\phi_k^{-1}(re^{\ii\theta}) & = \rho_k\circ u_k(g_{|b_k^j|}^{-1}(r)e^{\ii\theta}) \\
    & = \rho_k((g_{|b_k^j|}^{-1}(r))^{l_j}e^{\ii l_j\theta}) = g_{|a_k|}((g_{|b_k^j|}^{-1}(r))^{l_j})e^{\ii l_j\theta}
  \end{align*}
  where near $r=0$ the radial component is given by
  \begin{equation*}
    g_{|a_k|}((g_{|b_k^j|}^{-1}(r))^{l_j})=(\sqrt{|b_k^j|}+r)^{l_j}-\sqrt{|a_k|}=rP_{|b_k^j|}(r)
  \end{equation*}
  where $P_{|b_k^j|}(r)$ is a polynomial in $r$ whose coefficients depend polynomially on $\sqrt{|b_k^j|}$. Thus the map $\rho_k\circ u_k\circ\phi_k^{-1}$ converges uniformly to $z\mapsto z^{l_j}$ on the discs for $|a_k|\to 0$. This shows one direction of the proposition.

  For the other direction consider a sequence $[C_k,u_k,X_k,\bq_k,\bp_k]\in |\mcR_{g,k,h,n}(T)|$ converging to $[C,u,X,\bq,\bp]\in|\mcR_{g,k,h,n}(T)|$ in the sense of the proposition. Then we can use Theorem~13.16 from \cite{robbin_construction_2006} to see that after removing finitely many elements in the sequence the surfaces $[X_k,\bp_k]$ are biholomorphic to fibres in the universal unfolding of $(X,\bp)$. So choose discs on $X$ as in \cref{sec:complex-gluing} such that $(X_k,\bp_k)$ is biholomorphic to $X_{t_k,a_k}$, where $a_k$ is a tuple of gluing parameters for the additional nodes of $X$ and $a_k\lra 0$ as $k\to\infty$. We can use $u:C\lra X$ to pull back these discs to $C$ and by choosing the initial discs sufficiently small we can assume that the system forms a disc structure. Denote by $\rho_k$ and $\phi_k$ the diffeomorphisms outside the collapsing curves. As $\rho_k\circ u_k\circ \phi_k^{-1}$ converges uniformly to $u$ on $C$ they induce the same map on the fundamental group of the connected components of $C$ with the marked points removed. Thus there exists a family of diffeomorphisms $\psi_k:C\setminus\{\text{critical points}\}\lra C\setminus\{\text{critical points}\}$ such that $\rho_k\circ u_k\circ\phi_k^{-1}\circ\psi_k=u$. As these maps $\psi_k$ are holomorphic since locally all the ingredient maps are invertible and holomorphic, they extend to a biholomorphism $C\lra C$. In particular it follows that $u$ has the same degree at a node as $u_k$ has on the corresponding curve. Thus $(C_k,u_k,X_k,\bq_k,\bp_k)$ is isomorphic to a Hurwitz cover in some family $\Psi^{\lambda}$ centered around $(C,u,X,\bq,\bp)$ and built from this disc structure and appropriate gluing parameters $b^j_k$ and $a_k$ for source and target surface, respectively, again due to Theorem~13.16 in \cite{robbin_construction_2006}.\footnote{At this point we cheat a little bit as it is possible that $(C,u,X,\bq,\bp)$ was not contained in the set $\Lambda$ which we used for defining $\mcM_{g,k,h,n}(T)$. However, one can just choose a central Hurwitz cover in that set and replace the gluing parameters zero by some appropriate values $a$ and $b$, respectively.} As the sequence $a_k$ goes to zero we have that $b_k$ goes to zero as well and therefore the sequence converges in $O^{\lambda}\subset\Ob\mcM_{g,k,h,n}(T)$. But since the map $\Ob\mcM_{g,k,h,n}(T)\lra|\mcR_{g,k,h,n}(T)|$ is a quotient map this means that the sequence converges in this induced topology as well.
\end{proof}

\subsection{Bordered Hurwitz Covers}

\label{sec:topology-bordered-hurwitz-covers}

Next we will define the topology on the moduli spaces of bordered Hurwitz covers and their ``closed'' version, $|\wh{\mcR}_{g,k,h,n}(T)|$ and $|\wt{\mcR}_{g,nd+k,h,2n}(\wt{T})|$, respectively. Note that we have already defined topologies on the corresponding orbifolds $\wh{\mcM}^{\circ}_{g,k,h,n}(T)$ and $\wt{\mcM}_{g,nd+k,h,2n}(\wt{T})$. First, we will define a topology on $|\wt{\mcR}_{g,nd+k,h,2n}(\wt{T})|$ and then afterwards on $|\wh{\mcR}_{g,k,h,n}(T)|$ such that the map $\glue:\wh{\mcR}_{g,k,h,n}(T)\lra\wt{\mcR}_{g,nd+k,h,2n}(\wt{T})$ is continuous.

\begin{definition}
In the same way as in \cref{def:topology-moduli-space-hurwitz-covers} the topology on $|\wt{\mcR}_{g,nd+k,h,2n}(\wt{T})|$ will be the finest one such that the map
\begin{equation*}
  \iota:|\wt{\mcM}_{g,nd+k,h,2n}(\wt{T})|\lra|\wt{\mcR}_{g,nd+k,h,2n}(\wt{T})|
\end{equation*}
is continuous.
\end{definition}

\begin{rmk}
  As for the closed case this makes the map $\iota$ an open quotient map. Since $|\wt{\mcM}_{g,nd+k,h,2n}(\wt{T})|$ is second-countable the space $|\wt{\mcR}_{g,nd+k,h,2n}(\wt{T})|$ is second-countable, too, and therefore compactness is equivalent to sequential compactness. Furthermore $|\wt{\mcM}_{g,nd+k,h,2n}(\wt{T})|$ is locally compact and thus $|\wt{\mcR}_{g,nd+k,h,2n}(\wt{T})|$ is locally compact, too. Again we use that the map $\iota$ is open for this statement. Notice that this means that in order to prove that a map $|\wt{\mcR}_{g,nd+k,h,2n}(\wt{T})|\lra Z$ is proper it is enough to show that it is closed and preimages of points are compact.
\end{rmk}

Next we define a topology on $|\wh{\mcR}_{g,k,h,n}(T)|$. Note that we can not just require the map $\glue:|\wh{\mcR}_{g,k,h,n}(T)|\lra|\wt{\mcR}_{g,nd+k,h,2n}(\wt{T})|$ to be continuous and then take the coarsest topology as this would make the open sets contain the whole circle of marked points $z_i$ at a node $q_i$. The situation is similar to polar coordinates $\RR_{\geq 0}\times S^1\lra\CC$ where it is not immediately clear how to abstractly define the topology on $\RR_{\geq 0}\times S^1$ without knowing anything about the topology on $\RR^2$ or the cylinder. Notice that this toy map is not open as an open strip $\RR_{\geq 0}\times(-\epsilon,\epsilon)$ touching the boundary is mapped to a sector in $\CC$ including the origin which is not open.

Recall the maps from the following diagram.
\begin{equation}
  \xymatrix{
    \Ob\wh{\mcR}_{g,k,h,n}(T) \ar[rr] \ar[dr]^{\glue} & & |\wh{\mcR}_{g,k,h,n}(T)| \ar[dd]^{\glue} \\
    & \Ob\wt{\mcR}_{g,nd+k,h,2n}(\wt{T}) \ar[rd] & \\
    \Ob\wt{\mcM}_{g,nd+k,h,2n}(\wt{T}) \ar[r] \ar[ur]^{\iota} & |\wt{\mcM}_{g,nd+k,h,2n}(\wt{T})| \ar[r]^{\iota} & |\wt{\mcR}_{g,nd+k,h,2n}(\wt{T})|    
    }
  \label{eq:diag-maps-topology-bordered-hurwitz-covers}
\end{equation}
Observe that the lower horizontal line consists of spaces already equipped with topologies and those maps are open continuous quotient maps. Now let an element
\begin{equation*}
  [C,u,X,\bq,\bp,\bz]\in|\wh{\mcR}_{g,k,h,n}(T)|
\end{equation*}
be given. We will index a subbasis for the topology on this space by choosing central points of the neighborhoods as well as open sets in $\Ob\wt{\mcM}_{g,nd+k,h,2n}(\wt{T})$ and $(-\epsilon,\epsilon)\in\RR$.

Choose a sufficiently small open neighborhood $U\subset\Ob\wt{\mcM}_{g,nd+k,h,2n}(\wt{T})$ such that $\glue([C,u,X,\bq,\bp,\bz])\in \iota(|U|)$ and an $\epsilon>0$. By sufficiently small we mean that it is contained in one chart $O^{\lambda}$. Note that the data from $\lambda$ contains a set of simple essential free homotopy classes of $C$ decomposing the surface into pairs of pants. This allows us to fix a point on $\Gamma_j(C)$ by intersecting the unique geodesic going up the cusp which is perpendicular to one chosen boundary of the corresponding pair of pants (or going up this cusp if $C$ happens to have a node there) in the hyperbolic uniformization of $C$. Call this point $w_j\in\Gamma_j(C)$. Notice that we can do this for any $[C',u',X',\bq',\bp',\bz']\in U$ and we will denote the corresponding point by $w_j'$. Now we define
\begin{align*}
  \mcU_{[C,u,X,\bq,\bp,\bz]}(U,\epsilon)\coloneqq\{[C',u',X',\bq', & \bp',\bz']\mid \glue(C',u',X',\bq',\bp',\bz')\in\iota(U) \\
  & \text{and }|d_{\text{hyp}}(z_j,w_j)-d_{\text{hyp}}(z'_j,w'_j)|<\epsilon\},
\end{align*}
where by $d_{\text{hyp}}(z_j,w_j)$ we mean the hyperbolic distance between the two points \emph{along} the curve $\Gamma_j(C)$.

\begin{rmk}
  The point $w_j\in\Gamma_j(C)$ is used to make sense of the idea that the marked points $z_j$ can vary slightly on their reference curve in an open neighborhood for which we need a reference point to define parametrizations of the curves.
\end{rmk}

\begin{lem}
  The sets $\mcU_{[C,u,X,\bq,\bp,\bz]}(U,\epsilon)$ define a topology on $|\wh{\mcR}_{g,k,h,n}(T)|$ such that the space is locally compact and second-countable. Furthermore the map $\glue$ is continuous, closed and proper.
  \label{lem:prop-top-moduli-space-bordered-hurwitz-covers}
\end{lem}

\begin{proof}
  The sets $\mcU_{[C,u,X,\bq,\bp,\bz]}(U,\epsilon)$ clearly generate a topology as a subbasis. The space is second-countable as $\Ob\wt{\mcM}_{g,nd+k,h,2n}(\wt{T})$ and $\RR$ are second-countable and so we can find a countable basis by choosing appropriate basis open sets. The local compactness also follows immediately from the local compactness of $\Ob\wt{\mcM}_{g,nd+k,h,2n}(\wt{T})$ and $\RR$.

  Continuity of $\glue:|\wh{\mcR}_{g,k,h,n}(T)|\lra|\wt{\mcR}_{g,nd+k,h,2n}(\wt{T})|$ follows by looking at the commuting diagram in \cref{eq:diag-maps-topology-bordered-hurwitz-covers}: Suppose $\mcU\subset|\wt{\mcR}_{g,nd+k,h,2n}(\wt{T})|$ is an open set. Then for any point $[C,u,X,\bq,\bp,\bz]\in \glue^{-1}(\mcU)$ we can find an open set $[C,u,X,\bq,\bp,\bz]\in\mcV\subset\glue^{-1}(\mcU)$ by considering a sufficiently small subset in $\Ob\wt{\mcM}_{g,nd+k,h,2n}(\wt{T})$ containing $\glue([C,u,X,\bq,\bp,\bz])$ and taking the corresponding set in $|\wh{\mcR}_{g,k,h,n}(T)|$ as in the construction of the subbasis.

  Next we show that $\glue:|\wh{\mcR}_{g,k,h,n}(T)|\lra|\wt{\mcR}_{g,nd+k,h,2n}(\wt{T})|$ is a closed map. First notice that both spaces are second-countable and thus we can describe closed subsets as subsets such that every accumulation point is contained in the set and a point is an accumulation point if and only if there is a sequence converging to this point. Now consider a closed subset $\mcV\subset|\wh{\mcR}_{g,k,h,n}(T)|$. We need to show that $\glue(\mcV)\subset|\wh{\mcR}_{g,k,h,n}(T)|$ is closed, too. So consider a sequence $\{[\wt{C}_k,\wt{u}_k,\wt{X}_k,\wt{\bq}_k,\wt{\bp}_k,\wt{\Gamma}_k]\}_{k\in\NN}$ with
  \begin{equation*}
    [\wt{C}_k,\wt{u}_k,\wt{X}_k,\wt{\bq}_k,\wt{\bp}_k,\wt{\Gamma}_k]=\glue([C_k,u_k,X_k,\bq_k,\bp_k,\bz_k])\qquad\forall k\in\NN,
  \end{equation*}
  such that
  \begin{equation*}
    \lim_{k\to\infty}[\wt{C}_k,\wt{u}_k,\wt{X}_k,\wt{\bq}_k,\wt{\bp}_k,\wt{\Gamma}_k]=[\wt{C},\wt{u},\wt{X},\wt{\bq},\wt{\bp},\wt{\Gamma}]\in|\wt{\mcR}_{g,nd+k,h,2n}(\wt{T})|.
  \end{equation*}
  We need to show that there exists $[C,u,X,\bq,\bp,\bz]\in|\wh{\mcR}_{g,k,h,n}(T)|$ such that 
  \begin{equation*}
    \glue([C,u,X,\bq,\bp,\bz])=[\wt{C},\wt{u},\wt{X},\wt{\bq},\wt{\bp},\wt{\Gamma}].
  \end{equation*}
  For this purpose choose a sufficiently small open neighborhood $V$ of $(\wt{C},\wt{u},\wt{X},\wt{\bq},\wt{\bp},\wt{\Gamma})$ (or rather an equivalent Hurwitz cover) in $\Ob\wt{\mcM}_{g,nd+k,h,2n}(\wt{T})$ such that except for finitely many members all elements in the sequence $[\wt{C}_k,\wt{u}_k,\wt{X}_k,\wt{\bq}_k,\wt{\bp}_k,\wt{\Gamma}_k]$ have representatives $(\wt{C}_k,\wt{u}_k,\wt{X}_k,\wt{\bq}_k,\wt{\bp}_k,\wt{\Gamma}_k)\in \iota(V)$ in the groupoid $\wt{\mcM}_{g,nd+k,h,2n}(\wt{T})$. Now the map glue is injective on objects except regarding the marked points on reference curves close to punctures. In any case we can cut the Hurwitz cover $(\wt{C},\wt{u},\wt{X},\wt{\bq},\wt{\bp},\wt{\Gamma})$ along the hyperbolic geodesics in the free homotopy classes of $\wt{\Gamma}$ or at the corresponding node, respectively. We can do the same thing for all the elements in the sequence to obtain elements $(C_k,u_k,X_k,\bq_k,\bp_k,\bz_k)\in \Ob\wh{\mcR}_{g,k,h,n}(T)$ with $\glue(C_k,u_k,X_k,\bq_k,\bp_k,\bz_k)=C_k,u_k,X_k,\bq_k,\bp_k,\bz_k$. This is the same construction as in the proof of \cref{lem:glue-functor-surjectivity} and this procedure does not yet fix the points $(z_k)_j$ at punctures $q_j$. However, after choosing an arbitrary sequence of reference points $(z_k)_j$'s at these punctures such that $[C_k,u_k,X_k,\bq_k,\bp_k,\bz_k]\in\mcV$ we can pick a convergent subsequence because the corresponding $S^1$'s are compact and we can parametrize the $\Gamma_j(C_k)$ by choosing a geodesic as in the definition of the topology on $\wh{\mcR}_{g,k,h,n}(T)$. Now this subsequence obviously converges in $|\wh{\mcR}_{g,k,h,n}(T)|$ and is contained in $\mcV$ and therefore its limit is contained in $\mcV$. By continuity of $\glue$ this implies that the accumulation point of $\glue(\mcV)$ is contained in this image set and it is thus closed.

  In order to prove that $\glue$ is proper we thus only need to show preimages of points are compact. But preimages of points are homeomorphic to products of $S^1$'s corresponding to the nodal $\Gamma_j(C)$ and are thus compact.  
\end{proof}

\begin{rmk}
  Indeed the map $\glue:|\wh{\mcR}_{g,k,h,n}(T)|\lra|\wt{\mcR}_{g,nd+k,h,2n}(\wt{T})|$ is not open. To see this notice that a neighborhood of $[\wt{C},\wt{u},\wt{X},\wt{\bq},\wt{\bp},\wt{\Gamma}]$ with a node $\bp_j$ contains all possible marked points $u(z_j)$ as they correspond to possible endpoints of the unique hyperbolic geodesic going up the cusp $\wt{\bp}_{2j-1}$ and perpendicular to the geodesic corresponding to the collapsed node. But an open set in $|\wh{\mcR}_{g,k,h,n}(T)|$ might only contain a proper subset of this ``circle'' of possible marked points on the reference curve.
\end{rmk}

We have now defined two topologies on $|\wh{\mcM}^{\circ}_{g,k,h,n}(T)|$, one by pulling back the orbifold structure from $|\wt{\mcM}_{g,nd+k,h,2n}(\wt{T})|$ and the second one as a subspace $|\wh{\mcM}^{\circ}_{g,k,h,n}(T)|\subset |\wh{\mcR}_{g,k,h,n}(T)|$. These two topologies agree because the map $\glue:|\wh{\mcM}^{\circ}_{g,k,h,n}(T)|\lra|\wt{\mcM}^{\circ}_{g,nd+k,h,2n}(\wt{T})|$ is a homeomorphism by definition and for smooth Hurwitz covers the functor $\glue:\wh{\mcR}^{\circ}_{g,k,h,n}(T)\lra\wt{\mcR}^{\circ}_{g,nd+k,h,2n}(\wt{T})$ is injective on objects. This implies that the sets  $\mcU_{[C,u,X,\bq,\bp,\bz]}(U,\epsilon)$ are open in both topologies. Furthermore all forgetful and evaluation maps to corresponding Deligne--Mumford spaces are continuous by the same arguments as in \cref{sec:main-results}.

\section{Compactness of \texorpdfstring{$\boldsymbol{|\mcR_{g,k,h,n}(T)|}$}{|R\_gkhn(T)|}}

Now it remains to show that the obtained orbit spaces of our various orbifold categories are indeed compact. We will show this first for $|\mcR_{g,k,h,n}(T)|$ which implies compactness for $|\mcM_{g,k,h,n}(T)|$ by \cref{lem:inclusion-proper}. It is enough to show that they are sequentially compact as these spaces are first countable. The strategy is as follows. First we look at the sequence of complex structures of the target surface and pick a converging subsequence. Then we use SFT-compactness at collapsing curves in the target to extract a subsequence of Hurwitz covers converging to a holomorphic building in the SFT-sense. We then reformulate this in terms of Hurwitz covers and exclude various problematic cases that can happen in general. We are then left with a subsequence converging to a Hurwitz cover already considered in our moduli spaces thus proving its compactness.

So let us be given a sequence $[C_k,u_k,X_k,\bq_k,\bp_k]\in|\mcR_{g,k,h,n}(T)|$ with the usual combinatorial data denoted by $g,k,h,n\in\NN$ and $T=(T_1,\ldots,T_n,d,\nu,\{l_j\}_{1}^k)$. We will usually denote subsequences again by using the same index.

\subsection{Picking a Converging Subsequence of Complex Structures on the Target Surface}

First, note that the sequence $(X_k,\bp_k)$ is a sequence of stable nodal Riemann surfaces $X_k$ together with $n$ marked points $\bp_k$. Since the Deligne--Mumford space $|\mcR_{h,n}|$ is compact we can pass to a subsequence $(X_k,\bp_k)$ which converges to a nodal surface $(X,\bp)\in\obj\mcR_{h,n}$ in the sense of \cref{def:topology-dm}.

Recall from \cref{sec:collar-neighborhoods} that every closed simple geodesic and every cusp has a standard neighborhood and all these neighborhoods are pairwise disjoint if the geodesics and cusps do not meet. This implies that by passing to a subsequence we can assume that all the curves $\gamma_k^i$ and all cusps stay at least some fixed minimum hyperbolic distance away from each other. So wherever the hyperbolic metric $\psi_k^*g_k$ does not degenerate we can investigate the behavior of the maps $u_k$, the geodesics of $\psi_k^*g_k$, its nodes and its cusps individually. Rephrasing this means that on the target surface $X$ the nodes as well as the cusps cannot meet and the only degeneration possible is the one coming from collapsing hyperbolic geodesics. Away from these curves the hyperbolic structure converge, see e.g.\ \cite{hummel_gromovs_1997}. You can see this in \cref{fig:convergence-DM-nodes-collapsing-geodesics}.

\begin{figure}[H]
    \centering
    \def\svgwidth{0.6\textwidth}
    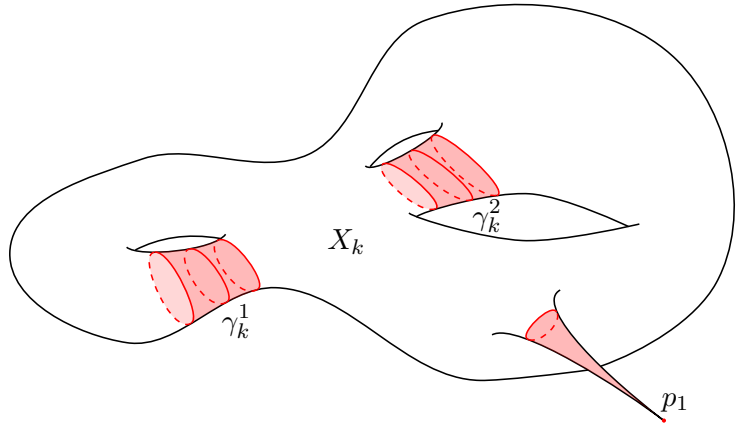
    \caption{The surface $X_k$ has a two collapsing geodesics $\gamma_k^1$ and $\gamma_k^2$ as well as a cusp $p_1$. Each of them has their collar or cusp neighborhood, respectively. By passing to a subsequence we assume that their lengths are uniformly bounded from above such that their collar and cusp neighborhoods, respectively, are pairwise disjoint. This means that the images of these neighborhoods around the cusps and nodes in $X$ are also disjoint for all $k\in\NN$.}
    \label{fig:convergence-DM-nodes-collapsing-geodesics}
\end{figure}

In total we have the following \cref{prop:target-j-subsequence}.

\begin{prop}
  Given a sequence $[C_k,u_k,X_k,\bq_k,\bp_k]\in|\mcR_{g,k,h,n}(T)|$ there is a subsequence denoted again by $[C_k,u_k,X_k,\bq_k,\bp_k]$ and a nodal surface $(X,\bp)$ together with a finite set of closed simple curves $\gamma_k^i\subset X_k$ for $i=1,\ldots,m$ and maps $\phi_k:X_k\lra X$ such that
  \begin{itemize}
    \item the maps $\phi_k$ map each $\gamma_k^i$ to a node in $X$,
    \item the maps $\phi_k|_{X_k\setminus\bigcup_{i=1}^m\gamma_k^i}:X_k\setminus\bigcup_{i=1}^m\gamma_k^i\lra X\setminus \bigcup_{i=1}^m\phi_k(\gamma_k^i)$ are diffeomorphisms of nodal surfaces and we denote their inverses by $\psi_k$,
    \item the unique complete hyperbolic structures of finite area $g_k$ on the punctured Riemann surface $X_k\setminus\{\text{nodes, marked points}\}$ induced by $J_k$ are such that $\psi_k^*g_k$ converge in $\cin_{\text{loc}}$ to $g$ on $X\setminus \left(\bigcup_{i=1}^m\phi_k(\gamma_k^i)\cup\{\text{nodes, marked points}\}\right)$ and
    \item the curves $\gamma_k^i$ are geodesics of $g_k$.
  \end{itemize}
  \label{prop:target-j-subsequence}
\end{prop}

\begin{rmk}
  So, in the hyperbolic picture we see curves degenerating to nodes on the target surface. Next we will describe this as stretching-the-neck and thus apply SFT-compactness. All other compactness phenomena then have to be excluded on the domain $C$.
\end{rmk}

\subsection{General SFT Compactness in Neck-Stretching Sequences}

\subsubsection{Neck-Stretching}

\label{sec:neck-stretching}

In this chapter we will discuss the SFT compactness theorem and necessary changes as well as definitions from Cieliebak--Mohnke~\cite{cieliebak_compactness_2005}. There are two generalizations from that paper that we will need in our case: First, we will allow to glue in arbitrary sequences of increasingly thinner cylinders instead of their fixed sequence and secondly we will allow for changing almost complex structures away from the neck-stretching region.

\begin{definition}
  A closed \emph{stable} hypersurface $M$ of a closed connected symplectic manifold $(X^{2n},\omega)$ is a closed submanifold of $X$ of codimension one such that $\omega_M\coloneqq\omega|_M$ is a stable Hamiltonian structure, i.e.\ $\omega_M$ is a closed $2$-form of maximal rank $n-1$ and there exists a $1$-form $\lambda$ on $M$ with $\lambda|_{\ker\omega_M}\neq 0$ and $\ker\omega_M\subset\ker\dd\lambda$ everywhere.
  \label{def:stable-hypersurface}
\end{definition}

Now let $(X_0^{2n},\omega)$ be a closed connected symplectic manifold, $M\subset X_0$ a closed stable hypersurface and $\lambda$ the corresponding $1$-form on $M$. We fix a parametrization of a bi-collar neighborhood of $M$ by $[-\epsilon,\epsilon]\times M$ for some $\epsilon >0$ such that $\omega=\omega_M+\dd (r\lambda)$, were $r\in[-\epsilon,\epsilon]$. Now we pick a sequence of tamed almost complex structures $\{\ol{J}_k\}_{k\in\NN}$ on $(X_0,\omega)$ such that their restriction $\ol{J}_k|_M$ to $[-\epsilon,\epsilon]\times M$ is the restriction of an SFT-like almost complex structure on $\RR\times M$, defined in \cref{def:tamed-complex-structure}. Furthermore we require that the sequence $\ol{J}_k$ converges in $\cin$ outside $[-\epsilon,\epsilon]\times M$ to an almost complex structure $J$ on $X_0\setminus[-\epsilon,\epsilon]\times M$.\footnote{This will later force us to include all collapsing curves on the target surfaces $X_k$ in the hypersurface $M$.}

\begin{definition}
  The \emph{Reeb vector field} $R$ on $M$ is defined by $\omega_M(R,\cdot)=0$ and $\lambda(R)=1$. Furthermore we define the CR-structure $\xi\coloneqq\ker\lambda$.
\end{definition}

\index{Reeb Vector Field}
\index{SFT-like Almost Complex Structure}

\begin{definition}
  A \emph{SFT-like almost complex structure} $J$ on $(X_0,\omega)$ for a closed stable hypersurface $M$ with a chosen stabilizing $1$-form $\lambda$ on $M$ and a fixed bi-collar $[-\epsilon,\epsilon]\times M\subset X_0$ is an almost complex structure $J$ on $X_0$ which tames $\omega$ and its restriction $J_M\coloneqq J|_{[-\epsilon,\epsilon]\times M}$ satisfies the following properties:%\footnote{We choose the stabilizing $1$-form $\lambda\in\Omega^1(M)$ such that $\lambda\left(\dd\pi_M\cdot J\left(\frac{\del}{\del r}\right)\right)=1$ and $\lambda|_{\xi}=0$ for $\xi\coloneqq\ts(\{0\}\times M)\cap J(\ts(\{0\}\times M))$.}
  \begin{enumerate}[label=(\roman*), ref=(\roman*)]
    \item $J_M\left(\frac{\del}{\del r}\right)=R$ is the Reeb vector field on $M$ and
    \item $\omega_M(v,J_Mv)>0$ for all $v\in\xi\setminus\{0\}$,
  \end{enumerate}
  where $\xi\coloneqq\ker\lambda$.
  \label{def:tamed-complex-structure}
\end{definition}

\begin{rmk}
  As is explained in \cite{cieliebak_compactness_2005} such tamed complex structures exist and the set of tamed complex structures for a given symplectic manifold and closed stable hypersurface is contractible. Note that we pick in fact a \emph{sequence} of such tamed complex structures as we want to allow for varying complex structures away from the neck-stretching. So we require that on the bi-collar neighborhood the $\ol{J}_k$ are independent of $k$ and are equal to one fixed SFT-like almost complex structure on $\RR\times M$.
\end{rmk}

At this point Cieliebak and Mohnke introduce a sequence of modified symplectic manifolds $X_k$ which consist of $X_0$ and replace the bi-collar neighborhood by $[-\epsilon-k,\epsilon]\times M$ with $k\in\NN$. However we will need a more general sequence and will replace $k$ by a sequence $\{w_k\}_{k\in\NN}\subset\RR_+$ with the property $\lim_{k\to\infty}w_k=+\infty$ and $w_0=0$ allowing for more general widths of collar neighborhoods. Furthermore, if $M=\bigsqcup M^{\nu}$ has several connected components we might choose different sequences $w_k^{\nu}$ for every $\nu$. As this complicates the notation further we will drop the index and treat every connected component separately.

Now we choose a sequence of diffeomorphisms $\phi_k:[-w_k-\epsilon,\epsilon]\lra[-\epsilon,\epsilon]$ with $\phi_k'=1$ near the end points of the intervals (again, for every connected component $M^{\nu}$ of $M$) and define

\begin{enumerate}[label=(\roman*), ref=(\roman*)]
  \item $\mathring{X}_0\coloneqq X_0\setminus[-\epsilon,\epsilon]\times M$,
  \item $M_{\pm}\coloneqq \{\pm\epsilon\}\times M$,
  \item $X_k\coloneqq \mathring{X}_0\cup_{M_-\sqcup M_+}[-w_k-\epsilon,\epsilon]\times M$,
  \item $J_k\coloneqq
    \begin{cases}
      \ol{J}_k & \text{on }\mathring{X}_0 \\
      J_M & \text{on }[-w_k-\epsilon,\epsilon]\times M,
    \end{cases}$
  \item $\omega_k\coloneqq
    \begin{cases}
      \omega & \text{on }\mathring{X}_0 \\
      \omega_M+\dd(\phi_k\lambda) & \text{on }[-w_k-\epsilon,\epsilon]\times M.
    \end{cases}$
\end{enumerate}

\begin{definition}
  To recap, given a stable hypersurface $M$ in a closed connected symplectic manifold $(X_0,\omega)$ together with a sequence $\{w_k\}_{k\in\NN}\subset\RR_+$ as above for every connected component of $M$, a sequence $\ol{J}_k$ of tamed almost complex structures on $X_0$ as well as a bi-collar neighborhood $[-\epsilon,\epsilon]\times M\subset X_0$ the above procedure defines a \emph{neck-stretching sequence} $(X_k,\omega_k,J_k)$.
  \label{def:neck-stretching-sequence}
\end{definition}

\index{Neck-Stretching Sequence}

The ``limit'' object of such a neck-stretching sequence will be of the following kind. For an integer $N\geq 1$  define
\begin{equation*}
  (X^{\nu},J^{\nu})\coloneqq
  \begin{cases}
    (\mathring{X}_0\sqcup[-\epsilon,\infty)\times M\sqcup(-\infty,\epsilon)\times M,J) & \text{for }\nu=0,N+1 \\
    (\RR\times M,J_M) & \text{for }\nu=1,\ldots,N
  \end{cases}
\end{equation*}
as well as
\begin{equation*}
  X^*\coloneqq X^0\sqcup\bigsqcup_{\nu=1}^NX^{\nu}
\end{equation*}
equipped with the almost complex structure $J^*$ induced by the $J^{\nu}$. Also one can glue the positive boundary component of the compactification $\ol{X^{\nu}}$ by adding copies of $M$ to the negative boundary component of $\ol{X^{\nu+1}}$. Call this compact topological space $\ol{X}$ which is homeomorphic to $X_0$ and is illustrated in the upper part of \cref{fig:map-phi}.

\begin{rmk}
  Note that the definition of $X^*$ does not use a ($N+1$)-st level. This is just used for simplifying the notation for the broken ends as the positive ends of $X^N$ are glued to the negative ends of $X^{N+1}=X^0$.

  Also notice that it is not immediately clear that this notion of convergence of the neck-stretching sequence $(X_k,J_k)$ to $(X^*,J^*)$ is the same as Deligne--Mumford convergence. In our case, however, we will see that the final object does indeed arise as an actual limit of degenerating hyperbolic surfaces in the Deligne--Mumford sense.
\end{rmk}

Furthermore we need to choose certain homeomorphisms to translate between $X_0$ and $\ol{X}$ and to obtain more structures on the ``limits'' $\ol{X}$ and $X^*$. To this end choose a homeomorphism
\begin{equation*}
  \Phi:\ol{X}\lra X_0
\end{equation*}
given by
\begin{equation*}
  \Phi(x)\coloneqq
  \begin{cases}
    x & x\in X_0\setminus (-\epsilon,\epsilon)\times M \\
    (\phi(r),y) & x=(r,y)\in \ol{(\RR\times M)^{\nu}}
  \end{cases}
\end{equation*}
with a homeomorphism
\begin{equation}
  \phi:[-\epsilon,+\infty]\cup_{\pm\infty}[-\infty,+\infty]\cup_{\pm\infty}\cdots \cup_{\pm\infty}[-\infty,\epsilon]\lra [-\epsilon,\epsilon]
  \label{eq:def-phi-shift}
\end{equation}
that is a diffeomorphism outside $\pm\infty$ and $\phi(r)=r$ near $\pm\epsilon$. Note that the domain of $\phi$ has $N+2$ intervals because in the definition of $X^*$ and $\ol{X}$ we used $X_0$ twice. Denote by $\phi^0$ the restriction of $\phi$ to the first and last interval and by $\phi^{\nu}$ the restriction to the $(\nu+1)$-st interval for $\nu=1,\ldots,N$.

We can use the map $\Phi$ to pull back $\omega$ to obtain a symplectic form $\omega_{\Phi}$ on $X^*$.

\begin{figure}[H]
    \centering
    \def\svgwidth{0.8\textwidth}
    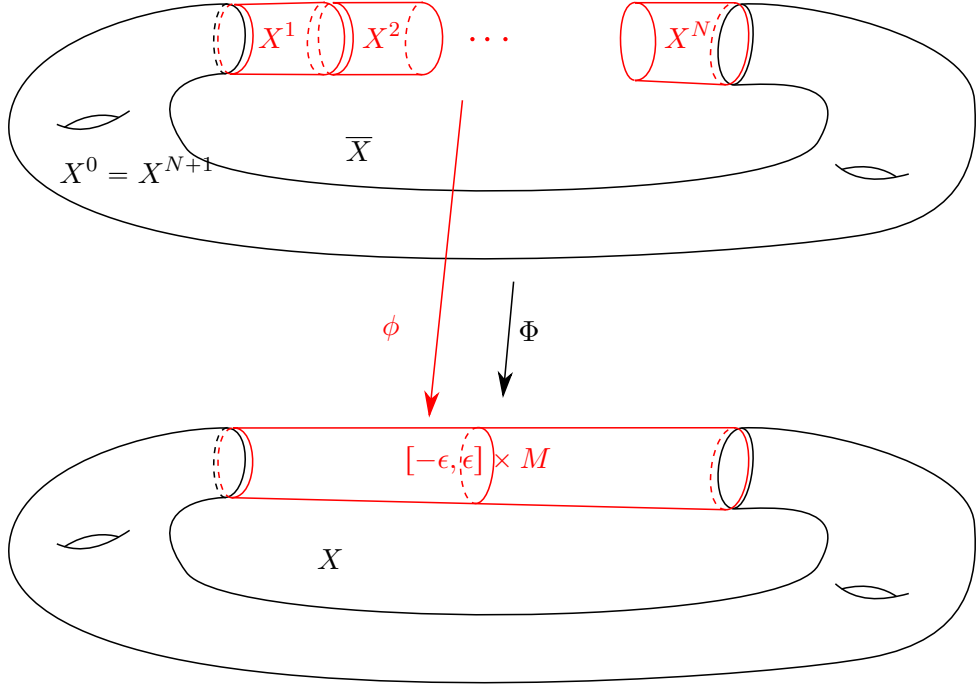
    \caption{This is an illustration of the map $\Phi:\ol{X}\lra X_0$ which will be used to relate various constructed objects on the limit broken curve to the surface $X_0$. Note that the ends of the tubes in the upper figure are given by $[-\infty,\infty]\times M$ so the cylindrical ends are open neighborhoods of the glued-in copy of $M$ at the boundaries between these tubes.}
    \label{fig:map-phi}
\end{figure}

\subsubsection{Broken Holomorphic Curves}

Before stating the definition of the convergence of a sequence of $J_k$-holomorphic curves in a neck-stretching sequence to a broken holomorphic curve we need to introduce various notions on surfaces.

First, suppose we are given a stable Hamiltonian structure $\omega_M$ on $\RR\times M$ stabilized by a $1$-form $\lambda$. This form defines a Reeb vector field $R$ on $M$ for which we are given a $T$-periodic closed orbit $\gamma:[0,T]\lra M$. Then we define

\begin{definition}
  A $J$-holomorphic map $f=(a,u):\DD\setminus\{0\}\lra\RR\times M$ is called \emph{positively asymptotic} to $\gamma$ if $\lim_{\rho\to 0}a(\rho e^{\ii\theta})=+\infty$ and $\lim_{\rho\to 0 }u(\rho e^{\ii\theta})=\gamma\left(\frac{T\theta}{2\pi}\right)$ uniformly in $\theta$. The map $f$ is negatively asymptotic to $\gamma$ if the same holds with $-\infty$ for $a$ and $\gamma\left(-\frac{T\theta}{2\pi}\right)$ for $u$ as the limits.
\end{definition}

\begin{figure}[H]
    \centering
    \def\svgwidth{0.8\textwidth}
    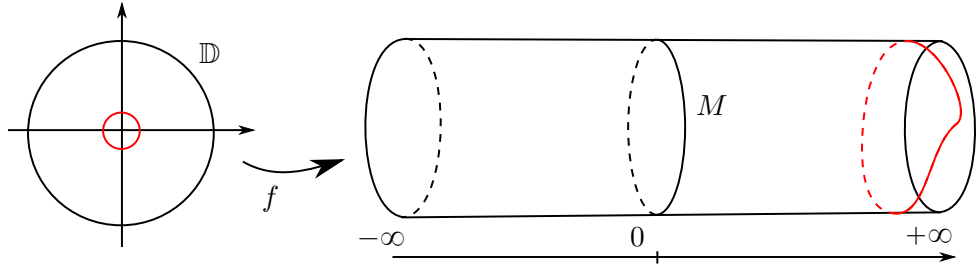
    \caption{A holomorphic map can be positively asymptotic at a puncture in its domain if the images of small circles around the puncture converge to $+\infty$ in the $\RR$-coordinate and converge uniformly to a parametrized Reeb orbit of $M$. Note that we have drawn the case $M=S^1$ which is indeed what we need in our considerations.}
    \label{fig:asymptotic-puncture}
\end{figure}

Now we can define the notion of a punctured holomorphic curve in $(X_0,J)$.

\begin{definition}
  A \emph{punctured holomorphic curve} is a tuple $(C,u,\ol{\bz},\ul{\bz},\ol{\Gamma},\ul{\Gamma})$ where
  \begin{itemize}
    \item $C$ is a nodal Riemann surface with pairwise disjoint marked points 
      \begin{equation*}
        \ol{\bz}=(\ol{z}_1,\ldots,\ol{z}_{\ol{p}})\text{ and }\ul{\bz}=(\ul{z}_1,\ldots,\ul{z}_{\ul{p}}),
      \end{equation*}
    \item corresponding vectors $\ol{\Gamma}=(\ol{\gamma}_1,\ldots,\ol{\gamma}_{\ol{p}})$ and $\ul{\Gamma}=(\ul{\gamma}_1,\ldots,\ul{\gamma}_{\ul{p}})$ of corresponding closed Reeb orbits in $M$ and
    \item a $j$-$J$-holomorphic map $u:C\setminus\ol{\bz}\cup\ul{\bz}\lra X_0$ which is positively asymptotic to $\ol{\gamma}_i$ at $\ol{z}_i$ for any $i=1,\ldots,\ol{p}$ and negatively asymptotic to $\ul{\gamma}_i$ at $\ul{z}_i$, respectively.
  \end{itemize}
  \label{def:punctured-holomorphic-curve}
\end{definition}

\begin{figure}[H]
    \centering
    \def\svgwidth{0.8\textwidth}
    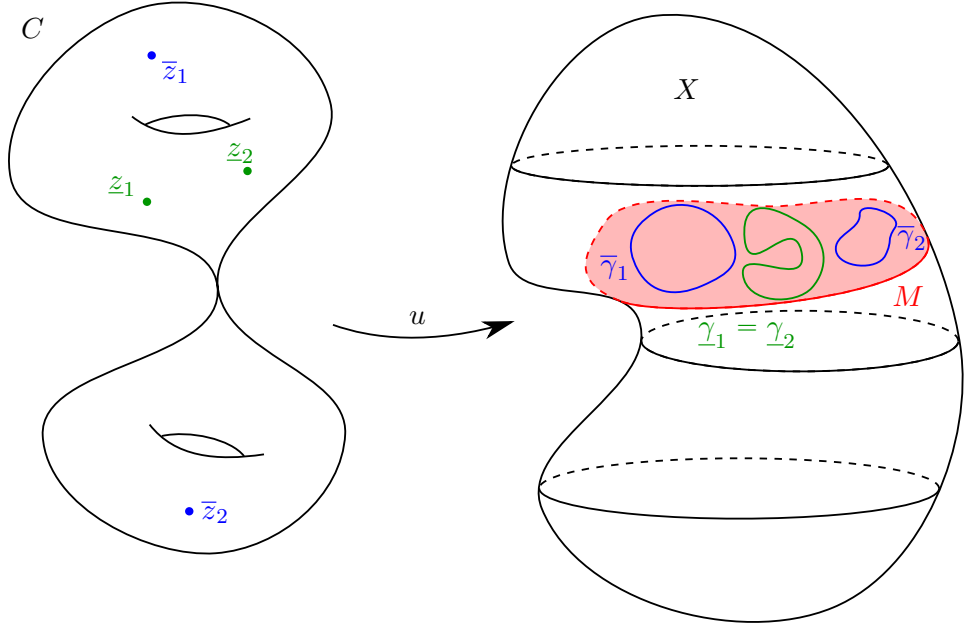
    \caption{Here one can see a punctured holomorphic curve (i.e.\ \emph{one} level of a broken holomorphic curve) as in \cref{def:punctured-holomorphic-curve} in a symplectic manifold $X$ of dimension $\dim X\geq 4$ with a stable hypersurface $M$. The curves in $M$ are Reeb orbits which are the limits of $u$ at the punctures $z_i$. Notice that nodes are allowed as they can be present in a sequence of (nodal) Hurwitz covers. There is no condition on the node as $X$ is not a surface.}
    \label{fig:punctured-holomorphic-curve}
\end{figure}

We will now define the notion of a broken holomorphic curve for the neck-stretching manifold $(X^*,J^*)$. Let $\ol{C}$ be a closed oriented connected surface with a tuple $\bz=(z_1,\ldots,z_k)\in \ol{C}^k$ of distinct marked points and two collections $\Delta_{\text{p}}\coloneqq \{\delta_i\}_{i=1}^p$ and $\Delta_{\text{n}}$ of finitely many disjoint and pairwise non-homotopic simple loops in $\ol{C}\setminus\bz$. We set $C\coloneqq \faktor{\ol{C}}{\Delta_n}$ where each loop in $\Delta_n$ is collapsed to a nodal point. Then we equip $C\setminus\Delta_{\text{p}}$ with a complex structure such that the loops in $\Delta_{\text{p}}$ become punctures from both sides and the collapsed curves $\Delta_{\text{n}}$ become nodes. This way $C\setminus\Delta_{\text{p}}$ is a closed oriented connected nodal Riemann surface with marked points $\bz$ and punctures $\delta_i$ for $i=1,\ldots,p$. Furthermore we choose a decomposition of $C\setminus\Delta_{\text{p}}$ into disjoint surfaces

\begin{equation*}
  C^*\coloneqq C\setminus\Delta_{\text{p}}\eqqcolon C^0\sqcup\bigsqcup_{\nu=1}^NC^{\nu}
\end{equation*}

such that all $C^{\nu}$ are pairwise disjoint not necessarily connected nodal Riemann surfaces with punctures. Note that the definition of the individual $C^{\nu}$'s is part of a choice. Also note that the curves $\delta_i$ give us an identification of pairs of punctures on $C^*$.

\begin{definition}
  A \emph{broken holomorphic curve} with $N+1$ levels
  \begin{equation*}
    F=(F^0,F^1,\ldots,F^N):C^*\lra(X^*,J^*)
  \end{equation*}
is a collection of punctured holomorphic curves $F^{\nu}:C^{\nu}\lra(X^{\nu},J^{\nu})$ such that $F: C^*\lra X^*$ extends to a continuous map $\ol{F}:C\lra \ol{X}$. It is called \emph{stable} if
  \begin{enumerate}[label=(\roman*), ref=(\roman*)]
    \item for $1\leq\nu\leq N$ the map $F^{\nu}$ is not the union of trivial cylinders over closed Reeb orbits without any marked points on them,
    \item no component of $C^*$ is a sphere with less than three special points\footnote{By special points we mean as usual either punctures, nodes or marked points.} on which $F$ is constant and
    \item no component of $C^*$ is a torus without any special points on which $F$ is constant.
  \end{enumerate}
  \label{def:broken-holomorphic-curve}
\end{definition}

\begin{rmk}
  Note that the continuity of $F$ implies that the number of positive and negative punctures as well as their corresponding Reeb orbits of consecutive punctured holomorphic curves coincide. Also, the map $\Phi\circ\ol{F}: C\lra X_0$ represents a well-defined homology class $A\in H_2(X_0;\ZZ)$. See \cref{fig:broken-holomorphic-curve} for an illustration. Note also that a ``trivial'' double cover of a trivial Reeb cylinder is in fact a trivial Reeb cylinder itself, namely over the Reeb orbit run through twice.
\end{rmk}

\begin{figure}[H]
    \centering
    \def\svgwidth{0.8\textwidth}
    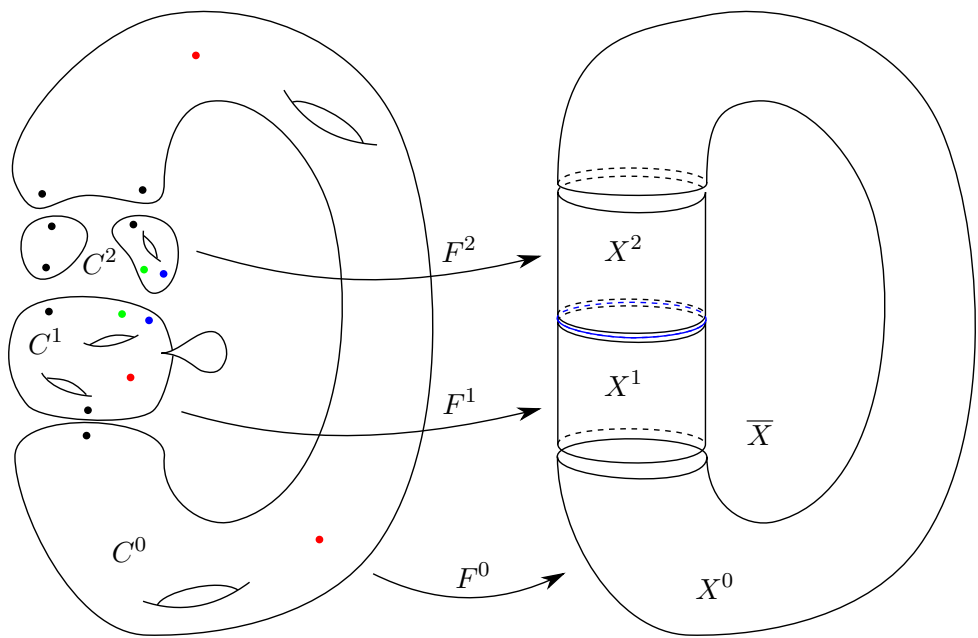
    \caption{This figure shows a broken holomorphic curve with $N=2$ levels. The black dots are asymptotic punctures whereas the red dots correspond to marked points. Notice that the broken holomorphic curves at each level can be disconnected and nodal. The pair of blue dots converge positively and negatively, respectively, to the blue Reeb orbit whereas the green ones might converge to the double of this Reeb orbit.}
    \label{fig:broken-holomorphic-curve}
\end{figure}

\subsubsection{Convergence of Broken Holomorphic Curves}

We will now state the definition of when a sequence of holomorphic curves in a neck-stretching sequence converges to a broken holomorphic curve. Unfortunately this involves quite a lot of additional notation that we will not need later on. This is why we put the definition of the various shift maps and diffeomorphisms into \cref{appendix-convergence-broken-holomorphic-curves}. Define the set $I^{\nu}$ as the set of indices $i\in\{1,\ldots,p\}$ such that $\delta_i\in\Delta_{\text{p}}$ is adjacent to $C^{\nu}$ and $C^{\nu+1}$. Also fix cylinders $A^i\subset C\setminus\Delta_{\text{n}}$ with a parametrization by $[-1,1]\times\faktor{\RR}{\ZZ}$ such that $\{-1\}\times\faktor{\RR}{\ZZ}\in C^{\nu}$ for $i\in I^{\nu}$ and $\{0\}\times\faktor{\RR}{\ZZ}=\delta^i$.

\begin{definition}
  A sequence of punctured holomorphic curves with $q$ marked points\footnote{To reduce confusion with the index $k$ of the member of the sequence we will temporarily use $q$ for the number of marked points.} $f_k:(C_k,j_k,\bz_k)\lra(X_k,J_k)$ in the neck-stretching sequence $(X_k,J_k)$ converges to a broken holomorphic curve with $q$ marked points $F:(C^*,j,\bz)\lra (X^*,J^*)$ if there exist orientation preserving diffeomorphisms $\phi_k:C_k\lra \ol{C}$ together with sequences of numbers
  \begin{equation*}
    -w_k=r_k^0<r_k^1<\cdots <r_k^{N+1}=0
  \end{equation*}
with $\lim_{k\to\infty}r_k^{\nu+1}-r_k^{\nu}=\infty$ such that the following hold:
\begin{enumerate}[label=(\roman*), ref=(\roman*)]
  \item $(\phi_k)_*j_k\to j$ in $\cin_{\text{loc}}$ on $C^*\setminus\{\text{nodes}\}$ and $\phi_k(\bz_k^l)\to\bz^l$ for $l=1,\ldots,q$,
  \item for every $i\in I^{\nu}$ the annulus $(A_i,(\phi_k)_*j_k)$ is conformally equivalent to a standard annulus $[-L_k^i,L_k^i]\times\faktor{\RR}{\ZZ}$ by a diffeomorphism of the form $(s,t)\longmapsto(\sigma(s),t)$ with $L_k^i\to\infty$ for $k\to\infty$,
  \item $f_k^{\nu}\circ\phi_k^{-1}\to F^{\nu}$  in $\cin_{\text{loc}}$ on $C^{\nu}\setminus\{\text{nodes}\}$ and in $C^0_{\text{loc}}$ on $C^{\nu}$,
  \item for every $i\in I^{\nu}$ the maps $\pi_M\circ f_k\circ\phi_k^{-1}\to\pi_M\circ\ol{F}$ converge uniformly on $A_i$, \label{item:def-convergence-broken-hol-curve}
  \item for every $R>0$ there exists $\rho>0$ and $K\in\NN$ such that $\pi_{\RR}\circ f_k\circ^{-1}(s,t)\in[r_k^{\nu}+R,r_k^{\nu+1}-R]$ for all $k\geq K$ and all $(s,t)\in A^i$ with $|s|\leq\rho$ and
  \item $\int_{C_k}f_k^*\omega_{\phi_k}\to\int_{C^*}F^*\omega_{\phi}=\omega([\ol{F}])$.
\end{enumerate}
\label{def:convergence-broken-hol-curve}
\end{definition}

\begin{figure}[h!]
    \centering
    \def\svgwidth{\textwidth}
    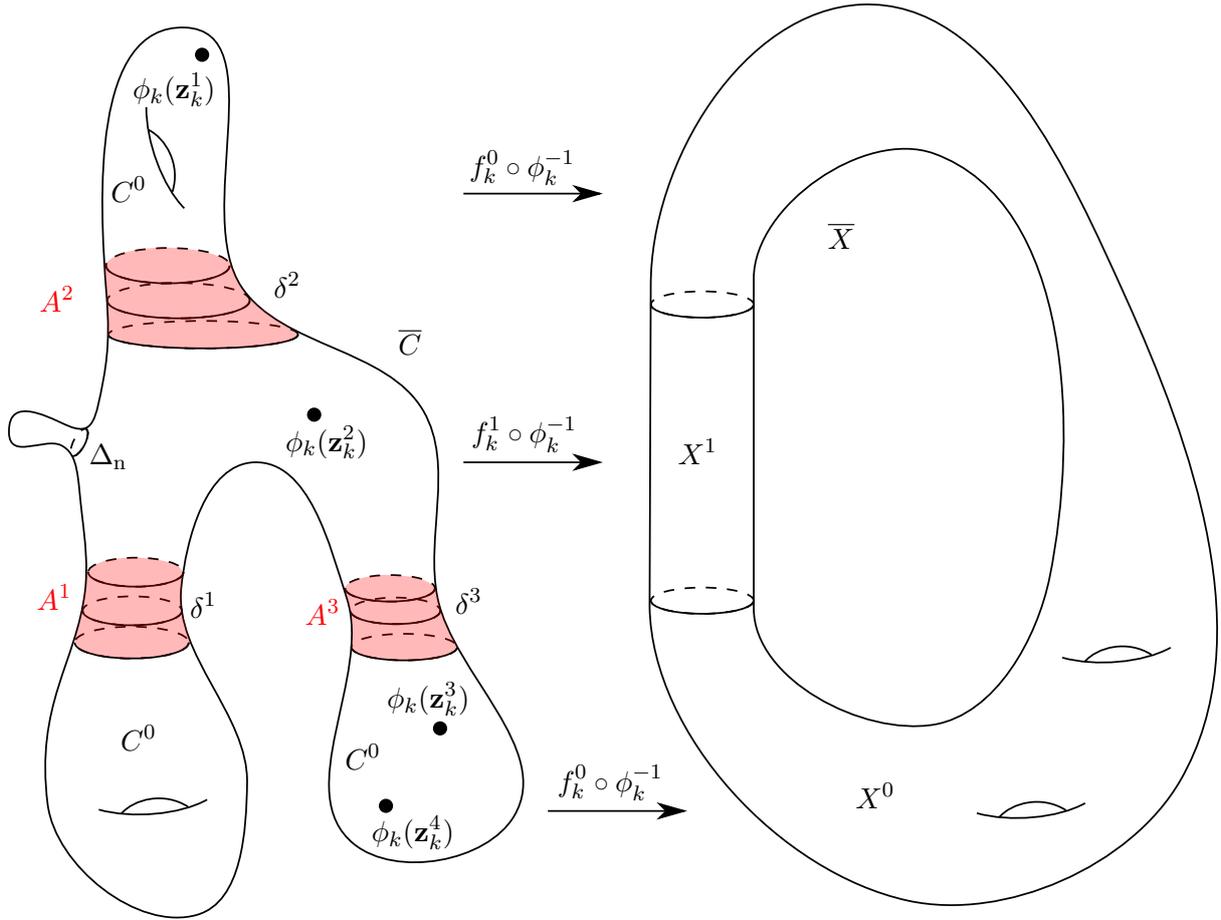
    \caption{This shows most of the objects appearing in the definition of convergence on the surface $\ol{C}$ meaning that everything is pulled back by the diffeomorphisms $\phi_k$ outside the breaking curves $\delta^i$. In particular the maps defined on the $C^{\nu}$ then converge in $\cin_{\text{loc}}$ outside these curves as well as those corresponding to nodes denoted by $\Delta_{\text{n}}$. Furthermore the definition includes statements on the behavior of the maps on the annuli $A^i$ as well as convergence of the symplectic area and the homology class of the continuous map $\ol{C}\lra\ol{X}\lra X_0$.}
    \label{fig:convergence-to-broken-curve}
\end{figure}

\subsubsection{SFT-Compactness}

\begin{thm}[Gromov--Hofer compactness]
  Let $(X_0,\omega)$ be a closed symplectic manifold and $M\subset X_0$ a stable closed hypersurface. Assume that all closed Reeb orbits on $(M,\lambda)$ are Bott non-degenerate. Let $(X_k,J_k)$ be a neck-stretching sequence as above. Let furthermore be $f_k:(C_k,\bz_k)\lra X_k$ a sequence of $J_k$-holomorphic curves of the same genus with $q$ marked points and with uniformly bounded area $\int_{C_k}f_k^*\omega\leq E_0$ where we identify $X_k\cong X$. Then a subsequence of $f_k$ converges to a broken holomorphic curve with $q$ marked points $F:(C^*,\bz)\lra(X^*,J^*)$ in the sense of \cref{def:convergence-broken-hol-curve}.
  \label{thm:sft-compactness}
\end{thm}

\begin{proof}
  This theorem is proven in Cieliebak--Mohnke~\cite{cieliebak_compactness_2005}. Also see the original version in \cite{bourgeois_compactness_2003}. However, the way we set up things we need two adaptations from the Cieliebak--Mohnke proof. Firstly, in our definition of the neck-stretching sequence $(X_k,J_k)$ the complex structure is allowed to vary away from the neck-stretching region and secondly the ``width'' of the necks is not fixed to be $k$ but we allow more general divergent sequences $w_k$. However, the proof of Cieliebak and Mohnke still works with the following comments:
  \begin{enumerate}[label=(\roman*), ref=(\roman*)]
    \item Outside the neck-stretching area $[-\epsilon,\epsilon]\times M$ we can just apply the usual Gromov convergence for varying target complex structures if they converge in $\cin_{\text{loc}}$ as described e.g.\ in \cite{hummel_gromovs_1997} instead of the one used in \cite{cieliebak_compactness_2005} for this region.
    \item In \cite{cieliebak_compactness_2005} the choice $w_k=k$ does not actually enter the calculations and was merely chosen to simplify the notation. Notice that most of the constructions only need that $\lim_{k\to\infty}w_k=+\infty$. 
  \end{enumerate}
\end{proof}

\subsection{Applying SFT-Compactness to the Case of Hurwitz Covers}

\label{sec:apply-sft-compactness-hurwitz-covers}

Now suppose that we are given the subsequence from \cref{prop:target-j-subsequence}, i.e.\ the sequence $\{[C_k,u_k,X_k,\bq_k,\bp_k]\}_{k\in\NN}\subset|\mcM_{g,k,h,n}(T)|$ is such that the target surfaces converge to a nodal hyperbolic surface $X$. We can add the punctures coming from the cusps corresponding to the marked points back to the surface $X$ and can thus regard $X$ also as a closed nodal Riemann surface $X$ with marked points $\bp\in X^n$. Our goal in this section is to find a description of this convergence process as a neck-stretching sequence in the sense of Cieliebak--Mohnke as described in \cref{sec:neck-stretching} and then apply SFT-compactness as in \cref{thm:sft-compactness}.

First, uniformize all punctured nodal Riemann surfaces $X_k$ to obtain hyperbolic metrics $g_k$ on $X_k$. By \cref{prop:target-j-subsequence} we have geodesic curves $\gamma_k^i$ with $i=1,\ldots,m$ on $X_k$ such that their hyperbolic lengths $l(\gamma_k^i)$ with respect to $g_k$ converge to zero. By assumption, all other lengths of simple closed geodesics stay away from zero.

By passing to a subsequence we can assume that we can construct all surfaces $X_k$  from $X$ by the usual gluing instruction in a neighborhood of the Deligne--Mumford-space of the limit curve $X$. This means that on $X$ we choose a disc structure
\begin{equation*}
  \{(V^j_i,\psi_i^j)\}_{j\in\{\ast,\dagger\},i=1,\ldots,m}
\end{equation*}
where $\psi_i^j:V_i^j\lra \DD$ is a biholomorphism and $\bz_i=V_i^{\ast}\cap V_i^{\dagger}$ are the nodal points coming from the sequence $\{\gamma_k^i\}_{k\in\NN}$ of collapsing curves. We choose the disc neighborhoods of the nodes such that they are isometric to $(-\infty,\ln \frac{3}{2})\times S^1$ with the hyperbolic metric $\dd\rho^2+e^{2\rho}\dd t^2$, i.e.\ such that the boundary of the disc is a horocycle of length $\frac{3}{2}$ in the uniformized hyperbolic metric. The reason for the strange number of $\frac{3}{2}$ is that we need some space between the discs used to glue in annuli in order to interpolate the given symplectic structure with a standard one on the annuli used in the Cieliebak--Mohnke \cref{thm:sft-compactness}. Thus note that there exists a larger disc neighborhood extending the given one isometric to $(-\infty,\ln 2)\times S^1$. Also note that here $S^1$ is parametrized by $[0,1)$.

We thus obtain a sequence of parameters in corresponding Teichmüller spaces as well as complex gluing parameters denoted by $(a^k_1,\ldots,a^k_m,\tau^k)$. Here $a_i^k\in\DD$ is the complex gluing parameter at node $\bz_i$ and $\tau^k\in\mcQ$ is a coordinate on the Teichmüller space from \cref{sec:vary-comp-str-away-from-node} which describes the variation of the complex structure on the nodal surface away from the node.

Furthermore we denote again the closed annulus by $A(v,w)=\{z\in\CC\mid v\leq|z|\leq w\}$ and the glued in cylinders by
\begin{equation*}
   Z_i^k:=\faktor{(V_i^{\dagger}\sqcup V_i^{\ast})\setminus\left((\psi_i^{\dagger})^{-1}(\mathring{B}_{|a^k_i|}(0))\sqcup(\psi_i^{\ast})^{-1}(\mathring{B}_{|a_i^k|}(0))\right)}{\sim}\subset X_{(a_1^k,\ldots,a_m^k,\tau^k)},
\end{equation*}
where 
\begin{align*}
  w\sim w' & \Longleftrightarrow w\in V_i^{\dagger}\setminus(\psi_i^{\dagger})^{-1}(\mathring{B}_{|a_i^k|}(0))\text{ and }w'\in V_i^{\ast}\setminus(\psi_i^{\ast})^{-1}(\mathring{B}_{|a_i^k|}(0)) \\
  & \phantom{\Longleftrightarrow} \text{ s.t. } \psi_i^{\dagger}(w)\cdot\psi_i^{\ast}(w')=a_i^k
\end{align*}
and the parameters were chosen such that $X_k$ is biholomorphic to $X_{(a_1^k,\ldots,a_m^k,\tau^k)}$. In the following we will always identify these surfaces and only talk about $X_k$.

Our goal is now to identify the suitable objects as needed in the SFT-compactness result by Cieliebak and Mohnke, including a symplectic surface $X_0$ together with a closed stable hypersurface $M\subset (X_0,\omega_0)$ with a bi-collar neighborhood $[-\epsilon,\epsilon]\times M\subset X_0$ where $\omega_0$ has a standard form.\footnote{Note that we are only interested in finding a convergent subsequence so we are free to modify the given sequence in finitely many members. In particular we  can just add a new first Hurwitz cover with target $X_0$.} Also we need a complex structure on $X_0$ compatible with $\omega_0$ such that its restriction to the bi-collar can be seen as the restriction of a translation invariant complex structure on $\RR\times M$. Since the sequence of surfaces will be constructed by replacing the bi-collar neighborhood by larger neighborhoods we can not use the standard hyperbolic neighborhoods of hyperbolic geodesics and will have to construct something different.

On the limit target surface $X$ we first choose $\omega$ to be the unique hyperbolic volume form compatible with its complex structure such that all marked points and nodes are cusps.\footnote{Note that $(X,\omega)$ is the (nodal) limit target surface and not the ``beginning'' symplectic manifold used in Cieliebak--Mohnke to construct the neck-stretching sequence. This will be constructed in the following.}

Secondly, any complex structure $J_k$ is compatible with $\omega$ as $X$ is a Riemann surface. To see this, notice that for any $v\in\ts_xX$ the vectors $(v,J_kv)$ form a positively oriented basis and thus $\omega(v,J_kv)>0$ for $v\neq 0$ and $\omega(J_kv,J_kJ_kv)=-\omega(J_kv,v)=\omega(v,J_kv)$ shows $\omega(J_k\cdot,J_k\cdot)=\omega(\cdot,\cdot)$.\footnote{Of course the $J_k$ are in fact the pulled-back complex structures from $X_k$ via maps $\phi_k$ which are diffeomorphisms outside the nodal points but we choose to simplify the notation a little bit.} Furthermore the sequences $a_i^k$ converge to zero for $k\to\infty$. By passing to yet another subsequence we can assume to start the sequence with some $X_1$ that is isomorphic to a surface constructed from $X$ by using variations $\tau^1\in\mcQ$ and parameters $a_1^1,\ldots,a_m^1$ such that $|a_i^k|< e^{-4\pi\ln\frac{3}{2}}$ for all $i=1,\ldots,m$ and $k\geq 1$. The reason for the last number lies in the fact that we will build the sequence $X_k$ from $X_0$ using a $\left[-\ln\frac{3}{2},\ln\frac{3}{2}\right]\times S^1$ neighborhood in the hyperbolic parametrization. Furthermore we restrict to a subsequence such that the twisting parameters $\arg a_i^k$ converge for $k\to\infty$, which is possible as $(S^1)^m$ is compact.

Next, we define the surface $X_0$ as constructed from $X$ using parameters $\tau^0,a_1^0,\ldots,a_m^0$ where $\tau^0\in\mcQ$ is arbitrary and $a_i^0=e^{-4\pi\ln\frac{3}{2}}$ for all $i=1,\ldots, m$.\footnote{We will denote the neck-stretching sequence by the same letters $X_k$. This sequence will be built such that the surfaces are biholomorphic to the given ones and thus we can simplify the notation a little bit.} This corresponds to no twisting as $\arg a_i^0=0$ and a cylinder isometric to $\left[-\ln\frac{3}{2},\ln \frac{3}{2}\right]\times S^1$. Recall that there are two parametrizations
\begin{align*}
  \psi_i^{\ast}: Z_i^0 & \lra A(|a_i^0|,1) \text{ and}\\
  \psi_i^{\dagger}:Z_i^0 &\lra A(|a_i^0|,1)
\end{align*}
of each cylinder. However, we chose disc neighborhoods such that it is also possible to parametrize the punctured discs $V_i^j\setminus\{z_i\}$ by hyperbolic cusp neighborhoods, i.e.\ by isometric maps
\begin{align*}
  \phi^j_i:V_i^j\setminus\{z_i\}\lra \left(-\infty,\ln \frac{3}{2}\right]\times S^1
\end{align*}
with the hyperbolic metric $g=\dd\rho^2+e^{2\rho}\dd t^2$ where $\rho\in \left(-\infty,\ln \frac{3}{2}\right]$ and $t\in S^1=\faktor{\RR}{\ZZ}$. The conformal coordinate change $\left(-\infty,\ln \frac{3}{2}\right]\times S^1\lra \DD\setminus\{0\}$ is given by
\begin{equation}
  (\rho,\theta)\longmapsto e^{2\pi(\rho-\ln\frac{3}{2}+\ii\theta)}
  \label{eq:conf-coord-change-bi-collar}
\end{equation}
which can be easily seen as there is only one biholomorphic map from the punctured disc on the punctured disc up to rotations.

Since the coordinates $\phi^j_i$ are isometries we know by a small calculation that the volume form $\omega$ on $X$ is given by $e^{\rho}\dd\rho\wedge\dd\theta$ in the local coordinates on $\left(-\infty,\ln \frac{3}{2}\right]\times S^1$.

These choices and definitions are summarized in \cref{fig:compactness-def-hyperbolic-cylinders} and \cref{prop:local-objects}.

\begin{figure}[h]
    \centering
    \def\svgwidth{\textwidth}
    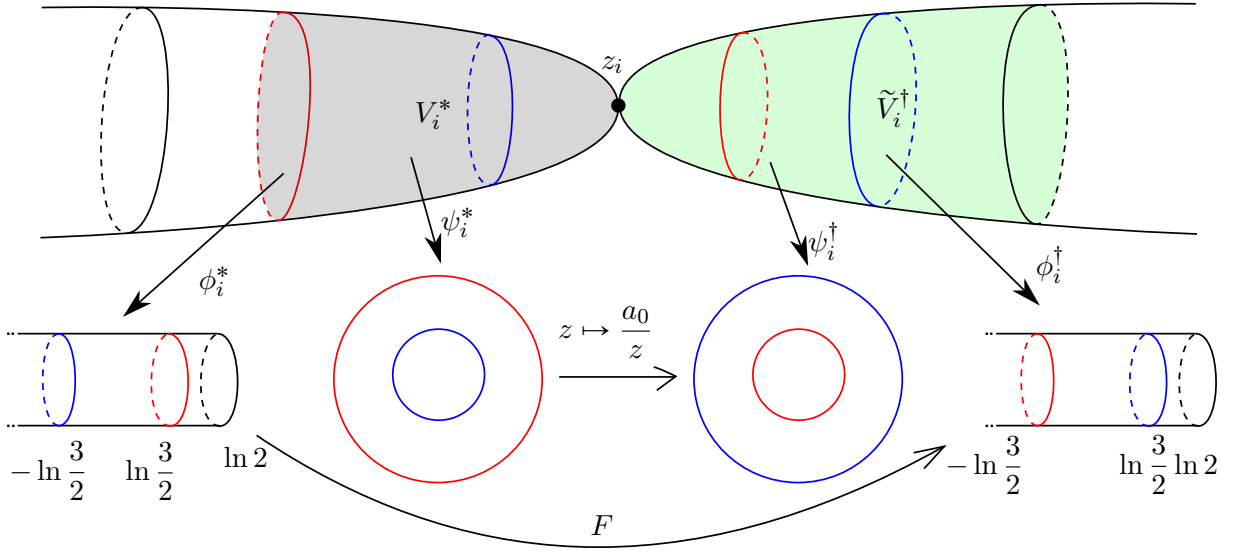
    \caption{When gluing with parameter $a_0$ we identify the two cylindrical regions according to the colors and via the map $z\mapsto \frac{a_0}{z}$. In order to simplify the notation a little bit we did not give names to all the individual coordinate charts and transition maps. However, notice that we do not modify the region between the black and green/red circles as we need these annuli in order to interpolate the symplectic forms. As these are the maximal disjoint cusp regions, our charts $\phi_i^{\ast}$ extend to these annuli and map them to $\left[\frac{3}{2},\ln 2\right]$. Also, the transition maps obviously only make sense on the glued surface. The gray set on the left hand side is the parametrized disc $V_i^{\ast}$ around the cusp $z_i$ but there exists a corresponding set $V_i^{\dagger}$ on the right hand side. However, we drew instead the set $\wt{V}_i^{\dagger}$ which includes the whole cusp neighborhood, i.e.\ $(\phi_i^{\dagger})^{-1}((-\infty,\ln2]\times S^1)$.}
    \label{fig:compactness-def-hyperbolic-cylinders}
\end{figure}

\begin{prop}
  In the hyperbolic charts $(\rho,\theta)\in\left(-\infty,\ln 2\right]\times S^1$ on the $\ast$-side as above we have
  \begin{itemize}
    \item $g=\dd\rho^2+e^{2\rho}\dd \theta^2$,
    \item $J=\begin{pmatrix} 0 & -1 \\ 1 & 0\end{pmatrix}$ and
    \item $\omega=e^{\rho}\dd\rho\wedge\dd\theta$.
  \end{itemize}
  Furthermore the transition maps on the glued surface where they are defined are given by
  \begin{align}
    \psi_i^{\ast}\circ(\phi_i^{\ast})^{-1}(\rho,\theta)&=e^{2\pi(\rho-\ln\frac{3}{2}+\ii\theta)}, \label{eq:hyp-complex-coord-change}\\
    F(\rho,\theta)&=\left(-\rho,\frac{\arg a_0}{2\pi}-\theta\right). \nonumber
  \end{align}
  \label{prop:local-objects}
\end{prop}

\begin{proof}
First, the conformal description together with its metric $g$ can be found in \cite{buser_geometry_2010}. The complex structure $J$ is just the standard complex structure on the cylinder as the chart $\phi_i^*$ is conformal and thus holomorphic. The transition function $\psi_i^{\ast}\circ(\phi_i^{\ast})^{-1}$ to the holomorphic disc chart used for gluing needs to map $\rho=-\infty$ to $0\in\DD$ and the boundary $\rho=\ln\frac{3}{2}$ to the boundary $\del\DD=\{z\in\CC\mid |z|=1\}$. This map is automatically unique up to a rotation which can be chosen arbitrarily. It is given by
\begin{equation*}
  (\rho,\theta)\longmapsto e^{2\pi(\rho-\ln\frac{3}{2}+\ii\theta)}
\end{equation*}
which is holomorphic and satisfies the boundary conditions. Thus the transition map $F$ is given by
\begin{align*}
  (\rho,\theta)\longmapsto e^{2\pi(\rho-\ln\frac{3}{2}+\ii\theta)} & \longmapsto a_i^0 e^{2\pi(-\rho+\ln\frac{3}{2}-\ii\theta)}=e^{2\pi\left(-\rho+\frac{1}{2\pi}\ln|a_i^0|+\ln\frac{3}{2}+\ii\left(-\theta+\frac{\arg a_i^0}{2\pi}\right)\right)} \\
  & \mapsto \left(-\rho+2\ln\frac{3}{2}+\frac{1}{2\pi}\ln|a^0_i|,\frac{\arg a^0_i}{2\pi}-\theta\right).
\end{align*}
Since we chose $a_i^0=e^{-4\pi\ln\frac{3}{2}}$ we get $F(\rho,\theta)=(-\rho,-\theta)$ or $F(\rho,\theta)=\left(-\rho,\frac{\arg a_0}{2\pi}-\theta\right)$ if we include the twisting.

It remains to calculate the hyperbolic volume form. Since for any positively oriented coordinates $(x_1,x_2)$ we have $\omega=\sqrt{\det g}\dd x_1\wedge \dd x_2$ we obtain $\omega=e^{\rho}\dd\rho\wedge\dd\theta$. Notice that under the coordinate change $F$ it is mapped to $F^*\omega = e^{-\rho}\dd\rho\wedge\dd\theta$.
\end{proof}

Before constructing $\omega_0$ on $X_0$ we define the stable hypersurface $M$ as
\begin{equation*}
  M\coloneqq \bigsqcup_{i=1}^m\left(\psi^{\ast}_i\right)^{-1}\left(\left\{z\in\DD\mid |z|=e^{-2\pi\ln(\frac{3}{2})}\right\}\right)
\end{equation*}
In the hyperbolic chart on the $\ast$-side this corresponds to the circle $\{0\}\times S^1$ as well as on the $\dagger$-side by choice of $|a_i^0|$. However, the gluing map identifying the two cylinders rotates the circles by $\theta\mapsto \frac{\arg a_i^0}{2\pi}-\theta$. The $1$-form $\lambda$ on this union of circles is chosen as $\dd \theta$ in the hyperbolic $\ast$-chart.

\begin{lem}
  The Reeb orbits of $\lambda$ are given by $t\mapsto t+c\in\faktor{\RR}{\ZZ}$ for some $c\in[0,1)$ corresponding to the base point. Thus for every degree or winding number there is an $S^1$-family of Reeb orbits.
  \label{lem:reeb-orbits}
\end{lem}

\begin{proof}
  This is clear as $\dd\lambda=0$ and $\lambda(R)=\dd\theta (R)=1$.
\end{proof}

In order to obtain a suitable symplectic form on $X_0$ we need to extend the required symplectic form on the cylinders $Z_i$ by a compatible symplectic form on the slightly larger discs on the cusp neighborhoods $(-\infty,\ln 2)\times S^1$ such that they fit with the two hyperbolic volume forms on both boundaries. The required form on the cylinder $\left[-\ln\frac{3}{2},\ln\frac{3}{2}\right]\times S^1$ from Cieliebak--Mohnke is $\dd(\rho\lambda)$. This will be done in the $\ast$-chart. By taking care of the volume form at the boundary $|a_i^0|$ we make sure that this gives a well-defined symplectic structure $\omega_0$ on $X_0$.

The local calculations above show that we are thus looking for symplectic forms $\eta_i^j\in\Omega^2\left(\left[\ln\frac{3}{2},\ln 2\right)\times S^1,\RR\right)$ for $i=1,\ldots,m$ and $j\in\{\ast,\dagger\}$ such that
\begin{enumerate}[label=(\roman*), ref=(\roman*)]
  \item close to $\rho=\ln 2$ we have $\eta_i=e^{\rho}\dd\rho\wedge\dd\theta$, \label{item:prop-symp-form-1}
  \item close to $\rho=\ln\frac{3}{2}$ we have $\eta_i=\dd(\rho\lambda)=\dd\rho\wedge\dd \theta$ and \label{item:prop-symp-form-2}
  \item $\eta_i$ is compatible with the complex structure on the cylinder. \label{item:prop-symp-form-3}
\end{enumerate}

Note that although we need to extend the symplectic forms on both sides of the glued in cylinder, the local calculations for the extension are actually the same in the two charts because the coordinate change between the conformal parametrizations are given by affine linear maps such that the pullback of $\dd\rho\wedge\dd\theta$ is again $\dd\rho\wedge\dd\theta$.

\begin{lem}
  There exists a symplectic form $\eta_i^j\in\Omega^2\left(\left[\ln\frac{3}{2},\ln 2\right)\times S^1,\RR\right)$ satisfying \cref{item:prop-symp-form-1}, \cref{item:prop-symp-form-2} and \cref{item:prop-symp-form-3} from above. In particular using the chart $\phi^j_i$ we find a bi-collar neighborhood $\left[-\ln\frac{3}{2},\ln\frac{3}{2}\right]\times \bigsqcup_{i=1}^mS^1$ of $M$ as required in \cref{sec:neck-stretching}.
  \label{lem:const-symp-form-bi-collar}
\end{lem}

\begin{proof}
  Choose a smooth function $f:\left[\ln\frac{3}{2},\ln 2\right]\lra \RR$ which is constant one in a neighborhood of $\ln\frac{3}{2}$ and equal to $e^{\rho}$ in a neighborhood of $\ln 2$ and is monotonely increasing. Define $\eta_i\coloneqq  f(\rho)\dd\rho\wedge\dd\theta$. This form satisfies \cref{item:prop-symp-form-1} and \cref{item:prop-symp-form-2} by construction and defines a symplectic form on the cylinder as $f(\rho)\neq 0$. Also it is compatible with the standard complex structure since this is a surface.
\end{proof}

Now we have defined $X_0, \omega_0, M$ and $\lambda$. Note that $\omega_M=0$ as $M$ is one-dimensional. Furthermore we already have fixed $J_0$ on $X_0$ which is compatible with the symplectic structure. Now we need to check that the parametrized bi-collar neighborhood constructed in \cref{lem:const-symp-form-bi-collar} is indeed of the correct form.

The bi-collar neighborhood was constructed such that $\omega_0=\dd (\rho\lambda)$. Furthermore its complex structure is given by $\frac{\del}{\del \rho}\mapsto \frac{\del}{\del \theta}$ and $\frac{\del}{\del \theta} \mapsto -\frac{\del}{\del \rho}$ as \cref{eq:conf-coord-change-bi-collar} is holomorphic with this structure. Thus it is translation invariant and in particular the restriction of the same complex structure on $\RR\times\bigsqcup_{i=1}^mS^1$. This complex structure is also of SFT-type, see \cref{def:tamed-complex-structure}, as
\begin{itemize}
  \item the Reeb vector field $R$ on $\bigsqcup_{i=1}^mS^1$ is given by $\frac{\del}{\del \theta}$ and $\dd\lambda=\omega_M=0$,
  \item $J\frac{\del}{\del \rho}=\frac{\del}{\del \theta}=R$ and
  \item the CR-structure $\xi=\{0\}$ implying that the last condition is vacuous.
\end{itemize}

All in all we see that we can apply the neck-stretching procedure to this data after choosing the following:

\begin{itemize}
  \item A sequence of positive real numbers $\{w_i^k\}_{k\in\NN}$ with $\lim_{k\to\infty}w_k=+\infty$ for every connected component $M^i$ for $i=1,\ldots,m$ and
  \item a sequence of diffeomorphisms $\phi_k:[-w_k-\epsilon,\epsilon]\lra[-\epsilon,\epsilon]$ such that $\phi'_k\equiv 1$ close to the boundaries.
\end{itemize}

Now we want to choose this data such that the resulting neck-stretching sequence $X_k$ is biholomorphic to the given sequence of target surfaces. By choosing the sequence of moduli $w_k$ of the annuli as
\begin{equation}
  w_i^k\coloneqq -\frac{|a_i^k|}{2\pi}-2\ln\frac{3}{2}
  \label{eq:choice-moduli-annuli}
\end{equation}
we recover the glued in cylinders $Z_i$ in $X_k$ in the neck-stretching sequence. However, as we cannot twist the cylinders in the neck-stretching procedure we need to include these twists given by $\arg a_i^k$ differently.

Recall from \cref{fig:compactness-def-hyperbolic-cylinders} that we used the $\ast$-side to glue in longer cylinders in the negative real direction. Thus we do not modify the complex structure in the region $\left[\ln\frac{3}{2},\ln 2\right]\times S^1$ when gluing in the longer cylinders, as we used this region only to interpolate between the given symplectic structure away from the nodes and the standard one on the cylinders. Thus we can use this region for implementing the twist coming from $\arg a_i^k$ and put the twisting in the variation within $J_k$.

For this purpose note that the holomorphic chart $\psi_i^{\dagger}:V_i^{\dagger}\lra\DD$ extends to a holomorphic disc $\psi_i^{\dagger}:\wt{V}_i^{\dagger}\lra B_{e^{2\pi\ln\frac{4}{3}}}(0)$ as the hyperbolic chart $\phi_i^{\dagger}$ extends to real coordinates up to $\ln 2$ corresponding to $|z|=e^{2\pi\ln\frac{4}{3}}$ under \cref{eq:hyp-complex-coord-change}. We will denote this larger set $(\phi_i^{\dagger})^{-1}((-\infty,\ln 2]\times S^2)$ by $\wt{V}_i^{\dagger}$. This is illustrated in \cref{fig:move-twist-outside}.

\begin{lem}
  Using the disc structure from \cref{fig:compactness-def-hyperbolic-cylinders} at one node $z_i$ and a gluing parameter $a\in\DD$ the identity map outside the gluing region extends to a well-defined biholomorphism between $X_a$ and
  \begin{equation*}
    \faktor{\wt{X}\setminus\left((\psi_i^{\ast})^{-1}(\mathring{B}_{|a|}(0))\sqcup(\psi_i^{\dagger})^{-1}(\mathring{B}_{1}(0))\right)\sqcup A(|a|,e^{2\pi\ln\frac{4}{3}})}{\sim},
  \end{equation*}
  where we identify
  \begin{equation*}
    z\in V_i^{\ast}\setminus (\psi_i^{\ast})^{-1}(\mathring{B}_{|a|}(0)) \sim w\in A(|a|,e^{2\pi\ln\frac{4}{3}}) \Longleftrightarrow \psi_i^{\ast}(z)\cdot w=|a|
  \end{equation*}
  and
  \begin{equation*}
    z\in \wt{V}_i^{\dagger}\setminus (\psi^{\dagger})^{-1}(\mathring{B}_{1}(0)) \sim w\in A(|a|,e^{2\pi\ln\frac{4}{3}}) \Longleftrightarrow \psi_i^{\dagger}(z)=\frac{a}{|a|}w.
  \end{equation*}
  \label{lem:move-twist-outside}
\end{lem}

\begin{proof}
  We need to check that the identity map gives well-defined holomorphic maps in all the charts. The calculations are illustrated in \cref{fig:move-twist-outside}.
\end{proof}

\begin{figure}[H]
    \centering
    \def\svgwidth{\textwidth}
    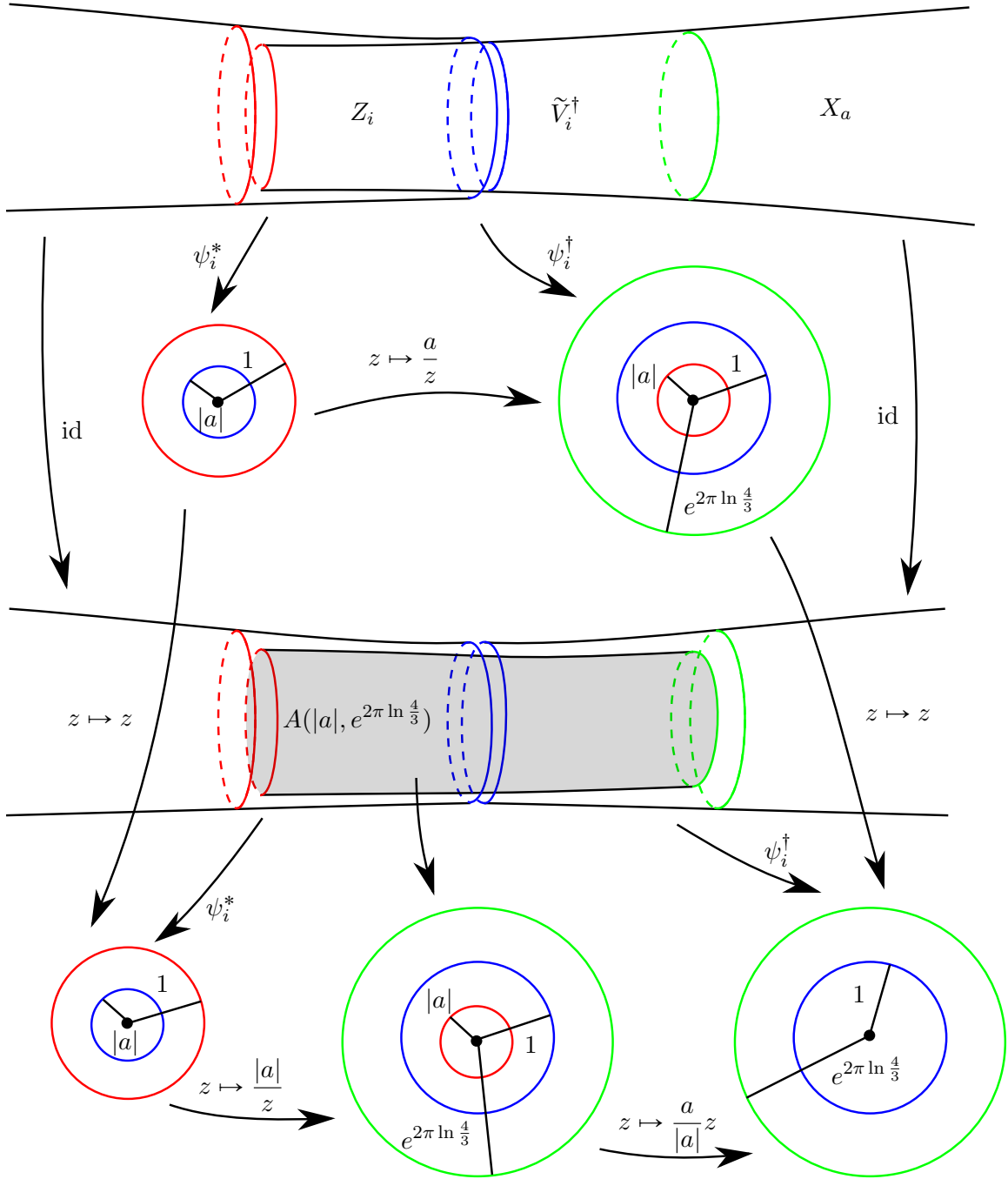
    \caption{This figure illustrates the construction of the surface in \cref{lem:move-twist-outside} as well as the identifications which need to be satisfied. Note that the construction is actually asymmetric as we glue in the longer annuli on the $\ast$-side whereas the twisting takes place on the $\dagger$-side in the blue--green area.}
    \label{fig:move-twist-outside}
\end{figure}

Using \cref{lem:move-twist-outside} we see that we can glue in the annuli as in the neck-stretching sequence with the moduli given by \cref{eq:choice-moduli-annuli} and the twisting contained in $\arg a_i^k$ can be captured in $J_k$ defined outside the bi-collar neighborhood by choosing appropriate diffeomorphisms in the setup of \cref{fig:move-twist-outside} outside the annulus and pulling back the complex structure. Now we can apply the SFT-compactness theorem \cref{thm:sft-compactness} to the sequence $u_k:C_k\lra X_k$.

\begin{prop}
  The sequence $u_k:C_k\lra X_k$ with the marked points $\bq_k$ satisfies the conditions of \cref{thm:sft-compactness}. Thus we obtain a subsequence $u_k:C_k\lra X_k$ which converges to a broken holomorphic curve $u^*:(C^*,\bq)\lra(X^*,J^*)$ with $k$ marked points in the sense of \cref{def:convergence-broken-hol-curve}.
\end{prop}

\begin{proof}
  We have seen that the sequence $X_k$ is a neck-stretching sequence starting at $(X_0,\omega)$ with $M$ given as above. The Reeb orbits are clearly Bott non-degenerate as they come in a one-parameter family as we have seen in \cref{lem:reeb-orbits}. The maps $u_k:C_k\lra X_k$ are $J$-holomorphic curves with $q$ marked points. As $u_k$ is a branched covering of Riemann surfaces and thus an actual covering of degree $d$ outside of a set of measure zero we have 
  \begin{equation}
    \int_{C_k}u_k^*\omega_k=d\int_{X_k}\omega_k\leq d\left(\vol_{\text{hyp}}(X)+2(3h-3+n)\ln \frac{8}{3}\right)
    \label{eq:uniform-bound-symp-area}
  \end{equation}
  which is independent of $k$. This can be seen as follows. The symplectic form on $X_0$ was built from the hyperbolic volume form on $X$ away from the collapsed nodes and the $\omega_k$ defined in \cref{def:neck-stretching-sequence}. In our construction this means that we have for every collapsed node an additional
  \begin{equation*}
    \int_{[-w_k-\epsilon,\epsilon]\times S^1}\dd(\phi_k\lambda)=\epsilon\int_{S^1}\lambda-(-\epsilon)\int_{S^1}\lambda=2\ln\frac{3}{2}
  \end{equation*}
  because of our choice $\epsilon=\ln\frac{3}{2}$ and $\phi_k:[-w_k-\epsilon,\epsilon]\lra[-\epsilon,\epsilon]$ with $\phi_k'=1$ near the end points of the intervals. Also we have the two interpolated regions from \cref{lem:const-symp-form-bi-collar} for each node giving an additional
  \begin{equation*}
    \int_{[\ln\frac{3}{2},\ln 2]\times S^1}\eta_i^j=\int_{[\ln\frac{3}{2},\ln 2]\times S^1}f(\rho)\dd\rho\wedge\dd\theta\leq 2\int_{\ln\frac{3}{2}}^{\ln 2}\dd\rho\int_0^1\dd\theta =2\ln\frac{4}{3}.
  \end{equation*}
  As we have at most $3h-3+n$ nodes on $X$ we obtain in total an additional area of $(3h-3+n)\left(2\ln\frac{3}{2}+2\cdot 2\ln\frac{4}{3}\right)$ which gives the upper bound in \cref{eq:uniform-bound-symp-area}.
\end{proof}

\begin{rmk}
  In the next section we will show that this limit curve is indeed a Hurwitz cover of type $T$ and that the sequence converges to this Hurwitz cover in the sense of \cref{prop:equivalent-formulation-topology}.

  Notice that as in Cieliebak--Mohnke the sequence $\omega_k$ does not converge to the symplectic structure $\omega$ on the Deligne--Mumford limit $(X,\omega)$. However, this is not important to us as we only need a neck-stretching sequence which is biholomorphic (and not symplectomorphic) to our given sequence in order to apply the SFT-compactness theorem.
\end{rmk}

\subsection{Relating Broken Holomorphic Curves with Hurwitz Covers}

\label{sec:relate-broken-holomorphic-curves-hurwitz-covers}

In the last section we have shown that there exists a subsequence of
\begin{equation*}
  \{[C_k,u_k,X_k,\bq_k,\bp_k]\}_{k\in\NN}\subset|\mcR_{g,k,h,n}(T)|
\end{equation*}
which converges to a broken holomorphic curve $u:(C^*,\bq)\lra(X^*,J^*)$ in the sense of \cref{def:convergence-broken-hol-curve}.

Recall that the stable hypersurface $M$ is a disjoint union of $S^1$'s and thus every connected component of $X^{\ast}$ is a non-compact Riemann surface with non-compact ends diffeomorphic to $(-\infty,\epsilon]\times S^1$ with the standard holomorphic structure as the $J$ from \cref{def:neck-stretching-sequence} is in our case induced by \cref{prop:local-objects}. Therefore every such cylindrical end is in fact biholomorphic to a punctured disc and we can compactify $X^*$ to obtain a nodal Riemann surface $X'$. This means that we glue at each of the two corresponding positive and negative cylindrical ends a holomorphic disc and identify the two additional points. At the moment we do not know much about this surface $X'$ but we will show that it is in fact biholomorphic to the DM-limit $X$ of the sequence $[X_k,\bp_k]$. Note that the intermediate levels $X^{\nu}$ for $\nu=1,\ldots,N$ compactify to a disjoint union of spheres with two nodal points and possibly marked points. This is illustrated in \cref{fig:compactify-annuli-to-spheres}.

\begin{figure}[h!]
    \centering
    \def\svgwidth{0.5\textwidth}
    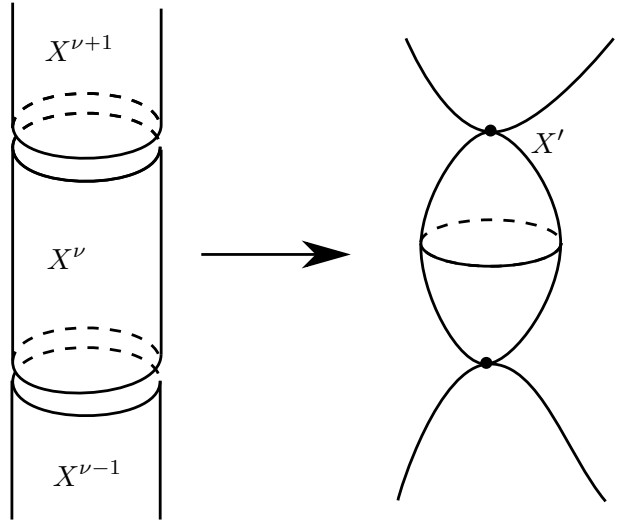
    \caption{By doing the same construction as in the proof of \cref{lem:nodal-extension-broken-holomorphic-curve} we glue in discs at the cylindrical ends of $X^*$ to obtain a nodal surface $X'$.}
    \label{fig:compactify-annuli-to-spheres}
\end{figure}

Next we do the same construction with the surface $C^*$ to obtain a nodal Riemann surface $C'$ such that $C^*\subset C'$ by adding these punctured discs at the cylindrical ends. We will show that the map $u:C^*\lra X^*$ extends to a holomorphic map between nodal Riemann surfaces $u:C'\lra X'$.

\begin{lem}
  The broken holomorphic curve $u:C^*\lra X^*$ extends to a holomorphic map between nodal Riemann surfaces $u:C'\lra X'$ with marked points $\bq\in C'$ and $u(\bq)\in X'$. Also all preimages of nodes are nodes, it is locally surjective at nodes\footnote{See \cref{def:hurwitz-cover} for an explanation of the term.} and the degree of the holomorphic map on both sides at a new node agrees with the multiplicity of the corresponding Reeb orbit.
  \label{lem:nodal-extension-broken-holomorphic-curve}
\end{lem}

\begin{proof}
  As we have explained above we obtain the nodal Riemann surfaces $C'$ and $X'$ by gluing in discs such that the negative cylindrical ends $(-\infty,0]\times S^1$ and positive cylindrical ends $[0,+\infty)\times S^1$ become punctured holomorphic discs. Note that the coordinate chart at a negative cylindrical end is given by
  \begin{align*}
    (-\infty,0]\times S^1 & \lra \DD \\
    (\rho,\theta) & \longmapsto e^{2\pi(\rho+\ii\theta)}
  \end{align*}
  and similarly for a positive end. From \cref{def:broken-holomorphic-curve} we know that the map $u:C^*\lra X^*$ is holomorphic everywhere and thus $u:C'\lra X'$ is holomorphic except at the added nodes. As the broken holomorphic curve $u:C^*\lra X^*$ respects the levels we see that $u$ is locally surjective at the nodes. Also if we restrict it to one disc in $C'$ and use the corresponding such glued-in disc in $X'$ the map $u$ is a bounded holomorphic function on the punctured disc and thus extends to a holomorphic one over the node. This works at all nodes and on every side. Thus $u:C'\lra X'$ is indeed a holomorphic map between nodal Riemann surfaces.

  Recall that $X$ was the DM-limit of $[X_k,\bp_k]$. Preimages of nodes $X'$ can come either from nodes that were already present in $X$ and in the sequence $X_k$ or from the procedure by gluing in discs at the cylindrical ends. Preimages of the nodes in $X_k$ under $u_k$ are nodes as this was a sequence of Hurwitz covers and preimages of the latter nodes are nodes as we glued in the discs at the cylindrical ends of $C^*$ as well.\footnote{We will see later that also nodes are mapped to nodes. However for this we need to exclude bubbling.}

  It remains to verify the statements about the degrees of $u:C'\lra X'$ at the nodes. If we write $z=e^{2\pi(\rho+\ii\theta)}$ and $f(z)=e^{2\pi(f_1(\rho,\theta)+\ii f_2(\rho,\theta))}$ then it is clear that the degree of $f$ at $0$ is the same as the winding number of $\theta\mapsto f_2(\rho,\theta)$ for sufficiently small $\rho$. This can be seen for example by writing $f(z)=z^dg(z)$ with $g(0)\neq 0$ for holomorphic $g$ and noting that by making $|z|$ small enough we have $0\not\in\Im g$ and therefore we can homotope the curve $g(\epsilon e^{2\pi\ii t})$ to a constant curve on $\CC\setminus\{0\}$ which does not change the winding number. Now recall from \cref{def:convergence-broken-hol-curve} and in particular \cref{item:def-convergence-broken-hol-curve} that the maps $\pi_M\circ u_k\circ\phi_k^{-1}$ converge uniformly on the cylinders $A_i\subset \ol{C}$ to $\pi_M\circ u$. In our case $M=S^1$ and thus the degrees of these maps converge to the degree of $\pi_M\circ u$. This means that the degree of $u$ at the glued-in node is equal to the winding number of $\epsilon e^{2\pi \ii t}$ for small enough $\epsilon$ which corresponds to $\rho$ very negative and thus this winding number is the order of the Reeb orbit. Since both sides of a node come from positive and negative cylindrical ends converging to the same Reeb orbit the corresponding nodal extensions have the same degree on both sides.
\end{proof}

Now we will show the following facts in this order:
\begin{enumerate}[label=(\roman*), ref=(\roman*)]
\item The sequence of source Riemann surfaces $(C_k,\bq_k)$ converges in DM to $(C',\bq)$.
  \item Any marked points $\bq$ of the limit curve are such that $u(q_j)=p_{\nu(j)}$ for the limit marked points $\bp$ in the Deligne--Mumford sequence $[X_k,\bp_k]$. Also all $q_j$ are contained in $C^0$ and all $p_i$ are contained in $X^0$, i.e.\ the non-cylindrical components.
  \item The degrees $\deg_{q_j}u_k$ cannot decrease in the limit $k\to\infty$.
  \item The homology class represented by $u$ is $d[X']$.
  \item There do not exist any constant components of $u'$ and there exist no bubbled-off holomorphic spheres in $X'$ or interior nodes in $C^1,\ldots, C^N$.
  \item The map $u:C'\lra X'$ is a Hurwitz cover except that some nodal components of the target might be unstable.
  \item All critical points of $u$ are contained in $\bq$.
  \item There cannot exist any cylindrical components in $X'$, i.e.\ $X^*$ consists only of $X^0$.
  \item Every component of $C'$ is stable as a nodal surface.
  \item The sequence $[C_k,u_k,X_k,\bq_k,\bp_k]$ converges to $[C',u',X',\bq',\bp']$ in the sense of \cref{prop:equivalent-formulation-topology} and $[X',\bp']=[X,\bp]$.
\end{enumerate}

\begin{lem}
  The subsequence of source surfaces $C_k$ converges to $C'$ in the topology of Deligne--Mumford space including the marked points $\bq_k\lra\bq$.
  \label{lem:DM-convergence-C-k}
\end{lem}

\begin{proof}
  By \cref{def:convergence-broken-hol-curve} we have orientation preserving diffeomorphisms $\phi_k:C_k\lra \ol{C}$ such that $(\phi_k)_*j_k\lra j$ in $C^{\infty}_{\text{loc}}$ on $C^*\setminus\Delta_n$.
 Recalling the definitions of $C,\ol{C}$ and $C^*$ we see that this is precisely the definition of Deligne--Mumford convergence.
\end{proof}

\begin{lem}
  Any marked points $\bq$ of the limit curve $C'$ satisfy $u(q_j)=p_{\nu(j)}$ for the limit marked points $\bp$ in the Deligne--Mumford sequence $[X_k,\bp_k]$. Also all $q_j$ are contained in $C^0$ and all $p_i$ are contained in $X_0$.
\end{lem}

\begin{proof}
  As $u$ converges in $\cin_{\text{loc}}$ away from the nodes we have $u(q_j)=\lim_{k\to\infty}u(q_j^k)=\lim_{k\to\infty}p_{\nu(j)}^k=p_{\nu(j)}$. Recall that we chose the bi-collar neighborhoods for the neck-stretching construction as hyperbolic neighborhoods of a geodesic. By \cref{lem:collar-and-cusp-neighborhood-disjoint} such a neighborhood is disjoint from the cusp neighborhoods around the marked points $p_i$. Therefore the marked points $\bp$ stay a distance bounded from below away from the bi-collar neighborhoods for all $k$. Again by $\cin$ convergence of all the objects in the sequence away from the nodes and the bi-collar neighborhoods the limit marked points $\bp$ are outside the cylindrical components, i.e.\ in $X^0$. Therefore the preimages $\bq$ are contained in $C^0$. 
\end{proof}

\begin{lem}
  We have for all $j=1,\ldots,n$ that $\deg_{q_j}u\geq l_j$ where $\bq=\{q_j\}_{j=1}^n$.
  \label{lem:degrees-limit-map-u}
\end{lem}

\begin{proof}
  This follows from the $\cinl$-convergence of $u_k\circ \phi_k^{-1}$ to $u\circ\phi^{-1}$ away from the nodes. As $\deg_{q_j^k}u_k=l_j$ for all $k\in\NN$ we get that $\deg_{q_j}u\geq l_j$.
\end{proof}

\begin{rmk}
  In principle it is possible to obtain new critical points or increasing the order of the critical points in the limit as the derivative could converge to zero. However, we will see shortly that this can in fact not happen.
\end{rmk}

\begin{lem}
  The homology class represented by $u$ is $d[X']$.
  \label{lem:homology-class-limit-u}
\end{lem}

\begin{proof}
  From the uniform convergence statement in \cref{def:convergence-broken-hol-curve} we can deduce that $[u_k\circ\phi_k^{-1}]$ converges to $[u]$ in $H_2(X',\ZZ)$. Since the maps $u_k$ were branched coverings of degree $d$, we have $[u]=d[X']$.
\end{proof}

\begin{lem}
  There do not exist components of $C$ where $u$ is constant. Also any $C^{\nu}$ with $\nu=1,\ldots,N$ has no interior nodes by which we mean nodes that do not come from cylindrical ends. Furthermore there do not exist any bubbled-off spheres.
\end{lem}

\begin{proof}
  Recall from the proof of \cref{lem:nodal-extension-broken-holomorphic-curve} that the degree of $u'$ at a node in $C'$ is given by the degree of the Reeb orbit that $u'$ converges to in the corresponding cylindrical end in $C^{\ast}$. As the curves $u_k$ are Hurwitz covers and therefore do not contain constant components these degrees of the Reeb orbits are all non-zero. By uniform convergence they stay non-zero in the limit and thus $u'$ cannot be constant on any component that has a cylindrical end.

  Next we show that there can not be any interior nodes in the levels $C^1,\ldots,C^N$. Suppose there is a node $z\in C^{\nu}$. As this node could not have been present in the sequence $C_k$ already as it would be contained in $C^0$ in that case, there exist curves $\phi_k^{-1}(\gamma)$ for $\gamma\in\Delta_n$ on $C_k$ which are collapsed to nodes in the limit $k\to\infty$. However, the nodal Riemann surfaces $C_k$ and  $X_k$ are hyperbolic and $u_k$ is a local isometry. If the curves $\phi_k^{-1}(\gamma)$ are collapsing then the hyperbolic length of the unique geodesic in this free homotopy class converges to zero. But then the length of the image geodesic in $X_k$ also converges to zero which means that we included it in the choice of $M\subset X_0$. But then $z$ is not contained in the interior of $C^1,\ldots,C^N$.

  The same argument also works on $C^0$ except of course for nodes that are already present in the sequence $C_k$. If there was a constant component then by \cref{def:broken-holomorphic-curve} it would need to contain special points where by \cref{lem:degrees-limit-map-u} and \cref{lem:nodal-extension-broken-holomorphic-curve} the degree of $u$ is greater than or equal to one. This means there can not be a constant component. This also shows there can not be any holomorphic spheres which are attached to $C$ via interior nodes.
\end{proof}

\begin{lem}
  The limit curve $u$ is a branched nodal cover of degree $d$ over $X'$ (which might not be $X$ and could still contain non-stable components in form of twice-punctured spheres) and such that the degrees of $u$ at the two sides of each node agree. The Riemann--Hurwitz formula in form of \cref{cor:nodal-riemann-hurwitz} thus applies.
  \label{lem:limit-u-hurwitz}
\end{lem}

\begin{proof}
  This is clear from the last lemmata.
\end{proof}

\begin{lem}
  All critical points of $u$ are contained in the set $\bq$ and the degrees for points $q_j\in\bq$ satisfy $\deg_{q_j}u=l_j$.
  \label{lem:critical-points-u-limit}
\end{lem}

\begin{proof}
  Since every Hurwitz cover $u_k:C_k\lra X_k$ satisfies Riemann--Hurwitz we have
  \begin{equation*}
    2-2g_{\text{a}}(C_k) = d_k(2-2g_{\text{a}}(X_k) - \sum_{j=1}^k(l_j-1),
  \end{equation*}
  abusing notation by using $k$ twice for different things.  By \cref{lem:DM-convergence-C-k} and \cref{lem:homology-class-limit-u} we know that the arithmetic genera converge in the limit and that the degree of the limit map is still $d$. We therefore obtain
  \begin{equation*}
    2-2g_{\text{a}}(C') = d(2-2g_{\text{a}}(X')) - \sum_{j=1}^k(l_j-1)
  \end{equation*}
  as well as
  \begin{equation*}
    2-2g_{\text{a}}(C') = d(2-2g_{\text{a}}(X')) - \sum_{\substack{w\in C \\ \dd_wu=0}}(\deg_wu-1)
  \end{equation*}
  since by \cref{lem:limit-u-hurwitz} the limit map is indeed a branched nodal cover. Thus we obtain
  \begin{equation*}
    \sum_{j=1}^k(l_j-1) = \sum_{\substack{w\in C \\ \dd_wu=0}}(\deg_wu-1) = \sum_{q_j\in\bq}(\deg_{q_j}u -1) + \sum_{w\in C\setminus\bq}(\deg_wu-1).
  \end{equation*}
  By \cref{lem:degrees-limit-map-u} we see that in
  \begin{equation*}
    0 = \sum_{q_j\in\bq}(\deg_{q_j}u -l_j -1) + \sum_{w\in C\setminus\bq}(\deg_wu-1)
  \end{equation*}
  every summand is non-negative and thus zero, concluding the proof for the degrees of the critical points. Thus $\bq$ contains all critical points of $u'$.
\end{proof}

\begin{lem}
  The broken holomorphic curve $u^*:C^*\lra X^*$ does not contain any cylindrical components.
\end{lem}

\begin{proof}
  Again, since the chosen neighborhood $[-\epsilon,\epsilon]\times M$ is disjoint from cusp neighborhoods around marked points $\bp$ we see that cylindrical components of $X^*$ do not contain any marked points. Thus by \cref{lem:critical-points-u-limit} the restriction of $u^*$ to a connected component $D$ of the preimage of a cylindrical component $Y$ is an unbranched covering. Therefore its Euler characteristics is given by $\chi(D)=d(u|_D)\chi(Y)=0$ and it has at least two punctures corresponding to the ends converging to the Reeb orbits or the nodal points, respectively. Thus $D$ is a cylinder.

  If we switch to the nodal picture by \cref{lem:nodal-extension-broken-holomorphic-curve} we see that $u|_D$ is given by a degree-$d(u|_D)$ holomorphic map from the sphere to the sphere which has exactly two branched points which are completely branched. From this it follows that it is a trivial Reeb cylinder which can we assume does not exist by the Cieliebak--Mohnke compactness theorem \cref{thm:sft-compactness}. Thus $u^*$ does not contain any cylindrical component.
\end{proof}

\begin{lem}
  Every component of the Riemann surface $C$ is stable as a nodal surface.
\end{lem}

\begin{proof}
  Restrict $u$ to a given smooth component $D\lra Y$ with $D\subset C'$ and $Y\subset X'$ which has thus various special points. We have marked all critical points (actually all fibres of branched points) and the target component $Y$ is stable. Remove all marked and nodal points from $Y$ as well as all their preimages on $D$ and denote these punctured Riemann surfaces $\mathring{D}$ and $\mathring{Y}$. This way the restriction of $u$ becomes an actual covering and we have that the Euler characteristics of the punctured target surface is negative. But then $\chi(\mathring{D})=d(u|_D)\chi(\mathring{Y})<0$ and thus $D$ is stable.
\end{proof}

Recall \cref{prop:equivalent-formulation-topology} and that $|\mcM_{g,k,h,n}(T)|$ is second countable. From this string of lemmata we deduce the following statement.

\begin{thm}
  Any sequence $\{[C_k,u_k,X_k,\bq_k,\bp_k]\}_{k\in\NN}\subset|\mcR_{g,k,h,n}(T)|$ of Hurwitz covers contains a subsequence which converges in the sense of \cref{prop:equivalent-formulation-topology} to a Hurwitz cover $[C,u,X,\bq,\bp]\in|\mcR_{g,k,h,n}(T)|$. Thus the moduli space $|\mcR_{g,k,h,n}(T)|$ is sequentially compact and therefore also compact. As $\iota:|\mcM_{g,k,h,n}(T)|\lra|\mcR_{g,k,h,n}(T)|$ is proper the space $|\mcM_{g,k,h,n}(T)|$ is compact, too.
  \label{thm:compactness-moduli-space-hurwitz-covers}
\end{thm}

\begin{proof}
  First notice that there exist no additional nodal components in $X'$, i.e.\ $X'$ consists only of the level $X^0$. By construction of the neck-stretching sequence this implies that $(X',\bp')$ is biholomorphic to $(X,\bp)$. Thus we obtain a Hurwitz cover $(C,u,X,\bq,\bp)\in\Ob\mcR_{g,k,h,n}(T)$, where we have abused notation by suppressing the biholomorphism $(X',\bp')\lra(X,\bp)$ and we have renamed $C\coloneqq C'$.

  It remains to relate the convergence of the broken holomorphic curve to the convergence of Hurwitz covers in the sense of \cref{prop:equivalent-formulation-topology}. Note that $\ol{C}$ was obtained from $C^*$ by adding copies of $M$ at the cylindrical ends in order to obtain a topological surface. The nodal surface $C$ was obtained from $C^*$ by gluing in holomorphic discs centered around the cylindrical ends. This means that the diffeomorphisms $\phi_k:C_k\lra \ol{C}$ can be interpreted as diffeomorphisms $C_k\setminus\Gamma_k\lra C\setminus\Delta_{\text{p}}$ where $\Delta_{\text{p}}$ corresponds to the newly collapsed nodes and $\Gamma_k\coloneqq \phi_k^{-1}(\Delta_{\text{p}})$ is a disjoint union of simple non-intersecting curves. The same argument on the target surface yields diffeomorphisms $\rho_k;X_k\setminus \Theta_k\lra X\setminus\{\text{collapsed nodes}\}$ where $\Theta_k\coloneqq \rho_k^{-1}(M)$. Notice that we had rewritten the sequence $(X_k,\bp_k)$ as a neck-stretching sequence and the diffeomorphisms $\rho_k$ are given by these identifications.

  By construction of $X$ and by \cref{def:convergence-broken-hol-curve}, respectively, we have that $(\phi_k)_*j_k\lra j$ on $C\setminus\Delta_{\text{p}}$ and $(\rho_k)_*J_k\lra J$ on $X\setminus\{\text{collapsed nodes}\}$ for $k\lra\infty$. The statement in \cref{prop:equivalent-formulation-topology} on the convergence of the marked points is clear. Also the $\cinl$-convergence of the maps $\rho_k\circ u_k\circ \phi_k^{-1}\lra u$ as well as the uniform convergence follows immediately from \cref{def:convergence-broken-hol-curve}.
\end{proof}

\section{Compactness of \texorpdfstring{$\boldsymbol{\wt{\mcM}_{g,k,h,n}(T)}$}{M\_gkhn(T)}}

\label{sec:compactness-moduli-space-bordered-hurwitz-covers}

The moduli space of bordered Hurwitz covers is not compact as the hyperbolic lengths of the boundary geodesics can escape to infinity.\footnote{Notice that this phenomenon is already taken care of for interior geodesics as in a hyperbolic pair of pants a diverging length of a boundary geodesic forces some other geodesic to collapse. However, if the geodesic is exterior this collapsing geodesic is only a geodesic arc in the surface.} This is why we will formulate the compactness theorem differently in the form of a proper function. Define $K\coloneqq\prod_{i=1}^nK_i$ where
\begin{equation*}
  K_i\coloneqq\lcm\{l_j\mid j\in\nu^{-1}(i)\}
\end{equation*}
and the $l_j$ are the local degrees of $u$ at the points $q_j$.

\begin{definition}
  We define 
  \begin{align}
    \wt{\mu}:\obj\wt{\mcR}_{g,nd+k,h,2n}(\wt{T}) & \lra\RR^n \\ 
    (C,u,X,\bq,\bp,\Gamma) & \longmapsto\frac{1}{2}(K_1L_1^2,\ldots,K_nL_n^2),
    \label{eq:def-mom-map-dm}
  \end{align}
  where $L_i$ denotes the hyperbolic length of the geodesic in the free homotopy class of $\Gamma_i$, i.e.\ the one on the \emph{target} surface. If that curve is contractible then the length is defined to be zero. Note that there are only $n$ curves contained in the tuple $\Gamma$ as they are required to bound a sphere or a pair of pants with two marked points on $X$. Similarly we define
  \begin{align*}
    \wh{\mu}:\obj\wh{\mcR}_{g,k,h,n}(T) & \lra\RR^n \\ 
    (C,u,X,\bq,\bp,\bz) & \longmapsto\frac{1}{2}(K_1L_1^2,\ldots,K_nL_n^2),
  \end{align*}
  where $L_i$ is the hyperbolic length of the boundary geodesic $\del_iX$.
\end{definition}

\begin{lem}
  The maps $\wh{\mu}$ and $\wt{\mu}$ are equivariant and descend to continuous maps on the orbit spaces. The pull-back of $\wt{\mu}$ to $\Ob\wt{\mcM}_{g,nd+k,h,2n}(\wt{T})$ is an equivariant smooth map. They satisfy
  \begin{equation*}
    \wh{\mu}=\wt{\mu}\circ \glue.
  \end{equation*}
  \label{lem:properties-length-maps}
\end{lem}

\begin{rmk}
  Of course we could also use the tuple of geodesic lengths as a map but this particular form will be used later in \cref{sec:sympl-geom-moduli} as a momentum map for some Hamiltonian $T^n$-action. 
\end{rmk}

\begin{proof}
  It is clear that the maps are equivariant because a biholomoprhism gives an isometry between the corresponding hyperbolic surfaces and thus their boundaries have the same hyperbolic lengths. The map $\wt{\mu}$ is continuous because the topology on $|\wt{\mcR}_{g,nd+k,h,2n}(\wt{T})|$ was constructed from open sets $O^{\lambda}\subset\Ob\mcM_{g,nd+k,h,n}(\wt{T})$ where we required that the curves $\bG$ come from one such multicurve in the central fibre of the family $\Psi^{\lambda}(O^{\lambda})$, see \cref{sec:orbifold-structure-wtmcr}. But then it is clear that $\wt{\mu}$ is continuous because the hyperbolic metrics depend continuously on the point in $O^{\lambda}$ and for a fixed homotopy class the length of the unique geodesic in this class then also depends continuously on the parameter in $O^{\lambda}$.

  As we build the map $\glue:\obj\wh{\mcR}_{g,k,h,n}(T)\lra\obj\wt{\mcR}_{g,nd+k,h,2n}(\wt{T})$ by gluing along the hyperbolic boundary geodesic it is clear that $\wh{\mu}=\wt{\mu}\circ \glue$ holds as we measure the lengths of the same geodesics. Since $\glue$ is continuous we get that $\wh{\mu}$ is continuous. Their pull-back or rather restriction to $\Ob\wt{\mcM}_{g,nd+k,h,2n}(\wt{T})$ is clearly still equivariant. It is smooth in the Fenchel--Nielsen differentiable structure as the geodesic length functions are smooth coordinate functions.

\end{proof}

Before proving the properness of these maps we will need a slight extension of the Deligne--Mumford compactness theorem. Consider a sequence $(X_k,\bp_k,\bG_k)$ where $X_k$ is a closed connected nodal stable oriented Riemann surface of genus $h$ with $2n$ marked points $\bp_k$ and a $n$-multicurve $\bG_k$ such that $(\bG_k)_i$ bounds a disc with $(\bp_k)_{2i-1}$ and $(\bp_k)_{2i}$ or it bounds a disc with a nodal sphere with those two marked points on it.\footnote{This is of course precisely the same data we have on the target surface of an element in $\Ob\wt{\mcR}_{g,nd+k,h,2n}(\wt{T})$.} Two such triples will be called equivalent if there exists a biholomorphism mapping the corresponding $\bp_k$ and $\bG_k$ onto each other keeping the enumeration. Note that if an element in $\bG_k$ bounds a disc with two marked points then we can consider its unique hyperbolic representative in the uniformized metric such that the marked points are cusps. We will refer to this particular curve which has a well-defined length without mentioning it explicitly. If the curve is contractible then the length is of course zero.

\begin{prop}[Hummel]
   Given a sequence $(X_k,\bp_k,\bG_k)$ as above such that the lengths of the curves in $\bG_k$ converge to some lengths $(L_1,\ldots,L_n)$, which might be zero, there exists a subsequence converging to some closed connected nodal stable oriented Riemann surface in the DM-sense, see \cref{prop:target-j-subsequence}, such that
   \begin{enumerate}[label=(\roman*), ref=(\roman*)]
     \item the geodesic representatives in $\bG_k$ whose lengths do not converge to zero do not intersect the geodesic representatives of the collapsing curves $\Theta_k\subset X_k$ and
     \item additionally the diffeomorphisms $\rho_k:X_k\setminus \bigcup\Theta_k\lra X\setminus\{\text{collapsed nodes}\}$ are such that every hyperbolic geodesic representative of $\rho_k(\bG_k)\subset X$ is independent of $k$.
   \end{enumerate}
   \label{prop:dm-improved}
\end{prop}

\begin{proof}
  Since the curves $\Theta_k$ collapse the lengths of the geodesic representatives become arbitrarily close to zero. Therefore the width of their collar neighborhoods becomes arbitrarily large and thus the length of the geodesics intersecting the curves in $\Theta_k$ becomes large. By assumption the lengths of the curves in $\bG_k$ converges and thus stays bounded. Of course if their lengths converge to zero then they are included in the set $\Theta_k$. This is also Corollary~4.1.2 in \cite{buser_geometry_2010}. In particular we see that the collapsing of curves happens away from the non-collapsing $\bG_k$.

  For the second statement we need to look at the proof of Deligne--Mumford compactness in \cite{hummel_gromovs_1997} in more detail. First notice that if the length of a curve in $\bG_k$ actually converges to zero then the second statement is already contained in the standard formulation of Deligne--Mumford compactness as the $\rho_k$ would map this curve in $\bG_k$ to a fixed node in $X$. So suppose a curve $\gamma_k\in\bG_k$ is such that its length converges to some non-zero limit.\footnote{Notice that Hummel cuts the surfaces $X_k$ first along the collapsing curves and considers these hyperbolic surfaces with collapsing boundaries and no internal collapsing geodesics separately. However, this does not change anything.} In the very first step Hummel chooses a pair of pants decomposition $(Y_k^1,\ldots,Y_k^N)$ of $(X_k,\bp_k)$\footnote{Our marked points are cusps in the hyperbolic surfaces.} by using Bers theorem to assume that the lengths of the boundary components of the pairs of pants are bounded independent of $k$ and then noting that the number of such pairs of pants decompositions is finite. Thus he chooses a subsequence in order to be able to assume that there exist diffeomorphisms $\chi_k:X_1\lra X_k$ with $\chi_k(Y_1^i)=Y_k^i$. He then proceeds to perturb these maps slightly and apply Whitneys extension theorem in order to obtain the diffeomorphisms $\rho_k$ or rather their inverses. Notice that because the length of our curve $\gamma_k$ converges it is in particular bounded and we can include the pair of pants bounded by $\gamma_k$ and the two marked points in our choice of the pair of pants decomposition $\{Y_k^i\}_{i=1,\ldots,N}$. By the construction in the proof of Proposition~5.1. in \cite{hummel_gromovs_1997} we see that the resulting maps $\rho_k:X_k\lra X$ map the pair of pants bounded by $\gamma_k$ and $(\bp_k)_{2i-1}$ and $(\bp_k)_{2i}$ to a fixed pair of pants and thus the free homotopy class of $\gamma_k$ is mapped to a free homotopy class on $X$ independent of $k$.
\end{proof}

\begin{thm}
  The maps $\wh{\mu}$ and $\wt{\mu}$ are proper.
  \label{thm:properness-mu}
\end{thm}

\begin{proof}
  Recall that we have the commuting diagram of maps
  \begin{equation}
    \xymatrix{
      |\wh{\mcR}_{g,k,h,n}(T)| \ar[dr]_{\wh{\mu}} \ar[r]^-{\glue} & |\wt{\mcR}_{g,nd+k,h,2n}(\wt{T})| \ar[d]^{\wt{\mu}} & |\wt{\mcM}_{g,nd+k,h,2n}(\wt{T})| \ar[dl]^{\wt{\mu}} \ar[l]_{\iota} \\
      & \RR^n &
      }.
  \end{equation}
  Furthermore we know that all spaces are locally compact by \cref{lem:prop-top-moduli-space-bordered-hurwitz-covers} and because $\Ob\wt{\mcM}_{g,nd+k,h,2n}(\wt{T})\lra|\wt{\mcM}_{g,nd+k,h,2n}(\wt{T})|\lra |\wt{\mcR}_{g,nd+k,h,2n}(\wt{T})|$ are all open quotient maps. Furthermore $\glue$ is proper and so we only need to show that the $\wt{\mu}$'s are proper as therefore $\wh{\mu}$ is proper, too.

  It remains to show that $\wt{\mu}:|\wt{\mcR}_{g,nd+k,h,2n}(\wt{T})|\lra\RR^n$ is proper. Consider a compact subset $K\subset\RR^n$. We need to show that any sequence in the preimage of $K$ under $\wt{\mu}:|\wt{\mcR}_{g,nd+k,h,2n}(\wt{T})|\lra \RR^n$ has a convergent subsequence. So consider some arbitrary sequence $[C_k,u_k,X_k,\bq_k,\bp_k,\Gamma_k]\in|\wt{\mcR}_{g,nd+k,h,2n}(\wt{T})|$ with $\wt{\mu}([C_k,u_k,X_k,\bq_k,\bp_k,\Gamma_k])\in K$. As $K$ is compact there exists a subsequence such that the hyperbolic lengths of the curves in $\Gamma_k$ converge to $(L_1,\ldots,L_n)$. Denote this new subsequence again by $[C_k,u_k,X_k,\bq_k,\bp_k,\Gamma_k]$.

  Now we apply \cref{prop:dm-improved} to find a subsequence such that the sequence of target surfaces $(X_k,\bp_k,\bG_k)$ converges in Deligne--Mumford space to a surface $(X,\bp,\bG)$ with the diffeomorphisms $\rho_k$ mapping $\bG_k$ to $\bG$. Now we can apply the SFT-compactness theorem in the same way as in \cref{sec:apply-sft-compactness-hurwitz-covers} and \cref{sec:relate-broken-holomorphic-curves-hurwitz-covers} to this sequence. This way we obtain a subsequence $[C_k,u_k,X_k,\bq_k,\bp_k,\Gamma_k]$ converging to some $[C,u,X,\bq,\bp]$ in the sense of \cref{sec:topology-moduli-spaces-closed-hurwitz-covers} with the additional property that the $\rho_k(\bG_k)$ is constant on $X$. It is clear from the definition of the topology on $\Ob\wt{\mcM}_{g,nd+k,h,2n}(\wt{T})$ which induces the topology on $|\wt{\mcR}_{g,nd+k,h,2n}(\wt{T})|$ as well as \cref{prop:equivalent-formulation-topology} that the subsequence $[C_k,u_k,X_k,\bq_k,\bp_k,\Gamma_k]$ converges to $[C,u,X,\bq,\bp,\bG]$ in $|\wt{\mcR}_{g,nd+k,h,2n}(\wt{T})|$.
\end{proof}

\chapter{Symplectic Geometry on the Moduli Space of Bordered Hurwitz Covers}

\label{sec:sympl-geom-moduli}

This section describes the symplectic geometry on the various introduced moduli spaces in order to apply Duistermaat--Heckman for obtaining relations between $\Psi$-intersections on Deligne--Mumford space and Hurwitz numbers.

Recall that we have the following diagram of maps involving moduli spaces. Also recall that these spaces are the orbifold versions, so in particular moduli spaces of Hurwitz covers are not the real moduli spaces because we ignored some lower-dimensional strata of morphisms.

\begin{equation}
  \label{eq:diag-mod-spaces}
  \xymatrix{ & \mcM_{g,k,h,n}(T)  \ar[rd]^{\ev}  \ar[ld]_{\fgt} & & \wt{\mcM}_{g,nd+k,h,2n}(\wt{T}) \ar@{~>}[ll] \ar[d]^{\wt{\ev}}   & &  \wh{\mcM}^{\square}_{g,k,h,n}(T)  \ar[ll]_{\glue} \ar[d]^{\wh{\ev}}  \\
    \mcM_{g,k} & & \mcM_{h,n}  &  \wt{\mcM}_{h,2n} \ar@{~>}[l]  & & \wh{\mcM}^{\square}_{h,n} \ar[ll]_{\sglue}    \\  
    }
\end{equation}

Recall that the index $\square$ means that the boundaries are non-degenerate but we do not impose any condition on interior geodesics. In contrast the index $\circ$ refers to completely smooth Hurwitz covers. In \cref{eq:diag-mod-spaces} we have the usual correspondence between Deligne--Mumford spaces and moduli spaces of Hurwitz covers, discussed in \cref{sec:main-results}. In particular $\fgt$ is a smooth map and $\ev$ is a branched morphism covering meaning it is in particular a morphism covering outside the nodal covers. In the middle column we have the orbifolds for the moduli spaces of glued Hurwitz covers as well as the moduli space of glued admissible surfaces, see \cref{sec:moduli-spaces-admissible-riemann-surfaces}. Note that the curved lines are not maps but refer to a symplectic quotient construction which we will work out in the next section. On $|\wt{\mcM}_{g,k,h,n}(T)|$ and $|\wt{\mcM}_{h,n}|$ we will consider symplectic structures together with Hamiltonian $T^n$-actions such that the symplectic quotients correspond to moduli spaces of Hurwitz covers and Riemann surfaces, respectively, of fixed boundary length. This will allow us to use Duistermaat--Heckman to formulate relations between certain cohomology classes on $|\mcM_{g,k,h,n}(T)|$ as well as on $|\mcM_{h,n}|$. The corresponding situation for Deligne--Mumford spaces was established in \cite{mirzakhani_weil-petersson_2007}. The right hand side corresponds to the moduli spaces of smooth bordered Hurwitz covers and Riemann surfaces. Its corresponding actual moduli space $\wh{\mcR}_{g,k,h,n}(T)$ is a space where the $T^n$-action will be free but which is unfortunately not an orbifold, see \cref{sec:pull-back-orbifold-structure-via-glue}.

\section{Symplectic Structures and Torus Actions}

Recall from \cref{sec:weil-peterss-sympl} that we have the Weil--Petersson symplectic structure on the Deligne--Mumford orbifold $\mcM_{g,k}$ and by construction on the moduli space of admissible hyperbolic surfaces $\wt{\mcM}_{h,2n}$, as this orbifold is locally diffeomorphic to $\mcM_{h,2n}$. Pulling back via $\ev$, $\sglue\circ\wt{\ev}$ and $\sglue$ we obtain thus symplectic orbifolds

\begin{equation*}
  (\wh{\mcM}^{\square}_{h,n},\wwp), (\mcM_{g,k,h,n}(T),\omega)\text{ and }(\wh{\mcM}_{g,k,h,n}(T),\wt{\omega}).
\end{equation*}

Also we have $T^n$-actions on the orbifolds $\wh{\mcM}^{\square}_{g,k,h,n}(T)$ and $\wh{\mcM}^{\square}_{h,n}$ by rotating the marked points $\bz$ close to the boundaries and cusps. However, we need to be a bit more careful regarding their parametrization, so let us give more details in the following.

We define the $T^n$-action on $\wh{\mcM}_{h,n}$ by defining it on the objects and morphisms, see \cref{def:group-action-on-orbifold} for a definition. So let

\begin{equation*}
  (C,\bz)\in\obj\wh{\mcM}_{h,n}
\end{equation*}

be a point  and define $\rot_t(\bz)\in \Gamma(C)$ for $t\in[0,1)^n$ as follows.\footnote{By $\Gamma(C)$ we denote the tuple of reference curves in $C$ which are the preimages of the reference curve in $X$. Compare with \cref{sec:orb-structure-mod-space-bordered-hurwitz-covers-definitions} and \cref{sec:reference-curve}.} First uniformize $C$ such that all special points are nodes and such that all boundaries are geodesics. In this metric the points $z_i$ lie on reference curves $\Gamma_i(C)$ close to the boundary $\del_iC$, see \cref{sec:reference-curve} for general comments and \cref{sec:orb-structure-mod-space-bordered-hurwitz-covers-definitions} for the precise condition such that preimages of reference curves on $X$ are reference curves on $C$. Now parametrize these curves from $[0,1]$ proportional to arc length: $\beta_j:[0,1]\lra \Gamma_j(C)$ such that $\beta_j(0)=z_j$ and define $\rot_t(\bz)$ with $\rot_t(\bz)_j\coloneqq \beta_j(t_jl_jF(l(\del_{\nu(j)}X)))$, where $F:\RR_{\geq 0}\lra\RR_{>0}$ is a function fixed beforehand from \cref{sec:orb-structure-mod-space-bordered-hurwitz-covers-definitions} and thus $l_jF(l(\del_{\nu(j)}X)$ is the hyperbolic length of $\Gamma_j(C)$. This defines a torus action as follows.

\begin{definition}
  We define an action of $T^n\coloneqq \faktor{\RR^n}{\ZZ^n}$ on $\wh{\mcM}_{h,n}$ by
  \begin{align*} 
	[0,1]^n\times \obj\wh{\mcM}_{h,n} & \lra \obj\wh{\mcM}_{h,n} \\
	(t,(C,\bz)) &\longmapsto (C,\rot_t(\bz))
  \end{align*}
  and correspondingly
  \begin{align*} 
	[0,1]^n\times \Mor\wh{\mcM}_{h,n} & \lra \Mor\wh{\mcM}_{h,n} \\
	(t,\Phi\in \Hom((C,\bz),(C',\bz')) &\longmapsto \Phi\in\Hom((C,\rot_t(\bz)),(C',\rot_t(\bz'))).
  \end{align*}
  \label{def:torus-action-DM-spaces}
\end{definition}

\begin{lem}
  The $T^n$-action in \cref{def:torus-action-DM-spaces} is a well-defined torus action on the groupoid $\wh{\mcM}_{h,n}$ and it is smooth on the orbifold $\wh{\mcM}^{\square}_{h,n}$ of admissible Riemann surfaces such that the boundaries are non-degenerate.
  \label{lem:torus-action-moduli-space-admissible-surfaces}
\end{lem}

\begin{rmk}
  Recall from \cref{def:group-action-on-orbifold} that a $T^n$-action on an orbifold $\mcG$ is a smooth group homomorphism $T^n\lra\Isom(\mcG)$ where $\Isom(\mcG)$ are invertible homomorphisms and smooth means that the maps $T^n\times\Ob\mcG\lra\Ob\mcG$ and $T^n\times\Mor\mcG\lra\Mor\mcG$ are smooth.
\end{rmk}

\begin{proof}
  By definition of $\rot_t(\bz)$ we see that $\rot_1(\bz)=\bz$ and thus we obtain indeed group actions by the torus $T^n$ on $\Ob\wh{\mcM}_{h,n}$ and $\Mor\wh{\mcM}_{h,n}$. Notice that $\Phi:C\lra C'$ is a biholomorphism and therefore an isometry for the hyperbolic metrics. This means that from $\Phi(\bz)=\bz'$  follows $\Phi(\rot_t(\bz))=\rot_t(\bz')$.

  As the rotation of the marked points is obviously reversible the functor $\wh{\mcM}_{h,n}\lra\wh{\mcM}_{h,n}$ defined by $\rot_t$ for $t\in T^n$ is indeed a category isomorphism. Furthermore it is smooth because we pulled back the manifold structure from $\wt{\mcM}_{h,2n}$ via the functor $\sglue$. This means that in local coordinates the map $\rot_t$ is given by Fenchel--Nielsen twists of length $t_j\in[0,1]$ around the corresponding geodesics $\gamma_j$ corresponding to the actual boundary components $\del_jC$ in the glued surface $\wt{C}$. By \cite{wolf_real_1992-1} this is smooth in the Fenchel--Nielsen differentiable structure on $\wh{\mcM}_{h,n}$ and in the Deligne--Mumford differentiable structure on $\wh{\mcM}^{\square}_{h,n}$.
\end{proof}

In a similar way we can define $\rot_t(C,X,u,\bp,\bq,\bz)$ for $(C,X,u,\bp,\bq,\bz)\in\obj\wh{\mcM}^{\square}_{g,k,h,n}$ and $t\in[0,1]^n$. The only modification we need is to rescale the speeds accordingly because we need to adhere to the condition $u(z_i)=u(z_j)$ for $i,j$ such that $\nu(i)=\nu(j)$. To this end define $K_i= \lcm\{l_j|\nu(j)=i\}$.

\begin{definition}
  We define a $T^n$-action on $\wh{\mcM}^{\square}_{g,k,h,n}$ by
  \begin{align} 
    [0,1]^n\times \obj\wh{\mcM}^{\square}_{g,k,h,n}(T) & \lra \obj\wh{\mcM}^{\square}_{g,k,h,n}(T) \label{eq:def-skew-torus-action} \\
    (t,(C,u,X,\bq,\bp,\bz)) &\longmapsto (C,u,X,\bq,\bp,\rot_{t'}(\bz)), \nonumber
  \end{align}
  and correspondingly
  \begin{align*} 
    [0,1]^n\times \Mor\wh{\mcM}^{\square}_{g,k,h,n}(T) & \lra \Mor\wh{\mcM}^{\square}_{g,k,h,n}(T) \\
    (t,(\Phi,\varphi))\in \Hom((C,u,X,\bq,\bp,\bz),(C',u',X',\bq',\bp',\bz')) &\longmapsto \\
    \Phi\in\Hom((C,u,X,\bq,\bp,\rot_{t'}(\bz)), & (C',u',X',\bq',\bp',\rot_{t'}(\bz')))
  \end{align*}
  with
  \begin{equation*}
    t'\coloneqq\left(\frac{K_{\nu(1)}}{l_1}t_{\nu(1)}, \cdots , \frac{K_{\nu(k)}}{l_k}t_{\nu(k)}\right)\in \RR^k.
    \label{eq:def-skew-torus}
  \end{equation*}
  \label{def:torus-action-moduli-spaces}
\end{definition}

% t'\coloneqq\left(\frac{\mfL}{l_1}\sum_{\substack{1\leq i\leq n\\ \nu(1)=i}}t_i, \cdots , \frac{\mfL}{l_k}\sum_{\substack{1\leq i\leq n\\ \nu(k)=i}}t_i\right)\in \RR^k.

\begin{prop}
  The action defined by \cref{eq:def-skew-torus-action} is a well defined smooth torus action on the orbifold $\wh{\mcM}^{\square}_{g,k,h,n}$. 
\end{prop}

\begin{proof}
  As the orbifold structure on $\wh{\mcM}^{\square}_{g,k,h,n}$ comes locally from a submanifold of the source moduli space of the glued Hurwitz cover it is clear that the action is smooth. Also it is clear as in the proof of \cref{lem:torus-action-moduli-space-admissible-surfaces} that $\Phi$ and $\varphi$ are isometries and thus $(\Phi,\varphi)$ is a morphism between the twisted Hurwitz covers. Furthermore for every $t$ the resulting functor is an isomorphism of categories as we can just rotate backwards.

  However, we need to show that with the definition of $t'$ in terms of $t$ this is indeed a torus action and that the condition $u(z_j)=u(z_i)$ for any $i,j$ such that $\nu(i)=\nu(j)$ is preserved as this was a condition for the marked points in $\wh{\mcM}_{g,k,h,n}(T)$. First notice that if $t=1$ then $t'\in\NN^n$ as $l_j\mid K_{\nu(j)}$ for every $j=1,\ldots,k$ and therefore $\rot_{t'}=\id$ and thus we have indeed a torus action. Secondly every $\del_jC$ covers $\del_{\nu(j)}X$ with the degree $l_j$ which by choice of the reference curves remains true for those. As $u:C\lra X$ is holomorphic and thus a local isometry for the hyperbolic metrics it preserves the hyperbolic lengths. Parametrize $\del_{\nu(j)}X$ proportional to arc length with the map $\beta:[0,1]\lra\del_{\nu(j)}X$ such that $\beta(0)=u(z_j)=u(z_i)$. Then

  \begin{align*}
    u(\rot_{t'}(z_j)) & =u\left(\beta_j\left(\frac{K_{\nu(j)}}{l_j}t_{\nu(j)}\right)\right)=\beta\left(l_j\cdot\left(\frac{K_{\nu(j)}}{l_j}t_{\nu(j)}\right)\right) \\
    & =\beta\left(K_{\nu(j)} t_{\nu(j)}\right)=u\left(\beta_i\left(\frac{K_{\nu(j)}}{l_i}t_{\nu(i)}\right)\right)=u(\rot_{t'}(z_i))
  \end{align*}

  for all $1\leq i,j\leq k$ such that $\nu(i)=\nu(j)$.
\end{proof}

We now have symplectic structures and torus actions on our moduli spaces and we want to show that they are indeed Hamiltonian. Recall from \cref{sec:compactness-moduli-space-bordered-hurwitz-covers} that we have defined a map $\wh{\mu}:\wh{\mcM}_{g,k,h,n}(T)\lra\RR^n$. Correspondingly we make the following definition.

\begin{definition}
  On $\wh{\mcM}_{h,n}$ we define
  \begin{align*}
    \mu:\obj\wh{\mcM}_{h,n} & \lra\RR^n \\ 
    (X,\bp) & \longmapsto\frac{1}{2}(L_1^2,\ldots,L_n^2), 
  \end{align*}
  where $L_i$ is the hyperbolic length of the boundary geodesic $\del_iX$.
  \label{def:mom-map-dm}
\end{definition}

\begin{rmk}
  Of course the proofs of \cref{lem:properties-length-maps} and \cref{thm:properness-mu} apply to $\mu$ as well and we see that it is a proper continuous map on $\wh{\mcM}_{h,n}$. Additionally it is smooth in the Fenchel--Nielsen coordinates and smooth in Deligne--Mumford coordinates on $\wh{\mcM}_{h,n}^{\square}$.
\end{rmk}

\begin{lem}
  The maps $\wh{\mu}$ and $\mu$ are momentum maps for the $T^n$-actions on $\wh{\mcM}^{\square}_{h,n}$ and $\wh{\mcM}^{\square}_{g,k,h,n}(T)$ defined above and the $T^n$-actions are Hamiltonian.
  \label{lem:moment-maps}
\end{lem}

\begin{proof}
  First notice that both maps are smooth if the boundaries are non-degenerate and by \cref{thm:properness-mu} the maps are proper on orbit spaces. As tori are Abelian the equivariance condition of $\wh{\mu}$ and $\mu$ becomes an invariance condition under the group action which is satisfied as by changing the marked points the hyperbolic length of the boundary geodesic does not change. Thus it remains to prove invariance of $\omega_{\Ob}$ under the $T^n$-action and $\dd\langle\wh{\mu},X\rangle =-i_{\ul{X}}\omega_{\Ob}$ for all $X\in\mft$.

  Recall that the symplectic form is the pull-back of $\wwp$ under the gluing map. Thus at a point $(C,u,X,\bq,\bp,\bz)\in\Ob\wh{\mcM}^{\square}_{g,k,h,n}(T)$ we can choose Fenchel--Nielsen coordinates $(\tau_i,L_i)$ including the curves $\wt{\bG}$ obtained after gluing. In these coordinates $\wwp$ has standard symplectic form with respect to the length and twist coordinates around these curves. So in particular $\omega\left(\frac{\del}{\del \tau_i},\cdot\right)=\dd\left(\frac{L_i^2}{2}\right)$. Also we invariance of $\wwp$ is well known from \cite{wolpert_weil-petersson_1985}. This proves that $\wh{\mcM}_{h,n}^{\square}$ is Hamiltonian. The calculations for $\wh{\omega}$ and $\wh{\mu}$ on $\wh{\mcM}^{\square}_{g,k,h,n}(T)$ are very similar. However, notice that
  \begin{equation*}
    \ul{e_i}=\sum_{j=1}^k\frac{K_i}{l_j}\frac{\del}{\del z_j},
  \end{equation*}
  where $\frac{\del}{\del z_j}$ is the vector field corresponding to a rotation of the marked point $z_j$. As $\wh{\omega}=\wh{\ev}^*\wwp$ we can calculate
  \begin{equation*}
    \wh{\omega}\left(\ul{e_i},\cdot\right)=\wwp\left(K_i\frac{\del}{\del \tau_i}\right)=\dd\left(\frac{K_iL_i^2}{2}\right)
  \end{equation*}
  showing the statement for the $T^n$-action on $\wh{\mcM}^{\square}_{g,k,h,n}(T)$. The equivariance follows in the same way as on $\wh{\mcM}_{h,n}^{\square}$.
\end{proof}

Note that we have some more $T^n$-actions as explained in the following statements.

\begin{lem}
  The $T^n$-action above on $\wh{\mcM}^{\square}_{g,k,h,n}(T)$ defines a continuous free action on $\wh{\mcR}_{g,k,h,n}(T)$ where free means that it acts freely on objects.
\end{lem}

\begin{proof}
  Take the same definition as in \cref{def:torus-action-moduli-spaces} on objects and on morphisms in $\wh{\mcR}_{g,k,h,n}(T)$. It is clearly continuous with respect to the topology as defined in \cref{lem:prop-top-moduli-space-bordered-hurwitz-covers}. However we can not say that it is smooth because $\wh{\mcR}_{g,k,h,n}(T)$ is not an orbifold groupoid. It is also clearly free on objects.
\end{proof}

\begin{definition}
  We define a $T^n$-action on $\wt{\mcM}_{g,nd+k,h,2n}(\wt{T})$ as follows. Consider
  \begin{equation*}
    (C,u,X,\bq,\bp,\bG)\in\Ob\wt{\mcM}_{g,nd+k,h,2n}(\wt{T}).
  \end{equation*}
  Keep the map $u$ as it is and for every non-contractible curve $\gamma_i\in\bG$ change the hyperbolic structure on $(X,\bp)$ by doing a Fenchel--Nielsen twist by $2\pi t_i$ in the positive direction. If the curve $\gamma_i$ is contractible then do nothing. On morphisms we define the action in the same way as in \cref{eq:def-skew-torus} and \cref{def:torus-action-moduli-spaces} by leaving the morphism map identical but changing the domain and target by doing the same Fenchel--Nielsen twists.
  \label{def:torus-action-tilde-mcm}
\end{definition}

\begin{prop}
  The \cref{def:torus-action-tilde-mcm} defines a smooth Hamiltonian torus action on $\wt{\mcM}_{g,nd+k,h,2n}(\wt{T})$. The number of contractible curves in $\bG$ is the dimension of the subtorus fixing the point $(C,u,X,\bq,\bp,\bG)$. The functor
  \begin{equation*}
    \glue:\wh{\mcM}^{\square}_{g,k,h,n}(T)\lra\wt{\mcM}_{g,nd+k,h,2n}(\wt{T})
  \end{equation*}
  is equivariant.
\end{prop}

\begin{proof}
  The equivariance of $\glue$ is clear from the definitions of the torus actions, in fact we have already used that argument in the proof of \cref{lem:moment-maps}. The same arguments as before show that this action is indeed a smooth Hamiltonian torus action on $\wt{\mcM}_{g,nd+k,h,2n}(\wt{T})$. At any contractible curve in $\bG$ the torus action does nothing so this corresponds to a $S^1$ fixing the point.
\end{proof}

Note that we can now do symplectic reduction on the orbifolds $(\wh{\mcM}^{\square}_{h,n},\wwp,\mu,T^n)$ and $(\wh{\mcM}^{\square}_{g,k,h,n}(T),\wh{\omega},\wh{\mu},T^n)$ at some points $\frac{1}{2}(L^2_1,\ldots,L_n^2)\in\RR_{>0}^n$ and $\frac{1}{2}(K_1L^2_1,\ldots,K_nL_n^2)\in\RR_{>0}^n$, respectively, as in \cref{thm:symplectic-reduction}.\footnote{Obviously we chose the values at which to do symplectic reduction such that the lengths of the boundaries agree as we have different normalizations of the momentum maps. We will denote the quotient always with an additional $[L]$ specifying these lengths.} Next we need to make sure that the evaluation maps and the symplectic reduction behave well with respect to these reductions.

\begin{lem}
  For the symplectic orbifolds
  \begin{equation*}
    (\wh{\mcM}^{\square}_{h,n},\wwp,\mu,T^n)\text{ and }(\wh{\mcM}^{\square}_{g,k,h,n}(T),\wh{\omega},\wh{\mu},T^n)
  \end{equation*}
  and the branched morphism covering $\wh{\ev}:\wh{\mcM}^{\square}_{g,k,h,n}\lra\wh{\mcM}^{\square}_{h,n}$ we have
  \begin{enumerate}[label=(\roman*), ref=(\roman*)]
  \item $\wh{\ev}$ descends to a branched morphism covering $\wh{\ev}$ of the symplectic quotients $\wh{\mcM}^{\square}_{g,k,h,n}(T)[L]$ and $\wh{\mcM}^{\square}_{h,n}[L]$ of the same degree and
  \item $\wh{\omega}_L=\wh{\ev}^*{\wwp}|_L$.
  \end{enumerate}
  \label{lem:ham-orb-covering-reduction}
\end{lem}

\begin{proof}
  First notice that the diagram of functors between orbifold categories
  \begin{equation}
    \xymatrix{
      \wh{\mcM}^{\square}_{g,k,h,n}(T) \ar[r]^{\wh{\ev}} & \wh{\mcM}^{\square}_{h,n} \\
      \wh{\mcM}^{\square}_{g,k,h,n}(T)|_{\wh{\mu}^{-1}(L)} \ar[r]^{\wh{\ev}} \ar@{^{(}->}[u] \ar[r]^{\wh{\ev}} \ar@{->>}[d] & \wh{\mcM}^{\square}_{h,n}|_{\mu^{-1}(L)} \ar@{^{(}->}[u] \ar@{->>}[d] \\ 
      \wh{\mcM}^{\square}_{g,k,h,n}(T)[L] \ar[r]^{\wh{\ev}} & \wh{\mcM}^{\square}_{h,n}[L]
      }
    \label{eq:diagram-evaluation-maps-symplectic-quotient}
  \end{equation}
  exists and commutes. The vertical arrows correspond to full subcategories and quotients by the torus action as defined in \cref{thm:symplectic-reduction}, respectively. The horizontal functor on the bottom is well-defined.

  Next we show that this functor $\wh{\ev}:\wh{\mcM}^{\square}_{g,k,h,n}(T)[L]\lra\wh{\mcM}^{\square}_{h,n}[L]$ is a branched morphism covering and has the same degree as the original $\wh{\ev}$. So consider some $(C,u,X,\bq,\bp,\bz)\in\Ob\wh{\mcM}^{\square}_{g,k,h,n}(T)$. Recall from the proof of \cref{thm:ev-local-structure} that we chose coordinate neighborhoods on the Teichmüller space of the smooth components of $X$ as well as discs for parametrizing the opening of all the nodes in the source surface $C$. We will now choose the coordinates on the Teichmüller factor such that they include the Fenchel--Nielsen coordinates around the curves $\bG$ from $\glue(C,u,X,\bq,\bp,\bz)\in \Ob\wt{\mcM}_{g,nd+k,h,2n}(\wt{T})$. As these curves are not collapsed by assumption the twist and length functions are smooth and we can complete this multicurve to a maximal one in order to obtain coordinates. Also we will use the same Fenchel--Nielsen coordinates around $\wh{\ev}(C,u,X,\bq,\bp,\bz)\in\Ob\wh{\mcM}_{h,n}$. By \cref{thm:ev-local-structure} the map $\wh{\ev}$ then looks like

  \begin{align*}
    \DD^{3k-3-N} \times \prod_{i=1}^n\left([L_i-\epsilon,L_i+\epsilon]\times S^1\right) \times \DD^N & \lra\DD^{3k-3-N} \times \prod_{i=1}^n\left([L_i-\epsilon,L_i+\epsilon]\times S^1\right) \times\DD^N \\
    (x,l_1,\theta_1,\ldots,l_n,\theta_n,z_1,\ldots,z_N) & \longmapsto \left(x,l_1,\theta_1,\ldots,l_n,\theta_n,z_1^{K_1},\ldots,z_N^{K_N}\right),
  \end{align*}

  where the $N$ discs correspond to \emph{interior} nodes. Now restricting to $l_i=L_i$ for all $i=1,\ldots,L_n$ we obtain a slice for the torus action corresponding to the Fenchel--Nielsen twists generated by the $\frac{\del}{\del \theta_i}$ by picking one fixed twist coordinate for every $i=1,\ldots,n$. First we see that the resulting map $\wh{\ev}:\wh{\mcM}^{\square}_{g,k,h,n}(T)[L]\lra\wh{\mcM}^{\square}_{h,n}[L]$ is still smooth and has the correct form in a neighborhood of a nodal Hurwitz cover. So it remains to show that $\wh{\ev}$ is a morphism covering on the (completely) smooth part of the moduli spaces and that the degree is the same as the one of

  \begin{equation*}
    \wh{\ev}:\wh{\mcM}^{\square}_{g,k,h,n}(T)\lra \wh{\mcM}^{\square}_{h,n}.
  \end{equation*}
  
  To this end recall \cref{eq:diagram-evaluation-maps-symplectic-quotient}. It is clear that restricting $\wh{\ev}$ to the full subcategory $\wh{\mcM}^{\square}_{g,k,h,n}(T)|_{\wh{\mu}^{-1}(L)}$ remains a covering on objects and morphisms as this is the preimage of $\wh{\mcM}^{\square}_{h,n}|_{\mu^{-1}(L)}$ under the functor $\wh{\ev}$. Furthermore taking a quotient with respect to the free $T^n$-action on both objects and morphisms gives a locally trivializable bundle showing that $\wh{\ev}$ is still a covering on the quotient $\wh{\mcM}^{\square}_{g,k,h,n}(T)[L]$. We need to show the lifting property on the full subcategory of completely smooth Hurwitz covers $\wh{\mcM}^{\circ}_{g,k,h,n}(T)[L]$.

  The lifting property is easily verified from the properties of the functor $\ev$. Consider $[C,u,X,\bq,\bp,\bz]\in\Ob\wh{\mcM}^{\circ}_{h,n}[L]$ and $\varphi:(X,\bp)\lra(X',\bp')\in\Mor\wh{\mcM}_{h,n}[L]$.\footnote{Note that taking the quotient with respect to the torus action on $\Ob\wh{\mcM}^{\square}_{h,n}$ is the same as forgetting the marked point. On $\Ob\wh{\mcM}^{\square}_{g,k,h,n}(T)$ however, the equivalence class still remembers some information about the relative differences between choices of the marked points $\bz$ in the fibre over $u(\bz)$. This issue will lead to \cref{lem:degrees-evaluation-map-hurwitz}.} There exists a morphism $(\id,\varphi):(C,u,X,\bq,\bp,\bz)\lra(C,\varphi\square u,\bq,\varphi(\bp),\bz)$ in $\wh{\mcM}^{\square}_{g,k,h,n}(T)$ which is mapped to $\varphi$ under $\wh{\ev}$ and $s(\id,\varphi)=(C,u,X,\bq,\bp)$ and gives a morphism $[C,u,X,\bq,\bp,\bz]\lra[C,\varphi\square u,\bq,\varphi(\bp),\bz]$ in $\wh{\mcM}^{\square}_{g,k,h,n}(T)[L]$ with the same properties.

  This shows that

  \begin{equation*}
    \wh{\ev}:\wh{\mcM}^{\circ}_{g,k,h,n}(T)[L]\lra \wh{\mcM}^{\circ}_{h,n}[L]
  \end{equation*}

  is a morphism covering and therefore

  \begin{equation*}
    \wh{\ev}:\wh{\mcM}^{\square}_{g,k,h,n}(T)[L]\lra \wh{\mcM}^{\square}_{h,n}[L]
  \end{equation*}

  is a branched morphism covering. Regarding the degree we can see that it does not change when passing to the symplectic quotient as follows. The restriction to the full subcategory $\wh{\mcM}^{\square}_{g,k,h,n}(T)|_{\wh{\mu}^{-1}(L)}$ does not change the degree. By taking quotients with respect to the free $S^1$-action we do not change the number of preimages of an object in $\Ob\mcM_{h,n}$ or their automorphisms as the evaluation map is also equivariant.

  It remains to calculate the symplectic volume forms. However, $\wh{\omega}_L$ is the unique form on $\Ob\wh{\mcM}^{\square}_{g,k,h,n}(T)[L]$ such that its pullback under $\pi:\Ob\wh{\mcM}^{\square}_{g,k,h,n}(T)|_{\wh{\mu}^{-1}(L)}\lra\Ob\wh{\mcM}^{\square}_{g,k,h,n}(T)[L]$ is equal to $\wh{\omega}|_{\wh{\mu}^{-1}(L)}$. By \cref{eq:diagram-evaluation-maps-symplectic-quotient} we have

  \begin{equation*}
    \pi^*\wh{\ev}^*\wwp|_L=\wh{\ev}^*\pi^*\wwp|_L=\wh{\ev}^*\wwp|_{\mu^{-1}(L)}=(\wh{\ev}^*\wwp)|_{\wh{\mu}^{-1}(L)}=\wh{\omega}|_{\wh{\mu}^{-1}(L)}
  \end{equation*}

  and therefore

  \begin{equation*}
    \wh{\omega}_L=\wh{\ev}^*\wwp|_L.
  \end{equation*}
\end{proof}

\section{Bundles and \texorpdfstring{$\boldsymbol{\Psi}$}{Psi}-classes}

\subsection{Complex Orbifold Vector Bundles over the Moduli Space of Closed Hurwitz Covers}

In this section we define two complex orbifold vector bundles $E$ and $F$ which are direct sums of complex orbifold line bundles and whose first Chern classes will be needed later. These Chern classes are certain kinds of $\Psi$-classes over moduli spaces.\footnote{We say ``kind'' of $\Psi$-classes because there are different notions due to compactification differences of Deligne--Mumford space and moduli spaces of $J$-holomorphic curves. However, as it will later turn out, due to the fact that we mark enough points, these will actually coincide in our case.} Recall from \cref{sec:orb-bundles-coverings} that a complex orbifold vector bundle over $\mcM_{g,k,h,n}(T)$ is a manifold $E$ together with an anchor $\pi_E:E\lra\obj\mcM_{g,k,h,n}(T)$ that is a complex vector bundle as well as an action $\mu_E:\Mor\mcM_{g,k,h,n}(T){}_s\times_{\pi} E\lra E$ satisfying the usual group action axioms and which acts linearly on the fibres.

Thus we make the following definition.

\begin{definition}
  Define the complex orbifold vector bundle $E$ over $\mcM_{g,k,h,n}(T)$ by
  \begin{align}
    E & \coloneqq \bigsqcup_{\lambda\in\Lambda}O^{\lambda}\times\CC^k \nonumber \\
    \pi_E: E & \lra \obj\mcM_{g,k,h,n}(t) \nonumber \\
    (b,z) & \longmapsto b \nonumber \\
    \mu_E: \Mor\mcM_{g,k,h,n}(T){_s\times_{\pi_E}} E & \lra E \nonumber \\
    (\underbrace{(b,(\Phi,\varphi),c)}_{\in M(\lambda,\lambda')}, \underbrace{(b,z)}_{\in O^{\lambda}\times \CC^k}) & \longmapsto \left(c,\left(\left(\D_0\rho_j^{c}\right)^{-1}\circ \D_{q^b_j}\Phi\circ \D_0\rho_j^{b}(z_j)\right)_{j=1}^k\right), \label{eq:transition-function-E}
  \end{align}
  where the map $\rho_j^{b}:\DD\lra C^b$ is defined below.
  \label{def:vector-bundle-E}
\end{definition}

In order to define the chart $\rho_j^b$ notice that $b\in O^{\lambda}$ contains a choice of decomposition of $C^b$ into pairs of pants as well as a uniformized hyperbolic metric on $C^b$ such that the marked points $q_j$ and cusps are nodes. In turn this hyperbolic metric defines a unique collar curve $\Gamma_j(C^b)$ with the length $F(0)$ close to $q_j$, as was described e.g.\ in \cref{sec:orb-structure-mod-space-bordered-hurwitz-covers-definitions}. Notice that there are exactly two hyperbolic geodesics $\gamma$ on this pair of pants going up the cusp $q_j$ and perpendicular on the geodesic boundary components or going up these cusps if they happen to be degenerate. For each of these curves there exists a unique point $m\in\Gamma_j(C)\cap\gamma$. For both choices there exists a unique biholomorphism $\rho_j^b:\DD\lra C^b$ onto the disc bounded by $\Gamma_j(C)$ such that $\rho_j^b(p_j)=0$ and $\rho_j^b(m)=1$. This is illustrated in \cref{fig:definition-chart-phi}.

Notice that we can make this choice for $\rho_j^c$ and $\rho_j^b$ independently, so we fix this by requiring that $\Phi(m^b)=m^c$ where $m^b$ is the intersection point of $\Gamma_j(C^b)$ with the chosen geodesic going up the cusp $q_j$ which is perpendicular to one of the boundary components of the hyperbolic pair of pants.
	
\begin{figure}[h!]
  \centering
  \def\svgwidth{\textwidth}
  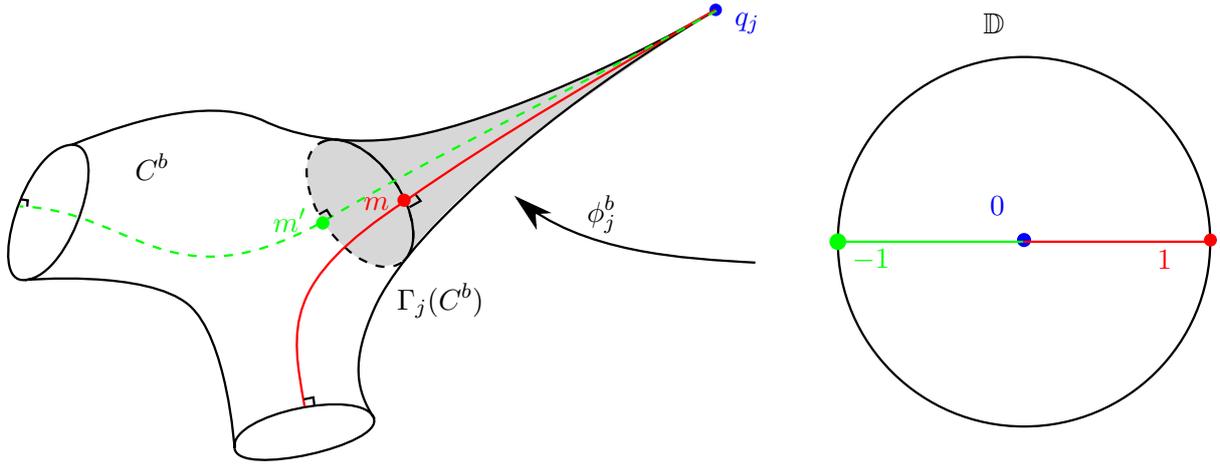
  \caption{This figure illustrates the definition of the map $\rho_j^b:\DD\lra C^b$ used in \cref{def:vector-bundle-E}.}
  \label{fig:definition-chart-phi}
\end{figure}

\begin{lem}
  $E$ as defined in \cref{def:vector-bundle-E} is a well-defined complex orbifold vector bundle and a direct sum of complex orbifold line bundles.
  \label{lem:E-orbibundle}
\end{lem}

\begin{proof}
  Before proving the properties of a complex orbifold vector bundle we need to make sure that the action of $\mcG$ on $E$ stated in \cref{def:vector-bundle-E} is well defined. The issue is that we have defined two different charts $\rho_j^b$ coming from the two possible choices of boundary geodesics of the pair of pants specified by $\lambda$. Note that in $\DD$ the real axis connecting $0$ and $1$ is a fixed point set of an anti-holomorphic involution and therefore geodesic, see \cref{sec:embedding-moduli-space-boundary}. Recall that a hyperbolic pair of pants can be cut into two isometric hyperbolic hexagons by cutting along the unique geodesics perpendicular to the pairs of boundary components (or going up the cusp). The isometry is a reflection with respect to these cut geodesics which include the two possible choices $\gamma$ and $\gamma'$ from the definition. This shows that that on $\DD$ the image of the other possible choice of $m'=\Gamma_j(C)\cap\gamma'$ is the point $-1$. This in turn proves that $\ol{\rho}_j^b(z)=\rho_j^b(-z)$ for $z\in\DD$ is the other chart. But then we see that

  \begin{align*}
    \left(\D_0\ol{\rho}_j^{c}\right)^{-1}\circ \D_{q^b_j}\Phi\circ \D_0\ol{\rho}_j^{b}(z_j) & =-\left(\D_0\rho_j^{c}\right)^{-1}\circ \D_{q^b_j}\Phi\circ \D_0\rho_j^{b}(-z_j) \\
    & = \left(\D_0\rho_j^{c}\right)^{-1}\circ \D_{q^b_j}\Phi\circ \D_0\rho_j^{b}(z_j)
  \end{align*}

  which shows that the map $\mu_E$ does not depend on that choice.

  Next we show that $E$ is indeed a complex orbifold vector bundle. First we notice that $\pi_E:E\lra\Ob\mcM_{g,k,h,n}(T)$ is indeed a complex vector bundle and the action $\Mor\mcM_{g,k,h,n}(T)$ on the fibres is indeed $\CC$-linear. It remains to show that $\mu_E:\Mor\mcM_{g,k,h,n}(T){_s\times_{\pi}} E \lra E$ is smooth and satisfies $\pi_E(g\cdot e)=t(g)$, $\id_x\cdot e=e$ and $g\cdot(h\cdot e)=(gh)\cdot e$ for all suitable $e\in E$ and $g,h\in\Mor\mcM_{g,k,h,n}(T)$.

  From \cref{def:vector-bundle-E} it is clear that
  \begin{equation*}
    \pi_E(\mu_E((b,(\Phi,\phi),c),z))=\pi_E(c,z')=c=t(b,(\Phi,\phi),c)
  \end{equation*}
  where $z'$ is an abbreviation for the long expression in \cref{eq:transition-function-E}. Secondly we have $\id_b\cdot e=e$ as in \cref{eq:transition-function-E} $\Phi$ is given by the identity.  And thirdly we have $g\cdot(h\cdot e)=(gh)\cdot e$ because this is true for the underlying morphisms on $\mcM_{g,k,h,n}(T)$ and in the $\CC$-part the two middle charts cancel out. This remains true even if the two maps are defined with different choices for the map $\rho_j^c$ as we obtain always an even number of factors of $-1$.

  The smoothness of $\mu_E$ can be shown as follows. As the map $\mu_E$ is linear on fibres it depends smoothly on that coordinate. For the dependence on the complex structure represented by the point $b\in O^{\lambda}$ we notice first that the uniformized hyperbolic metric depends smoothly on $b$ in a neighborhood of a cusp $p_j$. Since the curve $\Gamma_j(C)$ is determined by a length condition as well as the solution to a first-order ordinary differential equation whose coefficients depend smoothly on $b$ we see that the curves $\Gamma_j(C)$ depend smoothly on $b$. We choose a sufficiently large coordinate chart around $p_j$ and consider everything in $\CC$. The curves $\Gamma_j(C)$ bound contractible neighborhoods of $0$ which might depend on $b$. The chart $\rho_j^b$ is then the unique biholomorphism from this neighborhood on the unit disc after rotating the neighborhoods such that the point $m$ which might depend smoothly on $b$ points e.g.\ in the positive real direction. This way we obtain a smooth family of Jordan curves in $\CC$ such that $0$ is contained in the interior and which are all contractible. In \cite{caratheodory_untersuchungen_1912} it is shown that for one parameter families of such neighborhoods of zero with a small additional assumption the corresponding Riemann mappings converge in $\cin$ on compact sets to the Riemann mapping of the limit neighborhood.\footnote{The assumption is that the sequence of domains bounded by the curves converges versus its so-called core. This core is the largest domain $G$ such that any closed domain contained in $\mathring{G}$ is contained in all but finitely many members of the sequence of domains. Convergence means that any subsequence has the same core. This is satisfied if the boundary curves converge.} This shows that the map

  \begin{equation*}
    b\mapsto \left(\left(\D_0\rho_j^{c}\right)^{-1}\circ \D_{q^b_j}\Phi\circ \D_0\rho_j^{b}(z_j)\right)_{j=1}^k
  \end{equation*}
  
  is $\cin$ in all directions and thus a smooth map of $b\in O^{\lambda}$ which is used for the charts of $M(\lambda,\lambda')$.

  The bundle $E$ is a direct sum of complex orbifold line bundles because the bundle $\pi_E:E\lra\Ob\mcM_{g,k,h,n}(T)$ is a trivial bundle and the action $\mu_E$ acts diagonally on the fibres.
\end{proof}

We do need another similar bundle called $F$. This bundle will be essentially the bundle of $\Psi$-classes for the marked points on the target, so in particular it will have rank $n$ instead of rank $k$.

\begin{definition}
  Define the complex orbifold vector bundle $F$ over $\mcM_{g,k,h,n}(T)$ by
  \begin{align*}
    F & \coloneqq \bigsqcup_{\lambda\in\Lambda}O^{\lambda}\times\CC^n  \\
    \pi_F: F & \lra \obj\mcM_{g,k,h,n}(t)  \\
    (b,z) & \longmapsto b  \\
    \mu_F: \Mor\mcM_{g,k,h,n}(T){_s\times_{\pi_F}} F & \lra F \\
    (\underbrace{(b,(\Phi,\varphi),c)}_{\in M(\lambda,\lambda')}, \underbrace{(b,z)}_{\in O^{\lambda}\times \CC^n}) & \longmapsto \left(c,\left(\left(\D_0\eta_i^{c}\right)^{-1}\circ \D_{p^b_i}\varphi\circ \D_0\eta_i^{b}(z_i)\right)_{i=1}^n\right),
  \end{align*}
  where the map $\eta_i^{b}:\DD\lra X^b$ is defined in the same way as before only that we use the reference curve $\Gamma_i(X^b)$ and the pair of pants decomposition contained in $\lambda$ for the target surface.
  \label{def:vector-bundle-F}
\end{definition}

\begin{lem}
  $F$ as defined in \cref{def:vector-bundle-F} is indeed a well-defined complex orbifold vector bundle and a direct sum of complex orbifold line bundles.
\end{lem}

\begin{proof}
  The proof is exactly the same as the proof of \cref{lem:E-orbibundle}.
\end{proof}

\index{Good Vector Bundle}

\begin{rmk}
 Unfortunately it turns out that the orbifold vector bundle $E$ can be a bad bundle (i.e.\ not a good one), see \cref{def:good-vector-bundle} for a definition. The problem is as follows: Suppose we have an automorphism $(\Phi,\id):(C,u,X,\bq,\bp)\lra(C,u,X,\bq,\bp)$ on the completely smooth Hurwitz cover $x=(C,u,X,\bq,\bp)$. Then according to \cref{sec:constr-an-orbif} a neighborhood of this point $x\in\Ob\mcM_{g,k,h,n}(T)$ is described by a neighborhood of $X$ in its Teichmüller space whose complex structures are then pulled-back via $u$ to complex structures on $C$. It is then clear that $(\Phi,\id)$ also gives an automorphism of all the points $x'$ close to $x$ in $\Ob\mcM_{g,k,h,n}(T)$. This is because we did not modify the maps $u$ or $\Phi$ and thus everything still commutes and $u$ is holomorphic by construction.

 Thus $(\Phi,\id)$ is contained in the ineffective morphisms $\Mor_{\text{ineff}}\mcM_{g,k,h,n}(T)$ and it is also an automorphism of $x$ and therefore it needs to act trivially on the fibres of $E$ if the bundle was good. Looking at the definition in \cref{def:vector-bundle-E} we see that the action of $(\Phi,\phi)\in\Aut_{\mcM_{g,k,h,n}(T)}(x)$ on $\CC$ is given by multiplication with the so-called rotation number of $\Phi$ at $z_j$ in the $j$-th component. It is clear that there exist morphisms with rotation numbers different from the identity, as any non-trivial automorphism needs to have a non-trivial rotation number and from examples we can see that most cases for Hurwitz numbers do have automorphisms. Therefore these automorphisms exist on most moduli spaces and the bundle $E$ will be bad.

 Note that this does not contradict the fact that every stable Deligne--Mumford orbifold except $\mcM_{2,0}$ is reduced as the complex structures of $C$ admitting Hurwitz covers are actual submanifolds of the corresponding object manifolds of the source Deligne--Mumford orbifold. This means that it is possible for an automorphism to be effective on a neighborhood of the whole Deligne--Mumford space but when restricted to the constructed submanifolds it acts as an automorphism everywhere in an open set. This is precisely what happens here.

 Furthermore note that the bundle $F$ is actually good. The reason is that for the target surface we use a neighborhood of $X$ in the full target universal unfolding. But since only $\mcM_{2,0}$ and $\mcM_{1,1}$ are non-reduced there are generally no ineffective morphisms on the target Deligne--Mumford orbifold. The special case $h=2$ and $n=0$ is of no concern to us as in that case the bundles do not even exist. This means that if we have an automorphism $(\Phi,\phi)\in\Aut_{\mcM_{g,k,h,n}(T)}(C,u,X,\bq,\bp)$ then it never extends to automorphisms in a neighborhood of $(\Phi,\phi)$ except if $\phi=\id$. But in that case it acts trivially on the fibre of the bundle $F$.
  \label{rmk:orbibundles-e-f-bad}
\end{rmk}

\begin{rmk}
  The vector bundles $E$ and $F$ define vectors of Chern classes $c_1(E)\in H^2(|\mcM_{g,k,h,n}(T)|,\QQ)^k$ and $c_1(F)\in H^2(|\mcM_{g,k,h,n}(T)|,\QQ)^n$ as was explained in \cref{rmk:chern-classes-orbibundles}. In the case of $F$ this is done directly via Chern--Weil theory on the bundle whereas in the case of $E$ we have to go through some extension of the bundle to a good one as was described in \cite{seaton_characteristic_2007} and in \cref{rmk:chern-classes-orbibundles}. Note that in our case we do not have to pass to the vertical tangent bundle because $E$ is naturally the pull-back of a $\CC^k$-orbifold bundle on $\mcM_{g,k}$ which is reduced for all parameter cases of interest except $g=1$ and $k=1$. Thus this extension is a good bundle and the Chern class of $E$ is given by restricting the Chern--Weil Chern class of this good bundle over $\mcM_{g,k}$. In any case we can do our calculations as if the Chern class was defined via Chern--Weil theory. The local calculations just need to be done on the extended bundle and then pulled-back afterwards.
  \label{rmk:chern-classes-e-f}
\end{rmk}

Notice that there exists of course also a corresponding bundle on Deligne--Mumford space.

\begin{definition}
  We denote by $\LL_i^*\lra\Ob\mcM_{h,n}$ the \emph{$i$-th tangent line bundle} over Deligne--Mumford space $\mcM_{h,n}$ with objects and morphisms defined by universal unfoldings $U^{\lambda}$ for $\lambda\in\Lambda$ which were used in the construction for the $O^{\lambda}$ for the moduli space of Hurwitz covers. It is defined by
  \begin{align*}
    \LL_i^* & \coloneqq \bigsqcup_{\lambda\in\Lambda}U^{\lambda}\times\CC  \\
    \pi_{\LL_i^*}: \LL_i^* & \lra \obj\mcM_{h,n}  \\
    (b,z) & \longmapsto b  \\
    \mu_{\LL_i^*}: \Mor\mcM_{h,n}{_s\times_{\pi_{\LL_i^*}}} \LL_i^* & \lra \LL_i^* \\
    ((b,\varphi,c), (b,z)) & \longmapsto \left(c,\left(\D_0\rho_i^{c}\right)^{-1}\circ \D_{p^b_i}\varphi\circ \D_0\rho_i^{b}(z_i)\right),
  \end{align*}
  where the map $\rho_i^b$ is defined as above at the marked point $p_i$ of $X$ at $b=(X,\bp)\in\Ob\mcM_{h,n}$.  
\end{definition}

\begin{rmk}
Note that with this transition map this line bundle is in fact the dual of the usual line bundle which has the cotangent space at a marked point as a fibre. It is a good bundle as again $\mcM_{h,n}$ is reduced for all interesting cases of parameters and the fact that it is an orbifold bundle at all can be proven exactly in the same way as for \cref{lem:E-orbibundle}.
\end{rmk}

\begin{definition}
  The $i$-th $\psi$-class on $\mcM_{h,n}$ is defined as
  \begin{equation*}
    \psi_i\coloneqq c_1(\LL_i)\in H^2(|\mcM_{h,n}|,\QQ)
  \end{equation*}
  for $\LL_i\lra\Ob\mcM_{h,n}$.
  \label{def:psi-classes}
\end{definition}

\subsection{Relating Various \texorpdfstring{$\boldsymbol{\Psi}$}{Psi}-Classes}

\begin{lem}
  The complex orbifold vector bundles $E$ and $\bigoplus_{j=1}^k\fgt^*\LL_j$ are isomorphic as complex orbifold vector bundles. Furthermore $F$ and $\bigoplus_{i=1}^n\ev^*\LL_i$ are isomorphic as complex orbifold vector bundles.
  \label{lem:pulled-back-line-bundles}
\end{lem}

\begin{proof}
  Recall that $\ev:\mcM_{g,k,h,n}(T)\lra\mcM_{h,n}$ and $\fgt:\mcM_{g,k,h,n}(T)$ are homomorphisms and we can therefore pull back the orbibundles, see \cref{def:pull-back-orbibundle} for a definition and some discussion. In particular the pulled-back bundles are given by
  \begin{align*}
    \ev^*\LL_i & = \Ob\mcM_{g,k,h,n}(T){_{\ev}\times_{\pi}}\LL_i, \\
    \mu_{\ev^*\LL_i}((\Phi,\phi),(C,u,X,\bq,\bp),(X,\bp,z)) & = ((C',u',X',\bq',\bp'),(X',\bp',z')), \\
    \fgt^*\LL_j & = \Ob\mcM_{g,k,h,n}(T){_{\fgt}\times_{\pi}}\LL_j, \\
    \mu_{\fgt^*\LL_j}((\Phi,\phi),(C,u,X,\bq,\bp),(C,\bq,w)) & = ((C',u',X',\bq',\bp'),(C',\bq',w')) ,
  \end{align*}
  where $(\Phi,\phi):(C,u,X,\bq,\bp)\lra(C',u',X',\bq',\bp')$ is a morphism in $\mcM_{g,k,h,n}(T)$ and $z'$ and $w'$ are given by
  \begin{align*}
    z' & =\left(\D_0\eta_i^{X'}\right)^{-1}\circ \D_{p^X_i}\varphi\circ \D_0\eta_i^{X}(z) \text{ and } \\
    w' & =\left(\D_0\rho_j^{C'}\right)^{-1}\circ \D_{q^C_i}\varphi\circ \D_0\rho_j^{C}(w).
  \end{align*}
  Note that by definition of the orbifold atlases and the evaluation functor we have that the data in $\lambda$ is preserved and thus the maps $\eta_i$ and $\rho_j$ on $\mcM_{g,k,h,n}(T), \mcM_{g,k}$ and $\mcM_{h,n}$ correspond to each other. Therefore the obvious maps
  \begin{align*}
    \oplus_{i=1}^n\ev^*\LL_i & \lra F \\
    ((C,u,X,\bq,\bp),((X,\bp,z_i))_{i=1}^n) & \longmapsto (C,u,X,\bq,\bp,\bz) \text{ and} \\
    \oplus_{j=1}^k\fgt^*\LL_j & \lra E \\
    ((C,u,X,\bq,\bp),((X,\bp,w_j))_{j=1}^k) & \longmapsto (C,u,X,\bq,\bp,\bw)
  \end{align*}
  are clearly $\mcM_{g,k,h,n}(T)$-maps and smooth bijective fibre maps and therefore orbibundle isomorphisms.
\end{proof}

Now we will relate the Chern classes of the two orbifold vector bundles $E$ and $F$ over $\mcM_{g,k,h,n}(T)$. We want to express $c_1(F)$ in terms of $c_1(E)$ in $H^2(\mcM_{g,k,h,n}(T),\QQ)$. We will discuss in \cref{rmk:chern-classes-on-bundles} that we can also just look at the torus subbundles and compare their vectors of first Chern classes.

Relating Chern classes of $E$ and $F$ will be done in three steps: First we switch to the summand-wise unit-vector subbundles $\mathring{E}$ and $\mathring{F}$ which are torus principal bundles and have the same first Chern classes. Then we define a $n$-dimensional torus subbundle of $\mathring{E}$ corresponding to those unit tangent vectors at the $q_j$ which correspond to the same unit tangent vectors under $u$. And thirdly we see that there is a map from this subbundle to $\mathring{F}$ by just pushing forward these points via the Hurwitz cover $u$. This map will be a fibrewise covering of every circle summand. Note that we cannot describe this as a subbundle of $E$ by identifying tangent vectors that are mapped to each other under $u$ because the $q_j$ might be critical points.

\begin{definition}
  Consider a Hurwitz cover $(C,u,X,\bq,\bp)\in O^{\lambda}\subset\mcM_{g,k,h,n}(T)$. Then for every $j=1,\ldots,k$ we can define the map $\rho_j:\DD\lra C$ and $\eta_{\nu(j)}:\DD\lra X$ as in \cref{def:vector-bundle-E}. Now we define maps $\wt{\rho}_j$ and $\wt{\eta}_{\nu(j)}$ from $S^1\subset\CC$ to $\Gamma_j(C)$ and $\Gamma_{\nu(j)}(X)$, respectively. By $\D_0\rho_j:\CC\lra\ts_{q_j}C$ we obtain a tangent vector to $q_j$ for every point in $S^1$. This tangent vector in turn corresponds to a unique hyperbolic geodesic going up the cusp $q_j$ which intersects $\Gamma_j(C)$ perpendicularly. Now $\wt{\rho}_j$ maps a point in $S^1$ to this intersection point. The same construction gives a map $\wt{\eta}_{\nu(j)}:S^1\lra \Gamma_{\nu(j)}(X)$. These maps are obviously homeomorphisms.
  \label{def:s1-parametrizations}
\end{definition}

\begin{definition}
  We define the torus bundles $\mathring{E}$ and $\mathring{F}$ in the same way as in \cref{def:vector-bundle-E} and \cref{def:vector-bundle-F} but using $(S^1)^k\subset\CC^k$ as a fibre.
\end{definition}

\begin{lem}
  The subsets $\mathring{E}$ and $\mathring{F}$ are indeed $T^k$- and $T^n$-orbifold principal bundles, respectively.
  \label{lem:torus-subbundles-are-bundles}
\end{lem}

\begin{proof}
  We need to check that the $\mu_E$- and $\mu_F$-actions restrict to the $(S^1)^k$-subbundles. This is because for $(\Phi,\varphi):b\lra c$ the map $\rho_j^c\circ \Phi\circ\rho_j^b:\CC\lra\CC$ is a biholomorphism of the disc fixing zero and therefore a rotation. But then its differential at zero is also a rotation thus preserving $S^1$. This argument works for $\eta_i$, too, of course.
\end{proof}

\begin{definition}
  Define the subset $\mathring{E}'\subset \mathring{E}$ as follows. A point $(b,z)\in O^{\lambda}\times (S^1)^k\subset \mathring{E}$ with $b=(C,u,X,\bq,\bp)$ and $z=(z_1,\ldots,z_k)\in (S^1)^k$ is contained in $E'$ if and only if
  \begin{equation}
    (\wt{\eta}_{\nu(j)})^{-1}\circ u\circ\wt{\rho}_j(z_j)= (\wt{\eta}_{\nu(i)})^{-1}\circ u \circ\wt{\rho}_i(z_i)
    \label{eq:condition-subbundle-e-prime}
  \end{equation}
  for all $i,j\in\{1,\ldots,k\}$ such that $\nu(i)=\nu(j)$.
  \label{def:e-prime-mathring}
\end{definition}

\begin{prop}
  The subset $\mathring{E}'$ is a $T^n$-principal orbibundle over the orbifold groupoid
  \begin{equation*}
    \faktor{\mcM_{g,k,h,n}(T)\ltimes \mathring{E}'}{T^n}
  \end{equation*}
  with respect to the following $T^n$-action on $\mathring{E}'$. We define $K_i\coloneqq \lcm\{l_j\mid \nu(j)=i\}$ for $i=1,\ldots,n$ and set
  \begin{align*}
    T^n & \stackrel{\iota}{\lra} T^k \\
  (e^{2\pi\ii \theta_1},\ldots,e^{2\pi\ii \theta_n}) & \mapsto\left(e^{2\pi\ii \frac{K_{\nu(1)}}{l_1} \theta_{\nu(1)}},\ldots,e^{2\pi\ii \frac{K_{\nu(k)}}{l_k} \theta_{\nu(k)}}\right).
  \end{align*}
  This way, $T^n$ acts on $\mathring{E}'$ and on $\mcM_{g,k,h,n}(T)\ltimes \mathring{E}'$. In particular this action is free on $\mcM_{g,k,h,n}(T)\ltimes \mathring{E}'$. Also this orbifold principal torus bundle $\mathring{E}'$ is a sum of circle bundles.
  \label{prop:E-prime-E-subbundle}
\end{prop}

\begin{proof}
  We need to check the following things:
  \begin{enumerate}[label=(\roman*), ref=(\roman*)]
    \item $\mathring{E}'$ is a $\mcM_{g,k,h,n}(T)$-space, \label{item:e-prime-g-space}
    \item $T^n$ acts smoothly and freely on $\mathring{E}'$ and \label{item:smooth-free-tn-action}
    \item $\mathring{E}'\lra\Ob \faktor{\mcM_{g,k,h,n}(T)\ltimes \mathring{E}'}{T^n}$ is a principal $T^n$-bundle splitting as a sum of $S^1$-bundles. \label{item:principal-tn-orbibundle}
  \end{enumerate}
  To see \cref{item:e-prime-g-space} we calculate for
  \begin{equation*}
    g=(\Phi,\varphi):b\coloneqq(C,u,X,\bq,\bp)\lra c\coloneqq(C',u',X',\bq',\bp')
  \end{equation*}
  whether it preserves the condition \cref{eq:condition-subbundle-e-prime}. Notice that from the argument in the proof of \cref{lem:torus-subbundles-are-bundles} we see that $\rho_j^c\circ \Phi\circ\rho_j^b:\CC\lra\CC$ is a rotation and thus its tangent map equals the map under the obvious identification $\ts_0\CC\cong\CC$. As all involved maps are holomorphic and thus preserve angles this rotation agrees with $(\wt{\rho}_j^c)^{-1}\circ\Phi\circ\wt{\rho}_j^b:S^1\lra S^1$. In the same way we have $(\wt{\eta}_{\nu(j)}^c)^{-1}\circ\varphi\circ\wt{\eta}_{\nu(j)}^b=(\D_0\eta_{\nu(j)}^c)^{-1}\circ\D_{p_{\nu(j)}^b}\varphi\circ\D_0\eta_{\nu(j)}^b(z_j)$ on $S^1$.
  \begin{align*}
    (\wt{\eta}^c_{\nu(j)})^{-1}\circ u'\circ\wt{\rho}^c_j(g\cdot z_j) & = (\wt{\eta}^c_{\nu(j)})^{-1}\circ u' \circ\wt{\rho}^c_j\circ (\D_0\rho_j^c)^{-1}\circ\D_{q_j^b}\Phi\circ\D_0\rho_j^b(z_j) \\
    & = (\wt{\eta}^c_{\nu(j)})^{-1}\circ u' \circ\wt{\rho}^c_j\circ (\wt{\rho}_j^c)^{-1}\circ\Phi\circ \wt{\rho}_j^b(z_j) \\
    & = (\wt{\eta}^c_{\nu(j)})^{-1}\circ \varphi \circ u \circ \wt{\rho}_j^b(z_j) \\
    & = (\wt{\eta}^c_{\nu(j)})^{-1}\circ \varphi \circ \wt{\eta}_{\nu(j)}^b \circ (\wt{\eta}_{\nu(j)}^b)^{-1} \circ u \circ \wt{\rho}_j^b(z_j) \\
    & = (\wt{\eta}^c_{\nu(i)})^{-1}\circ \varphi \circ \wt{\eta}_{\nu(i)}^b \circ (\wt{\eta}_{\nu(i)})^{-1}\circ u \circ\wt{\rho}_i(z_i) \\
    & = (\wt{\eta}^c_{\nu(i)})^{-1}\circ u' \circ \Phi \circ \wt{\rho}_i^b(z_i) \\
    & = (\wt{\eta}^c_{\nu(i)})^{-1}\circ u' \circ \wt{\rho}^c_i \circ (\wt{\rho}_i^c)^{-1} \circ \Phi \circ \wt{\rho}_i^b(z_i) \\
    & = (\wt{\eta}^c_{\nu(i)})^{-1}\circ u' \circ\wt{\rho}^c_i\circ (\D_0\rho_i^c)^{-1}\circ\D_{q_i^b}\Phi\circ\D_0\rho_i^b(z_i) \\
    & = (\wt{\eta}^c_{\nu(i)})^{-1}\circ u'\circ\wt{\rho}^c_i(g\cdot z_i),
  \end{align*}
  where we have used $\nu(j)=\nu(i)$. The compatibility conditions of the $\mcM_{g,k,h,n}(T)$-action and the projection in \cref{def:g-space} are immediate from the fact that $\mathring{E}$ satisfies them.

  For \cref{item:smooth-free-tn-action} we need to show that $\iota:T^n\lra T^k$ is an injective smooth Lie-group homomorphism. It then defines automatically a smooth free action of $T^n$ on $\mathring{E}'$. Since $T^n$ is compact the action is properly discontinuous and the quotient map $\mathring{E}'\lra \faktor{\mathring{E}'}{T^n}$ is a $T^n$-principal fibre bundle. However, by the same argument as in the proof of \cref{lem5} the map $\iota$ is injective and it is clearly a smooth Lie group homomorphism.

  Recall that $\Ob\mcM_{g,k,h,n}(T)\ltimes \mathring{E}'=\mathring{E}'$ and therefore $\mathring{E}'\lra \faktor{\Ob\mcM_{g,k,h,n}(T)\ltimes \mathring{E}'}{T^n}$ is a $T^n$-principal bundle. Notice that $T^n$ acts freely on
  \begin{equation*}
    \Mor\mcM_{g,k,h,n}(T)\ltimes\mathring{E}'=\Mor\mcM_{g,k,h,n}(T){_s\times_{\pi}}\mathring{E}'
  \end{equation*}
  by acting on the second factor. We thus have a smooth $T^n$-action on $\mcM_{g,k,h,n}(T)\ltimes\mathring{E}'$ and we can form the quotient orbifold category $\faktor{\mcM_{g,k,h,n}(T)\ltimes \mathring{E}'}{T^n}$.

  For \cref{item:principal-tn-orbibundle} it remains to check that $\faktor{\Mor_{g,k,h,n}(T)\ltimes\mathring{E}'}{T^n}$ acts diagonally via fibrewise equivariant maps on $\mathring{E}'$ in the sense of \cref{def:g-space}.\footnote{Note that the $T^n$-action on $\Mor\mcM_{g,k,h,n}(T)\ltimes\mathring{E}'$ does not actually modify the morphism of Hurwitz covers but identifies the families of morphisms corresponding to the same Hurwitz covers with different marked points $\bz$.} However, all the pointwise properties are clear as they are satisfied by the action of $\mcM_{g,k,h,n}(T)$ on $\mathring{E}'$. The fact that $\faktor{\mcM_{g,k,h,n}(T)\ltimes\mathring{E}'}{T^n}$ acts on the fibres by rotations can again be seen in the same way as above in the proof of \cref{lem:torus-subbundles-are-bundles}. Thus the bundle also satisfies \cref{def:orbifold-bundles} as rotations are equivariant maps with respect to the torus actions which are also rotations. Therefore we have a $T^n$-principal orbibundle which clearly splits as a direct sum of $S^1$-principle orbibundles.
\end{proof}

Before discussing how $\mathring{E}'$ relates to $\mathring{F}$ let us investigate the maps $(\wt{\eta}_{\nu(j)})^{-1}\circ u\circ\wt{\rho}_j$ a bit further.

\begin{lem}
  We have
  \begin{equation*}
    (\wt{\eta}_{\nu(j)})^{-1}\circ u\circ\wt{\rho}_j(z_j)=e^{2\pi\ii\alpha_j}\cdot z_j^{l_j},
  \end{equation*}
  where $e^{2\pi\ii\alpha_j}=(\wt{\eta}_{\nu(j)})^{-1}\circ u\circ\wt{\rho}_j(1)$ depends only on $j$ and $z_j\in S^1$.
  \label{lem:local-description-transfer-map}
\end{lem}

\begin{proof}
  Notice that $\wt{\rho}_j:S^1\lra \Gamma_j(C)$ maps a point $e^{2\pi\ii\theta_j}$ with $\theta_j\in[0,1)$ to a point of distance $l_jF(0)$ on $\Gamma_j(C)$ in the positive direction from the point $\rho_j(1)\in\Gamma_j(C)$ by definition of the map, rotational symmetry of a cusp neighborhood and because $\rho_j$ is holomorphic and thus conformal. The map $u$ is an isometry and maps $\Gamma_j(C)$ to $\Gamma_{\nu(j)}(X)$ and thus maps an arc-length parametrized curve to an arc-length parametrized curve, see \cref{lem:hyperbolic-lift-geodesic}. We can repeat the argument for $\wt{\eta}_{\nu(j)}:S^1\lra\Gamma_{\nu(j)}(X)$ to see that this arc of length $l_jF(0)$ is mapped to an angle $e^{2\pi\ii l_j\theta_j}$. However, as we chose the reference points on $\Gamma_j(C)$ and $\Gamma_i(X)$ independently there might be an additional constant rotation depending on $j$ corresponding to the angle between $(\wt{\eta}_{\nu(j)})^{-1}\circ u\circ\wt{\rho}_j(1)$ and $1$. All in all we see that the maps are given by
  \begin{equation*}
    (\wt{\eta}_{\nu(j)})^{-1}\circ u\circ\wt{\rho}_j(z_j)=e^{2\pi\ii\alpha_j}\cdot z_j^{l_j}.
  \end{equation*}
\end{proof}

\begin{rmk}
  Notice that the angles $\alpha_j$ are non-zero because we did not require any relation between the choices for the reference points on the curves $\Gamma_j(C)$ and $\Gamma_i(X)$. In particular one could fix this by requiring that the points on $\Gamma_j(C)$ are preimages of $\Gamma_{\nu(j)}(X)$. In any case we still have that points in $\mathring{E}'$ are well-defined, they are just ``twisted'' within the torus $T^k$ by these angles.
\end{rmk}

\begin{prop}
  The map $\Theta:\mathring{E}'\lra \mathring{F}$ defined by
  \begin{equation*}
    \Theta(b,(z_1,\ldots,z_k))\coloneqq \left(b,\left((\wt{\eta}_i)^{-1}\circ  u^b \circ \wt{\rho}_j(z_j)\text{ for }j\text{ s.t.\ }\nu(j)=i\right)_{i=1}^n\right)
  \end{equation*}
  on $\mathring{E}'\cap (O^{\lambda}\times T^k)$ together with the homomorphism
  \begin{align*}
    \pi : \faktor{\mcM_{g,k,h,n}(T)\ltimes \mathring{E}'}{T^n} & \lra \mcM_{g,k,h,n}(T) \\
    [(b,(z_1,\ldots,z_k)] & \longmapsto b \\
    [(b,(\Phi,\varphi),c),(b,(z_1,\ldots,z_k))] & \longmapsto (b,(\Phi,\varphi),c)
  \end{align*}
  is a $T^n$-principal orbibundle morphism with the Lie group homomorphism
  \begin{align*}
    \beta:T^n & \lra T^n \\
    (e^{2\pi\ii\theta_1},\ldots,e^{2\pi\ii\theta_n}) & \longmapsto (e^{2\pi\ii K_1\theta_1},\ldots,e^{2\pi\ii K_n\theta_n}).
  \end{align*}
  The homomorphism $\pi$ is a full morphism covering of degree $K\coloneqq\prod_{i=1}^nK_i$.
  \label{prop:Tn-bundle-map-E-prime-F}
\end{prop}

\begin{proof}
  Recall from \cref{def:orbibundle-morphism} that we need to show the following properties:
  \begin{enumerate}[label=(\roman*), ref=(\roman*)]
    \item the map $\Theta$ is a well-defined smooth $T^n$-equivariant map $\mathring{E}'\lra\mathring{F}$ with $\Theta(g\cdot x)=\beta(g)\cdot \Theta(x)$ for $g\in T^n$ and $x\in \mathring{E}'$, \label{item:theta-well-defined}
    \item the induced map on quotient spaces is the same as $\pi_{\Ob}$, \label{item:induced-map-pi}
    \item the diagram
      \begin{equation}
        \xymatrix{
           \Mor\left(\faktor{\mcM_{g,k,h,n}(T)\ltimes \mathring{E}'}{T^n}\right){_s\times_{\pi}}\mathring{E}' \ar[d]_{\pi_{\Mor}\times\Theta} \ar[r]^-{\mu_{\mathring{E}'}} & \mathring{E}' \ar[d]^{\Theta} \\
           \Mor\mcM_{g,k,h,n}(T){_s\times_{\pi}}\mathring{F} \ar[r]^-{\mu_{\mathring{F}}} & \mathring{F}
          } \label{eq:diag-commuting-actions}
      \end{equation}
      commutes and \label{item:commuting-diagram-actions}
    \item the functor $\pi$ is a morphism covering whose morphism map is surjective when restricted to automorphisms. \label{item:pi-morphism-covering}
  \end{enumerate}
  Now let us check \cref{item:theta-well-defined}. By \cref{def:e-prime-mathring} this map is well-defined. It is smooth by the same argument as in the proof of \cref{lem:E-orbibundle} since the Riemann mapping chart depends smoothly on the smooth family of reference curves. For the equivariance we can calculate in the $i$-th $S^1$ summand of $\mathring{F}$ for $g=(e^{2\pi\ii \theta_1},\ldots,e^{2\pi\ii \theta_n})\in T^n$ and $j$ such that $\nu(j)=i$
  \begin{align*}
    \Theta(g\cdot (z_1,\ldots,z_k))_i & = (\wt{\eta}_{\nu(j)})^{-1}\circ u\circ\wt{\rho}_j\left(\iota(g)\cdot z_j\right) \\
    & = (\wt{\eta}_{\nu(j)})^{-1}\circ u\circ\wt{\rho}_j\left(e^{2\pi\ii\frac{K_i}{l_j}\theta_i}z_j\right) \\
    & = e^{2\pi\ii\alpha_j}e^{2\pi\ii K_i\theta_i}z_j^{l_j} \\
    & = e^{2\pi\ii K_i\theta_i}\Theta(z_1,\ldots,z_k)_i \\
    & = \beta(g)\cdot \Theta(z_1,\ldots,z_k)_i.
  \end{align*}
  
  The map $\Theta$ on $\Ob\faktor{\mcM_{g,k,h,n}(T)\ltimes \mathring{E}'}{T^n} = \faktor{\mathring{E}'}{T^n}$ maps $[b,(z_1,\ldots,z_k)]$ to $b$ which coincides with $\pi_{\Ob}$ and thus we have \cref{item:induced-map-pi}.
  
  Furthermore the diagram in \cref{eq:diag-commuting-actions} commutes because the maps are given by
  \begin{align*}
    \Theta(\mu_{\mathring{E}'}([(b,(\Phi,\varphi),c), & (b,(z_1,\ldots,z_k))]) =\Theta\left(c,\left(\left(\D_0\rho_j^c\right)^{-1}\circ \D_{q^b_j}\Phi\circ \D_0\rho_j^{b}(z_j)\right)^k_{j=1}\right) \\
    & = \left(c,\left((\wt{\eta}^c_{i})^{-1}\circ \varphi \circ u \circ \wt{\rho}_j^b(z_j) \text{ for }j\text{ s.t. }\nu(j)=i\right)_{i=1}^n\right)
  \end{align*}
  and
  \begin{align*}
    \mu_{\mathring{F}}(\pi([(b, & (\Phi,\varphi),c),(b,(z_1,\ldots,z_k))]), \Theta(b,(z_1,\ldots,z_k)) =  \\
    & = \mu_{\mathring{F}}\left((b,(\Phi,\varphi),c),\left(b,\left((\wt{\eta}^b_i)^{-1}\circ  u^b \circ \wt{\rho}^b_j(z_j)\text{ for }j\text{ s.t.\ }\nu(j)=i\right)_{i=1}^n\right)\right) \\
    & = \left(c,\left((\D_o\eta_i^c)^{-1}\circ\D_{p_i^b}\varphi\circ\D_0\eta_i^b\circ(\wt{\eta}^b_i)^{-1}\circ  u^b \circ \wt{\rho}^b_j(z_j)\text{ for }j\text{ s.t.\ }\nu(j)=i \right)_{i=1}^n\right) \\
    & = \left(c,\left((\wt{\eta}_i^c)^{-1}\circ\varphi\circ\wt{\eta}_i^b\circ(\wt{\eta}^b_i)^{-1}\circ  u^b \circ \wt{\rho}^b_j(z_j)\text{ for }j\text{ s.t.\ }\nu(j)=i \right)_{i=1}^n\right),
  \end{align*}
  where we have used the same calculations as in the proof of \cref{prop:E-prime-E-subbundle}. 
  
  It remains to prove that $\pi$ is a morphism covering and surjective when restricted to automorphism groups. However, it is clearly a smooth functor and also a finite covering on objects and morphisms as preimages of sufficiently small neighborhoods $\mcU$ are given by $\mcU\times \faktor{T^n}{\beta(T^n)}$ for open $\mcU$ in $\Ob\mcM_{g,k,h,n}(T)$ and $\mcU\subset\Mor\mcM_{g,k,h,n}(T)$. The lifting property is clear as any $(\Phi,\phi):b\lra c$ in $\mcM_{g,k,h,n}(T)$ defines a morphism $[(b,(\Phi,\phi),c),(b,(z_1,\ldots,z_k))]$ in $\Mor\faktor{\mcM_{g,k,h,n}(T)\ltimes \mathring{E}'}{T^n}$. The same argument shows that the surjectivity when restricted to automorphisms as any automorphism $(\Phi,\phi):b\lra b$ in $\mcM_{g,k,h,n}(T)$ defines an automorphism 
  \begin{equation*}
    [(b,(\Phi,\phi),b),(b,(z_1,\ldots,z_k))]\in\Mor\faktor{\mcM_{g,k,h,n}(T)\ltimes \mathring{E}'}{T^n}.
  \end{equation*}
  By definition of $\faktor{\mcM_{g,k,h,n}(T)\ltimes \mathring{E}'}{T^n}$ there are no morphisms between the elements in the fibre over an object in $\mcM_{g,k,h,n}(T)$ and thus the covering has degree $K=\prod_{i=1}^nK_i$.
\end{proof}

\begin{cor}
  The map $\pi:\left|\faktor{\mcM_{g,k,h,n}(T)\ltimes \mathring{E}'}{T^n}\right|\lra|\mcM_{g,k,h,n}(T)|$ is a topological covering of degree $K$.
  \label{cor:pi-top-covering}
\end{cor}

\begin{proof}
  \cref{prop:Tn-bundle-map-E-prime-F} and \cref{lem:morphism-covering-actual-covering} imply that $\pi$ is a topological covering on orbit spaces. It is easy to see that an element $b\in\Ob\mcM_{g,k,h,n}(T)$ has $K$ preimages under $\pi_{\Ob}$ which are pairwise not identified and thus $\pi$ has topological covering degree $K$.
\end{proof}

\begin{rmk}
  From \cref{sec:algebraic-topology-orbifolds}, \cref{rmk:chern-classes-e-f} and general knowledge about connection forms on $S^1$-principal bundles from e.g.\ \cite{chern_circle_1977} a rational Chern class of a $T^k$-bundle $\pi:\mathring{E}\lra\Ob\mcM_{g,k,h,n}(T)$ in Chern--Weil theory is defined as the cohomology class of an invariant $2$-form $\alpha\in\Omega^2(\mcM_{g,k,h,n}(T),\RR^k)$ such that its lift to $E$ satisfies $\pi^*\alpha=\frac{\ii}{2\pi}\dd A$, where $A$ is a connection $1$-form on $E$. So let us pick such a form $A\in\Omega^1(E,\mft)$ satisfying\footnote{Here, $\mft$ denotes the Lie-algebra of the Lie group, in our case $\mft=\ii\RR^k$.}

\begin{enumerate}[label=(\roman*), ref=(\roman*)]
\item $A(\ul{v})=v\quad\forall v\in\mft$,
\item $A_{g.p}(\D_pg\cdot v)=A_p(v)\qquad\forall p\in E, v\in\ts_p E\text{ and }g\in T^n$,
\end{enumerate}

where $\ul{v}$ denotes the vector field corresponding to the Lie algebra element $v\in\mft\cong\ii\RR^n$. Note that this definition gives the Chern class of the complex vector bundle $E$ whose unit-torus bundle is given by $\mathring{E}$, too.
\label{rmk:chern-classes-on-bundles}
\end{rmk}

Using \cref{prop:E-prime-E-subbundle} and \cref{prop:Tn-bundle-map-E-prime-F} we can now relate the Chern classes of $E$ and $F$.

\begin{prop}
  The Chern classes of $E, E'$ and $F$ satisfy
  \begin{align}
    c_1(E')_i & = \frac{1}{K_i}\pi^*c_1(F)_i\text{ and} \label{eq:chern-classes-2} \\
    c_1(E')_i & = \frac{1}{dK_i}\sum_{\substack{j=1,\ldots,k\\ \nu(j)=i}}l_j^2\pi^*c_1(E)_j, \label{eq:chern-classes-1}
  \end{align}
  where the index $i$ or $j$ denotes the $i$-th or $j$-th component of the Chern class vector, respectively.
  \label{prop:chern-class-calculations}
\end{prop}

\begin{proof}
  Notice first that by \cref{rmk:chern-classes-on-bundles} and \cref{rmk:chern-classes-e-f} we can prove this by showing the equivalent equalities for the Chern classes of the corresponding $T^k$- and $T^n$-bundles.

  For \cref{eq:chern-classes-2} pick a connection $1$-form $A\in\Omega^1(\mathring{F},\mft)$ as in \cref{rmk:chern-classes-on-bundles} and define
  \begin{equation*}
    A_i'\coloneqq \frac{1}{K_i}\Theta^*A_i.
  \end{equation*}
  Notice that we have the bundle map
  \begin{equation}
    \xymatrix{
      \mathring{E}' \ar[r]^{\Theta} \ar[d] & \mathring{F} \ar[d] \\
      \Ob\faktor{\mcM_{g,k,h,n}(T)\ltimes \mathring{E}'}{T^n} \ar[r]^-{\pi_{\Ob}} & \Ob\mcM_{g,k,h,n}(T)
    }
    \label{eq:commutin-diagram-bundles}
  \end{equation}
  which satisfies
  \begin{equation}
    \Theta(g\cdot p)=\beta(g)\cdot\Theta(p)
    \label{eq:theta-equivariance}
  \end{equation}
  for $g\in T^n,p\in \mathring{E}'$ and $\beta:T^n\lra T^n$ as in \cref{prop:Tn-bundle-map-E-prime-F}. We need to check that $A'=\sum_{i=1}^nA_i'e_i'$ satisfies both equations from \cref{rmk:chern-classes-on-bundles} with the standard basis $\{e_i'\}_{i=1}^n$ of $\RR^n$. We calculate
  \begin{align*}
    K_i\cdot(A'_i)_{g\cdot p}(\D_pg\cdot X) & = (A_i)_{\Theta(g\cdot p)}(\D_{g\cdot p}\Theta\circ\D_pg\cdot X) \\
    & = (A_i)_{\beta(g)\cdot\Theta(p)}(\D_{\Theta(p)}\beta(g)\circ\D_{\Theta(g\cdot p)}\beta(g)^{-1}\circ \D_{g\cdot p}\Theta\circ\D_pg\cdot X) \\
    & = (A_i)_{\Theta(p)}(\D_{\Theta(g\cdot p)}\beta(g)^{-1}\circ \D_{g\cdot p}\Theta\circ\D_pg\cdot X) \\
    & = (A_i)_{\Theta(p)}(\D_p(\beta(g^{-1})\cdot\Theta\circ g)\cdot X) \\
    & = (A_i)_{\Theta(p)}(\D_p\Theta\cdot X) \\
    & = (\Theta^*A_i)_p(X)
  \end{align*}
  for $X\in\ts_p\mathring{E}'$ where we have used again \cref{eq:theta-equivariance}. The second equation can be checked as follows. Here, $v=\sum_{i=1}^nv_ie_i'\in\mft$ and
  \begin{equation*}
    \ul{v}_p=\ddt \exp(2\pi\ii tv)\cdot p=\ddt \left(e^{2\pi\ii\frac{K_{\nu(1)}}{l_1} t v_{\nu(1)}},\ldots,e^{2\pi\ii\frac{K_{\nu(k)}}{l_k} t v_{\nu(k)}}\right)\cdot p
  \end{equation*}
  such that for $p=(b,(z_1,\ldots,z_k))\in \mathring{E}'$
  \begin{align*}
    (A')_p(\ul{v}(p)) & = \sum_{i=1}^n\frac{1}{K_i}(A_i)_{\Theta(p)}(\D_p\Theta\cdot \ul{v}(p)) \\
    & = \sum_{i=1}^n\frac{1}{K_i}(A_i)_{\Theta(p)}\left(\ddt\Theta(b,(e^{2\pi\ii\frac{K_{\nu(1)}}{l_1} t v_{\nu(1)}}z_1,\ldots,e^{2\pi\ii\frac{K_{\nu(k)}}{l_k} t v_{\nu(k)}}z_k))\right) \\
    & = \sum_{i=1}^n\frac{1}{K_i}(A_i)_{\Theta(p)}\left(\ddt \left(b,\left(e^{2\pi\ii\alpha_j}e^{2\pi\ii K_iv_it}z_j^{l_j}\text{ for }j\text{ s.t.\ }\nu(j)=i\right)_{i=1}^n\right)\right) \\
    & = \sum_{i=1}^n\frac{1}{K_i}(A_i)_{\Theta(p)}\left(\sum_{i=1}^nK_iv_i\ul{e_i}(\Theta(p))\right) \\
    & = \sum_{i=1}^n\frac{1}{K_i}(A_i)_{\Theta(p)}(K_iv_i\ul{e_i}) \\
    & = \sum_{i=1}^nv_ie_i = v.
  \end{align*}
  As \cref{eq:commutin-diagram-bundles} commutes and therefore $\pi_{\mathring{E}'}^*\pi^*c_1(\mathring{F})_i=\Theta^*\pi_{\mathring{F}}^*c_1(\mathring{F})_i=\frac{\ii}{2\pi}\Theta^*\dd A_i=\frac{\ii}{2\pi}\dd A_i'$, this proves \cref{eq:chern-classes-2}.

  Regarding \cref{eq:chern-classes-1} notice that the map $\epsilon:\mathring{E}' \hookrightarrow \mathring{E}$ is equivariant in the sense that for $g\in T^n$ and $p\in\mathring{E}'$ we have
  \begin{equation*}
    \epsilon(g\cdot p)=\iota(g)\cdot\epsilon(p)
  \end{equation*}
  which is the only fact that we needed in the first half of showing \cref{eq:chern-classes-2}. Thus $A_j|_{\mathring{E}'}=\epsilon^*A_j$ still satisfies the invariance condition for a connection form by the same calculation. Therefore, if we define
  \begin{equation*}
    A'\coloneqq\sum_{i=1}^n\frac{1}{dK_i}\sum_{\substack{j=1\\ \nu(j)=i}}^kl_j^2A_j|_{\mathring{E}'}e_i'\in\Omega^1(\mathring{E}',\mft)
  \end{equation*}
  this form $A'$ satisfies $(A')_{g\cdot p}(\D_pg\cdot X)=(A')_p(X)$ for all $X\in\ts_p\mathring{E}'$. It remains to check the other relation. To this end, we first calculate the Killing fields of the $T^n$-action on $\mathring{E}'$ in terms of those of the $T^k$-action on $\mathring{E}$. We get for $\ul{e_i'}(p)=\ddt\exp(t e_i')\cdot p$ with the standard basis  $\{e_i'\}_{i=1}^n$ of $\RR^n$ and $p\in\mathring{E}'$
  \begin{align*}
    \D_p\epsilon\cdot\ul{e_i'}(p) & = \ddt \epsilon(\exp(te_i)\cdot p) \\
    & = \ddt \iota(\exp(t e_i))\cdot \epsilon(p) \\
    & = \ddt \iota(1,\ldots,e^{2\pi\ii t},\ldots,1)\cdot\epsilon(p) \\
    & = \ddt \left( \begin{cases}e^{2\pi\ii\frac{K_i}{l_j}t} & \nu(j)=i \\ 1 & \nu(j)\neq i\end{cases}\right)_{j=}^k \cdot \epsilon(p) \\
    & = \sum_{\substack{j=1 \\ \nu(j)=i}}^k \frac{K_i}{l_j}\ul{e_j}(p),
  \end{align*}
  where $\{e_j\}_{j=1}^k$ is the standard basis of $\RR^k$. We therefore have
  \begin{align*}
    A'_p(\ul{e_l'}(p)) & = \sum_{i=1}^n\frac{1}{d K_i}\sum_{\substack{j=1 \\ \nu(j)=i}}^kl_j^2\sum_{\substack{m=1 \\ \nu(m)=l}}^k\frac{K_l}{l_m}(A_j)_p(\ul{e_m}(p))e_i' \\
    & = \sum_{i=1}^n\frac{1}{d K_i}\sum_{\substack{j=1 \\ \nu(j)=i}}^kl_j^2\sum_{\substack{m=1 \\ \nu(m)=l}}^k\frac{K_l}{l_m}\delta_{jm}e_i' \\
    & = \sum_{i=1}^n\frac{1}{d}\sum_{\substack{j=1 \\ \nu(j)=i}}^kl_j \delta_{il}e_i' \\
    & = e_l'
  \end{align*}
  which proves the other condition by linearity and thus $A'$ is a connection $1$-form for $\mathring{E}'$. Notice that the $\iota$-equivariant map $\epsilon$ gives a commuting diagram of $T^n$-principle bundles
  \begin{equation*}
    \xymatrix{
      \mathring{E}' \ar[d]_{\pi_{\mathring{E}'}} \ar[r]^{\epsilon} & \mathring{E} \ar[d]^{\pi_{\mathring{E}}} \\
      \Ob\faktor{\mcM_{g,k,h,n}(T)\ltimes \mathring{E}'}{T^n} \ar[r]^-{\pi_{\Ob}} & \Ob\mcM_{g,k,h,n}(T)
      }
  \end{equation*}
  and therefore we have again $\pi_{\mathring{E}'}^*\pi^* c_1(\mathring{E})_j=\epsilon^*\pi_{\mathring{E}}^*c_1(\mathring{E})_j=\frac{\ii}{2\pi}\epsilon^*\dd A_j=\frac{\ii}{2\pi}\dd A_j|_{\mathring{E}'}$ which proves \cref{eq:chern-classes-1}.
\end{proof}

Combining \cref{prop:chern-class-calculations}, \cref{lem:pulled-back-line-bundles} and \cref{def:psi-classes} we obtain \cref{cor:chern-class-relations}.

\begin{cor}
  The various $\Psi$-classes satisfy
  \begin{align*}
    \pi^*\ev^*\psi_i=\frac{1}{d}\sum_{\substack{j=1 \\ \nu(j)=i}}^kl_j^2\pi^*\fgt^*\psi_j
  \end{align*}
  in $H^2\left(\faktor{\mcM_{g,k,h,n}(T)\ltimes \mathring{E}'}{T^n},\QQ\right)$ for every $i=1,\ldots,n$.
  \label{cor:chern-class-relations}
\end{cor}

\subsection{The Limit \texorpdfstring{$\boldsymbol{L\to 0}$}{L->0}} 

We will now discuss how to understand the limit $L\lra 0$ of
\begin{equation*}
  \xymatrix{
    \wh{\mu}^{-1}(L) \ar[d] \ar[rr]^-{``L\to 0"} & & \mcM_{g,k,h,n}(T)\ltimes\mathring{E}' \ar[d] \\
    \wh{\mcM}_{g,k,h,n}^{\square}(T)[L] \ar[rr]^-{``L\to 0"} & & \faktor{\mcM_{g,k,h,n}(T)\ltimes\mathring{E}'}{T^n} \\
    }
\end{equation*}
by defining a family of orbifold isomorphisms
\begin{equation*}
  \Xi_L:\mcM_{g,k,h,n}(T)\ltimes\mathring{E}' \lra \wh{\mcM}^{\square}_{g,k,h,n}(T)|_{\wh{\mu}^{-1}(L)}
\end{equation*}
for every $L\in\RR_{>0}^n$. This $\Xi_L$ will depend continuously on $L$ and be $T^n$-equivariant for the corresponding $T^n$-actions and will allow us to compute the limits $L\to 0$.

Recall that a point in $\Ob\mcM_{g,k,h,n}(T)\ltimes\mathring{E}'$ is given by a tuple 
\begin{equation*}
  (b,\bz)\in O^{\lambda}\times T^k
\end{equation*}
satisfying \cref{def:e-prime-mathring}. Furthermore, recall that we defined a function $F:\RR_{\geq 0}\lra\RR_{>0}$ fixing the lengths of the reference curves. Another choice was a set $\Lambda$ of central Hurwitz covers with additional data used for defining the Hurwitz deformation families $\Psi^{\lambda}$ for $\lambda\in\Lambda$. As we will now consider the categories $\mcR_{g,k,h,n}(T)$ as well as $\wt{\mcR}_{g,nd+k,h,2n}(\wt{T})$ we will distinguish the sets $\Lambda$ and $\wt{\Lambda}$ and the families $\Psi^{\lambda}:O^{\lambda}\lra\Ob\mcR_{g,k,h,n}(T)$ and $\wt{\Psi}^{\lambda}:\wt{O}^{\lambda}\lra\Ob\wt{\mcR}_{g,nd+k,h,2n}(\wt{T})$, respectively. Furthermore we will include another choice with every $\lambda\in\Lambda$ as follows.

Via $\lambda\in\Lambda$ we have chosen for every $i=1,\ldots,n$ one interior pair of pants on $X^{\lambda}$ including the degenerate boundary $\del_iX$. Pick one interior boundary component of such a pair of pants and intersect the unique geodesic perpendicular to that boundary geodesic and going up the cusp $\del_iX^{\lambda}$ with $\Gamma_i(X^{\lambda})$. We denote this intersection point by $m_i$. Now choose one preimage $m_j'$ of $m_i$ under $u^{\lambda}$ on every $\Gamma_j(C^{\lambda})$ for all $j=1,\ldots,k$ such that $\nu(j)=i$. We will use these points to define maps $\wt{\eta}_i:S^1\lra\Gamma_i(X^{\nu})$ and $\wt{\rho}_j:S^1\lra\Gamma_j(C^{\nu})$ as in \cref{def:s1-parametrizations}. Note that this defines corresponding maps for any admissible Hurwitz cover close to $(C,u,X,\bq,\bp)$ as the construction for $m_i$ only depends on $\Gamma_i(X)$ and the choice of one boundary of the interior pair of pants. By passing to a small enough neighborhood in the moduli space of Hurwitz covers we also obtain a well-defined choice of preimages $m_j'$ which are close to the corresponding ones on the central Hurwitz cover. This is because the gluing construction on $C^{\lambda}$ in \cref{sec:local-parametrizations} actually gives a map from the boundary of the disc structure on $C^{\lambda}$ to the boundary of the glued-in annulus. See \cref{fig:choices-zeros-circles-1} and \cref{fig:choices-zeros-circles-2} for an illustration.

\begin{figure}[!ht]
  \centering
  \def\svgwidth{0.8\textwidth}
  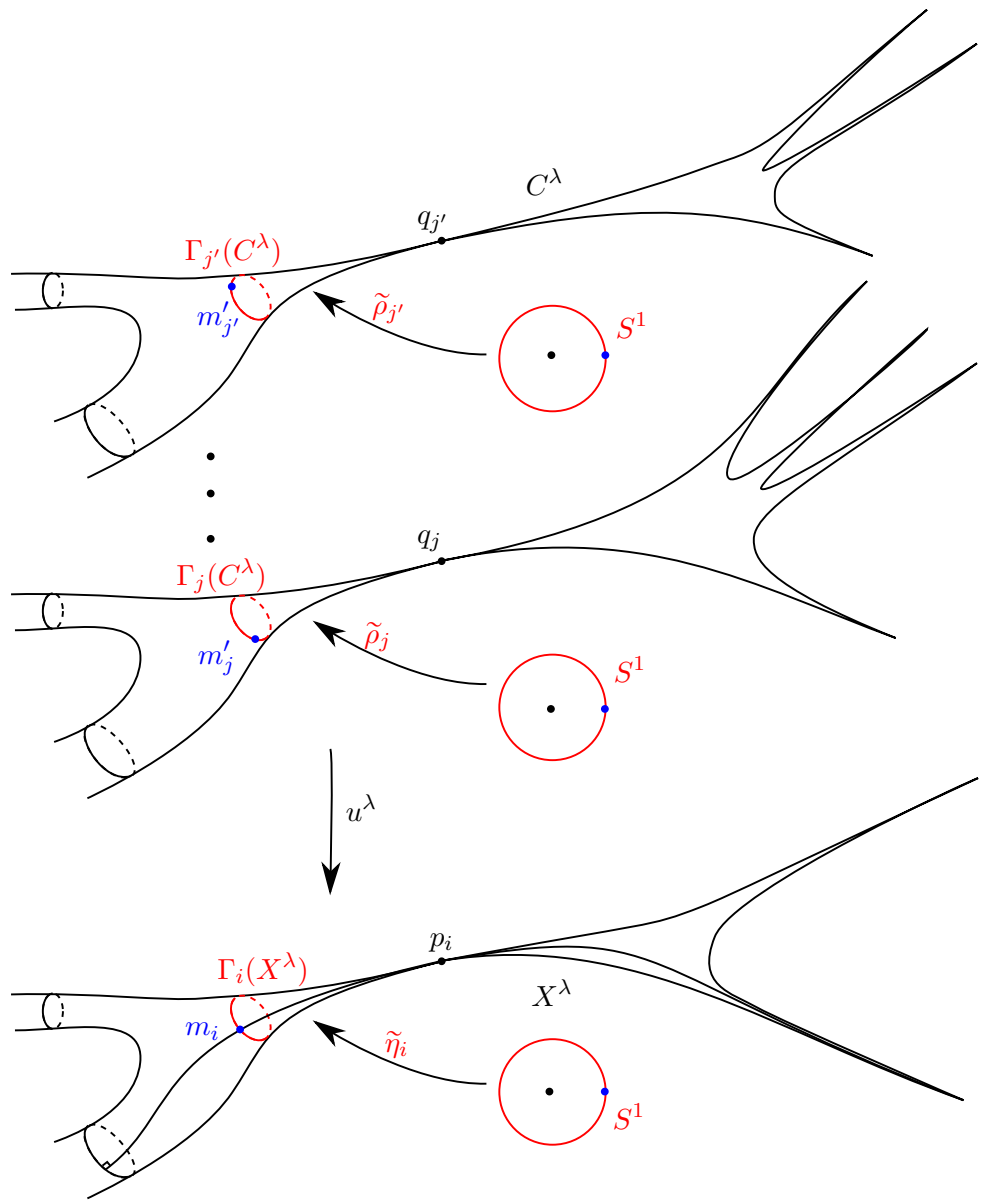
  \caption{This shows the Hurwitz cover $(C^{\lambda},u^{\lambda},X^{\lambda},\bq^{\lambda},\bp^{\lambda})\in\Ob\mcM_{g,k,h,n}(T)$ or rather its image under $\glue\circ\,\iota$, together with the additional choices in $\lambda$. The points $m_i$ and $m_j'$ are used as reference points on the reference curves such that we can measure hyperbolic distances from an origin. Notice that the pairs of pants on the left hand side are \emph{not} mapped onto each other, as a degree-$l_j$ preimage of a pair of pants cannot be a pair of pants.Thus we can do the geometric construction for $m_i\in X^{\lambda}$, but not at the same time for $m'_j\in C^{\lambda}$.}
  \label{fig:choices-zeros-circles-1}
\end{figure}

\begin{figure}[!ht]
    \centering
    \def\svgwidth{0.8\textwidth}
    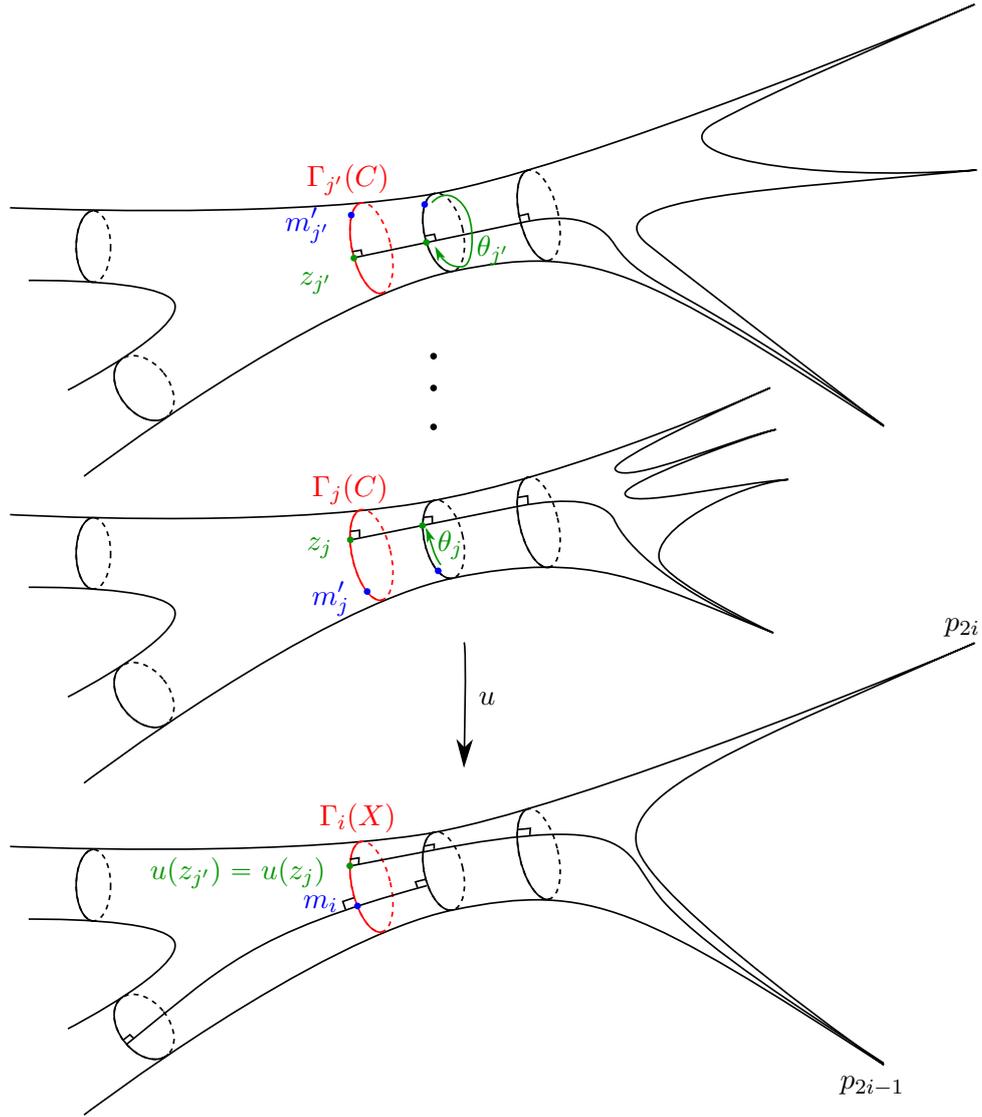
    \caption{In this figure one can see the Hurwitz cover $\Xi_L(b,(e^{2\pi\ii\theta_1},\ldots,e^{2\pi\ii\theta_k}))$ or rather its still glued version. Notice that the twisting in the construction is arbitrary as we cut the cover anyway. However, we still mark the point $z_j\in\Gamma_j(C)$ which would correspond to a Fenchel--Nielsen twist of $\theta_j$ by following the reference curve for the correct hyperbolic distance.}
    \label{fig:choices-zeros-circles-2}
\end{figure}

\begin{lem}
  The functor
  \begin{equation*}
    \iota:\mcM_{g,k,h,n}(T)\ltimes\mathring{E}'\lra\wh{\mcR}_{g,k,h,n}(T)
  \end{equation*}
  given by
  \begin{align*}
    \iota_{\Ob}(b,\bz) & \coloneqq (C,u,X,\bq,\bp,\bz') \\
    \iota_{\Mor}((b,(\Phi,\varphi),c),(b,\bz)) & \coloneqq (\iota_{\Ob}(b),(\Phi,\varphi),\iota_{\Ob}(c)) \\
  \end{align*}
  is well-defined, fully faithful and injective on objects. It is continuous on orbit spaces. Here the tuple $\bz'$ is defined as
  \begin{equation*}
    z_j'\coloneqq \wt{\rho}_j(z_j)
  \end{equation*}
  and $(C,u,X,\bq,\bp)\coloneqq\Psi^{\lambda}(b)$ for $b\in O^{\lambda}$ and the Hurwitz deformation family $\Psi^{\lambda}$.
  \label{lem:glue-spheres-to-closed-hurwitz-cover}
\end{lem}

\begin{proof}
  The condition on the points in $\mathring{E}'$ from \cref{def:e-prime-mathring} is precisely such that the points $\bz'$ satisfy $u(z'_i)=u(z'_j)$ for all $i,j$ such that $\nu(j)=\nu(i)$. Therefore, $\iota$ is a well-defined functor. Continuity on orbit spaces is also clear as we used the same type of construction to define the topology on $|\wh{\mcR}_{g,k,h,n}(T)|$ in \cref{sec:topology-bordered-hurwitz-covers}. The injectivity on objects is due to the fact that $\wt{\rho}_j$ is bijective.

  It remains to notice that for any $b,c\in\Ob\mcM_{g,k,h,n}(T)\ltimes\mathring{E}'$ the map
  \begin{equation*}
    \iota:\Hom_{\mcM_{g,k,h,n}(T)\ltimes\mathring{E}'}(b,c)\lra\Hom_{\wh{\mcR}_{g,k,h,n}(T)}(\iota(b),\iota(c))
  \end{equation*}
  is bijective. However, this is again obvious as the notion of morphisms in both categories is the same, namely a pair $(\Phi,\varphi)$ acting in the obvious way on the marked points on $\bG$ and on $(S^1)^k$, respectively.
\end{proof}

\begin{rmk}
  Notice that this means that objects in $\mcM_{g,k,h,n}(T)\ltimes\mathring{E}'$ define elements in $\Ob\wt{\mcM}_{g,nd+k,h,2n}(\wt{T})$ by gluing nodal spheres to the marked points, i.e. we obtain a functor
  \begin{equation*}
    \mcM_{g,k,h,n}(T)\ltimes\mathring{E}'\overset{\iota}{\lra}\wh{\mcR}_{g,k,h,n}(T)\overset{\glue}{\lra}\wt{\mcR}_{g,nd+k,h,n}(\wt{T}).
  \end{equation*}
  This glued surface has of course the obvious contractible loops $\bG$ as induced data but it has forgotten the points in the fibre of $\mathring{E}'$. Nevertheless, we can use it to define the functor $\Xi_L$.
  \label{rmk:glue-spheres-to-closed-hurwitz-cover}
\end{rmk}

We will choose the set $\wt{\Lambda}$ defining the orbifold structure on $\wt{\mcM}_{g,nd+k,h,2n}(\wt{T})$ such that it includes $\glue(\iota(\lambda,\bz))$ for any $\lambda\in\Lambda$. Notice that $\glue\circ\iota$ forgets the point $\bz$, so this statement makes sense.

\begin{definition}
  For any sufficiently small $L\in\RR_{>0}^n$ we define
  \begin{equation*}
    (\Xi_L)_{\Ob}:\Ob\mcM_{g,k,h,n}(T)\ltimes\mathring{E}'=\mathring{E}'\lra \Ob\wh{\mcM}^{\square}_{g,k,h,n}(T)|_{\wh{ev}^{-1}(L)}
  \end{equation*}
  as follows. For every $(b,(e^{2\pi\ii\theta_1},\ldots,e^{2\pi\ii\theta_k}))\in O^{\lambda}\times T^k$ satisfying \cref{def:e-prime-mathring} we choose a point $(\glue(\iota(b)),y_1,\ldots,y_n)\in O^{\lambda}\times\DD^n\cong \wt{O}^{\wt{\lambda}}$ such that the hyperbolic length of $\Gamma_i$ on $X$ of $(C,u,X,\bq,\bp,\bG)=\wt{\Psi}^{\wt{\lambda}}(\glue(\iota(b)),y_1,\ldots,y_n)$ is $L_i$ for $i=1,\ldots,n$. See \cref{rmk:y-variables} for a few comments on the $y_i$. We then cut the Hurwitz cover along $\bG$ and mark the point $z_j$ on $\Gamma_j(C)$ at a hyperbolic distance of $l_jF(L_{\nu(j)})\cdot \theta_j$ from $m_j'$ in the positive direction. We define this Hurwitz cover with marked points $\bz$ as $\Xi_L(b,(e^{2\pi\ii\theta_1},\ldots,e^{2\pi\ii\theta_k}))$.

  On morphisms we can define
  \begin{equation*}
    (\Xi_L)_{\Mor}:\Mor\mcM_{g,k,h,n}(T)\ltimes\mathring{E}'\lra \Mor\wh{\mcM}_{g,k,h,n}(T)|_{\wh{ev}^{-1}(L)}
  \end{equation*}
  in the same way by mapping $((b,(\Phi,\varphi),c),(b,(e^{2\pi\ii\theta_1},\ldots,e^{2\pi\ii\theta_k})))$ with $(b,(\Phi,\varphi),c)\in M(\lambda,\lambda')$ first to the same morphism in $M(\wt{\lambda},\wt{\lambda}')\subset\Mor\wt{\mcM}_{g,nd+k,h,2n}(\wt{T})$ and then to the unique morphism starting at $\Xi_L(b,(e^{2\pi\ii\theta_1},\ldots,e^{2\pi\ii\theta_k}))$ agreeing with $(\Phi,\varphi)$ outside the gluing region and then restricted to the cut surface. Here we use \cref{prop1}.
  \label{def:xi}
\end{definition}

\begin{rmk}
  The role of the $(y_1,\ldots,y_n)\in\DD^n$ is as follows. As we have glued pairs of pants to the marked points we turned these marked points into interior nodes. Thus we need $n$ complex gluing parameters in order to define a smoothened Hurwitz cover. By \cite{hubbard_analytic_2014} we know that there is a homeomorphism to Fenchel--Nielsen coordinates on a small enough neighborhood of $0\in\DD^n$ (hence why we restrict to sufficiently small $L$). We can therefore fix these Fenchel--Nielsen coordinates by using the prescribed parameters $L\in\RR_{>0}^n$ for the lengths of the target surface and arbitrary Fenchel--Nielsen twists. The reason the twists do not matter is because we immediately afterwards cut the cover along these geodesics. The $(y_1,\ldots,y_n)$ are just the complex gluing parameters such that the corresponding hyperbolic geodesics have the prescribed lengths. We will not actually need their precise values.

  Also note that the letter $\Gamma$ is used here in two different ways. The multicurve $\bG$ is only used to remember where the boundary was inside of the glued surface. So in our case we can replace it mentally by their hyperbolic geodesic representatives. In contrast the curves $\Gamma_j(C)$ are reference curves which are in a collar neighborhood of the corresponding boundary geodesic.
  \label{rmk:y-variables}
\end{rmk}

\begin{rmk}
See \cref{fig:choices-zeros-circles-2} for an illustration. Furthermore notice that the points $y_1,\ldots,y_n\in\DD^n$ are not uniquely determined by just fixing the hyperbolic lengths as we would also need to fix the Fenchel--Nielsen twist around those curves. However, as we cut the surface along these curves we do not need to specify this twist. Instead we add it later by marking the points $z_j$ correspondingly.
\end{rmk}

\begin{prop}
  For every $L\in\RR_{>0}^n$ the functor
  \begin{equation*}
    \Xi_L:\mcM_{g,k,h,n}(T)\ltimes\mathring{E}\lra \wh{\mcM}^{\square}_{g,k,h,n}(T)|_{\wh{ev}^{-1}(L)}
  \end{equation*}
  is well-defined and a $T^n$-equivariant orbifold isomorphism and therefore descends to an orbifold isomorphism
  \begin{equation*}
    \xi_L:\faktor{\mcM_{g,k,h,n}(T)\ltimes\mathring{E}}{T^n}\lra\wh{\mcM}^{\square}_{g,k,h,n}(T)[L].
  \end{equation*}
  \label{prop:xil-well-defined}
\end{prop}

\begin{proof}
  Notice that by definition the target surface of $\Xi_L(b,(e^{2\pi\ii\theta_1},\ldots,e^{2\pi\ii\theta_k}))$ has boundary lengths $L$. The condition for the elements in $\mathring{E}'$ yields the condition for the marked points in $\wh{\mcM}_{g,k,h,n}(T)$. Regarding the morphisms
  \begin{equation*}
    \Xi_L((b,(\Phi,\varphi),c),(b,(e^{2\pi\ii\theta_1},\ldots,e^{2\pi\ii\theta_k})))
  \end{equation*}
  we see that a biholomorphism is an isometry of the induced hyperbolic metrics and therefore maps the curves $\Gamma_j(C_b)$ to $\Gamma_j(C_c)$ and also preserves lengths along these curves. From this it follows that $\Xi_L$ is indeed a functor.

  For the smoothness of $\Xi$ notice that we fixed the lengths of the boundary geodesics and that the marked points $\bz$ depend smoothly on $(b,(e^{2\pi\ii\theta_1},\ldots,e^{2\pi\ii\theta_k}))$. Also the Hurwitz cover depends smoothly on the parameters $b$ as the map is actually the identity in coordinates, since we use the same pair of pants decompositions and disc structures to define the orbifold structures on $\mcM_{g,k,h,n}(T),\wh{\mcM}_{g,k,h,n}(T)$ and $\wt{\mcM}_{g,nd+k,h,2n}(\wt{T})$.

  The functor $\Xi_L$ is bijective on objects as we have chosen corresponding $\Lambda$ and $\wt{\Lambda}$ and because all $O^{\lambda}$ and $\wt{O}^{\wt{\lambda}}$ are small enough such that the Hurwitz deformation families are injective. Therefore we can recover unique gluing parameters and marked points from the surfaces. Recall for this argument that $\lambda$ and $\wt{\lambda}$ both contain pair of pants decompositions as well disc structures and marked reference points $m_j'$ used for calculating distances along the reference curves. Bijectivity on morphisms is the same argument together with \cref{prop:morphisms-nodal-hurwitz-covers}.

  It remains to check the $T^n$-equivariance. Let $t=e^{2\pi\ii\alpha_1},\ldots,e^{2\pi\ii\alpha_n}\in T^n$. Multiplication by $\iota(t)\in T^k$ rotates $\theta_j\mapsto \theta_j+\frac{K_{\nu(j)}}{l_j}\alpha_{\nu(j)}$ and thus moves the marked point $z_j$ on $\Gamma_j(C)$ by a hyperbolic distance of $l_jF(L_{\nu(j)})\frac{K_{\nu(j)}}{l_j}\alpha_{\nu(j)}=F(L_{\nu(j)})K_{\nu(j)}\alpha_{\nu(j)}$. Recalling the definition of the torus action on $\wh{\mcM}^{\square}_{g,k,h,n}[L]$ from \cref{def:torus-action-moduli-spaces} we see that this is the same as a rotation by
  \begin{equation*}
    \rot_{\iota(t)}(\bz)_j=\beta_j\left(\frac{K_{\nu(j)}}{l_j}\alpha_jl_jF(L_{\nu(j)})\right)=\beta_j(K_{\nu(j)}F(L_{\nu(j)})\alpha_j),
  \end{equation*}
  where $\beta_j:[0,1]\lra\Gamma_j(C)$ is a parametrization by arc length, showing that $\Xi_L$ is equivariant.
\end{proof}

\begin{lem}
  We have
  \begin{equation*}
    \lim_{L\to 0}\;\xi_L^*[\wh{\omega}_L]=\pi^*[\omega]
  \end{equation*}
  in $H^2\left(\faktor{\mcM_{g,k,h,n}(T)\ltimes\mathring{E}'}{T^n},\RR\right)$, where
  \begin{equation*}
    \pi : \faktor{\mcM_{g,k,h,n}(T)\ltimes \mathring{E}'}{T^n} \lra \mcM_{g,k,h,n}(T)
  \end{equation*}
  is the projection from \cref{prop:Tn-bundle-map-E-prime-F}.
  \label{lem:limit-wwp}
\end{lem}

\begin{proof}
  From \cref{prop:xil-well-defined} we obtain isomorphisms
  \begin{equation*}
    \xi_L^*:H^2(\wh{\mcM}^{\square}_{g,k,h,n}[L],\RR)\lra H^2\left(\faktor{\mcM_{g,k,h,n}(T)\ltimes\mathring{E}'}{T^n},\RR\right)
  \end{equation*}
  for every $L\in\RR_{>0}^n$. Notice that we have the following diagram
  \begin{equation*}
    \xymatrix{
      \mcM_{g,k,h,n}(T)\ltimes\mathring{E}' \ar[r]^{\Xi_L} \ar[d]^{\eta} \ar@/^2pc/[rr]^{\glue\circ\iota} & \wh{\mcM}^{\square}_{g,k,h,n}(T)|_{\wh{\mu}^{-1}(L)} \ar[d] \ar[r]^{\glue} & \wt{\mcM}_{g,nd+k,h,2n}(\wt{T}) \\
      \faktor{\mcM_{g,k,h,n}(T)\ltimes\mathring{E}'}{T^n} \ar[d]^{\pi} \ar[r]^{\xi_L} & \wh{\mcM}^{\square}_{g,k,h,n}(T)[L] &  \\
      \mcM_{g,k,h,n}(T) \ar@/_3pc/[rruu]^{F} 
      }
  \end{equation*}
  where $F$ just glues a standard sphere Hurwitz cover of the correct type to the punctures in the same way as in \cref{sec:glue-at-branch-point}. However, the diagram does \emph{not} commute. The only commuting parts are the $\Xi_l-\xi_L$-square in the middle and the outer square. The horizontal map in the upper triangle
  \begin{equation*}
    \glue:\wh{\mcM}^{\square}_{g,k,h,n}(T)|_{\wh{\mu}^{-1}(L)} \lra \wt{\mcM}_{g,nd+k,h,2n}(\wt{T})
  \end{equation*}
  has a disjoint image from $\glue\circ\iota$. It is enough to check that the pullbacks of these forms to the total spaces of the $T^n$-bundles satisfy
  \begin{equation*}
    \Xi_L^*[\wh{\omega}|_{\wh{\mu}^{-1}(L)}]\overset{L\to 0}{\lra} \eta^*\pi^*\omega,
  \end{equation*}
  because the pull-backs are injective.

  Notice that $F$ is actually symplectic. The orbifold $\wt{\mcM}_{g,nd+k,h,2n}(\wt{T})$ is equipped with the (target) Weil--Petersson symplectic structure since it is locally diffeomorphic to $\mcM_{g,nd+k,h,2n}(\wt{T})$ and we can use Fenchel--Nielsen coordinates to see this. The horizontal $\glue$-functor satisfies $\glue^*\wt{\omega}=\wh{\omega}|_{\wh{\mu}^{-1}(L)}$ because of the commuting diagram in \cref{eq:diag-mod-spaces}. Thus it is enough to show that
  \begin{equation*}
    \glue\circ\, \Xi_L\lra \glue\circ\,\iota
  \end{equation*}
  for $L\to 0$. However, this is clear from the construction.
\end{proof}

\begin{prop}
  We have for the vector of rational first Chern classes
  \begin{align*}
    c_1\left((\wh{\mu})^{-1}\left(L\right)\right) &  \coloneqq c_1\left(\wh{\mcM}^{\square}_{g,k,h,n}(T)|_{\wh{\mu}^{-1}(L)}\lra\wh{\mcM}^{\square}_{g,k,h,n}(T)[L]\right) \\
    & \hspace{3cm} \in H^2\left(\left|\wh{\mcM}^{\square}_{g,k,h,n}(T)[L]\right|,\QQ^n\right)
  \end{align*}
  the limit
  \begin{align}
    \left\langle c_1\left((\wh{\mu})^{-1}\left(L\right)\right)^J\wedge\wh{\omega}_L^{3h-3+2n-|J|}, \right. & \left. [\wh{\mcM}^{\square}_{g,k,h,n}(T)[L]] \right\rangle \nonumber \\
    & \overset{L\to 0}{\lra} \left\langle c_1(E')^J\wedge\pi^*\omega,\left[\faktor{\mcM_{g,k,h,n}(T)\ltimes \mathring{E}'}{T^n}\right]\right\rangle \label{eq:limit}
  \end{align}
  for any multi-index $J\in\NN^n$.
  \label{prop:limit-chern-classes}
\end{prop}

\begin{proof}
  Using the isomorphisms
  \begin{equation*}
    \xi_L:\faktor{\mcM_{g,k,h,n}(T)\ltimes\mathring{E}}{T^n}\lra\wh{\mcM}^{\square}_{g,k,h,n}(T)[L]
  \end{equation*}
  we see $\xi_L^*c_1\left(\wh{\mu}^{-1}(L)\right)=c_1(E')$ because it is covered by a $T^n$-principal bundle isomorphism $\Xi_L$ and
  \begin{equation*}
    (\xi_L)_* \left[\faktor{\mcM_{g,k,h,n}(T)\ltimes \mathring{E}'}{T^n}\right]=[\wh{\mcM}^{\square}_{g,k,h,n}(T)[L]].
  \end{equation*}
  Thus
  \begin{align*}
    & \left\langle c_1\left((\wh{\mu})^{-1}\left(L\right)\right)^J\wedge\wh{\omega}_L^{3h-3+2n-|J|},[\wh{\mcM}^{\square}_{g,k,h,n}(T)[L]] \right\rangle \\
    & = \left\langle \left(\xi_L^*c_1\left((\wh{\mu})^{-1}\left(L\right)\right)\right)^J\wedge\left(\xi_L^*\wh{\omega}_L\right)^{3h-3+2n-|J|},\left[\faktor{\mcM_{g,k,h,n}(T)\ltimes \mathring{E}'}{T^n}\right] \right\rangle \\
    & = \left\langle c_1(E')^J\wedge\left(\xi_L^*\wh{\omega}_L\right)^{3h-3+2n-|J|},\left[\faktor{\mcM_{g,k,h,n}(T)\ltimes \mathring{E}'}{T^n}\right] \right\rangle. \\  
  \end{align*}
  By \cref{lem:limit-wwp} this converges to \cref{eq:limit}.
\end{proof}

\begin{lem}
  We have $\deg\wh{\ev}=K\cdot H_{g,k,h,n}(T)$ for $\wh{\ev}:\wh{\mcM}^{\square}_{g,k,h,n}(T)[L]\lra\wh{\mcM}^{\square}_{h,n}[L]$.
  \label{lem:degrees-evaluation-map-hurwitz}
\end{lem}

\begin{proof}
  Notice that we can repeat the construction from \cref{def:xi} for $\Xi_L$ to obtain an orbifold isomorphism
  \begin{equation*}
    \rho_L:\mcM_{h,n}\lra\wh{\mcM}^{\square}_{h,n}[L]
  \end{equation*}
  by adding a trivial pair of pants, doing the usual gluing construction with specified lengths $L$ but arbitrary twisting and then cutting along the corresponding hyperbolic geodesic. This gives a well-defined orbifold homomorphism which we can make an isomorphism by choosing the sets $\Lambda$ correspondingly. In total we have the following commuting diagram
  \begin{equation*}
    \xymatrix{
      \faktor{\mcM_{g,k,h,n}(T)\ltimes\mathring{E}'}{T^n} \ar[r]^{\xi_L} \ar[d]^{\pi} & \wh{\mcM}^{\square}_{g,k,h,n}(T)[L] \ar[dd]^{\wh{\ev}} \\
      \mcM_{g,k,h,n}(T) \ar[d]^{\ev} & \\
      \mcM_{h,n} \ar[r]^{\rho_L} & \wh{\mcM}^{\square}_{h,n}[L]
      }
  \end{equation*}
  where the horizontal functors are isomorphisms and the vertical ones are morphism coverings of degrees $K,H_{g,k,h,n}(T)$ and $\deg\wh{\ev}$, respectively, giving the result.
\end{proof}

\section{Applying Duistermaat--Heckman}

\label{sec:application-duistermaat-heckman}

Now we are in a setting that we can apply the earlier statements on the various bundles. First recall the Duistermaat--Heckman theorem for the $T^n$-action on $\wh{\mcM}^{\square}_{g,k,h,n}(T)$ at the value $\wh{\mu}^{-1}(L_0)$.\footnote{Recall that $\wh{\mu}^{-1}(L_0)$ corresponds to geodesic boundary lengths of $L_o\in\RR^n$.}

\begin{lem}
  By using Duistermaat--Heckman for the symplectic reductions of the Hamiltonian system $(\wh{\mcM}_{g,k,h,n}(T),\omega)$ to $\wh{\mcM}_{g,k,h,n}(T)[L]$ we see
\begin{equation}
  [\wh{\omega}_L]=[\wh{\omega}_{L_0}]+\frac{1}{2}\sum_{i=1}^nc_1\left((\wh{\mu})^{-1}\left(L_0\right)\right)_i\cdot K_i\left(L_i^2-(L_0)_i^2\right)
\label{eq:duistermaat}
\end{equation}
in $H^2\left(\wh{\mcM}_{g,k,h,n}(T)[L],\QQ\right)$ for any $L,L_0\in\RR^n$ close enough.
\end{lem}

\begin{thm}
  We have
  \begin{align*}
     K \cdot H_{g,k,h,n}(T)V_{h,n}(L) & = \sum_{\substack{\alpha\in\NN^n,\\ |\alpha|\leq 3h+n-3}}\sum_{\substack{\beta_1\in\NN^{I_1}\\ |\beta_1|=\alpha_1}}\cdots \sum_{\substack{\beta_n\in\NN^{I_n}\\ |\beta_n|=\alpha_n}} \frac{L^{2\alpha}l_1^{2(\beta_{\nu(1)})_1}\cdots l_k^{2(\beta_{\nu(k)})_k}}{(2d)^{|\alpha|}\beta_1!\cdots\beta_n!(3h+n-3-|\alpha|)!} \cdot \\
 & \phantom{=} \cdot\left\langle [\omega]^{3h+n-3-|\alpha|}{\fgt^*\psi_1}^{(\beta_{\nu(1)})_1}\cdots{\fgt^*\psi_k}^{(\beta_{\nu(k)})_k},[\mcM_{g,k,h,n}(T)] \right\rangle, \numberthis \label{eq:hurvolpsi}
  \end{align*}
  where $H_{g,k,h,n}(T)\in\QQ$ is the Hurwitz number, $V_{h,n}=\int_{\wh{\mcM}_{h,n}[L]}\frac{\wh{\omega}_L^{3h+n-3}}{(3h+n-3)!}$ the Weil--Petersson volume, $\psi_j\in H^2(\mcM_{g,k,h,n}(T),\QQ)$ the $\Psi$-classes on the moduli space of closed Hurwitz covers and $K=\prod_{i=1}^nK_i$.
  \label{thm:main-result}  
\end{thm}

\begin{proof}
Recall that $\wh{\omega}_L=\wh{\ev}^*{\wwp}|_L$ and that $\wh{\ev}$ is a branched morphism covering due to \cref{lem:ham-orb-covering-reduction} and thus
\begin{align*}
  \int_{\mcM_{g,k,h,n}(T)[L]}\wh{\omega}_L^{3h+n-3} & =\int_{\mcM_{g,k,h,n}(T)[L]}\wh{\ev}^*{\wwp}_L^{3h+n-3} \\
  & = (\deg\wh{\ev})\cdot\int_{\mcM_{h,n}[L]}{\wwp}_L^{3h+n-3} \\
  & = K\cdot H_{g,k,h,n}(T)\int_{\mcM_{h,n}[L]}{\wwp}_L^{3h+n-3}
\end{align*}
using \cref{lem:degrees-evaluation-map-hurwitz}.

Integrating \cref{eq:duistermaat} to the power of $3h+n-3$ over $\wh{\mcM}_{g,k,h,n}(T)[L]$ we therefore obtain for the left-hand side
\begin{align*}
  \left\langle\frac{[\wh{\omega}_L]^{3h+n-3}}{(3h+n-3)!},[\wh{\mcM}_{g,k,h,n}(T)[L]]\right\rangle&=\left\langle\frac{\wh{\ev}^*[\wwp|_L]^{3h+n-3}}{(3h+n-3)!},[\wh{\mcM}_{g,k,h,n}(T)[L]]\right\rangle\\
&=K H_{g,k,h,n}(T)\left\langle\frac{[\wwp|_L]^{3h+n-3}}{(3h+n-3)!},[\wh{\mcM}_{h,n}[L]]\right\rangle \numberthis \label{eq:dh_ev_covering}
\end{align*}
and for the right-hand side
\begin{align*}
  \left\langle\frac{[\wh{\omega}_L]^{3h+n-3}}{(3h+n-3)!}, \right. & \left. [\wh{\mcM}_{g,k,h,n}(T)[L]]\right\rangle = \sum_{\substack{\alpha\in\NN^n,\\ |\alpha|\leq 3h+n-3}}\left(\frac{1}{2}\right)^{|\alpha|}\frac{K^{\alpha}(L^{2\alpha}-(L_0)^{2\alpha})}{\alpha!(3h+n-3-|\alpha|)!}\cdot \\
  & \qquad \cdot\left\langle[\wh{\omega}_{L_0}]^{3h+n-3-|\alpha|}c_1\left((\wh{\mu})^{-1}\left(L_0\right)\right)^{\alpha},[\wh{\mcM}_{g,k,h,n}(T)[L]]\right\rangle \\
  & = \sum_{\substack{\alpha\in\NN^n,\\ |\alpha|\leq 3h+n-3}}\left(\frac{1}{2}\right)^{|\alpha|}\frac{K^{\alpha}L^{2\alpha}}{\alpha!(3h+n-3-|\alpha|)!}\cdot \\
  & \qquad \cdot\left\langle[\pi^*\omega]^{3h+n-3-|\alpha|}c_1\left(E'\right)^{\alpha},\left[\faktor{\mcM_{g,k,h,n}(T)\ltimes\mathring{E}'}{T^n}\right]\right\rangle \numberthis \label{eq:apply_dh}
\end{align*}
where we have used \cref{prop:limit-chern-classes} for the limit $L_0\lra 0$. Denote the index set $I_i:=\{j\in\{1,\ldots,k\}\mid\nu(j)=i\}$ for which we see by using \cref{cor:chern-class-relations}
\begin{align*}
  c_1(E')_i^{\alpha_i} & = \left(\frac{1}{dK_i}\sum_{j\in I_i}l_j^2\pi^*\fgt^*\psi_j\right)^{\alpha_i} \\
  & = \frac{1}{(d K_i)^{\alpha_i}}\sum_{\substack{\beta_i\in\NN^{I_i}\\ |\beta_i|=\alpha_i}}\frac{\alpha_i!}{\beta_i!}\prod_{j\in I_i}l_j^{2(\beta_i)_j}{\pi^*\fgt^*\psi_j}^{(\beta_i)_j}.
\end{align*}
Thus we obtain for the expression in \cref{eq:apply_dh}
\begin{align*}
  \sum_{\substack{\alpha\in\NN^n,\\ |\alpha|\leq 3h+n-3}} \left(\frac{1}{2}\right)^{|\alpha|} & \frac{K^{\alpha}L^{2\alpha}}{\alpha!(3h+n-3-|\alpha|)!} \left\langle\pi^*[\omega]^{3h+n-3-|\alpha|} \prod_{i=1}^n\sum_{\substack{\beta_i\in\NN^{I_i},\\ |\beta_i|=\alpha_i}}\frac{\alpha_i!}{\beta_i!(d K_i)^{\alpha_i}}\cdot \right. \\
& \left. \prod_{j\in I_i} l_{j}^{2(\beta_i)_j} \pi^*\fgt^*\psi_{j}^{(\beta_i)_j},\left[\faktor{\mcM_{g,k,h,n}(T)\ltimes\mathring{E}'}{T^n}\right]\right\rangle \\
  = \sum_{\substack{\alpha\in\NN^n,\\ |\alpha|\leq 3h+n-3}} \left(\frac{1}{2}\right)^{|\alpha|} & \frac{L^{2\alpha}}{\alpha!(3h+n-3-|\alpha|)!d^{|\alpha|}} \left\langle[\omega]^{3h+n-3-|\alpha|} \prod_{i=1}^n\sum_{\substack{\beta_i\in\NN^{I_i},\\ |\beta_i|=\alpha_i}}\frac{\alpha_i!}{\beta_i!}\cdot \right. \\
& \left. \prod_{j\in I_i} l_{j}^{2(\beta_i)_j} \fgt^*\psi_{j}^{(\beta_i)_j},[\mcM_{g,k,h,n}(T)\right\rangle,
\end{align*}
where we have used \cref{cor:pi-top-covering}.

Note that the multi-indices $\beta_i$ are indexed by the elements of $I_i$, thus $(\beta_i)_j$ only exists for those $j$ such that $\nu(j)=i$. In the following intermediate step we denote by $(I_1)[j]$ the $j$-th element of $I_1$. Then we can see
\begin{align*}
  \prod_{i=1}^n & \sum_{\substack{\beta_i\in\NN^{I_i},\\ |\beta_i|=\alpha_i}} \frac{\alpha_i!}{\beta_i!}\prod_{j\in I_i}l_{j}^{2(\beta_i)_j} {\fgt^*\psi_{j}}^{(\beta_i)_j} = \\
 & = \sum_{\substack{\beta_1\in\NN^{I_1}\\ |\beta_1|=\alpha_1}}\frac{\alpha_1!}{\beta_1!}l_{(I_1)[1]}^{2(\beta_1)_{(I_1)[1]}}{\fgt^*\psi}^{(\beta_1)_{(I_1)[1]}}_{(I_1)[1]}\cdots l_{(I_1)[|I_1|]}^{2(\beta_1)_{(I_1)[|I_1|]}}{\fgt^*\psi}^{(\beta_1)_{(I_1)[|I_1|]}}_{(I_1)[|I_1|]}\cdots  \\
 & \phantom{=}  \cdots \sum_{\substack{\beta_n\in\NN^{I_n}\\ |\beta_n|=\alpha_n}}\frac{\alpha_n!}{\beta_n!}l_{(I_n)[1]}^{2(\beta_n)_{(I_n)[1]}}{\fgt^*\psi}^{(\beta_n)_{(I_n)[1]}}_{(I_n)[1]}\cdots l_{(I_n)[|I_n|]}^{2(\beta_n)_{(I_n)[|I_n|]}}{\fgt^*\psi}^{(\beta_n)_{(I_n)[|I_n|]}}_{(I_n)[|I_n|]} \\
& = \sum_{\substack{\beta_1\in\NN^{I_1}\\ |\beta_1|=\alpha_1}}\cdots \sum_{\substack{\beta_n\in\NN^{I_n}\\ |\beta_n|=\alpha_n}}\frac{\alpha!}{\beta_1!\cdots\beta_n!}l_1^{2(\beta_{\nu(1)})_1}\cdots l_k^{2(\beta_{\nu(k)})_k}{\fgt^*\psi_1}^{(\beta_{\nu(1)})_1}\cdots{\fgt^*\psi_k}^{(\beta_{\nu(k)})_k},
\end{align*}
because in the second line each $j\in\{1,\ldots,k\}$ appears in one of the $I_i$'s, namely the one such that $\nu(j)=i$. 

Now we can write
\begin{align*}
  K \cdot H_{g,k,h,n}(T)V_{h,n}(L) & = \sum_{\substack{\alpha\in\NN^n,\\ |\alpha|\leq 3h+n-3}}\sum_{\substack{\beta_1\in\NN^{I_1}\\ |\beta_1|=\alpha_1}}\cdots \sum_{\substack{\beta_n\in\NN^{I_n}\\ |\beta_n|=\alpha_n}} \frac{L^{2\alpha}l_1^{2(\beta_{\nu(1)})_1}\cdots l_k^{2(\beta_{\nu(k)})_k}}{(2d)^{|\alpha|}\beta_1!\cdots\beta_n!(3h+n-3-|\alpha|)!} \cdot \\
 & \phantom{=} \cdot\left\langle [\omega]^{3h+n-3-|\alpha|}{\fgt^*\psi_1}^{(\beta_{\nu(1)})_1}\cdots{\fgt^*\psi_k}^{(\beta_{\nu(k)})_k},[\mcM_{g,k,h,n}(T)] \right\rangle,
\end{align*}
where we denoted the Weil--Petersson volume of $\wh{\mcM}_{h,n}[L]$ by $V_{h,n}(L)$ .
\end{proof}

\begin{rmk}
This equation relates Hurwitz numbers, Weil--Petersson volumes and certain integrals of $\Psi$-classes over the moduli space of Hurwitz covers. We will apply the following corollary of \cref{thm:main-result} in the next section.
\end{rmk}

\begin{cor}
  We can rewrite \cref{eq:hurvolpsi} in degree $|\alpha|=3h+n-3$ as
  \begin{align*}
    & K \cdot H_{g,k,h,n}(T)V_{h,n}(L)[3h+n-3] = \\
    & \sum_{\substack{\alpha\in\NN^n,\\ |\alpha|= 3h+n-3}} \sum_{\substack{\beta_1\in\NN^{I_1}\\ |\beta_1|=\alpha_1}}\cdots \sum_{\substack{\beta_n\in\NN^{I_n}\\ |\beta_n|=\alpha_n}} \frac{L^{2\alpha}l_1^{2(\beta_{\nu(1)})_1}\cdots l_k^{2(\beta_{\nu(k)})_k}}{(2d)^{3h+n-3}\beta_1!\cdots\beta_n!} \left\langle {\psi_1}^{(\beta_{\nu(1)})_1}\cdots{\psi_k}^{(\beta_{\nu(k)})_k},D_{g,k,h,n}(T) \right\rangle, \numberthis \label{eq:hvpsi_top}
  \end{align*}
  where $V_{h,n}(L)[3h+n-3]$ denotes the homogeneous part of the polynomial of degree $3h+n-3$.
  \label{cor:main-result}
\end{cor}

\begin{rmk}
  Recall from \cite{mirzakhani_weil-petersson_2007} that $V_{h,n}(L)$ is a polynomial in $L$ whose coefficients are given by $\Psi$-intersections on Deligne--Mumford space. Thus we can rewrite \cref{cor:main-result} as an equation between pure $\Psi$-intersection numbers and $\Psi$-intersection numbers with the Hurwitz class. As the former numbers can be easily computed e.g.\ by using Mirzakhanis recursion relation for Weil--Petersson volumes we can deduce concrete equations for the $\Psi$-intersection numbers of the Hurwitz class. However, this particular conclusion from \cref{thm:main-result} we could have also obtained directly from \cref{cor:chern-class-relations} on pull-backs of $\Psi$-classes without applying Duistermaat--Heckman. 
\end{rmk}

\chapter{Applications}

\label{sec:appl-exampl}

This section calculates a few concrete examples for \cref{thm:main-result}. Note that the equation involves Hurwitz numbers, Weil--Petersson volumes of moduli spaces and pairings of the Hurwitz class with $\psi$-classes on the source moduli space. These pairings seem to be most difficult to understand, so we will calculate Hurwitz numbers with the help of a computer and deduce the Weil--Petersson volumes from the McShane identity.

\section{Recollections and Weil-Petersson Volumes}

Recall that we denote by
\begin{equation*}
  V_{h,n}(L)\coloneqq\int_{\mcM_{h,n}[L]}\frac{\wwp^{3h-3+n}}{(2\pi)^{3h-3+n}}
\end{equation*}
the Weil--Petersson volume of the moduli space of bordered Riemann surfaces. It can be calculated explicitly via Mirzakhanis recursion relation in \cite{mirzakhani_weil-petersson_2007}. As the combinatorial calculations become somewhat involved rather quickly we calculated the volumes via a SAGE/Python program whose source code is also included in \cref{appendix:calculation-weil-petersson-volumes}.

\index{Hurwitz Number!Standard}
\index{Hurwitz Number}

Also recall from \cref{sec:hurwitz-numbers} and \cref{sec:main-results} that we have Hurwitz numbers $H_{g,k,h,n}(T)$ and standard Hurwitz numbers $\mcH_h(T_1,\ldots,T_n)$ which are related by a factor of $\mfK$ corresponding to the number of possible enumerations of the fibres, i.e.\

\begin{equation*}
  \mfK\coloneqq\prod_{i=1}^n\prod_{u=1}^d\left(\# \{1\leq j \leq k\mid \nu(j)=i, l_j=u\}\right)!
\end{equation*}

Also we denote by
\begin{equation*}
  D_{g,k,h,n}(T)=\fgt_*[\mcM_{g,k,h,n}(T)]\in H_{6h-6+2n}(|\mcM_{g,k}|,\QQ)
\end{equation*}
the Hurwitz class. Another combinatorial term appearing is $K=\prod_{j=1}^kK_i$ where
\begin{equation*}
  K_i=\lcm\{l_j\mid \nu(j)=i\}.
\end{equation*}
\index{Hurwitz Class}

\section{Examples}

In the following examples the numbers at the symbols stand for the index of the corresponding marked point. The degree of the Hurwitz cover at $q_j$, i.e.\ the integer $l_j$ will be drawn with an obvious pictogram. This means that the map $\nu:\{1,\ldots,k\}\lra\{1,\ldots,n\}$ can be read off from the diagram by following the arrow.

\subsection{The Case \texorpdfstring{$\boldsymbol{g=h=0, k=4,n=3, d=2}$}{gh0k4n3d2}}

\label{ex:0-4-0-3}

Assume the branching profile looks as follows:

\begin{figure}[!ht]
    \centering
    \def\svgwidth{0.6\textwidth}
    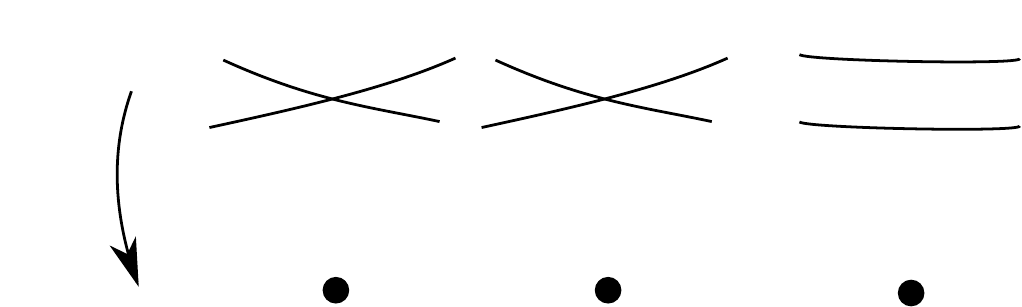
\end{figure}

Computer calculations show that $\mcH_0(2,2,1+1)=\frac{1}{2}$ as well as $\mfK=2!=2$ and thus $H_{0,4,0,3}(T)=1$. There is one Hurwitz cover with two automorphisms. Furthermore, we have $K=K_1\cdot K_2\cdot K_3=2\cdot 2\cdot 1=4$ for the factor corresponding to the least common multiplies.

We also have $V_{0,3}(L)=1$ and $\dim\mcM_{g,k,h,n}(T)=6h-6+2n=0$, i.e.\ $D_{0,4,0,3}(T)\in H_0(|\mcM_{0,4}|,\QQ)\cong\QQ$ is just a $0$-cycle on $|\mcM_{0,4}|\cong \CC P^1$ and \cref{thm:main-result} or rather \cref{cor:main-result} becomes
\begin{equation*}
  \langle 1, D_{0,4,0,3}(T)\rangle=4
\end{equation*}
and therefore
\begin{equation*}
  D_{0,4,0,3}(T)=4[\pt].
\end{equation*}

\subsection{The Case \texorpdfstring{$\boldsymbol{g=1, h=0, k=4, n=4, d=2}$}{g1h0k4n4d2}}

Now we consider the branching profile

\begin{figure}[!ht]
    \centering
    \def\svgwidth{0.9\textwidth}
    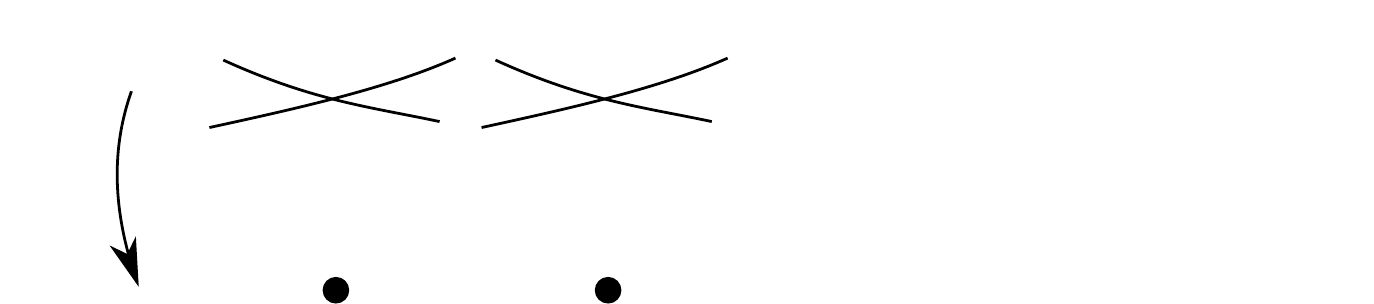
\end{figure}

Computer calculations show $\mcH_{0}(2,2,2,2)=\frac{1}{2}$ as there is one Hurwitz cover with two automorphisms, one of them the identity and the other one the sheet change. Together with $\mfK=1$ we get $H_{1,4,0,4}(T)=\frac{1}{2}$.

We have $V_{0,4}(L)=\frac{1}{2}(4\pi^2+L_1^2+L_2^2+L_3^2+L_4^2)$. Also, in degree $|\alpha|=1$ in \cref{eq:hvpsi_top} only one of the $\beta_1,\ldots,\beta_4$ can be nonzero and thus equal to one. As $l_1=l_2=l_3=l_4=2$ we obtain $K=2^4=16$ and thus
\begin{align*}
  4\sum_{i=1}^4L_i^2&=\sum_{i=1}^4\frac{L_i^24}{4}\langle\psi_i,D_{1,4,0,4}(T)\rangle \\
 &=\sum_{i=1}^4L_i^2\langle\psi_i,D_{1,4,0,4}(T)\rangle
\end{align*}
and therefore $\langle\psi_i,D_{1,4,0,4}(T)\rangle=4$ for $i=1,\ldots,4$ by comparing coefficients of the polynomials.

\subsection{The Case \texorpdfstring{$\boldsymbol{g=h=0, k=10, n=4,d=4}$}{gh0k10n4d4}}

Now we consider the following branching profile:

\begin{figure}[!ht]
    \centering
    \def\svgwidth{0.8\textwidth}
    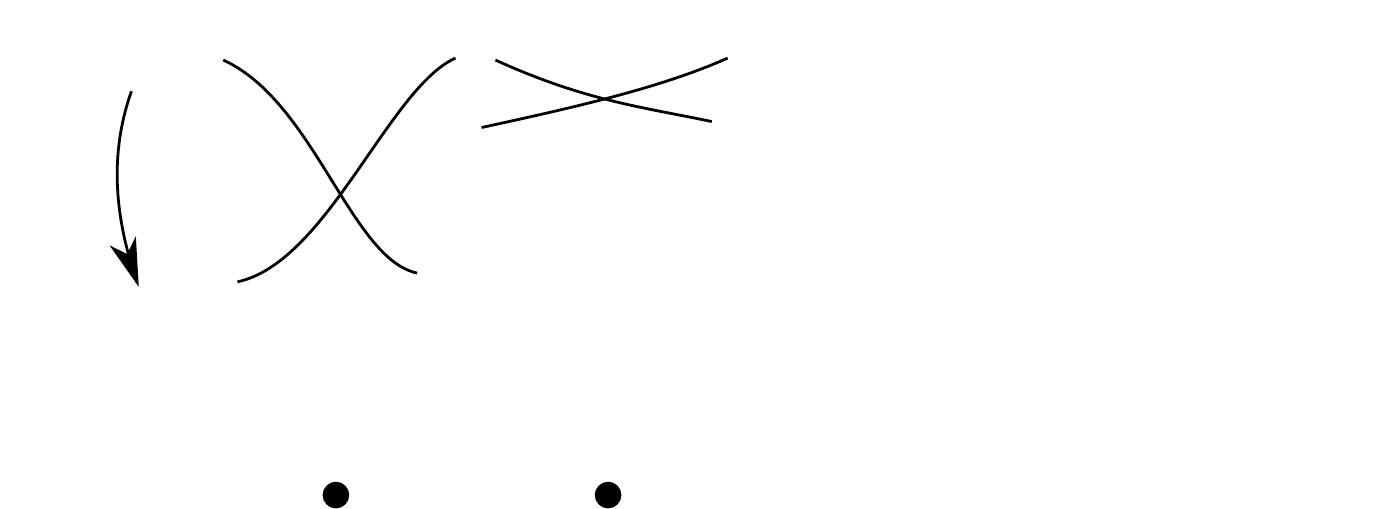
\end{figure}

Computer calculations show that $\mcH_0(3,2+1+1,2+1+1,2+1+1)=4$ with four Hurwitz covers, none of which have non-trivial automorphisms. As $\mfK=(2!)^3=8$ we have $H_{0,10,0,4}(T)=32$.

Here we have $K=2^3\cdot4=32$ and $V_{0,4}=\frac{1}{2}\left(4\pi^2+\sum_{i=1}^4L_i^2\right)$. Since $|\alpha|=1$ we can instead sum over the index of the component of $\alpha$ which is one instead of zero. Then we obtain
\begin{align*}
  512\sum_{i=1}^4L_i^2&=\sum_{i=1}^4\sum_{\substack{\beta_i\in\NN^{I_i}\\ |\beta_i|=1}}\frac{L_i^2l_1^{2(\beta_1)_1}l_2^{2(\beta_2)_2}l_5^{2(\beta_3)_5}l_8^{2(\beta_4)_8}}{8\cdot 1}\langle\psi_1^{(\beta_{\nu(1)})_1}\cdots\psi_{10}^{(\beta_{\nu(10)})_{10}},D_{0,10,0,4}(T)\rangle \\
  & = 2L_1^2\langle\psi_1,D_{0,10,0,4}(T)\rangle + \\
  & \phantom{=} + L_2^2\left(\frac{1}{2}\langle\psi_2,D_{0,10,0,4}(T)\rangle+\frac{1}{8}\langle\psi_3,D_{0,10,0,4}(T)\rangle+\frac{1}{8}\langle\psi_4,D_{0,10,0,4}(T)\rangle\right)+ \\
  & \phantom{=} + L_3^2\left(\frac{1}{2}\langle\psi_5,D_{0,10,0,4}(T)\rangle+\frac{1}{8}\langle\psi_6,D_{0,10,0,4}(T)\rangle+\frac{1}{8}\langle\psi_7,D_{0,10,0,4}(T)\rangle\right)+ \\
  & \phantom{=} + L_4^2\left(\frac{1}{2}\langle\psi_8,D_{0,10,0,4}(T)\rangle+\frac{1}{8}\langle\psi_9,D_{0,10,0,4}(T)\rangle+\frac{1}{8}\langle\psi_{10},D_{0,10,0,4}(T)\rangle\right)
\end{align*}
for $D_{0,10,0,4}(T)\in H_2(|\mcM_{0,10}|,\QQ)$ implying e.g.\ $\langle \psi_1,D_{0,10,0,4}(T)\rangle=256$ from comparing the coefficients in front of $L_1^2$.

\subsection{The Case \texorpdfstring{$\boldsymbol{g=h=0, k=8, n=4, d=3}$}{gh0k8n4d3}}

\label{ex:0-8-0-4}

Next we assume the branching profile looks like

\begin{figure}[!ht]
    \centering
    \def\svgwidth{0.8\textwidth}
    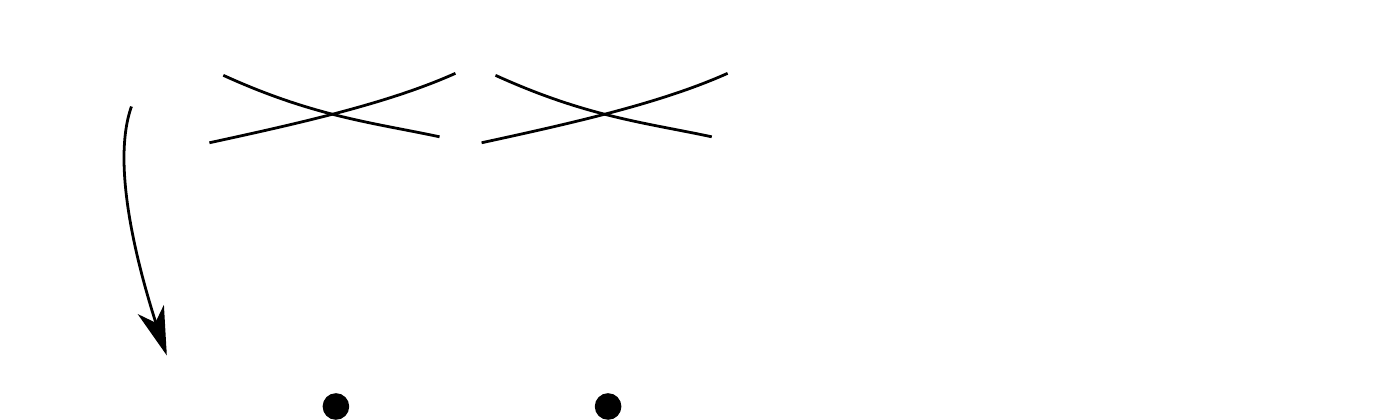
\end{figure}

Again, computer calculations show that $\mcH_0(2+1,2+1,2+1,2+1)=\frac{9}{2}$ as we have five Hurwitz covers, four of which have only the identity as an automorphism and the fifth one with two automorphisms. Since $\mfK=1$ we have therefore $H_{0,8,0,4}(T)=\frac{9}{2}$.

Also we see $K=2^4=16, I_i=\{2i-1,2i\}$ and $V_{0,4}=\frac{1}{2}\left(4\pi^2+\sum_{i=1}^4L_i^2\right)$. Thus
\begin{align*}
  36\sum_{i=1}^4L_i^2&=\sum_{i=1}^4\sum_{\substack{\beta_i\in\NN^{I_i}\\ |\beta_i|=1}}\frac{L_i^2l_1^{2(\beta_1)_1}l_3^{2(\beta_2)_3}l_5^{2(\beta_3)_5}l_7^{2(\beta_4)_7}}{(2\cdot 3)^1\cdot 1}\langle\psi_1^{(\beta_{\nu(1)})_1}\cdots\psi_{10}^{(\beta_{\nu(10)})_{10}},D_{0,8,0,4}(T)\rangle \\
  & = L_1^2\left(\frac{2}{3}\langle\psi_1,D_{0,8,0,4}(T)\rangle+\frac{1}{6}\langle\psi_2,D_{0,8,0,4}(T)\rangle\right) + \\
  & \phantom{=} + L_2^2\left(\frac{2}{3}\langle\psi_3,D_{0,8,0,4}(T)\rangle+\frac{1}{6}\langle\psi_4,D_{0,8,0,4}(T)\rangle\right)+\\
  & \phantom{=} + L_3^2\left(\frac{2}{3}\langle\psi_5,D_{0,8,0,4}(T)\rangle+\frac{1}{6}\langle\psi_6,D_{0,8,0,4}(T)\rangle\right) \\
  & \phantom{=} + L_4^2\left(\frac{2}{3}\langle\psi_7,D_{0,8,0,4}(T)\rangle+\frac{1}{6}\langle\psi_8,D_{0,8,0,4}(T)\rangle\right)
\end{align*}
for $D_{0,8,0,4}(T)\in H_2(|\mcM_{0,8}|,\QQ)$ which shows for example $\langle 4\psi_1+\psi_2,D_{0,8,0,4}(T)\rangle=216$.

\subsection{The Case \texorpdfstring{$\boldsymbol{g=h=0, k=11, n=5, d=3}$}{gh0k11n5d3}}

The next example takes place in a higher dimension than before and is the same as \cref{ex:0-8-0-4} with an added trivial fibre. Assume the branching profile is

\begin{figure}[!ht]
    \centering
    \def\svgwidth{0.9\textwidth}
    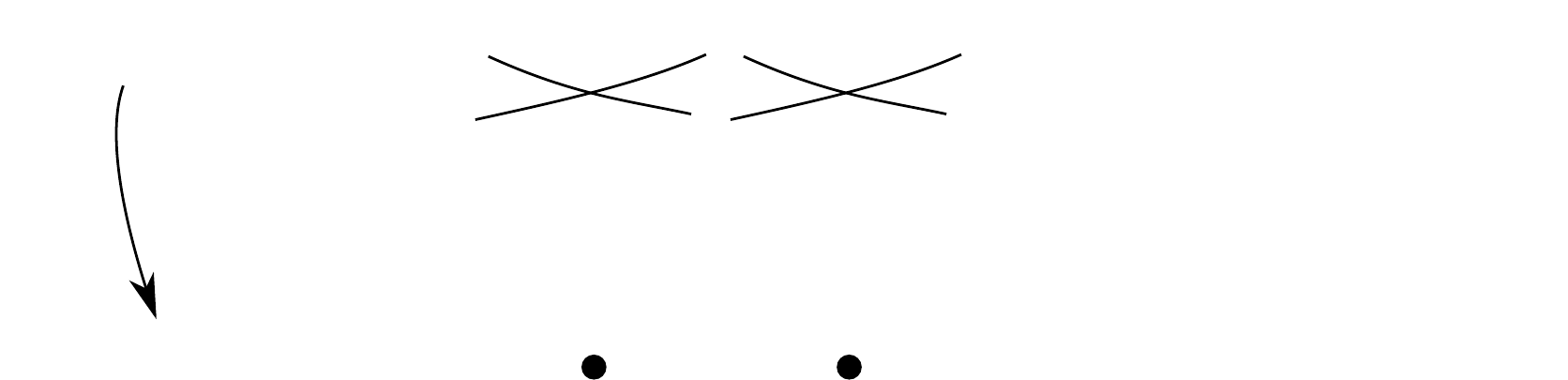
\end{figure}

As usual computer calculations show that $\mcH_0(1+1+1,2+1,2+1,2+1,2+1)=\frac{9}{2}$ which is expected because we have of course the same Hurwitz covers as in \cref{ex:0-8-0-4}. However, we now have $\mfK=3!=6$ and therefore $H_{0,11,0,5}(T)=27$.

Here we have $K=1\cdot 2^4=16$ and $V_{0,5}[2]=\frac{1}{8}\sum_{i=1}^5L_i^4+\frac{1}{2}\sum_{\substack{i,j=1}\\ i\neq j}^5L_i^2L_j^2$ as is seen by the calculation in \cref{sec:calculation-v-05}. Therefore
\begin{align*}
  432\left(\frac{1}{8}\sum_{i=1}^5L_i^4+\frac{1}{2}\sum_{\substack{i,j=1\\ i\neq j}}^5L_i^2L_j^2\right)=&\sum_{\substack{\alpha\in\NN^5\\ |\alpha|=2}}\sum_{\substack{\beta_1\in\NN^{\{1,2,3\}}\\|\beta_1|=\alpha_1}}\sum_{\substack{\beta_2\in\NN^{\{4,8\}}\\|\beta_2|=\alpha_2}}\sum_{\substack{\beta_3\in\NN^{\{5,9\}}\\|\beta_3|=\alpha_3}}\sum_{\substack{\beta_4\in\NN^{\{6,10\}}\\|\beta_4|=\alpha_4}}\sum_{\substack{\beta_5\in\NN^{\{7,11\}}\\|\beta_5|=\alpha_5}} \\[0.4em]
& \phantom{=} \frac{L^{2\alpha}4^{(\beta_2)_4+(\beta_3)_6+(\beta_4)_8+(\beta_5)_{10}}}{6^2\beta_1!\beta_2!\beta_3!\beta_4!\beta_5!}\left\langle\prod_{i=1}^{11}\psi_{i}^{(\beta_{\nu(i)})_i},D_{0,11,0,5}(T)\right\rangle,
\end{align*}
where $D_{0,11,0,5}\in H_4(|\mcM_{0,11}|,\QQ)$.

Expanding this expression with the help of a computer shows that the right hand side contains for example the following summands
\begin{align*}
&  L_1^4\left(\frac{1}{72}\langle\psi_1^2,D_{0,11,0,5}(T)\rangle+\frac{1}{72}\langle\psi_2^2,D_{0,11,0,5}(T)\rangle+\frac{1}{72}\langle\psi_3^2,D_{0,11,0,5}(T)\rangle\right.+ \\
& \left. +\frac{1}{36}\langle\psi_1\psi_2,D_{0,11,0,5}(T)\rangle+\frac{1}{36}\langle\psi_1\psi_3,D_{0,11,0,5}(T)\rangle+\frac{1}{36}\langle\psi_2\psi_3,D_{0,11,0,5}(T)\rangle\right)+ \\
& + L_2^4\left(\frac{2}{9}\langle\psi_4^2,D_{0,11,0,5}(T)\rangle+\frac{1}{9}\langle\psi_4\psi_5,D_{0,11,0,5}(T)\rangle+\frac{1}{72}\langle\psi_5^2,D_{0,11,0,5}(T)\rangle\right) + \\
& + L_2^2L_3^2\left(\frac{4}{9}\langle\psi_4\psi_6,D_{0,11,0,5}(T)\rangle +\frac{1}{9}\langle\psi_5\psi_6,D_{0,11,0,5}(T)\rangle+\frac{1}{9}\langle\psi_4\psi_7,D_{0,11,0,5}(T)\rangle + \right. \\
& + \left. \frac{1}{36}\langle\psi_5\psi_7,D_{0,11,0,5}(T)\rangle\right) + \ldots
\end{align*}
This proves for example that $\langle(\psi_1+\psi_2+\psi_3)^2,D_{0,11,0,5}(T)\rangle=3888$ and similar formulas by comparing the coefficients in front of $L_1^4$ and simplifying the polynomial of $\Psi$-classes.

\subsection{The Case \texorpdfstring{$\boldsymbol{g=h=1, k=4, n= 2, d=2}$}{gh1k4n2d2}}

Assume the branching profile looks as follows:

\begin{figure}[!ht]
    \centering
    \def\svgwidth{0.5\textwidth}
    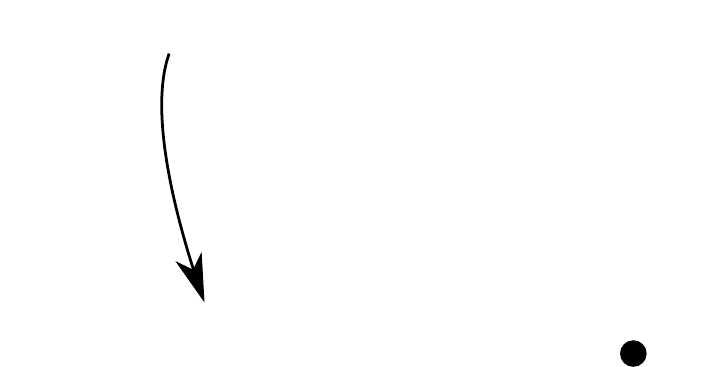
\end{figure}

In this case $K=1$ and $V_{1,2}(L)=\frac{1}{192}\left(L_1^4+2L_1^2L_2^2+L_2^4\right)+\frac{\pi^2}{12}(L_1^2+L_2^2)+\frac{\pi^4}{4}$ as is calculated in \cref{sec:calculation-wp-volume-1-2}. Also we have $\mcH_1(1+1,1+1)=2$ using a computer program which shows that there are four Hurwitz covers having each two automorphisms. Notice that this case corresponds to an unbranched two-fold covering of the torus by a torus so you can actually write down these covers explicitly. Since $\mfK=(2!)^2=4$ we obtain $H_{1,4,1,2}(T)=8$ and thus
\begin{align*}
  \frac{1}{24} & \left(L_1^4+2L_1^2L_2^2+L_2^4\right)= \\
&= \sum_{\substack{\alpha_1,\alpha_2\in\NN\\ \alpha_1+\alpha_2=2}}\sum_{\substack{\beta_1,\beta_2\in\NN\\ \beta_1+\beta_2=\alpha_1}}\sum_{\substack{\beta_3,\beta_4\in\NN\\ \beta_3+\beta_4=\alpha_2}} \frac{L_1^{2\alpha_1}L_2^{2\alpha_2}}{16\cdot \beta_1!\beta_2!\beta_3!\beta_4!}\langle\psi_1^{\beta_1}\psi_2^{\beta_2}\psi_3^{\beta_3}\psi_4^{\beta_4},D_{1,4,1,2}(T)\rangle \\
&=\frac{1}{16}\left[L_1^4\left(\frac{1}{2}\langle\psi_1^2,D_{1,4,1,2}(T)\rangle+\langle\psi_1\psi_2,D_{1,4,1,2}(T)\rangle+\frac{1}{2}\langle\psi_2^2,D_{1,4,1,2}(T)\rangle\right)\right.+ \\
&\phantom{=} +L_1^2L_2^2\left(\langle\psi_1\psi_3,D_{1,4,1,2}(T)\rangle+\langle\psi_1\psi_4,D_{1,4,1,2}(T)\rangle \right. + \\
&\phantom{= +L_1^2L_2^2(} \left. +\langle\psi_2\psi_3,D_{1,4,1,2}(T)\rangle+\langle\psi_2\psi_4,D_{1,4,1,2}(T)\rangle\right)+ \\
&\phantom{=} \left.+L_2^4\left(\frac{1}{2}\langle\psi_3^2,D_{1,4,1,2}(T)\rangle+\langle\psi_3\psi_4,D_{1,4,1,2}(T)\rangle+\frac{1}{2}\langle\psi_4^2,D_{1,4,1,2}(T)\rangle\right)\right],
\end{align*}
where $D_{1,4,1,2}(T)\in H_4(|\mcM_{1,4}|,\QQ)$. Thus we have e.g.\ $\frac{4}{3}=\langle(\psi_1+\psi_2)^2,D_{1,4,1,2}(T)\rangle$.

\subsection{The Case \texorpdfstring{$\boldsymbol{g=1,h=0,k=6,n=4,d=3}$}{g1h0k6n4d3}}

As in \cref{ex:0-8-0-4} we have $V_{0,4}(L)=\frac{1}{2}\left(4\pi^2+\sum_{i=1}^4L_i^2\right)$. However with the parameters above we have two possible subcases of distributions of critical points.

\subsubsection{The Subcase \texorpdfstring{$\boldsymbol{\{3,3,2,2,1,1\}}$}{332211}}

Assume first that the branching profile $T$ looks as follows:

\begin{figure}[!ht]
    \centering
    \def\svgwidth{0.6\textwidth}
    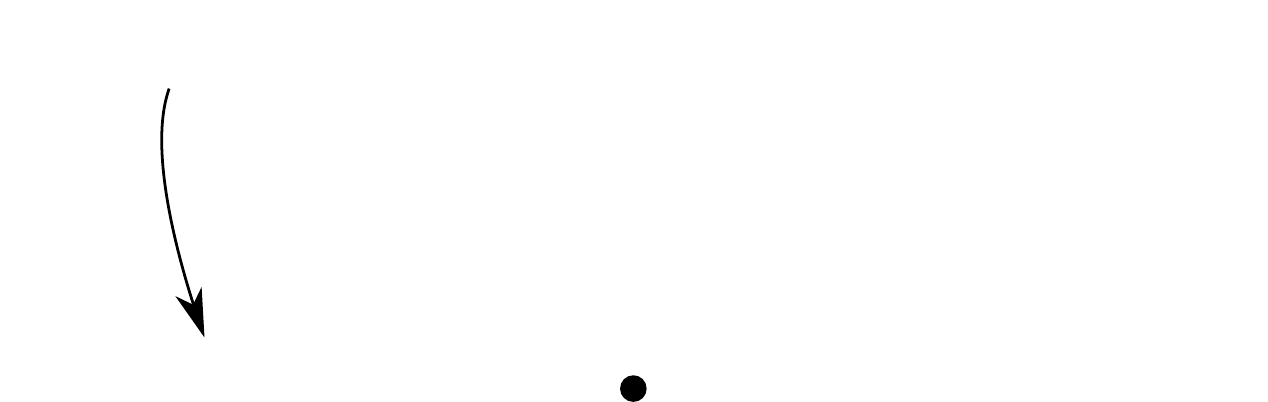
\end{figure}

Computer calculations show $\mcH_0(3,3,2+1,2+1)=2$ because there are two Hurwitz covers with just one automorphism and as $\mfK=1$ we have $H_{1,6,0,4}(T)=2$.

Then $K=3^2\cdot 2^2=36$ such that
\begin{align*}
  36\sum_{i=1}^4L_i^2&=\sum_{i=1}^4\sum_{\substack{\beta_i\in\NN^{I_i}\\ |\beta_i|=1}} \frac{L_i^2l_1^{2(\beta_1)_1}l_2^{2(\beta_2)_2}l_3^{2(\beta_3)_3}l_5^{2(\beta_4)_5}}{6^1}\left\langle\psi_1^{(\beta_{\nu(1)})_1}\cdots\psi_6^{(\beta_{\nu(6)})_6},D_{1,6,0,4}(T)\right\rangle \\
&=\frac{3}{2}\langle\psi_1,D_{1,6,0,4}(T)\rangle L_1^2+\frac{3}{2}\langle\psi_2,D_{1,6,0,4}(T)\rangle L_2^2+ \\
& \qquad + \left(\frac{2}{3}\langle\psi_3,D_{1,6,0,4}(T)\rangle+\frac{1}{6}\langle\psi_4,D_{1,6,0,4}(T)\rangle\right)L_3^2+ \\
& \qquad + \left(\frac{2}{3}\langle\psi_5,D_{1,6,0,4}(T)\rangle+\frac{1}{6}\langle\psi_6,D_{1,6,0,4}(T)\rangle\right)L_4^2,
\end{align*}
where $D_{1,6,0,4}(T)\in H_2(|\mcM_{1,6}|,\QQ)$. This implies e.g.\ $\langle\psi_1,D_{1,6,0,4}(T)\rangle=24$ as well as $\langle 4\psi_5+\psi_6,D_{1,6,0,4}(T)\rangle=216$ by comparing coefficients in front of $L_1^2$ and $L_4^2$, respectively.

\subsubsection{The Subcase \texorpdfstring{$\boldsymbol{\{3,3,3,1,1,1\}}$}{333111}}

\label{ex:1-6-0-4}

Assume now that the branching profile $T'$ looks like

\begin{figure}[!ht]
    \centering
    \def\svgwidth{0.6\textwidth}
    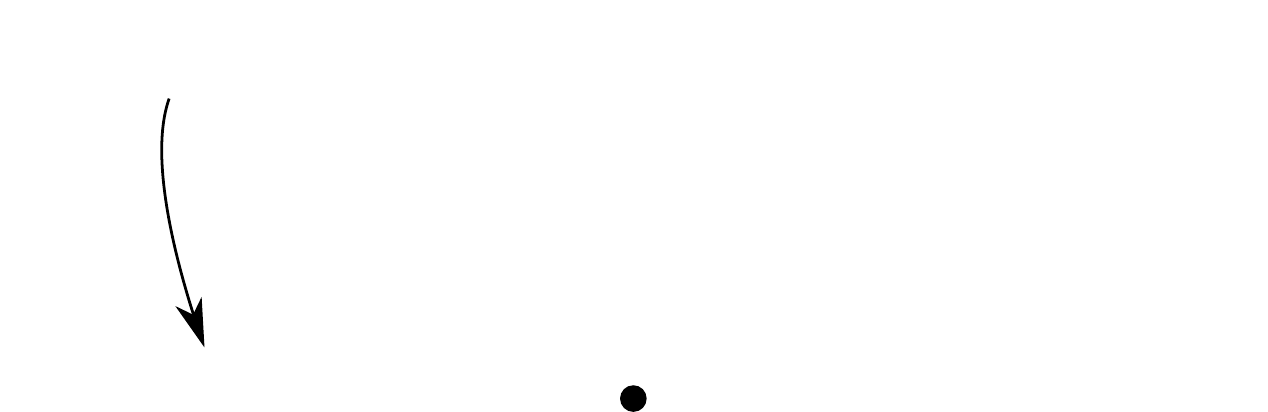
\end{figure}

Here, computer calculations show that $H_0(3,3,3,1+1+1)=\frac{1}{3}$ because there is only one Hurwitz cover with three automorphisms and since $\mfK=6$ we have $H_{1,6,0,4}(T')=2$. Note that this is the same as \cref{ex:1-3-0-3} but with one added trivial fibre.

Then we see $K=27$ such that
\begin{align*}
  27\sum_{i=1}^4L_i^2&=\sum_{i=1}^4\sum_{\substack{\beta_i\in\NN^{I_i}\\ |\beta_i|=1}} \frac{L_i^2l_1^{2(\beta_1)_1}l_2^{2(\beta_2)_2}l_3^{2(\beta_3)_3}}{6^1}\left\langle\psi_1^{(\beta_{\nu(1)})_1}\cdots\psi_6^{(\beta_{\nu(6)})_6},D_{1,6,0,4}(T')\right\rangle \\
&=\frac{3}{2}\langle\psi_1,D_{1,6,0,4}(T')\rangle L_1^2+\frac{3}{2}\langle\psi_2,D_{1,6,0,4}(T')\rangle L_2^2+\frac{3}{2}\langle\psi_3,D_{1,6,0,4}(T')\rangle L_3^2+ \\
&\phantom{=}+\frac{1}{6}(\langle\psi_4,D_{1,6,0,4}(T')\rangle+\langle\psi_5,D_{1,6,0,4}(T')\rangle+\langle\psi_6,D_{1,6,0,4}(T')\rangle)L_4^2
\end{align*}
where $D_{1,6,0,4}(T')\in H_2(|\mcM_{1,6}|,\QQ)$ implying for example $\langle\psi_1,D_{1,6,0,4}(T')\rangle=18$.

\subsection{The Case \texorpdfstring{$\boldsymbol{g=1,h=0,k=3,n=3,d=3}$}{g1h0k3n3d3}}

\label{ex:1-3-0-3}

This is a somewhat special case as we are looking at a target sphere with three punctures and therefore a target moduli space consisting of just a single point without automorphisms. Let us consider the following branching profile anyway

\begin{figure}[!ht]
    \centering
    \def\svgwidth{0.6\textwidth}
    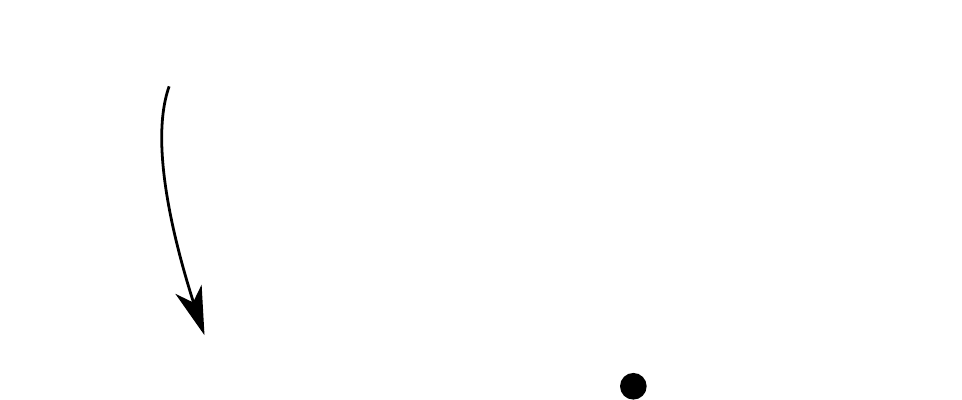
\end{figure}

Computer calculations show that $\mcH_0(3,3,3)=\frac{1}{3}$ as in \cref{ex:1-6-0-4} because there is one Hurwitz cover with three automorphisms. Since $\mfK=1$ we have $H_{1,3,0,3}(T)=\frac{1}{3}$ and $K=3^3=27$.

We therefore have $V_{0,3}(L)=1$ and as in \cref{ex:0-4-0-3} we get for $D_{1,3,0,3}(T)\in H_0(|\mcM_{1,3}|,\QQ)\cong \QQ$ that\footnote{Deligne--Mumford spaces are connected as Teichmüller spaces are connected and nodal curves in Deligne--Mumford space can be smoothened to smooth curves.}

\begin{equation*}
  D_{1,3,0,3}(T)=27\cdot\frac{1}{3}[\pt]=9[\pt].
\end{equation*}

This can be interpreted as follows. It implies that $\ev:\mcM_{1,3,0,3}(T)\lra\mcM_{0,3}$ is an actual morphism covering as there cannot be any nodal Hurwitz covers and $\dim\mcM_{1,3,0,3}(T)=0$. Therefore $\mcM_{1,3,0,3}(T)$ is a disjoint union of points with finite automorphism groups such that their inverses add up to $\frac{1}{3}$. As there was only one Hurwitz cover with three automorphisms this means there is only one such point with an automorphism group with three elements.

In this particular example one can also see the result slightly differently. If there is one Hurwitz cover then it has at least three automorphisms from permuting the sheets cyclically. As $S_3$ has six elements it would be possible there are two Hurwitz covers with each six automorphisms. However, pick a small disc close to one of the branch points in a standard neighborhood as in \cref{lem2} and consider its three disjoint preimages. Any automorphism of the discs has to biholomorphically map these discs to themselves and needs to extend to a map fixing the critical point. This means there cannot be an automorphism fixing one of these discs as this would need to be the identity.

\subsection{Hyperelliptic Curves of Genus \texorpdfstring{$\boldsymbol{2}$}{2}}

Consider the case $g=2, d=2, h=0$ such that $k=n=2g+2=6$ and $l_1=\ldots=l_6=2$. This is the case of genus-$2$ hyperelliptic surfaces. Denote the marked points such that the marked points are related by $\nu=\id$. We can calculate the usual factor $K=2^6=64$ and because of $I_i=\{i\}$ we have $\beta_i\in\NN^{\{i\}}\cong\NN$ and thus $\beta_i=\alpha_i$. Then \cref{eq:hvpsi_top} reads
\begin{equation*}
64\cdot H_{2,6,0,6}V_{0,6}(L)=\sum_{\substack{\alpha\in\NN^6\\ |\alpha|=3}}\frac{L^{2\alpha}}{(2\cdot 2)^{3}}\frac{2^6}{\alpha_1!\cdots\alpha_6!}\left\langle\psi_1^{\alpha_1}\cdots\psi_6^{\alpha_6},D_{2,6,0,6}(T)\right\rangle
\end{equation*}
Now we have to understand what $D_{2,6,0,6}(T)\in H_6(|\mcM_{2,6}|,\QQ)$ is. One might think that because every genus 2 surface is hyperelliptic the map $\fgt:|\mcM_{2,6,0,6}(T)|\lra|\mcM_{2,6}|$ should be surjective and thus the image of the fundamental class should go to some multiple of the fundamental class (ignoring compactification issues). However, a quick dimension count shows that this cannot be and in fact the map $\fgt$ records the position of the  critical points. But given a hyperelliptic surface $C$, the critical points are the Weierstrass points which are uniquely determined by the underlying source curve, thus it is not possible to prescribe the curve together with the position of the critical points. In fact, composing $\fgt$ with the map to $|\mcM_{2,0}|$ forgetting the marked points is surjective.

However, we have $\mcH_0(2,2,2,2,2,2)=H_{2,6,0,6}(T)=\frac{1}{2}$ as $\mfK=1$. This is easily seen by looking at the combinatorics as $S_2$ only has two elements, by using a computer program or by noticing that a hyperelliptic surface has precisely one non-trivial automorphism which corresponds to the map interchanging the sheets of the two-fold covering over $S^2$. A computer program then shows
\begin{align*}
  V_{0,6}(L)=&\frac{1}{48}\sum_{i=1}^6L_i^6 + \frac{3}{16}\sum_{\substack{i,j=1\ldots 6\\ i\neq j}}L_i^4L_j^2 + \frac{3}{4}\sum_{\substack{i,j,k=1\ldots 6\\ i\neq j\neq k}}L_i^2L_j^2L_k^2 +\\
&+ \frac{3\pi^2}{2}\sum_{i=1}^6L_i^4+6\pi^2\sum_{\substack{i,j=1\ldots 6\\ i\neq j}}L_i^2L_j^2+26\pi^4\sum_{i=1}^6L_i^2+\frac{244\pi^6}{3}.
\end{align*}
This gives in the highest order
\begin{align*}
 \frac{2}{3}\sum_{i=1}^6L_i^6 & + 6\sum_{\substack{i,j=1\ldots 6\\ i\neq j}}L_i^4L_j^2 + 24\sum_{\substack{i,j,k=1\ldots 6\\ i\neq j\neq k}}L_i^2L_j^2L_k^2= \frac{1}{6}\sum_{i=1}^6\langle\psi_i^3,D_{2,6,0,6}(T)\rangle L_i^6+ \\
& +\frac{1}{2}\sum_{\substack{i,j=1\ldots 6\\ i\neq j}}\langle \psi_i^2\psi_j,D_{2,6,0,6}(T)\rangle L_i^4L_j^2+\sum_{\substack{i,j,k=1\ldots 6\\ i\neq j\neq k}}\langle\psi_i\psi_j\psi_k,D_{2,6,0,6}(T)\rangle L_i^2L_j^2L_k^2.
\end{align*}
So for example we have $\langle\psi_1^3,D_{2,6,0,6}(T)\rangle=4$.

\appendix

\chapter{Appendix}
\label{appendix}

\section{Overview of Moduli Spaces}

\label{sec:moduli-spaces-3}

\index{Moduli Space!Of Admissible Riemann Surfaces}
\index{Moduli Space!Of Closed Riemann Surfaces With Multicurve}
\index{Deligne--Mumford Space}

\begin{table}[H]
\begin{tabular*}{\textwidth}{c p{12cm}}
\toprule
\textbf{Notation} & \textbf{Name and Description} \\
\midrule \\

$\mcR_{g,k}$ & {\RaggedRight \textbf{Deligne--Mumford space}:} \linebreak
Equivalence classes of stable nodal closed connected curves of genus $g$ with $k$ marked points. $$\{(C,\bq)\mid C\text{ stable cx.\ curve of genus }g,\bq\in C^k\setminus\Delta\}$$ \\

$\wh{\mcR}_{g,k}$ & {\RaggedRight \textbf{Moduli space of admissible Riemann surfaces}:} \linebreak Equivalence classes of stable nodal connected Riemann surfaces $C$ of genus $g$ with $k$ enumerated boundaries or cusps together with a marked point on a reference curve $\Gamma_i(C)$ close to each boundary component or cusp $\del_iC$. $$\{(C,\bz)\mid C\text{ admissible Riemann surface}, z_i\in\Gamma_i(C)\;\forall\; 1\leq i\leq k\}$$ \\

$\wt{\mcR}_{g,k}$ & {\RaggedRight \textbf{Moduli space of complex surfaces with a multicurve}:} \linebreak Equivalence classes of closed stable connected nodal surfaces $C$ of genus $g$ with $k\in 2\ZZ$ enumerated marked points $\bq$ together with a $k$-multicurve $\bG$ such that $q_{2i-1}$ and $q_{2i}$ are either on a sphere component and one corresponding element of $\bG$ is contractible or they are contained in a disc bounded by a curve in the multicurve. $$ \left\{ \begin{aligned} (C,\bq,\Gamma)\mid  \; & C\text{ cl.\ stable nodal surface}, \bq\in C^k\setminus\Delta, \bG\text{ a multicurve}  \\ & \text{whose elements bound pairs of pants with } \bq \end{aligned} \right\}$$ \\

\bottomrule
\end{tabular*}

\caption{This table shows the objects of the groupoid categories corresponding to moduli spaces of surfaces. The morphisms are always the natural ones. Note that in $\wh{\mcR}_{g,k}$ we allow boundary components as well as degenerate boundary components, i.e.\ cusps. In contrast $\wt{\mcR}_{g,k}$ and $\mcR_{g,k}$ contain only closed surfaces. These categories correspond to those used by Mirzakhani in \cite{mirzakhani_weil-petersson_2007}. The orbifold groupoids $\mcM_{g,k},\wt{\mcM}_{g,k}$ and $\wh{\mcM}_{g,k}$ were constructed in \cite{robbin_construction_2006} and the gluing idea is contained in \cite{mirzakhani_weil-petersson_2007}. Thick diagonals are denoted by $\Delta$, i.e.\ $\Delta\subset X^n$ is the subset $\{(x_1,\ldots,x_n)\mid \exists\, i,j:x_i=x_j\}$.}

\label{tab:moduli-spaces-wo-maps}
\end{table}

\index{Hurwitz Cover!Bordered}
\index{Moduli Space!Of Closed Hurwitz Covers}
\index{Moduli Space!Of Admissible Hurwitz Covers}
\index{Moduli Space!Of Closed Hurwitz Covers With Multicurve}

\begin{table}[H]
\begin{tabular*}{\textwidth}{c p{12cm}}
\toprule
\textbf{Notation} & \textbf{Name and Description} \\
\midrule \\

$\mcR_{g,k,h,n}(T)$ & {\RaggedRight \textbf{Moduli space of closed Hurwitz covers}:} \linebreak
Equivalence classes of holomorphic maps $u:C\lra X$ with closed nodal complex curves $C$ of genus $g$ and $X$  of genus $h$ of type $T$ where all branched points and their preimages are marked and enumerated. $$ \begin{aligned} \{(C & , X, u , \bq, \bp)\mid C,X\text{ closed stable nodal Riemann surfaces of genus} \\ & g,h,\text{ respectively}, u:C\rightarrow X \text{ hol., sat. }T,\bq\in C^k\setminus\Delta,\bp\in X^n\setminus\Delta\} \end{aligned} $$ \\

$\wh{\mcR}_{g,k,h,n}(T)$ & {\RaggedRight \textbf{Moduli space of admissible Hurwitz covers}:} \linebreak
Equivalence classes of holomorphic maps $u:C\lra X$ with admissible complex curves\footnotemark $C$ of genus $g$ and $X$ of genus $h$ of type $T$ where all branched points and their preimages as well as boundaries and their preimages are marked and enumerated together with one marked point $z_j$ per boundary component or cusp on $C$ such that $u(z_j)=u(z_i)$ if $\nu(i)=\nu(j)$. $$ \left\{ \begin{aligned} (C,X,u,\bq,\bp,\bz)\mid \; & C,X\text{ adm.}, u: C\rightarrow X \text{ hol., sat. }T, \bq\in C^k\setminus\Delta, \\ & \bp\in X^n\setminus\Delta, z_i\in\Gamma_i(C)\; \forall\; 1\leq i\leq k, u(z_i)=u(z_j) \\ & \forall\; i,j\text{ s.t. } \nu(i)=\nu(j) \end{aligned} \right\} $$ \\

$\wt{\mcR}_{g,k,h,n}(T)$ & {\RaggedRight \textbf{Moduli space of closed Hurwitz covers with a multicurve}:} \linebreak
Equivalence classes of holomorphic maps $u:C\lra X$ with closed complex curves $C$ of genus $g$ and $X$ of genus $h$ of type $T$ where all $n\in 2\ZZ$ branched points and their preimages are marked and enumerated together with a multicurve $\bG$ on $X$ such that consecutive pairs of marked points on $X$ are either on a sphere component and there is a contractible curve in $\bG$ or they are contained in a disc bounded by an element in the multicurve. $$ \left\{ \begin{aligned} (C,X,u,\bq,\bp,\bG)\mid \; & C,X\text{ cl. stable nodal surface}, u: C\rightarrow X \text{ hol.,} \\ & \text{sat. }T, \bq\in C^k\setminus\Delta, \bp\in X^n\setminus\Delta, \Gamma \text{ a multicurve} \\ & \text{whose elements bound pairs of pants with } \bp \end{aligned} \right\} $$ \\

\bottomrule
\end{tabular*}

\caption{This table shows most categories used in this thesis that include maps. Note however that the usage of $\mcR$ means that these are the general categories and written are the objects only. The morphisms are always (pairs of) maps between objects preserving everything. The corresponding orbifold categories will use an $\mcM$ and have far fewer objects and morphisms. The hat and tilde symbols correspond to the ones in \cref{tab:moduli-spaces-wo-maps}.}

\label{tab:moduli-spaces-with-maps}
\end{table}
\footnotetext{{See \cref{def:admissible-riemann-surface}.}}

\section{Definition of Convergence to a Broken Holomorphic Curve}

\label{appendix-convergence-broken-holomorphic-curves}

In this chapter we will give the definitions in order to understand convergence of sequences of $J$-holomorphic curves in a neck-stretching sequence to a broken holomorphic curve as stated in \cref{def:convergence-broken-hol-curve}. We follow again Cieliebak--Mohnke in \cite{cieliebak_compactness_2005}. First suppose we are given sequences of numbers
\begin{equation*}
  -w_k=r_k^0<r_k^1<\cdots <r_k^{N+1}=0
\end{equation*}
such that $r_k^{\nu+1}-r_k^{\nu}\to\infty$ as $k\to\infty$. The map $\cdot + R$ for $R\in\RR$ is defined as
\begin{align*}
  X_k & \lra \mathring{X}_0\cup_{M_-\sqcup M_+}[-w_k-\epsilon+R,\epsilon+R]\times M \\
  x & \longmapsto x+R\coloneqq
  \begin{cases}
    x & x\in\mathring{X}_0, \\
    (r+R,y) & x=(r,y)\in [-\epsilon-w_k,\epsilon]\times M.
  \end{cases}
\end{align*}
Then we define
\begin{align*}
  X_k^{\nu} & \coloneqq \ol{X}_0\cup_{M_-\sqcup M_+}[-w_k-\epsilon-r_k^{\nu},\epsilon+-r_k^{\nu}]\times M \qquad\text{for }\nu=1,\ldots,N, \\
  X_k^0 & \coloneqq \faktor{\left[-\frac{w_k}{2},\epsilon\right]\times M\cup_{M_+}\mathring{X}_0\cup_{M_-}\left[-\epsilon,\frac{w_k}{2}\right]\times M}{\left(-\frac{w_k}{2},x\right)\sim \left(\frac{w_k}{2},x\right)}
\end{align*}
and using the maps $f_k: C_k\lra X_k$ we also define
\begin{align*}
  f_k^{\nu}:C_k & \lra X_k^{\nu} \\
  z & \longmapsto f_k^{\nu}(z)\coloneqq f_k(z)-r_k^{\nu}
\end{align*}
for $\nu=1,\ldots, N$ as well as
\begin{align*}
  f_k^0:C_k & \lra X_k^0 \\
  z & \longmapsto f_k^0(z)\coloneqq
  \begin{cases}
    f_k(z) & f_k(z)\in \left[-\frac{w_k}{2},\epsilon\right]\times M\cup \mathring{X}_0, \\
    f_k(z)+w_k & f_k(z)\in \left[-w_k-\epsilon,-\frac{w_k}{2}\right]\times M.
  \end{cases}
\end{align*}

\begin{figure}[H]
    \centering
    \def\svgwidth{\textwidth}
    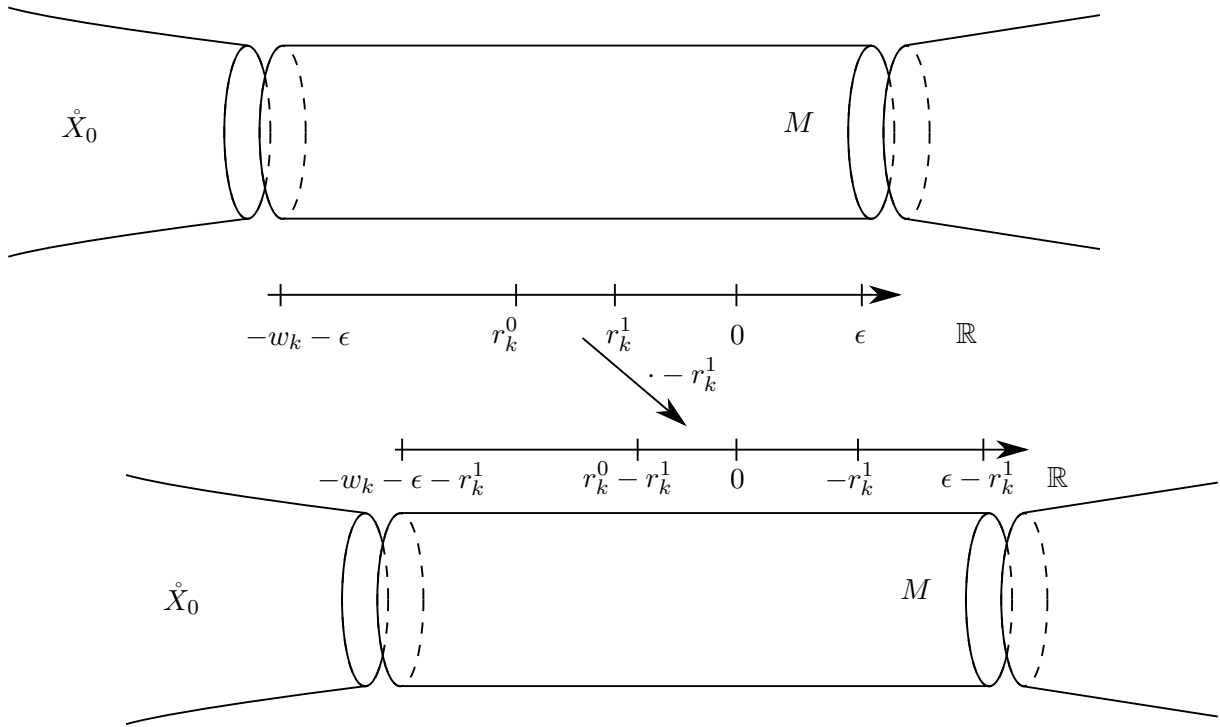
    \caption{This shows an example of the shift maps $\cdot+r_k^{\nu}$. Note that for every $k$ and $\nu$ we define a map such that the points $r_k^{\nu}$ are shifted to zero. The points $r_k^{\nu}$ will then be chosen to correspond to a sequence of values for which the corresponding holomorphic curve $f_k$ splits and develops an asymptotic puncture. Thus in the shifted picture this takes place at the real coordinate zero. Similarly the maps $f_k^0$ are defined using shifts such that again the splitting happens at real coordinate zero.}
    \label{fig:shift-maps}
\end{figure}

We denote the projections on $\RR\times M$ by
\begin{equation*}
  \xymatrix{
    & \RR\times M \ar[dl]_{\pi_{\RR}} \ar[dr]^{\pi_M} & \\
    \RR & & M
    }.
\end{equation*}

Also we choose a sequence of diffeomorphisms $\phi_k:[-\epsilon-w_k,\epsilon]\lra[-\epsilon,\epsilon]$ with $\phi_k'=1$ near the boundaries. Then we define shifts of these diffeomorphisms by
\begin{align*}
  \phi_k^{\nu}:[-\epsilon-w_k-r_k^{\nu},\epsilon-r_k^{\nu}] & \lra[-\epsilon,\epsilon] \\
  z & \longmapsto \phi_k^{\nu}(z)\coloneqq \phi_k(z+r_k^{\nu})
\end{align*}
for $\nu=1,\ldots,N$ and
\begin{align*}
  \phi_k^0:\faktor{\left[-\frac{w_k}{2},\epsilon\right]\cup \left[-\epsilon,\frac{w_k}{2}\right]}{\left(-\frac{w_k}{2}\sim \frac{w_k}{2}\right)}  & \lra[-\epsilon,\epsilon] \\
  z & \longmapsto \phi_k^0(z)\coloneqq 
  \begin{cases}
    \phi_k(z) & z\in \left[-\frac{w_k}{2},\epsilon\right], \\
    \phi_k(z)-w_k & z\in \left[-\epsilon,\frac{w_k}{2}\right].
  \end{cases}
\end{align*}
Last we assume that $\lim_{k\to\infty}\phi_k^{\nu}=\phi^{\nu}$ in $\cin_{\text{loc}}$ for $\nu=0,\ldots,N$ where $\phi^{\nu}$ is the diffeomorphism defined after \cref{eq:def-phi-shift}.

\section{Calculations of Weil--Petersson Volumes}

\label{appendix:calculation-weil-petersson-volumes}

\subsection{The Recursion Relation of Mirzakhani}

Before stating the McShane identity let us define a few functions. In the following, $V_{h,n}(L)$ is the Weil--Petersson symplectic volume of the moduli space of bordered Riemann surfaces of genus $h$ and $n$ boundary components of lengths $L_i$ for $i=1,\ldots,n$. We will also need the functions $H(x,y)\coloneqq\frac{1}{1+\exp(\frac{x+y}{2})}+\frac{1}{1+\exp(\frac{x-y}{2})}$ and $m(h,n)\coloneqq\delta_{h,1}\delta_{n,1}$. Then we can define

\begin{align*}
  \mcA_{h,n}^{\text{con}} & \coloneqq\frac{1}{2}\int_0^{\infty}\int_0^{\infty}xy\frac{V_{h-1,n+1}(x,y,L_2,\ldots,L_n)H(x+y,L_1)}{2^{m(h-1,n+1)}}\dd x\dd y, \\
  \mcA_{h,n}^{\text{dcon}} & \coloneqq\frac{1}{2}\int_0^{\infty}\int_0^{\infty}xy\sum_{a\in\mcI_{h,n}}\frac{V_{h_1,n_1+1}(x,L_{I_1})}{2^{m(h_1,n_1+1)}}\frac{V_{h_2,n_2+2}(x,L_{I_2})}{2^{m(h_2,n_2+1)}}H(x+y,L_1)\dd x\dd y, \\
  \mcB_{h,n}(L) & \coloneqq \frac{2^{-m(h,n-1)}}{2}\sum_{j=2}^n\int_0^{\infty}x(H(x,L_1+L_j)+H(x,L_1-L_j))\times \\
  & \phantom{\coloneqq \frac{2^{-m(h,n-1)}}{2}\sum_{j=2}^n\int_0^{\infty}x}\times V_{h,n-1}(x,L_2,\ldots,\wh{L_j},\ldots,L_n)\dd x,
\end{align*}

where $\mcI_{h,n}\coloneqq\{(h_1,I_1),(h_2,I_2)\}$ is the set of ordered pairs with $I_1,I_2\subset\{2,\ldots,n\}$ and $0\leq h_1,h_2\leq h$ such that $I_1$ and $I_2$ are disjoint and $I_1\cup I_2=\{2,\ldots,n\}$, $h_1, h_2$ and $n_i=|I_i|$ satisfy $2\leq 2h_i+n_i$ for $i=1,2$ and $h_1+h_2=h$. Then the McShane identity is given as follows.

\begin{thm}[Mirzakhani]
  With the definitions as above it was proven in \cite{mirzakhani_weil-petersson_2007} that one has
  \begin{equation}
     \frac{\partial}{\partial L_1}L_1V_{h,n}(L)=\mathcal{A}^{\text{con}}_{h,n}(L)+\mathcal{A}^{\text{dcon}}_{h,n}(L)+\mathcal{B}_{h,n}(L)
  \label{eq:mcshane}
  \end{equation}
  as well as the initial conditions
  \begin{align*}
    V_{0,3}(L) & = 1, \\
    V_{1,1}(L) & = \frac{\pi^2}{6} + \frac{L^2}{24}.
  \end{align*}
\end{thm}

\begin{rmk}
  This result allows to compute recursively arbitrary Weil--Petersson volumes $V_{h,n}(L)$. In particular the volumes are polynomials in $L$ allowing us to deduce \cref{cor:main-result} from \cref{thm:main-result}.
\end{rmk}

\subsection{Calculation of \texorpdfstring{$\boldsymbol{V_{1,2}(L_1,L_2)}$}{V_12(L_1,L_2)}}

\label{sec:calculation-wp-volume-1-2}

As there is no pair of pants bounding $\del X_1$ such that its complement is a disconnected stable surface the recursion relation for $V_{1,2}(L_1,L_2)$ has no $\mcA^{\text{dcon}}$-term and reads
\begin{align}
  \frac{\partial}{\partial L_1}(L_1V_{1,2}(L_1,L_2))=&\frac{1}{2}\int_0^{\infty}\int_0^{\infty}xyH(x+y,L_1)\dd x\dd y +  \\
&+\frac{1}{4}\int_0^{\infty}x\big(H(x,L_1+L_2)+H(x,L_1-L_2)\big)V_{1,1}(x)\dd x, \numberthis \label{eq:rec_rel_vol}
\end{align}
where $H(x,y)=\frac{1}{1+e^{x+y}}+\frac{1}{1+e^{x-y}}$ and $V_{1,1}(L_1)=\frac{L_1^2}{24}+\frac{\pi^2}{6}$. It can be shown that
\begin{align*}
  \int_0^{\infty}\int_0^{\infty}x^{2i+1} & y^{2j+1}H(x+y,t)\dd x\dd y= \\
  &=(2i+1)!(2j+1)!\sum_{k=0}^{i+j+2}\zeta(2k)(2^{2k+1}-4)\frac{t^{2i+2j+4-2k}}{(2i+2j+4-2k)!}
\end{align*}
and
\begin{equation*}
  \int_0^{\infty}x^{2i+1}H(x,t)\dd x=(2i+1)!\sum_{k=0}^{i+1}\zeta(2k)(2^{2k+1}-4)\frac{t^{2i+2-2k}}{(2i+2-2k)!}.
\end{equation*}
Using these equations we obtain for \ref{eq:rec_rel_vol}
\begin{align*}
  \frac{\partial}{\partial L_1} & (L_1V_{1,2}(L_1,L_2))= \\
  &=\frac{L_1^4}{48}+\frac{\pi^2}{6}L_1^2+\frac{7}{45}\pi^4+\frac{\pi^2}{24}\left(\frac{(L_1+L_2)^2}{2}+\frac{2}{3}\pi^2+\frac{(L_1-L_2)^2}{2}+\frac{2}{3}\pi^2\right)+ \\
 +\frac{1}{96} & \left(\frac{(L_1+L_2)^4}{4}+2\pi^2(L_1+L_2)^2+\frac{28}{15}\pi^4+\frac{(L_1-L_2)^4}{4}+2\pi^2(L_1-L_2)^2+\frac{28}{15}\pi^4\right) \\
  &=\frac{5}{192}L_1^4+\frac{1}{32}L_1^2L_2^2+\frac{1}{192}L_2^4+\frac{\pi^2}{4}L_1^2+\frac{\pi^2}{12}L_2^2+\frac{\pi^4}{4}.
\end{align*}
As the argument on the left hand side is zero for $L_1=0$ we can integrate this equation without an additive constant and obtain
\begin{equation*}
  V_{1,2}(L_1,L_2)=\frac{1}{192}\left(L_1^4+2L_1^2L_2^2+L_2^4\right)+\frac{\pi^2}{12}(L_1^2+L_2^2)+\frac{\pi^4}{4}.
\end{equation*}

\subsection{Calculation of \texorpdfstring{$\boldsymbol{V_{0,5}(L)}$}{V_05(L)}}

\label{sec:calculation-v-05}

Notice first that there is no $\mcA^{\text{con}}$-term because there is no pair of geodesics bounding a pair of pants with a boundary component such that the complement is connected (because this would have to have genus $-1$). The disconnected sum consists of volumes $V_{0,3}$ which are equal to one, the sum however counts \emph{ordered} pairs of decompositions into disconnected surfaces (again by a pair of geodesics bounding a pair of pants with a specified boundary component), of which there are six. Therefore we obtain
\begin{align*}
  \mcA^{\text{dcon}}_{0,5}(L)&=\frac{6}{2}\int_0^{\infty}\int_0^{\infty}xyH(x+y,L_1)\dd x\dd y, \\
  &=3\frac{1}{6}F_3(L_1)\\
  &=\frac{1}{8}L_1^4+\pi^2L_1^2+\frac{14}{15}\pi^4.
\end{align*}
The sum over pairs of boundary components and their bounded pairs of pants is somewhat longer and is given by
\begin{align*}
  \mcB_{0,5}(L)&=\frac{1}{2}\sum_{j=2}^5\int_0^{\infty}x(H(x,L_1+L_j)+H(x,L_1-L_j))V_{0,4}(x,L_2,\ldots,\wh{L_j},\ldots,L_5)\dd x\\
  &=\frac{1}{8}(L_2^4+L_3^4+L_4^4+L_5^4)+L_1^2(\frac{3}{2}(L_2^2+L_3^2+L_4^2+L_5^2)+8\pi^2)+\\
  &\phantom{=}+\frac{1}{2}L_1^4+\frac{1}{2}\sum_{\substack{i,j=2\\ i\neq j}}^5L_i^2L_j^2+3\pi^2(L_2^2+L_3^2+L_4^2+L_5^2)+\frac{1}{8}(L_2^4+L_3^4+L_4^4+L_5^4),
\end{align*}
where we have used
\begin{align*}
  \int_0^{\infty}x^3H(x,L_1+L_j)&=\frac{1}{4}(L_1+L_j)^4+2\pi^2(L_1+L_j)^2+\frac{28}{15}\pi^4,\\
  \int_0^{\infty}xH(x,L_1+L_j)&=\frac{1}{2}(L_1+L_j)^2+\frac{2}{3}\pi^2
\end{align*}
as well as many simplifications. Integrating (\ref{eq:mcshane}) with respect to $L_1$ and dividing by $L_1$ we obtain
\begin{equation*}
  V_{0,5}(L)=\frac{1}{8}\sum_{i=1}^5L_i^4+\frac{1}{2}\sum_{\substack{i,j=1}\\ i\neq j}^5L_i^2L_j^2+3\pi^2\sum_{i=1}^5L_i^2+10\pi^4.
\end{equation*}

\chapter{Source Code for Calculations of Hurwitz Numbers}

The following SAGE/Python code was used to calculate the standard Hurwitz numbers $\mcH_h(T_1,\ldots,T_n)$ used in \cref{sec:appl-exampl} and defined in \cref{sec:hurwitz-numbers} via the combinatorial description in \cref{sec:comb-desciption-hurwitz-numbers}.

\begin{lstlisting}[language=Python, basicstyle=\footnotesize]
import itertools

def calculate_hurwitz_numbers(h, profile):
    ##### Parameters
    
    # Profile is given by a tuple of tuples where the outer tuple enumerates the marked points of the target surface and the inner tuples are (ordered) degrees of critical points in that fibre. h is the target genus.
    
    ##### Preparations
    
    # Other parameters calculated from the profile.
    n = len(profile)
    d = sum(profile[0])
    
    # Symmetric group acting on the fibres.
    
    G = SymmetricGroup(d)
    e = G.identity()
    
    # Initialize variables for results.
    
    hurwitz_number = 0
    possible_decompositions = []
    common_conjugacy_classes = [[()]]
    automorphisms = []
    
    
    ##### Find all Decompositions
    
    # This loop goes over all possible (n+2h)-tuples of elements in G and first checks whether they have the correct (admissible) cycle type. If yes it then verifies if the tuple satisfies the correct equation and saves it to "possible_decompositions". This algorithm ignores some obvious optimizations.        
    
    for group_tuple in itertools.product(G, repeat = n+2*h):
        admissible = True
        for i in range(2*h,n+2*h):
            if Permutation(group_tuple[i]).cycle_type() != profile[i-2*h]:
                admissible = False
        if admissible == True:
            relation = e
            for k in range(h):
                relation = relation * group_tuple[2*k] * group_tuple[2*k+1] * group_tuple[2*k].inverse() * group_tuple[2*k+1].inverse()
            for k in range(n+2*h-1,2*h-1,-1):
                relation = relation * group_tuple[k].inverse()
            if relation == e:
                possible_decompositions.append(group_tuple)
                                
    ##### Sort all Decompositions into Common Conjugacy Classes
    
    # Loop goes through all found decompositions and sorts them into common conjugacy classes.
    
    i=0
    j=0
    while i in range(len(possible_decompositions)):
        j=0
        N = len(common_conjugacy_classes)
        assigned = False
        while j in range(N):
            if possible_decompositions[i] in eval("gap.function_call('Orbit',[gap.SymmetricGroup(d),common_conjugacy_classes[j][0],'OnTuples'])"):
                common_conjugacy_classes[j].append(possible_decompositions[i])
                assigned = True
            j = j + 1
        if assigned == False:
            common_conjugacy_classes.append([possible_decompositions[i]])
            automorphisms.append(gap.function_call('Size',[gap.function_call('Stabilizer',[gap.SymmetricGroup(d),possible_decompositions[i],'OnTuples'])]))
        i = i + 1        
    
    ##### Remove Dummy Conjugacy Class
    
    del common_conjugacy_classes[0]
       
    ##### Calculate Hurwitz number
    
    for a in common_conjugacy_classes:
        hurwitz_number = hurwitz_number + 1 / int(automorphisms[common_conjugacy_classes.index(a)])
    
    ##### Return Result
    
    return hurwitz_number, common_conjugacy_classes, automorphisms
\end{lstlisting}

\cleardoublepage
%\manualmark
%\markboth{\spacedlowsmallcaps{\bibname}}{\spacedlowsmallcaps{\bibname}}
\refstepcounter{dummy}
\addtocontents{toc}{\protect\vspace{1cm}}
\addcontentsline{toc}{chapter}{Bibliography}
\label{app:bibliography}
\printbibliography

\cleardoublepage
%\manualmark
%\markboth{\spacedlowsmallcaps{Index}}{\spacedlowsmallcaps{Index}} % work-around to
\addtocontents{toc}{\protect\vspace{1cm}} % to have the bib a bit from the rest in the toc
\refstepcounter{dummy}
\label{app:index}
\addcontentsline{toc}{chapter}{Index}
\printindex

\end{document}